\documentclass{article}
\usepackage{latexsym, amscd, amsfonts, eucal, mathrsfs, amsmath, amssymb, amsthm, xypic,xr, makeidx, stmaryrd}
\def\Z{\mathbf{Z}}

\DeclareMathOperator{\CatP}{{\mathfrak P}}

\DeclareMathOperator{\exc}{exc}
\DeclareMathOperator{\Lin}{{\mathcal L}in}

\DeclareMathOperator{\Spectra}{Sp}
\DeclareMathOperator{\PreSpectra}{PSp}
\newcommand{\ICAT}{{\bf Cat}_{\infty}}
\newcommand{\dICAT}{{\bf Cat}^{\Delta}_{\infty}}
\DeclareMathOperator{\wellp}{wp}
\DeclareMathOperator{\ex}{ex}

\DeclareMathOperator{\sNerve}{N}
\DeclareMathOperator{\sCat}{\mathcal{C}at_{\Delta}}
\DeclareMathOperator{\sMap}{M}

\DeclareMathOperator{\tNerve}{N}

\newcommand{\LCone}[1]{C(#1)}

\newcommand{\LCones}[2]{C_{#1}(#2)}
\DeclareMathOperator{\scUn}{Un^{sc}}

\DeclareMathOperator{\mCat}{{{\mathcal C}at}_{\Delta}^{+}}

\DeclareMathOperator{\scSt}{St^{sc}}
\DeclareMathOperator{\Fr}{Fr}

\DeclareMathOperator{\scSet}{{\mathcal S}et_{\Delta}^{sc}} 
\newcommand{\PreSeg}[1]{Seg_{#1}}

\DeclareMathOperator{\scNerve}{N^{sc}} 
\DeclareMathOperator{\scCoNerve}{{\mathfrak C}^{sc}} 
\newcommand{\mset}[1]{{(\mathcal S}et^{+}_{\Delta})_{/#1}} 
\DeclareMathOperator{\sFun}{Fun^{sc}}

\DeclareMathOperator{\UnPre}{UnPre}
\newcommand{\CSS}[1]{{CSS_{#1}}}
\newcommand{\SegS}[1]{SS_{#1}}
\newcommand{\sdd}{sd}

\newcommand{\Andre}{Andr\'{e}}

\DeclareMathOperator{\mSet}{{\mathcal{S}et}_{\Delta}^{+}}

\newcommand{\Un}{Un}

\DeclareMathOperator{\calH}{\mathcal{H}}

\DeclareMathOperator{\LFun}{\Fun^{L}}

\newcommand{\toposref}[1]{T.\ref{HTT-#1}}
\newcommand{\stableref}[1]{S.\ref{STA-#1}}
\newcommand{\monoidref}[1]{M.\ref{MON-#1}}
\newcommand{\symmetricref}[1]{C.\ref{SYM-#1}}

\newcommand{\degree}{\text{o}}
\newcommand{\bfA}{{\mathbf A}}
\newcommand{\bfB}{{\mathbf B}}
\newcommand{\bfC}{{\mathbf C}}

\DeclareMathOperator{\Gp}{Gp} 
\DeclareMathOperator{\cDelta}{{\bf \Delta}}
\DeclareMathOperator{\inv}{\sim}
\DeclareMathOperator{\Grp}{\mathcal{G}pd}
\DeclareMathOperator{\Set}{\mathcal{S}et}
\DeclareMathOperator{\sSet}{\mathcal{S}et_{\Delta}}
\DeclareMathOperator{\sCoNerve}{\mathfrak{C}}
\DeclareMathOperator{\Nerve}{N}
\DeclareMathOperator{\Cat}{\mathcal{C}at}
\newcommand{\h}[1]{\rm{h} \! #1}

\newcommand{\Adjoint}[4]{\xymatrix@1{#2 \ar@<.4ex>[r]^-{#1} & #3 \ar@<.4ex>[l]^-{#4}}}

\newcommand{\et}{\'{e}t}

\newcommand{\bigdot}{\bullet}

\DeclareMathOperator{\Stab}{Stab}

\DeclareMathOperator{\bHom}{Map}

\newcommand{\Cech}{\v{C}ech\,}

\DeclareMathOperator{\Mod}{\mathcal{M}od}

\DeclareMathOperator{\calM}{\mathcal{M}}

 \DeclareMathOperator{\G}{\mathbf{G}}
\DeclareMathOperator{\calZ}{\mathcal{Z}}

\DeclareMathOperator{\DerRing}{\mathcal{SCR}}

\DeclareMathOperator{\colim}{colim}

\DeclareMathOperator{\calA}{\mathcal{A}}
\DeclareMathOperator{\bd}{\partial}

\DeclareMathOperator{\calE}{\mathcal{E}}

\DeclareMathOperator{\calO}{\mathcal{O}}

\DeclareMathOperator{\calK}{\mathcal{K}}
\DeclareMathOperator{\Spec}{{\bf Spec}}

\DeclareMathOperator{\calF}{\mathcal{F}}
\DeclareMathOperator{\calG}{\mathcal{G}}
\DeclareMathOperator{\Hom}{Hom} 
\DeclareMathOperator{\HH}{H} 
\DeclareMathOperator{\id}{id} \DeclareMathOperator{\Fun}{Fun}
\DeclareMathOperator{\calC}{\mathcal{C}}
\DeclareMathOperator{\calI}{\mathcal{I}}
\DeclareMathOperator{\calN}{\mathcal{N}}
\DeclareMathOperator{\calJ}{\mathcal{J}}

\DeclareMathOperator{\SSet}{\mathcal{S}}

\newcommand{\St}{St}

\newcommand{\lNerve}{{\mathfrak F}}

\DeclareMathOperator{\calX}{\mathcal{X}}

\DeclareMathOperator{\hocolim}{hocolim}
\DeclareMathOperator{\calY}{\mathcal{Y}}

\DeclareMathOperator{\calD}{\mathcal{D}}
 
\DeclareMathOperator{\calP}{\mathcal{P}} \topmargin=0in

\newtheorem{theorem}{Theorem}[subsection]
\newtheorem{lemma}[theorem]{Lemma}
\newtheorem{proposition}[theorem]{Proposition}

\newtheorem{corollary}[theorem]{Corollary}

\theoremstyle{definition}
\newtheorem{definition}[theorem]{Definition}
\newtheorem{construction}[theorem]{Construction}
\newtheorem{example}[theorem]{Example}

\newtheorem{notation}[theorem]{Notation}

\newtheorem{warning}[theorem]{Warning}
\newtheorem{remark}[theorem]{Remark}
\newtheorem{variant}[theorem]{Variant}

\makeindex

\begin{document}

\title{$(\infty,2)$-Categories and the Goodwillie Calculus I}

\maketitle
\tableofcontents

\section*{Introduction}

The theory of higher categories is notorious for having an excessive proliferation of definitions,
many of which are difficult to compare with one another. Nevertheless, a general consensus has emerged in some special cases. Let us use the term {\it $(\infty,n)$-category} to indicate a higher category in which all $k$-morphisms are assumed to be invertible for $k > n$. It has long been understood that the theory of $(\infty,0)$-categories (that is, higher categories in which {\em all} morphisms are required to be invertible) should be equivalent to classical homotopy theory.
Consequently, for practical purposes one can {\em define} an $(\infty,0)$-category to be a topological
space, or a simplicial set which satisfies the Kan extension condition.

The theory of $(\infty,1)$-categories is also quite well understood, though in this case there
is a variety of possible approaches. Arguably the simplest of these is the Boardman-Vogt theory {\it weak Kan complexes}: that is, simplicial sets which satisfy a weaker version
of the Kan extension condition (these are also known as {\it quasicategories} in the literature;
we will follow the terminology of \cite{topoi} and refer to them simply as $\infty$-categories).
However, there are a number of other possible approaches: for example, one could define an $(\infty,1)$-category to be a topological
category (that is, a category $\calC$ in which every mapping set $\Hom_{\calC}(X,Y)$ is endowed
with a topology, such that the composition of morphisms is continuous), a simplicial category,
a Segal category, or a complete Segal space. These notions are equivalent to one another.
More precisely, we have the following:

\begin{theorem}\label{onever}
There is a diagram of right Quillen equivalences
$$ \xymatrix{ \sSet & \sCat \ar[l]_{ \Nerve } \ar[d] \\
\Fun( \cDelta^{op} , \sSet) \ar[u]_{G_0} \ar[r] & \PreSeg{\sSet}, }$$
with the following features:

\begin{itemize}
\item[$(A1)$] In the upper left hand corner, we have the category $\sSet$ of simplicial sets,
endowed with the Joyal model structure (see \S \toposref{compp3}).
The fibrant objects of $\sSet$ are precisely the $\infty$-categories.

\item[$(A2)$] In the upper right corner, we have the category $\sCat$ of simplicial categories,
with the model structure constructed by Bergner in \cite{bergner}; see also
\S \toposref{compp4}.

\item[$(A3)$] In the lower right corner, we have the category $\PreSeg{\sSet}$ of
{\it preSegal categories}: that is, bisimplicial sets $X_{\bigdot \bigdot}$ with
the property that the $0$th column $X_{\bigdot 0}$ is a constant simplicial set.
These we endow with the projective model structure (see Theorem \ref{castle2}).

\item[$(A4)$] In the lower left corner, we have the category $\Fun( \cDelta^{op}, \sSet)$ of 
all bisimplicial sets, which we endow with the {\it complete Segal model structure}
introduced by Rezk (see \cite{completesegal} or Proposition \ref{camper}).

\item[$(B1)$] The upper horizontal arrow is given by the homotopy coherent nerve
functor $\Nerve: \sCat \rightarrow \sSet$ of Cordier and Porter. This 
is a right Quillen equivalence by virtue of Theorem \toposref{biggier}
(an alternative proof is given in \cite{joyalsimp}). We denote the left adjoint of this
functor by $\sCoNerve: \sSet \rightarrow \sCat$.

\item[$(B2)$] The left vertical arrow is given by the forgetful functor 
$G_0: \Fun( \cDelta^{op}, \sSet) \rightarrow \sSet$ which carries
a bisimplicial set $X_{\bigdot \bigdot}$ to the $0$th row
$X_{0 \bigdot}$. Joyal and Tierney have shown that this functor is a right Quillen equivalence
(see \cite{joyalt}); we will reprove this result as Corollary \ref{JT}.

\item[$(B3)$] The right vertical arrow associates to every simplicial
category $\calC_{\bigdot}$ the bisimplicial set $\Nerve(\calC_{\bigdot})_{\bigdot}$
obtained by applying the nerve construction degreewise. It determines
a fully faithful embedding from the category $\sCat$ of simplicial categories
to the category $\PreSeg{\sSet}$ of preSegal categories, and is a right Quillen equivalence
(a proof of this result is given in \cite{bergner4}; we will prove a generalization of this result as Theorem \ref{castle2}).

\item[$(B4)$] The lower horizontal arrow denotes a right adjoint to the inclusion
$\PreSeg{\sSet} \subseteq \Fun( \cDelta^{op}, \sSet)$. Bergner has shown
that this functor is a right Quillen equivalence (see \cite{bergner4};
we will prove a generalization of this result as Proposition \ref{curt}).
\end{itemize}
\end{theorem}

\begin{remark}\label{spade}
There is a sense in which the diagram of Theorem \ref{onever} is commutative up to homotopy. Let $F_0: \sSet \rightarrow \Fun( \cDelta^{op}, \sSet)$ be a left
adjoint to $G_0$, and let $F_1: \sSet \rightarrow \Fun( \cDelta^{op}, \sSet)$ denote the composition
$$ \sSet \stackrel{ \sCoNerve}{\rightarrow} \sCat \subseteq \PreSeg{\sSet}
\subseteq \Fun( \cDelta^{op}, \sSet).$$
There exists another functor $F: \sSet \rightarrow \Fun( \cDelta^{op}, \sSet)$
and a pair of natural transformations $F_0 \leftarrow F \rightarrow F_1$ with the following
property: for every simplicial set $X$, the induced maps $F_0(X) \leftarrow F(X) \rightarrow F_1(X)$
are weak equivalences (with respect to the complete Segal model structure on
$\Fun( \cDelta^{op}, \sSet)$). 

To describe the functor $F: \sSet \rightarrow \Fun( \cDelta^{op}, \sSet)$, 
it is convenient to introduce yet another object:
the category $(\sSet)_{/ \Nerve( \cDelta^{op})}$ of simplicial sets $X$ equipped with
a map $X \rightarrow \Nerve( \cDelta^{op})$, where $\cDelta$ denotes the category
of simplices. This category is related to the category $\Fun( \cDelta^{op}, \sSet)$
by a pair of adjoint functors
$$ \Adjoint{ \lNerve_{\bigdot}( \cDelta^{op} ) }{ (\sSet)_{/ \Nerve(\cDelta)^{op} }}{\Fun( \cDelta^{op}, \sSet)}{\Nerve_{\bigdot}( \cDelta^{op} )}$$
where $\Nerve_{\bigdot}(\cDelta^{op})$ denotes the {\it relative nerve} functor
introduced in \S \toposref{altstr}. It follows from Proposition \toposref{kudd} that these adjoint
functors determine a Quillen equivalence between $(\sSet)_{/ \Nerve(\cDelta)^{op} }$
(endowed with the covariant model structure) and $\Fun( \cDelta)^{op}, \sSet)$ 
(endowed with the injective model structure). Passing to localizations, we
deduce the existence of a Quillen equivalence between the category
$\Fun( \cDelta^{op}, \sSet)$ (endowed with the complete Segal model structure)
and $(\sSet)_{/ \Nerve(\cDelta)^{op}}$ (endowed with a suitable localization
of the covariant model structure).

Let $X$ be a simplicial set, viewed as a covariant functor from
$\cDelta^{op}$ into the category of sets. By means of a Grothendieck construction,
we can think of $X$ in a different way: as a category {\it cofibered in sets} over
$\cDelta^{op}$. More precisely, let $\cDelta_{X}$ denote the {\it category of simplices}
of $X$: the objects of $\cDelta_{X}$ are given by maps of simplicial sets $\Delta^n \rightarrow X$
where $n \geq 0$, and morphisms by commutative diagrams
$$ \xymatrix{ \Delta^m \ar[dr] \ar[rr] & & \ar[dl] \Delta^n \\
& X. & }$$
We can then view the nerve $\Nerve( \cDelta_{X})^{op}$ as an object of
$(\sSet)_{/ \Nerve(\cDelta)^{op}}$. This construction determines a fully faithful embedding
of $\sSet$ into $(\sSet)_{/ \Nerve(\cDelta)^{op}}$, which we will denote by $\sdd$.
The functor $F$ is defined to be the composition
$$\sSet \stackrel{\sdd}{\rightarrow} 
(\sSet)_{/ \Nerve(\cDelta)^{op} } \stackrel{ \lNerve_{\bigdot}(\cDelta^{op})}{\rightarrow}
\Fun( \cDelta^{op}, \sSet).$$
(For a definition of the natural transformations $F_0 \leftarrow F \rightarrow F_1$,
and the verification that they have the asserted properties, we refer the reader to
\S \ref{subber}). 
\end{remark}

We can describe the situation informally as follows: there are a number of models
for the theory of $(\infty,1)$-categories which are known to be equivalent to one another,
via more-or-less explicit combinatorial constructions. The main goal of this paper is to establish
an analogous picture for the theory of $(\infty,2)$-categories; moreover, all of the essential
players are slightly more elaborate versions of their $(\infty,1)$-categorical counterparts.
There is one exception: we do not have an obvious analogue of the adjoint functors
$(F_0,G_0)$ in the $(\infty,2)$-categorical setting. However, we do have an analogue of the functor
$F$, which serves as an adequate replacement. Our main results can be summarized as follows:

\begin{theorem}\label{toothygrin}
There is a diagram of model categories and right Quillen equivalences
$$ \xymatrix{ \scSet & \Cat_{\mSet} \ar[l] \ar[d] \\
\mset{\Nerve(\cDelta)^{op} } \ar[u] & \PreSeg{\mSet} \\
\Fun( \cDelta^{op}, \mSet) \ar[r] \ar[u] & \PreSeg{ \mSet} \ar[u] , }$$
where:
\begin{itemize}
\item[$(A1)$] In the upper left hand corner, we have the category $\scSet$
of {\it scaled simplicial sets}: that is, pairs $(X,T)$ where $X$ is a simplicial set
and $T$ is a collection of $2$-simplices in $X$, which includes all degenerate $2$-simplices.

\item[$(A2)$] In the upper right hand corner, we have the category $\Cat_{\mSet}$
of {\it $\mSet$-enriched categories}, or {\it marked simplicial categories}. Here $\mSet$ denotes the category of {\it marked simplicial sets}: that is, pairs $(X,M)$ where $X$ is a simplicial set and
$M$ a collection of $1$-simplices of $X$, which includes all degenerate $1$-simplices.
(We can think of an object of $\Cat_{\mSet}$ as a simplicial category $\calC$, such that
for every pair of objects $x,y \in \calC$, the simplicial set $\bHom_{\calC}(x,y)$ comes
equipped with a distinguished class of edges, which are required to be stable under composition.)

\item[$(A3)$] In the lower right corner of the diagram, we have the category
$\PreSeg( \mSet)$ of $\mSet$-enriched preSegal categories, endowed with
the projective model structure of Theorem \ref{castle2}. This can
be viewed as the full subcategory of $\Fun( \cDelta^{op}, \mSet)$ spanned
by those marked bisimplicial sets $X_{\bigdot \bigdot}$ such that
each $X_{\bigdot 0}$ is a constant simplicial set.

\item[$(A4)$] In the lower left corner of the diagram, we have the category
$\Fun( \cDelta^{op}, \mSet)$ of simplicial objects of $\mSet$. We regard this category
as endowed with a localization of the injective model structure which we will refer to
as the {\it complete Segal model structure} (see Proposition \ref{camper}).

\item[$(A5)$] In the middle left side of the diagram, we have the category
$\mset{ \Nerve(\cDelta)^{op} }$ whose objects are simplicial sets $X$ equipped with
both a marking and a map $X \rightarrow \Nerve(\cDelta)^{op}$. We endow
this category with a suitable localization of the coCartesian model structure, which is described in  Proposition \ref{camperr}.

\item[$(A6)$] In the middle right side of the diagram, we have the category
$\PreSeg{\mSet}$ of $\mSet$-enriched preSegal categories, endowed with
the {\it injective} model structure of Proposition \ref{curt}.

\item[$(B1)$] The upper horizontal functor $\scNerve: \Cat_{\mSet} \rightarrow \scSet$
is a decorated version of the homotopy coherent nerve $\Nerve$: it carries a marked
simplicial category $\calC$ to the pair $( \Nerve(\calC), T)$, where $T$ is a collection
of $2$-simplices in $\calC$ which depends on the collection of marked edges in
the mapping spaces $\bHom_{\calC}(x,y)$ (for the complete definition, we
refer the reader to Definition \ref{ilmut}). The functor $\scNerve$ admits a left adjoint, which
we will denote by $\scCoNerve$.

\item[$(B2)$] The upper left vertical
arrow $\mset{ \Nerve(\cDelta)^{op} } \rightarrow \scSet$ is defined as the right
adjoint to a functor $\sdd^{+}: \scSet \rightarrow \mset{ \Nerve(\cDelta)^{op} }$,
which is a decorated version of the functor $\sdd: \sSet \rightarrow (\sSet)_{/ \Nerve(\cDelta)^{op}}$
described in Remark \ref{spade} (in other words, we have $\sdd^{+}(X,T)= (\sdd(X), M)$, where
$M$ is a collection of edges in $\sdd(X) = \Nerve( \cDelta_{X})^{op}$ which depends on the collection $T$ of $2$-simplices in $X$). 

\item[$(B3)$] The upper right vertical arrow is a fully faithful embedding
$\Cat_{\mSet} \hookrightarrow \PreSeg{\mSet}$, which is a decorated
version of the embedding $\Cat_{\sSet} \rightarrow \PreSeg{\sSet}$:
to a marked simplicial category $\calC_{\bigdot}$, it associates the bisimplicial
set $\Nerve( \calC_{\bigdot})_{\bigdot}$ obtained by applying the nerve construction degreewise
(endowed with a suitable marking).

\item[$(B4)$] The bottom horizontal arrow is a right adjoint to the inclusion
functor $\PreSeg{\mSet} \hookrightarrow \Fun( \cDelta^{op}, \mSet)$.

\item[$(B5)$] The lower left vertical arrow is given by the marked relative nerve
functor $\Nerve^{+}_{\bigdot}( \cDelta^{op} )$ described in \S \toposref{altstr}.
We denote the left adjoint to this functor by $\lNerve^{+}_{\bigdot}( \cDelta^{op} )$. 

\item[$(B6)$] The lower right vertical arrow is the identity functor, which is
a right Quillen equivalence by virtue of Proposition \ref{curt2}.
\end{itemize}

Moreover, this diagram is commutative in the following sense:
there exists a natural transformation of functors
$\alpha: F \rightarrow F'$ where $F$ denotes the composition
$$ \scSet \stackrel{\sdd^{+}}{\rightarrow} \mset{ \Nerve(\cDelta)^{op} }
\stackrel{ \lNerve^{+}_{\bigdot}(\cDelta^{op})}{\rightarrow}
\Fun( \cDelta^{op}, \mSet)$$
and $F'$ the composition
$$ \scSet \stackrel{\scCoNerve}{\rightarrow} \Cat_{\mSet}
\hookrightarrow \PreSeg{\mSet} \hookrightarrow \Fun( \cDelta^{op}, \mSet).$$
Furthermore, for every scaled simplicial set $\overline{X}$, the map
$\alpha( \overline{X}): F(\overline{X}) \rightarrow F'( \overline{X} )$ is a weak
equivalence (with respect to the complete Segal model structure on
$\Fun( \cDelta^{op}, \mSet)$).
\end{theorem}

\begin{remark}
In Theorem \ref{toothygrin}, the category $\mSet$ of marked simplicial sets
plays the role of a good model for the theory of $(\infty,1)$-categories. Some
of the assertions of Theorem \ref{toothygrin} continue to hold if we replace
$\mSet$ by other models. For example, the forgetful functor
$\mSet \rightarrow \sSet$ determines a right Quillen equivalence if we
endow $\sSet$ with the Joyal model structure (Theorem \toposref{bigdiag}).
This gives rise to a commutative diagram of right Quillen equivalences
$$ \xymatrix{ \Cat_{\mSet} \ar[r] \ar[d] & \Cat_{\sSet} \ar[d] \\
\PreSeg{\mSet} \ar[r] & \PreSeg{\sSet} \\
\Fun( \cDelta^{op}, \mSet) \ar[u] \ar[r] & \Fun( \cDelta^{op}, \sSet) \ar[u], }$$
which gives us another three models for the theory of $(\infty,2)$-categories.
(Here the left column consists of Quillen equivalences appearing
Theorem \ref{toothygrin}, and the right column is defined analogously.)
\end{remark}

The bulk of this paper will be devoted to constructing the diagram described
in Theorem \ref{toothygrin} and verifying that it has the desired properties.
We begin in \S \ref{bisecA} by reviewing Rezk's theory of complete Segal spaces.
We will present a variation on his definitions, which will allow us to define the model categories described in $(A3)$ and $(A5)$, and to establish the Quillen equivalence described in $(B5)$. 
In \S \ref{bisecB}, we will review the formalism of Segal categories, which we will use to 
define the model categories described in $(A4)$ and $(A6)$ and to establish the Quillen equivalences
of $(B3)$, $(B4)$, and $(B6)$.

The notions of Segal categories and complete Segal spaces have the virtue of generalizing
to higher dimensions: using the work of Simpson-Tamsamani or Barwick, one can give a definition of $(\infty,n)$-category for any nonnegative integer $n$, using induction on $n$. However, in
either case, the resulting theory describes an $(\infty,n)$-category as an $(n+1)$-uple simplicial set.
Consequently, these definitions increase in complexity with $n$, and become somewhat cumbersome
to work with directly. Our goal in \S \ref{bisecD} is to provide an alternative definition in the case
$n=2$ which does not share this defect. We will achieve this goal by introducing a theory of {\it scaled simplicial sets}. Our definition was inspired by Verity's work on {\it stratified simplicial sets} (see \cite{verity} and \cite{verity2}), though our goals are considerably less ambitious. We will use it to define the model category
described in $(A1)$, and to construct the Quillen equivalences of $(B1)$ and $(B2)$. In order
to carry out the details, we will need an analogue of straightening and
unstraightening constructions of \S \toposref{strsec} to the setting of locally coCartesian fibrations,
which we provide in \S \ref{bisecC}. These constructions will be presecan be regarded as providing a higher-categorical version of the Grothendieck construction for lax functors, and should be useful in a variety of other contexts.

Our original goal in developing the theory described in this paper is to have an adequate higher-categorical language for describing Goodwillie's calculus of functors. In \S \ref{bisecE}, we will review the rudiments of Goodwillie's theory (namely, the theory of first derivatives) from a higher-categorical point of view. The material of \S \ref{bisecE} is almost entirely independent of the remainder of the paper, except for the final section (\S \ref{finalx}) where we explain how to interpret
the theory of Goodwillie derivatives as giving rise to a functor between $(\infty,2)$-categories.

\begin{remark}
The material presented here is really only the first step in a much larger project,
whose aim is to understand the Goodwillie calculus in terms of higher category theory.
We plan to return to this subject in \cite{calculus}.
\end{remark}

\section{Complete Segal Spaces}\label{bisecA}

In this section, we will review Rezk's theory of complete Segal spaces, and the higher dimensional generalization thereof (due to Barwick). Let us begin by sketching the basic idea. Suppose
that $\calC$ is an $(\infty,n)$-category. We would like to describe $\calC$ in terms of
invariants of a less sophisticated nature: for example, $(\infty,k)$-categories for $k < n$.
We can begin by extracting an $(\infty,0)$-category $\calC_0$ from $\calC$, by 
discarding all of the noninvertible morphisms in $\calC$ at all levels.
The passage from $\calC$ to $\calC_0$ involves a loss of information:
$\calC_0$ knows everything about the objects of $\calC$, but nothing about
noninvertible morphisms between them. To retain this information, we first note that
for every pair of objects $X,Y \in \calC$, we expect to have an $(\infty,n-1)$-category
of morphisms $\bHom_{\calC}(X,Y)$. This $(\infty,n-1)$-category depends functorially
on the pair $X,Y \in \calC_0$. Consequently, we can organize the collection
of all of the $(\infty,n-1)$-categories $\{ \bHom_{\calC}(X,Y) \}_{X,Y \in \calC}$ into
a single $(\infty,n-1)$-category $\calC_1$, whose objects are given by triples
$(X \in \calC_0, Y \in \calC_0, f \in \bHom_{\calC}(X,Y) )$.
More generally, for each $k \geq 0$, we can consider an
$(\infty,n-1)$-category $\calC_k$ consisting of $(2k+1)$-tuples
$$ (X_0 \in \calC_0, X_1 \in \calC_0, \ldots, X_k \in \calC_0,
f_1 \in \bHom_{\calC}(X_0, X_1), \ldots, f_k \in \bHom_{\calC}( X_{k-1}, X_k)) \}$$
in other words, composable sequences of morphisms
$$ X_0 \stackrel{f_1}{\rightarrow} X_1 \stackrel{f_2}{\rightarrow} X_2 \stackrel{f_2}{\rightarrow} \ldots \stackrel{f_k}{\rightarrow} X_k.$$
The collection of $(\infty,n-1)$-categories $\{ \calC_{k} \}_{k \geq 0}$ forms a simplicial
$(\infty,n-1)$-category $\calC_{\bigdot}$ satisfying the following {\it Segal condition}:
\begin{itemize}
\item[$(A1)$] For each $k \geq 0$, the canonical map $$\calC_{k} \rightarrow
\calC_1 \times_{\calC_0} \calC_1 \times_{\calC_0} \ldots
\times_{\calC_0} \calC_1$$ is an equivalence of $(\infty,n-1)$-categories.
\end{itemize}
This simply encodes the idea that an object of $\calC_{k}$ consists of a sequence
of $k$ morphisms $\{ f_i \}_{1 \leq i \leq k}$ in $\calC$, constrained only by the requirement that the domain of each $f_{i+1}$ is the codomain of $f_i$.

Conversely, if we are given a simplicial $(\infty,n-1)$-category $\calC_{\bigdot}$ satisfying
the Segal condition $(A1)$, then we should be able to extract an $(\infty,n)$-category $\calC$
as follows:
\begin{itemize}
\item The objects of $\calC$ are the objects of $\calC_0$.
\item Given a pair of objects $X,Y \in \calC_0$, the $(\infty,n-1)$-category of maps
$\bHom_{\calC}(X,Y)$ is given by the fiber product
$\{X\} \times_{ \calC_0} \calC_1 \times_{\calC_0} \{Y\}$.
\item Given a sequence of objects $X_0, \ldots, X_k \in \calC_0$, the composition law
$$ \bHom_{\calC}(X_0, X_1) \times \ldots \bHom_{\calC}(X_{k-1}, X_k) \rightarrow \bHom_{\calC}(X_0, X_k)$$
is given by the composite map
\begin{eqnarray*} 
(\{X_0 \} \times_{\calC_0} \calC_1 \times_{\calC_0} \{X_1\})
\times \ldots \times ( \{X_{k-1}\} \times_{\calC_0} \calC_1 \times_{\calC_0} \{X_k\})
& \stackrel{\sim}{\leftarrow} & \calC_k \times_{ \calC_0 \times \ldots \times \calC_0 }
( \{X_0\} \times \ldots \times \{X_k\}) \\
& \rightarrow & \{X_0\} \times_{\calC_0} \calC_1 \times_{\calC_0} \{X_k\}. 
\end{eqnarray*}
Here the invertibility of the first map follows from assumption $(A1)$.
\end{itemize}
This construction determines a left inverse (up to equivalence) to the earlier
process which extracts a simplicial $(\infty,n-1)$-category from
an $(\infty,n)$-category. However, it is not generally a right inverse, because
generally the underlying $(\infty,0)$-category of $\calC$ does not agree with
the $\calC_0$. To rule out this phenomenon, we need to make two additional assumptions
on $\calC_{\bigdot}$: 
\begin{itemize}
\item[$(A2)$] The $(\infty,n-1)$-category $\calC_0$ is an $(\infty,0)$-category.
\item[$(A3)$] The simplicial $(\infty,n-1)$-category $\calC_{\bigdot}$ is {\it complete}
(see Definition \ref{compseg}). 
\end{itemize}

Our objective is to make the above ideas precise, working in a general
$\infty$-categorical context. We begin in \S \ref{bisec1.1} by introducing the
notion of a {\it category object} of an $\infty$-category $\calY$: that is, a
simplicial object of $\calY$ satisfying axiom $(A1)$ (the case
of interest is that in which $\calY$ is some version of the theory
of $(\infty,n-1)$-categories). In order to formulate axiom $(A2)$, we need
need to assume that $\calY$ is equipped with a suitable subcategory
$\calX \subseteq \calY$ of ``$\infty$-groupoids''. In \S \ref{bisec1.2}, we will
formulate analogues of $(A2)$ and $(A3)$ under the assumption
that the inclusion $\calX \subseteq \calY$ is a {\it distributor} (see Definition \ref{disty}).
By imposing $(A1)$, $(A2)$, and $(A3)$, we will obtain a full subcategory
$\CSS{\calX \subseteq \calY} \subseteq \Fun( \Nerve(\cDelta)^{op}, \calY)$ which
we will refer to as the {\it $\infty$-category of complete Segal space objects of $\calY$}.
In \S \ref{bisec1.3}, we will show that the diagonal embedding
$\calX \rightarrow \CSS{\calX \subseteq \calY}$ gives rise to another distributor,
so the construction $\calY \mapsto \CSS{\calX \subseteq \calY}$ can be iterated.
To obtain the theory of $(\infty,n)$-categories, we simply apply this construction
$n$ times, with initial data $\calX = \calY = \SSet$. In \S \ref{bisec1.4} we will
present an alternative construction: we can begin instead with the
distributor $\SSet \subseteq \Cat_{\infty}$ and apply the above construction $(n-1)$ times.
(The fact that these two constructions give the same result is not obvious: it depends on the 
equivalence between our theory of $\infty$-categories and the theory of complete
Segal spaces. This is a result of Joyal and Tierney which we will later reprove as
Corollary \ref{JT}.)

We will conclude this section with \S \ref{bisec1.5}, where we reformulate the theory
of complete Segal space objects in the language of model categories and use it to explain part
of Theorem \ref{toothygrin}.

\subsection{Category Objects and Groupoid Objects}\label{bisec1.1}

Let $\calE$ be an ordinary category. Then $\calE$ is determined (up to canonical isomorphism)
by the simplicial set $\Nerve(\calE)$. In other words, we can regard the ordinary category
$\Cat$ of small categories as a full subcategory of the category $\sSet$ of simplicial sets.
Moreover, we can give a simple explicit characterization of this subcategory:
a simplicial set $X_{\bigdot}$ is isomorphic to the nerve of a category if and only if,
for every $n \geq 0$, the canonical map
$X_{n} \rightarrow X_1 \times_{X_0} X_1 \times_{X_0} \ldots \times_{X_0} X_1$
is a bijection. Motivated by this observation, we introduce the following definition:

\begin{definition}\label{unwise}
Let $\calC$ be an $\infty$-category. A {\it category object} of $\calC$ is a simplicial object
$X \in \Fun( \Nerve(\cDelta)^{op}, \calC)$ with the following property:
for every integer $n \geq 0$, the functor $X$ exhibits $X([n])$ as a limit
of the diagram
$$ \xymatrix{ & X( \{0,1\}) \ar[dr] \ar[dl] & & \ldots \ar[dl] \ar[dr]  &   & X(\{n-1, n\}) \ar[dl] \ar[dr] & \\
X( \{0\})  & & \ldots & & \ldots & & X(\{n\}). }$$
$($Here abuse notation identifing a nonempty finite linearly ordered set $I$ with the corresponding
object of $\cDelta^{op}$.$)$ Let $\Cat(\calC)$ denote the full subcategory of $\Fun( \Nerve(\cDelta)^{op}, \calC)$ spanned by the category objects.
\end{definition}

\begin{example}
Let $\calC$ be (the nerve of) the category of sets. Then we can identify category objects
of $\calC$ with ordinary categories.
\end{example}

In other words, a simplicial object $X_{\bigdot}$ is a category object if, for each $n \geq 0$, the canonical map
$$ X_{n} \rightarrow X_1 \times_{X_0} X_1 \times_{X_0} \times \ldots \times_{X_0} X_1$$
is an equivalence in $\calC$; here the right hand side is well-defined so long as $\calC$ admits pullbacks.

\begin{example}\label{cuckoe}
Let $\calC$ be an $\infty$-category, and let $\Grp(\calC)$ denote the full subcategory of
$\Fun( \Nerve(\cDelta)^{op}, \calC)$ spanned by the groupoid objects of $\calC$
(see Definition \toposref{grpddef}). Then $\Grp(\calC) \subseteq \Cat(\calC)$. 
\end{example}

\begin{example}\label{defcon1}
Let $\calC$ be an $\infty$-category. For every object $C \in \calC$, the constant functor
$\Nerve(\cDelta)^{op} \rightarrow \{C\} \subseteq \calC$ is a groupoid object of $\calC$.
This construction determines a fully faithful embedding $\delta: \calC \rightarrow \Grp(\calC)$.
We will say that a groupoid object of $\calC$ is {\it constant} if it lies in the essential image
of $\delta$.

If $\calC$ admits small colimits, then the functor $\delta$ admits a left adjoint, given by the geometric realization construction
$$ \Grp(\calC) \subseteq \Fun( \Nerve(\cDelta)^{op}, \calC) \stackrel{ \colim}{\rightarrow} \calC.$$
It follows that if $\calC$ is presentable, then the full subcategory of constant groupoid objects of
$\calC$ is an accessible localization $\Grp(\calC)$ (see \S \toposref{invloc}).
\end{example}

\begin{remark}\label{supling}
Since the simplicial set $\Nerve(\cDelta)^{op}$ is weakly contractible, a simplicial object
$X_{\bigdot}$ of an $\infty$-category $\calC$ is constant if and only if, for every morphism
$[m] \rightarrow [n]$ in $\cDelta$, the induced map $X_n \rightarrow X_m$ is an equivalence.
\end{remark}

For our purposes, the most important special case of Definition \ref{unwise} is that
in which $\calC$ is the $\infty$-category $\SSet$ of spaces. A category object of
$\SSet$ is usually called a {\it Segal space}. We will later see that every
Segal space $X_{\bigdot}$ determines an $\infty$-category. In particular, we can
extract a {\it homotopy category} from $X_{\bigdot}$, which is enriched over the
homotopy category $\calH$ of spaces. Our next goal is to explain how to extract this
homotopy category directly from $X_{\bigdot}$:

\begin{definition}\label{cabler}
Let $X_{\bigdot}$ be a category object of $\SSet$, and let $\calH = \h{\SSet}$ denote the homotopy category of spaces. We define a $\calH$-enriched category, the {\it homotopy category $\h{X_{\bigdot}}$}, as follows:
\begin{itemize}
\item[$(1)$] The objects of $\h{ X_{\bigdot}}$ are the points of $X_{0}$.
\item[$(2)$] Given a pair of points $x,y \in X_0$, we define
$\bHom_{ \h{X_{\bigdot}}}( x,y)$ to be the homotopy fiber product
$$ \{x\} \times_{ X_0 } X_1 \times_{ X_0 } \{y\} \in \calH.$$
\item[$(3)$] Given a sequence of points $x_0, \ldots, x_n \in X_0$, the associated composition law is
given by the composition
\begin{eqnarray*}
\prod_{1 \leq i \leq n} \bHom_{\h{X_{\bigdot}}}(x_{i-1}, x_i) & \stackrel{\sim}{\rightarrow} & \{x_0\} \times_{ X_0} X_1 \times_{ X_0 } \{x_1\} \times_{ X_0} X_1 \times_{X_0} \ldots \times_{ X_{0}} \{ x_{n-1} \} \times_{X_0} X_1 \times_{X_0} \{x_n\} \\
& \rightarrow & \{x_0 \} \times_{X_0} X_1 \times_{X_0} X_1 \times_{X_0} \ldots \times_{X_0} X_1 \times_{X_0} \{x_{n} \} \\
& \stackrel{\sim}{\leftarrow} & \{x_0\} \times_{X_0} X_n \times_{ X_0} \{x_n\} \\
& \simeq & \bHom_{ \h{X_{\bigdot}}}( x_0, x_n).
\end{eqnarray*}
\end{itemize}
Note that every point $f \in X_{1}$ determines a morphism in the homotopy category
$\h{X_{\bigdot}}$, which we will denote by $[f]$.
\end{definition}

According to Example \ref{cuckoe}, every groupoid object of $\SSet$ is a Segal space.
Our next goal is to establish a partial converse: a Segal space $X_{\bigdot}$ is a groupoid object
of $\SSet$ if and only its homotopy category $\h{X_{\bigdot}}$ is a groupoid (Proposition \ref{cuckop}).
First, we need to introduce a bit of terminology. 

\begin{notation}
Let $X_{\bigdot}$ be a simplicial object of an $\infty$-category $\calC$. For every simplicial set
$K$, let $\cDelta_{/K}$ denote the category of simplices of $K$ (see \S \toposref{quasilimit1}). We
let $X(K)$ denote a limit of the composite diagram
$$ \Nerve( \cDelta_{/K})^{op} \rightarrow \Nerve( \cDelta)^{op} \stackrel{X}{\rightarrow} \calC,$$
if such a limit exists. This limit always exists, for example, if $K$ is finite and $\calC$ admits finite limits.
\end{notation}

\begin{proposition}\label{cuckop}
Let $X_{\bigdot}$ be a category object of $\SSet$. The following conditions are equivalent:
\begin{itemize}
\item[$(1)$] The category object $X_{\bigdot}$ is a groupoid object of $\SSet$.
\item[$(2)$] For every point $f \in X_1$, the morphism $[f]$ is invertible in the homotopy category $\h{X_{\bigdot}}$.
\end{itemize}
\end{proposition}

\begin{proof}
We have a pullback diagram
$$ \xymatrix{ X( \Delta^2) \ar[r]^{p'} \ar[d]^{q'} & X( \Lambda^2_0 ) \ar[d]^{q} \\
X( \Lambda^2_1 ) \ar[r]^{p} & X( \Delta^{ \{0,1\} } \coprod \{2\} ). }$$
The proof of Proposition \toposref{grpobjdef} shows that $X_{\bigdot}$ is a groupoid object if and only if the
map $p'$ is a homotopy equivalence. Since $X_{\bigdot}$ is a category object, the map
$q'$ is an equivalence. The composition $p' \circ {q'}^{-1}$ determines a map from
$r: X( \Lambda^2_1 ) \rightarrow X( \Lambda^2_0)$ in the $\infty$-category
$\SSet_{/ X( \Delta^{ \{0,1\} }\coprod \{2\} )}$, and $(1)$ is equivalent to the assertion that
$r$ is a homotopy equivalence. This can be reformulated as follows:
\begin{itemize}
\item[$(1')$] For every point $\eta$ of $X( \Delta^{ \{0,1\} } \coprod \{2\} )$, the induced map
$$r_{\eta}: X( \Lambda^2_1) \times_{  X( \Delta^{ \{0,1\} }\coprod \{2\} ) } \{ \eta \}
\rightarrow X( \Lambda^2_0) \times_{  X( \Delta^{ \{0,1\} }\coprod \{2\} ) } \{ \eta \}$$
is an equivalence in $\SSet$.
\end{itemize}
In the situation of $(1')$, we can identify $\eta$ with a pair $(f, z)$, where
$f \in X_1$ and $z \in X_0$. Let $x$ and $y$ denote the images of $f$ in
$X_0$. Unwinding the definitions, we see that $r_{\eta}$ can be identified with
the map
$$ \bHom_{ \h{X_{\bigdot}}}(y,z) \rightarrow \bHom_{ \h{X_{\bigdot}}}(x,z)$$
given by composition with $[f]$. This map is a homotopy equivalence for
every point $z \in X_0$ if and only if $[f]$ is invertible in $\h{X_{\bigdot}}$. The requirement that this condition holds for {\em every} $f \in X_1$ is equivalent to $(2)$.
\end{proof}

\begin{notation}
Let $X_{\bigdot}$ be a category object in $\SSet$. Every point $f \in X_{n}$ determines a composable sequence of morphisms $$ x_0 \stackrel{[f_1]}{\rightarrow} x_1 \stackrel{[f_2]}{\rightarrow} x_2 \stackrel{[f_3]}{\rightarrow} \ldots \stackrel{[f_n]}{\rightarrow} x_n$$
in the homotopy category $\h{ X_{\bigdot} }$ which is well-defined up to isomorphism, and
(up to isomorphism) depends only on the connected component of $f$ in $X_n$.
We will say that $f$ is {\it invertible} if each $[f_i]$ is an invertible morphism of $\h{ X_{\bigdot} }$.
We let $X_{n}^{\inv}$ denote full simplicial subset of $X_{n}$ spanned by the invertible points,
so that $X_{n}^{\inv}$ is a union of connected components of $X_{n}$. The simplicial subsets
$\{ X_{n}^{\inv} \subseteq X_{n} \}_{n \geq 0}$ assemble to form a new simplicial object of
$\SSet$, which we will denote by $X_{\bigdot}^{\inv}$.
\end{notation}

It is straightforward to establish the following universal property of $X_{\bigdot}^{\inv}$:

\begin{proposition}\label{hutch}
Let $X_{\bigdot}$ be a category object of $\SSet$. Then:
\begin{itemize}
\item[$(1)$] The simplicial object $X_{\bigdot}^{\inv}$ is a groupoid object of $\SSet$.
\item[$(2)$] Let $Y_{\bigdot}$ be a groupoid object of $\SSet$. Then composition with
the canonical map $X_{\bigdot}^{\inv} \rightarrow X_{\bigdot}$ induces a homotopy
equivalence
$$ \bHom_{ \Fun( \Nerve(\cDelta)^{op}, \SSet)}( Y_{\bigdot}, X_{\bigdot}^{\inv})
\rightarrow \bHom_{ \Fun( \Nerve(\cDelta)^{op}, \SSet)}( Y_{\bigdot}, X_{\bigdot} ).$$
\end{itemize}
\end{proposition}

\begin{corollary}\label{slam}
The inclusion $\Grp(\SSet) \subseteq \Cat(\SSet)$ admits a right adjoint, given by
$X_{\bigdot} \mapsto X_{\bigdot}^{\sim}$.
\end{corollary}

We can informally summarize Proposition \ref{hutch} as follows: $X_{\bigdot}^{\inv}$ is the
largest groupoid object contained in the category object $X_{\bigdot}$. Our goal for the remainder of this section is to obtain a similar construction when $X_{\bigdot}$ is a category object
of an arbitrary $\infty$-category $\calC$ which admits finite limits. Our first step
is characterize $X_{\bigdot}^{\inv}$ in a different way.

\begin{notation}\label{amplecrown}
Let $K$ denote the simplicial set $$\Delta^0 \coprod_{ \Delta^{ \{0,2\} }} \Delta^3 \coprod_{ \Delta^{ \{1,3\} }} \Delta^0$$ obtained from $\Delta^3$ by collapsing the edges $\Delta^{ \{0,2\} }$ and $\Delta^{ \{1,3\} }$.
We let $K^{0} \subseteq K$ denote the image of the edge $\Delta^{ \{1,2\} } \subseteq \Delta^3$ in $K$.
\end{notation}

\begin{proposition}\label{canman}
Let $K^0 \subseteq K$ be as in Notation \ref{amplecrown}.
\begin{itemize}
\item[$(1)$] Let $X_{\bigdot}$ be a category object of $\SSet$. Then the canonical map
$\phi: X^{\inv}(K) \rightarrow X(K)$ is a homotopy equivalence.
\item[$(2)$] Let $\calC$ be an $\infty$-category which admits finite limits, and let
$Y_{\bigdot}$ be a groupoid object of $\calC$. Then the canonical map
$Y(K) \rightarrow Y(K^0)$ is an equivalence in $\calC$.
\end{itemize}
\end{proposition}

\begin{proof}
It follows immediately from the definitions that the map $\phi$ is a homotopy equivalence onto its essential image, which consists of all points of $X(K)$ such that the induced diagram
$$ \xymatrix{ & y \ar[r] \ar[drr]_{id} & x \ar[dr] & \\
x \ar[ur] \ar[urr]_{\id} \ar[rrr] & & & y }$$
in the homotopy category $\h{X_{\bigdot}}$ consists entirely of isomorphisms. The essential surjectivity
now follows by a simple diagram chase. This proves $(1)$. Assertion $(2)$ follows from Proposition \toposref{grpobjdef}, since the inclusion $K^0 \subseteq K$ is a homotopy equivalence which is bijective on vertices.
\end{proof}

\begin{proposition}\label{cavem}
Let $\calC$ be an $\infty$-category which admits finite limits.
\begin{itemize}
\item[$(1)$] The inclusion
$\Grp(\calC) \subseteq \Cat(\calC)$ admits a right adjoint, which we will denote by
$X_{\bigdot} \mapsto X_{\bigdot}^{\inv}$. 
\item[$(2)$] For every category object $X_{\bigdot}$ of $\calC$, the canonical maps
$$ X_1^{\inv} = X^{\inv}(K^0) \leftarrow X^{\inv}(K) \rightarrow X(K) \quad \quad  X_0^{\inv} \rightarrow X_0$$
are equivalences.
\end{itemize}
\end{proposition}

\begin{proof}
We first treat the case where $\calC = \calP(\calD)$ is the $\infty$-category of presheaves on
another $\infty$-category $\calD$. In this case, we have canonical isomorphisms
$$ \Grp(\calC) \simeq \Fun( \calD^{op}, \Grp(\SSet) ) \quad \quad \Cat(\calC) \simeq \Fun( \calD^{op}, \Cat(\SSet)),$$
so assertion $(1)$ follows from Corollary \ref{slam}. To verify $(2)$, it suffices to check that the resulting maps are equivalences after evaluation at every object $D \in \calD$. We may therefore reduce to the case where $\calC = \SSet$, where the desired result follows from Proposition \ref{canman}.

We now consider the general case. Without loss of generality we may suppose that $\calC$ is small.
Let $j: \calC \rightarrow \calP(\calC)$ denote the Yoneda embedding. Since $j$ preserves finite limits, it induces fully faithful embeddings
$$ \Grp(\calC) \subseteq \Grp( \calP(\calC) ) \quad \quad \Cat(\calC) \subseteq \Cat( \calP(\calC) ).$$
The first part of the proof shows that the inclusion $\Grp( \calP(\calC) ) \subseteq \Cat( \calP(\calC) )$ admits a right adjoint $X_{\bigdot} \mapsto X_{\bigdot}^{\inv}$. To complete the proof, it will suffice
to show that if $X_{\bigdot}$ is a category object of $\calP(\calC)$ such that each
$X_{n}$ is representable (that is, each $X_n$ lies in the essential image of the Yoneda embedding
$j: \calC \rightarrow \calP(\calC)$), then the simplicial object $X_{\bigdot}^{\inv}$ has the same property.
Since $X_{\bigdot}^{\inv}$ is a category object of $\calP(\calC)$ and the collection of representable functors is stable under finite limits (because $\calC$ admits finite limits and $j$ is left exact), 
it will suffice to show that $X_0^{\inv}$ and $X_1^{\inv}$ are representable. The first part of the proof shows that these objects are equivalent to $X_0$ and $X(K)$, respectively; the first is a representable functor by assumption, and the second is a finite limit of representable functors and therefore representable.
\end{proof}

\subsection{Segal Spaces and Complete Segal Spaces}\label{bisec1.2}

In \S \ref{bisec1.1}, we introduced the definition of a {\it category object} of an
arbitary $\infty$-category $\calY$, which gives a precise articulation of
axiom $(A1)$ appearing in the introduction to \S \ref{bisecA}. Our goal in this
section is to do the same for axioms $(A2)$ and $(A3)$. To obtain a sensible theory, we need to introduce some assumptions on $\calY$.

\begin{definition}\label{disty}
A {\it distributor} consists of an $\infty$-category $\calY$ together with a full subcategory $\calX$ satisfying the following conditions:
\begin{itemize}
\item[$(1)$] The $\infty$-categories $\calX$ and $\calY$ are presentable.

\item[$(2)$] The full subcategory $\calX \subseteq \calY$ is stable under small limits and colimits
in $\calY$.

\item[$(3)$] Let $Y \rightarrow X$ be a morphism in $\calY$ such that $X \in \calX$. Then the pullback functor $\calX_{/X} \rightarrow \calY_{/Y}$ preserves small colimits.

\item[$(4)$] 
Let $\calO$ denote the full subcategory of $\Fun( \Delta^1, \calY)$ spanned by those
morphisms $f: Y \rightarrow X$ such that $X \in \calX$, and let $\pi: \calO \rightarrow \calX$ be the functor given by evaluation at $\{1\} \subseteq \Delta^1$. Since $\calY$ admits pullbacks, the evaluation functor 
$\Fun( \Delta^1, \calY ) \rightarrow \Fun( \{1\}, \calY) \simeq \calY$ is a Cartesian fibration, so that
$\pi$ is likewise a Cartesian fibration. Let $\chi: \calX \rightarrow \widehat{\Cat}^{op}_{\infty}$ be a functor which classifies $\pi$. Then $\chi: \calX \rightarrow \widehat{\Cat}_{\infty}^{op}$ preserves small limits.
\end{itemize}
\end{definition}

\begin{remark}\label{cavem2}
Condition $(2)$ of Definition \ref{disty} is equivalent to the requirement that the inclusion
$i: \calX \subseteq \calY$ preserves small limits and colimits. In view of Corollary \toposref{adjointfunctor}, this is equivalent to the requirement that $i$ admits both left and right adjoints.
\end{remark}


\begin{example}\label{jina}
Suppose that $\calX = \calY$ in Definition \ref{disty}. Then condition $(3)$ is equivalent to the requirement that colimits in $\calX$ are universal, and condition $(4)$ is equivalent to the assertion that the collection of all morphisms in $\calX$ is local (in the sense of Definition \toposref{localitie}). It follows from Theorem \toposref{mainchar} and \toposref{charleschar} that the inclusion $\calX \subseteq \calY$ is a distributor if and only if $\calX$ is an $\infty$-topos.
\end{example}

Our first goal is to reformulate conditions $(3)$ and $(4)$ of Definition \ref{disty} in more concrete terms.

\begin{proposition}\label{maid}
Let $\calX \subseteq \calY$ be a fully faithful inclusion of $\infty$-categories satisfying conditions
$(1)$ and $(2)$ of Definition \ref{disty}, let $K$ be a small simplicial set, and let
$\overline{q}: K^{\triangleright} \rightarrow \calX$ be a colimit diagram. Consider the following conditions:

\begin{itemize}
\item[$(a)$] Let $X \in \calX$ denote the image under $\overline{q}$ of the cone point
of $K^{\triangleright}$, so that we may view $\overline{q}$ as defining a map
$\widetilde{q}: K \rightarrow \calX_{/X}$. Let $f: Y \rightarrow X$ be a morphism in $\calY$. Then
the morphism $K \rightarrow \calY_{/Y}$ obtained by pullback of $\widetilde{q}$ along the
map $f$ classifies a colimit diagram $K^{\triangleright} \rightarrow \calY$.

\item[$(b)$] The composition $\chi \circ \overline{q}: K^{\triangleright} \rightarrow \widehat{\Cat}_{\infty}^{op}$ is a colimit diagram. Here $\chi: \calX \rightarrow \widehat{\Cat}_{\infty}^{op}$ classifies the
Cartesian fibration $\pi: \calO \rightarrow \calX$ appearing in Definition \ref{disty}. 

\item[$(c)$] Let $\overline{\alpha}: \overline{p} \rightarrow \overline{q}$ be a natural transformation between diagrams $\overline{p}, \overline{q}: K^{\triangleright} \rightarrow \calY$, where
$\overline{q}$ is colimit diagram $\calX$ and $\alpha = \overline{\alpha} | K$ is a Cartesian transformation.
Then $\overline{\alpha}$ is a Cartesian transformation if and only if $\overline{p}$ is a colimit diagram.
\end{itemize}
Then $(a) \wedge (b) \Leftrightarrow (c)$.
\end{proposition}

\begin{proof}
Let $\overline{\calC} = \Fun(K^{\triangleright}, \calY)^{/\overline{q}}$ and
$\calC = \Fun(K,\calY)^{/q}$. Let $\overline{\calC}^{0}$ denote the full subcategory of
$\overline{\calC}$ spanned by {\em Cartesian} natural tranformations $\overline{\alpha}: \overline{p} \rightarrow \overline{q}$, and let $\calC^0$ be defined similarly. Finally, let
$\overline{\calC}^{1}$ denote the full subcategory of $\overline{\calC}$ spanned by those natural transformations $\overline{\alpha}: \overline{p} \rightarrow \overline{q}$ such that $\overline{p}$ is a colimit diagram and $\alpha = \overline{\alpha} | K$ is a Cartesian transformation. 
Assertion $(a)$ is the requirement that $\widehat{\Cat}^{1} \subseteq \widehat{\calC}^0$, and
assertion $(c)$ is the requirement that $\widehat{\Cat}^{1} = \widehat{\calC}^0$. The implication
$(c) \Rightarrow (a)$ is obvious. Let us assume that $(a)$ is satisfied; we wish to show that
$(b)$ holds if and only if $\widehat{\Cat}^0 \subseteq \widehat{\calC}^1$.

Let $\calD$ denote the full subcategory of $\Fun(K^{\triangleright}, \calY)$ spanned by the colimit diagrams. Proposition \toposref{lklk} asserts that the restriction map $\calD \rightarrow \Fun(K,\calY)$ is a trivial fibration. It follows that the associated map $\calD^{/\overline{q}} \rightarrow \Fun(K,\calY)^{/q}$ is also a trivial fibration, and therefore restricts to a trivial fibration $\overline{\calC}^{1} \rightarrow \calC^{0}$.

According to Proposition \toposref{charcatlimit}, condition $(b)$ is equivalent to the assertion that
the projection $\overline{\calC}^{0} \rightarrow \calC^0$ is an equivalence of $\infty$-categories.
In view of the above argument, this is equivalent to the assertion that the fully faithful inclusion
$\overline{\calC}^{0} \subseteq \overline{\calC}^{1}$ is essentially surjective. Since
$\overline{\calC}^{0}$ is clearly stable under equivalence in $\overline{\calC}$, $(b)$
holds if and only if $\overline{\calC}^0 = \overline{\calC}^1$, as desired.
\end{proof}

\begin{corollary}\label{summa}
Let $\calX \subseteq \calY$ be a fully faithful inclusion of $\infty$-categories satisfying conditions
$(1)$ and $(2)$ of Definition \ref{disty}. Then $\calX \subseteq \calY$ is a distributor if and only if the following condition is satisfied: for every small simplicial set $K$ and every natural
transformation $\overline{\alpha}: \overline{p} \rightarrow \overline{q}$, if
$\overline{q}$ is a colimit diagram in $\calX$ and $\alpha = \overline{\alpha} | K$ is
Cartesian, then $\overline{\alpha}$ is Cartesian if and only if $\overline{p}$ is a colimit diagram.
\end{corollary}

\begin{remark}
It follows from the characterization given in Corollary \ref{summa} that if $\calX \subseteq \calY$ is
a distributor, then $\calX \subseteq \calX$ is also a distributor; in particular, $\calX$ is an $\infty$-topos
(Example \ref{jina}).
\end{remark}

\begin{definition}
Let $\calX \subseteq \calY$ be a distributor. We will say that a simplicial object
$Y_{\bigdot} \in \Fun( \Nerve(\cDelta)^{op}, \calY)$ is a {\it Segal space object} if the following conditions are satisfied:
\begin{itemize}
\item[$(1)$] The simplicial object $Y_{\bigdot}$ is a category object of $\calY$.
\item[$(2)$] The object $Y_0$ belongs to $\calX$.
\end{itemize}
We let $\SegS{\calX \subseteq \calY}$ denote the full subcategory of $\Fun( \Nerve(\cDelta)^{op}, \calY)$ spanned by the Segal space objects.
\end{definition}

\begin{remark}
We have a homotopy pullback diagram of $\infty$-categories
$$ \xymatrix{ \SegS{ \calX \subseteq \calY} \ar[r] \ar[d] & \Cat(\calY) \ar[d] \\
\calX \ar[r] & \calY, }$$
where $\Cat(\calY) \subseteq \Fun( \Nerve(\cDelta)^{op}, \calY)$ is the $\infty$-category of category objects of $\calY$. It follows from Theorem \toposref{surbus} that the $\infty$-category
$\SegS{ \calX \subseteq \calY}$ of Segal space objects of $\calY$ is presentable, and that each
functor in this diagram admits a left adjoint.
\end{remark}

\begin{notation}\label{righton}
Let $\calX \subseteq \calY$ be a distributor. Then the inclusions
$$ \Grp(\calX) \subseteq \Cat(\calX) \subseteq \Cat(\calY)$$
admit right adjoints, by Proposition \ref{cavem} and Remark \ref{cavem2}. It follows that the inclusion
$\Grp(\calX) \subseteq \SegS{\calX \subseteq \calY}$ admits a right adjoint, which we will denote by
$\Gp$. Note that Proposition \ref{cavem} implies that for every Segal space object $Y_{\bigdot}
\in \SegS{ \calX \subseteq \calY}$, the colocalization map $(\Gp Y)_{\bigdot} \rightarrow Y_{\bigdot}$
induces an equivalence $(\Gp Y)_0 \rightarrow Y_0$.
\end{notation}

\begin{definition}\label{compseg}
Let $\calX \subseteq \calY$ be a distributor. We will say that a Segal space object
$Y_{\bigdot}$ of $\calY$ is {\it complete} if the groupoid object $(\Gp Y)_{\bigdot} \in \Grp(\calX)$ is constant (Example \ref{defcon1}). 
We let $\CSS{ \calX \subseteq \calY}$ denote the full subcategory of $\Fun( \Nerve(\cDelta)^{op}, \calY)$ spanned by the complete Segal space objects.
\end{definition}

\begin{remark}\label{sluther}
Let $\calX \subseteq \calY$ be a distributor. It follows from Lemma \toposref{stur2} that $\CSS{ \calX \subseteq \calY}$ is an accessible localization of
$\SegS{ \calX \subseteq \calY}$, and therefore an accessible localization of $\Fun( \Nerve(\cDelta)^{op}, \calY)$.
\end{remark}

Our goal for the remainder of this section is to describe the localization functor
$L: \SegS{ \calX \subseteq \calY} \rightarrow \CSS{ \calX \subseteq \calY}$ more
explicitly. For example, given a morphism $f: Y_{\bigdot} \rightarrow Y'_{\bigdot}$ between
Segal space objects of $\calY$, we would like a simple criterion for testing whether
or not $Lf$ is an equivalence. This criterion can be formulated as follows:

\begin{definition}\label{jinke}
Let $\calX \subseteq \calY$ be a distributor, and let $f_{\bigdot}: Y_{\bigdot} \rightarrow Y'_{\bigdot}$ be
a map between Segal space objects of $\calY$. We will say that $f_{\bigdot}$ is a {\it Segal equivalence} if the following conditions are satisfied:
\begin{itemize}
\item[$(a)$] The map $| \Gp Y_{\bigdot} | \rightarrow | \Gp Y'_{\bigdot} |$ is an equivalence in
the $\infty$-topos $\calX$.

\item[$(b)$] The induced diagram
$$ \xymatrix{ Y_1 \ar[r] \ar[d] & Y'_1 \ar[d] \\
Y_0 \times Y_0 \ar[r] & Y'_0 \times Y'_0 }$$
is a pullback square in $\calY$.
\end{itemize}
\end{definition}

We can now state the main result of this section.

\begin{theorem}\label{segmin}
Let $\calX \subseteq \calY$ be a distributor. Then:
\begin{itemize}
\item[$(1)$] The inclusion $\CSS{ \calX \subseteq \calY} \subseteq \SegS{ \calX \subseteq \calY}$
admits a left adjoint $L$.
\item[$(2)$] Let $f: Y_{\bigdot} \rightarrow Y'_{\bigdot}$ be a morphism between Segal space objects of $\calY$. Then $Lf$ is an equivalence if and only if $f$ is a Segal equivalence.
In particular, for each $Y_{\bigdot} \in \SegS{ \calX \subseteq \calY}$, the localization map
$Y_{\bigdot} \rightarrow LY_{\bigdot}$ is a Segal equivalence.
\end{itemize}
\end{theorem}

The remainder of this section is devoted to the proof of Theorem \ref{segmin}.
We begin with a few easy observations about the collection of Segal equivalences
in $\SegS{\calX \subseteq \calY}$.

\begin{remark}\label{swinn}
Let $\calX \subseteq \calY$ be a distributor, and let $f: Y_{\bigdot} \rightarrow Y'_{\bigdot}$ 
be a Segal equivalence. Let $K$ be any simplicial set, and let $V$ be the set of vertices of
$K$. Then the diagram
$$ \xymatrix{ Y(K) \ar[r] \ar[d] & Y'(K) \ar[d] \\
\prod_{v \in V} Y_0 \ar[r] & \prod_{v \in V} Y'_0 }$$
is a pullback square.
\end{remark}

\begin{remark}
Let $\calX \subseteq \calY$ be a distributor, and let $G: \calY \rightarrow \calX$ be a right adjoint to the inclusion of $\calX$ into $\calY$. Then composition with $G$ carries Segal equivalences in
$\SegS{ \calX \subseteq \calY}$ to Segal equivalences in $\SegS{ \calX \subseteq \calX}$. Combining
this observation with Remark \ref{swinn} and Proposition \ref{canman}, we deduce that for every
Segal equivalence $f: Y_{\bigdot} \rightarrow Y'_{\bigdot}$ in $\calY$, the induced diagram
$$ \xymatrix{ (\Gp Y)_{1} \ar[r] \ar[d] & ( \Gp Y')_1 \ar[d] \\
Y_0 \times Y_0 \ar[r] & Y'_0 \times Y'_0 }$$
is a pullback square; in other words, the induced map $\Gp Y_{\bigdot} \rightarrow \Gp Y'_{\bigdot}$
is a Segal equivalence.
\end{remark}

\begin{remark}\label{humma}
Let $\calX \subseteq \calY$ be a distributor, and suppose given a commutative diagram
$$ \xymatrix{ & Y_{\bigdot} \ar[dr]^{g} & \\
X_{\bigdot} \ar[ur]^{f} \ar[rr]^{h} & & Z_{\bigdot} }$$
of Segal space objects of $\calY$. Suppose further that $g$ is a Segal equivalence. Then
$f$ is a Segal equivalence if and only if $h$ is a Segal equivalence; this follows immediately
from the definition. In fact, there is a converse to this statement: if $f$ and $h$ are Segal equivalences, then $g$ is a Segal equivalence. This does not follow immediately from the definition, but it is a consequence of Theorem \ref{segmin}, which we will prove below.
\end{remark}

\begin{remark}\label{swamm}
Let $\calX \subseteq \calY$ be a distributor, and let $f: Y_{\bigdot} \rightarrow Y'_{\bigdot}$ be a map between Segal space objects of $\calY$. Suppose that $f$ satisfies condition $(b)$ of Definition
\ref{jinke}. Then condition $(a)$ is equivalent to the following:
\begin{itemize}
\item[$(a')$] The map $Y_0 \rightarrow | \Gp Y'_\bigdot |$ is an effective epimorphism in the
$\infty$-topos $\calX$.
\end{itemize}
The implication $(a) \Rightarrow (a')$ is clear, since the map $Y_0 \rightarrow | \Gp Y_{\bigdot} |$ is an effective epimorphism. Conversely, suppose that $(a')$ is satisfied; we wish to show that the
canonical map $| \Gp Y_{\bigdot} | \rightarrow | \Gp Y'_{\bigdot} |$ is an equivalence. This map determines an augmented simplicial object $\overline{Y}_{\bigdot}: \Nerve(\cDelta_{+})^{op} \rightarrow \calY$ such that $\overline{Y}_{-1} = | \Gp Y'_{\bigdot} |$ and $\overline{Y}_{\bigdot} | \Nerve(\cDelta)^{op} = \Gp Y_{\bigdot}$. We wish to prove that $\overline{Y}_{\bigdot}$ is a colimit diagram. Since $\calX$ is an $\infty$-topos and $Y_{\bigdot}$ is a groupoid. Since $\calX$ is an $\infty$-topos and
the augmentation map $u: Y_0 \rightarrow | \Gp Y'_{\bigdot} |$ is an effective epimorphism, it will suffice to show that the augmented simplicial object $\overline{Y}_{\bigdot}$ exhibits $\Gp Y_{\bigdot}$
as a \Cech nerve of $u$. Since $\Gp Y_{\bigdot}$ is a groupoid object by assumption, we are reduced to proving that the map
$q: Y_1 \rightarrow Y_0 \times_{ | \Gp Y'_{\bigdot} | } Y_0$ is an equivalence.
This follows from $(b)$, since $q$ is a pullback of the equivalence 
$Y'_1 \rightarrow Y'_0 \times_{ | \Gp Y'_{\bigdot} | } Y'_0$.
\end{remark}

\begin{remark}\label{swammer}
Let $\calX \subseteq \calY$ be a distributor, and let $f: Y_{\bigdot} \rightarrow Y'_{\bigdot}$ be a map between Segal space objects of $\calY$. Suppose that $f$ satisfies condition $(b)$ of Definition
\ref{jinke}. Then condition $(a)$ is satisfied whenever $f_0: Y_0 \rightarrow Y'_0$ is an effective epimorphism. This follows immediately from Remark \ref{swamm}, since the canonical map
$Y'_0 \rightarrow | \Gp Y'_{\bigdot} |$ is an effective epimorphism.
\end{remark}

We will prove Theorem \ref{segmin} by applying the following general principle:

\begin{proposition}\label{genprip}
Let $\calC$ be an $\infty$-category. Suppose given a full subcategory $\calC^{0} \subseteq \calC$,
a collection $S$ of morphisms in $\calC$, and a functor $L: \calC \rightarrow \calC$ equipped with
a natural transformation $\alpha: \id \rightarrow L$, satisfying the following conditions:
\begin{itemize}
\item[$(1)$] The full subcategory $\calC^0$ and the collection of morphisms $S$ are stable under equivalence.
\item[$(2)$] Suppose given a commutative diagram
$$ \xymatrix{ & Y \ar[dr]^{g} & \\
X \ar[ur]^{f} \ar[rr]^{h} & & Z }$$
in $\calC$, where $g$ belongs to $S$. Then $f \in S$ if and only if $h \in S$.
\item[$(3)$] Let $f: X \rightarrow Y$ be a morphism in $\calC$. If $X,Y \in \calC^0$ and $f \in S$, then $f$ is an equivalence.
\item[$(4)$] For every object $X \in \calC$, the morphism $\alpha_{X}: X \rightarrow LX$ belongs to
$S$, and the object $LX$ belongs to $\calC^0$.
\item[$(5)$] The functor $L$ carries morphisms of $S$ to morphisms of $S$ (in view of assumption
$(3)$, this is equivalent to the requirement that $L$ carries morphisms of $S$ to equivalences).
\end{itemize}
Then:
\begin{itemize}
\item[$(a)$] The essential image of $L$ coincides with $\calC^0$.
\item[$(b)$] The functor $L: \calC \rightarrow \calC^0$ is left adjoint to the inclusion of
$\calC^0$ into $\calC$.
\item[$(c)$] A morphism $f$ of $\calC$ belongs to $S$ if and only if $Lf$ is an equivalence.
\end{itemize}
\end{proposition}

\begin{proof}
We first prove $(a)$. Note that condition $(4)$ guarantees that $L$ factors through
$\calC^0$. Conversely, suppose that $X \in \calC^0$. Condition $(4)$ guarantees that
$LX \in \calC^0$ and that $\alpha_X: X \rightarrow LX$ belongs to $S$, so that
$\alpha_X$ is an equivalence by $(3)$; it follows that $X$ belongs to the essential image of
$L$.

We now prove $(b)$. In view of Proposition \toposref{recloc}, it will suffice to prove that
for $X \in \calC$, the maps $\alpha_{LX}: LX \rightarrow LLX$ and $L \alpha_X: LX
\rightarrow LLX$ are equivalences. We note that condition $(3)$ implies that
$\alpha_{LX}$ and $\alpha_X$ belong to $S$. Applying $(5)$, we deduce that $L \alpha_X$
is an equivalence. To prove that $\alpha_{LX}$ is an equivalence, it will suffice (by virtue
of $(3)$) to show that $LX$ and $LLX$ belong to $\calC^0$, which follows from assumption $(4)$.

To prove $(c)$, we must show that if $f: X \rightarrow Y$ is a morphism such that $Lf$ is an equivalence, then $f \in S$. Consider the diagram
$$ \xymatrix{ X \ar[r]^{f} \ar[dr]^{g} \ar[d]^{\alpha_X} & Y \ar[d]^{\alpha_Y} \\
LX \ar[r]^{Lf} & LY. }$$
Assumption $(4)$ guarantees that the vertical morphisms belong to $S$. Applying
$(2)$, we deduce that $f \in S$ as desired.
\end{proof}

We now proceed to deduce Theorem \ref{segmin} by showing that the hypotheses of Proposition \ref{genprip} are satisfied, if we take $\calC = \SegS{ \calX \subseteq \calY}$, $\calC^{0} = \CSS{ \calX \subseteq \calY}$, and $S$ to be the class of Segal equivalences. To prove this, we will need to construct
a functor $L: \SegS{ \calX \subseteq \calY} \rightarrow \SegS{ \calX \subseteq \calY}$ and a natural
transformation $\alpha: \id \rightarrow L$ satisfying conditions $(4)$ and $(5)$.

\begin{construction}\label{tubal}
Let $\calX \subseteq \calY$ be a distributor. We let $\calJ$ denote the $\infty$-category $\Nerve( \Fun( [1], \cDelta) )^{op}$.
Let $i: \Nerve( \cDelta)^{op} \rightarrow \calJ$ denote the fully faithful inclusion which
carries an object $[n] \in \cDelta$ to the morphism $[n] \rightarrow [0]$, and let
$i_{\ast}: \Fun( \Nerve(\cDelta)^{op}, \calY ) \rightarrow \Fun( \calJ, \calY)$ denote the associated right Kan extension functor. For $0 \leq j \leq 1$, let $e_{i}: \calJ \rightarrow \Nerve( \cDelta)^{op}$ be the functor given by evaluation at the object $j \in [1]$, so that the composition $e_{j} \circ i$ is the identity on $\Nerve(\cDelta)^{op}$. Let $e_0^{\ast}: \Fun( \Nerve(\cDelta)^{op}, \calY) \rightarrow \Fun( \calJ, \calY)$ be given by composition
with $e_0$.

For every Segal space object $Y_{\bigdot}$ of $\calY$, the canonical identification
$(i \circ e_0)^{\ast} Y_{\bigdot} \simeq Y_{\bigdot}$ induces a map $e_0^{\ast} Y_{\bigdot} \rightarrow
i_{\ast} Y_{\bigdot}$. Form a pullback diagram
$$ \xymatrix{ DY_{\bigdot} \ar[r] \ar[d] & e_0^{\ast} Y_{\bigdot} \ar[d] \\
i_{\ast} (\Gp Y)_{\bigdot} \ar[r] & i_{\ast} Y_{\bigdot}. }$$
The construction $Y_{\bigdot} \mapsto DY_{\bigdot}$ determines a functor from
$\SegS{ \calX \subseteq \calY}$ to $\Fun( \calJ, \calY)$, which we will denote by $D$. 

We define full subcategories $\calJ_0 \subseteq \calJ_1 \subseteq \calJ$ as follows:
\begin{itemize}
\item An object $f: [n] \rightarrow [m]$ of $\calJ$ belongs to $\calJ_1$ if and only if $f$ is surjective.
\item An object $f: [n] \rightarrow [m]$ of $\calJ$ belongs to $\calJ_0$ if and only if $f$ is bijective.
\end{itemize}
Let $\pi^0$ and $\pi^1$ denote the restrictions of $e_1$ to $\calJ_0$ and $\calJ_1$, respectively. We note that $\pi^0$ is an isomorphism. Let
$L: \SegS{ \calX \subseteq \calY} \rightarrow \Fun( \Nerve(\cDelta)^{op}, \calY)$ denote 
the composition
$$ \SegS{ \calX \subseteq \calY} \stackrel{ D}{\rightarrow} \Fun( \calJ, \calY)
\stackrel{| \calJ_1}{\rightarrow} \Fun( \calJ_1, \calY) \stackrel{ \pi^1_! }{\rightarrow} \Fun(\Nerve(\cDelta)^{op}, \calY),$$ 
where $\pi^1_{!}$ is given by left Kan extension along $\pi^1$.

For any Segal space object $Y_{\bigdot}$ of $\calY$, the canonical map
$DY_{\bigdot} \rightarrow \pi^{\ast} Y_{\bigdot}$ induces an identification
$DY_{\bigdot} | \calJ_0 \simeq (\pi^0)^{\ast} Y_{\bigdot}$ (this follows from the observation that
$(\Gp Y)_{0} \simeq Y_0$; see Notation \ref{righton}). We therefore obtain a canonical map
$$ Y_{\bigdot} \simeq \pi^0_! (DY_{\bigdot} | \calJ_0) \rightarrow \pi^{1}_! DY_{\bigdot}.$$
This construction determines a natural transformation $\alpha: \id \rightarrow L$ of functors from
$\SegS{ \calX \subseteq \calY}$ to $\Fun( \Nerve(\cDelta)^{op}, \calY)$. 
\end{construction}

\begin{remark}
More informally, the functor $D$ may be described as follows. Let $Y_{\bigdot}$ be
a Segal space object of $\calY$, and let $f: [n] \rightarrow [m]$ be an object of
$\calJ^1$. Then $(DY_{\bigdot})(f)$ is given by the fiber product
$Y_{n} \times_{ \prod_{0 \leq i \leq m} Y( f^{-1} \{i\} ) } \prod_{0 \leq i \leq m} (\Gp Y)(f^{-1} \{i\} ).$
\end{remark}

\begin{lemma}\label{precooke}
Let $\calX \subseteq \calY$ be a distributor. Suppose given a commutative diagram
$$ \xymatrix{ Z \ar[r]^{g} \ar[d] & Y \ar[r] \ar[d] & X \ar[d]^{f} \\
Z' \ar[r]^{g'} & Y' \ar[r] & X'. }$$
Assume further that:
\begin{itemize}
\item[$(1)$] Every square in the above diagram is a pullback.
\item[$(2)$] The map $f$ is an effective epimorphism in the $\infty$-topos $\calX$.
\item[$(3)$] The map $g$ is an equivalence in $\calY$.
\end{itemize}
Then the map $g'$ is an equivalence in $\calY$.
\end{lemma}

\begin{proof}
Let $X_{\bigdot}: \Nerve(\cDelta)^{op}$ be a \Cech nerve of $f$, so that
$X_0 \simeq X$.
Define simplicial objects $Y_{\bigdot}$ and $Z_{\bigdot}$ by the formulas
$$ Y_{n} = Y' \times_{X'} X_n \quad \quad Z_n = Z' \times_{X'} X_n,$$
so that we have a natural transformation of simplicial objects
$g_{\bigdot}: Y_{\bigdot} \rightarrow Z_{\bigdot}$. Assumption $(1)$ guarantees that
$g_0$ is equivalent to $g$, and therefore an equivalence by assumption $(3)$. Since
$g_{\bigdot}$ is a Cartesian transformation, we conclude that each $g_{n}$ is an equivalence. It follows
that $g_{\bigdot}$ induces an equivalence $|Z_{\bigdot} | \rightarrow |Y_{\bigdot} |$. 
Since $\calX \subseteq \calY$ is a distributor, this map can be identified with the pullback of
$g'$ along the monomorphism $j: | X_{\bigdot} | \rightarrow X'$. We complete the proof by observing that condition $(2)$ guarantees that $j$ is an equivalence.
\end{proof}

\begin{lemma}\label{cooke}
Let $\calX \subseteq \calY$ be a distributor, and let $f: Y_{\bigdot} \rightarrow Y'_{\bigdot}$ be
a map between simplicial objects of $\calY$. Suppose that:
\begin{itemize}
\item[$(1)$] The object $Y_{\bigdot}$ is a Segal space object of $\calY$.
\item[$(2)$] The induced map $Y_0 \rightarrow Y'_0$ is an effective epimorphism in $\calX$.
\item[$(3)$] For each $n \geq 0$, the diagram
$$ \xymatrix{ Y_{n} \ar[r] \ar[d] & Y'_{n} \ar[d] \\
\prod_{0 \leq i \leq n} Y_0 \ar[r] & \prod_{0 \leq i \leq n} Y'_0 }$$
is a pullback square in $\calY$.
\end{itemize}
Then $Y'_{\bigdot}$ is a Segal space object of $\calY$, and the map $f$ is a Segal equivalence.
\end{lemma}

\begin{proof}
We first show that $Y'_{\bigdot}$ is a Segal space object of $\calY$.
Since $Y'_0 \in \calX$ by assumption, it will suffice to show that $Y'_{\bigdot}$ is a category
object of $\calY$. It follows from $(1)$ and $(3)$ that the canonical map
$$ \psi: Y'_{n} \rightarrow Y'_1 \times_{ Y'_0} \ldots \times_{ Y'_0} Y'_1$$
becomes an equivalence after pullback along the map $g: \prod_{ 0 \leq i \leq n} Y_0 \rightarrow
\prod_{0 \leq i \leq n} Y'_0$. Assumption $(2)$ implies that $g$ is an effective epimorphism in $\calX$, so that $\psi$ is an equivalence by Lemma \ref{precooke}.

We now claim that $f$ is a Segal equivalence. This follows immediately from assumption $(3)$
(in the case $n=1$) and Remark \ref{swammer}.
\end{proof}

\begin{lemma}\label{sowork}
Let $\calX \subseteq \calY$ be a distributor, and suppose given a diagram
$$ \xymatrix{ & Y_{\bigdot} \ar[dr]^{g}  & \\
X_{\bigdot} \ar[ur]^{f} \ar[rr]^{h} & & Z_{\bigdot} }$$
of Segal space objects of $\calY$. If $f$ and $h$ are Segal equivalences and the map
$f$ induces an effective epimorphism $X_0 \rightarrow Y_0$, then $g$ is a Segal equivalence.
\end{lemma}

\begin{proof}
The only nontrivial point is to verify that the map
$$ Y_1 \rightarrow Y_0 \times_{ Z_0} Z_1 \times_{ Z_0} Y_0$$
is an equivalence in $\calY$. Since $f$ and $h$ are Segal equivalences, this map becomes
an equivalence after pullback along the effective epimorphism $X_0 \times X_0 \rightarrow Y_0 \times Y_0$. We conclude by applying Lemma \ref{precooke}.
\end{proof}

\begin{lemma}\label{slavish}
Let $\calX$ be an $\infty$-topos, and let $f: X_{\bigdot} \rightarrow Y_{\bigdot}$ be a Segal equivalence between category objects of $\calX$. Suppose that $X_{\bigdot}$ is a groupoid object and that the map 
$X_0 \rightarrow Y_0$ is an effective epimorphism. Then $Y_{\bigdot}$ is a groupoid object.
\end{lemma}

\begin{proof}
In view of Proposition \toposref{grpobjdef}, it will suffice to show that if $K \rightarrow K'$ is a weak homotopy equivalence of finite simplicial sets which is bijective on vertices, then the induced map
$\phi: Y(K') \rightarrow Y(K)$ is an equivalence in $\calX$. Let $V$ denote the common vertex set of $K$ and $K'$. Since the map $X_0 \rightarrow Y_0$ is an effective epimorphism, it will suffice to check that
$\phi$ is an equivalence after pullback along the map $\prod_{v \in V} X_0 \rightarrow \prod_{v \in V} Y_0$. Invoking Remark \ref{swinn}, we can identify the pullback of $\phi$ with the map
$X(K') \rightarrow X(K)$, which is an equivalence in view of our assumption that $X_{\bigdot}$ is a groupoid object (Proposition \toposref{grpobjdef}).
\end{proof}

\begin{lemma}\label{colk}
Fix integers $m, n \geq 0$, and let $\calC$ denote the full subcategory of
$\cDelta_{ / [m]} \times_{ \cDelta} \cDelta_{ / [n] }$ spanned by those diagrams
$$ [m] \stackrel{f}{\leftarrow} [k] \stackrel{g}{\rightarrow} [n]$$
for which $f$ is surjective. Then $\Nerve(\calC)$ is weakly contractible.
\end{lemma}

\begin{proof}
Let $\calC'$ denote the full subcategory of $\calC$ spanned by those objects for which
the map $f \times g: [k] \rightarrow [m] \times [n]$ is injective. The inclusion
$\calC' \rightarrow \calC$ admits a left adjoint, given by $[k] \mapsto [k]/ \sim$, where
$\sim$ is the equivalence relation defined by the requirement that $i \sim j$ if and only if
$f(i) = f(j)$ and $g(i) = g(j)$. Consequently, it will suffice to prove that $\Nerve(\calC')$ is weakly contractible. We observe that $\calC'$ can be identified with the partially ordered set of
linearly ordered subsets $S \subseteq [m] \times [n]$, such that the projection map
$S \rightarrow [m]$ is surjective.

We now proceed by induction on $n$. If $n = 0$, then $\calC'$ has a single object
(corresponding to the subset $S = [m] \times [0]$) and the result is obvious. We may therefore assume that $n > 0$. For $0 \leq i \leq m+1$, let $\calC'_{i}$ denote the full subcategory of
$\calC'$ spanned by those subsets $S \subseteq [m] \times [n]$ which do not contain
$(j, n)$ for $j \geq i$. We have a chain of inclusions
$$ \calC'_0 \subseteq \calC'_{1} \subseteq \ldots \subseteq \calC'_{m+1} = \calC'.$$
The inductive hypothesis guarantees that $\Nerve(\calC'_{i})$ is weakly contractible.
To complete the proof, it will suffice to show that each of the inclusions
$\Nerve(\calC'_{i}) \subseteq \Nerve( \calC'_{i+1})$ is a weak homotopy equivalence.

Let $\calD \subseteq \calC'_{i+1}$ be the full subcategory spanned by those
subsets $S \subseteq [m] \times [n]$ satisfying the following condition:
if $(i,n) \in S$, then $(i, n-1) \in S$. We will prove that the inclusions
$\Nerve( \calC'_{i}) \subseteq \Nerve(\calD) \subseteq \Nerve(\calC'_{i+1})$ are weakly contractible.
To prove this, it suffices to observe that the inclusion $\calC'_{i} \subseteq \calD$ has a right
adjoint (given by $S \mapsto \begin{cases} S & \text{ if } (i,n) \notin S \\
S - \{ (i,n) \} & \text{ if } (i,n) \in S. \end{cases}$)
and the inclusion $\calD \subseteq \Nerve( \calC'_{i+1})$ has a left adjoint
(given by $S \mapsto \begin{cases} S & \text{if } (i,n) \notin S \\
S \cup \{ (i,n-1) \} & \text{ if } (i,n) \in S. \end{cases}$). 
\end{proof}

\begin{proposition}\label{mainwork}
Let $\calX \subseteq \calY$ be a distributor, and let $L: \SegS{ \calX \subseteq \calY}
\rightarrow \Fun( \Nerve(\cDelta)^{op}, \calY)$ and $\alpha: \id \rightarrow L$ be as defined
in Construction \ref{tubal}. Then:
\begin{itemize}
\item[$(1)$] For every Segal space object $Y_{\bigdot}$ of $\calY$, the simplicial object
$LY_{\bigdot}$ is a Segal space object of $\calY$.
\item[$(2)$] For every Segal space object $Y_{\bigdot}$ of $\calY$, the natural transformation
$\alpha$ induces a Segal equivalence $Y_{\bigdot} \rightarrow LY_{\bigdot}$.
\item[$(3)$] If $f: Y_{\bigdot} \rightarrow Y'_{\bigdot}$ is a Segal equivalence of Segal
space objects of $\calY$, then $Lf$ is a Segal equivalence.
\item[$(4)$] If $Y_{\bigdot}$ is a groupoid object of $\calX$, then $L Y_{\bigdot}$ is again a groupoid object
of $\calX$.
\item[$(5)$] For every Segal space object $Y_{\bigdot}$ of $\calY$, the Segal space object
$LY_{\bigdot}$ is complete.
\end{itemize}
\end{proposition}

\begin{proof}
Throughout the proof, we will employ the notation of Construction \ref{tubal}.
For each object $[n] \in \Nerve(\cDelta)^{op}$, let $\calJ^1_{[n]}$ denote the fiber
product $\calJ^1 \times_{ \Nerve(\cDelta)^{op} } \{ [n] \}$ (where
$\calJ^{1}$ maps to $\Nerve(\cDelta)^{op}$ by the projection $\pi^{1}$). 
We first prove:
\begin{itemize}
\item[$(\ast)$] The inclusion $\calJ^{1}_{[n]} \subseteq \calJ^{1} \times_{ \Nerve(\cDelta)^{op}}  \Nerve(\cDelta)^{op})_{/[n]}$ is cofinal. 
\end{itemize}
In view of Theorem \toposref{hollowtt}, it will suffice to show that for each object
$J \in \calJ^{1} \times_{ \Nerve(\cDelta)^{op} } (\Nerve( \cDelta)^{op})_{/[n]}$, the
$\infty$-category $( \calJ^{1}_{[n]})_{J/}$ is weakly contractible. We can identify
$X$ with a commutative diagram 
$$ \xymatrix{ & [m'] \ar[d]^{\alpha} \\
[n] \ar[r]^{\beta} & [m] }$$
in $\cDelta$, where $\alpha$ is surjective. The $\infty$-category
$(\calJ^{1}_{[n]})_{J/}$ can then be identified with the opposite of the nerve
of the category $\calC( [n] \stackrel{\beta}{\rightarrow} [m] \stackrel{\alpha}{\leftarrow} [k])$ whose objects are commutative diagrams
$$ \xymatrix{ [n'] \ar[r] \ar[d]^{\alpha'} & [m'] \ar[d]^{\alpha} \\
[n] \ar[r]^{\beta} & [m] }$$
where $\alpha'$ is also surjective. We observe that
$$\calC( [n] \stackrel{\beta}{\rightarrow} [m] \stackrel{\alpha}{\leftarrow} [m']])$$
is equivalent to a product $$\prod_{ 0 \leq i \leq m} \calC( \beta^{-1} \{i\} \rightarrow \{i\}  \leftarrow \alpha^{-1} \{i\})$$ (where, by convention, the category $\calC( \beta^{-1} \{i\} \rightarrow \{i\}  \leftarrow \alpha^{-1} \{i\})$ consists of a single object if $\beta^{-1} \{i\}$ is empty).
It therefore suffices to treat the case where $m=0$, which follows from
Lemma \ref{colk}.

Let $\calJ^{2}_{[n]}$ denote the subcategory of $\calJ^{1}_{[n]}$ consisting of all the objects, together with those morphisms which correspond to diagrams of {\em surjective} maps
$$ \xymatrix{ [m] \ar[rr] \ar[dr] & & [m'] \ar[dl] \\
& [n]. & }$$
We observe that the inclusion $\calJ^{2}_{[n]} \subseteq \calJ^{1}_{[n]}$ is equivalent
to the $(n+1)$st power of the inclusion $\Nerve(\cDelta_{s})^{op} \subseteq \Nerve(\cDelta)^{op}$, and
therefore cofinal (Lemma \toposref{bball4}). Combining this observation with $(\ast)$, we conclude that for any functor $F: \calJ^1 \rightarrow \calY$ there is a canonical equivalence
$$(\pi^1_{!} F)([n]) \simeq \colim F | \calJ^2_{[n]}.$$

Let $Y_{\bigdot}$ be a Segal space object of $\calY$. We have a pullback diagram
$$ \xymatrix{ DY_{\bigdot} \ar[r] \ar[d] & e_0^{\ast} Y_{\bigdot}  \ar[d] \\
i_{\ast} \Gp Y_{\bigdot} | \calJ^2 \ar[r] & i_{\ast} Y_{\bigdot} }$$
of functors from $\calJ$ to $\calY$. Invoking the assumption that $Y_{\bigdot}$ is a category object of $\calY$, we deduce that the right vertical map induces a Cartesian transformation
$(e_0^{\ast} Y_{\bigdot})| \calJ^2_{[n]} \rightarrow ( i_{\ast} Y_{\bigdot} )| \calJ^2_{[n]}$ for
each $n \geq 0$.
It follows that the left vertical map also induces a Cartesian transformation
$$ DY_{\bigdot} | \calJ^2_{[n]} \rightarrow ( i_{\ast} \Gp Y_{\bigdot} ) | \calJ^2.$$
Since $\calX \subseteq \calY$ is a distributor, Proposition \ref{maid} guarantees that
the diagram
$$ \xymatrix{ DY_{\bigdot}(C) \ar[r] \ar[d] & LY_{\bigdot}( [n] ) \ar[d] \\
(i_{\ast} \Gp Y_{\bigdot})(C) \ar[r] & \colim (i_{\ast} \Gp Y_{\bigdot} )| \calJ^2_{[n]} }$$
is a pullback square in $\calY$, for each object $C \in \calJ^{2}_{[n]}$. In particular, taking
$C$ to be the initial object of $\calJ^{2}_{[n]}$ (and using the fact that colimits in $\calX$
commute with products), we see that the diagram
$$ \xymatrix{ Y_{n} \ar[r] \ar[d] & LY_n \ar[d] \\
\prod_{0 \leq i \leq n} Y_0 \ar[r] & \prod_{0 \leq i \leq n} Y'}$$
where 
$$Y' = | \Gp Y_{\bigdot} | = \colim DY_{\bigdot}| \calJ^{1}_{[0]} \simeq LY_0.$$
We observe that this is the diagram induced by the natural transformation $\alpha$.
Since the map $Y_0 \rightarrow Y'$ is an effective epimorphism, it follows from
Lemma \ref{cooke} that $LY_{\bigdot}$ is a Segal space object of $\calY$ and that
the map $\alpha: Y_{\bigdot} \rightarrow LY_{\bigdot}$ is a Segal equivalence.
This proves $(1)$ and $(2)$.

To prove $(3)$, let us suppose that $f: Y_{\bigdot} \rightarrow Y'_{\bigdot}$ is a Segal equivalence.
We wish to prove that the induced map $Lf: LY_{\bigdot} \rightarrow LY'_{\bigdot}$ is an equivalence of simplicial object of $\calY$. Consider the diagram
$$ \xymatrix{ Y_{\bigdot} \ar[dr]^{g} \ar[r]^{f} \ar[d]^{\alpha_{Y_{\bigdot}}} & Y'_{\bigdot} \ar[d]^{\alpha_{Y'_{\bigdot}}} \\
LY_{\bigdot} \ar[r]^{Lf} & LY'_{\bigdot}. }$$
Since $f$ and $\alpha_{Y'_{\bigdot}}$ are Segal equivalences, we deduce that $g$ is a Segal
equivalence (Remark \ref{humma}). Using the fact that $g$ and $\alpha_{Y_{\bigdot}}$ are Segal equivalences, together with the fact that the map $\alpha_{Y_{\bigdot}}$ induces an
effective epimorphism $Y_0 \rightarrow LY_0 \simeq | \Gp Y_{\bigdot} |$, we deduce that
$Lf$ is a Segal equivalence (Lemma \ref{sowork}).

To prove assertion $(4)$, we may assume without loss of generality that $\calX = \calY$; the desired result then follows from $(2)$ and Lemma \ref{slavish}. It remains to prove $(5)$. Assertion $(4)$ guarantees that
$L \Gp Y_{\bigdot}$ is a groupoid object of $\calX$ admitting a map to $L Y_{\bigdot}$. We therefore obtain a canonical factorization
$$ \xymatrix{ & \Gp LY_{\bigdot} \ar[dr] & \\
L \Gp Y_{\bigdot} \ar[ur] \ar[rr] & & L Y_{\bigdot}. }$$
Let $\calJ^{1}_{/[0]}$ denote the fiber product $\calJ^{1} \times_{ \Nerve(\cDelta)^{op} } ( \Nerve(\cDelta)^{op} )_{/ [0] }$. We have a commutative diagram
$$ \xymatrix{ \colim (D \Gp Y_{\bigdot}) | \calJ^0 \ar[r]^-{p_0} \ar[d]^{q_0} & | \Gp Y_{\bigdot} | \ar[r]^{s_0} & | L \Gp Y_{\bigdot} | \ar[dd]^{s_1} \\
\colim (D \Gp Y_{\bigdot}) | \calJ^{1}_{/[0]} \ar[r]^-{p_1} \ar[d] & (L \Gp Y)_0 \ar[d]^{q_1} & \\
\colim (DY_{\bigdot}) | \calJ^{1}_{/[0]} \ar[r]^-{p_2} & (LY)_0 \ar[r]^{r} & | (\Gp L Y)_{\bigdot}. }$$
We now argue as follows:
\begin{itemize}
\item The maps $p_0$, $p_1$, and $p_2$ are equivalences by construction. 
\item The map $q_1$ is an equivalence (since $L Y_{0}$ depends
only on $\Gp Y_{\bigdot}$, and the functor $\Gp Y_{\bigdot}$ is idempotent). 
\item The map $q_0$ is an equivalence. To prove this, we observe that
$(D \Gp Y)_{\bigdot}$ can be identified with the composition of $Y_{\bigdot}$ with the projection
$\pi: \calJ \rightarrow \Nerve(\cDelta)^{op}$. It follows that $\pi$ induces a map
$\colim (D \Gp Y)_{\bigdot}| \calJ^{1}_{/[0]} \rightarrow | \Gp Y_{\bigdot} |$ which is left
inverse to $q_0$. To prove that this left inverse is an equivalence, it suffices to observe that
the projection $\calJ^{1}_{/[0]} \rightarrow \Nerve(\cDelta)^{op}$ is cofinal (this follows from
$(\ast)$ and Proposition \toposref{cofbasic}). 

\item The composition $s_1 \circ s_0$ is an equivalence, since the map $Y_{\bigdot} \rightarrow LY_{\bigdot}$ is a Segal equivalence by $(2)$.
\end{itemize}
It now follows by a diagram chase that $r$ is an equivalence, so that $LY_{\bigdot}$ is a complete
Segal space object of $\calY$ as desired.
\end{proof}

\begin{lemma}\label{nowork}
Let $\calX \subseteq \calY$ be a distributor, and let $f_{\bigdot}: Y_{\bigdot} \rightarrow Y'_{\bigdot}$ be
a Segal equivalence between complete Segal space objects of $\calY$. Then $f$ is an equivalence.
\end{lemma}

\begin{proof}
We must show that $f_{n}: Y_n \rightarrow Y'_n$ is an equivalence for $n \geq 0$. Since
$Y_{\bigdot}$ and $Y'_{\bigdot}$ are category objects of $\calY$, it will suffice to treat the
case $n \leq 1$.

We first consider the case $n=0$. We have a commutative diagram
$$ \xymatrix{ | \Gp Y_{\bigdot} | \ar[r]^{f'} & | \Gp Y'_{\bigdot} | \\
(\Gp Y)_0 \ar[u] \ar[r] \ar[d] & ( \Gp Y')_0 \ar[u] \ar[d] \\
Y_0 \ar[r]^{f_0} & Y'_0 }$$
The map $f'$ is an equivalence because $f$ is assumed to be a Segal equivalence. The upper
vertical maps are equivalences since $Y_{\bigdot}$ and $Y'_{\bigdot}$ are complete, and the lower
vertical maps are always equivalences (see Notation \ref{righton}); it follows by a two-out-of-three argument that $f_0$ is an equivalence.

We now treat the case $n=1$. Since $f_{\bigdot}$ is a Segal equivalence, we have a pullback diagram
$$ \xymatrix{ Y_1 \ar[r]^{f_1} \ar[d] & Y'_1 \ar[d] \\
Y_0 \times Y_0 \ar[r] & Y'_0 \times Y'_0. }$$ 
The first part of the proof shows that the lower horizontal map is an equivalence, so that the upper horizontal map is also an equivalence as desired.
\end{proof}

\begin{proof}[Proof of Theorem \ref{segmin}]
Combine Propositions \ref{genprip} and \ref{mainwork}, Lemma \ref{nowork}, and
Remark \ref{humma}.
\end{proof}             

For later use, we record the following property of distributors:

\begin{proposition}\label{presmashet}
Let $\calX \subseteq \calY$ be a distributor. Suppose that:
\begin{itemize}
\item[$(a)$] Filtered colimits in $\calY$ are left exact (Definition \toposref{leftexactcolim}).
\item[$(b)$] Let $G: \calY \rightarrow \calX$ be a right adjoint to the inclusion. Then $G$ commutes with filtered colimits.
\end{itemize}
Then:
\begin{itemize}
\item[$(1)$] The full subcategory $\CSS{ \calX \subseteq \calY}$ is stable under small filtered colimits
in $\calY$.
\item[$(2)$] Let $L: \Fun( \Nerve(\cDelta)^{op}, \calY) \rightarrow \CSS{ \calX \subseteq \calY}$
be a left adjoint to the inclusion. Then $L$ preserves small filtered colimits, when regarded
as a functor from $\Fun( \Nerve(\cDelta)^{op}, \calY)$ to itself.
\item[$(3)$] Filtered colimits in $\CSS{ \calX \subseteq \calY}$ are left exact.
\item[$(4)$] Let $\calX' \subseteq \CSS{ \calX \subseteq \calY}$ be the essential image
of the diagonal map $\calX \rightarrow \Fun( \Nerve(\cDelta)^{op}, \calY)$, and let
$G'$ be a right adjoint to the inclusion $\calX' \subseteq \CSS{ \calX \subseteq \calY}$. Then
$G'$ preserves small filtered colimits.
\end{itemize}
\end{proposition}

\begin{remark}
Suppose that $\calY$ is an absolute distributor. In this case, we can identify
$\calX$ with $\SSet$ so that the functor $G: \calY \rightarrow \calX$ is corepresented by
the final object ${\bf 1} \in \calY$. Condition $(b)$ is equivalent to the requirement that
${\bf 1}$ is a compact object of $\calY$.
\end{remark}

\begin{proof}
We first prove $(1)$. Let $\calC$ be a small filtered $\infty$-category, $p: \calC \rightarrow \CSS{ \calX \subseteq \calY}$ a diagram, and $Y_{\bigdot}$ a colimit of $p$ in $\Fun( \Nerve(\cDelta)^{op}, \calY)$. We wish to prove that $Y_{\bigdot}$ is a complete Segal space object of $\calY$. We first invoke
$(a)$ to deduce that $Y_{\bigdot}$ is a category object of $\calY$. Since $Y_0$ is a filtered colimit of objects in $\calX$, we deduce that $Y_0 \in \calX$, so that $Y_{\bigdot} \in \SegS{ \calX \subseteq \calY}$.
To prove that $Y_{\bigdot}$ is complete, we must show that the groupoid $\Gp Y_{\bigdot}$ is constant. Since the collection of constant groupoid objects of $\calX$ is stable under filtered colimits, it will suffice to show that the functor
$$\Gp: \SegS{ \calX \subseteq \calY} \rightarrow \Grp(\calX)$$
preserves filtered colimit. We observe that $\Gp$ can be defined by first composing with
$G$ pointwise (which preserves filtered colimits by virtue of $(b)$) and then applying
a right adjoint $U$ to the inclusion $\Grp(\calX) \subseteq \Cat(\calX)$. It will therefore
suffice to show that $U$ preserves filtered colimits. Since $\Grp(\calX)$ is stable under
filtered colimits in $\calX$ (since filtered colimits in $\calX$ are left exact by Example \toposref{tucka}),
it will suffice to show that for each $n \geq 0$, the functor
$$ \Cat(\calX) \rightarrow \calX$$
$$ X_{\bigdot} \mapsto U(X)_{n}$$
preserves filtered colimits. Since $U(X)_{n} \simeq U(X)_1 \times_{U(X)_0} \ldots \times_{U(X)_0} U(X)_1$, we can reduce to the case where $n \leq 1$. The desired result now follows from Proposition \ref{cavem} (and that fact that filtered colimits in $\calX$ are left exact).

Assertion $(2)$ follows immediately from $(1)$. Assertion $(3)$ follows from $(a)$, since
$\CSS{ \calX \subseteq \calY}$ is stable under filtered colimits and finite limits in
$\Fun( \Nerve(\cDelta)^{op}, \calY)$. Finally, assertion $(4)$ follows from the observation
that $G'$ is the restriction of $\Gp$ to the full subcategory $\CSS{ \calX \subseteq \calY} \subseteq \SegS{\calX \subseteq \calY}$, and the functor $\Gp$ preserves filtered colimits by the above argument.
\end{proof}

\subsection{Higher-Dimensional Complete Segal Spaces}\label{bisec1.3}

In the last section, we saw that to every distributor $\calX \subseteq \calY$, one can
associate a new $\infty$-category $\CSS{\calX \subseteq \calY}$ of complete Segal objects of $\calY$. 
Our first goal in this section is to show that this construction can be iterated: according to 
Proposition \ref{ittaboy}, the inclusion $\calX' \subseteq \CSS{ \calX \subseteq \calY}$
is again a distributor, where $\calX'$ denotes the essential image of the
(fully faithful) diagonal embedding $\calX \rightarrow \Fun( \Nerve(\cDelta)^{op}, \calY)$.
First, we need to establish a preliminary result.

\begin{lemma}\label{stumple}
Let $\calX \subseteq \calY$ be a distributor. Suppose given an effective epimorphism
$\coprod_{v \in V} X_{v} \rightarrow X$ in the $\infty$-topos $\calX$. For each $v \in V$, let
$\phi_v: \calY^{/X} \rightarrow \calY^{/X_v}$ denote the associated pullback functor.
Let $\overline{Y}_{\bigdot}$ be a simplicial object of $\calY^{/X}$ such that, for each $v \in V$, the composite functor
$$ Y^{v}_{\bigdot}: \Nerve(\cDelta)^{op} \stackrel{\overline{Y}_{\bigdot}}{\rightarrow} \calY^{/X} \stackrel{ \phi_v}{\rightarrow}
\calY^{/X_v} \rightarrow \calY$$
is a complete Segal space object of $\calY$. Then the composite functor
$$ Y^{\bigdot}: \Nerve(\cDelta)^{op} \stackrel{ \overline{Y}^{\bigdot} }{\rightarrow} \calY$$
is a complete Segal space object of $\calY$.
\end{lemma}

\begin{proof}
We first show that $Y_{\bigdot}$ is a category object of $\calY$. Choose $n \geq 0$, and consider the canonical map $f: Y_n \rightarrow Y_1 \times_{Y_0} \ldots \times_{Y_0} Y_1$. We wish to show that
$f$ is an equivalence. By assumption, $f$ becomes an equivalence after pullback along each
of the maps $X_{v} \rightarrow X$. It follows that $f$ is an equivalence after pullback along the
effective epimorphism $\coprod_{v} X_v \rightarrow X$, so that $f$ is itself an equivalence by virtue of Lemma \ref{precooke}.

We next show that $Y_{\bigdot}$ is a Segal space object of $\calY$. Let $G: \calY \rightarrow \calX$
denote a right adjoint to the inclusion. For each $v \in V$ and each $n \geq 0$, we have a pullback diagram
$$ \xymatrix{ Y^{v}_n \ar[r] \ar[d] & Y_n \ar[d] \\
X_{v} \ar[r] & X. }$$
Since $G$ preserves pullback diagrams, we obtain a canonical identification
$GY^{v}_n \simeq GY_n \times_{ X} X_v$. Since each $Y^v$ is a Segal space object of
$\calY$, the maps $GY^v_0 \rightarrow Y^v_0$ are equivalences. It follows that the map
$GY_0 \rightarrow Y_0$ becomes an equivalence after pullback along each of the maps
$X_v \rightarrow X$. Arguing as above, we conclude that $GY_0 \rightarrow Y_0$ is an equivalence,
so that $Y_0 \in \calX$ as desired.

It remains to show that $Y_{\bigdot}$ is complete. For each $v \in V$, let $\overline{X}_v$ denote
the constant simplicial object of $\calX$ taking the value $X_v$, and define $\overline{X}$ similarly.
We have a pullback diagram of Segal space objects
$$ \xymatrix{ Y^{v}_{\bigdot} \ar[r] \ar[d] & Y_{\bigdot} \ar[d] \\
\overline{X}_v \ar[r] & \overline{X} }$$
for each $v \in V$. It follows that the induced diagram
$$ \xymatrix{ \Gp Y^{v}_{\bigdot} \ar[r] \ar[d] & \Gp Y_{\bigdot} \ar[d] \\
\Gp \overline{X}_v \ar[r] & \Gp \overline{X} }$$
is a pullback square of groupoid objects in $\calX$. Since each $Y^{v}_{\bigdot}$ is complete, we conclude that the groupoid object $\Gp Y_{\bigdot}$ becomes constant after pullback along each of the maps $X_v \rightarrow X$. It follows that $\Gp Y_{\bigdot}$ becomes constant after pullback along the effective epimorphism $\coprod_{v} X_v \rightarrow X$. The desired result now follows from Lemma \ref{cooke} and Remark \ref{supling}.
\end{proof}

\begin{proposition}\label{ittaboy}
Let $\calX \subseteq \calY$ be a distributor, and let $\calX'$ denote the essential image of the diagonal embedding $\calX \rightarrow \Fun( \Nerve(\cDelta)^{op}, \calX) \subseteq \Fun( \Nerve(\cDelta)^{op}, \calY)$. Then the inclusion $\calX' \subseteq \CSS{ \calX \subseteq \calY }$ is a distributor.
\end{proposition}

\begin{proof}
Since the $\infty$-category $\Nerve(\cDelta)^{op}$ is weakly contractible, the diagonal functor
$\calX \rightarrow \calX'$ is an equivalence of $\infty$-categories. In particular, $\calX'$ is presentable.
Since $\calX \subseteq \calY$ is stable under small limits and colimits, the full subcategory
$\calX'$ is stable under limits and colimits in $\Fun( \Nerve(\cDelta)^{op}, \calY)$, and in particular
in $\CSS{ \calX \subseteq \calY}$. To complete the proof, we will show that 
the inclusion $\calX' \subseteq \CSS{ \calX \subseteq \calY}$ satisfies the conditions of Corollary \ref{summa}. Let $K$ be a small simplicial set and let $\alpha: \overline{p} \rightarrow \overline{q}$
be a natural transformation of diagrams $\overline{p}, \overline{q}: K^{\triangleright} \rightarrow
\CSS{ \calX \subseteq \calY }$ such that $\overline{q}$ is a colimit diagram in $\calX'$ and
the restriction $\alpha = \overline{\alpha} | K$ is a Cartesian transformation. We wish to prove that the following assertions are equivalent:
\begin{itemize}
\item[$(a)$] The natural transformation $\overline{\alpha}$ is Cartesian.
\item[$(b)$] The diagram $\overline{p}$ is a colimit diagram in $\CSS{ \calX \subseteq \calY}$.
\end{itemize}
For every object $[n] \in \Nerve(\cDelta)^{op}$, the induced map $\overline{q}_{[n]}: K^{\triangleright}
\rightarrow \calX$ is a colimit diagram, and the induced transformation $\alpha_{[n]}$ is Cartesian.
Applying Corollary \ref{summa}, we deduce that $(a)$ is equivalent to the following requirement:
\begin{itemize}
\item[$(b')$] For each $[n] \in \Nerve(\cDelta)^{op}$, the map $\overline{p}_{[n]}: K^{\triangleright} \rightarrow \calY$ is a colimit diagram. In other words, $\overline{p}$ is a colimit diagram
in the $\infty$-category $\Fun( \Nerve(\cDelta)^{op}, \calY)$. 
\end{itemize}

It is clear that $(b')$ implies $(b)$. To prove the converse, let $Y_{\bigdot}$ be a colimit of
$p = \overline{p} | K$ in the $\infty$-category $\Fun( \Nerve(\cDelta)^{op}, \calY)$; we must show that
$Y_{\bigdot}$ is a complete Segal space object of $\calY$.

Without loss of generality, we may assume that $\overline{q}$ factors as a composition
$$ K^{\triangleright} \stackrel{ \overline{q}'}{\rightarrow} \calX \stackrel{\overline{q}''}{\rightarrow} \calX',$$
where $\overline{q}''$ is the diagonal map. Let $X \in \calX$ denote the image of the cone point
under $\overline{q}'$. For every vertex $v$ of $K$, let $X_{v}$ denote the image of $v$ under
$\overline{q}'$, so that we have an effective epimorphism $\coprod_{v} X_v \rightarrow X$ in the $\infty$-topos $\calX$. The map $\overline{\alpha}$ determines a lifting of $Y_{\bigdot}$ to a simplicial object
of $\calY^{/X}$. Using Corollary \ref{summa} and the fact that $p$ factors through $\CSS{ \calX \subseteq \calY}$, we deduce that each of the induced diagrams
$$ \Nerve(\cDelta)^{op} \rightarrow \calY^{/X} \rightarrow \calY^{/X_v} \rightarrow \calY$$
is a complete Segal space object of $\calY$. It follows from Lemma \ref{stumple} that
$Y_{\bigdot}$ is a complete Segal space object of $\calY$, as desired.
\end{proof}

In practice, we are primarily interested in the case of distributors $\calX \subseteq \calY$
where $\calY$ is an $\infty$-category of $(\infty,n)$-categories, and $\calX$ is the full subcategory spanned by the $\infty$-groupoids. In this case, $\calX$ is equivalent to the $\infty$-category
of spaces and its inclusion into $\calY$ is uniquely determined. Consequently, Definition \ref{disty} can be rephrased entirely in terms of the ambient category $\calY$:

\begin{definition}\label{custo}
Let $\calY$ be a presentable $\infty$-category. It follows from Theorem \toposref{charpresheaf} that
there exists a functor $F: \SSet \rightarrow \calY$ which preserves small colimits and final objects; moreover, $F$ is uniquely determined up to equivalence. We will say that $\calY$ is an {\it absolute distributor} if the following conditions are satisfied:
\begin{itemize}
\item[$(1)$] The functor $F$ is fully faithful.
\item[$(2)$] The inclusion $\calX \subseteq \calY$ is a distributor, where $\calX$ denotes the essential image of $F$.
\end{itemize}
In this case, we let $\SegS{ \calY }$ and $\CSS{ \calY}$ denote the $\infty$-categories
$\SegS{ \calX \subseteq \calY}$ and $\CSS{ \calX \subseteq \calY}$ of Segal space objects and
complete Segal space objects associated to the distributor $\calX \subseteq \calY$.
\end{definition}


\begin{corollary}\label{cus1}
Let $\calY$ be an absolute distributor. Then the $\infty$-category $\CSS{ \calY}$ of complete Segal space objects of $\calY$ is again an absolute distributor.
\end{corollary}

\begin{example}\label{cus2}
The $\infty$-category $\SSet$ of spaces is an absolute distributor.
\end{example}

\begin{definition}
We define $\infty$-categories $\Cat_{(\infty, n)}$ by induction on $n$ as follows:
\begin{itemize}
\item If $n =0$, we let $\Cat_{(\infty, n)}$ denote the $\infty$-category $\SSet$ of spaces.
\item If $n > 0$, we let $\Cat_{(\infty, n)}$ denote the $\infty$-category
$\CSS{ \Cat_{ (\infty, n-1)}}$ of complete Segal space objects of $\Cat_{ (\infty, n-1)}$. 
\end{itemize}
\end{definition}

\begin{remark}
We will later show that the $\infty$-category $\Cat_{( \infty, 1)}$ is equivalent to the $\infty$-category
$\Cat_{\infty}$ of $\infty$-categories (Corollary \ref{presquare}).
\end{remark}

\begin{variant}
Let $\calX$ be an $\infty$-topos. Then we can define a sequence of distributors
$\calX_n \subseteq \calY_n$ by induction as follows:
\begin{itemize}
\item If $n=0$, we let $\calX_n = \calY_n = \calX$.
\item For $n \geq 0$, we let $\calY_{n+1} = \CSS{ \calX_{n} \subseteq \calY_n }$, and
$\calX_{n+1}$ denote the essential image of the diagonal embedding $\calX_{n} \rightarrow
\calY_{n+1}$.  
\end{itemize}
We can think of $\calY_{n}$ as the $\infty$-category of {\em stacks} of $(\infty,n)$-categories
on $\calX$.
\end{variant}

\subsection{$\Cat_{\infty}$ as an Absolute Distributor}\label{bisec1.4}

Our goal in this section is to prove the following result:

\begin{theorem}\label{cuppo}
The $\infty$-category $\Cat_{\infty}$ is an absolute distributor (see Definition \ref{custo}).
\end{theorem}

It is possible to deduce Theorem \ref{cuppo} by combining
Example \ref{cus2}, Corollary \ref{cus1}, and Corollary \ref{presquare}. However, we will give a direct proof in this section, which will yield some additional dividends. 

We begin by observing that $\SSet$ can be identified with the full subcategory of $\Cat_{\infty}$ spanned by the Kan complexes.

\begin{lemma}\label{succ}
The inclusion $\SSet \subset \Cat_{\infty}$ admits left and right adjoints.
\end{lemma}

\begin{proof}
It will suffice to show that for every $\infty$-category $\calC$, there exist Kan complexes
$X$ and $Y$ and maps
$$ X \stackrel{f}{\rightarrow} \calC \stackrel{g}{\rightarrow} Y$$
with the following properties:
\begin{itemize}
\item For every Kan complex $Z$, composition with $f$ and $g$ induces homotopy equivalences
$$ \bHom_{ \Cat_{\infty} }( Z, X) \rightarrow \bHom_{ \Cat_{\infty}}( Z, \calC)$$
$$ \bHom_{ \Cat_{\infty} }( Y, Z) \rightarrow \bHom_{ \Cat_{\infty} }(\calC, Z).$$
\end{itemize}
These conditions are satisfied if we choose $X$ to be the largest Kan complex contained in
$\calC$, and $g: \calC \rightarrow Y$ to be any weak homotopy equivalence such that $Y$ is a Kan complex.
\end{proof}

Theorem \ref{cuppo} follows immediately from Lemma \ref{succ}, Corollary \ref{summa}, and the following
result:

\begin{proposition}\label{custa0}
Let $\calE$ be a small $\infty$-category, and suppose given a natural transformation
$\overline{\alpha}: \overline{p} \rightarrow \overline{q}$ of diagrams
$\overline{p}, \overline{q}: \calE^{\triangleright} \rightarrow \Cat_{\infty}$. Assume that
$\overline{q}$ is a colimit diagram in $\SSet$, and that $\alpha = \overline{\alpha} | \calE$ is
Cartesian. Then $\overline{\alpha}$ is Cartesian if and only if $\overline{p}$ is a colimit diagram.
\end{proposition} 

We will deduce Proposition \ref{custa0} from a more general result, which does not require
the hypothesis that $\overline{q}$ factors through $\SSet$. To formulate this result, we need to introduce a definition:

\begin{definition}
Let $f: \calC \rightarrow \calD$ be a morphism in $\Cat_{\infty}$. We will say that $f$ is
{\it essentially a coCartesian fibration} if $f$ factors as a composition
$$ \calC \stackrel{f'}{\rightarrow} \calC' \stackrel{f''}{\rightarrow} \calD$$
where $f'$ is an equivalence of $\infty$-categories and $f''$ is a coCartesian fibration.
\end{definition}

\begin{remark}
Let $f: \calC \rightarrow \calD$ be a morphism in $\Cat_{\infty}$, and choose {\em any} factorization
$$ \calC \stackrel{f'}{\rightarrow} \calC' \stackrel{f''}{\rightarrow} \calD$$
where $f'$ is an equivalence of $\infty$-categories and $f''$ is a categorical fibration. Then
$f$ is essentially a coCartesian fibration if and only if $f''$ is a coCartesian fibration.
\end{remark}

\begin{remark}\label{stooper}
Let $f: \calC \rightarrow \calD$ be a morphism in $\Cat_{\infty}$, where $\calD$ is a Kan complex.
Then $f$ is essentially a coCartesian fibration. This follows immediately from Proposition \toposref{groob}.
\end{remark}

In view of Remark \ref{stooper}, Proposition \ref{custa0} is an immediate consequence of the following result:
\begin{proposition}\label{custa1}
Let $\calE$ be a small $\infty$-category, and suppose given a natural transformation
$\overline{\alpha}: \overline{p} \rightarrow \overline{q}$ of diagrams
$\overline{p}, \overline{q}: \calE^{\triangleright} \rightarrow \Cat_{\infty}$. Assume that:
\begin{itemize}
\item[$(a)$] For every object $E \in \calE$, the map $\overline{p}(E) \rightarrow \overline{q}(E)$
is essentially a coCartesian fibration. 
\item[$(b)$] The natural transformation $\alpha = \overline{\alpha} | \calE$ is Cartesian.
\item[$(c)$] The map $\overline{q}$ is a colimit diagram.
\end{itemize}
Then the following conditions are equivalent:
\begin{itemize}
\item[$(1)$] The map $\overline{p}$ is a colimit diagram.
\item[$(2)$] The natural transformation $\overline{\alpha}$ is Cartesian, and the map
$\overline{p}(v) \rightarrow \overline{q}(v)$ is essentially a coCartesian fibration, where
$v$ denotes the cone point of $\calE^{\triangleright}$.
\end{itemize}
\end{proposition} 

We proceed to reduce Proposition \ref{custa1} to a more concrete statement. Let
$q = \overline{q} | \calE$. Then $q$ is classified by a coCartesian fibraton of simplicial sets
$\pi: \calD \rightarrow \calE$. Similarly, the diagram $p = \overline{p} | \calE$ is classified by a coCartesian fibration $\calC \rightarrow \calE$, and the natural transformation $\overline{\alpha}$ is encoded by a commutative diagram
$$ \xymatrix{ \calC \ar[rr]^{r} \ar[dr] & & \calD \ar[dl] \\
& \calE, & }$$
where the functor $r$ preserves coCartesian edges. 

Let $\calD^{\natural}$ denote the marked simplicial set $(\calD,S)$, where
$S$ is the collection of $\pi$-coCartesian edges of $\calD$, and let $\calC^{\natural}$ be defined likewise. In view of Proposition \toposref{charcatcolimit}, we can identify the colimit of $q$ with an $\infty$-category
$\calD'$ equipped with a weak equivalence $\calD^{\natural} \rightarrow {\calD'}^{\natural}$ of marked
simplicial sets; here ${\calD'}^{\natural}$ is the marked simplicial set $(\calD', S')$ where $S'$ denotes the collection of all equivalences in $\calD'$. Similarly, the diagram $\overline{p}$ is encoded by a map
of marked simplicial sets $\calC^{\natural} \rightarrow {\calC'}^{\natural}$ which is a weak equivalence if and only if $\overline{p}$ is a colimit diagram, and the natural transformation $\overline{\alpha}$ is encoded by a commutative diagram
$$ \xymatrix{ \calC' \ar[r] & \calD' \\
\calC \ar[r] \ar[u] & \calD. \ar[u] }$$
We may therefore reformulate Proposition \ref{custa1} as follows:

\begin{proposition}\label{custa2}
Suppose given a commutative diagram of $\infty$-categories:
$$ \xymatrix{ \calC' \ar[rr]^{r'} &  & \calD' \\
\calC \ar[rr]^{r} \ar[u]^{f} \ar[dr]^{p} & & \calD \ar[u]^{f'} \ar[dl]^{q} \\
& \calE & }$$
satisfying the following conditions:
\begin{itemize}
\item[$(a)$] The maps $p$ and $q$ are coCartesian fibrations.
\item[$(b)$] The map $f$ carries $p$-coCartesian edges of $\calC$ to equivalences in $\calC'$, 
and the map $f'$ carries $q$-coCartesian edges of $\calD$ to equivalences in $\calD'$.
\item[$(c)$] For every object $E \in \calE$, the induced map $\calC_{E} \rightarrow \calD_{E}$ is
essentially a coCartesian fibration.
\item[$(d)$] For every morphism $E \rightarrow E'$ in $\calE$, the induced homotopy coherent diagram
$$ \xymatrix{ \calC_{E} \ar[r] \ar[d] & \calD_{E} \ar[d] \\
\calC_{E'} \ar[r] & \calD_{E'} }$$
is a homotopy pullback diagram of $\infty$-categories.
\item[$(e)$] The map $\calD^{\natural} \rightarrow {\calD'}^{\natural}$ is a weak equivalence of marked simplicial sets; here we regard an edge of $\calD$ as marked in $\calD^{\natural}$ if it is $q$-coCartesian, while an edge of $\calD'$ is marked in ${\calD'}^{\natural}$ if it is an equivalence.
\end{itemize}
Then the following conditions are equivalent:
\begin{itemize}
\item[$(1)$] The induced map $\calC^{\natural} \rightarrow { \calC'}^{\natural}$ is a weak equivalence of marked simplicial sets, where $\calC^{\natural}$ and ${ \calC' }^{\natural}$ are defined as in $(e)$.

\item[$(2)$] The map $r'$ is essentially a coCartesian fibration, and for each $E \in \calE$ the diagram
$$ \xymatrix{ \calC_{E} \ar[r] \ar[d] & \calD_{E} \ar[d] \\
\calC' \ar[r] & \calD' }$$
is a homotopy pullback square of $\infty$-categories.
\end{itemize}
\end{proposition}

The proof will require a few preliminary results. We first introduce a bit of terminology:

\begin{definition}
We will say that a map $p: (X, \calE) \rightarrow (Y, \calE')$ of marked simplicial sets
is a {\it marked coCartesian fibration} if the following conditions are satisfied:
\begin{itemize}
\item[$(1)$] The underlying map of simplicial sets $p: X \rightarrow Y$ is a coCartesian fibration.
\item[$(2)$] An edge $f$ of $X$ belongs to $\calE$ if and only if $f$ is $p$-coCartesian and
$p(f) \in \calE'$.
\item[$(3)$] For every marked edge $f: y \rightarrow y'$ in $Y$, the induced map
$X_{y} \rightarrow X_{y'}$ is an equivalence of $\infty$-categories.
\end{itemize}
\end{definition}

\begin{lemma}\label{basty1}
Let $f: (X, \calE) \rightarrow (Y, \calE')$ be a marked coCartesian fibration, and suppose
given an $\infty$-category $\calD$ and a map $h: Y \rightarrow \calD'$ which carries each edge in $\calE$ to an equivalence in $\calD'$. Then there exists a commutative diagram of marked simplicial sets
$$ \xymatrix{ (X, \calE) \ar[r]^{f} \ar[d]^{g'} & (Y, \calE') \ar[dr]^{h} \ar[d]^{g} & \\
\calC^{\natural} \ar[r]^{f'} & \calD^{\natural} \ar[r]^{j} & {\calD'}^{\natural} }$$
where:
\begin{itemize}
\item[$(a)$] The marked simplicial set $\calD^{\natural} = (\calD, \calE_{\calD})$, where $\calE_{\calD}$ denotes the collection of all equivalences in $\calD$ (and ${\calD'}^{\natural}$ is defined likewise).
\item[$(b)$] The marked simplicial set $\calC^{\natural} = (\calC, \calE_{\calC})$, where $\calC$
is an $\infty$-category and $\calE_{\calC}$ consists of the collection of all $f'$-coCartesian morphism in $\calC$. 
\item[$(c)$] The map $f'$ is a coCartesian fibration.
\item[$(d)$] The maps $g$ and $g'$ are weak equivalences of marked simplicial sets.
\item[$(e)$] For every vertex $y$ of $Y$, the induced map $X_y \rightarrow \calC_{g(y)}$ is an equivalence of $\infty$-categories.
\item[$(f)$] The map $j$ is a categorical fibration.
\end{itemize}

In particular, the morphism in $\Cat_{\infty}$ determined by $f$ is essentially a coCartesian fibration.
\end{lemma}

\begin{proof}
Let $K$ be a contractible Kan complex equipped with a monomorphism $\Delta^1 \rightarrow K$ which is bijective on vertices. The map $h$ admits a factorization
$$ Y \stackrel{h'}{\rightarrow} Y \coprod_{\calE \times \Delta^1} (\calE \times K)
\stackrel{h''}{\rightarrow} \calD \stackrel{h'''}{\rightarrow} \calD'$$
where $h''$ is a cofibration and a categorical equivalence, and $h'''$ is an categorical fibration (so that $\calD$ is an $\infty$-category).
The coCartesian fibration $f$ is classified by a map $\chi: Y \rightarrow \Cat_{\infty}$, which carries
every edge in $\calE'$ to an equivalence in $\Cat_{\infty}$. It follows that $\chi$ can be extended to a map
$\overline{\chi}: \calD \rightarrow \Cat_{\infty}$. This map classifies a coCartesian fibration
$f': \calC \rightarrow \calD$. Without loss of generality, we may replace $X \rightarrow Y$ by the equivalent coCartesian fibration $Y \times_{\calD} \calC \rightarrow Y$. We claim that this construction has the desired properties. The only nontrivial point is to verify that the map $\phi:(X, \calE) \rightarrow \calC^{\natural}$ is a weak equivalence of marked simplicial sets. We can factor $\phi$ as a composition
$$ (X, \calE) \stackrel{\phi'}{\rightarrow} \overline{Z} \stackrel{\phi''}{\rightarrow} \calC^{\natural},$$
where $\overline{Z} = (Z, \calE_{Z})$ denotes the marked simplicial set
$$ (Y^{\flat} \coprod_{ (\calE \times \Delta^1)^{\flat}} (\calE \times K)^{\sharp})
\times_{ \calD^{\natural} } \calC^{\natural}. $$
It will therefore suffice to show that $\phi'$ and $\phi''$ are weak equivalences of marked simplicial sets.

The map $\phi'$ is an iterated pushout of inclusions of the form
$$ ( \Delta^1)^{\sharp} \times_{ \calD^{\natural} } \calC^{\natural}
\subseteq K^{\sharp} \times_{ \calD^{\natural} } \calC^{\natural}.$$
Since $K$ is contractible, there exists a retraction $r: K \rightarrow \Delta^1$ and a homotopy
$H: K \times \Delta^1 \rightarrow K$ from the identity to $r$. Choosing a coCartesian lifting of this
homotopy, we obtain map
$$ \overline{H}: (K^{\sharp} \times_{ \calD^{\natural} } \calC^{\natural}) \times (\Delta^1)^{\sharp}
\rightarrow K^{\sharp} \times_{ \calD^{\natural} } \calC^{\natural},$$
which exhibits $( \Delta^1)^{\sharp} \times_{ \calD^{\natural} } \calC^{\natural}$ as a deformation
retract of $K^{\sharp} \times_{ \calD^{\natural} } \calC^{\natural}$ in the category of marked simplicial sets. It follows that $\phi'$ is a weak equivalence.

To prove that $\phi''$ is a weak equivalence, we consider the commutative diagram
$$ \xymatrix{ Z^{\flat} \ar[r]^{\phi'''} \ar[d] & \calC^{\flat} \ar[d] \\
\overline{Z} \ar[r]^{\phi''} & \calC^{\natural}. }$$
Both of the vertical arrows are weak equivalences (since they can be obtained as iterated
pushouts of inclusions of the form $S^{\flat} \subseteq S^{\sharp}$, where $S$ is a Kan complex).
Consequently, it will suffice to show that $Z^{\flat} \rightarrow \calC^{\flat}$ is an equivalence of marked simplicial sets: in other words, that the inclusion $Z \subseteq \calC$ is an equivalence of $\infty$-categories. This follows from Corollary \toposref{basety}, since $f'$ is a coCartesian fibration and
$h'''$ is a categorical equivalence.
\end{proof}

\begin{remark}
In the situation of Lemma \ref{basty1}, we can assume without loss of generality that the map
$X \rightarrow \calC$ is a monomorphism of simplicial sets: if necessary, we can replace
$\calC$ by $\calC \times K$, where $K$ is a contractible Kan complex equipped with a monomorphism $X \rightarrow K$.
\end{remark}

\begin{lemma}\label{quack}
Suppose given a commutative diagram of marked simplicial sets
$$ \xymatrix{ \overline{X} \ar[r]^{p} \ar[d] & \overline{Y} \ar[d]^{q} \\
\overline{X}' \ar[r]^{p'} & \overline{Y}' }$$
where $p$ and $p'$ are marked coCartesian fibrations. This diagram is a
homotopy pullback square of marked simplicial sets if and only if, for every vertex
$y$ of $\overline{Y}$, the induced map of fibers $X_{y} \rightarrow X'_{q(y)}$ is an
equivalence of $\infty$-categories. Here $X$ and $X'$ denote the underlying simplicial
sets of $\overline{X}$ and $\overline{X}'$, respectively.
\end{lemma}

\begin{remark}
The hypothesis of Lemma \ref{quack} is satisfied, in particular, if the diagram
$$ \xymatrix{ \overline{X} \ar[r] \ar[d] & \overline{Y} \ar[d] \\
\overline{X}' \ar[r] & \overline{Y}' }$$
is a pullback square.
\end{remark}

\begin{proof}
Choose a commutative diagram
$$ \xymatrix{ \overline{X}' \ar[d] \ar[r] & \overline{Y}' \ar[d] \ar[dr] & \\
{\calC'}^{\natural} \ar[r] & {\calD'}^{\natural} \ar[r] & (\Delta^0)^{\sharp} }$$
satisfying the conditions of Lemma \ref{basty1}. It will now suffice to prove Lemma \ref{quack} after replacing $\overline{X}'$ by ${\calC'}^{\natural}$ and $\overline{Y}$ by ${ \calD'}^{\natural}$.
Now choose another commutative diagram
$$ \xymatrix{ \overline{X} \ar[d]^{f} \ar[r] & \overline{Y} \ar[d] \ar[dr] & \\
{\calC}^{\natural} \ar[r] & {\calD}^{\natural} \ar[r] & {\calD'}^{\natural} }$$
satisfying the conditions of Lemma \ref{basty1}, such that the map $f$ is a monomorphism.
Then $f$ is a trivial cofibration of marked simplicial sets, so that the mapping problem
depicted in the diagram
$$ \xymatrix{ \overline{X} \ar[d]^{f} \ar[r] & { \calC' }^{\natural} \ar[d] \\
\calC^{\natural} \ar[r] \ar@{-->}[ur] & { \calD' }^{\natural } }$$
admits a solution. Since every object of $\calD$ is equivalent to an object lying in the image of the map
$\overline{Y} \rightarrow \calC^{\natural}$, we can replace $\overline{X}$ by $\calC^{\natural}$ and
$\overline{Y}$ by $\calD^{\natural}$. We are now reduced to proving that the diagram of $\infty$-categories
$$ \xymatrix{ \calC \ar[r] \ar[d] & \calD \ar[d]^{g} \\
\calC' \ar[r] & \calD' }$$
(in which the horizontal maps are coCartesian fibrations) is a homotopy pullback square if and only
if the induced map $\calC_{D} \rightarrow \calC'_{g(D)}$ is an equivalence, for each object
$D \in \calD$. This follows from Corollary \toposref{basety} and Proposition \toposref{apple1}.
\end{proof}


\begin{lemma}\label{collam}
Suppose given a commutative diagram of $\infty$-categories
$$ \xymatrix{ \calC \ar[rr]^{r} \ar[dr]^{p} & & \calD \ar[dl]^{q} \\
& \calE & }$$
where $q$ is a coCartesian fibration. Then $r$ is a coCartesian fibration if and only if the following conditions are satisfied:
\begin{itemize}
\item[$(a)$] The map $p$ is a coCartesian fibration, and the map $r$ is an inner fibration.
\item[$(b)$] For each object $E \in \calE$, the induced map $r_{E}: \calC_{E} \rightarrow \calD_{E}$ is
a coCartesian fibration.
\item[$(c)$] For every morphism $E \rightarrow E'$ in $\calE$, the induced functor
$\calC_{E} \rightarrow \calC_{E'}$ carries $r_{E}$-coCartesian morphisms to $r_{E'}$-coCartesian morphisms.
\end{itemize}
\end{lemma}

\begin{proof}
It is clear that if $r$ is a coCartesian fibration, then conditions $(a)$ and $(b)$ are satisfied.
We may therefore assume that $(a)$ and $(b)$ hold. Invoking Proposition \toposref{fibertest}, we conclude that $r$ is a locally coCartesian fibration. Moreover, a morphism $f: X \rightarrow Y$ in $\calC$ is locally $r$-coCartesian if and only if it admits a factorization 
$$X \stackrel{f'}{\rightarrow} Z \stackrel{f''}{\rightarrow} Y$$
where $f'$ is a $p$-coCartesian morphism of $\calC$, and $f''$ is an $r_{E}$-coCartesian morphism
of $\calC_{E}$ for some $E \in \calE$. According to Proposition \toposref{gotta}, the map $r$ is a
coCartesian fibration if and only if the collection of locally $r$-coCartesian morphisms in $\calC$ is stable under composition. Since the collection of locally $r$-coCartesian morphisms is obviously stable under composition on the right by $p$-coCartesian morphisms and under composition on the left by
$r_{E}$-coCartesian morphisms for each $E \in \calE$, we see that $r$ is a coCartesian fibration
if and only if the following condition is satisfied:
\begin{itemize}
\item[$(\ast)$] Suppose given a pair of morphisms $X \stackrel{f}{\rightarrow} Y \stackrel{g}{\rightarrow} Z$
in $\calC$, where $f$ is an $r_E$-coCartesian morphism in $\calC_{E}$ for some $E \in \calE$, and 
the map $g$ is $p$-coCartesian. Then $g \circ f$ is locally $r$-coCartesian.
\end{itemize}
Let $g': E \rightarrow E'$ be the morphism in $\calE$ induced by $g$, and choose a
$p$-coCartesian morphism $X \rightarrow X'$ lifting $X$. We obtain a commutative diagram
$$ \xymatrix{ X \ar[d] \ar[r]^{f} & Y \ar[d]^{g} \\
X' \ar[r]^{f'} & Z. }$$
covering the diagram
$$ \xymatrix{ E \ar[r] \ar[d] & E \ar[d] \\
E' \ar[r] & E' }$$
in $\calE$. Note that $g \circ f$ is locally $r$-coCartesian if and only if $f'$ is 
an $r_{E'}$-coCartesian morphism in $\calC_{E'}$, and that $f'$ is the image of $f$ under
the associated functor $\calC_{E} \rightarrow \calC_{E'}$. The condition that this holds
for {\em every} morphism $f: X \rightarrow Y$ in $\calC_{E}$ and every map $p(Y) = E \rightarrow E'$
in $\calE$ (which then admits an essentially unique $p$-coCartesian lifting $g: Y \rightarrow Z$) is
manifestly equivalent to $(c)$.
\end{proof}

\begin{proof}[Proof of Proposition \ref{custa2}]
Without loss of generality, we may assume that $r$ is a categorical fibration. It follows that for every object $E \in \calE$, the induced map $r_E: \calC_{E} \rightarrow \calD_{E}$ is a categorical fibration which is essentially a coCartesian fibration, and therefore a coCartesian fibration. In view of Proposition \toposref{basechangefunky}, condition $(d)$ of Proposition \ref{custa2} is equivalent to the requirement that
every morphism $E \rightarrow E'$ in $\calE$ induces an equivalence of $\infty$-categories
$$ \phi: \calC_{E} \rightarrow \calD_{E} \times_{ \calD_{E'} } \calC_{E'}.$$
This implies in particular that the induced map $\calC_{E} \rightarrow \calC_{E'}$ carries
$r_{E}$-coCartesian morphisms to $r_{E'}$-coCartesian morphisms. Invoking Lemma \ref{collam}, we deduce that $r$ is a coCartesian fibration.

Since $\phi$ is a categorical equivalence of coCartesian fibrations over $\calD_{E}$, it induces
an equivalence after passing to the fibers over any vertex $D \in \calD_{E}$. Unwinding the
definitions, we conclude that for every $q$-coCartesian morphism $D \rightarrow D'$, the induced map $\calC_{D} \rightarrow \calC_{D'}$ is an equivalence of $\infty$-categories. Suppose first that
$(1)$ is satisfied; we will prove $(2)$. The assertion that $r'$ is essentially a coCartesian fibration follows immediately from Lemma \ref{basty1}. For the second assertion, we consider the commutative diagram
of marked simplicial sets
$$ \xymatrix{ \calC^{\natural}_{E} \ar[r] \ar[d] & \calD^{\natural}_{E} \ar[d] \\
\calC^{\natural} \ar[r] \ar[d] & \calD^{\natural} \ar[d] \\
{\calC'}^{\natural} \ar[r] & { \calD'}^{\natural}. }$$
The lower square is a homotopy pullback because the lower vertical maps are weak equivalences, and
the upper square is a homotopy pullback by Lemma \ref{quack}. It follows that the outer square
is a homotopy pullback diagram, so that
$$ \xymatrix{ \calC_{E} \ar[r] \ar[d] & \calD_{E} \ar[d] \\
\calC' \ar[r] & \calD' }$$
is a homotopy pullback diagram of $\infty$-categories.

Let us now prove that $(2) \Rightarrow (1)$. Without loss of generality that $r'$ is a categorical fibration. 
Condition $(2)$ then guarantees that $r'$ is a coCartesian fibration, and (by virtue of Lemma \ref{quack}) that for each $E \in \calE$, the induced map
$$ \calC_{E} \rightarrow \calD_{E} \times_{ \calD' } \calC'$$
induces an equivalence of $\infty$-categories after passing to the fiber over any vertex of $\calD_{E}$.
It follows that the map $\calC \rightarrow \calD \times_{ \calD' } \calC'$ induces an equivalence of
$\infty$-categories after passing to the fiber over any vertex of $\calD$, so that (by Lemma \ref{quack}) the diagram of marked simplicial sets
$$ \xymatrix{ \calC^{\natural} \ar[r] \ar[d] & \calD^{\natural} \ar[d] \\
{ \calC' }^{\natural} \ar[r] & { \calD' }^{\natural }}$$
is a homotopy pullback square. Since the right vertical map is a weak equivalence, we conclude that the left vertical map is also a weak equivalence, as desired.
\end{proof}

\begin{remark}\label{postsmashet}
The inclusion $\SSet \subseteq \Cat_{\infty}$ satisfies the hypotheses of Proposition \ref{presmashet}. That is:
\begin{itemize}
\item[$(a)$] Filtered colimits in $\Cat_{\infty}$ are left exact.
\item[$(b)$] Let $G: \Cat_{\infty} \rightarrow \SSet$ denote a right adjoint to the inclusion. Then
$G$ commutes with filtered colimits.
\end{itemize}
Assertion $(b)$ follows from the observation that the functor which assigns to each $\infty$-category $\calC$ the largest Kan complex contained in $\calC$ commutes with filtered colimits.
To prove $(a)$, it will suffice to show that the collection of pullback diagrams
in $\Cat_{\infty}$ is stable under small filtered colimits. In other words, we must show that
if $\calJ$ is a small filtered $\infty$-category and $F: \calJ^{\triangleright} \times (\Delta^1 \times \Delta^1) \rightarrow \Cat_{\infty}$ is a diagram with the property that 
$F | \calJ^{\triangleright} \times \{v\}$ is a colimit diagram for each vertex $v$ of $\Delta^1 \times \Delta^1$, and $F | \{J\} \times \Delta^1 \times \Delta^1$ is a pullback square for each $J \in \calJ$, then
$F | \{\infty\} \times \Delta^1 \times \Delta^1$ is again a pullback square, where $\infty$ denotes the cone point of $\calJ^{\triangleright}$. Without loss of generality, we may suppose that $\calJ$ is the
nerve of a filtered partially ordered set $A$ (Proposition \toposref{rot}), and that
$F$ is induced by an injectively fibrant diagram $(A \cup \{\infty\}) \times [1] \times [1] \rightarrow \mSet$
(Proposition \toposref{gumby444}). Let $G$ denote the diagram obtained by composing with
the forgetful functor $\mSet \rightarrow \sSet$, which is a right Quillen equivalence if we endow
$\sSet$ with the Joyal model structure. We will regard $G$ as giving a commutative square
$$ \xymatrix{ X \ar[r] \ar[d] & X' \ar[d] \\
Y \ar[r] & Y' }$$
of functors from $A \cup \{ \infty \}$ into the category of simplicial sets. Note that each
$Y'(a)$ is an $\infty$-category, and each of the maps $X'(a) \rightarrow Y'(a) \leftarrow Y(a)$ is
a categorical fibration. It follows that the map $X(a) \rightarrow X'(a) \times_{ Y'(a)} Y(a)$ is
a categorical equivalence for $a \in A$, so that the induced map
$$ \colim_{a \in A} X(a) \rightarrow \colim_{a \in A} (X'(a) \times_{ Y'(a)} Y(a) )
\simeq (\colim_{a \in A} X'(a)) \times_{ \colim_{a \in A} Y'(a) } ( \colim_{a \in A} Y(a) )$$
is again a weak equivalence. Using Corollary \toposref{gottaput}, we deduce that
$\colim_{a \in A} Y'(a)$ is an $\infty$-category and the maps
$$ \colim_{a \in A} X'(a) \rightarrow \colim_{a \in A} Y'(a) \leftarrow \colim_{a \in A} Y(a)$$
are categorical fibrations, so that the diagram
$$ \xymatrix{ \colim_{a \in A} X(a) \ar[r] \ar[d] & \colim_{a \in A} X'(a) \ar[d] \\
\colim_{a \in A} Y(a) \ar[r] & \colim_{a \in A} Y'(a) }$$
is a homotopy pullback square. It follows that the weakly equivalent diagram
$$ \xymatrix{ X(\infty) \ar[r] \ar[d] & X'(\infty) \ar[d] \\
Y(\infty) \ar[r] & Y'(\infty) }$$
is also a homotopy pullback square in $\SSet$, as desired.
\end{remark}

\subsection{Models for Complete Segal Spaces}\label{bisec1.5}

In \S \ref{bisec1.3}, we defined the $\infty$-category
$\CSS{\calY}$ of complete Segal objects of $\calY$, where $\calY$
is any $\infty$-category which is an absolute distributor. In this section, we will show that
if $\calY$ can be realized as the underlying $\infty$-category of a well-behaved model category, then
$\CSS{\calY}$ has the same property. Applying this in the case $\calY = \Cat_{\infty}$
(which is an absolute distributor by virtue of Theorem \ref{cuppo}), we will obtain some of
the models for the theory of $(\infty,2)$-categories which appear in the statement of
Theorem \ref{toothygrin}.

\begin{definition}\label{kayloo}
We will say that a simplicial model category $\bfA$ is an {\it absolute distributor} if the following
conditions are satisfied:
\begin{itemize}
\item[$(1)$] The model category $\bfA$ is combinatorial and left proper.
\item[$(2)$] The underlying $\infty$-category $\Nerve( \bfA^{\degree})$ is an
absolute distributor, in the sense of Definition \ref{custo}.
\end{itemize}
\end{definition}

\begin{example}
The category of simplicial sets is an absolute distributor when equipped with the Kan model structure.
\end{example}

\begin{example}
The category $\mSet$ of marked simplicial sets is an absolute distributor; this follows from
Theorem \ref{cuppo}.
\end{example}

\begin{proposition}\label{camper}
Let $\bfA$ be a simplicial model category which is an absolute distributor. Then category
$\Fun( \Delta^{op}, \bfA)$ admits a simplicial model structure which is characterized by the following properties:
\begin{itemize}
\item[$(C)$] A morphism $f: X_{\bigdot} \rightarrow Y_{\bigdot}$ of simplicial objects of
$\bfA$ is a cofibration if each of the maps $X_{n} \rightarrow Y_{n}$ is a cofibration; that is,
if and only if $f$ is an injective cofibration.

\item[$(W)$] Let $f: X_{\bigdot} \rightarrow Y_{\bigdot}$ be a morphism between simplicial objects of $\bfA$. If the induced map $X_{n} \rightarrow Y_{n}$ is a weak equivalence in $\bfA$ for each $n \geq 0$, then $f$ is a weak equivalence. 
\end{itemize}

(In other words, the model structure on $\Fun( \Delta^{op}, \bfA)$
is a localization of the injective model structure.)

\begin{itemize}
\item[$(F)$] A simplicial object $X_{\bigdot}$ of $\bfA$ is fibrant if and only if
it is injectively fibrant, and the induced diagram $\Nerve( \cDelta)^{op} \rightarrow \Nerve( \bfA^{\degree} )$ (which is well-defined up to equivalence) is a complete Segal space object
of $\Nerve(\bfA^{\degree})$. 
\end{itemize}

Moreover, this model structure on $\Fun( \Delta^{op}, \bfA)$ is again an absolute distributor.
\end{proposition}

\begin{proof}
Let $\bfB = \Fun( \cDelta^{op}, \bfA)$, endowed with the injective model structure. Using
Proposition \toposref{gumby444}, we deduce that the underlying $\infty$-category
$\Nerve( \bfB^{\degree} )$ is canonically equivalent to $\Fun( \Nerve(\cDelta)^{op}, \bfA)$.
The desired result now follows from Proposition \toposref{cabber} and Remark \ref{sluther},
and the final assertion follows from Corollary \ref{cus1}.
\end{proof}

We will refer to the model structure of Proposition \ref{camper} as the {\it complete Segal model structure}. 

\begin{remark}\label{sweetuma}
Let $\bfA$ be as in Proposition \ref{camper}, and let $X_{\bigdot}$ be a fibrant object of
$\Fun( \cDelta^{op}, \bfA)$. In particular, $X_{\bigdot}$ is an injectively fibrant diagram. It follows that:
\begin{itemize}
\item[$(a)$] The object $X_0 \in \bfA$ is fibrant.
\item[$(b)$] The map $X_1 \rightarrow X_0 \times X_0$ is a fibration.
\item[$(c)$] Since $X_{\bigdot}$ determines a category object in the underlying $\infty$-category
$\Nerve(\bfA^{\degree})$, the map $X_{n} \rightarrow X_1 \times_{X_0} \ldots \times_{X_0} X_1$
exhibits $X_{n}$ as a homotopy limit of the diagram
$$ \xymatrix{ & X_1 \ar[dl] \ar[dr] & & \ar[dl] \ldots \ar[dr]  & & X_1 \ar[dl] \ar[dr] & \\
X_0 & & X_0 & \ldots & X_0 & & X_0. }$$
In view of $(a)$ and $(b)$, this homotopy limit coincides with the usual limit, so the map
$f: X_{n} \rightarrow X_1 \times_{X_0} \ldots \times_{X_0} X_1$ is a weak equivalence.
\item[$(d)$] Since $X_{\bigdot}$ is injectively fibrant, the map $f$ is fibration. It is therefore
a trivial fibration in $\bfA$.
\end{itemize}
\end{remark}

\begin{remark}
We have defined the complete Segal model structure on $\Fun( \cDelta^{op}, \bfA)$ as a localization of the injective model structure. However, we could just as well have begun with the projective or Reedy model structures on $\Fun( \cDelta^{op}, \bfA)$. After an appropriate localization, we would obtain a model category which is Quillen equivalent to the one described in Proposition \ref{camper}.
\end{remark}

We can combine Proposition \ref{camper} with the Quillen equivalence of
Proposition \toposref{kudd} to obtain another model for the $\infty$-category
$\CSS{ \Cat_{\infty}}$:

\begin{proposition}\label{camperr}
The category $\mset{ \Nerve( \cDelta)^{op} }$ admits a simplicial model structure, which is
characterized by the following properties:
\begin{itemize}
\item[$(C)$] A morphism $\overline{X} \rightarrow \overline{Y}$ in $\mset{ \Nerve(\cDelta)^{op} }$
is a cofibration if and only if the underlying map of simplicial sets is a monomorphism.

\item[$(W)$] Every coCartesian equivalence in $\mset{ \Nerve(\cDelta)^{op} }$ is a
weak equivalence.
\end{itemize}
$($In other words, the model structure on $\mset{ \Nerve(\cDelta)^{op} }$ is a localization
of the coCartesian model structure.$)$
\begin{itemize}
\item[$(F)$] An object $(X, \calE)$ of $\mset{ \Nerve(\cDelta)^{op} }$ is
fibrant if and only if the map $p: X \rightarrow \Nerve(\cDelta)^{op}$ is a coCartesian fibration,
$\calE$ is the set of all $p$-coCartesian edges of $X$, and $p$ is classified by a map
$\Nerve(\cDelta)^{op} \rightarrow \Cat_{\infty}$ which is a complete Segal space object
of $\Cat_{\infty}$.
\end{itemize}
Moreover, the adjoint functors
$$ \Adjoint{ \lNerve^{+}_{\bigdot}( \cDelta^{op}) }{ \mset{ \Nerve(\cDelta)^{op} }}{
\Fun( \cDelta^{op}, \mSet)}{\Nerve_{\bigdot}^{+}( \cDelta^{op})}$$
determine a Quillen equivalence of the category $\mset{ \Nerve(\cDelta)^{op} }$ 
$($with the model structure described above$)$ with the category $\Fun( \cDelta^{op}, \mSet)$
$($with the complete Segal model structure of Proposition \ref{camper}$)$.
In particular, $\mset{ \Nerve(\cDelta)^{op} }$ is an absolute distributor.
\end{proposition}

\begin{proof}
Let $\bfA$ denote the category $\mset{ \Nerve(\cDelta)^{op} }$ endowed with the
coCartesian model structure, and $\bfB$ the category $\Fun( \cDelta^{op}, \mSet)$ endowed
with the injective model structure. Proposition \toposref{kudd} implies that
the relative nerve functor $f \mapsto \Nerve_{f}^{+}( \cDelta^{op} )$ determines a right
Quillen equivalence from $\bfB$ to $\bfA$. Note that $\Nerve^{+}_{f}$ has the structure of a simplicial functor. It follows from Corollary \toposref{urchug} that the induced map
$\Nerve( \bfB^{\degree} ) \rightarrow \Nerve( \bfA^{\degree} )$ is an equivalence of
$\infty$-categories. In view of Proposition \toposref{gumby444}, we conclude that
$\Nerve( \bfB^{\degree} )$ is equivalent to $\Fun( \Nerve(\cDelta)^{op}, \Cat_{\infty} )$. 
The existence of the model structure in question now follows from
Proposition \toposref{cabber}, Remark \ref{sluther}, and Theorem \ref{cuppo}. 

It remains only to prove that the adjoint functors 
$\lNerve^{+}_{\bigdot}( \cDelta^{op})$ and $\Nerve^{+}_{\bigdot}( \cDelta^{op})$ determine
a Quillen equivalence. It is easy to verify that these adjoint functors satisfy the hypotheses of Lemma \ref{speed}, and therefore determine a Quillen adjunction. We conclude by observing that the right derived functor of $\Nerve^{+}_{\bigdot}( \cDelta^{op} )$ determines an equivalence from the homotopy
category $\h{\bfB}$ to the homotopy category $\h{\bfA}$, which restricts to an equivalence
of categories between $\h{ \Fun( \cDelta^{op}, \mSet)} \subseteq \h{\bfB}$ to
$\h{ \mset{ \Nerve(\cDelta)^{op} }} \subseteq \h{\bfA}$. 
\end{proof}

We can use the same reasoning to obtain the following simpler result:

\begin{proposition}\label{campe}
The category $\sSet_{/\Nerve( \cDelta)^{op}}$ admits a simplicial model structure, which is
characterized by the following properties:
\begin{itemize}
\item[$(C)$] A morphism $X \rightarrow Y$ in $\sSet_{/\Nerve( \cDelta)^{op}}$
is a cofibration if and only if the underlying map of simplicial sets is a monomorphism.

\item[$(W)$] Every covariant equivalence in $\sSet_{/\Nerve( \cDelta)^{op}}$ is a
weak equivalence.
\end{itemize}
$($In other words, the model structure on $\sSet_{/\Nerve( \cDelta)^{op}}$ is a localization
of the covariant model structure.$)$
\begin{itemize}
\item[$(F)$] An object $X$ of $\sSet_{/\Nerve( \cDelta)^{op}}$ is
fibrant if and only if the map $X \rightarrow \Nerve(\cDelta)^{op}$ is a left fibration,
classified by a complete Segal object $\chi: \Nerve(\cDelta)^{op} \rightarrow \SSet$
of the $\infty$-category $\SSet$ of spaces.
\end{itemize}
Moreover, the adjoint functors
$$ \Adjoint{ \lNerve_{\bigdot}( \cDelta^{op}) }{ (\sSet)_{/ \Nerve(\cDelta)^{op}} }{
\Fun( \cDelta^{op}, \sSet)}{\Nerve_{\bigdot}( \cDelta^{op})}$$
determine a Quillen equivalence of the category $\sSet_{/\Nerve( \cDelta)^{op}}$
$($with the model structure described above$)$ with the category $\Fun( \cDelta^{op}, \sSet)$
$($with the complete Segal model structure of Proposition \ref{camper}$)$.
In particular, $(\sSet)_{/ \Nerve(\cDelta)^{op} }$ is an absolute distributor.
\end{proposition}

We will abuse terminology by referring to the model structures of Propositions \ref{camperr} and \ref{campe} as the {\it complete Segal model structures}.

\begin{remark}\label{cooperr}
Since the inclusion $\SSet \subseteq \Cat_{\infty}$ carries complete Segal space objects of
$\SSet$ to complete Segal space objects of $\Cat_{\infty}$, Proposition \toposref{stake} implies
that the forgetful functor $\mset{ \Nerve(\cDelta)^{op} } \rightarrow (\sSet)_{/ \Nerve(\cDelta)^{op} }$ is a left Quillen functor (with respect to the complete Segal model structures).
\end{remark}

We conclude this section with a technical result about the behavior of the complete
Segal model categories of Proposition \ref{camper}:

\begin{proposition}\label{smashet}
Let $\bfA$ be a simplicial model category satisfying the following conditions:
\begin{itemize}
\item[$(a)$] The model category $\bfA$ is an absolute distributor.
\item[$(b)$] The collection of weak equivalences in $\bfA$ is stable under filtered colimits.
\item[$(c)$] Filtered colimits are left exact in the underlying $\infty$-category $\Nerve( \bfA)^{\degree}$.
\item[$(d)$] The final object of $\Nerve( \bfA)^{\degree}$ is compact. 
\end{itemize}
Then the collection of weak equivalences in $\Fun( \cDelta^{op}, \bfA)$
$($with respect to the complete Segal model structure$)$ is stable under small filtered colimits.
\end{proposition}

\begin{proof}
Let $\bfB$ denote the category $\Fun( \cDelta^{op}, \bfA)$ endowed with the
injective model structure. Let $\calJ$ be a small filtered category, and let $\alpha: X \rightarrow Y$ be a natural transformation of functors $X,Y: \calJ \rightarrow \bfA$ such that, for each
$J \in \calJ$, the map $X(J) \rightarrow Y(J)$ is a weak equivalence with respect
to the complete Segal model structure. Choose a diagram
$$ \xymatrix{ X \ar[r] \ar[d] & Y \ar[d] \\
X' \ar[r] & Y' \\
X'' \ar[u] \ar[r] & Y'' \ar[u] }$$
where the vertical morphisms are weak equivalences with respect to the projective model
structure on $\Fun( \calJ, \bfB)$, and the objects $X'', Y'' \in \Fun( \calJ, \bfB)$ are
projectively fibrant and cofibrant. Assumption $(b)$ guarantees that the collection of
weak equivalences for the projective model structure is stable under filtered colimits. We may therefore replace $X$ by $X''$ and $Y$ by $Y''$, and thereby reduce to the case where
$X$ and $Y$ are projectively fibrant and cofibrant. In this case, the colimits of $X$ and $Y$ can be identified with their homotopy colimits, which are also (by virtue of Theorem \toposref{colimcomparee})
the colimits of the induced diagrams $\overline{X}, \overline{Y}: \Nerve(\calJ) \rightarrow 
\Nerve(\bfB^{\degree})$.

Let $L: \Nerve( \bfB^{\degree}) \rightarrow \CSS{ \Nerve(\bfA^{\degree})}$ denote
a composition of a right adjoint to the inclusion $\CSS{ \Nerve(\bfA^{\degree})} 
\subseteq \Fun( \Nerve(\cDelta)^{op}, \Nerve(\bfA^{\degree}))$ 
with the equivalence $\Fun( \Nerve(\cDelta)^{op}, \Nerve(\bfA^{\degree})) \simeq \Nerve( \bfB^{\degree} )$. By assumption, the map $\overline{X}(J) \rightarrow \overline{Y}(J)$ becomes an
equivalence in $\CSS{ \Nerve(\bfA^{\degree})}$ after applying the functor $L$.
Since $L$ commutes with filtered colimits (this follows from $(c)$ and $(d)$ by virtue of Proposition \ref{presmashet}), we deduce that the induced map $L \colim \overline{X} \rightarrow L \colim \overline{Y}$ is an equivalence, so that the map $\colim X \rightarrow \colim Y$ is a weak equivalence with respect to the complete Segal model structure as desired.
\end{proof}

\begin{example}\label{smashett}
The conditions of Proposition \ref{smashet} are satisfied if $\bfA$ is the category of simplicial sets
(endowed with the Kan model structure), or if $\bfA$ is the category of marked simplicial sets.
\end{example}

\section{Segal Categories}\label{bisecB}

In \S \ref{bisecA}, we introduced the notion of a {\it Segal space}. To every
Segal space $X_{\bigdot}$, one can associate an $(\infty,1)$-category
$\calC$ whose objects are the points of $X_0$. The resulting correspondence between
Segal spaces and $(\infty,1)$-categories is not one-to-one. In general, a given
$(\infty,1)$-category can be obtained in this way from many different Segal spaces.
One way to eliminate this ambiguity is to restrict attention to the {\em complete} Segal spaces:
in other words, to require that $X_0$ be as ``large'' as possible. In this section, we will
adopt another approach, where we require instead that $X_0$ be very ``small''. More precisely, we will consider Segal spaces $X_{\bigdot}$ with the property that the set $X_0$ is {\em discrete}. A Segal
space with this property is called a {\it Segal category}. 

The theory of Segal categories is another approach to the foundations of higher category theory.
It has a number of advantages and disadvantages when compared with the theory of
complete Segal spaces:
\begin{itemize}
\item[$(1)$] A Segal category $X_{\bigdot}$ has an underlying {\em set} of objects
$X_0$. This makes it possible to directly compare the notion of a Segal category with a more
naive approach to higher category theory, like the theory of simplicial or topological
categories.
\item[$(2)$] Let $X_{\bigdot}$ be a Segal category, and let $S = X_0$ be the set of objects of $X_{\bigdot}$. For each $n \geq 0$, we have a canonical map $p: X_{n} \rightarrow S^{n+1}$, which
determines a decomposition 
$$ X_{n} \simeq \coprod_{ s_0, \ldots, s_n \in S} X[s_0, \ldots, s_n]$$
where $X[s_0, \ldots, s_n]$ denotes the fiber of $p$ over $(s_0, \ldots, s_n)$.
In practice, it is usually easier to work directly with the summands
$X[ s_0, \ldots, s_n]$ of $X_{n}$ than it is to work with $X_{n}$ itself. This is somewhat
advantageous when we work with Segal categories {\em enriched} over a model category
category $\bfA$: we will not need to assume that $\bfA$ is an absolute distributor.

\item[$(3)$] Restricting our attention to Segal spaces $X_{\bigdot}$ such that $X_0$ is discrete
does not fully eliminate the problem that a given $(\infty,1)$-category can be represented
by many different Segal spaces. This is a phenomenon which is apparent even at the level
of ordinary categories: the construction which assigns to each category $\calC$ its underlying
{\em set} of objects is not invariant under equivalence. A practical consequence of this phenomenon is that it is not easy describe the weak equivalences between Segal categories directly.
\end{itemize}

Our goal in this section is to outline the theory of Segal categories and its relationship to other models
for higher category theory. We will begin in \S \ref{bisec2.1} by defining the notion of a
{\it $\bfA$-enriched preSegal category}, for any category $\bfA$. The collection of 
$\bfA$-enriched preSegal categories can be organized into a category $\PreSeg{\bfA}$.
The idea is that if $\bfA$ is a sufficiently nice model for the theory of $(\infty,n-1)$-categories, then 
$\PreSeg{\bfA}$ will be a model for the theory of $(\infty,n)$-categories. For example, in the
case $n=1$ we can take $\bfA$ to be the category of simplicial sets; in the case
$n=2$ (the primary case of interest to us in this paper) we can take $\bfA$ to be the category
of marked simplicial sets.

For any category $\bfA$, there is a fully faithful embedding $i$ from the category
$\Cat_{\bfA}$ of $\bfA$-enriched categories to the category $\PreSeg{\bfA}$ of $\bfA$-enriched
preSegal categories. In \S \ref{bisec2.2}, we endow $\PreSeg{\bfA}$ with a {\it projective} model structure and show that the functor $i$ is a right Quillen equivalence, provided that
$\bfA$ satisfies some mild hypotheses.

In \S \ref{bisec2.3}, we will compare the theory Segal categories with the theory of
Segal spaces. More specifically, we will describe a functor from the category
$\PreSeg{\bfA}$ of $\bfA$-enriched preSegal categories to the category of
$\Fun( \cDelta^{op}, \bfA)$ of simplicial objects of $\bfA$ (in many cases of interest, such
as the case where $\bfA = \sSet$ or $\bfA = \mSet$, this functor is a fully faithful embedding).
Under suitable hypotheses, we will endow $\PreSeg{\bfA}$ with an {\it injective} model structure and prove that the functor $\PreSeg{\bfA} \rightarrow \Fun( \cDelta^{op}, \bfA)$ is a left Quillen equivalence, where $\Fun( \cDelta^{op}, \bfA)$ is endowed with the complete Segal model structure
of Proposition \ref{camper}. We will also show that the projective and injective model
structures on $\PreSeg{\bfA}$ are Quillen equivalent, provided that both are defined.

\subsection{Basic Definitions}\label{bisec2.1}

In this section, we will introduce the notion of a $\bfA$-enriched preSegal category, where
$\bfA$ is another category. We begin by establishing some terminology.

\begin{definition}
Let $S$ be a set. We define a category $\cDelta_{S}$ as follows:
\begin{itemize}
\item An object of $\cDelta_{S}$ is an object $[n] \in \cDelta$ together with a map of sets
$$ c: [n] = \{ 0, \ldots, n \} \rightarrow S.$$ 
\item Given a pair of objects $([n], c)$ and $([n'], c')$ in $\cDelta_{S}$, a morphism from
$([n], c)$ to $([n'],c')$ is a map of linearly ordered sets $f: [n] \rightarrow [n']$ such that
$c = c' \circ f$.
\item Composition of morphisms is defined in the obvious way.
\end{itemize}
\end{definition}

\begin{notation}
Given an $(n+1)$-tuple of elements $( s_0, \ldots, s_n) \in S^{n+1}$, we let
$[ s_0, s_1, \ldots, s_n ]$ denote the object $( [n], c) \in \cDelta_{S}$, where
$c: [n] \rightarrow S$ is the map $i \mapsto s_i$.
\end{notation}

\begin{definition}
Let $\bfA$ be a category. A {\it $\bfA$-enriched preSegal category} consists of the following data:
\begin{itemize}
\item A set $S$, called the {\it set of objects}.
\item A functor $X: \cDelta_{S}^{op} \rightarrow \bfA$ such that, for every element
$s \in S$, the value $X( [s] )$ is a final object of $\bfA$.
\end{itemize}
Given preSegal categories $(S, X: \cDelta_{S}^{op} \rightarrow \bfA)$ and
$(S', X': \cDelta_{S}^{op} \rightarrow \bfA)$, a map of preSegal categories from
$(S,X)$ to $(S',X')$ is a pair $(f, \alpha)$, where $f: S \rightarrow S'$ is a map of sets and
$\alpha$ is a natural transformation from $X$ to the composite functor
$$ \cDelta^{op}_{S} \stackrel{f \circ }{\rightarrow} \cDelta_{S'} \stackrel{X'}{\rightarrow} \bfA.$$
We let $\PreSeg{\bfA}$ denote the category of $\bfA$-enriched preSegal categories.
\end{definition}

\begin{example}\label{sabb}
Let $\bfA$ be a category which admits finite products, and regard
$\bfA$ as endowed with the Cartesian monoidal structure. Then
every $\bfA$-enriched category $\calC$ determines a 
$\bfA$-enriched preSegal category $(S,X)$ as follows:
\begin{itemize}
\item We take $S$ to be the set of objects of $\calC$.
\item For every sequence of elements $s_0, \ldots, s_n \in S$, let
$X[ s_0, \ldots, s_n ] = \Hom_{\calC}(s_0, s_1) \times \ldots
\times \Hom_{\calC}(s_{n-1}, s_n)$.
\end{itemize}
This construction determines a fully faithful embedding from
the category $\Cat_{\bfA}$ of $\bfA$-enriched categories to the
category $\PreSeg{\bfA}$. The essential image of this embedding
consists of those $\bfA$-enriched preSegal categories
$(S,X)$ with the following additional property:
\begin{itemize}
\item[$(\ast)$] For
every sequence of objects $s_0, \ldots, s_n \in S$, the
map
$$ X[s_0, s_1, \ldots, s_n] \rightarrow X[s_0, s_1] \times \ldots \times X[s_{n-1},s_n]$$
is an isomorphism in $\bfA$.
\end{itemize}
\end{example}

In view of Example \ref{sabb}, we can regard the notion of a $\bfA$-enriched preSegal
category as a {\it generalization} of the notion of a $\bfA$-enriched category. To ensure that
this generalization is not too drastic, it is natural to impose a homotopy-theoretic analogue of condition $(\ast)$:

\begin{definition}\label{calhoun}
Let $\bfA$ be a model category. We say that a $\bfA$-enriched preSegal category
$(S,X)$ is a {\it $\bfA$-enriched Segal category} if, for every sequence of objects
$s_0, \ldots, s_n \in S$, the map
$$ \phi: X[s_0, s_1, \ldots, s_n] \rightarrow X[s_0, s_1] \times \ldots \times X[s_{n-1},s_n]$$
exhibits $X[ s_0, \ldots, s_n]$ as a homotopy product of the objects
$\{ X[s_{i-1}, s_i ] \}_{1 \leq i \leq n}$. 

We will say that $(S,X)$ is {\it locally fibrant} if $X[s_0, \ldots, s_n]$ is
a fibrant object of $\bfA$, for every sequence of objects $s_0, \ldots, s_n$.
\end{definition}

\begin{remark}
If $(S,X)$ is locally fibrant, or if the collection of weak equivalences in $\bfA$ is
stable under the formation of finite products, then $(S,X)$ is a $\bfA$-enriched
Segal category if and only if the map
$$X[s_0, \ldots, s_n] \rightarrow X[s_0, s_1] \times \ldots \times X[s_{n-1}, s_n]$$
is a weak equivalence for every sequence of elements $s_0, \ldots, s_n \in S$.
\end{remark}

\begin{example}\label{stepp1}
Let $\calC$ be a $\bfA$-enriched category, where $\bfA$ is a model category. Then the $\bfA$-enriched preSegal category associated to $\calC$ (see Example \ref{sabb}) is a $\bfA$-enriched Segal category.
\end{example}

In \S \ref{bisec2.2}, we will see that if $\bfA$ is a sufficiently nice model category, then
$\PreSeg{\bfA}$ inherits a model structure whose fibrant objects are precisely the
locally fibrant $\bfA$-enriched Segal categories. Our goal for the remainder of this section
is to lay some of the groundwork for the proof of this statement.

\begin{notation}
Let $n \geq 0$, and let $A \in \bfA$ be an object. We define a $\bfA$-enriched preSegal category
$\Fr^{n}(A)=(S, X)$ as follows:
\begin{itemize}
\item The underlying set $S$ is $[n] = \{0, \ldots, n \}$.
\item Suppose given an object $( [m], c: [m] \rightarrow S)$ of $\cDelta_{S}$. We define
$X( [m], c)$ to be an initial object of $\bfA$ if the map $c$ is not monotone, 
a final object of $\bfA$ if the map $c$ is constant, and the object $A$ otherwise.
\end{itemize}
We observe that $\Fr^{n}(A)$ can be described by the following universal property:
given any $\bfA$-enriched preSegal category $(S,X)$, giving a map
$\Fr^{n}(A) \rightarrow (S,X)$ is equivalent to giving a sequence of objects
$s_0, \ldots, s_n \in S$ and a map $A \rightarrow X( [s_0, \ldots, s_n ] )$.
\end{notation}

\begin{definition}\label{scanpl}
Let $\bfA$ be a model category.
We will say that a map $f: (S,X) \rightarrow (S',X')$ of $\bfA$-enriched preSegal categories
is a {\it generating projective cofibration} if one of the following conditions holds:
\begin{itemize}
\item[$(a)$] The set $S$ is empty, $S'$ consists of a single element, and the functor
$X': \cDelta_{S'}^{op} \rightarrow \bfA$ is a constant functor taking value equal to a final object
${\bf 1} \in \bfA$.
\item[$(b)$] There exists a cofibration $f_0: A \rightarrow B$ in $\bfA$ and an integer $n \geq 0$ such that $f$ is the induced map $\Fr^n(A) \rightarrow \Fr^{n}(B)$.
\end{itemize}
We will say that a morphism in $\PreSeg{\bfA}$ is a {\it projective cofibration} if it belongs to the smallest
weakly saturated class of morphisms in $\PreSeg{\bfA}$ containing all generating cofibrations.

We will say that a $\bfA$-enriched preSegal category $(S,X)$ is {\it cofibrant} if the map
$\emptyset \rightarrow (S,X)$ is a cofibration, where $\emptyset$ denotes the initial object
of $\PreSeg{\bfA}$.
\end{definition}

\begin{remark}\label{smuther}
In Definition \ref{scanpl}, we can replace the class of morphisms $(b)$ by the collection 
of all morphisms $\{ \Fr^n(f): \Fr^n(A) \rightarrow \Fr^n(B) \}$, where $f: A \rightarrow B$ ranges over
a collection of generating projective cofibrations of $\bfA$. It follows that if $\bfA$ is a combinatorial model category, then the collection of projective cofibrations in
$\PreSeg{\bfA}$ is of small generation (as a weakly saturated class), so that we can apply the small object argument.
\end{remark}

\begin{remark}\label{soccer}
Assume that $\bfA$ is a model category in which every object is cofibrant.
Then each of the generating projective cofibrations of Definition \ref{scanpl} is a morphism between {\em cofibrant} objects
of $\PreSeg{\bfA}$.
\end{remark}

\begin{lemma}\label{spottus}
Let $\bfA$ be a combinatorial model category, and let $f: (S,X) \rightarrow (S', X')$ be a map of $\bfA$-enriched preSegal categories. The following conditions are equivalent:
\begin{itemize}
\item[$(1)$] The map $f$ is a projective cofibration.
\item[$(2)$] The map $f$ admits a factorization
$$ (S,X) \stackrel{f'}{\rightarrow} (S'', X'') \stackrel{f''}{\rightarrow} (S', X')$$
where:
\begin{itemize}
\item The map $f'$ is an iterated pushout of morphisms of type $(a)$ appearing in Definition \ref{scanpl}
(in other words, $(S'', X'')$ is obtained from $(S,X)$ by freely adjoining new objects).
\item The map $f''$ is a retract of a transfinite composition of pushouts of morphisms of type
$(b)$ appearing in Definition \ref{scanpl}.
\end{itemize}
\end{itemize}
\end{lemma}

\begin{proof}
The implication $(2) \Rightarrow (1)$ is obvious. For the converse, suppose that
$f$ is a projective cofibration. Using the small object argument (see Remark \ref{smuther}), we deduce
that there exists a transfinite sequence of $\bfA$-enriched preSegal categories
$\{ (S_{\beta}, X_{\beta}) \}_{ \beta < \alpha}$ with the following properties:
\begin{itemize}
\item The pair $(S_0, X_0)$ coincides with $(S,X)$.
\item For $0 < \beta < \alpha$, the map 
$$\colim_{\beta' < \beta} (S_{\beta'}, X_{\beta'}) \rightarrow (S_{\beta}, X_{\beta} )$$
is a pushout of a generating projective cofibration $\phi_{\beta}$.
\item The $\bfA$-enriched preSegal category $(S',X')$ is a retract of
the colimit $\colim_{ \beta < \alpha} (S_{\beta}, X_{\beta})$ in the category
$(\PreSeg{\bfA})_{ (S,X)/}$. 
\end{itemize}
Reordering the generating projective cofibrations $\phi_{\beta}$ if necessary, we may suppose that
there exists an ordinal $\alpha_0 \leq \alpha$ such that $\phi_{\beta}$ is of type
$(a)$ appearing in Definition \ref{scanpl} for $\beta < \alpha_0$, and of type $(b)$
otherwise. Let $( \overline{S}, \overline{X} )$ denote the colimit of the sequence
$\{ (S_{\beta}, X_{\beta} ) \}_{ \beta < \alpha}$, so that we have a retraction diagram
$$ \xymatrix{ & (\overline{S}, \overline{X}) \ar[dr]^{r} & \\
(S', X') \ar[ur]^{s} \ar[rr]^{\id} & & (S', X'). }$$

Let $\overline{S}'' \subseteq \overline{S}$ denote the image of $S'$ in $\overline{S}$.
For each $\beta < \alpha$, let $S''_{\beta}$ denote the inverse image of $\overline{S}''$ in
$S_{\beta}$ (so that $S''_{\beta} = S_{\beta}$ if $\beta \geq \alpha$), and let
$X''_{\beta}$ denote the restriction of $X_{\beta}$ to $\cDelta^{op}_{ S''_{\beta} }$. 
We now define $(X'',S'')$ to be the pair $(X''_{\alpha_0}, S''_{\alpha_0})$, 
$f'$ the canonical map $(X,S) = (X_0,S_0) \rightarrow (X''_{\alpha_0}, S''_{\alpha_0})$, and
$f''$ the composition
$$ ( X''_{\alpha_0}, S''_{\alpha_0}) \rightarrow (X_{\alpha_0}, S_{\alpha_0})
\rightarrow ( \overline{X}, \overline{S} ) \stackrel{r}{\rightarrow} (S', X').$$
It is easy to see that $f'$ and $f''$ have the desired properties.
\end{proof}

\begin{example}\label{swaggius}
Let $f: A \rightarrow B$ be a cofibration in $\bfA$. We define a $\bfA$-enriched preSegal category
$\Fr^{n}(f)=([n], X)$ as follows:
\begin{itemize}
\item For every sequence of maps $i_0, \ldots, i_k \in [n]$, we let
$$X( [i_0, \ldots, i_k]) = \begin{cases} {\bf 1} & \text{if } i_0 = \ldots = i_k \\
B & \text{if } i_0 \leq i_1 \leq \ldots \leq i_k = i_0 + 1 \\
A & \text{if } i_0 \leq i_1 \leq \ldots \leq i_k > i_0 +1 \\
\emptyset & \text{otherwise.} \end{cases}$$
\end{itemize}
Here ${\bf 1}$ and $\emptyset$ denote final and initial objects of $\bfA$, respectively.
We have a canonical map $\psi: \Fr^{n}(f) \rightarrow \Fr^n(B)$. The domain and codomain of $\psi$ are cofibrant objects of $\PreSeg{\bfA}$, and the induced map $F(\psi): F( \Fr^{n}(f) ) \rightarrow F( \Fr^{n}(B))$ is an isomorphism of $\bfA$-enriched categories. It follows that $\psi$ is a weak equivalence.
Corollary \ref{curry} now implies that every pushout of $\psi$ is a weak equivalence in
$\PreSeg{\bfA}$.
\end{example}

\subsection{The Projective Model Structure on $\PreSeg{\bfA}$}\label{bisec2.2}

Throughout this section, we will assume that $\bfA$ is a combinatorial model category
satisfying the following conditions:
\begin{itemize}
\item[$(A1)$] Every object of $\bfA$ is cofibrant.
\item[$(A2)$] For every object $X \in \bfA$, the functor $Y \mapsto X \times Y$ preserves small colimits.
\item[$(A3)$] The Cartesian product on $\bfA$ endows $\bfA$ with the structure of a monoidal model category. In other words, given a pair of cofibrations $f: A \rightarrow A'$, $g: B \rightarrow B'$, the induced map
$$ f \wedge g: (A \times B') \coprod_{ A \times B} (A' \times B) \rightarrow A' \times B'$$
is again a cofibration, which is trivial if either $f$ or $g$ is trivial.
\item[$(A4)$] The collection of weak equivalences in $\bfA$ is stable under filtered colimits.
\end{itemize}

These conditions guarantee that the category $\Cat_{\bfA}$ of $\bfA$-enriched categories
admits a model structure; see \S \toposref{compp4}. Our goal in this section is exhibit
a model structure on $\PreSeg{\bfA}$ such that the fully faithful embedding
$G: \Cat_{\bfA} \hookrightarrow \PreSeg{\bfA}$ of Example \ref{sabb} is a right Quillen equivalence.
Our first step is to give an explicit construction of a left adjoint to this embedding.

\begin{remark}\label{makemr}
Suppose given a $\bfA$-enriched preSegal category $(S,X)$ and a $\bfA$-enriched category
$\calC$. A map of $\bfA$-enriched preSegal categories $(S,X) \rightarrow G(\calC)$ is
determined by the following data:
\begin{itemize}
\item For each element $s \in S$, an object $\alpha(s) \in \calC$.
\item For each sequence of elements $s_0, \ldots, s_n \in \calC$ and every pair of integers
$0 \leq i \leq j \leq n$, a map
$$ \beta^{i,j}_{s_0, \ldots, s_n}: X( [s_0, \ldots, s_n] ) \rightarrow \bHom_{\calC}( \alpha(s_i), \alpha(s_j) ).$$
\end{itemize}
Moreover, such data determines a map of $\bfA$-enriched preSegal categories provided that the following conditions are satisfied:
\begin{itemize}
\item If $0 \leq i = j \leq n$, then $\beta^{i,j}_{s_0, \ldots, s_n}$ is given by the composition
$$ X( [s_0, \ldots, s_n] ) \rightarrow {\bf 1} \stackrel{\id}{\rightarrow} \bHom_{\calC}( \alpha(s_i), \alpha(s_j) ),$$
where ${\bf 1}$ denotes the final object of $\bfA$.
\item If $0 \leq i \leq j \leq k \leq n$, then $\beta^{i,k}_{s_0, \ldots, s_n}$ is obtained by composing
$\beta^{i,j}_{s_0, \ldots, s_n}$ with $\beta^{j,k}_{s_0, \ldots, s_n}$ in the category $\calC$.
\item Given a map $f: [s_0, \ldots, s_n] \rightarrow [s'_0, \ldots, s'_{n'} ]$ in $\cDelta_{S}$ and
a pair of integers $0 \leq i \leq j \leq n$, the map $\beta^{ f(i), f(j)}_{ s'_0, \ldots, s'_{n'} }$ is given
by composing $\beta^{i,j}_{s_0, \ldots, s_n}$ with $X(f)$.
\end{itemize}
\end{remark}

\begin{notation}
Let $S$ be a set containing a pair of elements $x$ and $y$. We define a category
$\calJ_{x,y}(S)$ as follows:
\begin{itemize}
\item An object of $\calJ_{x,y}(S)$ consists of a sequence of elements
$(s_0, s_1, \ldots, s_n) \in S^{n+1}$ such that $s_0 = x$ and $s_n = y$, together with
a sequence of integers $\{ 0 = i_0 < i_1 < \ldots < i_k = n \}$.

\item A morphism from $( (s_0, \ldots, s_n), \{ 0 = i_0 < \ldots < i_k = n\} )$ to
$( ( s'_0, \ldots, s'_{n'}), \{ 0 = i'_0 < \ldots < i'_{k'} = n' \})$ in the category $\calJ_{x,y}(S)$
is a map of linearly ordered sets $f: [n'] \rightarrow [n]$ satisfying the following conditions:
\begin{itemize}
\item We have $f(0) = 0$ and $f(n') = n$.
\item For $0 \leq m \leq n'$, we have $s_{f(m)} = s'_{m}$.
\item For each $0 \leq j' < k'$, there exists $0 \leq j < k$ such that
$$i_{j} \leq f( i'_{j'} ) \leq f( i'_{j'+1}) \leq i_{j+1}.$$
(Note that $j$ is uniquely determined unless
$i_j = f(i'_{j'}) = f(i'_{j'+1})$ or $f(i'_{j'}) = f(i'_{j'+1}) = i_{j+1}$.)
\end{itemize}
\end{itemize}
Given a sequence of elements $x_0, \ldots, x_m$ in $S$, there is an evident concatenation
functor
$$ \calJ_{x_0,x_1}(S) \times \ldots \times \calJ_{ x_{m-1}, x_m}(S) \rightarrow \calJ_{x_0, x_m}(S).$$

Suppose now that we are given a $\bfA$-enriched preSegal category
$(S,X)$. Given a pair of elements $x,y \in S$, we define a functor
$H^{X}_{x,y}: \calJ_{x,y}(S) \rightarrow \bfA$ as follows:
\begin{itemize}
\item[$(\ast)$] Let $\sigma = ( (s_0, \ldots, s_n), \{ 0 = i_0 < \ldots < i_k = n\} )$ be an object of $\calJ_{x,y}$. We then define $H^{X}_{x,y}(\sigma)$ to be the product
$$X( [s_0, \ldots, s_{i_1}] ) \times X( [ s_{i_1}, \ldots, s_{i_2}] ) \times \ldots \times X( [s_{i_{k-1}}, \ldots, s_{n}]).$$
\end{itemize}
\end{notation}

\begin{remark}\label{kuli}
Given a sequence of elements $x_0, \ldots, x_m$ of $S$, the composition
$$ \calJ_{x_0, x_1}(S) \times \ldots \times \calJ_{x_{m-1}, x_m}(S)
\rightarrow \calJ_{x_0, x_m}(S) \stackrel{ H^X_{x_0, x_m}}{\rightarrow} \bfA$$
is canonically isomorphic with the external product of the functors $\{ H^{X}_{x_i, x_{i+1}} \}_{0 \leq i < m}.$
\end{remark}

\begin{proposition}\label{koon}
Let $(S,X)$ be a $\bfA$-enriched preSegal category. 
\begin{itemize}
\item[$(1)$] There exists a $\bfA$-enriched category
$F(S,X)$ which may be described as follows: 
\begin{itemize}
\item[$(a)$] The objects of $F(S,X)$ are the elements of $S$.
\item[$(b)$] Given a pair of objects $x,y \in S$, the mapping object
$\bHom_{F(S,X)}(x,y)$ is given by the colimit of the diagram
$$ H^{X}_{x,y}: \calJ_{x,y}(S) \rightarrow \bfA.$$
\item[$(c)$] Given a sequence of objects $x_0, \ldots, x_m \in S$, the composition law
$$ \bHom_{F(S,X)}(x_0, x_1) \times \ldots \times \bHom_{F(S,X)}(x_{m-1}, x_m)
\rightarrow \bHom_{ F(S,X)}( x_0, x_m)$$
is given by the chain of morphisms
\begin{eqnarray*}
\colim \prod_{1 \leq i \leq m} H^{X}_{x_{i-1}, x_i} & \stackrel{\phi}{\simeq} & 
\colim (\prod_{1 \leq i \leq m} H^{X}_{x_{i-1}, x_i}) | \prod_{1 \leq i \leq m} \calJ_{x_{i-1}, x_i}(S) \\
& \stackrel{\phi'}{\simeq} & \colim H^{X}_{x_0, x_m} | \prod_{ 1 \leq i \leq m} \calJ_{x_{i-1}, x_i}(S) \\
& \rightarrow & \colim H^{X}_{x_0, x_m}. \end{eqnarray*}
Here the isomorphism $\phi$ results from our assumption that the Cartesian product preserves colimits separately in each variable, and the isomorphism $\phi'$ from Remark \ref{kuli}.
\end{itemize}
\item[$(2)$] There exists a map of $\bfA$-enriched Segal categories
$$ u: (S,X) \rightarrow G( F(S,X) )$$ which is the identity on objects and satisfies the following universal property: for every $\bfA$-enriched category $\calC$,
the composite map
$$ \phi: \Hom_{ \Cat_{\bfA}}( F(S,X), \calC) \rightarrow \Hom_{ \PreSeg{\bfA}}( G( F(S,X) ), G(\calC) )
\stackrel{ \circ u}{\rightarrow} \Hom_{ \PreSeg{\bfA}}( (S,X), G(\calC) ).$$
\end{itemize}
\end{proposition}

\begin{proof}
Assertion $(1)$ is evident. To construct the map $u$ described in $(2)$, we invoke Remark \ref{makemr}: it will suffice to define, for every sequence of elements $s_0, \ldots, s_n \in S$ and every
pair of integers $0 \leq i \leq j \leq n$, a map
$$ \beta^{i,j}_{s_0, \ldots, s_n}: X( [s_0, \ldots, s_n] )
\rightarrow \bHom_{ F(S,X)}( s_i, s_j)$$
satisfying some evident compatibility conditions. We will take $\beta^{i,j}_{s_0, \ldots, s_n}$ to be
given by the composition
\begin{eqnarray*}
X( [s_0, \ldots, s_n] ) & \rightarrow & X( [ s_i, s_i+1, \ldots, s_j ]) \\
& = & H^{X}_{s_i, s_j}( (s_i, s_i + 1, \ldots, s_j), \{ 0 \leq j-i \} ) \\
& \rightarrow & \colim H^{X}_{s_i, s_j} \\
& = & \bHom_{ F(S,X)}( s_i, s_j). \end{eqnarray*}
It is not difficult to check that these maps satisfy the conditions of Remark \ref{makemr} and therefore determine a map $u$ of $\bfA$-enriched preSegal categories. 

To complete the proof, it will suffice to show that for every $\bfA$-enriched category $\calC$, the map $\phi$ is a bijection. The proof proceeds by explicitly constructing a map 
$\psi: \Hom_{ \PreSeg{\bfA}}( (S,X), G(\calC) ) \rightarrow \Hom_{\Cat_{\bfA}}( F(S,X), \calC)$ inverse
to $\phi$. Suppose given a map $f$ of $\bfA$-enriched preSegal categories $(S,X) \rightarrow G(\calC)$.
This data determines a map $\alpha$ from $S$ to the set of objects of $\calC$, and a collection of maps
$$ \beta^{i,j}_{s_0, \ldots, s_n}: X( [s_0, \ldots, s_n]) \rightarrow \bHom_{\calC}( \alpha(s_i), \alpha(s_j) )$$
as explained in Remark \ref{makemr}. We will define a $\bfA$-enriched functor $\psi(f): F(S,X) \rightarrow \calC$ as follows. On objects, $\psi(f)$ is given by the map $\alpha$. To define
$\psi(f)$ on morphisms, we must give for every pair of elements $x,y \in S$ a map
$$ \colim H^{X}_{x,y} \rightarrow \bHom_{\calC}( \alpha(x), \alpha(y) ).$$
This is equivalent to giving a compatible family of maps
$$ H^{X}_{x,y}(\sigma) \rightarrow \bHom_{\calC}( \alpha(x), \alpha(y) ),$$
where $\sigma = ( (s_0, \ldots, s_n ), \{ 0 = i_0 < i_1 < \ldots < i_k = n\} )$ ranges over
the objects of $\calJ_{x,y}(S)$. We take this to be the map given by the composition
\begin{eqnarray*}
H^{X}_{x,y}(\sigma) & = & X([s_0, \ldots, s_{i_1}]) \times \ldots \times X( [s_{i_{k-1}}, \ldots, s_n])
\\
& \stackrel{ \beta{0, i_1}_{ s_0, \ldots, s_{i_1}} \times \ldots \times \beta^{0, n- i_{k-1}}_{ s_{i_{k-1}}, \ldots, s_{n} }}{\rightarrow} & \bHom_{\calC}( \alpha(s_0), \alpha(s_{i_1})) \times \ldots
\times \bHom_{\calC}( \alpha(s_{k-1}), \alpha(s_n) ) \\
& \rightarrow & \bHom_{\calC}( \alpha(x), \alpha(y) ). \end{eqnarray*}
It is straightforward to verify that this collection of maps has the desired properties, and that this construction determines a map $\psi$ which is inverse to $\phi$.
\end{proof}

It follows from Proposition \ref{koon} that the construction $(S,X) \mapsto F(S,X)$ is functorial in
the $\bfA$-enriched preSegal category $(S,X)$, and determines a left adjoint to the functor of Example \ref{sabb}. Our next goal is to understand the homotopy types of the mapping objects in
$F(S,X)$. Since these mapping objects are given by colimits in $\bfA$, we would like a criterion
to guarantee that these colimits are also homotopy colimits (see Proposition \ref{hunk} below). First, we need a very formal preliminary result:

\begin{lemma}\label{quen}
Let $\bfA$, $\bfB$, and $\bfC$ be combinatorial model categories, and let
$F: \bfA \times \bfB \rightarrow \bfC$ be a left Quillen bifunctor.
Let $\calI$ and $\calJ$ be small categories. Then the induced functor
$$ \Fun(\calI, \bfA) \times \Fun( \calJ, \bfB) \rightarrow \Fun( \calI \times \calJ, \bfC)$$
is again a left Quillen bifunctor; here we regard $\Fun( \calI, \bfA)$, 
$\Fun(\calJ, \bfB)$, and $\Fun( \calI \times \calJ, \bfC)$ as endowed with the projective model structure.
\end{lemma}

\begin{proof}
Let $f: A \rightarrow A'$ be a projective cofibration in $\Fun( \calI, \bfA)$ and let
$g: B \rightarrow B'$ be a projective cofibration in $\Fun( \calJ, \bfB)$. Define new functors
$$A' \otimes B', f \wedge g: \calI \times \calJ \rightarrow \calC$$
by the formulas
$$ (A' \otimes B')(I,J) = F( A'(I), B'(J) )$$
$$ (f \wedge g)(I,J) = F( A(I), B'(J) ) \coprod_{ F( A(I), B(J) )} F( A'(I), B(J) ).$$
We have an induced map $\phi_{f,g}: f \wedge g \rightarrow A' \otimes B'$. We wish to prove that
$\phi_{f,g}$ is a projective cofibration, which is trivial if either $f$ or $g$ is a trivial cofibration.
We will prove the first assertion; the second follows by the same argument.

Let us first regard $f$ as fixed, and consider the collection of all $g$ such that $\phi_{f,g}$ is a 
projective cofibration. It is not difficult to verify that this collection is weakly saturated. It will therefore
suffice to prove the assertion as $g$ ranges over a collection of generators for the weakly saturated class of all projective cofibrations in $\Fun( \calJ, \bfB)$. We may therefore assume that
$g = \psi_{!} \overline{g}$, where $\psi: [0] \rightarrow \calJ$ is a functor (corresponding to an object of $\calJ$), $\psi_{!}$ the corresponding left Kan extension functor, and $\overline{g}$ is a cofibration
in $\Fun( [0], \bfB) \simeq \bfB)$. Replacing $\calJ$ by $[0]$, we may reduce to the case where
the category $\calJ$ consists of a single object. We now regard $g$ as fixed and apply the same argument to reduce to the case where $\calI$ consists of a single object. We are therefore reduced to proving that the original functor
$$F: \bfA \times \bfB \rightarrow \bfC$$
is a left Quillen bifunctor, which is true by assumption.
\end{proof}

\begin{proposition}\label{hunk}
Let $f: (S, X) \rightarrow (S, X')$ be a projective cofibration between cofibrant $\bfA$-enriched preSegal categories
which is the identity on the object set $S$. For every pair of objects $x,y \in S$, the induced map
$$ H^{X}_{x,y} \rightarrow H^{X'}_{x,y}$$ 
is a projective cofibration in the category of functors from $\calJ_{x,y}(S)$ to $\bfA$.
\end{proposition}

\begin{proof}
We will prove Proposition \ref{hunk} under the following additional assumption:
\begin{itemize}
\item[$(\ast)$] For every pair of elements $x,y \in S$, the diagram
$H^{X}_{x,y}$ is projectively cofibrant.
\end{itemize}
Assume for the moment that this weaker version of the Proposition holds. Since $(S,X)$ is cofibrant, Lemma \ref{spottus} implies that we have a projective cofibration $f_0: (S, X_0) \rightarrow (S,X)$, where
$X_0: \cDelta^{op}_{S} \rightarrow \bfA$ is defined by the formula
$$ X_0( [s_0, \ldots, s_n]) = \begin{cases} {\bf 1} & \text{ if } s_0 = \ldots = s_n \\
\emptyset & \text{ otherwise.} \end{cases}$$
Here ${\bf 1}$ and $\emptyset$ denote final and initial objects of $\bfA$, respectively.
Since $(\ast)$ is satisfied by $X_0$, we deduce that the maps
$H^{X_0}_{x,y} \rightarrow H^{X}_{x,y}$ are projective cofibrations, so that
assumption $(\ast)$ is also satisfied by $X$: in other words, condition $(\ast)$ is automatic, so that Proposition \ref{hunk} holds in general.

Let us now assume that $X$ satisfies $(\ast)$. Using Lemma \ref{spottus}, we deduce the existence of
a transfinite sequence of $\bfA$-enriched preSegal categories $\{ (S, X_{\beta} \}_{ \beta \leq \alpha}$
with the following properties:
\begin{itemize}
\item We have $X_0 = X$.
\item If $\beta \leq \alpha$ is a nonzero limit ordinal, then $X_{\beta}$ is isomorphic to the colimit of the diagram $\{ X_{\beta'} \}_{ \beta' < \beta}$.
\item If $\beta < \alpha$, then the map $X_{\beta} \rightarrow X_{\beta'}$ is a pushout of a generating projective cofibration of type $(b)$ appearing in Definition \ref{scanpl}.
\item The morphism $X \rightarrow X'$ is a retract of $X_0 \rightarrow X_{\alpha}$.
\end{itemize}

Since the collection of projective cofibrations is stable under retracts, we may assume without loss of generality that $f$ is the map $X_0 \rightarrow X_{\alpha}$. We will prove the following:
\begin{itemize}
\item[$(a)$] For each $\gamma \leq \beta \leq \alpha$, the map $H^{X_{\gamma}}_{x,y} \rightarrow
H^{X_{\beta}}_{x,y}$ is a projective cofibration. In particular (taking $\gamma = 0$ and applying assumption $(\ast)$), the diagram $H^{X_{\beta}}_{x,y}$ is projectively cofibrant.
\end{itemize}
We will prove $(a)$ using induction on $\beta$, regarding the ordinal $\gamma$ as fixed.
If $\beta = \gamma$ there is nothing to prove, and if $\beta$ is a limit ordinal then the result follows from the inductive hypothesis, since the construction $X \mapsto H^{X}_{x,y}$ preserves filtered colimits.
We may therefore suppose that $\beta = \beta_0 + 1$ is a successor ordinal larger than $\gamma$. The inductive hypothesis guarantees that the map $H^{X_{\gamma}}_{x,y} \rightarrow H^{X_{\beta_0}}_{x,y}$ is a projective cofibration. It will therefore suffice to show that the map
$H^{X_{\beta_0}}_{x,y} \rightarrow H^{X_{\beta}}_{x,y}$ is a projective cofibration. 
We map now replace $X$ by $X_{\beta_0}$ (observe that
our inductive hypothesis guarantees that $(\ast)$ is still satisfied) and $X'$ by $X_{\beta}$ to
reduce to proving the original form of Proposition \ref{hunk} under the following additional assumption:

\begin{itemize}
\item[$(\ast')$] There exists a cofibration $u: A \rightarrow B$ in $\bfA$, an integer $n \geq 0$, and a pushout square
$$ \xymatrix{ \Fr^{n}(A) \ar[r]^{ \Fr^{n}(u)} \ar[d]^{\phi} & \Fr^{n}(B) \ar[d] \\
(S, X) \ar[r] & (S, X'). }$$
\end{itemize}
This diagram is classified by a sequence of elements $s_0, \ldots, s_n \in S$ and a commutative
diagram
$$ \xymatrix{ A \ar[r] \ar[d] & B \ar[d]^{v} \\
X( [s_0, \ldots, s_n]) \ar[r] & X'( [s_0, \ldots, s_n] ) }$$
in the category $\bfA$.

For each $k \geq 0$ and each subset $S \subseteq [k]$, let
$$\calF^{[k]}_{S}: \calJ_{x, s_0}(S) \times \calJ_{s_n, s_0}(S)^{k} \times \calJ_{s_n, y}(S)
\rightarrow \bfA$$
be the functor given by the formula
$$( \sigma_0, \sigma_1, \ldots, \sigma_{k+1}) \mapsto
H^{X}_{x,s_0}( \sigma_0) \times C_0 \times
H^{X}_{s_n,s_0}(\sigma_1) \times C_1 \times \ldots \times C_k \times H^{X}_{s_n, y}(\sigma_{k+1}),$$
where $C_i$ is equal to $A$ if $i \notin S$ and $B$ otherwise. Let $\calF^{[k]}_{+}$ denote the functor
$\calF^{[k]}_{[k]}$, and let $\calF^{[k]}_{-}$ denote the colimit of the functors
$\{ \calF^{[k]}_{S} \}_{S \subset [k] }$, so that we have a map $\phi^{k}: \calF^{[k]}_{-} \rightarrow
\calF^{[k]}_{+}$.

There is a natural concatenation functor
$$\psi^k: \calJ_{x, s_0}(S) \times \calJ_{s_n, s_0}(S)^{k} \times \calJ_{s_n, y}(S)
\rightarrow \calJ_{x,y}(S),$$
and the map $v$ determines a natural transformation from each
$\calF^{[k]}_{+}$ to $H^{X'}_{x,y} \circ \psi^k$. 
Let 
$$\psi^k_{!}: \Fun(  \calJ_{x, s_0}(S) \times \calJ_{s_n, s_0}(S)^{k} \times \calJ_{s_n, y}(S), \bfA)
\rightarrow \Fun( \calJ_{x,y}(S), \bfA)$$
denote the left Kan extension functor.

The diagram $H^{X'}_{x,y}$ can be realized as the direct limit of a sequence of diagrams
$$ H^{X}_{x,y} \simeq H^{(0)} \rightarrow H^{(1)} \rightarrow H^{(2)} \rightarrow \ldots$$
with the following property:
\begin{itemize}
\item For each $k > 0$, there is a pushout diagram (in the category $\Fun( \calJ_{x,y}(S), \bfA)$)
$$ \xymatrix{ \psi^{k}_{!} \calF^{[k]}_{-} \ar[r]^{\psi^{k}_{!}(\phi^k)} \ar[d] & \psi^{k}_{!} \calF^{[k]}_{+} \ar[d] \\
H^{(k-1)} \ar[r] & H^{(k)}. }$$
\end{itemize}
To complete the proof, it will suffice to show that each $\psi^{k}_{!}( \phi^k)$ is a projective
cofibration. Since $\psi^{k}_{!}$ is a left Quillen functor, we are reduced to proving that
$\phi^{k}$ is a projective cofibration. This follows from Lemma \ref{quen}, since
$\phi^{k}$ is an iterated external smash product of maps of the form
$\emptyset \rightarrow H^{X}_{x',y'}$ (here $\emptyset$ denotes the initial object
of $\Fun( \calJ_{x',y'}(S), \bfA)$; this map is projective cofibration by virtue of assumption $(\ast)$)
and $u: A \rightarrow B$ (which is a projective cofibration in the category $\Fun([0], \bfA)$).
\end{proof}

It follows from Proposition \ref{hunk} that the construction $(S,X) \mapsto F(S,X)$ is
homotopy invariant when restricted to cofibrant $\bfA$-enriched preSegal categories
(Proposition \ref{camble} below). To make this precise, we need to introduce some terminology.

\begin{definition}
Let $\phi: (S,X) \rightarrow (S',X')$ be a morphism of $\bfA$-enriched preSegal categories. We will say that $\phi$ is a {\it pointwise fully faithful} if, for every sequence of objects $s_0, \ldots, s_n \in S$, the map $X( [s_0, \ldots, s_n] ) \rightarrow X'( [\phi(s_0), \ldots, \phi(s_n) ] )$ is a weak equivalence in $\bfA$.
\end{definition}

\begin{lemma}\label{cuppl}
Let $\phi: \calI \rightarrow \calJ$ be a functor between small categories such that the induced map
$\Nerve( \calI ) \rightarrow \Nerve(\calJ)$ is cofinal, and let $\bfA$ be a combinatorial model category.
Then, for any functor $X: \calJ \rightarrow \bfA$, the canonical map
$$ \hocolim( X \circ \phi) \rightarrow \hocolim(X)$$
is an isomorphism in the homotopy category $\h{\bfA}$.
\end{lemma}

\begin{proof}
Replacing $\bfA$ by a Quillen equivalent combinatorial model category if necessary, we may assume that $\bfA$ is simplicial (see, for example, \cite{combmodel}). Without loss of generality, we may assume that $X$ takes values in the full subcategory $\bfA^{\degree}$ of fibrant-cofibrant objects of $\bfA$, and
therefore determines a functor $x: \Nerve(\calJ) \rightarrow \Nerve( \bfA^{\degree})$. 
Using Theorem \toposref{colimcomparee}, we are reduced to proving that the induced map
$\colim(x \circ \phi) \rightarrow \colim(x)$ is an equivalence in the $\infty$-category
$\Nerve(\bfA^{\degree})$, which follows from Proposition \toposref{gute}.
\end{proof}

\begin{proposition}\label{camble}
Let $\phi: (S,X) \rightarrow (S', X')$ be a morphism of $\bfA$-enriched preSegal categories. Assume that:
\begin{itemize}
\item[$(1)$] The $\bfA$-enriched preSegal categories $(S,X)$ and $(S, X')$ are cofibrant.
\item[$(2)$] The induced map $S \rightarrow S'$ on objects is surjective.
\item[$(3)$] The map $\phi$ is pointwise fully faithful.
\end{itemize}
Then the induced map $F(\phi): F(S,X) \rightarrow F(S',X')$ is an equivalence of $\bfA$-enriched categories.
\end{proposition}

\begin{proof}
Assumption $(2)$ guarantees that $F(\phi)$ is essentially surjective. It will therefore suffice to prove that, for every pair of objects $x,y \in S$, the induced map
$$ \bHom_{ F(S,X)}(x,y) \rightarrow \bHom_{ F(S', X') }( \phi(x), \phi(y) )$$
is a weak equivalence in $\bfA$. In other words, we must show that the map
$$ \colim ( H^{X}_{x,y}: \calJ_{x,y}(S) \rightarrow \bfA) \rightarrow
\colim (H^{X'}_{\phi(x),\phi(y)}: \calJ_{x',y'}(S') \rightarrow \bfA )$$
is a weak equivalence. Assumption $(1)$ and Proposition \ref{hunk} guarantee that
the diagrams $H^{X}_{x,y}$ and $H^{X'}_{\phi(x), \phi(y)}$ are projectively cofibrant, so it will suffice to show that the horizontal map in the diagram
$$ \xymatrix{ \hocolim (H^{X}_{x,y}: \calJ_{x,y}(S) \rightarrow \bfA) \ar[r] \ar[dr]^{\phi} & \hocolim (H^{X'}_{\phi(x), \phi(y) }: \calJ_{x,y}(S') \rightarrow \bfA) \\
& \hocolim ( \calJ_{x,y}(S) \rightarrow \calJ_{\phi(x), \phi(y)}(S')
\stackrel{ H^{X'}_{\phi(x), \phi(y) }}{\rightarrow} \bfA) \ar[u]^{\psi} }$$
is a weak equivalence. By the two-out-of-three property, it will suffice to show that $\phi$ and
$\psi$ are weak equivalences. The map $\phi$ is a weak equivalence because the transformation
$H^{X}_{x,y} \rightarrow H^{X'}_{ \phi(x), \phi(y)} | \calJ_{x,y}(S)$ is a pointwise weak equivalence
(in view of assumption $(3)$ together with the observation that every product in $\bfA$ is a homotopy product, since every object is cofibrant and the Cartesian product functor on $\bfA$ is a left Quillen bifunctor). To prove that $\psi$ is a weak equivalence, it will suffice (by virtue of Lemma \toposref{gute}) to show that the map $\Nerve(\calJ_{x,y}(S)) \rightarrow \Nerve(\calJ_{\phi(x), \phi(y)}(S'))$ is cofinal. 
In view of Theorem \toposref{hollowtt}, this is equivalent to the following assertion:
for every object $\sigma \in \calJ_{\phi(x), \phi(y)}(S')$, the fiber product category
$$ \calJ_{x,y}(S)_{\sigma / } = \calJ_{x,y}(S) \times_{ \calJ_{ \phi(x), \phi(y) }(S') } \calJ_{\phi(x), \phi(y)}(S')_{\sigma/}$$
has weakly contractible nerve. 

Let $\sigma = ( (s'_0=\phi(x), s'_1, \ldots, s'_{n-1}, s'_n=\phi(y)), \{ 0 = i_0 < i_1 < \ldots < i_k = n \})$.
We define full subcategories $\calC \subseteq \calD \subseteq (\calJ_{x,y}(S)_{\sigma/})^{op}$ as follows: an object of $(\calJ_{x,y}(S)_{\sigma/})^{op}$ belongs to $\calD$ if and only if
the corresponding map
$$\alpha: \sigma \rightarrow ( (\phi(x), \overline{s}'_1, \ldots, \overline{s}'_{\overline{n}-1}, \phi(y)), \{ 0 = \overline{i}_0 < \overline{i}_1 < \ldots < \overline{i}_{ \overline{k}} = \overline{n} \})$$
in $\calJ_{\phi(x), \phi(y)}(S')$ has the property that 
$\alpha( \overline{i}_{j} ) \in \{ 0 = i_0 < \ldots < i_k = n \}$ for $0 \leq j \leq \overline{k}$. Such
an object belongs to $\calC$ if and only if $\alpha( [ \overline{n}] ) \subseteq \{ 0 = i_0 < \ldots < i_k = n \}$. We note that the inclusion $\calD \subseteq (\calJ_{x,y}(S)_{\sigma/})^{op}$ admits a left
adjoint, and the inclusion $\calC \subseteq \calD$ admits a right adjoint. It follows that the inclusions
$$\Nerve(\calC) \subseteq \Nerve(\calD) \subseteq \Nerve( \calJ_{ \phi(x), \phi(y)}(S')_{ \sigma/ })$$
are weak homotopy equivalences. It will therefore suffice to show that $\Nerve(\calC)$ is weakly contractible.

The category $\calC$ is equivalent to a product $\calC_0 \times \calC_1 \times \ldots \times \calC_n$, where the categories $\calC_i$ may be described as follows:
\begin{itemize}
\item An object of $\calC_i$ is given by a nonempty sequence
$s_0, \ldots, s_m \in \phi^{-1} \{ s'_i \}$, together with a nonempty subset $K \subseteq [m]$. 
If $i=0$, we have the additional requirement that $s_0 = x$ and $0 \in K$, while if $i =n$
we have the additional requirement that $s_m = y$ and $m \in K$.

\item A morphism from $(s_0, \ldots, s_m; K)$ to $( \overline{s}_0, \ldots, \overline{s}_{ \overline{m}; \overline{K} }$
in $\calC_i$ is a monotone map $\alpha: [m] \rightarrow [ \overline{m} ]$ such that
$s_{j} = \overline{s}_{ \alpha(j) }$ for $0 \leq j \leq m$, and $\overline{K} \subseteq \alpha(K)$.
If $i=0$, we have the additional requirement that $\alpha(0) = 0$, and if $i=n$ we have the additional requirement that $\alpha( m) = \overline{m}$.
\end{itemize}

We will show that each $\calC_i$ has a weakly contractible nerve. There are several cases to
consider:
\begin{itemize}
\item[$(a)$] Suppose that $i= 0 = n$. Let $\calC'_i$ be the full subcategory of $\calC_i$ spanned
by those objects $( s_0, \ldots, s_m ; K)$ such that $K = \{0, m\}$. The inclusion
$\calC'_i \subseteq \calC_i$ admits a left adjoint, so it will suffice to show that
$\calC'_i$ has weakly contractible nerve. We now observe that $\calC'_i$ has an initial object,
given by the sequence $(x,y; \{0,1\})$.
\item[$(b)$] Suppose that $i = 0 < n$. Let $\calC'_i$ be the full subcategory of $\calC_i$ spanned by those objects $(s_0, \ldots, s_m; K)$ such that $K = \{0\}$. We again observe that
the inclusion $\calC'_i \subseteq \calC_i$ admits a left adjoint, and that $\calC'_i$ has
an initial object, this time given by $(x; \{0\})$. 
\item[$(c)$] Suppose that $i = n > 0$. We then argue as in case $(b)$.
\item[$(d)$] Suppose that $0 < i < n$. Consider the following functors from $\calC_i \times \calC_i$ to
$\calC_i$:
\begin{itemize}
\item Let $\pi_1, \pi_2: \calC_i \times \calC_i \rightarrow \calC_i$ denote projection onto the first and second factor, respectively.
\item Let $F: \calC_i \times \calC_i \rightarrow \calC_i$ be the concatenation functor, so that
$$F( (s_0, \ldots, s_m; K), ( \overline{s}_0, \ldots, \overline{s}_{\overline{m}}; K'))=
( ( s_0, \ldots, s_m, \overline{s}_0, \ldots, \overline{s}_{\overline{m}}), K'')$$
with $K'' = K \cup \{ j + m + 1: j \in K' \}$.  
\item Let $F_{-}: \calC_i \times \calC_i \rightarrow \calC_i$ be the variant on the
concatenation functor described by the formula
$$F_{-}( (s_0, \ldots, s_m; K), ( \overline{s}_0, \ldots, \overline{s}_{\overline{m}}; K'))=
( ( s_0, \ldots, s_m, \overline{s}_0, \ldots, \overline{s}_{\overline{m}}), K).$$
\item Let $F_{+}: \calC_i \times \calC_i \rightarrow \calC_i$ be the variant on the
concatenation functor described by the formula
$$F_{+}( (s_0, \ldots, s_m; K), ( \overline{s}_0, \ldots, \overline{s}_{\overline{m}}; K'))=
( ( s_0, \ldots, s_m, \overline{s}_0, \ldots, \overline{s}_{\overline{m}}), K''),$$
where $K'' = \{ j + m + 1 | j \in K' \}$.
\end{itemize}
We have natural transformations of functors
$$ \pi_1 \rightarrow F_{-} \leftarrow F \rightarrow F_{+} \leftarrow \pi_2.$$
It follows that $\pi_1$ and $\pi_2$ induce homotopic maps from
$\Nerve( \calC_i) \times \Nerve(\calC_i)$ to $\calC_i$.
Fix an object $C \in \calC_i$ (this is possible in view of assumption $(2)$). The identity map from $\Nerve(\calC_i)$ to itself,
which is given by the composition
$$ \Nerve(\calC_i) \simeq \Nerve(\calC_i) \times \{C\} \subseteq \Nerve(\calC_i) \times \Nerve(\calC_i) \stackrel{ \Nerve(\pi_1) }{\rightarrow} \Nerve(\calC_i),$$
is homotopic to the composition
$$ \Nerve(\calC_i) \simeq \Nerve(\calC_i) \times \{C\} \subseteq \Nerve(\calC_i) \times \Nerve(\calC_i) \stackrel{ \Nerve(\pi_2)}{\rightarrow} \Nerve(\calC_i),$$
which is a constant map taking the value $C$. It follows that $\Nerve( \calC_i)$ is weakly contractible as desired.
\end{itemize}
\end{proof}

\begin{definition}
We will say that a morphism $(S, X) \rightarrow (S', X')$ of $\bfA$-enriched preSegal categories
is a {\it cofibrant refinement} if it is pointwise fully faithful, the map $S \rightarrow S'$ is
surjective, and $(S,X)$ is cofibrant.
\end{definition}

\begin{remark}
For every $\bfA$-enriched Segal category $(S', X')$, there exists a cofibrant refinement
$\phi: (S,X) \rightarrow (S', X')$. This follows from the small object argument: we can choose
$\phi$ so that $(S,X)$ is cofibrant and $\phi$ has the right lifting property with respect to all projective cofibrations. Unwinding the definition, the latter condition amounts to the requirement that $S \rightarrow S'$ is surjective and the map $X( [s_0, \ldots, s_n] ) \rightarrow X'( [ \phi(s_0), \ldots, \phi(s_n) ] )$ is a trivial fibration in $\bfA$, for every sequence of elements $s_0, \ldots, s_n \in S$. In particular, this condition guarantees that $\phi$ is pointwise fully faithful.
\end{remark}

\begin{corollary}\label{stafff}
Let $\phi: (S, X) \rightarrow (S', X')$ be a morphism of $\bfA$-enriched preSegal categories.
The following conditions are equivalent:
\begin{itemize}
\item[$(1)$] Suppose given a commutative diagram
$$ \xymatrix{ ( \overline{S}, \overline{X}) \ar[r]^{\overline{\phi}} \ar[d] & ( \overline{S}', \overline{X}') \ar[d] \\
(S, X) \ar[r]^{\phi} & (S', X') }$$
such that the vertical maps are cofibrant refinements.
Then the induced map $F( \overline{S}, \overline{X}) \rightarrow F( \overline{S}', \overline{X}')$ is
an equivalence of $\bfA$-enriched categories.
\item[$(2)$] There exists a commutative diagram satisfying the conditions listed in $(1)$.
\end{itemize}
\end{corollary}

\begin{proof}
Let us say that a diagram $\sigma:$
$$ \xymatrix{ ( \overline{S}, \overline{X}) \ar[r]^{\overline{\phi}} \ar[d] & ( \overline{S}', \overline{X}') \ar[d] \\
(S, X) \ar[r]^{\phi} & (S', X') }$$
is {\it good} if the vertical maps are cofibrant refinements, and {\it excellent} if it is good and the map
$F( \overline{S}, \overline{X}) \rightarrow F( \overline{S}', \overline{X}')$ is an equivalence of
$\bfA$-enriched categories. Using the small object argument, we can construct a good
diagram $\sigma_0$:
$$ \xymatrix{ (S, \overline{X}_0) \ar[r] \ar[d] & ( S', \overline{X}'_0) \ar[d] \\
(S, X) \ar[r]^{\phi} & (S', X') }$$
with the following additional properties:
\begin{itemize}
\item The vertical maps are the identity on objects.
\item For every sequence of elements $s_0, s_1, \ldots, s_n \in S$, the map
$\overline{X}_0( [ s_0, \ldots, s_n] ) \rightarrow X( [s_0, \ldots, s_n] )$ is a trivial fibration in $\bfA$.
\item For every sequence of elements $s'_0, \ldots, s'_n \in S'$, the map
$ \overline{X}'_0( [ s'_0, \ldots, s'_n] ) \rightarrow X'( [s'_0, \ldots, s'_n])$ is a trivial fibration in $\bfA$.
\end{itemize}
If $(1)$ is satisfied, then $\sigma_0$ is excellent, which proves $(2)$.

Conversely, suppose that $(2)$ is satisfied, so that there exists an excellent diagram. We wish to prove $(1)$: that is, we wish to prove that every good diagram $\sigma$ is excellent. For this, it will suffice to show that a good diagram $\sigma$ is excellent if and only if $\sigma_0$ is excellent.

We first observe the following immediate consequence of Proposition \ref{camble}:
\begin{itemize}
\item[$(\ast)$] Suppose given a commutative diagram
$$ \xymatrix{ 
( \widetilde{S}, \widetilde{X}) \ar[r] \ar[d] & ( \widetilde{S}', \widetilde{X}') \ar[d] \\
( \overline{S}, \overline{X}) \ar[r]^{\overline{\phi}} \ar[d] & ( \overline{S}', \overline{X}') \ar[d] \\
(S, X) \ar[r]^{\phi} & (S', X') }$$
where the vertical morphisms are cofibrant refinements. Then the lower square is good if and only if the outer rectangle is good.
\end{itemize}

In particular, given any good diagram $\sigma$:
$$ \xymatrix{ ( \overline{S}, \overline{X}) \ar[r]^{\overline{\phi}} \ar[d] & ( \overline{S}', \overline{X}') \ar[d] \\
(S, X) \ar[r]^{\phi} & (S', X'), }$$
we can choose a factorization of $\overline{\phi}$ as a composition
$$ ( \overline{S}, \overline{X}) \stackrel{ \overline{\phi}'}{\rightarrow} ( \widetilde{S}', \widetilde{X}')
\stackrel{ \overline{\phi}''}{\rightarrow} ( \overline{S}', \overline{X}')$$
where $\overline{\phi}'$ is a projective cofibration and $\overline{\phi}''$ is a cofibrant refinement. Taking
$( \widetilde{S}, \widetilde{X}) = ( \overline{S}, \overline{X})$ and applying $(\ast)$, we see that
$\sigma$ is excellent if and only if the diagram $\sigma'$:
$$ \xymatrix{ ( \overline{S}, \overline{X}) \ar[r]^{\overline{\phi}} \ar[d] & ( \widetilde{S}', \widetilde{X}') \ar[d] \\
(S, X) \ar[r]^{\phi} & (S', X') }$$
is excellent. It will therefore suffice to prove that $\sigma'$ is excellent if and only if $\sigma_0$ is excellent, which follows from $(\ast)$ together with the existence of a commutative diagram
$$ \xymatrix{ ( \overline{S}, \overline{X}) \ar[r] \ar[d] & ( \widetilde{S}', \widetilde{X}') \ar[d] \\
(S, \overline{X}_0) \ar[r] \ar[d] & ( S', \overline{X}'_0) \ar[d] \\
(S, X) \ar[r]^{\phi} & (S', X'). }$$
\end{proof}

\begin{lemma}\label{stade}
Let $\phi: (S,X) \rightarrow (S', X')$ be a projective cofibration of $\bfA$-enriched preSegal categories.
Then the induced map $F(\phi): F(S,X) \rightarrow F(S', X')$ is a cofibration of
$\bfA$-enriched categories.
\end{lemma}

\begin{proof}
The collection of morphisms $\phi$ for which $F(\phi)$ is a cofibration is weakly saturated.
It will therefore suffice to prove that this collection contains each of the generating projective cofibrations
of Definition \ref{scanpl}. There are two cases to consider:
\begin{itemize}
\item[$(a)$] The set $S$ is empty, $S'$ consists of a single element, and the functor
$X': \cDelta_{S'}^{op} \rightarrow \bfA$ is a constant functor taking value equal to a final object
${\bf 1} \in \bfA$. In this case, $F(\phi)$ is the inclusion from the empty $\bfA$-enriched category
to the $\bfA$-enriched category with a single object (and endomorphisms given by
${\bf 1} \in \bfA$); this is a generator for the class of cofibrations in $\Cat_{\bfA}$.

\item[$(b)$] There exists a cofibration $\phi_0: A \rightarrow B$ in $\bfA$ and an integer $n \geq 0$ such that $\phi$ is the induced map $\Fr^n(A) \rightarrow \Fr^{n}(B)$. In this case, we can identify
$F(S,X)$ with the $\bfA$-enriched category freely generated by objects $x_0, x_1, \ldots, x_n$ and
maps $\{ \psi_i: A \rightarrow \bHom(x_i, x_{i+1}) \}_{0 \leq i < n}$. Similarly, $F(S',X')$ is freely generated by objects $x_0, \ldots, x_n$ and maps $\{ \overline{\psi}_i: B \rightarrow \bHom(x_i, x_{i+1}) \}_{0 \leq i < n}$. The functor $F(S,X) \rightarrow F(S',X')$ is easily seen to be a pushout of $n$ generating cofibrations in the category $\Cat_{\bfA}$.
\end{itemize}
\end{proof}

\begin{lemma}\label{speed}
Suppose given a pair of adjoint functors $\Adjoint{F}{\bfA}{\bfB}{G}$, where $\bfA$ and $\bfB$ are combinatorial model categories. Assume that:
\begin{itemize}
\item[$(1)$] The collection of cofibrations in $\bfA$ is generated (as a weakly saturated class of morphisms) by cofibrations between cofibrant objects.
\item[$(2)$] The functor $F$ preserves cofibrations and preserves weak equivalences between cofibrant objects.
\item[$(3)$] The model categories $\bfA$ and $\bfB$ are left proper.
\item[$(4)$] The collection of weak equivalences in $\bfA$ and $\bfB$ are stable under filtered colimits.
\end{itemize}
Then $F$ and $G$ determine a Quillen adjunction between $\bfA$ and $\bfB$.
\end{lemma}

\begin{proof}
In view of $(2)$, it will suffice to show that the functor $F$ preserves trivial cofibrations.
We first introduce a bit of terminology. We will say that a commutative square
$$ \xymatrix{ X \ar[r]^{f'} \ar[d]^{g} & X' \ar[d]^{g'} \\
Y \ar[r]^{f} & Y' }$$
in ${\bfA}$ is {\it good} if the maps $g$ and $g'$ are weak equivalences, the objects $X$ and $X'$ are cofibrant, and the induced diagram
$$ \xymatrix{ FX \ar[r] \ar[d] & FX' \ar[d] \\
FY \ar[r] & FY' }$$
is a homotopy pushout square in $\bfB$. We will say that a morphism $f: Y \rightarrow Y'$ in
$\bfA$ is {\it good} if, for every trivial fibration $g: X \rightarrow Y$ such that $X$ is cofibrant, there exists a good square
$$ \xymatrix{ X \ar[r]^{f'} \ar[d]^{g} & X' \ar[d]^{g'} \\
Y \ar[r]^{f} & Y'. }$$
We will prove the following:
\begin{itemize}
\item[$(\ast)$] Every cofibration in $\bfA$ is good.
\end{itemize}
Assuming $(\ast)$ for the moment, we can complete the proof as follows. Let $f: Y \rightarrow Y'$ be
a trivial cofibration; we wish to show that $Ff$ is a trivial cofibration. Assumption $(2)$ guarantees that $Ff$ is a cofibration, so it will suffice to show that $Ff$ is a weak equivalence.
Choose a trivial fibration $g: X \rightarrow Y$, where $X$ is cofibrant.
Invoking $(\ast)$, we deduce the existence of a good square
$$ \xymatrix{ X \ar[r]^{f'} \ar[d]^{g} & X' \ar[d]^{g'} \\
Y \ar[r]^{f} & Y'. }$$
Since $f$ is a weak equivalence, a two-out-of-three argument implies that $f'$ is a weak equivalence.
Consider the diagram
$$ \xymatrix{ FX \ar[r]^{Ff'} \ar[d] & FX' \ar[d] \\
FY \ar[r]^{Ff} & FY'. }$$
By assumption, this is a homotopy pushout square in $\bfB$. It will therefore suffice to show that
$Ff'$ is a weak equivalence, which follows from assumption $(2)$.

We now prove $(\ast)$. Let $f: Y \rightarrow Y'$ be a cofibration in $\bfA$. Using assumption $(1)$ and the small object argument, we deduce the existence of a transfinite sequence of objects
$\{ Y_{\beta} \}_{\beta \leq \alpha}$ with the following properties:
\begin{itemize}
\item The object $Y_0$ coincides with $Y$.
\item For every nonzero limit ordinal $\beta \leq \alpha$, we have $Y_{\beta} \simeq \colim_{\gamma < \beta} Y_{\gamma}$.
\item For every ordinal $\beta < \alpha$, we have a pushout diagram
$$ \xymatrix{ Z_{\beta} \ar[r]^{h_{\beta}} \ar[d] & Z'_{\beta} \ar[d] \\
Y_{\beta} \ar[r] & Y_{\beta + 1} }$$
where $h_{\beta}$ is a cofibration between cofibrant objects.
\item The object $Y'$ is a retract of $Y_{\alpha}$ in $\bfA_{Y/}$.
\end{itemize}

Choose a trivial fibration $g: X \rightarrow Y$, where $X$ is cofibrant. We first construct a transfinite sequence of trivial fibrations 
$\{ g_{\beta}: X_{\beta} \rightarrow Y_{\beta} \}_{\beta \leq \alpha}$ as follows:
\begin{itemize}
\item If $\beta = 0$, we set $g_{\beta} = f$.
\item Let $\beta \leq \alpha$ be a nonzero limit ordinal, and let 
$g_{< \beta}: \colim_{ \gamma < \beta} X_{\gamma} \rightarrow Y_{\beta}$ be the colimit of the maps
$g_{\gamma}$ for $\gamma < \beta$. Assumption $(4)$ guarantees that $g_{< \beta}$ is a weak equivalence. It follows that $g_{< \beta}$ admits a factorization
$$ \colim_{ \gamma < \beta} X_{\gamma} \stackrel{ g'_{< \beta}}{\rightarrow} X_{\beta}
\stackrel{g_{\beta}}{\rightarrow} Y_{\beta}$$
where $g'_{< \beta}$ is a trivial cofibration and $g_{\beta}$ is a trivial fibration.
\item Suppose that $\beta < \alpha$. Since $Z_{\beta}$ is cofibrant and the map $h_{\beta}$ is
a trivial fibration, the attaching map $Z_{\beta} \rightarrow Y_{\beta}$ factors through $X_{\beta}$.
Let $X'_{\beta}$ denote the pushout $X_{\beta} \coprod_{ Z_{\beta} } Z'_{\beta}$. Since $\bfA$
is left proper, the induced map $X'_{\beta} \rightarrow Y_{\beta+1}$ is a weak equivalence.
We may therefore choose a factorization
$$ X'_{\beta} \stackrel{p}{\rightarrow} X_{\beta+1} \stackrel{ g_{\beta+1}}{\rightarrow} Y_{\beta+1}$$
where $p$ is a trivial cofibration and $g_{\beta+1}$ is a trivial fibration. 
\end{itemize}
We now prove that for each $\beta \leq \alpha$, the square
$$ \xymatrix{ X \ar[r]^{f'_{\beta}} \ar[d]^{g} & X_{\beta} \ar[d]^{g_{\beta}} \\
Y \ar[r] & Y_{\beta} }$$
is good. It is clear from the construction that the upper horizontal map is a cofibration (which implies
that $X_{\beta}$ is cofibrant, since $X$ is cofibrant by assumption), and that the map $g_{\beta}$ is a trivial fibration. The only nontrivial point is to verify that the induced square
$$ \xymatrix{ FX \ar[r] \ar[d]^{Fg} & FX_{\beta} \ar[d]^{Fg_{\beta}} \\
FY \ar[r] & FY_{\beta} }$$
is a homotopy pushout square in $\bfB$. The proof proceeds by induction on $\beta$. If $\beta = 0$, there is nothing to prove. Suppose that $\beta$ is a nonzero limit ordinal. We have a commutative rectangle
$$ \xymatrix{ FX \ar[r] \ar[d] & F X_{< \beta} \ar[r] \ar[d] & F X_{\beta} \ar[d] \\
FY \ar[r] & F Y_{\beta} \ar[r]^{\id} & F Y_{\beta }.}$$
Using assumption $(4)$ and the inductive hypothesis, we deduce that the left square is a homotopy pushout. Assumption $(2)$ guarantees that the upper horizontal map in the right square is a weak equivalence, so that the right square is also a homotopy pushout. It follows that the outer square is again a homotopy pushout, as desired.

The case of successor ordinals is treated similarly: suppose that $\beta < \alpha$, and consider the diagram
$$ \xymatrix{ FX \ar[r] \ar[d] & F X_{\beta} \ar[r] \ar[d] & F( X_{\beta} \coprod_{Z_{\beta}} Z'_{\beta}) \ar[r] \ar[d] & F X_{\beta+1} \ar[d] \\
FY \ar[r] & FY_{\beta} \ar[r] & FY_{\beta+1} \ar[r]^{\id} & FY_{\beta+1}. }$$
The inductive hypothesis guarantees that the leftmost square is a homotopy pushout, and assumption
$(2)$ guarantees that that the rightmost square is a homotopy pushout. It will therefore suffice to show that
the middle square is a homotopy pushout. Since the functor $F$ preserves colimits, the middle square is a pushout. Moreover, the horizontal maps in the middle square are cofibrations, by virtue of assumption
$(2)$. Since $\bfB$ is left-proper (assumption $(3)$), the desired result follows.

Since $Y'$ is a retract of $Y_{\alpha}$ in $\bfA_{Y/}$, there exists a map $r: Y_{\alpha} \rightarrow Y_{\alpha}$ such that $r^2 = r$ and $r \circ f_{\alpha} = f_{\alpha}$, such that we can identify
$Y'$ with the colimit of the sequence 
$$Y_{\alpha} \stackrel{r}{\rightarrow} Y_{\alpha} \stackrel{r}{\rightarrow} Y_{\alpha} \stackrel{r}{\rightarrow} \ldots.$$ Consider the diagram
$$ \xymatrix{ X \ar[d]^{ f'_{\alpha}} \ar[r]^{f'_{\alpha}} & X_{\alpha} \ar[d]^{g_{\alpha}} \\
X_{\alpha} \ar@{-->}[ur]^{\overline{r}} \ar[r]^{r \circ g_{\alpha}} & Y_{\alpha}. }$$
Since $g_{\alpha}$ is a trivial fibration and $f'_{\alpha}$ is a cofibration, there exists a map
$\overline{r}: X_{\alpha} \rightarrow X_{\alpha}$ making the diagram commute.

We will construct a commutative ladder
$$ \xymatrix{ X \ar[d]^{\id} \ar[r]^{q^0} & X^0 \ar[r]^{q^1} \ar[d]^{s^0} & X^1 \ar[r]^{q^2} \ar[d]^{s^1} & \ldots \\
X \ar[r]^{ f'_{\alpha} } \ar[d] & X_{\alpha} \ar[d]^{g_{\alpha}} \ar[r]^{\overline{r}} & X_{\alpha} \ar[r]^{\overline{r}} \ar[d]^{g_{\alpha}} & \ldots \\
Y \ar[r]^{ f_{\alpha} } & Y_{\alpha} \ar[r]^{r} & Y_{\alpha} \ar[r]^{r} & \ldots }$$
as follows:
\begin{itemize}
\item Let $X^0 = X_{\alpha}$, $q^0 = f'_{\alpha}$, and $s^0 = \id$.
\item For $i \geq 0$, we factor the map $X^i \stackrel{s^i}{\rightarrow} X_{\alpha} \stackrel{ \overline{r}}{\rightarrow} X_{\alpha}$ as a composition
$$ X^i \stackrel{ q^{i+1}}{\rightarrow} X^{i+1} \stackrel{ s^{i+1}}{\rightarrow} X_{\alpha}$$
where $q^{i+1}$ is a cofibration and $s^{i+1}$ is a trivial fibration.
\end{itemize}
Let $X' = \colim_{i} X^i$, so that we have a commutative diagram
$$ \xymatrix{ X \ar[r] \ar[d]^{g} & X' \ar[d]^{g'} \\
Y \ar[r]^{f} & Y'. }$$
We claim that this square is good. By construction, the upper horizontal map is a cofibration,
so that $X'$ is cofibrant. The map $g'$ is a filtered colimit of the compositions $g_{\alpha} \circ s^i$, each of which is a trivial fibration; assumption $(4)$ guarantees that $g'$ is a weak equivalence. 
The only nontrivial point is to guarantee that the diagram
$$ \xymatrix{ FX \ar[r] \ar[d]^{g} & FX' \ar[d] \\
FY \ar[r] & FY' }$$
is a homotopy pushout square in $\bfB$. Assumption $(4)$ guarantees that the collection of homotopy pushout squares in $\bfB$ is stable under filtered colimits; it will therefore suffice to show that for each $i \geq 0$, the outer square in the diagram
$$ \xymatrix{ FX \ar[r] \ar[d]^{g} & FX^{i} \ar[d]^{Fs^i} \\
FX \ar[r] \ar[d] & FX_{\alpha} \ar[d] \\
FY \ar[r] & FY_{\alpha}. }$$
We conclude by observing that the upper square is a homotopy pushout (since $Fs^i$ is a weak equivalence by virtue of assumption $(2)$) and the lower square is a homotopy pushout thanks to our previous efforts.
\end{proof}

\begin{lemma}\label{stepp2}
Let $(S,X)$ be a projectively cofibrant $\bfA$-enriched Segal category, and let $x,y \in S$. Then the canonical map $u: X( [x,y]) \rightarrow \bHom_{ F(S,X)}( x,y)$ is a weak equivalence in $\bfA$.
\end{lemma}

\begin{proof}
Let $\calJ'_{x,y}(S)$ denote the full subcategory of $\calJ_{x,y}(S)$ spanned by those objects of the form
$( (s_0 = x, s_1, \ldots, s_n = y), \{ 0 < n \} )$. Let $H: \calJ'_{x,y}(S) \rightarrow \bfA$ denote the restriction of $H^{X}_{x,y}$ to $\calJ'_{x,y}$. The category $\calJ'_{x,y}(S)$ has a final object, given by
$( (x,y), \{0 < 1\} )$. It follows that we have canonical identifications
$$ \colim H \simeq \hocolim H \simeq X( [x,y] ).$$
Moreover, we can identify the map $u$ with the map
$$\colim H \rightarrow \colim H^{X}_{x,y}$$
induced by the inclusion $\calJ'_{x,y} \subseteq \calJ_{x,y}$. 

We observe that the inclusion $\calJ'_{x,y}(S) \subseteq \calJ_{x,y}(S)$ admits a right adjoint $R$. We have a canonical natural transformation $\alpha: H^X_{x,y} \rightarrow H$ of functors
from $\calJ_{x,y}(S)$ to $\bfA$. The map $R$ induces a morphism
$v: \colim H^{X}_{x,y} \rightarrow \colim H$ which is left inverse to $u$. It will therefore suffice
to show that $v$ is a weak equivalence in $\bfA$. 

Since $(S,X)$ is cofibrant, Proposition \ref{hunk} implies that the diagram $H^{X}_{x,y}$ is projectively cofibrant diagram. It will therefore suffice to show that $R$ induces an
isomorphism $\hocolim H^{X}_{x,y} \rightarrow \hocolim H$ in the homotopy category $\h{\bfA}$.
This map factors as a composition
$$ \hocolim H^{X}_{x,y} \rightarrow \hocolim (H \circ R) \rightarrow \hocolim H.$$
The first map is an isomorphism in $\h{\bfA}$ since the natural tranformation $H^{X}_{x,y} \rightarrow H \circ R$ is a weak equivalence of diagrams $\calJ_{x,y}(S) \rightarrow \bfA$, by virtue of our assumption
that $(S,X)$ is $\bfA$-enriched Segal category. The second map is an isomorphism by Lemma \ref{cuppl}
(the functor $R$ admits a left adjoint, and therefore induces a cofinal map $\Nerve( \calJ_{x,y}(S) ) \rightarrow \Nerve( \calJ'_{x,y}(S) )$).
\end{proof}

\begin{theorem}\label{castle2}
There exists a left proper combinatorial model structure on $\PreSeg{\bfA}$ which may be described as follows:
\begin{itemize}
\item[$(C)$] A morphism $\phi: (S,X) \rightarrow (S',X')$ of $\bfA$-enriched preSegal categories is a {\it cofibration} if it is a projective cofibration in the sense of Definition \ref{scanpl}.
\item[$(W)$] A morphism $\phi: (S,X) \rightarrow (S',X')$ of $\bfA$-enriched preSegal categories is a {\it weak equivalence} if it satisfies the equivalent conditions of Corollary \ref{stafff}.
\item[$(F)$] A morphism $\phi: (S,X) \rightarrow (S',X')$ of $\bfA$-enriched preSegal categories is
a {\it fibration} if it has the right lifting property with respect to all morphisms satisfying
$(C)$ and $(W)$.
\end{itemize}
Moreover:
\begin{itemize}
\item[$(a)$] The collection of weak equivalences in $\PreSeg{\bfA}$ is stable under filtered colimits.
\item[$(b)$] The adjoint functors
$$ \Adjoint{F}{ \PreSeg{\bfA}}{ \Cat_{\bfA}}{G}$$
determine a Quillen equivalence between $\PreSeg{\bfA}$ and $\Cat_{\bfA}$.
\end{itemize}
\end{theorem}

\begin{remark}
We will refer to the model structure of Theorem \ref{castle2} as the {\it projective} model structure
on $\PreSeg{\bfA}$.
\end{remark}

\begin{proof}
To prove that $\PreSeg{\bfA}$ is a left proper combinatorial model category, it will suffice to show that
the hypotheses of Proposition \toposref{goot} are satisfied. We consider each in turn:
\begin{itemize}
\item[$(1)$] The collection of weak equivalences in $\PreSeg{\bfA}$ is perfect (in the sense of
Definition \toposref{perfequiv}). Using Remark \ref{smuther} and the small object argument, we deduce the
existence of an accessible functor $T: \PreSeg{\bfA} \rightarrow \PreSeg{\bfA}$ and a natural transformation
$\alpha: T \rightarrow \id$ with the following property: for every $\bfA$-enriched preSegal category $(S,X)$, the transformation $\alpha$ induces a map $T(S,X) \rightarrow (S,X)$ which is a cofibrant refinement.

It follows that a morphism $\phi$ in $\PreSeg{\bfA}$ is a weak equivalence if and only if the induced map
$(F \circ T)(\alpha)$ is an equivalence of $\bfA$-enriched categories. It follows immediately that
the collection of weak equivalences in $\PreSeg{\bfA}$ is an accessible subcategory of $\bfA^{[1]}$ which satisfies the two-out-of-three property.

It remains to show that the collection of weak equivalences in $\bfA$ is stable under filtered colimit.
Let $\calJ$ be a small filtered category, and suppose that $\alpha: \calF \rightarrow \calF'$ is a natural
transformation of functors $\calF, \calF': \calJ \rightarrow \PreSeg{\bfA}$ such that, for each
$J \in \calJ$, the induced map $\calF(J) \rightarrow \calF'(J)$ is a weak equivalence. We wish to
prove that the induced map $\colim \calF \rightarrow \colim \calF'$ is a weak equivalence.
Let us say that a morphism in $\Fun(\calJ, \PreSeg{\bfA})$ is {\it projective cofibration} if it belongs
to the weakly saturated class of morphisms generated by $\{ i^{J}_{!}(f) \}$, where $f$ ranges
over projective cofibrations in $\PreSeg{\bfA}$, $J$ over the collection of all objects in $\calJ$, and
$i^{J}_{!}$ denotes the functor of left Kan extension along the inclusion $\{J\} \subseteq \calJ$.
We will say that an object $\calG$ of $\Fun(\calJ, \PreSeg{\bfA})$ is {\it projectively cofibrant} if the map
$\emptyset \rightarrow \calG$ is a projective cofibration, where $\emptyset$ denotes the initial object of $\Fun( \calJ, \PreSeg{\bfA})$. Using the small object argument, we deduce the existence of a commutative diagram
$$ \xymatrix{ \overline{\calF} \ar[r] \ar[d] & \overline{\calF}' \ar[d] \\
\calF \ar[r] & \calF' }$$
in $\Fun( \calJ, \PreSeg{\bfA})$, where $\overline{\calF}$ and $\overline{\calF}'$ are projectively cofibrant, and the vertical maps are cofibrant refinements after evaluation at each $J \in \calJ$. 
We observe that $\colim \overline{\calF}$ and $\colim \overline{\calF}'$ are cofibrant objects
of $\PreSeg{\bfA}$. Since the class of weak equivalences in $\bfA$ is stable under filtered colimits, the vertical maps in the diagram
$$ \xymatrix{ \colim \overline{\calF} \ar[r] \ar[d] & \colim \overline{\calF}' \ar[d] \\
\colim \calF \ar[r] & \colim \calF' }$$
are cofibrant refinements. It will therefore suffice to show that the induced map
$F( \colim \overline{\calF} ) \rightarrow F( \colim \overline \calF' )$ is an equivalence of
$\bfA$-enriched categories. Since $F$ commutes with colimits and the collection of equivalences
in $\Cat_{\bfA}$ is stable under filtered colimits, it will suffice to show that $F( \overline{\calF}(J) ) \rightarrow
F( \overline{\calF}'(J) )$ is a weak equivalence for each $J \in \calJ$, which follows from our assumption that $\calF(J) \rightarrow \calF'(J)$ is a weak equivalence in $\PreSeg(\bfA)$.

\item[$(2)$] The collection of weak equivalences is stable under pushouts by generating projective cofibrations.
Let $f: (T, Y) \rightarrow (T', Y')$ be one of the generating projective cofibrations of Definition \ref{scanpl}, and let $\phi: (S,X) \rightarrow (S', X')$ be a weak equivalence. Suppose we are given
a map $\chi: (T, Y) \rightarrow (S,X)$; we wish to prove that the induced map
$$ (S,X) \coprod_{ (T, Y) } (T', Y') \rightarrow (S', X') \coprod_{ (T,Y)} (T', Y')$$
is a weak equivalence. Choose a commutative diagram
$$ \xymatrix{ (S, \overline{X}_0) \ar[r] \ar[d] & ( S', \overline{X}'_0) \ar[d] \\
(S, X) \ar[r]^{\phi} & (S', X') }$$
as in the proof of Corollary \ref{stafff}. We observe that $(T,Y)$ is cofibrant, so we can lift
$\chi$ to a map $\chi_0: (T,Y) \rightarrow (S, \overline{X}_0)$. Consider the induced diagram
$$ \xymatrix{ (S, \overline{X}_0) \coprod_{(T,Y)} (T',Y') \ar[r]^{\psi} \ar[d] & ( S', \overline{X}'_0)  \coprod_{(T,Y)} (T',Y') \ar[d] \\
(S, X)  \coprod_{(T,Y)} (T',Y') \ar[r] & (S', X')  \coprod_{(T,Y)} (T',Y').}$$
Using the left-properness of $\bfA$, we deduce that the vertical maps are cofibrant refinements.
It will therefore suffice to prove that $F(\psi)$ is an equivalence of $\bfA$-enriched categories.
Since $F$ preserves colimits, this is equivalent to the assertion that the map
$$ F(S, \overline{X}_0) \coprod_{ F( T,Y)} F(T',Y') \rightarrow F(S', \overline{X}'_0) \coprod_{ F(T,Y)} F(T',Y')$$
is an equivalence of $\bfA$-enriched categories. Our assumption that $\phi$ is a weak equivalence
guarantees that $F(S, \overline{X}_0) \rightarrow F(S', \overline{X}'_0)$ is an equivalence of $\bfA$-enriched categories. The desired result now follows from Lemma \ref{stade} and the left-properness of
$\Cat_{\bfA}$.

\item[$(3)$] Let $\phi: (S,X) \rightarrow (S', X')$ be a morphism which has the right lifting property with respect to every projective cofibration; we wish to prove that $\phi$ is a weak equivalence. We note that 
$\phi$ induces a surjection $S \rightarrow S'$ and a trivial fibration
$X( [s_0, \ldots, s_n]) \rightarrow X'( [ \phi(s_0), \ldots, \phi(s_n)])$ for every sequence of elements
$s_0, s_1, \ldots, s_n \in S$. In particular, $\phi$ is pointwise fully faithful.

Choose a commutative diagram 
$$ \xymatrix{ (S, \overline{X}_0) \ar[r] \ar[d] & ( S', \overline{X}'_0) \ar[d] \\
(S, X) \ar[r]^{\phi} & (S', X') }$$
as in the proof of Corollary \ref{stafff}. To prove that $\phi$ is a weak equivalence, it will suffice to prove that the induced map $F(S, \overline{X}_0) \rightarrow F(S', \overline{X}'_0)$ is an equivalence of $\bfA$-enriched categories, which follows immediately from Proposition \ref{camble}.
\end{itemize}

We now claim that $F$ is a left Quillen functor. This follows from Lemma \ref{stade}, Lemma \ref{speed}, Remark \ref{soccer}, and the following trivial observation:
\begin{itemize}
\item[$(\ast)$] A morphism $\phi$ between cofibrant objects of $\PreSeg{\bfA}$ is a weak equivalence if and only if $F(\phi)$ is an equivalence of $\Cat_{\bfA}$-enriched categories.
\end{itemize}
We now complete the proof by showing that $F$ induces a Quillen equivalence between
$\PreSeg{\bfA}$ and $\Cat_{\bfA}$. Assertion $(\ast)$ implies that the left derived functor $LF$
is conservative. It will therefore suffice to show that the right derived functor $RG$ is fully faithful.
In other words, we must show that if $\calC$ is a $\bfA$-enriched category, then the counit map
$$ (LF \circ RG): \calC \rightarrow \calC$$
is an isomorphism in the homotopy category $\h{\bfA}$. Without loss of generality, we may assume that $\calC$ is a fibrant $\bfA$-enriched category, so that we can identify $RG(\calC)$ with $G(\calC)$.
Choose a cofibrant refinement $\phi: (S,X) \rightarrow G\calC$; we wish to prove that the adjoint map
$\psi: F(S,X) \rightarrow \calC$ is an equivalence of $\bfA$-enriched categories. Since the functor
$\psi$ is surjective on objects, it will suffice to show that for every pair of objects $s,s' \in S$, the map
$$u: \bHom_{F(S,X)}( s, s' ) \rightarrow \bHom_{\calC}(s,s')$$
is a weak equivalence in $\bfA$. We have a commutative diagram
$$ \xymatrix{ & \bHom_{F(S,X)}(s,s') \ar[dr]^{u} & \\
X( [s,s']) \ar[ur]^{v} \ar[rr]^{w} & & \bHom_{\calC}(s,s'). }$$
The map $w$ is an isomorphism in $\bfA$, and the map $v$ is a weak equivalence by
virtue of Example \ref{stepp1} and Lemma \ref{stepp2}. It follows from the two-out-of-three property that
$u$ is a weak equivalence, as desired.
\end{proof}

\begin{remark}\label{hseg}
Let $(S,X)$ be a $\bfA$-enriched Segal category. We define a new category
$\h{(S,X)}$, the {\it homotopy category of $\h{(S,X)}$}, which is enriched over the homotopy
category $\h{\bfA}$ of $\bfA$:
\begin{itemize}
\item The objects of $\h{(S,X)}$ are the elements of $S$.
\item Given a pair of objects $x,y \in S$, we let $\bHom_{ \h{(S,X)} }(x,y)$ denote the
image of $X( [x,y] ) \in \bfA$ in the homotopy category $\h{\bfA}$.
\item Given a sequence of objects $s_0, \ldots, s_n \in S$, the composition law
$$ \bHom_{ \h{(S,X)}}( s_0, s_1) \times \ldots \times
\bHom_{ \h{(S,X)}}( s_{n-1}, s_n) \rightarrow \bHom_{ \h{(S,X)}}(s_0, s_n)$$
is given by composing the map $X( [s_0, \ldots, s_n]) \rightarrow X([s_0, s_n])$ with
the inverse of the weak equivalence $X( [s_0, \ldots, s_n] ) \rightarrow
X([ s_0, s_1]) \times \ldots \times X( [s_{n-1}, s_n]).$
\end{itemize} 

We observe that if $\calC$ is a $\bfA$-enriched category, then the homotopy category
of $G(\calC)$ is canonically isomorphic to the homotopy category $\h{\calC}$. 
It follows from Lemma \ref{stepp2} that if $(S,X)$ is a cofibrant $\bfA$-enriched Segal category,
then there is a canonical isomorphism $\h{(S,X)} \simeq \h{ F(S,X)}$. 
For a general (not necessarily cofibrant) $\bfA$-enriched Segal category $(S,X)$, we can
identify $\h{(S,X)}$ with the homotopy category $F(S, X')$, where $(S,X')$ is a cofibrant refinement of 
$(S,X)$.
\end{remark}

\begin{remark}\label{schwe}
Combining Remark \ref{hseg} with the definitions of collections of weak equivalences in
$\PreSeg{\bfA}$ and $\Cat_{\bfA}$, we deduce the following result:
\begin{itemize}
\item[$(\ast)$] Let $f: (S, X) \rightarrow (S', X')$ be a map of $\bfA$-enriched Segal categories.
Then $f$ is a weak equivalence in $\PreSeg{\bfA}$ if and only if the induced map
$\h{(S,X)} \rightarrow \h{(S',X')}$ is an equivalence of $\h{\bfA}$-enriched categories.
\end{itemize}
In particular, the fully faithful embedding $\Cat_{\bfA} \hookrightarrow \PreSeg{\bfA}$ preserves weak equivalences.
\end{remark}

\subsection{The Injective Model Structure on $\PreSeg{\bfA}$}\label{bisec2.3}

Our goal in this section is to describe a general context in which we can compare
the theory of Segal categories with the theory of complete Segal Spaces.
We begin by observing that for every model category $\bfA$, there
is a functor
$$ \UnPre: \PreSeg{\bfA} \rightarrow \Fun( \cDelta^{op}, \bfA),$$
which carries a $\bfA$-enriched preSegal category $(S,X)$ to the simplicial object
$\pi_{!} X$, where $\pi: \cDelta^{op}_{S} \rightarrow \cDelta^{op}$ is the forgetful functor
and $\pi_{!}$ is given by left Kan extension along $\pi$. More concretely, 
$\UnPre$ is given by the formula
$$ \UnPre(S,X)_{n} = \coprod_{ s_0, \ldots, s_n \in S} X( [s_0, \ldots, s_n] ).$$
We can now state the main result of this section as follows:

\begin{proposition}\label{curt}
Let $\bfA$ be a combinatorial simplicial model category satisfying the following conditions:
\begin{itemize}
\item[$(a)$] The simplicial model category $\bfA$ is an absolute distributor.
\item[$(b)$] The collection of weak equivalences in $\bfA$ is stable under filtered colimits.
\item[$(c)$] Filtered colimits are left exact in the underlying $\infty$-category $\Nerve( \bfA^{\degree})$.
\item[$(d)$] The final object of $\bfA$ is cofibrant, and determines a compact object
of $\Nerve( \bfA^{\degree})$. 
\item[$(e)$] For every finite collection of objects $\{ X_s \}_{s \in S} \in \bfA$, the
coproduct $\coprod_{s \in S} X_s$ is also a homotopy coproduct
(this is automatic if, for example, every object of $\bfA$ is cofibrant).
\end{itemize}

Then there exists a left proper, combinatorial model structure on the category $\PreSeg{\bfA}$ of $\bfA$-enriched preSegal categories, which may be described as follows:
\begin{itemize}
\item[$(C)$] A map $f: (S, X) \rightarrow (S', X')$ of $\bfA$-enriched preSegal
categories is an {\it injective cofibration} if the map $S \rightarrow S'$
is injective and, for every sequence of elements $s_0, \ldots, s_n \in S$, the induced map
$X( s_0, \ldots, s_n) \rightarrow X'( f(s_0), \ldots, f(s_n) )$ is a cofibration in $\bfA$.

\item[$(W)$] A map $f: (S,X) \rightarrow (S',X')$ of $\bfA$-enriched 
preSegal categories is a {\it weak equivalence} if and only if the induced map
$\UnPre(f)$ is a weak equivalence with respect to the complete Segal model structure
on $\Fun( \cDelta^{op}, \bfA)$.

\item[$(F)$] A map $f: (S,X) \rightarrow (S',X')$ of $\bfA$-enriched preSegal categories
is an {\it injective fibration} if and only if it has the right lifting property with respect to
all morphisms which satisfy $(C)$ and $(W)$.
\end{itemize}

Moreover, the functor $\UnPre: \PreSeg(\bfA) \rightarrow \Fun( \cDelta^{op}, \bfA)$
is a left Quillen equivalence, where $\Fun( \cDelta^{op}, \bfA)$ is endowed with
the complete Segal model structure.
\end{proposition}

\begin{remark}
We will refer to the model structure of Proposition \ref{curt} as the
{\it injective model structure} on $\PreSeg{\bfA}$.
\end{remark}

We will give the proof of Proposition \ref{curt} after establishing a few preliminary results.

\begin{remark}
Let $\bfA$ be a model category, and let $f: \Set \rightarrow \bfA$ be the functor described by the formula
$f(S) = \coprod_{s \in S} {\bf 1}$. Suppose that $\bfA$ satisfies the following conditions:
\begin{itemize}
\item[$(a)$] The functor $f$ is fully faithful.
\item[$(b)$] For every set $S$, the coproduct functor
$$ \prod_{s \in S} \bfA \simeq \prod_{s \in S} \bfA_{/ {\bf 1} }
\stackrel{ \coprod}{\rightarrow} \bfA_{/f(S)}$$
is an equivalence of categories. 
\end{itemize}
Then the functor $\UnPre: \PreSeg{ \bfA} \rightarrow \Fun( \cDelta^{op}, \bfA)$ is fully faithful.
Its essential image consists of those simplicial objects $X_{\bigdot}$ of $\bfA$ such that
$X_0$ lies in the essential image of $f$. Conditions $(a)$ and $(b)$ are satisfied in
many cases of interest: for example, if $\bfA$ is the category of simplicial sets, or the category of marked simplicial sets.
\end{remark}

\begin{lemma}\label{swagg}
Let $\bfA$ be as in Proposition \ref{curt}, and let $f: A \rightarrow B$ be a cofibration in $\bfA$.
Let $\psi: \Fr^{n}(f) \rightarrow \Fr^{n}(B)$ be the morphism of $\bfA$-enriched preSegal categories
described in Example \ref{swaggius}. Then the induced map $\UnPre(f)$ is a trivial cofibration
in $\Fun( \cDelta^{op}, \bfA)$ (with respect to the complete Segal model structure).
\end{lemma}

\begin{proof}
It is clear that $\UnPre(f)$ is a cofibration in $\Fun( \cDelta^{op}, \bfA)$. To prove that it is a trivial cofibration, it will suffice to show that for every fibrant object $X_{\bigdot}$ of
$\Fun( \cDelta^{op}, \bfA)$, the induced map
$$ \phi: \bHom_{\Fun( \cDelta^{op}, \bfA)}(\UnPre \Fr^{n}(B), X_{\bigdot} ) \rightarrow \bHom_{\Fun(\cDelta^{op}, \bfA)}(\UnPre \Fr^{n}(f), X_{\bigdot})$$
is a trivial fibration of simplicial sets. Unwinding the definitions, we see that $\phi$ is a pullback
of the map
$$ \phi': \bHom_{\bfA}(B, X_n) \rightarrow \bHom_{\bfA}(A, X_n)
\times_{ \bHom_{\bfA}(A, X_1 \times_{X_0} \times \ldots \times_{X_0} X_1)}
\bHom_{\bfA}(B, X_1 \times_{X_0} \times \ldots \times_{X_0} X_1).$$
This map is a trivial fibration, since $f$ is a cofibration by assumption and the map
$X_n \rightarrow X_1 \times_{X_0} \ldots \times_{X_0} X_1$ is a trivial fibration
(Remark \ref{sweetuma}).
\end{proof}

\begin{lemma}
Let $\bfA$ satisfy the hypotheses of Proposition \ref{curt}, and let
$f: (S,X) \rightarrow (S,Y)$ be a map of $\bfA$-enriched preSegal categories
which is the identity on objects, such that the induced map
$X[ s_0, \ldots, s_n] \rightarrow Y[s_0, \ldots, s_n]$ is a weak equivalence in
$\bfA$ for every sequence $s_0, \ldots, s_n \in S$. The the induced map
$\UnPre(f)$ is a weak equivalence in $\Fun( \cDelta^{op}, \bfA)$.
\end{lemma}

\begin{proof}
It follows immediately from assumption $(e)$ of Proposition \ref{curt} that
$\UnPre(f)$ is a levelwise equivalence of simplicial objects of $\bfA$.
\end{proof}

\begin{example}\label{haley}
Let $\bfA$ be as in Proposition \ref{curt}, let
$i: A \rightarrow B$ be a trivial cofibration in $\bfA$. and let
$f: \Fr^{n}(A) \rightarrow \Fr^{n}(B)$ be the induced map of
$\bfA$-enriched preSegal categories. Then $\UnPre(f)$ is a trivial
cofibration in $\Fun( \cDelta^{op}, \bfA)$ (with respect to the injective model structure,
and therefore with respect to the complete Segal model structure).
\end{example}

\begin{remark}\label{sinker}
Using Lemma \ref{swagg}, Example \ref{haley} and the small object argument, we deduce that
for every $\bfA$-enriched preSegal category $(S,X)$, there exists a map
$f: (S,X) \rightarrow (S,Y)$ with the following properties:
\begin{itemize}
\item[$(i)$] The morphism $\UnPre(f)$ is a trivial cofibration with respect
to the complete Segal model structure on $\Fun( \cDelta^{op}, \bfA)$.
\item[$(ii)$] For every pair of elements $s,s' \in S$, the object
$Y[s,s'] \in \bfA$ is fibrant.
\item[$(iii)$] For every sequence of elements $s_0, \ldots, s_n \in S$, the map
$Y[s_0, \ldots, s_n] \rightarrow Y[s_0, s_1] \times \ldots \times Y[s_{n-1}, s_n]$
is a trivial fibration in $\bfA$.
\item[$(iv)$] For every sequence of objects $s_0, \ldots, s_n \in S$, the map
$X[s_0, \ldots, s_n] \rightarrow Y[s_0, \ldots, s_n]$ is a cofibration in $\bfA$.
\end{itemize}
\end{remark}

\begin{lemma}\label{strongbad1}
Let $\bfA$ be a simplicial model category satisfying the hypotheses of Proposition
\ref{curt}, and let $f: (S,X) \rightarrow (S',X')$ be a map of $\bfA$-enriched preSegal
categories satisfying the following conditions:
\begin{itemize}
\item[$(1)$] The underlying map of sets $S \rightarrow S'$ is surjective.
\item[$(2)$] For every sequence of elements $s_0, \ldots, s_n \in S$, the induced map
$X( s_0, \ldots, s_n) \rightarrow X'( f(s_0), \ldots, f(s_n) )$ is a weak equivalence in $\bfA$.
\end{itemize}
Then $f$ satisfies condition $(W)$ of Proposition \ref{curt}.
\end{lemma}

\begin{proof}
Choose a map $(S,X) \rightarrow (S,Y)$ as in Remark \ref{sinker}.
We have a pushout diagram
$$ \xymatrix{ (S,X) \ar[r] \ar[d] & (S,Y) \ar[d] \\
(S',X') \ar[r] & (S,Y) \coprod_{ (S,X) } (S',X'). }$$
Remark \ref{sinker} guarantees that the functor $\UnPre$ carries the upper horizontal map
to a trivial cofibration. It therefore carries the lower horizontal map to a trivial cofibration as well.
Since $\bfA$ is left proper, the right vertical map continues to satisfy assumptions
$(1)$ and $(2)$. We may therefore replace $(S,X)$ by $(S,Y)$, and thereby assume that the map
$X[s_0, \ldots, s_n] \rightarrow X[s_0, s_1] \times \ldots \times X[s_{n-1}, s_n]$ is
a trivial fibration between fibrant objects of $\bfA$. It follows from $(1)$ and $(2)$
that for every sequence of objects $s'_0, \ldots, s'_n \in S'$, the functor $X'$
exhibits $X'[ s'_0, \ldots, s'_n] \in \bfA$ as a homotopy product of the objects
$\{ X'[ s'_{i-1}, s'_i] \in \bfA \}_{1 \leq i \leq n}$. In other words, we may assume that
$(S,X)$ and $(S',X')$ are $\bfA$-enriched Segal categories.

Let $\calY = \Nerve( \bfA^{\degree})$ denote the underlying $\infty$-category of $\bfA$, and let
$Z_{\bigdot}$ and $Z'_{\bigdot}$ be the simplicial objects of $\calY$ determined by
$\UnPre(S,X)$ and $\UnPre(S',X')$. Using the assumption that $\calY$ is an absolute distributor
and that $(S,X)$ and $(S',X')$ are $\bfA$-enriched Segal categories, we deduce that 
$Z_{\bigdot}$ and $Z'_{\bigdot}$ are Segal space objects of $\calY$. To complete the
proof, it will suffice to show that $f$ induces a Segal equivalence $Z_{\bigdot} \rightarrow
Z'_{\bigdot}$ (Theorem \ref{segmin}). This follows immediately from the criterion of
Remark \ref{swammer}.
\end{proof}

\begin{proof}[Proof of Proposition \ref{curt}]
To deduce the existence of the desired model structure on $\PreSeg{\bfA}$, we will apply
the criterion of Proposition \toposref{goot}. There are only three nontrivial points to check:
\begin{itemize}
\item[$(1)$] The collection of weak equivalences in $\PreSeg{\bfA}$ is stable under filtered colimits.
This follows from Proposition \ref{smashet}, since the functor $\UnPre$ preserves all colimits.
\item[$(2)$] The collection of weak equivalences in $\PreSeg{\bfA}$ is stable under pushouts by injective cofibrations. This follows from the left properness of the complete Segal model structure on $\Fun( \cDelta^{op}, \bfA)$. 
\item[$(3)$] Let $f: (S,X) \rightarrow (S', X')$ be a map of $\bfA$-enriched preSegal categories which has the right lifting property with respect to all injective cofibrations. We must show that $f$ is a weak equivalence. We observe that each of the generating cofibrations of Definition 
\ref{scanpl} is an injective cofibration. Consequently, the map $S \rightarrow S'$ is surjective, and
each of the maps $X(s_0, \ldots, s_n) \rightarrow X'( f(s_0), \ldots, f(s_n) )$ is a trivial fibration
in $\bfA$. The desired result now follows from Lemma \ref{strongbad1}.
\end{itemize}

To complete the proof of Proposition \ref{curt}, it will suffice to show that
$\UnPre$ is a left Quillen equivalence. The functor $\UnPre$ obviously preserves cofibrations and weak equivalences. Let ${\bf 1}$ denote the final object of $\bfA$.
We observe that $\UnPre$ admits a right adjoint
$G: \Fun( \cDelta^{op}, \bfA) \rightarrow \PreSeg{\bfA}$, which may be described as follows:
for every simplicial object $Y_{\bigdot}$ of $\bfA$, we let
$G(Y_{\bigdot}) = (S, X)$, where $S = \Hom_{\bfA}( {\bf 1}, Y_0)$ and
$X( [ s_0, \ldots, s_n] ) = Y_{n} \times_{ Y_0^{n+1}} {\bf 1},$ where the map
from ${\bf 1}$ to $Y_0^{n+1}$ is determined by the sequence $s_0, \ldots, s_n \in 
S = \Hom_{\bfA}( {\bf 1}, Y_0)$. It follows that $\UnPre$ is a left Quillen functor.

By construction, the left derived functor $L \UnPre$ is conservative.
Consequently, to show that $\UnPre$ is a left Quillen equivalence, it will suffice
to show that the counit map $L \UnPre \circ RG \rightarrow \id$ is
an isomorphism of functors from the homotopy category $\h{ \Fun( \cDelta^{op}, \bfA) }$
to itself. Because $\UnPre$ preserves weak equivalences, it can be identified
with its own left derived functor. Consequently, it suffices to prove the following:

\begin{itemize}
\item[$(\ast)$] Let $X_{\bigdot}$ be a fibrant object of $\Fun( \cDelta^{op}, \bfA)$.
Then the induced map $\UnPre G(X_{\bigdot}) \rightarrow X_{\bigdot}$ is
a weak equivalence in $\Fun( \cDelta^{op}, \bfA)$. 
\end{itemize}

Let $\calY = \Nerve( \bfA^{\degree} )$ denote the underlying $\infty$-category of
$\bfA$. Let $Y_{\bigdot}$ denote the simplicial object of $\calY$ determined
by $X_{\bigdot}$, and $Y'_{\bigdot}$ the simplicial object determined by
$\UnPre G(X_{\bigdot})$. We first claim that $Y'_{\bigdot}$ is a Segal space object
of $\calY$. Let $\calX \subseteq \calY$ be the essential image of a functor
$\SSet \rightarrow \calY$ which preserves small colimits and final objects;
we first claim that $Y'_0 \in \calX$. By construction, $(\UnPre G(X_{\bigdot}))_0
= \coprod_{s \in S} {\bf 1} \in \bfA$, where ${\bf 1}$ denotes a final object
of $\bfA$ and $S = \Hom_{\bfA}( {\bf 1}, X_0)$. Assumption
$(e)$ guarantees that this coproduct is also a homotopy coproduct, so that
$Y'_0$ is a coproduct in $\calY$ of final objects. Since the inclusion
$\calX \subseteq \calY$ preserves final objects and small coproducts, we conclude
that $Y'_0 \in \calX$ as desired.

We next claim that $Y'_{\bigdot}$ is a category object of $\calY$. To prove this, we must show
that for each $n \geq 0$, the canonical map
$$ \phi: Y'_n \rightarrow Y'_1 \times_{Y'_0} Y'_1 \times_{Y'_0} \ldots \times_{Y'_0} Y'_1$$
is an equivalence. Using $(e)$ and the assumption that $\calX \subseteq \calY$
is a distributor, we deduce that $\phi$ is a coproduct of maps
$$ \phi_{s_0, \ldots, s_n}: Z(s_0, \ldots, s_n) \rightarrow Z(s_0, s_1) \times \ldots
\times Z(s_{n-1}, s_n),$$
where $Z(t_0, \ldots, t_k)$ denotes the object of $\calY$ associated
to the fiber product $X_{k} \times_{ X_0^{k+1}} {\bf 1} \in \bfA$,
where the map ${\bf 1} \rightarrow X_0^{k+1}$ is given by $(t_0, \ldots, t_k)$.
Since $X_{\bigdot}$ is injectively fibrant, the map $X_k \rightarrow X_0^{k+1}$
is a fibration between fibrant objects of $\bfA$, so that the fiber
product is also a homotopy fiber product (Proposition \toposref{leftpropsquare}).
It follows that $\phi_{s_0, \ldots, s_n}$ is a pullback of the map $Y_{n} \rightarrow Y_1 \times_{Y_0} \ldots \times_{Y_0} Y_1$, which is an equivalence because $Y_{\bigdot}$ is a Segal
space object of $\bfA$.

To complete the proof of $(\ast)$, it will suffice to show that the map $Y'_{\bigdot} \rightarrow Y_{\bigdot}$ is a Segal equivalence (Theorem \ref{segmin}). In view of Remark \ref{swammer}, this will follow
if we show that the map $Y'_0 \rightarrow Y_0$ is an effective epimorphism in $\calX$.
Since $\calX \simeq \SSet$, this is equivalent to the following assertion: every map 
in the homotopy category $\h{\calX}$ from the final object to $Y_0$ factors through $Y'_0$.
Using assumption $(d)$, we deduce that any such map is represented by a morphism
${\bf 1} \rightarrow X_0$ in $\bfA$. The corresponding element of $S$ determines the desired
factorization.
\end{proof}

We conclude this section by comparing the injective model structure of
Proposition \ref{curt} with the projective model structure introduced in
\S \ref{bisec2.2}. These model structures are Quillen equivalent to one another
provided that we are in a situation where both are well-defined:

\begin{proposition}\label{curt2}
Let $\bfA$ be a combinatorial simplicial model category satisfying the following conditions:
\begin{itemize}
\item[$(a)$] The collection of weak equivalences in $\bfA$ is stable under filtered colimits.
\item[$(b)$] The underlying $\infty$-category $\Nerve( \bfA^{\degree})$ is an absolute distributor.
\item[$(c)$] Every object of $\bfA$ is cofibrant.
\item[$(d)$] Filtered colimits are left exact in the underlying $\infty$-category
$\Nerve( \bfA^{\degree})$.
\item[$(e)$] The final object of $\Nerve( \bfA^{\degree})$ is compact.
\item[$(f)$] For every object $X \in \bfA$, the functor $Y \mapsto X \times Y$ preserves small colimits.
\item[$(g)$] The Cartesian product on $\bfA$ endows $\bfA$ with the structure of a monoidal model category. In other words, given a pair of cofibrations $f: A \rightarrow A'$, $g: B \rightarrow B'$, the induced map
$$ f \wedge g: (A \times B') \coprod_{ A \times B} (A' \times B) \rightarrow A' \times B'$$
is again a cofibration, which is trivial if either $f$ or $g$ is trivial.
\end{itemize}
Then the identity functor $\id$ determines a left Quillen equivalence
from $\PreSeg{ \bfA}$ (endowed with the projective model structure of Theorem \ref{castle2})
to $\PreSeg{\bfA}$ (endowed with the injective model structure of Proposition \ref{curt2}).
\end{proposition}

\begin{remark}
The hypotheses of Proposition \ref{curt2} are satisfied if
$\bfA$ is the category $\sSet$ of simplicial sets (with the Kan model structure), or if
$\bfA$ is the category $\mSet$ of marked simplicial sets (with the model structure of
\S \toposref{markmodel}).
\end{remark}

We will give the proof of Proposition \ref{curt2} after establishing a few preliminary results.
We first note that the left properness of the projective model structure on $\PreSeg{\bfA}$
can be strengthened as follows:

\begin{proposition}
Let $\bfA$ be a combinatorial model category satisfying assumptions $(A1)$ through
$(A4)$ of \S \ref{bisec2.2}. Suppose given a diagram $\sigma$:
$$ \xymatrix{ (S,X) \ar[r] \ar[d] & (S', X') \ar[d] \\
(S,Y) \ar[r] & (S', Y') }$$
of $\bfA$-enriched preSegal categories satisfying the following conditions:
\begin{itemize}
\item[$(1)$] The vertical maps are the identity on objects.
\item[$(2)$] For every sequence of objects $s'_0, \ldots, s'_n \in S'$, 
the diagram
$$ \xymatrix{ \coprod_{ s_0, \ldots, s_n} X( [s_0, \ldots, s_n]) \ar[r] \ar[d] & X'( [s'_0, \ldots, s'_n]) \ar[d] \\
\coprod_{ s_0, \ldots, s_n} Y( [s_0, \ldots, s_n] ) \ar[r] & Y'( [s'_0, \ldots, s'_n] ) }$$
is a homotopy pushout square in $\bfA$. Here the coproducts are taken over all
$s_0, \ldots, s_n \in S$ lifting the sequence $s'_0, \ldots, s'_n \in S'$.
\end{itemize}
Then $\sigma$ is a homotopy pushout diagram with respect to the projective
model structure on $\PreSeg{\bfA}$.
\end{proposition}

\begin{proof}
Using the small object argument, we can choose a map
$(S, \overline{X}) \rightarrow (S,X)$ with the following properties:
\begin{itemize}
\item[$(a)$] Let ${\bf 1}$ denote the final object in $\PreSeg{\bfA}$. Then
the canonical map $\coprod_{s \in S} {\bf 1} \rightarrow (S, \overline{X})$ is an
iterated pushouts of morphisms of type $(b)$ appearing in Definition \ref{scanpl}. 
\item[$(b)$] For every sequence of elements $s_0, \ldots, s_n \in S$, the induced map
$\overline{X}( [s_0, \ldots, s_n] ) \rightarrow X( [s_0, \ldots, s_n] )$ is a trivial fibration in $\bfA$.
\end{itemize}
In particular, $(S, \overline{X})$ is a cofibrant refinement of $(S, X)$. Choose a cofibrant
refinement $(S', \overline{X}') \rightarrow (S', X')$ similarly. Using the small object
argument again, we can factor the map $(S, \overline{X}) \rightarrow (S, Y)$ as a composition
$$ (S, \overline{X}) \rightarrow (S, \overline{Y}) \rightarrow (S, Y)$$
with the following properties:
\begin{itemize}
\item[$(c)$] The map $(S, \overline{X}) \rightarrow (S, \overline{Y})$ is an iterated
pushout of morphisms of type $(b)$ appearing in Definition \ref{scanpl}.
\item[$(d)$] For every sequence of objects $s_0, \ldots, s_n \in S$, the map
$\overline{Y}( [s_0, \ldots, s_n] ) \rightarrow Y( [s_0, \ldots, s_n] )$ is a trivial fibration in $\bfA$.
\end{itemize}
Let $(S', \overline{Y}')$ denote the pushout $(S, \overline{X}') \coprod_{ (S, \overline{X} )} (S, \overline{Y})$. Condition $(2)$ guarantees that the map $(S', \overline{Y}') \rightarrow (S', Y')$ is a cofibrant refinement. Consequently, it will suffice to show that the diagram
$$ \xymatrix{ (S, \overline{X}) \ar[r] \ar[d] & (S', \overline{X}') \ar[d] \\
(S,\overline{Y}) \ar[r] & (S', \overline{Y}') }$$
is a homotopy pushout square in $\PreSeg{\bfA}$. This follows from the left-properness of
$\PreSeg{\bfA}$, since the vertical maps are cofibrations by construction.
\end{proof}

\begin{corollary}\label{curry}
Let $\bfA$ be a combinatorial model category satisfying assumptions $(A1)$ through
$(A4)$ of \S \ref{bisec2.2}. 
Let $f: (S, X) \rightarrow (S', Y')$ be a map of $\bfA$-enriched preSegal categories with the following
properties:
\begin{itemize}
\item[$(a)$] The map $f$ is bijective on objects.
\item[$(b)$] The map $f$ is a weak equivalence.
\item[$(c)$] For every sequence of elements $s_0, \ldots, s_n \in S$, the induced map
$X( [s_0, \ldots, s_n] ) \rightarrow X([ f(s_0), \ldots, f(s_n)] )$ is a cofibration.
\end{itemize}
Then any pushout of $f$ is again a weak equivalence (with respect to the projective
model structure on $\PreSeg{\bfA}$).
\end{corollary}

\begin{remark}\label{coper}
The collection of morphisms $f$ which satisfy the hypotheses of Corollary \ref{curry} is evidently stable under retracts and transfinite composition (since the collection of weak equivalences in $\PreSeg{\bfA}$ is stable under filtered colimits). It follows from Corollary \ref{curry} that this collection is weakly saturated. (In the situation of Proposition \ref{curt2}, it is precisely the collection of trivial cofibrations
with respect to the injective model structure on $\PreSeg{\bfA}$.)
\end{remark}

\begin{lemma}\label{sacurt}
Let $\bfA$ be as in Proposition \ref{curt2}. Then the functor
$\UnPre: \PreSeg{\bfA} \rightarrow \Fun( \cDelta^{op}, \bfA)$ carries
projective weak equivalences in $\PreSeg{\bfA}$ to weak equivalences
in $\Fun( \cDelta^{op}, \bfA)$ (with respect to the complete Segal model structure).
\end{lemma}

\begin{proof}
Let $W$ denote the collection of all morphisms in $\PreSeg{\bfA}$ of the type
$\psi: \Fr^{n}(f) \rightarrow \Fr^{n}(B)$ described in Example \ref{swaggius}, where
$f: A \rightarrow B$ ranges over a collection of generating cofibrations for $\bfA$.
Let $\overline{W}$ denote the weakly saturated class of morphisms generated by
$W$. We now observe:
\begin{itemize}
\item[$(a)$] The functor $\UnPre$ carries every morphism in $W$ to a trivial cofibration in
$\Fun( \cDelta^{op}, \bfA)$, and therefore carries every morphism in $\overline{W}$ to a trivial cofibration in $\Fun( \cDelta^{op}, \bfA)$.

\item[$(b)$] Every morphism in $\overline{W}$ is a weak equivalence in $\PreSeg{\bfA}$;
this follows from Example \ref{swaggius} and Remark \ref{coper}.
\end{itemize}

Let $f: (S,X) \rightarrow (S', X')$ be a weak equivalence in $\PreSeg{\bfA}$. 
We wish to prove that $\UnPre(f)$ is a weak equivalence. Using the small
object argument, we can choose a commutative diagram
$$ \xymatrix{ (S, X ) \ar[r]^{f} \ar[d] & (S', X') \ar[d] \\
(S, \overline{X}) \ar[r]^{\overline{f}} & (S', \overline{X}') }$$
where the vertical maps belong to $\overline{W}$ and the objects
$(S, \overline{X})$ and $(S', \overline{X}')$ have the extension property with respect
to each morphism in $W$. In view of $(a)$, it will suffice to prove that
$\UnPre( \overline{f} )$ is a weak equivalence. In view of $(b)$, the map
$f'$ is itself a weak equivalence. We may therefore replace
$(S,X)$ and $(S', X')$ by $(S, \overline{X})$ and $(S', \overline{X}')$, and thereby reduce to the case where $(S,X)$ and $(S', X')$ have the extension property with respect to every morphism in $W$.
Unwinding the definitions, we conclude that for every sequence of elements $s_0, \ldots, s_n \in S$,
the map $X( [s_0, \ldots, s_n]) \rightarrow X( [s_0, s_1]) \times \ldots \times X( [s_{n-1}, s_n])$ is
a trivial fibration in $\bfA$. In particular, the pair $(S,X)$ is a $\bfA$-enriched Segal category;
similarly, $(S', X')$ is a $\bfA$-enriched Segal category. 

Since the map $f$ is a weak equivalence in $\PreSeg{\bfA}$, Remark \ref{schwe} implies
that the induced map of homotopy categories $\h{(S,X)} \rightarrow \h{(S',X')}$ is an equivalence of
$\h{\bfA}$-enriched categories.

Let $\calY$ denote the underlying $\infty$-category $\Nerve( \bfA^{\degree} )$ of $\bfA$.
Since $\calY$ is an absolute distributor, there exists a fully faithful functor $t: \SSet \rightarrow \calY$ which preserves colimits and final objects. Let $\calX$ denote the essential image of $t$, so that
$\calX \subseteq \calY$ is a distributor. Let $Y_{\bigdot}$ and $Y'_{\bigdot}$ denote simplicial objects of $\calY$ determined by $\UnPre(S,X)$ and $\UnPre(S',X')$, respectively. 
For every pair of objects $x,y \in S$, let $H(x,y)$ denote the object of $\calY$ corresponding to
$X([x,y]) \in \bfA$, and define $H'(x',y') \in \calY$ for $x',y' \in S'$ similarly. By construction, we have
for each $n \geq 0$ canonical equivalences
$$ Y_n \simeq \coprod_{ s_0, \ldots, s_n \in S} H(s_0, s_1) \times \ldots \times H(s_{n-1}, s_n)$$
$$ Y'_n \simeq \coprod_{ s'_0, \ldots, s'_n \in S'} H'(s'_0, s'_1) \times \ldots \times H'( s'_{n-1}, s'_n).$$
In particular, the objects $Y_0$ and $Y'_0$ are coproducts of final objects of $\calY$, and therefore belong to $\calX$. Using the fact that $\calX \subseteq \calY$ is a distributor, we deduce
that $Y_{\bigdot}$ and $Y'_{\bigdot}$ are Segal space objects of $\calY$. 
To complete the proof, it will suffice to show that the map $Y_{\bigdot} \rightarrow Y'_{\bigdot}$ is a Segal equivalence. 

We first show that the map $\phi: Y_1 \rightarrow Y_0 \times_{Y'_0} Y'_1 \times_{Y'_0} Y_0$
is an equivalence in $\calY$. Using the fact that $\calX \subseteq \calY$ is a distributor, 
we deduce that the right hand side can be identified with the coproduct
$\coprod_{ x,y \in S} H'(f(x),f(y)).$ Under this identification, the map $\phi$ corresponds to the coproduct of the maps $H(x,y) \rightarrow H'( f(x), f(y) )$, which is an equivalence because 
$\h{(S,X)} \rightarrow \h{(S',X')}$ is a fully faithful functor between $\h{\bfA}$-enriched categories.

To complete the proof that $Y_{\bigdot} \rightarrow Y'_{\bigdot}$ is a Segal equivalence, it will suffice (by Remark \ref{swamm}) to show that the map $Y_0 \rightarrow | \Gp Y'_{\bigdot} |$ is an effective epimorphism in the $\infty$-topos $\calX$. Under the equivalence of $\calX$ with
$\SSet$, the object $Y_0$ corresponds to the discrete space $S$, while $| \Gp Y'_{\bigdot} |$ correponds to some Kan complex $K$. We wish to prove that the canonical map
$S \rightarrow \pi_0 K$ is surjective. 

Choose a connected component $\eta \in \pi_0 K$; we wish to prove that
$\eta$ lies in the image of the map $S \rightarrow \pi_0 K$.
The effective epimorphism $Y'_0 \rightarrow | \Gp Y'_{\bigdot} |$ determines a map
$S' \rightarrow K$ which is surjective on connected components, so there exists
$s' \in S'$ whose image in $\pi_0 K$ coincides with $\eta$. Since the functor $\h{ (S,X)} \rightarrow \h{(S',X')}$ is an equivalence of categories, there exists an element $s \in S$ whose image in $S'$ is
isomorphic to $s'$. It follows that the image of $s$ in $\pi_0 K$ also coincides with $\eta$, so that
$\eta$ lies in the image of $S$ as desired.
\end{proof}

\begin{proof}[Proof of Proposition \ref{curt2}]
It is easy to see that every projective cofibration is an injective cofibration.
Combining this observation with Lemma \ref{sacurt}, we deduce that the
identity functor is a left Quillen functor from the projective model structure
on $\PreSeg{\bfA}$ to the injective model structure on $\PreSeg{\bfA}$.
To complete the proof, we wish to show that the identity functor is
a left Quillen equivalence. In view of Proposition \ref{curt}, it will suffice
to show that the functor $\UnPre: \PreSeg{\bfA} \rightarrow \Fun( \cDelta^{op}, \bfA)$
is a left Quillen equivalence from the projective model structure on
$\PreSeg{\bfA}$ to the complete Segal model structure on $\Fun( \cDelta^{op}, \bfA)$.

Let $G$ denote a right adjoint to $\UnPre$. To prove that $(\UnPre, G)$ is a Quillen equivalence, we consider an arbitrary object $(S,X) \in \PreSeg{\bfA}$ and a fibrant object $Y_{\bigdot} \in \Fun( \cDelta^{op}, \bfA)$. We will to show that a map $\phi: (S,X) \rightarrow G Y_{\bigdot}$ is a weak equivalence
in $\PreSeg{\bfA}$ if and only if the adjoint map $\psi: \UnPre(S,X) \rightarrow Y_{\bigdot}$ is
a weak equivalence in $\Fun( \cDelta^{op}, \bfA)$. Let $W$ and $\overline{W}$ be the classes of morphisms in $\PreSeg{\bfA}$ defined in Lemma \ref{sacurt}. Using the small object argument, we can choose a morphism $u: (S, X) \rightarrow (S, X')$ belonging to $\overline{W}$ such that
$(S,X')$ has the extension property with respect to every morphism in $W$. Lemma \ref{swagg} implies that $\psi$ factors
through $(S, X')$ and that the map $\UnPre(u)$ is a weak equivalence in $\Fun( \cDelta^{op}, \bfA)$, while Example \ref{swaggius} and Remark \ref{coper} imply that $u$ is a weak equivalence
in $\PreSeg{\bfA}$. We may therefore replace $(S,X)$ by $(S,X')$ and thereby reduce to the case where
$(S,X)$ is a $\bfA$-enriched Segal category. 

In view of Remark \ref{schwe}, the map $\phi$ is a weak equivalence if and only if the
induced map of homotopy categories $\h{(S,X)} \rightarrow \h{(GY_{\bigdot})}$ is an equivalence of $\bfA$-enriched categories (note that $GY_{\bigdot}$ has the extension property with respect
to every morphism in $W$ by Lemma \ref{swagg}, and is therefore a $\bfA$-enriched Segal category). Unwinding the definitions, this amounts to the following pair of conditions:
\begin{itemize}
\item[$(i)$] For every pair of elements $x,y \in S$, the induced map
$$X([x,y]) \rightarrow {\bf 1} \times_{Y_0} Y_1 \times_{Y_0} \{ \bf 1 \}$$
is a weak equivalence in $\bfA$.
\item[$(ii)$] For every morphism $\eta: {\bf 1} \rightarrow Y_0$, there exists
an element $s \in S$ whose image in $GY_0$ is equivalent to $\eta$ in the homotopy
category $\h{(GY_{\bigdot})}$. 
\end{itemize}

As in the proof of Lemma \ref{sacurt}, we note that $\UnPre(S,X)$ determines a Segal space object $Z_{\bigdot}$ in the $\infty$-category $\calY = \Nerve( \bfA^{\degree})$. Similarly, $Y_{\bigdot}$ determines
a complete Segal space object $Z'_{\bigdot}$ of $\Nerve( \bfA^{\degree})$, and the map
$\psi$ determines a natural transformation $\overline{\psi}: Z_{\bigdot} \rightarrow Z'_{\bigdot}$. 
Let $\calX \subseteq \calY$ denote the full subcategory generated under colimits by the final object, so that $\calX$ is equivalent to the $\infty$-category of spaces and the inclusion $\calX \subseteq \calY$ is a distributor. We note that $\psi$ is a weak equivalence if and only if $\overline{\psi}$ is a Segal equivalence. In view of Remark \ref{swamm}, this is equivalent to the following pair of conditions:
\begin{itemize}
\item[$(i')$] The canonical map
$$ Z_1 \rightarrow Z_0 \times_{Z'_0} Z'_1 \times_{ Z'_0} Z_0$$
is an equivalence in $\Nerve( \bfA^{\degree} )$. 
\item[$(ii')$] The map $Z_0 \rightarrow | \Gp Z'_{\bigdot} |$ is an effective epimorphism
in the $\infty$-topos $\calX$.
\end{itemize}

Using the fact that $Z_0$ is equivalent to the coproduct $\coprod_{s \in S} {\bf 1} \in \calX$
and the fact that $\calX \subseteq \calY$ is a distributor, we conclude that $(i')$ is equivalent
to $(a)$. To prove $(ii')$, we note that under the equivalence $\calX \simeq \SSet$,
the map $Z_0 \rightarrow | \Gp Z'_{\bigdot} |$ corresponds to a map from
the discrete space $S$ to a Kan complex $K$; assertion $(ii')$ is equivalent to the requirement that
the map $S \rightarrow \pi_0 K$ is surjective. We conclude by observing that $\pi_0 K$
can be identified with the set of equivalence classes of objects in the homotopy category
$\h{(GY_{\bigdot})}$, so that the surjectivity is equivalent to $(ii)$.
\end{proof}

\section{Straightening for Locally coCartesian Fibrations}\label{bisecC}

Let $S$ be a simplicial set. It follows from Theorem \toposref{straightthm} that the
straightening and unstraightening functors $\St_{S}$ and $\Un_{S}$ of \S \toposref{strsec} determine an equivalence
between the following types of data:
\begin{itemize}
\item[$(1)$] Cartesian fibrations of simplicial sets $X \rightarrow S$.
\item[$(2)$] Simplicial functors from $\sCoNerve[S]^{op}$ into the category
$\mSet$ of marked simplicial sets.
\end{itemize}
In other words, Cartesian fibrations $X \rightarrow S$ are classified (up to equivalence)
by functors $S^{op} \rightarrow \Cat_{\infty}$. Our goal in this section is to provide a generalization of this classification scheme to the setting of {\em locally} Cartesian fibrations.

Our first step is to find an appropriate replacement for the base $S$ of the fibration. 
Recall that a locally Cartesian fibration $X \rightarrow S$ is a Cartesian fibration if and
only if every restriction $X \times_{S} \Delta^2 \rightarrow \Delta^2$ is a Cartesian fibration
(Proposition \toposref{gotta}). More generally, if we are given some collection $T$ of
$2$-simplices of $S$, then we could restrict our attention to locally Cartesian fibrations
$X \rightarrow S$ whose restriction to each $2$-simplex of $T$ is a Cartesian fibration.
There is no loss of generality in assuming that $T$ contains every degenerate $2$-simplex
of $S$ (since any locally Cartesian fibration $X \rightarrow \Delta^1$ is automatically
a Cartesian fibration). A pair $(S,T)$ consisting of a simplicial set $S$ and
a subset $T \subseteq \Hom( \Delta^2,S)$ containing every degenerate $2$-simplex is called a {\it scaled simplicial set}. In \S \ref{bisec3.1}, we will introduce the language of scaled simplicial sets,
and show that every scaled simplicial set $\overline{S} = (S,T)$ determines an
$\mSet$-enriched category $\scCoNerve[ \overline{S} ]$, whose underlying simplicial category
agrees with $\sCoNerve[S]$.

In \S \ref{bisec3.2}, we will introduce the relevant analogue of the Cartesian model structure
to the setting of scaled simplicial sets. More precisely, we will show that for every
scaled simplicial set $\overline{S} = (S,T)$, there is a combinatorial simplicial model structure
on the category $\mset{S}$ of marked simplicial sets over $S$, whose fibrant objects can
be identified with locally coCartesian fibrations $X \rightarrow S$ whose restriction to each
$2$-simplex of $T$ is a coCartesian fibration (here we work with locally coCartesian fibrations
rather than locally Cartesian fibrations for reasons of technical convenience).

Suppose now that we are given a scaled simplicial set $\overline{S} = (S,T)$.
In \S \ref{bisec3.3}, we will introduce a pair of adjoint functors
$$ \Adjoint{\scSt_{ \overline{S} }}{ \mset{S}}{ (\mSet)^{\scCoNerve[\overline{S}]},}{ \scUn_{ \overline{S} }}$$
which we will refer to as the {\it scaled straightening functor} and the {\it scaled unstraightening functor}.
The main result of this section is that these adjoint functors determine a Quillen equivalence
between $\mset{S}$ and $(\mSet)^{\scCoNerve[ \overline{S}]}$. The proof follows the same
basic pattern as that of the analogous result in \S \toposref{strsec}:

\begin{itemize}
\item[$(a)$] We first treat the case where $S$ consists of a single vertex (\S \ref{bisec3.4}). In this case,
we can identify $\scSt_{ \overline{S} }$ and $\scUn_{ \overline{S} }$ with functors from the category
$\mSet$ of marked simplicial sets to itself. The desired result can be deduced in this case by
comparing both functors to the identity.

\item[$(b)$] We next consider the case where $S$ is a simplex (\S \ref{bisec3.5}). This requires an analysis of the structure of a locally coCartesian fibration over $S$, generalizing the work
of \S \toposref{funkystructure}.

\item[$(c)$] Finally, in \S \ref{bisec3.6}, we will handle the case of a general simplicial set $S$
by writing $S$ as a colimit of its simplices.
\end{itemize}

\begin{warning}
Our notation in this section might be slightly misleading. The scaled straightening and unstraightening functors $\scSt_{ \overline{S} }$ and $\scUn_{ \overline{S} }$ are not merely decorated versions of the analogous functors $\St_{S}$ and $\Un_{S}$ defined in \S \toposref{strsec}. In fact, it is very difficult to compare the constructions directly. Nevertheless, we will see in \S \ref{bisec4.x} that they must be related by virtue of the universal properties enjoyed by each.
\end{warning}

\subsection{Scaled Simplicial Sets}\label{bisec3.1}

Our goals in this section are the following:

\begin{itemize}
\item[$(a)$] To introduce the category $\scSet$ of {\it scaled simplicial sets}
(Definition \ref{cabman}).

\item[$(b)$] To introduce the class of {\it scaled anodyne} morphisms in
$\scSet$, and to establish some of its basic properties.
In particular, we will show that the class of scaled anodyne maps is stable
under pushout products by arbitrary monomorphisms (Proposition \ref{notred}).

\item[$(c)$] To define the functor $\scCoNerve: \scSet \rightarrow \Cat_{\mSet}$ and
its right adjoint $\scNerve$, the {\it scaled nerve} functor. We will show later
$\scSet$ can be endowed with a model structure such that the adjoint functors
$(\scCoNerve, \scNerve)$ determine a Quillen equivalence of $\scSet$ with
$\Cat_{\mSet}$, as indicated in Theorem \ref{toothygrin} (so that the underlying
homotopy theory of $\scSet$ is the theory of $(\infty,2)$-categories). We will carry
out one crucial step of the proof in this section: the verification that
$\scCoNerve$ carries scaled anodyne morphisms in $\scSet$ to
trivial cofibrations in $\Cat_{\mSet}$ (Proposition \ref{presus}).
\end{itemize}

We begin with the basic definitions.

\begin{definition}\label{cabman}
A {\it scaled simplicial set} is a pair $(X, T)$, where
$X$ is a simplicial set, and $T$ is
a set of $2$-simplices of $X$ which contains every degenerate $2$-simplex.
We will refer to the elements of $T$ as {\it thin}.

Let $(X,T)$ and $(X',T')$ be scaled simplicial sets. A {\it morphism} from 
$(X,T)$ to $(X',T')$ is a map of simplicial sets $f: X \rightarrow X'$ which
carries $T$ into $T'$. The collection of scaled simplicial sets and their morphisms
forms a category, which we will denote by $\scSet$.
\end{definition}

\begin{notation}
Let $X$ be an arbitrary simplicial set. We let
$\deg(X)$ denote the collection of degenerate $2$-simplices of $X$.
We let $X_{\flat} = (X, \deg(X) )$ denote the scaled simplicial set whose underlying simplicial
set is $X$ in which degenerate $2$-simplices are flat, and $X_{\sharp} = (X, X_{2} )$ the scaled simplicial set whose underlying simplicial set $X$ in which all $2$-simplices are flat.
\end{notation}

\begin{definition}\label{slapper}
The collection of {\it scaled anodyne maps} is the weakly saturated collection of morphisms
of $\scSet$ generated by the following maps:
\begin{itemize}
\item[$(A)$] For each $0 < i < n$, the inclusion
$$ ( \Lambda^n_i, (\deg(\Delta^n) \cup \{ \sigma \}) \cap \Hom( \Delta^2, \Lambda^n_i) ) \subseteq ( \Delta^n, \deg( \Delta^n) \cup \{ \sigma \}),$$
where $\sigma$ denotes the $2$-simplex $\Delta^{ \{i-1, i, i+1\} } \subseteq \Delta^n$.

\item[$(B)$] The inclusion
$$ (\Delta^4, T ) \subseteq (\Delta^4, T
\cup \{ \Delta^{ \{0,3,4\} }, \Delta^{ \{1,3,4\} } \} \}$$
where $T$ is the collection of all degenerate $2$-simplices of $\Delta^4$ together with the simplices $\{ \Delta^{ \{0,2,4\} }$, $\Delta^{ \{1,2,3 \} }$, $\Delta^{ \{ 0, 1, 3 \} }$, $\Delta^{ \{1,3, 4 \} } \}$, and
$\Delta^{ \{0,1,2\} }$.

\item[$(C)$] The inclusion
$$ ( \Lambda^{n}_0 \coprod_{ \Delta^{ \{0,1\} }} \Delta^0, T)
\subseteq ( \Delta^n \coprod_{ \Delta^{ \{0,1\} } } \Delta^0, T),$$
where $n > 2$ and $T$ is the collection of all degenerate $2$-simplices of
$\Delta^n \coprod_{ \Delta^{ \{0,1\} }} \Delta^0$ together with the image
of the simplex $\Delta^{ \{0,1,n\} }$. 
\end{itemize}
\end{definition}

\begin{remark}\label{slapperB}
For $i \in \{1,2\}$, the inclusion of scaled simplicial sets
$f_i: ( \Delta^3, T) \subseteq ( \Delta^3, \Hom( \Delta^2, \Delta^3) ),$
is scaled anodyne, where $T$ is the collection of all $2$-simplices of $\Delta^3$ other than
$\Delta^{ \{0, i, 3 \} }$. To see this, it suffices to observe that each $f_i$ is a pushout of a morphism of
type $(B)$ appearing in Definition \ref{slapper}, where the pushouts are formed along the surjective map of simplicial sets $p_i: \Delta^4 \rightarrow \Delta^3$ characterized by the equations
$$p_1^{-1} \{2\} = \Delta^{ \{2,3\} } \quad \quad p_2^{-1} \{1\} = \Delta^{ \{1,2\} }.$$
\end{remark}

\begin{remark}\label{tweepus}
If $A \rightarrow B$ is an inner anodyne map of simplicial sets, then the induced map
$A_{\sharp} \rightarrow B_{\sharp}$ is scaled anodyne.
\end{remark}

\begin{remark}
Definition \ref{slapper} is not self-opposite: if $f: (X,T) \rightarrow (Y,T')$ is a scaled anodyne mapmap of scaled simplicial sets, then the induced map $(X^{op}, T) \rightarrow (Y^{op}, T')$ need not be scaled anodyne. 
\end{remark}

 \begin{definition}
We will say that a map $(X,T) \rightarrow (X',T')$ of scaled simplicial sets is
a {\it cofibration} if the underlying map of simplicial sets $X \rightarrow X'$ is a monomorphism.
\end{definition}

\begin{proposition}\label{notred}
Let $f: X \rightarrow X'$ be a cofibration of scaled simplicial sets, and let
$g: Y \rightarrow Y'$ be a scaled anodyne map. Then the pushout product
$$ f \wedge g: (X \times Y') \coprod_{ X \times Y} (X' \times Y) \rightarrow X' \times Y'$$
is a scaled anodyne map.
\end{proposition}

\begin{proof}
Without loss of generality, we may assume that $f$ is a generating cofibration of one of the following forms:
\begin{itemize}
\item[$(1)$] The inclusion $(\bd \Delta^n)_{\flat} \subseteq \Delta^n_{\flat}$ for some $n \geq 0$.
\item[$(2)$] The inclusion $\Delta^2_{\flat} \subseteq \Delta^2_{\sharp}$.
\end{itemize}
Similarly, we may assume that $g$ is one of the generators for the class of scaled anodyne maps
specified in Definition \ref{slapper}. There are seven cases to consider:
\begin{itemize}
\item[$(1A)$] The map $f$ is an inclusion $(\bd \Delta^n)_{\flat} \subseteq \Delta^n_{\flat}$ and
$g$ is an inclusion of the form 
$$ ( \Lambda^m_i, T \cap \Hom( \Delta^2, \Lambda^m_i) ) \subseteq ( \Delta^m, T ),$$
where $0 < i < m$ and $T = \deg( \Delta^m) \cup \{ \Delta^{ \{ i-1, i, i+1 \} } \}$.
Let $S$ denote the collection of all simplices $\sigma: \Delta^{k(\sigma)} \rightarrow \Delta^n \times \Delta^m$ with the following properties:
\begin{itemize}
\item The simplex $\sigma$ is nondegenerate, and induces surjections
$\Delta^{k(\sigma)} \rightarrow \Delta^n$ and $\Delta^{k(\sigma)} \rightarrow \Delta^m$. 
\item There exist integers $0 < p(\sigma) < k$ and $0 \leq j(\sigma) \leq n$ (automatically unique) such that
$\sigma( p(\sigma) ) = ( j(\sigma), i)$ and $\sigma( p(\sigma) - 1) = (j(\sigma), i-1)$.
\end{itemize}
Choose an ordering $S = \{ \sigma_1 < \ldots < \sigma_{q} \}$, such that
$a<b$ if $\dim( \sigma_a) < \dim(\sigma_b)$, or if
$\dim( \sigma_a) = \dim( \sigma_b)$ and $j(\sigma_{a}) < j(\sigma_{b})$.
For every index $1 \leq a \leq q$, let $T_{a}$ denote the collection of all 
$2$-simplices of $\Delta^{ k( \sigma_a) }$ which are either degenerate or have the form
$\Delta^{ \{ p(\sigma_a)-1, p(\sigma_a), p(\sigma_a) + 1 \} }$, and let
$T'_a$ denote the collection of all $2$-simplices of $T_a$ which belong to
$\Lambda^{ k(\sigma_a)}_{ p(\sigma_a)}$. 
We define a sequence of scaled simplicial subsets $\{ Z_a \subseteq X' \times Y' \}_{0 \leq a \leq q}$
as follows. Set 
$Z_0 = (X \times Y') \coprod_{ X \times Y} (X' \times Y)$, and for $a > 0$ define
$Z_{a}$ by the pushout diagram 
$$ \xymatrix{ ( \Lambda^{ k( \sigma_a) }_{p( \sigma_a)}, T'_a) \ar@{^{(}->}[d] \ar[r] & Z_{a-1} \ar[d] \\
( \Delta^{ k( \sigma_a) }, T_a) \ar[r] & Z_{a}, }$$
using the map $\sigma_{a}$ to extend the inclusion of $Z_{a-1}$ into $X' \times Y'$ to an
inclusion of $Z_{a}$ into $X' \times Y'$. 

By construction, the inclusion $Z_0 \subseteq Z_q$
is a scaled anodyne map. The inclusion $\phi: Z_{q} \subseteq X' \times Y'$ is an
isomorphism of the underlying simplicial sets. If $n \neq 1$ or $m \neq 2$, then
$\phi$ is even an isomorphism of scaled simplicial sets, and the proof is complete.
The special case where $n = 1$ and $m = 2$ requires a bit more care: in this case,
we observe that $\phi$ can be obtained as a as a pushout of three scaled anodyne maps 
appearing in Remark \ref{slapperB}.

\item[$(1B)$] The map $f$ is the inclusion $( \bd \Delta^n)_{\flat} \subseteq \Delta^n_{\flat}$, and
$g$ is an inclusion
$(\Delta^4, T ) \subseteq (\Delta^4, T
\cup \{ \Delta^{ \{0,1,4 \} }, \Delta^{ \{0,3,4\} } \} \}$,
where $T$ is defined as in part $(B)$ of Definition \ref{slapper}. 
If $n > 1$, then $f \wedge g$ is an isomorphism. If $n=0$, then $f \wedge g$ is isomorphic to
$g$. If $n =1$, then $f \wedge g$ is isomorphic to an iterated pushout of copies of the morphism $g$; this follows from the fact that every map from $\Delta^{ \{0,1,4\} }$ or $\Delta^{ \{0,3,4\} }$ into
$\Delta^1$ can be extended to a map from $\Delta^4$ into $\Delta^1$.

\item[$(1C)$] The map $f$ is the inclusion $(\bd \Delta^n)_{\flat} \subseteq \Delta^n_{\flat}$,
and $g$ is the inclusion
$$ ( \Lambda^{m}_0 \coprod_{ \Delta^{ \{0,1\} }} \Delta^0, T)
\subseteq ( \Delta^m \coprod_{ \Delta^{ \{0,1\} } } \Delta^0, T)$$
where $m > 2$ and $T$ is defined as in part $(C)$ of Definition \ref{slapper}.
The proof is similar to that of part $(1A)$. Let $S$ denote the collection of all simplices $\sigma: \Delta^{k(\sigma)} \rightarrow \Delta^n \times \Delta^m$ with the following properties:
\begin{itemize}
\item The simplex $\sigma$ is nondegenerate, and induces surjections
$\Delta^{k(\sigma)} \rightarrow \Delta^n$ and $\Delta^{k(\sigma)} \rightarrow \Delta^m$. 
\item There exist integers $0 \leq p(\sigma) < k$ and $0 \leq j(\sigma) \leq n$ (automatically unique) such that
$\sigma( p(\sigma) ) = ( j(\sigma), 0)$ and $\sigma( p(\sigma) + 1) = (j(\sigma), 1)$.
\end{itemize}
Choose an ordering $S = \{ \sigma_1 < \ldots < \sigma_{q} \}$, such that
$a<b$ if $\dim( \sigma_a) < \dim(\sigma_b)$, or if
$\dim( \sigma_a) = \dim( \sigma_b)$ and $j(\sigma_{a}) > j(\sigma_{b})$.
For every index $1 \leq a \leq q$, let $T_{a}$ denote the collection of all 
$2$-simplices of $\Delta^{ k( \sigma_a) }$ which are either degenerate, have the form
$\Delta^{ \{ p(\sigma_a)-1, p(\sigma_a), p(\sigma_a) + 1 \} }$ if $p(\sigma_a)>0$, or
have the form $\Delta^{ \{0, 1, k(\sigma_a) \} }$ if $p(\sigma_a)=0$.
Let $T'_a$ denote the collection of all $2$-simplices of $T_a$ which belong to
$\Lambda^{ k(\sigma_a)}_{ p(\sigma_a)}$. 
Set  $Z_0 = (X \times Y') \coprod_{ X \times Y} (X' \times Y)$, and for $a > 0$ define
$Z_{a}$ by the pushout diagram 
$$ \xymatrix{ ( \Lambda^{ k( \sigma_a) }_{ p( \sigma_a)}, T'_a) \ar@{^{(}->}[d] \ar[r] & Z_{a-1} \ar[d] \\
( \Delta^{ k( \sigma_a) }, T_a) \ar[r] & Z_{a}, }$$
using the map $\sigma_{a}$ to extend the inclusion of $Z_{a-1}$ into $X' \times Y'$ to an
inclusion of $Z_{a}$ into $X' \times Y'$. 
It follows by induction on $a$ that each inclusion $Z_0 \subseteq Z_a$ is a scaled anodyne map; taking $a=q$ we deduce that $f \wedge g$ is scaled anodyne as desired.

\item[$(2A)$] The map $f$ is the inclusion $\Delta^2_{\flat} \subseteq \Delta^2_{\sharp}$, and
$g$ is an inclusion of the form $( \Lambda^2_{1})_{\flat} \subseteq \Delta^2_{\sharp}$. In this
case, $f \wedge g$ is an isomorphism on the underlying simplicial sets, and can be obtained
by composing three scaled anodyne maps belonging to the type $(B)$ of Definition \ref{slapper}.

\item[$(2A')$] The map $f$ is the inclusion $\Delta^2_{\flat} \subseteq \Delta^2_{\sharp}$, and
$g$ is an inclusion of the form $$( \Lambda^m_i, \deg( \Lambda^n_i) \cup \{ \Delta^{ \{i-1, i, i+1\}} \}
) \subseteq ( \Delta^m, \deg( \Delta^n) \cup \{ \Delta^{ \{i-1, i, i+1\}} \}),$$ where $m > 2$ and
$0 < i < m$. In this case, $f \wedge g$ is an isomorphism.

\item[$(2B)$] The map $f$ is the inclusion $\Delta^2_{\flat} \subseteq \Delta^2_{\sharp}$, and
$g$ is the inclusion
$(\Delta^4, T ) \subseteq (\Delta^4, T
\cup \{ \Delta^{ \{0,1,4 \} }, \Delta^{ \{0,3,4\} } \} \}$ where $T$ is defined as in part
$(B)$ of Definition \ref{slapper}. Then $f \wedge g$ is an iterated pushout of morphisms isomorphic to $g$; this follows from the observation that every map from $\Delta^{ \{0,1,4\} }$ or $\Delta^{ \{0,3,4\} }$ to $\Delta^2$ can be extended from a map from $\Delta^4$ to $\Delta^2$.

\item[$(2C)$] The map $f$ is the inclusion $\Delta^2_{\flat} \subseteq \Delta^2_{\sharp}$
and $g$ is the inclusion
$$ ( \Lambda^{n}_0 \coprod_{ \Delta^{ \{0,1\} }} \Delta^0, T)
\subseteq ( \Delta^n \coprod_{ \Delta^{ \{0,1\} } } \Delta^0, T)$$
where $n > 2$ and $T$ is defined as in part $(C)$ of Definition \ref{slapper}. In this case,
$f \wedge g$ is an isomorphism of scaled simplicial sets.
\end{itemize}
\end{proof}

\begin{notation}
The category $\scSet$ of scaled simplicial sets is {\it Cartesian closed}. That is, for
every pair of objects $X,Y \in \scSet$, we can define a new scaled simplicial set
$\sFun( X, Y)$ and a map $e: \sFun(X,Y) \times X \rightarrow Y$ with the following universal property: for every scaled simplicial set $Z$, composition with $e$ induces a bijection
$$ \Hom_{ \scSet}( Z, \sFun(X,Y) ) \rightarrow \Hom_{ \scSet}( Z \times X, Y ).$$
\end{notation}

Let $\mSet$ denote the category of marked simplicial sets, as defined in \S \toposref{twuf}.
We regard $\mSet$ as endowed with the Cartesian model structure of \S \toposref{markmodel},
so that the forgetful functor $\mSet \rightarrow \sSet$ is a right Quillen functor which
determines a Quillen equivalence between $\mSet$ and $\sSet$ (where $\sSet$ is
endowed with the Joyal model structure). We let $\mCat$ denote the category
of $\mSet$-enriched categories, endowed with the model structure of \S \toposref{compp4}.

\begin{definition}\label{ilmut}
Let $\calC$ be a category enriched over marked simplicial sets. We
define a scaled simplicial set $\scNerve( \calC) = ( \Nerve(\calC), T)$
as follows:
\begin{itemize}
\item[$(1)$] The underlying simplicial set $\Nerve(\calC)$ is the simplicial
nerve of $\calC$, where we regard $\calC$ as a simplicial category via the forgetful
functor $\mSet \rightarrow \sSet$.
\item[$(2)$] Suppose given a $2$-simplex $\sigma$ of $\Nerve(\calC)$, corresponding to
a (noncommutative) diagram
$$ \xymatrix{ & Y \ar[dr]^{g} & \\
X \ar[ur]^{f} \ar[rr]^{h} & & Z }$$
in $\calC$ and an edge $\alpha: \Delta^1 \rightarrow \bHom_{\calC}( X, Z)$ joining
$h$ to $g \circ f$. Then $\sigma$ is thin in $\Nerve(\calC)$ if and only if $\alpha$
is a marked edge of $\bHom_{\calC}(X,Z)$. 
\end{itemize}
\end{definition}

The functor $\scNerve: \mCat \rightarrow \scSet$ admits a left adjoint, which we will denote
by $\scCoNerve: \scSet \rightarrow \mCat$. We can describe the functor
$\scCoNerve$ concretely as follows. For every scaled simplicial set $\overline{S} = (S,T)$, 
the underlying simplicial category of $\scCoNerve[\overline{S}]$ can be identified with
$\sCoNerve[S]$. Given a pair of vertices $x,y \in S$, an edge
$\alpha$ of  $\bHom_{\sCoNerve[S]}(x,y)$ is marked if and only if there
exists a sequence of vertices 
$$ x = x_0, x_1, \ldots, x_n = y$$
of $S$ and a sequence of thin $2$-simplices $\sigma_i:$
$$ \xymatrix{ & y_{i} \ar[dr]^{g_i} & \\
x_{i-1} \ar[ur]^{f_i} \ar[rr]^{h_i} & & x_{i} }$$
classifying edges $\alpha_{i}$ of $\bHom_{\sCoNerve[S]}(x_{i-1}, x_i)$
joining $h_{i}$ to $g_i \circ f_i$, such that 
$$ \alpha = \alpha_{n} \circ \ldots \circ \alpha_{1} \in \Hom_{\sSet}( \Delta^1, \bHom_{\sCoNerve[S]}(x,y)).$$

\begin{lemma}\label{preperc}
Fix $n > 0$. Let $Y'$ denote the marked simplicial set
$$ ( \Delta^n \times \Delta^1, M' ),$$
where $M'$ denotes the collection of all degenerate edges of $\Delta^n \times \Delta^1$ together
with the edge $\{n\} \times \Delta^1$, and let
$$ Y = ( (\Delta^n \times \{1\}) \coprod_{ ( \bd \Delta^n \times \{1\} ) } ( \bd \Delta^n \times \Delta^1), M
) \subseteq Y'$$
where $M$ is defined similarly. Then the inclusion $Y \subseteq Y'$ is a trivial cofibration of marked simplicial sets.
\end{lemma}

\begin{proof}
For $0 \leq i \leq n$, let $\sigma_{i}$ denote the $(n+1)$-simplex of $\Delta^n \times \Delta^1$
described by the formula
$$ \sigma_{i}(j) = \begin{cases} (j,0) & \text{if } j \leq i \\
(j-1, 1) & \text{if } j > i. \end{cases}$$
Let $Y_{i}$ be the union $Y \cup \sigma_0^{\flat} \cup \ldots \cup \sigma_i^{\flat}$, regarded as a marked
simplicial subset of $Y'$. By convention, we will say that $Y_{-1} = Y$.
It will suffice to show that each of the inclusions
$$ Y \stackrel{f_0}{\rightarrow} Y_0 \stackrel{f_1}{\rightarrow} Y_1 \rightarrow \ldots
\stackrel{ f_{n} }{\rightarrow} Y_{n} = Y'$$
is a marked anodyne morphism.

If $i < n$, we have a pushout diagram
$$ \xymatrix{ (\Lambda^{n+1}_{i+1})^{\flat} \ar@{^{(}->}[d] \ar[r] & Y_{i-1} \ar[d]^{f_{i}} \\
( \Delta^{n+1})^{\flat} \ar[r] & Y_{i} }$$
of marked simplicial sets. If $i = n$, we
instead have a pushout diagram
$$ \xymatrix{ ( \Lambda^{n+1}_{n+1})^{\flat} \coprod_{ (\Delta^{ \{n, n+1\} })^{\flat} } 
( \Delta^{ \{n, n+1\} })^{\sharp} \ar[r] \ar@{^{(}->}[d] & Y_{n-1} \ar[d]^{f_n} \\
(\Delta^{n+1})^{\flat} \coprod_{ ( \Delta^{ \{n, n+1\} })^{\flat}} ( \Delta^{ \{n, n+1\} })^{\sharp}
\ar[r] & X_{n}. }$$
In either case, the diagram exhibits $f_{i}$ as a pushout of a marked anodyne morphism, which
is therefore marked anodyne.
\end{proof}

\begin{lemma}\label{perc}
Let $k > 0$, let $C$ denote the cube $( \Delta^1 )^{k}$, and let $v = (1, \ldots, 1)$ denote the final
vertex of $C$. Let $Y'$ denote the marked simplicial set
$(C \times \Delta^1, M')$, where $M$ is the collection of all degenerate edges of
$C \times \Delta^1$ together with the edge $\{v\} \times \Delta^1$, and let
$$Y = ( (\bd C \times \Delta^1) \coprod_{ \bd C \times \{1\} } ( C \times \{1\} ), M) \subseteq Y'$$
where $M$ is defined similarly. Then the inclusion $Y \subseteq Y'$ is a trivial cofibration of marked simplicial sets.
\end{lemma}

\begin{proof}
We observe that every simplex of $C$ either contains the vertex $v$ as a final vertex
or belongs to the boundary $\bd C$. The desired result therefore follows from repeated application of Lemma \ref{preperc}.
\end{proof}

\begin{proposition}\label{presus}
Let $f: X \rightarrow Y$ be a map of scaled simplicial sets. Then:
\begin{itemize}
\item[$(1)$] If $f$ is a cofibration, then the induced map $\scCoNerve[X] \rightarrow
\scCoNerve[Y]$ is a cofibration of $\mSet$-enriched categories.
\item[$(2)$] If $f$ is a scaled anodyne map, then the induced map
$\scCoNerve[X] \rightarrow \scCoNerve[Y]$ is a trivial cofibration of
$\mSet$-enriched categories.
\end{itemize}
\end{proposition}

\begin{proof}
We begin by recalling some notation from \S \toposref{compp4}. For every marked simplicial
set $S$, we let $[1]_{S}$ denote the $\mSet$-enriched category whose set of objects is
$[1] = \{0, 1\}$, with
$$ \bHom_{ [1]_{S} }(i, j) = \begin{cases} ( \Delta^0)^{\sharp} & \text{if } i = j \\
\emptyset & \text{if } i > j \\
S & \text{if } i < j. \end{cases}$$
Note that if $p: S \rightarrow S'$ is a cofibration of marked simplicial sets, then
the induced map $[1]_{S} \rightarrow [1]_{S'}$ is a cofibration of $\mSet$-enriched categories,
which is a weak equivalence provided that $p$ is a weak equivalence.

We now prove $(1)$. Since the collection of all morphisms $f$ for which $\scCoNerve[f]$ is 
a cofibration is weakly saturated, it will suffice to prove the result for a collection of morphisms
which generate the class of cofibrations in $\scSet$. There are three cases to consider:

\begin{itemize}
\item[$(a)$] The map $f$ has the form $\Delta^2_{\flat} \rightarrow \Delta^2_{\sharp}$. In
this case, $\scCoNerve[f]$ is a pushout of the cofibration
$[1]_{ (\Delta^1)^{\flat} } \rightarrow [1]_{ ( \Delta^1)^{\sharp} }.$
\item[$(b)$] The map $f$ has the form $\emptyset = \bd \Delta^0 \subseteq \Delta^0$. 
In this case, $\scCoNerve[f]$ can be identified with the inclusion of the initial
$\mSet$-enriched category into the final $\mSet$-enriched category, which is again a cofibration.
\item[$(c)$] The map $f$ is an inclusion of the form $\bd \Delta^n \subseteq \Delta^n$, for $n > 0$.
Let $\sigma$ denote the cube $( \Delta^1)^{n-1}$ and $\bd \sigma \subseteq \sigma$ its boundary.
The morphism $\scCoNerve[f]$ is then a pushout of the cofibration
$[1]_{ \bd \sigma^{\flat} } \rightarrow [1]_{ \sigma^{\flat} }$.
\end{itemize}

The proof of $(2)$ is a bit more involved. We may again assume without loss of generality that
$f$ is a generator for the class of scaled anodyne maps. There are four cases to consider.
\begin{itemize}
\item[$(A)$] The map $f$ has the form $(\Lambda^2_1)_{\flat} \subseteq \Delta^2_{\sharp}$. In
this case, $\scCoNerve[f]$ is a pushout of the trivial cofibration
$[1]_{ (\Delta^{\{1\}})^{\sharp} } \rightarrow [1]_{ (\Delta^1)^{\sharp} }$.

\item[$(A')$] There exist integers $n > 2$ and $0 < i < n$, such that
$f$ is of the form 
$$( \Lambda^n_i, \deg( \Lambda^n_{i}) \cup \{ \Delta^{ \{i-1, i, i+1 \}} \}
) \subseteq ( \Delta^n, \deg( \Delta^n) \cup \{ \Delta^{ \{i-1, i, i+1 \} } \}.$$
In this case, $\scCoNerve[f]$ is a pushout of a morphism of the form
$[1]_{Y} \rightarrow [1]_{Y'}$, where $Y \rightarrow Y'$ is the inclusion of
marked simplicial sets appearing in the statement of Lemma \ref{perc}.
Since $Y \rightarrow Y'$ is a trivial cofibration, we conclude that
$\scCoNerve[f]$ is a trivial cofibration as well.


\item[$(B)$] The map $f$ is an inclusion
$(\Delta^4, T ) \subseteq (\Delta^4, T
\cup \{ \Delta^{ \{0,1,4 \} }, \Delta^{ \{0,3,4\} } \} \})$ 
where $T$ is defined as in part $(B)$ of Definition \ref{slapper}.
In this case, we observe that the induced map $\scCoNerve[f]$ is a pushout of a cofibration
$[1]_{S} \subseteq [1]_{S'}$. Here $S$ is the marked simplicial set $( \Delta^1 \times \Delta^1 \times \Delta^1, M)$ depicted in the following diagram:
$$ 
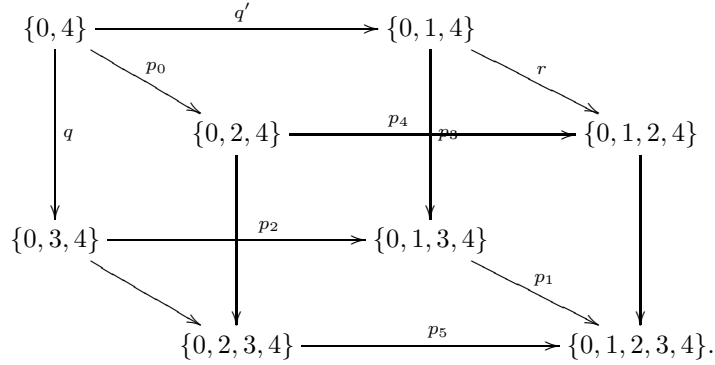
\xymatrix{ \{0,4\} \ar[rr]^{q'} \ar[dd]^{q} \ar[dr]^{p_0} & & \{0,1,4\} \ar[dd]^{p_3} \ar[dr]^{r} &  \\
& \{0,2,4\} \ar[rr]^{p_4 \quad \quad \quad} \ar[dd] & & \{ 0,1,2,4\} \ar[dd] \\
\{0,3,4\} \ar[rr]^{\quad \quad p_2} \ar[dr] & & \{0,1,3,4\} \ar[dr]^{p_1} & \\
& \{ 0, 2, 3, 4 \} \ar[rr]^{p_5} & & \{0,1,2,3,4\}. }$$
Here the marked edges of $S$ are the degenerate edges, together with the edges
$\{ p_i \}_{0 \leq i \leq 5}$, and $S' = ( \Delta^1 \times \Delta^1 \times \Delta^1, M \cup \{q,q'\})$.
To prove that $\scCoNerve[f]$ is a trivial cofibration, it will suffice to show that for every $\infty$-category $\calC$ and every map $g: \Delta^1 \times \Delta^1 \times \Delta^1 \rightarrow \calC$, if
$g$ carries the edges $\{ p_i \}_{0 \leq i \leq 5}$ to equivalences in $\calC$, then $g$ carries
$q$ and $q'$ to equivalences in $\calC$. To prove this, we observe that $g(q)$ has an inverse in the homotopy category $\h{\calC}$, given by the composition
$g(p_0)^{-1} g(p_4)^{-1} g(r) g(p_3)^{-1} g(p_2)$, and that $g(q')$ has an inverse given by the composition $g(q)^{-1} g(p_2)^{-1} g(p_3)$.

\item[$(C)$] The map $f$ is an inclusion
$$ ( \Lambda^{n}_0 \coprod_{ \Delta^{ \{0,1\} }} \Delta^0, T)
\subseteq ( \Delta^n \coprod_{ \Delta^{ \{0,1\} } } \Delta^0, T),$$
where $n > 2$ and $T$ is defined as in part $(C)$ of Definition \ref{slapper}.
The desired result in this case is merely a translation of Lemma \ref{prewise}.
\end{itemize}
\end{proof}

\begin{proposition}\label{twop1}
Let $S$ be a simplicial set, let $\calC = \scCoNerve[ S_{\sharp} ]$. We
define a new $\mSet$-enriched category $\calC_{S}$ as follows:
\begin{itemize}
\item The objects of $\calC_{S}$ are the objects of $\calC$ (the vertices of $S$).
\item Given a pair of objects $x,y \in \calC_{S}$ such that
$\bHom_{\calC}(x,y) = (X,M)$, we set $\bHom_{\calC_S}(x,y) = X^{\sharp}$.
\end{itemize}
Then the evident functor $f_{S}: \calC \rightarrow \calC_{S}$ is a trivial cofibration of $\mSet$-enriched categories.
\end{proposition}

\begin{proof}
The map $f_{S}$ can be obtained as an iterated pushout of trivial cofibrations of the form
$$[1]_{ ( \Lambda^2_1)^{\sharp} \coprod_{ ( \Lambda^2_1)^{\flat} } (\Delta^2)^{\flat}}
\subseteq [1]_{ (\Delta^2)^{\sharp}},$$
(with notation as in the proof of Proposition \ref{presus}). 
\end{proof}

\subsection{Locally coCartesian Model Structures}\label{bisec3.2}

Let $S$ be a simplicial set. In \S \toposref{markmodel}, we saw that there is
a simplicial model structure on the category $(\mSet)_{/S^{\sharp}}$ of marked simplicial
sets over $S$, whose fibrant objects correspond precisely to coCartesian fibrations
$X \rightarrow S$. Our goal in this section is to introduce a generalization of this model structure.

\begin{definition}\label{catpat}
Let $S$ be a simplicial set. A {\it categorical pattern} on $S$ is a triple
$(M_S, T, \{ p_{\alpha}: K_{\alpha}^{\triangleleft} \rightarrow S' \}_{\alpha \in A})$, where
$M_S$ is a marking of $S$ (that is, a collection of edges of $S$ which contains all degenerate edges),
$T$ is a scaling of $S$ (that is, a collection of $2$-simplices of $S$ which contains all degenerate
$2$-simplices), and $\{ p_{\alpha}: K_{\alpha}^{\triangleleft} \rightarrow S \}_{\alpha \in A}$ is a collection of maps of simplicial sets which carry each edge of $K_{\alpha}^{\triangleleft}$ into $M_S$ and
each $2$-simplex of $K_{\alpha}^{\triangleleft}$ into $T$.

Suppose we are given a categorical pattern 
$\CatP = (M_S, T, \{ p_{\alpha}: K_{\alpha}^{\triangleleft} \rightarrow S \}_{\alpha \in A})$
on $S$. A {\it marked simplicial set} over $\CatP$ is a marked simplicial set
$\overline{X} = (X,M)$ equipped with a map $f: X \rightarrow S$ satisfying the following condition:
for every edge $e$ of $X$ which belongs to $M$, $f(e)$ belongs to $M_S$. We let
$\mset{ \CatP}$ denote the category of marked simplicial sets over $\CatP$.

We will say that an object $\overline{X} \in \mset{ \CatP}$ is {\it $\CatP$-fibered} if the following conditions are satisfied:
\begin{itemize}
\item[$(1)$] The underlying map of simplicial sets $f: X \rightarrow S$ is an inner fibration.
\item[$(2)$] For each edge $\Delta^1 \rightarrow S$ belonging to $M_S$, the induced map
$f': X \times_{S} \Delta^1 \rightarrow \Delta^1$ is a coCartesian fibration.
\item[$(3)$] An edge $e$ of $X$ belongs to $M$ if and only if $f(e)$ belongs to $M_S$ and
$e$ is an $f'$-coCartesian edge of $X \times_{S} \Delta^1$.
\item[$(4)$] Given a commutative diagram
$$ \xymatrix{ \Delta^{ \{0,1\} } \ar[d] \ar[r]^{e} & X \ar[d] \\
\Delta^2 \ar[r]^{\sigma} & S, }$$
if $e \in M$ and $\sigma \in T$, then $e$ determines an
$f'$-coCartesian edge of $X \times_{S} \Delta^2$, where
$f': X \times_{S} \Delta^2 \rightarrow \Delta^2$ denotes the projection map.

\item[$(5)$] For every index $\alpha \in A$, the induced coCartesian fibration
$f_{\alpha}: X \times_{S} K_{\alpha}^{\triangleleft} \rightarrow K_{\alpha}^{\triangleleft}$ is classified
by a limit diagram $K_{\alpha}^{\triangleleft} \rightarrow \Cat_{\infty}$.

\item[$(6)$] For every index $\alpha \in A$ and every coCartesian section
$s$ of the map $f_{\alpha}$, $s$ is an $f$-limit diagram in $X$.
\end{itemize}
\end{definition}

\begin{remark}\label{caroline}
Let $\CatP$ be a categorical pattern on a simplicial set $S$. We will sometimes abuse terminology by saying that a map of simplicial sets $X \rightarrow S$ is {\it $\CatP$-fibered} if there exists a collection
of edges $M$ in $X$ such that $\overline{X} = (X,M)$ is a $\CatP$-fibered object of $\mset{ \CatP}$.
In this case, the set $M$ is uniquely determined (requirement $(3)$ of Definition \ref{catpat}). 
\end{remark}

\begin{remark}\label{sichwell}
In the situation of Definition \ref{catpat}, conditions $(5)$ and $(6)$ are automatic
whenever the simplicial set $K_{\alpha}$ is weakly contractible and the diagram
$p_{\alpha}$ is constant.
\end{remark}

\begin{remark}
Let $\CatP$ be a categorical pattern on a simplicial set $S$. For every pair of objects $\overline{X}, \overline{Y} \in \mset{ \CatP}$, there exists a simplicial set $\bHom_{S}^{\sharp}( \overline{X}, \overline{Y})$ with the following universal property: for every
simplicial set $K$, we have a canonical bijection
$$ \Hom_{ \sSet}( K, \bHom_{S}^{\sharp}( \overline{X}, \overline{Y} ))
\simeq \Hom_{ \mset{ \CatP}}( K^{\sharp} \times \overline{X}, \overline{Y} ).$$
This definition of mapping spaces endows $\mset{ \CatP}$ with the structure of a
simplicial category.
\end{remark}

\begin{remark}\label{ratpat}
Let $\CatP = (M_S, T, \{ p_{\alpha}: K_{\alpha}^{\triangleleft} \rightarrow S' \}_{\alpha \in A})$ be a categorical pattern on a simplicial set $S$ and let $\overline{X} = (X,M)$ be an object of $\mSet$ satisfying conditions $(1)$ through $(4)$ of Definition \ref{catpat}. For each index $\alpha \in A$, let
$X_{\alpha} = X \times_{S} K_{\alpha}^{\triangleleft}$. Then the projection map
$q: X_{\alpha} \rightarrow K_{\alpha}^{\triangleleft}$ is a coCartesian fibration, classified by a functor
$\chi: K_{\alpha}^{\triangleleft} \rightarrow \Cat_{\infty}$. According to Proposition \toposref{charcatlimit}, the map $\chi$ is a limit diagram if and only if the restriction functor $r: Z \rightarrow Z_0$ is an equivalence of $\infty$-categories, where $Z$ denotes the $\infty$-category of coCartesian sections of $q$ and $Z_0$ the $\infty$-category of coCartesian sections of the restriction
$X \times_{S} K_{\alpha} \rightarrow K_{\alpha}$. 

Now suppose that $\overline{X}$ also satisfies condition $(6)$ of Definition \ref{catpat}. In this case,
every coCartesian section $s$ of $q$ is a $q$-limit diagram, so that the map
$\bHom_{Z}(s',s) \rightarrow \bHom_{Z_0}(s' |K ,s|K)$ is a homotopy equivalence for
any $s' \in Z$ (in fact, the analogous statement is true for {\em any} section of $q$). 
It follows that the functor $r$ is automatically fully faithful. Now $r$ is an equivalence of $\infty$-categories if and only if it is essentially surjective, which (since $r$ is evidently a categorical fibration) is equivalent to the requirement that $r$ be surjective on vertices. Consequently, in the definition of
a $\CatP$-fibered object of $\mset{ \CatP}$, we are free to replace assumption $(5)$ by the following apparently weaker condition:
\begin{itemize}
\item[$(5')$] For each $\alpha \in A$ and every coCartesian section $s_0$ of the projection
$X \times_{S} K_{\alpha} \rightarrow K_{\alpha}$, there exists a coCartesian section $s$ of
$X \times_{S} K_{\alpha}^{\triangleleft} \rightarrow K_{\alpha}^{\triangleleft}$ extending
$s_0$.
\end{itemize}
\end{remark}

Our main goal in this section is to prove the following result:

\begin{theorem}\label{theo}
Let $\CatP$ be a categorical pattern on a simplicial set $S$. Then there exists a left proper combinatorial simplicial model structure on $\mset{\CatP}$, which is uniquely characterized by the following properties:
\begin{itemize}
\item[$(C)$] A morphism $f: \overline{X} \rightarrow \overline{Y}$ in $\mset{ \CatP}$ is a cofibration
if and only if $f$ induces a monomorphism between the underlying simplicial sets.
\item[$(F)$] An object $\overline{X} \in \mset{ \CatP}$ is fibrant if and only if
$\overline{X}$ is $\CatP$-fibered.
\end{itemize}
\end{theorem}

\begin{example}\label{sich}
Let $S$ be a simplicial set. The {\it canonical categorical pattern} on $S$ is the categorical
pattern $\CatP = (M_S, T, \emptyset)$, where $M_S$ consists of all edges of $S$ and $T$ consists of all
$2$-simplices of $S$. Then $\mset{\CatP}$ admits a unique model structure satisfying the conditions
of Theorem \ref{theo}: the {\it coCartesian model structure} described in \S \toposref{markmodel}.
\end{example}

\begin{example}\label{sipper}
Let $S$ be a simplicial set, and suppose we are given a categorical pattern
$\CatP = (M_S, T, \{ p_{\alpha}: K_{\alpha}^{\triangleleft} \rightarrow S \}_{\alpha \in A})$, where $M_S$ consists of all edges of $S$, $T$ consists of all
$2$-simplices of $S$, each of the simplicial sets $K_{\alpha}$ is weakly contractible and
each of the maps $p_{\alpha}$ is constant. Then the model structure on $\mset{ \CatP}$ described
by Theorem \ref{theo} coincides with the coCartesian model structure of Example \ref{sich}: this follows immediately from Remark \ref{sichwell}. 
\end{example}

\begin{example}\label{user}
Let $\overline{S} = (S,T)$ be a scaled simplicial set. Then $\overline{S}$ determines a
categorical pattern $\CatP_{ \overline{S}} = (M_S, T, \emptyset)$ on $S$, where $M_S$ is the collection of all edges of $S$. In this situation, we will denote $\mset{ \CatP}$ by $\mset{ \overline{S} }$, and refer to the
model structure of Theorem \ref{theo} as the {\it locally coCartesian model structure} on $\mset{ \overline{S}}$ (note that the underlying category of $\mset{ \overline{S} }$ depends only on
the underlying simplicial set $S$ over $\overline{S}$; however, the model structure on
$\mset{ \overline{S} }$ depends on the collection of thin $2$-simplices of $S$). We will say that
an object $\overline{X} = (X,M) \in \mset{ \overline{S} }$ is {\it $\overline{S}$-fibered} if it is
$\CatP$-fibered: that is, if the underlying map $p: X \rightarrow S$ is a locally coCartesian fibration, 
$M$ is the set of locally $p$-coCartesian edges of $X$, and the restriction of $p$ to every
thin $2$-simplex of $S$ is a coCartesian fibration.
\end{example}

The main step in proving Theorem \ref{theo} is to show that there is a sufficiently large supply of
trivial cofibrations in $\mset{\CatP}$. To this end, we introduce the following definition:

\begin{definition}\label{postspunt}
Let $\CatP = (M_S, T, \{ p_{\alpha}: K_{\alpha}^{\triangleleft} \rightarrow S\}_{\alpha \in A})$
be a categorical pattern on a simplicial set $S$. The collection of {\it $\CatP$-anodyne} morphisms in $\mset{ \CatP}$ is the smallest weakly saturated class of morphisms which contain all morphisms of the following types:

\begin{itemize}

\item[$(A_0)$] The inclusion $( \Lambda^2_1)^{\sharp} \coprod_{ (\Lambda^2_{1})^{\flat}} (\Delta^2)^{\flat}  \subseteq (\Delta^2)^{\sharp},$ for every map $\Delta^2 \rightarrow S$ belonging to $T$ which carries every edge into $M_S$.

\item[$(A_1)$] The inclusion $Q^{\flat} \subseteq Q^{\sharp}$, where
$Q = \Delta^0 \coprod_{ \Delta^{ \{0,2\} }} \Delta^3 \coprod_{ \Delta^{ \{1,3\} }} \Delta^0$
and the map $Q \rightarrow S$ carries every edge of $Q$ into $M_S$ and every $2$-simplex
of $Q$ into $T$.


\item[$(B_0)$] The inclusion $\{0\}^{\sharp} \subseteq (\Delta^1)^{\sharp}$ lying over an
edge of $M_S$.

\item[$(B_1)$] For each $\alpha \in A$, the inclusion $K_{\alpha}^{\sharp} \subseteq
(K_{\alpha}^{\triangleleft})^{\sharp}$ (where $K_{\alpha}^{\triangleleft}$ maps to $S$ via
$p_{\alpha}$).

\item[$(C_0)$] The inclusion 
$$ (\Lambda^n_0)^{\flat} \coprod_{ (\Delta^{ \{0,1\} })^{\flat} } ( \Delta^{ \{0,1\} })^{\sharp}
\subseteq ( \Delta^n )^{\flat} \coprod_{ ( \Delta^{ \{0,1\} })^{\flat} } ( \Delta^{ \{0,1\} })^{\sharp},$$
for every $n > 1$ and every map $\Delta^n \rightarrow S$ whose restriction to
$\Delta^{ \{0, 1, n\} }$ belongs to $T$.

\item[$(C_1)$]  The inclusion $(\Lambda^{n}_i)^{\flat} \subseteq (\Delta^n)^{\flat}$, for every $0 < i < n$ and every map $\Delta^n \rightarrow S$. 

\item[$(C_2)$] For each $n \geq 1$, $\alpha \in A$, and map
$f: \Delta^n \star K_{\alpha} \rightarrow S$ extending $p_{\alpha}: \{n\} \star K_{\alpha} \rightarrow S$,
the inclusion
$$ ( \bd \Delta^n \star K_{\alpha})^{\flat} \coprod_{ ( \{n\} \star K_{\alpha})^{\flat} }
( \{n\} \star K_{\alpha})^{\sharp} \subseteq 
( \Delta^n \star K_{\alpha})^{\flat} \coprod_{ ( \{n\} \star K_{\alpha})^{\flat} }
( \{n\} \star K_{\alpha} )^{\sharp} ).$$
\end{itemize}
\end{definition}

\begin{example}\label{ina}
Let $\CatP$ be a categorical pattern on a simplicial set $S$, and suppose we are given maps of simplicial sets
$A \stackrel{i}{\rightarrow} B \rightarrow S$. If $i$ is inner anodyne, then the induced map
$A^{\flat} \rightarrow B^{\flat}$ is a $\CatP$-anodyne morphism in $\mset{ \CatP}$.
\end{example}

\begin{example}\label{singer}
Let $\CatP$ be a categorical pattern on a simplicial set $S$, and let $e: \Delta^1 \rightarrow S$
be a marked edge of $S$. For every simplicial set $A$, let $\overline{ A}^{\triangleleft}$
denote the marked simplicial set $( A^{\triangleleft}, M_A)$, where $M_A$ is the collection
of all edges of $A^{\triangleleft}$ which are either degenerate or contain the cone point.
We regard $\overline{A}^{\triangleleft}$ as an object of $\mset{ \CatP}$ via the map
$A^{\triangleleft} \rightarrow ( \Delta^0)^{\triangleleft} \simeq \Delta^1 \stackrel{e}{\rightarrow} S$. For any cofibration of simplicial sets $i: A \rightarrow B$, the induced map
$j: \overline{A}^{\triangleleft} \rightarrow \overline{B}^{\triangleleft}$ is $\CatP$-anodyne. 
To prove this, it suffices to treat the basic case where $B = \Delta^n$ and $A = \bd \Delta^n$, in which case the map $j$ is a generating $\CatP$-anodyne map which is either of type
$(B_0)$ (if $n=0$) or $(C_0)$ (if $n > 0$).
\end{example}

\begin{example}\label{urtime}
Let $\CatP = (M_S, T, \{ p_{\alpha}: K_{\alpha}^{\triangleleft} \rightarrow S\}_{\alpha \in A})$
be a categorical pattern on a simplicial set $S$. Let $B_0 \subseteq B$ be a simplicial sets containing
a vertex $b$, and let $f: B \star K_{\alpha} \rightarrow S$ be a map whose restriction to
$\{b\} \star K_{\alpha} \simeq K_{\alpha}^{\triangleleft}$ is given by $p_{\alpha}$.
Suppose that every simplex of $B$ either belongs to $B_0$ or contains $b$ as a final vertex.
Then the inclusion
$$ (B_0 \star K_{\alpha})^{\flat} \coprod_{ ( \{b\} \star K_{\alpha})^{\flat} }
( \{b\} \star K_{\alpha})^{\sharp} \subseteq 
( B \star K_{\alpha})^{\flat} \coprod_{ ( \{b\} \star K_{\alpha})^{\flat} }
( \{b\} \star K_{\alpha} )^{\sharp} )$$
is $\CatP$-anodyne, because it can be obtained as an iterated pushout of $\CatP$-anodyne inclusions of type $(C_2)$.
\end{example}

\begin{remark}\label{stouffer}
Let $\CatP = (M_S, T, \{ p_{\alpha}: K_{\alpha}^{\triangleleft} \rightarrow S\}_{\alpha \in A})$
be a categorical pattern on a simplicial set $S$, and let $\overline{X} = (X,M)$ be an object of $\mset{ \CatP}$. Let $T'$ denote the inverse image of $T$ in $\Hom_{ \sSet}( \Delta^2, X)$, and let
$B$ denote the set of pairs $\beta = (\alpha, \overline{p}_{\beta})$ where $\alpha \in A$ and
$\overline{p}_{\beta}: K_{\alpha}^{\triangleleft} \rightarrow X$ is a map lifting $p_{\alpha}$.
Then $\CatP_{ \overline{X} } = ( M, T', \{ \overline{p}_{\beta} \}_{ \beta \in B} )$ is a categorical
pattern on $X$. Unwinding the definitions, we deduce that a morphism in
$\mset{ \CatP_{ \overline{X} }}$ is $\CatP_{\overline{X}}$-anodyne if and only if it is
$\CatP$-anodyne.
\end{remark}

\begin{definition}\label{prodpat}
Let $S$ and $S'$ be simplicial sets, and let
$\CatP = (M_S, T, \{ p_{\alpha}: K_{\alpha}^{\triangleleft} \rightarrow S\}_{\alpha \in A})$
and $\CatP' = (M'_{S'}, T', \{ q_{\beta}: L_{\beta}^{\triangleleft} \rightarrow S'\}_{\beta \in B})$
be categorical patterns on $S$ and $S'$, respectively. We let
$\CatP \times \CatP'$ denote the categorical pattern
$$ (M_S \times M'_{S'}, T \times T', \{ K_{\alpha}^{\triangleleft}
\stackrel{p_{\alpha}}{\rightarrow} S \times \{s'\} \rightarrow S' \}_{\alpha \in A, s' \in S'}
\cup \{ L_{\beta}^{\triangleleft} \stackrel{q_{\beta}}{\rightarrow} \{s\} \times S'
\rightarrow S \times S' \}_{s \in S, \beta \in B} )$$
on $S \times S'$.
\end{definition}

We will need the following technical results about the theory of 
$\CatP$-anodyne maps:

\begin{proposition}\label{poststaf}
Let $\CatP$ be a categorical pattern on a simplicial set $S$, and let $\overline{X} \in \mset{ \CatP}$. 
The following conditions are equivalent:
\begin{itemize}
\item[$(1)$] The object $\overline{X}$ has the extension property with respect to every
$\CatP$-anodyne morphism in $\mset{ \CatP}$.
\item[$(2)$] The object $\overline{X}$ is $\CatP$-fibered.
\end{itemize}
\end{proposition}

\begin{proposition}\label{postprod}
Let $\CatP$ and $\CatP'$ be categorical patterns on simplicial sets $S$ and $S'$. Let
$f: \overline{X} \rightarrow \overline{Y}$ be a cofibration in $\mset{\CatP}$, and let
$f': \overline{X}' \rightarrow \overline{Y}'$ be a cofibration in $\mset{\CatP'}$.
If $f$ is $\CatP$-anodyne or $f'$ is $\CatP'$-anodyne, then the induced map
$$ f \wedge f': ( \overline{Y} \times \overline{X}' ) \coprod_{ \overline{X} \times \overline{X}' }
( \overline{X} \times \overline{Y}' ) \rightarrow \overline{Y} \times \overline{Y}'$$
is $\CatP \times \CatP'$-anodyne.
\end{proposition}

We will give the proofs of Proposition \ref{poststaf} and \ref{postprod} at the end of this section.
Our next goal is to use them to prove Theorem \ref{theo}. First, we need to establish a few preliminaries.

\begin{lemma}\label{catchrat}
Let $\CatP = (M_S, T, \{ p_{\alpha}: K_{\alpha}^{\triangleleft} \rightarrow S\}_{\alpha \in A})$
be a categorical pattern on a simplicial set $S$, and let $\Delta^2 \rightarrow S$ be a $2$-simplex which belongs to $T$. Then the inclusion
$i: ( \Lambda^2_0)^{\sharp} \coprod_{ (\Lambda^2_{0})^{\flat}} (\Delta^2)^{\flat}  \subseteq (\Delta^2)^{\sharp}$ is a $\CatP$-anodyne morphism in $\mset{ \CatP}$.
\end{lemma}

\begin{proof}
We must show that $i$ has the left lifting property with respect to every morphism
morphism $f: \overline{X} \rightarrow \overline{Y}$ in $\mset{ \CatP}$, provided that $f$ has
the right lifting property with respect to every $\CatP$-anodyne morphism in $\mset{ \CatP}$.
Replacing $\CatP$ by $\CatP_{ \overline{Y} }$ (and invoking Remark \ref{stouffer}), we
are reduced to showing that $\overline{X}$ has the extension property with respect to $i$, provided that $\overline{X}$ has the extension property with respect to every $\CatP$-anodyne morphism.
In view of Proposition \ref{poststaf}, we may assume that $\overline{X}$ is $\CatP$-fibered.
The desired result is now an immediate consequence of Proposition \toposref{protohermes}.

\end{proof}

\begin{lemma}\label{swindler}
Let $\CatP = (M_S, T, \{ p_{\alpha}: K_{\alpha}^{\triangleleft} \rightarrow S\}_{\alpha \in A})$
be a categorical pattern on a simplicial set $S$. Fix $\alpha \in A$, let $M$ be the
collection of all edges of $\Delta^1 \star K_{\alpha}$ except for the initial edge
$\Delta^1 \subseteq \Delta^1 \star K_{\alpha}$. Let $f: \Delta^1 \star K_{\alpha} \rightarrow S$
be a map such which carries each edge into $M_S$, each $2$-simplex into $T$, and such that $f | ( \{1\} \star K_{\alpha} )$ agrees with $p_{\alpha}$.
Then the inclusion $i: ( \Delta^1 \star K_{\alpha}, M) \subseteq ( \Delta^1 \star K_{\alpha})^{\sharp}$ is
a $\CatP$-anodyne morphism in $\mset{ \CatP}$.
\end{lemma}

\begin{proof}
Let $g: \overline{X} \rightarrow \overline{Y}$ be a morphism in $\mset{ \CatP}$ which has
the right lifting property with respect to every $\CatP$-anodyne morphism; we will show that
$g$ has the right lifting property with respect to $i$. Replacing $\CatP$ by $\CatP_{ \overline{Y} }$
(and invoking Remark \ref{stouffer}), we may assume that $\overline{Y}$ is a final object of
$\mset{ \CatP}$. Proposition \ref{poststaf} now guarantees that $\overline{X}$ is $\CatP$-fibered.
Let $X'$ denote the fiber product $X \times_{S} ( \Delta^1 \star K_{\alpha})$, so that the projection
map $q: X' \rightarrow ( \Delta^1 \star K_{\alpha})$ is a coCartesian fibration. Unwinding the definitions, we must show the following:
\begin{itemize}
\item[$(\ast)$] Let $s$ be a section of $q$. If $s$ carries each edge of $M$ to a $q$-coCartesian edge
of $X'$, then $s$ carries every edge of $\Delta^1 \star K_{\alpha}$ to a $q$-coCartesian edge of $X'$.
\end{itemize}
To prove $(\ast)$, let us write rewrite the domain of $s$ as $\{x\} \star \{z\} \star K_{\alpha}$.
Choose a $q$-coCartesian edge $e: s(x) \rightarrow \overline{y}$ in $X'$ covering the initial edge
$\Delta^1 \subseteq \Delta^1 \star K_{\alpha}$. Since $e$ is $q$-coCartesian, we can extend
$s$ to a map $s': \{x\} \star \{y\} \star \{z\}  \star K_{\alpha} \rightarrow X'$ carrying
$\{x\} \star \{y\}$ to $e$. It follows from Proposition
\toposref{protohermes} that, for every vertex $k$ of $K_{\alpha}$, $s'$ carries the edge
$\{y\} \star \{k\}$ to a $q'$-coCartesian edge of $X'$. Using the fact that
$\overline{X}$ is $\CatP$-fibered, we deduce that $s' | \{y\} \star K_{\alpha}$ and
$s' | \{z\} \star K_{\alpha}$ are $q'$-limit diagrams, so that $s'$ carries $\{y\} \star \{z\}$ to
an equivalence in $X'_{y}$. It follows that $s$ carries the edge $\{x\} \star \{z\}$ into a 
composition of $q'$-coCartesian edges $s'( \{x\} \star \{y\})$ and $s'( \{y\} \star \{z\})$, 
which is again a $q'$-coCartesian edge (Proposition \toposref{protohermes}).
\end{proof}

\begin{lemma}\label{kalel}
Let $\calP_0$ denote the categorical pattern $( \Delta^0, \Hom_{\sSet}(\Delta^1, \Delta^0), \Hom_{\sSet}(\Delta^2, \Delta^0), \emptyset)$, so that $\mset{ \CatP_0}$ is equivalent to $\mSet$. For every left anodyne inclusion of simplicial sets $A \subseteq B$, the induced map $j: A^{\sharp} \subseteq B^{\sharp}$ is $\calP_0$-anodyne.
\end{lemma}

\begin{proof}
Without loss of generality, we may assume that $B = \Delta^n$ and $A = \Lambda^n_i$, for some
$0 \leq i < n$, where $n >0$. Suppose first that $n > 2$. If $0 < i < n$, then $j$ is a pushout of the inclusion $j_0: ( \Lambda^n_i)^{\flat} \rightarrow (\Delta^n)^{\flat}$, and therefore $\calP_0$-anodyne
(case $(C_1)$ of Definition \ref{postspunt}). If $i = 0$, then $j$ is a pushout of the inclusion
 $$j_0: (\Lambda^n_0)^{\flat} \coprod_{ (\Delta^{ \{0,1\} })^{\flat} } ( \Delta^{ \{0,1\} })^{\sharp}
\rightarrow ( \Delta^n )^{\flat} \coprod_{ ( \Delta^{ \{0,1\} })^{\flat} } ( \Delta^{ \{0,1\} })^{\sharp}$$
which is $\calP_0$-anodyne (case $(C_0)$ of Definition \ref{postspunt}). 


Now suppose that $n=2$. We observe that $j$ can be obtained as a composite $j'' \circ j'$, where
$j'$ is a pushout of the morphism $j_0$ considered above, and $j''$ is either a generating
$\CatP$-anodyne morphism of type $(A_0)$ or the $\CatP$-anodyne morphism described in
Lemma \ref{catchrat}.

Finally, in the case $n=1$, $j$ is itself a morphism of type $(B_0)$ appearing in Definition \ref{postspunt}.
\end{proof}

\begin{proposition}\label{stilk}
Let $\CatP$ be a categorical pattern on a simplicial set $S$. Let 
$f: \overline{X} \rightarrow \overline{Y}$ be a cofibration in $\mset{ \CatP}$, and let
$\overline{Z}$ be a $\CatP$-fibered object of $\mset{ \CatP}$. Then the induced map
$$q: \bHom^{\sharp}_{S}( \overline{Y}, \overline{Z} ) \rightarrow \bHom^{\sharp}_{S}( \overline{X}, \overline{Z} )$$ is a Kan fibration between Kan complexes. If $f$ is $\calP_0$-anodyne, then $q$ is a trivial Kan fibration.
\end{proposition}

\begin{proof}
We first show that $q$ is a left fibration by showing that $q$ has the right lifting property with respect to every left anodyne inclusion  of simplicial sets $A \subseteq B$ (or {\em every} inclusion of simplicial sets, in the case where
$f$ is $\CatP$-anodyne). Unwinding the definitions, this is equivalent to the assertion that $\overline{Z}$ has the extension property with respect to the induced inclusion
$$f': (B^{\sharp} \times \overline{X} ) \coprod_{ A^{\sharp} \times \overline{X} }( A^{\sharp} \times \overline{Y}) \rightarrow B^{\sharp} \times \overline{Y}.$$
It follows from Proposition \ref{postprod} and Lemma \ref{kalel} that $f'$ is $\CatP$-anodyne, so that the desired result  is a consequence of Proposition \ref{poststaf}.

Applying the above result to the inclusion $\emptyset \subseteq \overline{X}$, we deduce that
the projection map $\bHom^{\sharp}_{S}( \overline{X}, \overline{Z} ) \rightarrow \Delta^0$ is
a left fibration, so that $\bHom^{\sharp}_{S}( \overline{X}, \overline{Z} )$ is a Kan complex.
It follows that $q$ is a Kan fibration as desired (Lemma \toposref{toothie2}).
\end{proof}

Our next goal is to show how to use Propositions \ref{poststaf} and \ref{postprod} to prove Theorem \ref{theo}. We begin by describing the class of weak equivalences in $\mset{ \CatP}$.

\begin{definition}
Let $\CatP$ be a categorical pattern on a simplicial set $S$. We will say that a morphism $f: \overline{X} \rightarrow \overline{Y}$ in $\mset{ \CatP}$ is a {\it $\CatP$-equivalence} if, for every $\CatP$-fibered object $\overline{Z} \in \mset{ \CatP}$, the induced map
$$ \bHom^{\sharp}_{S}( \overline{Y}, \overline{Z}) \rightarrow \bHom^{\sharp}_{S}( \overline{X}, \overline{Z} )$$
is a homotopy equivalence of Kan complexes.
\end{definition}

\begin{example}\label{inker}
Any $\CatP$-anodyne morphism is a $\CatP$-equivalence; this follows immediately
from Proposition \ref{stilk}.
\end{example}

\begin{lemma}\label{caker}
Let $\CatP$ be a categorical pattern on a simplicial set $S$, and suppose we are given a pushout diagram $$ \xymatrix{ \overline{X} \ar[r]^{f} \ar[d] & \overline{Y} \ar[d] \\
\overline{X}' \ar[r]^{f'} & \overline{Y}' }$$
in $\mset{ \CatP}$. Assume that the vertical maps are cofibrations. If $f$ is a $\CatP$-equivalence,
then $f'$ is a $\CatP$-equivalence.
\end{lemma}

\begin{proof}
Let $\overline{Z} \in \mset{ \CatP}$ be $\CatP$-fibered. We have a pullback diagram of simplicial sets
$$ \xymatrix{ \bHom^{\sharp}_{S}( \overline{X}, \overline{Z} ) & \bHom^{\sharp}_{S}( \overline{Y}, \overline{Z} ) \ar[l] \\
\bHom^{\sharp}_{S}( \overline{X}', \overline{Z} ) \ar[u] & \bHom^{\sharp}_{S}( \overline{Y}', \overline{Z}). }$$
Proposition \ref{stilk} implies that the vertical maps are Kan fibrations, so that the diagram is also a homotopy pullback square. Since $f$ is a $\CatP$-equivalence, the upper horizontal maps is
a homotopy equivalence of Kan complexes. It follows that the lower horizontal map is also a homotopy equivalence of Kan complexes, as desired.
\end{proof}

\begin{lemma}\label{piner}
Let $\CatP = (M_S, T, \{ p_{\alpha}: K_{\alpha}^{\triangleleft} \rightarrow S \}_{\alpha \in A})$ be a categorical pattern on a simplicial set $S$, and let $f: \overline{X} \rightarrow \overline{Y}$ be a map between $\CatP$-fibered objects $\overline{X} = (X,M), \overline{Y} = (Y,M')$ of $\mset{ \CatP}$. The
following conditions are equivalent:
\begin{itemize}
\item[$(1)$] The map $f$ is a $\CatP$-equivalence.
\item[$(2)$] The map $f$ admits a homotopy inverse; that is, there exists a map
$g: \overline{Y} \rightarrow \overline{X}$ in $\mset{ \CatP}$ and homotopies
$$h: ( \Delta^1)^{\sharp} \times \overline{X} \rightarrow \overline{X}
\quad \quad h': ( \Delta^1)^{\sharp} \times \overline{Y} \rightarrow \overline{Y}$$
connecting $g \circ f$ and $f \circ g$ to $\id_{ \overline{X}}$ and $\id_{ \overline{Y} }$, respectively.
\item[$(3)$] For every edge $\Delta^1 \rightarrow S$, the induced map
$X \times_{ S} \Delta^1 \rightarrow Y \times_{S} \Delta^1$ is an equivalence of $\infty$-categories.
\end{itemize}
If every edge of $S$ belongs to $M_S$, then $(3)$ can be replaced by the following apparently
weaker condition:
\begin{itemize}
\item[$(3')$] For every vertex $s \in S$, the induced map $X_{s} \rightarrow Y_{s}$ is an equivalence
of $\infty$-categories.
\end{itemize}
\end{lemma}

\begin{proof}
The equivalence of $(1)$ and $(2)$ is formal, and the implications $(2) \Rightarrow (3) \Rightarrow
(3')$ are clear. If every edge of $S$ belongs to $M_S$, then the implication $(3') \Rightarrow (3)$ follows
from Corollary \toposref{usefir}. To complete the proof, let us suppose that $f$ satisfies
$(3)$. We will say that an object $\overline{W} = (W,M'') \in \mset{ \CatP}$ is {\it good} if composition
with $f$ induces a homotopy equivalence
$$ \bHom_{S}^{\sharp}( \overline{W}, \overline{X} ) \rightarrow \bHom_{S}^{\sharp}( \overline{W}, \overline{Y} ).$$
Our goal is to prove that every object $\overline{W} \in \mset{ \CatP}$ is good. The proof proceeds in several steps:

\begin{itemize}
\item[$(a)$] We have a commutative diagram
$$ \xymatrix{  \bHom_{S}^{\sharp}( \overline{W}, \overline{X} ) \ar[r] \ar[d] & \bHom_{S}^{\sharp}( \overline{W}, \overline{Y} ) \ar[d] \\
 \bHom_{S}^{\sharp}( W^{\flat}, \overline{X} ) \ar[r] &  \bHom_{S}^{\sharp}( W^{\flat}, \overline{Y} ). }$$
The left vertical map exhibits $\bHom_{S}^{\sharp} ( \overline{W}, \overline{X} )$ as the full simplicial
subset of $\bHom_{S}^{\sharp}( W^{\flat}, \overline{X} )$ spanned by those maps
$W \rightarrow X$ which carry every edge in $M''$ to a locally $p$-coCartesian edge of $X$, where
$p: X \rightarrow S$ denotes the projection, and the right vertical map admits a similar description
in terms of the projection $q: Y \rightarrow S$.
Assumption $(3)$ implies that an edge of $X$ is locally $p$-coCartesian if and only if its image in $Y$ is locally $q$-coCartesian. Consequently, to prove that $\overline{W}$ is good, it will suffice to show that
$W^{\flat}$ is good.

\item[$(b)$] Suppose given a pushout diagram
$$ \xymatrix{ V \ar[r] \ar[d] & V' \ar[d] \\
W \ar[r] & W' }$$
in the category of simplicial sets over $S$, where the vertical maps are cofibrations. We then obtain pullback diagram
$$ \xymatrix{ \bHom_{S}^{\sharp}( V^{\flat}, \overline{X} ) & \bHom_{S}^{\sharp}( {V'}^{\flat}, \overline{X} ) \ar[l] & \bHom_{S}^{\sharp}( V^{\flat}, \overline{Y} ) & \bHom_{S}^{\sharp}( {V'}^{\flat}, \overline{Y}) \ar[l] \\
\bHom_{S}^{\sharp}( W^{\flat}, \overline{X}) \ar[u] & \bHom_{S}^{\sharp}( {W'}^{\flat}, \overline{X}) \ar[u] \ar[l] & \bHom_{S}^{\sharp}( W^{\flat}, \overline{Y}) \ar[u] & \bHom_{S}^{\sharp}( W^{\flat}, \overline{Y}) \ar[u] \ar[l]. }$$
Proposition \ref{stilk} implies that the vertical maps are Kan fibrations, so both diagrams are homotopy pullback squares. It follows that if $V^{\flat}$, ${V'}^{\flat}$, and $W^{\flat}$ are good, then
${W'}^{\flat}$ is good.

\item[$(c)$] Let $\Delta^n \rightarrow S$ be a map; then $( \Delta^n)^{\flat}$ is good for $n \leq 1$; this follows immediately from $(3)$.

\item[$(d)$] For any map $\Delta^n \rightarrow S$, the object
$( \Delta^{ \{0,1\} } \coprod_{ \{0\} } \ldots \coprod_{ \{n-1\} } \Delta^{ \{n-1,n\} })^{\flat} \in \mset{ \CatP}$
is good; this follows from $(b)$ and $(c)$.

\item[$(e)$] Let $u: \overline{W} \rightarrow \overline{W}'$ be a $\CatP$-equivalence (for
example, any $\CatP$-anodyne map). Then $\overline{W}$ is good if and only if $\overline{W}'$ is good.

\item[$(f)$] For any map $\Delta^n \rightarrow S$, the resulting object
$( \Delta^n)^{\flat} \in \mset{ \CatP}$ is good. This follows from $(e)$ and $(d)$, since the
inclusion 
$( \Delta^{ \{0,1\} } \coprod_{ \{0\} } \ldots \coprod_{ \{n-1\} } \Delta^{ \{n-1,n\} })^{\flat} \subseteq (\Delta^n)^{\flat}$ is $\CatP$-anodyne (Example \ref{inker}).

\item[$(g)$] The collection of good objects in $\mset{\CatP}$ is closed under coproducts
(since a product of homotopy equivalences between Kan complexes is again a homotopy equivalence).

\item[$(h)$] If the simplicial set $W$ is finite-dimensional, then $W^{\flat} \in \mset{ \CatP}$ is good.
The proof goes by induction on the dimension $n \geq 0$ of $W$. If $W$ is empty, then the result is obvious. Otherwise, let $K$ denote the set of nondegenerate $n$-simplices of $W$, and let
$W'$ denote the $(n-1)$-skeleton of $W$. We have a pushout diagram
$$ \xymatrix{ K \times \bd \Delta^n \ar[r] \ar[d] & W' \ar[d] \\
K \times \Delta^n \ar[r] & W. }$$
The inductive hypothesis guarantees that $(K \times \bd \Delta^n)^{\flat}$ and ${W'}^{\flat}$ are good, and $(K \times \Delta^n)^{\flat}$ is good by virtue of $(g)$ and $(f)$. It follows from
$(b)$ that $W^{\flat}$ is good.

\item[$(i)$] Suppose that $W$ is obtained as the direct limit of a sequence of inclusions
$$ W(0) \rightarrow W(1) \rightarrow W(2) \rightarrow \ldots $$
Then $\bHom_{S}^{\sharp}( W^{\flat}, \overline{X} )$ can be obtained as the homotopy inverse limit of the tower $\{ \bHom_{S}^{\sharp}( W(n)^{\flat}, \overline{X} ) \}_{n \geq 0}$, and $
\bHom_{S}^{\sharp}( W^{\flat}, \overline{Y} )$ can be described similarly. It follows that
if each $W(n)^{\flat}$ is good, then $W^{\flat}$ is good.

\item[$(j)$] For every map of simplicial sets $W \rightarrow S$, the object
$W^{\flat} \in \mset{ \CatP}$ is good. This follows from $(h)$ and $(i)$,
if we take $W(n)$ to be the $n$-skeleton of $W$.
\end{itemize}
\end{proof}

We are now ready to prove Theorem \ref{theo}:

\begin{proof}
Let $\CatP = (M_S, T, \{ p_{\alpha}: K_{\alpha}^{\triangleleft} \rightarrow S \}_{\alpha \in A})$
be a categorical pattern on the simplicial set $S$. Assume for the moment that each of the simplicial sets $K_{\alpha}$ is finite. It follows from the small object argument that there exists a functor
$E: \mset{ \CatP} \rightarrow \mset{ \CatP}$ and a natural transformation
$\alpha: \id \rightarrow T$ with the following properties:
\begin{itemize}
\item[$(a)$] The functor $E$ commutes with filtered colimits.
\item[$(b)$] For every object $\overline{X} \in \mset{ \CatP}$, the object
$E \overline{X} \in \mset{ \CatP}$ has the extension property with respect to every
$\CatP$-anodyne map (and is therefore $\CatP$-fibered, by virtue of Proposition \ref{poststaf}.
\item[$(c)$] For every object $\overline{X} \in \mset{\CatP}$, the map 
$\overline{X} \rightarrow E \overline{X}$ is $\CatP$-anodyne.
\end{itemize}
Let $f: \overline{X} \rightarrow \overline{Y}$ be a morphism in $\mset{\CatP}$. 
It follows from $(c)$ and Example \ref{inker} that $f$ is a $\CatP$-equivalence if and only
if $E(f)$ is a $\CatP$-equivalence. Using $(b)$ and Lemma \ref{piner}, we deduce that 
$f$ is an equivalence if and only if for each edge $e: \Delta^1 \rightarrow S$, the map
$E(f)$ induces a categorical equivalence of simplicial sets after pulling back along $e$.
Using $(a)$ and Corollary \toposref{perfpull}, we deduce that the collection of $\CatP$-equivalences
in $\mset{ \CatP}$ is perfect, in the sense of Definition \toposref{perfequiv}.

We now wish to deduce the existence of a left proper, combinaorial model structure on
$\mset{ \CatP}$ such that the cofibrations are the monomorphisms and the weak equivalences
are given by the $\CatP$-equivalences. It will suffice to show that $\mset{ \CatP}$ satisfies
the hypotheses of Proposition \toposref{goot}:

\begin{itemize}
\item[$(1)$] The collection of $\CatP$-equivalences is perfect: this follows from the above arguments.
\item[$(2)$] The collection of $\CatP$-equivalences is stable under pushouts by cofibrations:
this follows from Lemma \ref{caker}.
\item[$(3)$] Let $f: \overline{X} \rightarrow \overline{Y}$ be a morphism in $\mset{\CatP}$ which has
the right lifting property with respect to every cofibration; we wish to show that $f$ is a 
$\CatP$-equivalence. To prove this, it suffices to observe that $f$ admits a section $s$
and that the composition $s \circ f: \overline{X} \rightarrow \overline{X}$ is homotopic to the identity
(that is, there exists a homotopy $h: \overline{X} \times (\Delta^1)^{\sharp} \rightarrow \overline{X}$
from $\id_{\overline{X}}$ to $s \circ f$ in the category $\mset{ \CatP}$).
\end{itemize}

We next claim that the simplicial structure on $\mset{ \CatP}$ is compatible with its model structure.
In view of Proposition \toposref{testsimpmodel}, it will suffice to prove that for every object
$\overline{X} \in \mset{ \CatP}$ and each $n \geq 0$, the projection map
$p: \overline{X} \times ( \Delta^n)^{\sharp} \rightarrow \overline{X}$ is a $\CatP$-equivalence.
The inclusion $i: \{0\}^{\sharp} \subseteq (\Delta^n)^{\sharp}$ determines a section $s$ of
$p$; it will therefore suffice to show that $s$ is a $\CatP$-equivalence. Lemma \ref{kalel} implies
that $i$ is $\CatP_0$-anodyne (where $\CatP_0$ is defined as in the statement of Lemma \ref{kalel}).
Using Proposition \ref{postprod}, we conclude that $s$ is $\CatP$-anodyne, so that 
$s$ is a $\CatP$-equivalence by Example \ref{inker}.

We now discuss the case of a general categorical pattern
$\CatP = (M_S, T, \{ p_{\alpha}: K_{\alpha}^{\triangleleft} \rightarrow S' \}_{\alpha \in A})$ on $S$.
Let $\CatP' = (M_S, T, \emptyset)$. We have already shown that $\mset{\CatP'}$ has the structure of a left proper combinatorial simplicial model category. We may therefore define a model structure on the category $\mset{ \CatP}$ to be the localization of $\mset{ \CatP}$ with respect to the generating $\CatP$-anodyne maps appearing in Definition \ref{postspunt}. It follows from Proposition \toposref{suritu} that
$\mset{ \CatP}$ is again a left proper combinatorial simplicial model category.

To complete the proof, it will suffice to show that an object $\overline{X} \in \mset{ \CatP}$ is 
fibrant if and only if it is $\CatP$-fibered. It follows from Proposition \toposref{suritu} that
$\overline{X}$ is fibrant if and only if the following conditions are satisfied:
\begin{itemize}
\item[$(i)$] The object $\overline{X}$ is fibrant in $\mset{ \CatP'}$: that is, 
$\overline{X}$ has the extension property with respect to every cofibration
$f: \overline{Y} \rightarrow \overline{Y}'$ which is a $\CatP'$-equivalence.
\item[$(ii)$] For every generating $\CatP$-anodyne map $f: \overline{Y} \rightarrow \overline{Y}'$,
the induced map $q: \bHom_{S}^{\sharp}( \overline{Y}', \overline{X} )
\rightarrow \bHom_{S}^{\sharp}( \overline{Y}, \overline{X} )$ is a homotopy equivalence of Kan complexes.
\end{itemize}
Suppose that $\overline{X}$ satisfies $(ii)$. Note that for every $\CatP$-anodyne map $f$,
the map $q$ is a Kan fibration (Proposition \ref{stilk}), and therefore a trivial Kan fibration.
It follows that $q$ is surjective on vertices, so that $\overline{X}$ has the extension property
with respect to every $\CatP$-anodyne map and is therefore $\CatP$-fibered by
virtue of Proposition \ref{poststaf}.

Conversely, suppose that $\overline{X}$ is $\CatP$-fibered; we wish to show that
$\overline{X}$ satisfies conditions $(i)$ and $(ii)$. To prove $(i)$, consider the map
$q: \bHom_{S}^{\sharp}( \overline{Y}', \overline{X} ) \rightarrow \bHom_{S}^{\sharp}( \overline{Y}, \overline{X} )$. This map is a Kan fibration (Proposition \ref{stilk}) and a homotopy equivalence
by virtue of our assumption that $f$ is a $\CatP'$-equivalence (since $\overline{X}$ is $\CatP'$-fibered).
It follows that $q$ is a trivial Kan fibration and therefore surjective on vertices, which proves $(i)$.
To prove $(ii)$, it will suffice to show that $q$ is a trivial Kan fibration whenever $f$
is $\CatP$-anodyne. To see that $q$ has the right lifting property with respect to the inclusion
$\bd \Delta^n \subseteq \Delta^n$, we need to show that $\overline{X}$ has the extension property with respect to the induced inclusion
$$f' = ( \overline{Y} \times (\Delta^n)^{\sharp} ) \coprod_{ \overline{Y} \times (\bd \Delta^n)^{\sharp}
} ( \overline{Y}' \times ( \bd \Delta^n)^{\sharp}) \subseteq \overline{Y}' \times (\Delta^n)^{\sharp}.$$
This follows from Proposition \ref{poststaf}, since $f'$ is $\CatP$-anodyne by virtue of
Proposition \ref{postprod}.
\end{proof}

\begin{remark}\label{slaver}
Let $\CatP$ and $\CatP'$ be categorical patterns, and let $\CatP \times \CatP'$ be defined as in
Definition \ref{prodpat}. The formation of Cartesian products induces a functor
$$ F: \mset{ \CatP} \times \mset{ \CatP'} \rightarrow \mset{ \CatP \times \CatP' }.$$
With respect to the model structures of Theorem \ref{theo}, the map $F$ is a left Quillen bifunctor.
To prove this, we must show that if $f: \overline{X} \rightarrow \overline{X}'$ is
a cofibration in $\mset{ \CatP}$ and $g: \overline{Y} \rightarrow \overline{Y}'$ is a cofibration
in $\mset{ \CatP'}$, then the induced map
$$ f \wedge g: ( \overline{X}' \times \overline{Y}) \coprod_{ \overline{X} \times \overline{Y}}
( \overline{X} \times \overline{Y}') \rightarrow \overline{X}' \times \overline{Y}'$$
is a cofibration, which is trivial if either $f$ or $g$ is trivial. The first claim is obvious, and the second is equivalent to the requirement that the diagram
$$ \xymatrix{ \overline{X} \times \overline{Y} \ar[d] \ar[r]^{i} & \overline{X}' \times \overline{Y} \ar[d] \\
\overline{X} \times \overline{Y}' \ar[r]^{j} & \overline{X}' \times \overline{Y}' }$$
is a homotopy pushout square. For this, it suffices to show that the horizontal maps are weak equivalences. We will prove that $i$ is a weak equivalence; the proof that $j$ is a weak equivalence is similar. Choose a $\CatP$-anodyne map $f': \overline{X}' \rightarrow \overline{X}''$, where
$\overline{X}''$ is $\CatP$-fibered. Proposition \ref{postprod} guarantees that 
the map $\overline{X}' \times \overline{Y} \rightarrow \overline{X}'' \rightarrow \overline{Y}$ is
$(\CatP \times \CatP')$-anodyne. It therefore suffices to show that the composite map
$\overline{X} \times \overline{Y} \rightarrow \overline{X}'' \times \overline{Y}$ is a 
$(\CatP \times \CatP')$-equivalence. We may therefore replace $\overline{X}'$ by
$\overline{X}''$ and thereby reduce to the case where $\overline{X}'$ is $\CatP$-fibered.
By a similar argument, we can assume that the map $\overline{X} \rightarrow \overline{X}'$
has the right lifting property with respect to all $\CatP$-anodyne morphisms, so that
$\overline{X}$ is $\CatP$-fibered as well. The $\CatP$-equivalence $f$ now admits a homotopy inverse, so that the induced map $\overline{X} \times \overline{Y} \rightarrow \overline{X}' \times \overline{Y}$ admits a homotopy inverse as well.
\end{remark}

\begin{remark}
Let $\CatP$ be a categorical pattern, and let $\mset{ \CatP}$ be endowed with the model structure of
Theorem \ref{theo}. Then the weak equivalences in $\mset{ \CatP}$ are precisely the
$\CatP$-equivalences.
\end{remark}

The remainder of this section is devoted to the proofs of Propositions \ref{poststaf} and
\ref{postprod}.

\begin{lemma}\label{kisker}
Let $n \geq 2$, and let $p: X \rightarrow \Delta^{n}$ be an inner fibration of simplicial sets. Consider a commutative diagram
$$ \xymatrix{ \Lambda^{n}_0 \ar[r]^{f_0} \ar@{^{(}->}[d] & X \ar[d]^{p} \\
\Delta^n \ar[r]^{\id} \ar@{-->}[ur]^{f} & \Delta^{n}, }$$
where $f_0$ carries $\Delta^{ \{0,1\} } \subseteq \Lambda^n_0$ to a locally $p'$-coCartesian edge
of $X \times_{ \Delta^n} \Delta^{ \{0,1,n\}}$, where $p'$ denotes the projection
$X \times_{ \Delta^n} \Delta^{ \{0,1,n\} } \rightarrow \Delta^{ \{0,1,n\} }$. 
Then there exists a map $f: \Delta^n \rightarrow X$ as indicated, rendering the diagram commutative.
\end{lemma}

\begin{proof}
To prove the assertion, it will suffice to show that $f_0$ extends to an $n$-simplex of $X$
(the compatibility with the projection $p$ is automatic, since $\Lambda^{n}_0$ contains every vertex of
$\Delta^n$). Choose a categorical equivalence $X \rightarrow \tNerve(\calC)$, where $\calC$
is a topological category (for example, we could take $\calC = | \sCoNerve[X] |$). Note that
the projection $p$ factors (uniquely) through some projection map $\tNerve(\calC) \rightarrow \Delta^n$.
Since $p$ is an inner fibration, the simplicial set $X$ is an $\infty$-category, and therefore fibrant with respect to the Joyal model structure. Consequently, it will suffice to prove the existence of the desired extension after replacing $X$ by $\tNerve(\calC)$. We may therefore assume that $X$ is the nerve of a topological category $\calC$.

The functor $f_0$ determines the following data in the topological category $\calC$:
\begin{itemize}
\item[$(1)$] A collection of objects $C_{i} = f_0( \{i \})$. 
\item[$(2)$] A morphism $\alpha: C_0 \rightarrow C_1$ in $\calC$, given by evaluating
$f_0$ on the edge $\Delta^{ \{0,1\} } \subseteq \Lambda^{n}_0$.
Let $q: \bHom_{\calC}(C_1, C_n) \rightarrow \bHom_{\calC}(C_0, C_n)$ be the map induced by composition with $\alpha$. Since $\alpha$ is locally $p$-coCartesian, it is coCartesian with
respect to the projection $X \times_{ \Delta^n} \Delta^{ \{0,1, n\}} \rightarrow
\Delta^{ \{0,1,n\} }$, so that $q$ is a weak homotopy equivalence.
\item[$(3)$] A continuous map $g_0: \bd [0,1]^{n-2} \rightarrow \bHom_{\calC}( C_1, C_n)$, given by
evaluating $f_0$ on $\bd \Delta^{ \{1, 2, \ldots, n} \}$. 
\item[$(4)$] Another continuous map 
$$H_0: ((\bd [0,1]^{n-2}) \times [0,1]) \coprod_{ \bd [0,1]^{n-2} \times \{0\} } ([0,1]^{n-2} \times \{0\})
\rightarrow \bHom_{\calC}(C_0, C_n)$$ such that the restriction
$H_0 | ( \bd [0,1]^{n-2} \times \{1\} )$ coincides with the composition
$$ \bd [0,1]^{n-2} \stackrel{g_0}{\rightarrow} \bHom_{\calC}(C_1, C_n)
\stackrel{q}{\rightarrow} \bHom_{\calC}(C_0, C_n).$$
\end{itemize}

Let $h_1 = H_0 | ( [0,1]^{n-2} \times \{0\} )$.
We can regard the restriction $H_0 | ( \bd [0,1]^{n-2} \times [0,1] )$ as a homotopy
from $q \circ g_0$ to $h_1 | \bd [0,1]^{n-2}$. Unwinding the definitions, we see that
producing the desired extension $f$ is equivalent to extending $H_0$ to a homotopy
from $q \circ g$ to $h_1$, for some continuous map $g: [0,1]^{n-2} \rightarrow \bHom_{\calC}(C_1, C_n)$. The existence of $H$ (and $g$) now follows easily from the fact that $q$ is a weak homotopy equivalence.
\end{proof}

\begin{proof}[Proof of Proposition \ref{poststaf}]
Let $\CatP = (M_S, T, \{ p_{\alpha}: K_{\alpha}^{\triangleleft} \rightarrow S \}_{\alpha \in A})$ be
a categorical pattern on the simplicial set $S$, and let $\overline{X}$ be an object of $\mset{ \CatP}$. We wish to show that $\overline{X}$ is $\CatP$-fibered if and only if it has the extension property with respect to every $\CatP$-anodyne morphism. We begin by proving the ``if'' direction. Let $\overline{X} = (X,M)$, and let $q: X \rightarrow S$ denote the underlying map of simplicial sets.  We will show that
$\overline{X}$ satisfies conditions $(1)$, $(2)$, $(3)$, $(4)$ and $(6)$ of Definition \ref{catpat},
together with condition $(5')$ of Remark \ref{ratpat}:

\begin{itemize}
\item[$(1)$] We must show that $q: X \rightarrow S$ is an inner fibration. This is equivalent
to our assumption that $\overline{X}$ has the unique extension property with respect to every
morphism of type $(C_1)$ appearing in Definition \ref{postspunt}.

\item[$(2)$] Let $\overline{e}: \Delta^1 \rightarrow S$ belong to $M_S$, let
$X' = X \times_{S} \Delta^1$, and let $q': X' \rightarrow \Delta^1$ denote the projection map.
We wish to prove that $q'$ is a coCartesian fibration. Let $M'$ denote the collection
of edges in $X'$ whose image in $X$ belongs to $M$.
Since $\overline{X}$ has the extension property with respect to morphisms of
the type $(C_0)$ appearing in Definition \ref{postspunt}, we deduce that every
edge of $M'$ is $q'$-coCartesian. The existence of a sufficient supply of such edges follows
from the assumption that $q$ has the extension property with respect to morphisms of type
$(B_0)$.

\item[$(3)$] Let $\overline{e}$, $X'$, and $q'$ be as in $(2)$. 
We claim that an edge $e: x \rightarrow y$ of $X'$ lifting $\overline{e}$ is $q'$-coCartesian if and only if
$e \in M'$. The ``if'' direction follows from the above arguments. To prove the converse, we
first treat the case where the edge $\overline{e}$ is degenerate, corresponding to a vertex
$s \in S$. Let $X_{s}$ denote the $\infty$-category $X \times_{S} \{s\}$, so that
$e$ is an equivalence in $X_{s}$ and therefore belongs to the largest Kan complex
$Y$ contained in $X_{s}$. Let $Q = \Delta^0 \coprod_{ \Delta^{ \{0,2\} }}
\Delta^3 \coprod_{ \Delta^{ \{1,3\} } } \Delta^0$, and let $Q'$ denote the image
of $\Delta^{ \{1,2\} } \subseteq \Delta^3$ in $K$. The inclusion $Q' \subseteq Q$
is a weak homotopy equivalence. Consequently, the map $Q' \rightarrow Y$ determined
by the edge $e$ extends to a map $Q \rightarrow Y$. Since $\overline{X}$
has the extension property with respect to morphisms of type $(A_1)$ appearing in Definition \ref{postspunt}, we deduce that the induced map $Q \rightarrow X$ carries each edge of $Q$ into
$M$, so that $e \in M$.

We now treat the general case where $\overline{e}$ is not assumed to be degenerate.
Using the extension property with respect to morphisms of type $(B_0)$, we can choose
an edge $e': x \rightarrow y'$ in $M'$ which lies over $\overline{e}$.
Since $e'$ is $q'$-coCartesian, we can choose a $2$-simplex
$$ \xymatrix{ & y' \ar[dr]^{e''} & \\
x \ar[ur]^{e'} \ar[rr]^{e} & & y }$$
lying over the edge $\overline{e}$ in $S'$, where $e''$ is an edge of the fiber $X'_{q'(y)}$. Since $e$ is also $q'$-coCartesian, we
deduce that $e''$ is an equivalence in $X'_{q'(y)}$, so that $e'' \in M$ by the above argument.
Invoking our assumption that $\overline{X}$ has the extension property with respect
to morphisms of the type $(A_0)$, we deduce that $e \in M'$, as desired.

\item[$(4)$] Let $\Delta^2 \rightarrow S$ be a $2$-simplex which belongs to $T$, let
$X' = X \times_{S} \Delta^2$, and let $e$ be an edge of $X'$ lying over
$\Delta^{ \{0,1\} }$ whose image in $X$ belongs to $M$. We wish to prove that
$e$ is $q'$-coCartesian, where $q'$ denotes the projection map
$X' \rightarrow \Delta^2$. This follows immediately from our assumption that
$\overline{X}$ has the extension property with respect to morphisms of the type
$(C_0)$.

\item[$(5')$] Fix an index $\alpha \in A$. Let $q_{\alpha}: X \times_{S} K_{\alpha}^{\triangleleft} \rightarrow K_{\alpha}^{\triangleleft}$ denote the projection map, and let
$q_{\alpha}^{0}: K \times_{S} K_{\alpha} \rightarrow K_{\alpha}$ its restriction.
We must show that every coCartesian section of $q_{\alpha}^{0}$ can be extended to
a coCartesian section of $q_{\alpha}$. In view of $(3)$, this is equivalent to the requirement
that $\overline{X}$ have the extension property with respect to morphisms of type $(B_1)$ in
Definition \ref{postspunt}.

\item[$(6)$] Let $\alpha$ and $q_{\alpha}$ be as in $(4')$; we must show that every
coCartesian section of $q_{\alpha}$ is a $q$-limit diagram. In view of $(3)$, this is equivalent
the requirement that $\overline{X}$ has the extension property with respect to all morphisms
of type $(C_2)$ appearing in Definition \ref{postspunt}.
\end{itemize}

We now prove the ``only if'' direction. Assume that $\overline{X}$ is $\CatP$-fibered.
We will show that $\overline{X}$ has the extension property with respect to every
$\CatP$-anodyne morphism $f: A \rightarrow B$ in $\mset{ \CatP}$. It will suffice to
treat the case where $f$ is one of the generating $\CatP$-anodyne morphisms
appearing in Definition \ref{postspunt}. For morphisms of the types
$(B_1)$, $(C_1)$, and $(C_2)$, the relevant assertion follows from the arguments given above
in cases $(5')$, $(1)$, and $(6)$, respectively. There are several more cases to consider:

\begin{itemize}
\item[$(A_0)$] The map $f$ is an inclusion $( \Lambda^2_1)^{\sharp} \coprod_{ (\Lambda^2_{1})^{\flat}} (\Delta^2)^{\flat}  \subseteq (\Delta^2)^{\sharp},$ for some $2$-simplex $\Delta^2 \rightarrow S$
belonging to $T$. Let $X' = X \times_{S} \Delta^2$, and let $q': X' \rightarrow \Delta^2$ denote the projection. To prove that $\overline{X}$ has the extension property with respect to $f$, we must show that if we are given a $2$-simplex 
$$\xymatrix{ & y \ar[dr]^{g''} & \\
x \ar[ur]^{g'} \ar[rr]^{g} & & z }$$
in $X'$ such that $g'$ and $g''$ are locally $q'$-coCartesian, then $g$ is locally $q'$-coCartesian.
We observe that $g''$ is automatically $q'$-coCartesian, and the hypothesis that
$\overline{X}$ is $\CatP$-fibered guarantees that $g'$ is $q'$-coCartesian. It follows from
Proposition \toposref{protohermes} that $g$ is $q'$-coCartesian.


\item[$(A_1)$] The map $f$ is an inclusion $Q^{\flat} \subseteq Q^{\sharp}$, where
$Q = \Delta^0 \coprod_{ \Delta^{ \{0,2\} }} \Delta^3 \coprod_{ \Delta^{ \{1,3\} }} \Delta^0$,
and the map $Q \rightarrow S$ carries each edge of $Q$ into $M_S$ and each $2$-simplex of
$Q$ into $T$. Let $X' = X \times_{S} Q$ and let $q': X' \rightarrow Q$ denote the projection map.
It follows from Corollary \toposref{gottaprime} that $q'$ is a coCartesian fibration, classified
by some functor $\chi: Q \rightarrow \Cat_{\infty}$. Since the projection $Q \rightarrow \Delta^0$ is a categorical equivalence, the functor $\chi$ is equivalent to a constant functor; it follows that
$X'$ is equivalent to a product $Q \times \calC$, for some $\infty$-category $\calC$.
To show that $\overline{X}$ has the extension property with respect to $f$, it suffices to show that
every section of $q'$ is coCartesian. Replacing $X'$ by $Q \times \calC$, we are reduced to proving
that every diagram $Q \rightarrow \calC$ carries each edge of $Q$ to an equivalence in $\calC$, which follows from a simple diagram chase.

\item[$(B_0)$] The map $f$ is an inclusion $\{0\}^{\sharp} \subseteq (\Delta^1)^{\sharp}$ lying over an
edge of $M_S$. Since the induced map $X \times_{S} \Delta^1 \rightarrow \Delta^1$ is a coCartesian fibration, the object $\overline{X}$ has the extension property with respect to $f$.

\item[$(C_0)$] The map $f$ is an inclusion 
$$ (\Lambda^n_0)^{\flat} \coprod_{ (\Delta^{ \{0,1\} })^{\flat} } ( \Delta^{ \{0,1\} })^{\sharp}
\subseteq ( \Delta^n )^{\flat} \coprod_{ ( \Delta^{ \{0,1\} })^{\flat} } ( \Delta^{ \{0,1\} })^{\sharp},$$
for every $n > 1$ and every map $\Delta^n \rightarrow S$ which carries 
$\Delta^{ \{0, 1, n\} }$ into $S'$. The desired result in this case is a reformulation of
Lemma \ref{kisker}.
\end{itemize}
\end{proof}

\begin{lemma}\label{carpal}
Let $\CatP = (M_S, T, \{ K_{\alpha}^{\triangleleft} \rightarrow S \}_{\alpha \in A} )$ be a categorical
pattern on a simplicial set $S$. Let $B_0 \subseteq B$ be an inclusion of simplicial sets, and
let $f: \Delta^1 \times B \rightarrow S$ be a map with the following properties:

\begin{itemize}
\item For every simplex $\sigma: \Delta^n \rightarrow B$ which does not belong to $B_0$,
let $\tau$ be the $2$-simplex of $\Delta^1 \times \Delta^n$ spanned by
$(0,0)$, $(1,0)$ and $(1,n)$. Then the induced map
$$ \Delta^2 \stackrel{\tau}{\rightarrow} \Delta^1 \times \Delta^n
\stackrel{\sigma}{\rightarrow} \Delta^1 \times B \stackrel{f}{\rightarrow} S$$
belongs to $T$. 

\item For every vertex $b$ of $B$, the map $f$ carries $\Delta^1 \times \{b\}$ into $M_S$.
\end{itemize}

Then the inclusion 
$$((\Delta^1)^{\sharp} \times B_0^{\flat}) \coprod_{ \{0\}^{\sharp} \times B_0^{\flat} }
(\{0\}^{\sharp} \times B^{\flat}) \subseteq (\Delta^1)^{\sharp} \times B^{\flat}$$
is $\CatP$-anodyne.
\end{lemma}

\begin{proof}
Working simplex-by-simplex, we can reduce to the case where
$B = \Delta^n$ and $B_0 = \bd \Delta^n$. The simplicial
set $\Delta^1 \times \Delta^n$ admits a filtration
$$ (\{0\} \times \Delta^n) \coprod_{ \{0\} \times \bd \Delta^n} ( \Delta^1 \times \bd \Delta^n)
= Z_0 \subset Z_1 \subset \ldots \subset Z_{n} \subseteq Z_{n+1} = \Delta^1 \times \Delta^n,$$
where each $Z_{i+1}$ is obtained from $Z_{i}$ by adjoining the $(n+1)$-simplex of
$\Delta^1 \times \Delta^n$ corresponding to the map
$$ \sigma_i: [n+1] \rightarrow [1] \times [n]$$
$$ \sigma_i(j) = \begin{cases} (0,j) & \text{if } j \leq n-i \\
(1, j-1) & \text{if } j > n-i. \end{cases}$$
Let $\overline{Z_i} = (Z_i, M_i)$ denote the marked simplicial set whose marked edges
are precisely those edges which are marked in $(\Delta^1)^{\sharp} \times (\Delta^1)^{\flat}$. We wish to show that the inclusion
$\overline{Z}_0 \subseteq \overline{Z}_{n+1}$ is $\CatP$-anodyne. For this, it will suffice to show that each of the inclusions $h_i: \overline{Z}_i \subseteq \overline{Z}_{i+1}$ is $\CatP$-anodyne. If 
$i=n=0$, then $h_i$ is a generating $\CatP$-anodyne morphism of type $(B_0)$.
If $0 \leq i < n$, then $h_i$ is a pushout of a generating $\CatP$-anodyne morphism of type $(C_1)$. If If $i=n>0$, then $h_i$ is a pushout of a generating $\CatP$-anodyne morphism of type $(C_0)$. 
\end{proof}

\begin{proof}[Proof of Proposition \ref{postprod}]
Let 
$\CatP = (M_S, T, \{ p_{\alpha}: K_{\alpha}^{\triangleleft} \rightarrow S\}_{\alpha \in A})$
and $\CatP' = (M'_{S'}, T', \{ q_{\beta}: L_{\beta}^{\triangleleft} \rightarrow S'\}_{\beta \in B})$
be categorical patterns on simplicial sets $S$ and $S'$, respectively.
Let $f: \overline{X} \rightarrow \overline{Y}$ be a $\CatP$-anodyne morphism
in $\mset{ \CatP}$, and let
$f': \overline{X}' \rightarrow \overline{Y}'$ be an arbitrary cofibration in $\mset{ \CatP'}$.
We wish to show that $f \wedge f'$ is $\CatP \times \CatP'$-anodyne. Without loss of generality,
we may assume that $f'$ is a generator for the class of cofibrations in $\mset{ \CatP'}$, having
either the form $(\Delta^1)^{\flat} \subseteq (\Delta^1)^{\sharp}$ or $(\bd \Delta^m)^{\flat} \subseteq (\Delta^m)^{\flat}$. Similarly, we may assume that $f$ is one of the generating
$\CatP$-anodyne morphisms described in Definition \ref{postspunt}. There
are fourteen cases to consider:
\begin{itemize}
\item[$(A_0)$] The map $f$ is an inclusion $( \Lambda^2_1)^{\sharp} \coprod_{ (\Lambda^2_{1})^{\flat}} (\Delta^2)^{\flat}  \subseteq (\Delta^2)^{\sharp}$ where $\Delta^2 \rightarrow S$ belongs to
$T$ and carries every edge into $M_S$, and $f'$ is an inclusion
$(\Delta^1)^{\flat} \subseteq (\Delta^1)^{\sharp}$. In this case, $f \wedge f'$ can be obtained as a composition of two morphisms, each of which is a pushout of a morphism having type $(A_0)$.

\item[$(A_1)$] The map $f$ is an inclusion $Q^{\flat} \subseteq Q^{\sharp}$, where
$Q = \Delta^0 \coprod_{ \Delta^{ \{0,2\} }} \Delta^3 \coprod_{ \Delta^{ \{1,3\} }} \Delta^0$
and the map $Q \rightarrow S$ carries every edge of $Q$ into $M_S$ and every $2$-simplex
of $Q$ into $T$, and $f'$ is an inclusion $(\Delta^1)^{\flat} \subseteq (\Delta^1)^{\sharp}$.
In this case, $f \wedge f'$ can be obtained as a successive pushout of two morphisms
of type $(A_0)$.

\item[$(B_0)$] The map $f$ is an inclusion $\{0\}^{\sharp} \subseteq (\Delta^1)^{\sharp}$, for some
edge $\Delta^1 \rightarrow S$ belonging to $M_S$, and $f'$ is an inclusion
$(\Delta^1)^{\flat} \subseteq (\Delta^1)^{\sharp}$. In this case, $f \wedge f'$ can be obtained
as a composition of two morphisms which are pushouts of maps of type $(A_0)$ and
the $\CatP$-anodyne morphism of Lemma \ref{catchrat}.

\item[$(B_1)$] For some $\alpha \in A$, the map $f$ is an inclusion
$K_{\alpha}^{\sharp} \subseteq
(K_{\alpha}^{\triangleleft})^{\sharp}$ (where $K_{\alpha}^{\triangleleft}$ maps to $S$ via
$p_{\alpha}$), and $f'$ is an inclusion
$(\Delta^1)^{\flat} \subseteq (\Delta^1)^{\sharp}$. We can factor the morphism
$f \wedge f'$ as a composition
$$ ( K_{\alpha}^{\triangleleft} \times \Delta^1, M) \stackrel{g}{\rightarrow}
(K_{\alpha}^{\triangleleft} \times \Delta^1, M') \stackrel{g'}{\rightarrow}
(K_{\alpha}^{\triangleleft} \times \Delta^1)^{\sharp},$$
where $M'$ is the collection of all edges of $K_{\alpha}^{\triangleleft} \times \Delta^1$
except for $\{v\} \times \Delta^1$, where $v$ is the cone point of $K_{\alpha}^{\triangleleft}$, and
$M \subseteq M'$ is the collection of all those edges which do not join $(v, 0)$ to
a vertex of $K_{\alpha}^{\triangleleft} \times \{1\}$. We begin by observing that $g$ is a pushout of a
coproduct of morphisms of type $(A_0)$, indexed by the collection of vertices of $K_{\alpha}$.
It will therefore suffice to show that $g'$ is $(\CatP \times \CatP')$-anodyne, which follows from
the observation that $g'$ is a pushout of a morphism of the type described in Lemma \ref{swindler}.

\item[$(C)$] The map $f$ is a generating $\CatP$-anodyne morphism of one of the types
$(C_0)$, $(C_1)$, or $(C_2)$ described in Definition \ref{postspunt}, and $f'$ is an inclusion
$(\Delta^1)^{\flat} \subseteq (\Delta^1)^{\sharp}$. In this case, $f \wedge f'$ is an isomorphism and there is nothing to prove.

\item[$(A'_0)$] The map $f$ is an inclusion $( \Lambda^2_1)^{\sharp} \coprod_{ (\Lambda^2_{1})^{\flat}} (\Delta^2)^{\flat}  \subseteq (\Delta^2)^{\sharp}$ where $\Delta^2 \rightarrow S$ belongs to
$T$ and carries every edge into $M_S$, and $f'$ is an inclusion
$(\bd \Delta^m)^{\flat} \subseteq (\Delta^m)^{\flat}$. If $m=0$, then $f \wedge f'$ is a
generating $(\CatP \times \CatP')$-anodyne morphism of type $(A_0)$. If $m > 0$, then
$f \wedge f'$ is an isomorphism.

\item[$(A'_1)$] The map $f$ is an inclusion $Q^{\flat} \subseteq Q^{\sharp}$, where
$Q = \Delta^0 \coprod_{ \Delta^{ \{0,2\} }} \Delta^3 \coprod_{ \Delta^{ \{1,3\} }} \Delta^0$
and the map $Q \rightarrow S$ carries every edge of $Q$ into $M_S$ and every $2$-simplex
of $S$ into $T$, and $f'$ is an inclusion
$(\bd \Delta^m)^{\flat} \subseteq (\Delta^m)^{\flat}$. If $m=0$ then $f \wedge f'$ is
a generating $(\CatP \times \CatP')$-anodyne morphism of type $(A_1)$, and if
$m > 0$ then $f \wedge f'$ is an isomorphism.

\item[$(B'_0)$] The map $f$ is an inclusion $\{0\}^{\sharp} \subseteq (\Delta^1)^{\sharp}$, for some
edge $\Delta^1 \rightarrow S$ belonging to $M_S$, and $f'$ is an inclusion
$(\bd \Delta^m)^{\flat} \subseteq (\Delta^m)^{\flat}$. If $m=0$, then $f \wedge f'$ is a 
generating $(\CatP \times \CatP')$-anodyne morphism of type $(B_0)$. Let us assume therefore
that $m > 0$. For $0 \leq k \leq m$, let $\sigma_k: \Delta^{m+1} \rightarrow \Delta^1 \times \Delta^m$
denote the simplex determined by the map of partially ordered sets $[m+1] \rightarrow [1] \times [m]$ given by the formula
$$ j \mapsto \begin{cases} (0,j) & \text{if } j \leq m-k \\
(1, j-1) & \text{otherwise.} \end{cases}$$
We have a sequence of simplicial sets
$$ Z_0 \subseteq Z_1 \subseteq \ldots \subseteq Z_{m+1} = \Delta^1 \times \Delta^m$$
where $Z_i$ is the simplicial subset of $\Delta^1 \times \Delta^m$ generated by
$\Delta^1 \times (\bd \Delta^m)$, $\{0\} \times \Delta^m$, and $\{ \sigma_j \}_{j < i}$.
Let $M$ denote the collection of edges of $\Delta^1 \times \Delta^m$ whose image in
$\Delta^m$ is degenerate, and let $\overline{Z}_i = (Z_i, M)$. To prove that
$f \wedge f'$ is $(\CatP \times \CatP')$-anodyne, it will suffice to show that each of the inclusions
$g_i: \overline{Z}_i \subseteq \overline{Z}_{i+1}$ is $\CatP$-anodyne. For $0 \leq i < m$, we observe that $g_i$ is a pushout of a generating $(\CatP \times \CatP')$-anodyne morphism of type
$(C_1)$. For $i = m$, we note that $g_i$ is a pushout of a generating $(\CatP \times \CatP')$-anodyme morphism of type $(C_0)$.

\item[$(B'_1)$] For some $\alpha \in A$, the map $f$ is an inclusion
$K_{\alpha}^{\sharp} \subseteq (K_{\alpha}^{\triangleleft})^{\sharp}$ (where $K_{\alpha}^{\triangleleft}$ maps to $S$ via $p_{\alpha}$), and $f'$ is an inclusion
$(\bd \Delta^m)^{\flat} \subseteq (\Delta^m)^{\flat}$. If $m=0$, then $f \wedge f'$ is a generating $\CatP \times \CatP'$-anodyne morphism of type $(B_1)$ and there is nothing to prove. Let us assume
therefore that $m > 0$. Let $v$ denote the cone point of $K_{\alpha}^{\triangleleft}$
We define a filtration 
$$ Z_0 \subseteq Z_1 \subseteq \ldots \subseteq Z_m \subseteq Z_{m+1} = K_{\alpha}^{\triangleleft} \times \Delta^m$$ as follows. For each $i \leq m$, let $Z_i$ denote the simplicial subset of
$K_{\alpha}^{\triangleleft} \times \Delta^m$ generated by those simplices $\sigma$ such that
either $\sigma \cap ( \{v\} \times \Delta^m) \subseteq \{v \} \times \Delta^{ \{0, \ldots, i-1 \} }$ or
the projection map $\sigma \rightarrow \Delta^m$ is not surjective. Let $\overline{Z}_i$ denote the
marked simplicial set $(Z_i, M_i)$, where $M_i$ is the collection of those edges of $Z_i$ whose
image in $\Delta^m$ is degenerate. The map $f \wedge f'$ can be identified with the inclusion
$\overline{Z}_0 \subseteq \overline{Z}_{m+1}$. It will therefore suffice to show that each of the inclusions $g_i: \overline{Z}_i \subseteq \overline{Z}_{i+1}$ is $(\CatP \times \CatP')$-anodyne. If $i < m$, then
$g_i$ is a pushout of the inclusion $B^{\flat} \subseteq ( \Delta^{i} \star (K_{\alpha} \times \Delta^{m-i} )^{\flat}$, where $B$ denotes the pushout 
$$(\bd \Delta^i \star (K_{\alpha} \times \Delta^{m-i })
\coprod_{ \bd \Delta^i \star (K_{\alpha} \times \Lambda^{m-i}_0)}
( \Delta^i \star (K_{\alpha} \times \Delta^{m-i} ) ).$$
In view of Example \ref{ina}, it will suffice to show that the inclusion of simplicial sets
$B \subseteq \Delta^i \star (K_{\alpha} \times \Delta^{m-i})$ is inner anodyne. This follows from
Lemma \toposref{precough}, since the inclusion $K_{\alpha} \times \Lambda^{m-i}_0
\subseteq K_{\alpha} \times \Delta^{m-i}$ is left anodyne (Corollary \toposref{prodprod1}).

In the case $i=m$, we observe that $g_i$ is a pushout of the inclusion
$$ ((\bd \Delta^m) \star K_{\alpha})^{\flat}
\coprod_{ (\{m\} \star K_{\alpha})^{\flat} }
( \{m\} \star K_{\alpha} )^{\sharp} 
\subseteq ( \Delta^m \star K_{\alpha})^{\flat} \coprod_{ (\{m\} \star K_{\alpha})^{\flat} }
( \{m\} \star K_{\alpha} )^{\sharp},$$
which is a $(\CatP \times \CatP')$-anodyne morphism of type $(C_2)$.

\item[$(C'_0)$] The map $f$ is an inclusion
$$ (\Lambda^n_0)^{\flat} \coprod_{ (\Delta^{ \{0,1\} })^{\flat} } ( \Delta^{ \{0,1\} })^{\sharp}
\subseteq ( \Delta^n )^{\flat} \coprod_{ ( \Delta^{ \{0,1\} })^{\flat} } ( \Delta^{ \{0,1\} })^{\sharp},$$
for some $n > 1$ such that the map $\Delta^n \rightarrow S$ carries 
$\Delta^{ \{0, 1, n\} }$ to a $2$-simplex belonging to $T$, and 
$f'$ is an inclusion $( \bd \Delta^m)^{\flat} \subseteq (\Delta^m)^{\flat}$.
If $m = 0$, then $f \wedge f'$ is a $(\CatP \times \CatP')$-anodyne morphism of
type $(C_0)$. We may therefore assume without loss of generality that $m > 0$.
We define maps
$$ \Delta^n \stackrel{s}{\rightarrow} \Delta^{1} \times \Delta^{n} \stackrel{r}{\rightarrow} \Delta^n$$
by the formulae
$$ s(i) = (1,i)$$
$$ r(i,j) = \begin{cases} 0 & \text{if } i=0, j=1 \\
j & \text{otherwise.} \end{cases}$$
These maps exhibit $f$ as a retract of the inclusion
$$g: ((\Delta^1)^{\sharp} \times (\Lambda^n_0)^{\flat}) \coprod_{ \{0\}^{\sharp} \times (\Lambda^n_0)^{\flat} } ( \{0\}^{\sharp} \times (\Delta^n)^{\flat} ) \subseteq (\Delta^1)^{\sharp} \times (\Delta^n)^{\flat}.$$
We regard $(\Delta^{1})^{\sharp} \times (\Delta^n)^{\flat}$ as an object of $\mset( \CatP )$
via the composition
$$ \Delta^1 \times \Delta^{n} \stackrel{r}{\rightarrow} \Delta^{n} \rightarrow S.$$
Since $f$ is a retract of $g$, it will suffice to show that $g \wedge f'$ is $(\CatP \times \CatP')$-anodyne, which follows
immediately from Lemma \ref{carpal}.

\item[$(C'_1)$] The map $f$ is an inclusion $(\Lambda^{n}_i)^{\flat} \subseteq (\Delta^n)^{\flat}$, for 
where $0 < i < n$, and $f'$ is an inclusion $(\bd \Delta^m)^{\flat} \subseteq (\Delta^m)^{\flat}$.
In this case, $f \wedge f'$ is a morphism of the form $B_0^{\flat} \subseteq B^{\flat}$, where
$B_0 \subseteq B$ is an inner anodyne inclusion of simplicial sets (Corollary \toposref{prodprod2}).
It follows from Example \ref{ina} that $f \wedge f'$ is $(\CatP \times \CatP')$-anodyne.

\item[$(C'_2)$]  The map $f$ has the form 
$$ ( \bd \Delta^n \star K_{\alpha})^{\flat} \coprod_{ ( \{n\} \star K_{\alpha})^{\flat} }
( \{n\} \star K_{\alpha})^{\sharp} \subseteq 
( \Delta^n \star K_{\alpha})^{\flat} \coprod_{ ( \{n\} \star K_{\alpha})^{\flat} }
( \{n\} \star K_{\alpha} )^{\sharp} )$$
for some $\alpha \in A$ and $n > 0$, where $\Delta^n \star K_{\alpha} \rightarrow S$
extends $p_{\alpha}$, and $f'$ is an inclusion of the form $(\bd \Delta^m)^{\flat}
\subseteq (\Delta^m)^{\flat}$. The treatment of this case is similar to that of $(B'_1)$.
If $m=0$, then $f \wedge f'$ is a generating $\CatP \times \CatP'$-anodyne morphism of type $(C_2)$ and there is nothing to prove. Let us assume
therefore that $m > 0$. We define a filtration 
$$ Z_0 \subseteq Z_1 \subseteq \ldots \subseteq Z_m \subseteq Z_{m+1} = ( \Delta^n \star K_{\alpha}) \times \Delta^m$$ as follows. For each $i \leq m$, let $Z_i$ denote the simplicial subset of
$( \Delta^n \star K_{\alpha}) \times \Delta^m$ generated by those simplices $\sigma$ such that
either $\sigma \cap ( \Delta^n \times \Delta^m) \subseteq \Delta^n \times \Delta^{ \{0, \ldots, i-1 \} }$ or
the projection map $\sigma \rightarrow \Delta^m$ is not surjective. Let $\overline{Z}_i$ denote the
marked simplicial set $(Z_i, M_i)$, where $M_i$ is the collection of those edges of $Z_i$ which are marked in $$( \Delta^n \star K_{\alpha})^{\flat} \coprod_{ ( \{n\} \star K_{\alpha})^{\flat} }
( \{n\} \star K_{\alpha} )^{\sharp} ) \times (\Delta^m)^{\flat}.$$

The map $f \wedge f'$ can be identified with the inclusion
$\overline{Z}_0 \subseteq \overline{Z}_{m+1}$. It will therefore suffice to show that each of the inclusions $g_i: \overline{Z}_i \subseteq \overline{Z}_{i+1}$ is $(\CatP \times \CatP')$-anodyne. If $i < m$, then
$g_i$ is a pushout of the inclusion $B^{\flat} \subseteq ( (\Delta^n \times \Delta^{i}) \star (K_{\alpha} \times \Delta^{m-i} )^{\flat}$, where $B$ denotes the pushout 
$$(\bd (\Delta^n \times \Delta^i) \star (K_{\alpha} \times \Delta^{m-i})
\coprod_{ \bd(\Delta^n \times \Delta^i) \star (K_{\alpha} \times \Lambda^{m-i}_0)}
( (\Delta^n \times \Delta^i) \star (K_{\alpha} \times \Delta^{m-i} ) ).$$
In view of Example \ref{ina}, it will suffice to show that the inclusion of simplicial sets
$B \subseteq (\Delta^n \times \Delta^i) \star (K_{\alpha} \times \Delta^{m-i})$ is inner anodyne. This follows from Lemma \toposref{precough}, since the inclusion $K_{\alpha} \times \Lambda^{m-i}_0
\subseteq K_{\alpha} \times \Delta^{m-i}$ is left anodyne (Corollary \toposref{prodprod1}).

In the case $i=m$, we observe that $g_i$ is a pushout of the inclusion
$$ (\bd( \Delta^n \times \Delta^m) \star K_{\alpha})^{\flat}
\coprod_{ (\{(n,m)\} \star K_{\alpha})^{\flat} }
( \{(n,m\} \star K_{\alpha} )^{\sharp} 
\subseteq ( (\Delta^n \times \Delta^m) \star K_{\alpha})^{\flat} \coprod_{ (\{(n,m)\} \star K_{\alpha})^{\flat} }
( \{(n,m)\} \star K_{\alpha} )^{\sharp},$$
which is $(\CatP \times \CatP')$-anodyne (Example \ref{urtime}).
\end{itemize}
\end{proof}

\subsection{Flat Inner Fibrations}\label{flatinn}

\begin{lemma}\label{gooby}
Let $\calC$ be a simplicial category equipped with a functor
$\calC \rightarrow [1]$, where $[1]$ denotes the (discrete) category
$\{ 0 < 1 \}$. Suppose that the inclusion $\calC_0 \hookrightarrow \calC$
is a cofibration of simplicial categories. Then, for every object $D \in \calC_{1}$, the
functor $C \mapsto \bHom_{ \calC}( C, D)$ is a projectively cofibrant object of
$F \in (\sSet)^{\calC_0^{op}}$. 
\end{lemma}

\begin{proof}
We must show that every trivial projective fibration $\alpha: G \rightarrow G'$ in
$(\sSet)^{\calC_0^{op}}$ has the right lifting property with respect to $F$. 
Define a new simplicial category $\calC[G]$ as follows:
\begin{itemize}
\item[$(i)$] The objects of $\calC[G]$ are the objects of $\calC$.
\item[$(ii)$] For $C, C' \in \calC$, we have
$$ \bHom_{\calC[G]}( C, C' ) = \begin{cases} \emptyset & \text{if } C \in \calC_1, C' \in \calC_0 \\
\bHom_{\calC}(C,C') \times G(C)^{ \bHom_{\calC}( C',D) } & \text{if } C \in \calC_0, C' \in \calC_1 \\
\bHom_{\calC}(C,C') & \text{otherwise} \end{cases}$$
\end{itemize}
Let $\calC[G']$ be defined similarly. Unwinding the definitions, we see that
$\alpha$ has the right lifting property with respect to $F$ if and only if the induced map
$\overline{\alpha}: \calC[G] \rightarrow \calC[G']$ has the right lifting property with
respect to the inclusion $i: \calC_0 \subseteq \calC$. Since $i$ is a cofibration, this follows from the observation that $\overline{\alpha}$ is a trivial fibration of simplicial categories.
\end{proof}

\begin{lemma}\label{cane}
Suppose we are given an inner fibration of simplicial sets $p: X \rightarrow \Lambda^2_1$.
Let $C$ be an initial object of $\calM = p^{-1} \Delta^{ \{0,1\} }$, let $E$ be a final object of
$\calN = p^{-1} \Delta^{ \{1,2\} }$, let $\calD = \calM \cap \calN = p^{-1} \{1\}$, 
and let $f: X \rightarrow \calM$ be a categorical equivalence
from $X$ to an $\infty$-category $\calM$. Then there is a canonical isomorphism
$\bHom_{\calM}( f(C), f(E)) \simeq [\calD]$ in the homotopy category $\calH$ of spaces.
\end{lemma}

\begin{proof}
We can identify $\bHom_{\calM}( f(C), f(E))$ with the simplicial set
$\bHom_{ \sCoNerve[X]}(C,E)$. Let $F: \sCoNerve[ \calD] \rightarrow \sSet$
be the functor given by the formula $F(D) = \bHom_{ \sCoNerve[\calM]}(C,D)$, and let
$G: \sCoNerve[\calD]^{op} \rightarrow \sSet$ be given by the formula
$G(D) = \bHom_{ \sCoNerve[\calN]}(D, E)$. Since $\sCoNerve[X]$ is isomorphic to the pushout
$\sCoNerve[ \calM] \coprod_{ \sCoNerve[\calD] } \sCoNerve[\calN]$, the simplicial set
$\bHom_{ \sCoNerve[X]}(C,E)$ can be computed as the coend
$$ \int_{ D \in \sCoNerve[\calD] } F(D) \otimes G(D).$$
Lemma \ref{gooby} guarantees that the functor $G$ is projectively cofibrant
projectively cofibrant, so the construction
$$ H \mapsto \int_{D \in \sCoNerve[\calD]} H(D) \times G(D)$$
carries weak equivalences between injectively cofibrant objects 
of $(\sSet)^{ \sCoNerve[\calD]}$ to weak homotopy equivalences of simplicial sets (Remark \toposref{cabler}). Since $C$ is an initial object of $\calM$, the canonical map
$F \rightarrow F_0$ is a weak equivalence, where $F_0: \sCoNerve[\calD] \rightarrow \sSet$ is
the constant functor taking the value $\Delta^0$. It follows that $\alpha$ induces a homotopy equivalence
$$ \bHom_{ \sCoNerve[X]}(C,E) \rightarrow \varinjlim G.$$
Since $E \in \calN$ is final, we also have a weak equivalence
$G \rightarrow G_0$, where $G_0: \sCoNerve[ \calD]^{op} \rightarrow \sSet$ denotes
the constant functor taking the value $0$. It follows that $G$ is a cofibrant
replacement for $G_0$ with respect to the projective model structure on
$(\sSet)^{ \sCoNerve[\calD]^{op} }$, so we can identify $\varinjlim G$ with
a homotopy colimit of the diagram $G_0$. Applying Theorem \toposref{colimcomparee}, we
can identify this homotopy colimit with a colimit of the constant diagram
$\calD^{op} \rightarrow \SSet$ taking the value $\Delta^0$. This colimit
is represented by the simplicial set $\calD$ in the homotopy category $\calH$
(Corollary \toposref{needka}).
\end{proof}

\begin{proposition}\label{seda}
Let $\calM$ be an $\infty$-category equipped with an inner fibration
$p: \calM \rightarrow \Delta^2$, and let $X = \calM \times_{ \Delta^2} \Lambda^2_1$.
Let $\calC = p^{-1} \{0\}$, let $\calD = p^{-1} \{1\}$, and let $\calE = p^{-1} \{2\}$. 
The following conditions are equivalent:
\begin{itemize}
\item[$(1)$] The inclusion $X \hookrightarrow \calM$ is a categorical equivalence.
\item[$(2)$] For every morphism $f: C \rightarrow E$ in $\calM$ from an
object $C \in \calC$ to an object $E \in \calE$, the $\infty$-category
$\calD_{C/ \, /E} = \calD \times_{ \calM} \calM_{ C/ \, /E}$ is weakly contractible.
\end{itemize}
\end{proposition}

\begin{proof}
Using the small object argument, we can factor the inclusion $X \hookrightarrow \calM$
as a composition
$$ X \stackrel{i}{\hookrightarrow} \calM' \stackrel{q}{\rightarrow} \calM$$
where $i$ is inner anodyne, the map $q$ is an inner fibration, and 
$i$ induces an isomorphism $X \rightarrow \calM' \times_{ \Delta^2} \Lambda^2_1$.
We will abuse notation by identifying $X$ (and therefore also the $\infty$-categories
$\calC, \calD, \calE \subseteq X$) with a simplicial subset of $\calM'$ via the map $i$.

Condition $(1)$ is equivalent to the assertion that $q$ is 
an equivalence of $\infty$-categories. Since $q$ is bijective on vertices, this is equivalent to the assertion that $q$ induces a homotopy equivalence $\theta: \bHom_{\calM'}(C,E) \rightarrow
\bHom_{\calM}(C,E)$ for every pair of objects $C,E \in \calM'$. This condition is obvious
unless $C \in \calC$ and $E \in \calE$. In the latter case, it is equivalent to the requirement that 
for every morphism $f: C \rightarrow E$ in $\calM$, the homotopy fiber of the map
$\theta$ (taken over the point $f \in \bHom_{\calM}(C,E)$) is contractible. It will therefore
suffice to prove the equivalence of the following conditions:
\begin{itemize}
\item[$(1')$] The homotopy fiber of $\theta$ over $\{f\}$ is contractible.
\item[$(2')$] The $\infty$-category $\calD_{C/ \, /E}$ is weakly contractible.
\end{itemize}

Suppose we are given a right fibration $\overline{\calM} \rightarrow \calM$,
and that we can lift $f$ to a morphism $\overline{f}: \overline{C} \rightarrow \overline{E}$ in $\overline{\calM}$. Let
$\overline{\calM}' = \calM' \times_{ \calM} \overline{\calM}$; it follows from Proposition
\toposref{basechangefunky} that the inclusion
$\overline{\calM} \times_{ \Delta^2} \Lambda^2_1 \hookrightarrow \overline{\calM}'$ remains a
categorical equivalence. Using Proposition
\toposref{compspaces}, we deduce the existence of a homotopy pullback diagram
$$ \xymatrix{ \bHom_{ \overline{\calM}'}( \overline{C}, \overline{E} ) \ar[r]^{ \overline{\theta}} \ar[d] & \bHom_{\overline{\calM}}( \overline{C}, \overline{E} ) \ar[d] \\
\bHom_{\calM'}(C,E) \ar[r]^{\theta} & \bHom_{\calM}(C,E). }$$
It follows that $(1')$ is satisfied by the morphism $f$ of $\calM$ if and only if
it is satisfied by the morphism $\overline{f}$ over $\overline{\calM}$. Proposition
\toposref{sharpen2} guarantees that the map $\overline{\calM}_{ \overline{C}/ \, /\overline{E}}
\rightarrow \calM_{ C/ \, /E}$ is a trivial Kan fibration, so that $(2')$ is satisfied by
$f$ if and only if it is satisfied by $\overline{f}$. It follows that we are free to replace
$\calM$ by $\overline{\calM} = \calM_{/E}$, and thereby reduce to the case where
$E$ is a final object of $\calM$. A similar argument shows that we can assume that
$C$ is an initial object of $\calM$. In this special case, the space
$\bHom_{ \calM}( C, E)$ is contractible, so we can reformulate $(1')$ as follows:
\begin{itemize}
\item[$(1'')$] The space $\bHom_{ \calM'}( C, E)$ is contractible.
\end{itemize}
If $C$ is an initial object of $\calM$, then $\calM_{C/} \rightarrow \calM$ is a trivial Kan fibration.
Moreover, if $E$ is a final object of $\calM$ then it is a final object of $\calM_{C/}$ (Proposition \toposref{needed17}), so the projection $\calM_{C/\,/E} \rightarrow \calM_{C/}$ is also a trivial Kan fibration. We therefore obtain the following reformulation of condition $(2')$: 
\begin{itemize}
\item[$(2'')$] The $\infty$-category $\calD$ is weakly contractible.
\end{itemize}
The equivalence of $(1'')$ and $(2'')$ now follows from Lemma \ref{cane}.
\end{proof}

\begin{definition}\label{kal}
We will say that an inner fibration $\calM \rightarrow \Delta^2$ of simplicial sets
is {\it flat} if it satisfies the equivalent conditions of Proposition \ref{seda}.
\end{definition}

\begin{example}\label{gabbe}
Let $p: \calM \rightarrow \Delta^2$ be an inner fibration of simplicial sets.
Let $\calC = p^{-1} \{0\}$, $\calD = p^{-1} \{1\}$, and $\calE = p^{-1} \{2\}$.
Suppose that for every object $C \in \calC$, there exists a $p$-coCartesian
morphism $f: C \rightarrow D$, where $D \in \calD$. Then $p$ is flat.

To prove this, consider an arbitrary morphism $g: C \rightarrow E$ in
$\calM$, where $C \in \calC$ and $E \in \calE$. Choose a $p$-coCartesian morphism
$f: C \rightarrow D$ in $\calM$ for $D \in \calD$. Using the assumption that
$f$ is $p$-coCartesian, we can find a commutative diagram
$$ \xymatrix{ & D \ar[dr]^{h} & \\
C \ar[ur]^{f} \ar[rr]^{g} & & E }$$
which we can identify with an object $\overline{D} \in \calD_{C/ \, /E}$ lifting
$D$. To show that $\calD_{C/ \, /E}$ is weakly contractible, it suffices to show that
$\overline{D}$ is an initial object of $\calD_{C/ \, /E}$. In view of Proposition
\toposref{needed17}, it will suffice to show that $\overline{D}$ is an initial object of
$\calD_{C/}$, which is equivalent to the assertion that $f$ is locally $p$-coCartesian.
\end{example}

\begin{example}\label{gabbe2}
Let $p: \calM \rightarrow \Delta^2$ be an inner fibration of simplicial sets.
Let $\calC = p^{-1} \{0\}$, $\calD = p^{-1} \{1\}$, and $\calE = p^{-1} \{2\}$.
Suppose that for every object $E \in \calE$, there exists a $p$-Cartesian
morphism $f: D \rightarrow E$, where $D \in \calD$. Then $p$ is flat. The proof
is identical to that of Example \ref{gabbe}.
\end{example}

\begin{definition}\label{kol}
Let $p: X \rightarrow S$ be an inner fibration of simplicial sets, and let
$\sigma$ be a $2$-simplex of $S$. We will say that $p$ is {\it flat over $\sigma$} if
the induced inner fibration $X \times_{S} \Delta^2 \rightarrow \Delta^2$ is flat, in the sense of Definition \ref{kal}. We will say that $p$ is {\it flat} if it is flat over every $2$-simplex of $S$.
\end{definition}

\begin{example}
Let $p: X \rightarrow S$ be an inner fibration of simplicial sets. Then $p$ is flat over
any degenerate $2$-simplex of $S$, since the induced functor $X \times_{ S} \Delta^2 \rightarrow \Delta^2$ satisfies the hypotheses of either Example \ref{gabbe} or Example \ref{gabbe2}.
It follows that an inner fibration $p: X \rightarrow \Delta^2$ is flat in the sense of
Definition \ref{kol} if and only if it is flat in the sense of Definition \ref{kal}.
\end{example}

\begin{example}\label{cartflat}
Let $p: X \rightarrow S$ be a coCartesian fibration of simplicial sets. Then
$p$ is a flat categorical fibration: this is an immediate consequence of Example \ref{gabbe}. Similarly, if
$p$ is a Cartesian fibration, then $p$ is flat.
\end{example}

\begin{remark}
Suppose given a pullback diagram of simplicial sets
$$ \xymatrix{ X' \ar[r]^{q} \ar[d]^{p'} & X \ar[d]^{p} \\
S' \ar[r] & S. }$$
If $p$ is a flat inner fibration, then so is $p'$. 
\end{remark}

Proposition \ref{seda} admits the following generalization:

\begin{proposition}\label{capper}
Let $p: X \rightarrow S$ be an inner fibration of simplicial sets. The following conditions
are equivalent:
\begin{itemize}
\item[$(1)$] For every inner anodyne map $A \hookrightarrow B$ of simplicial sets
and every map $B \rightarrow S$, the induced map $X \times_{S} A \rightarrow X \times_{S} B$
is a categorical equivalence.
\item[$(2)$] The inner fibration $p$ is flat.
\end{itemize}
\end{proposition}

Proposition \ref{capper} will require some preliminaries.

\begin{proposition}\label{balder}
Let $p: \calM \rightarrow \Delta^1$ be a correspondence from an $\infty$-category
$\calC = \calM \times_{ \Delta^1} \{0\}$ to $\calD = \calM \times_{ \Delta^1 } \{1\}$.
Let $\calX = \bHom_{ \Delta^1}( \Delta^1, \calM)$ be the $\infty$-category of sections
of the map $p$. Then the canonical map
$$ \calC \coprod_{ \calX \times \{0\} } (\calX \times \Delta^1) \coprod_{ \calX \times \{1\} } \calD \rightarrow \calM$$
is a categorical equivalence.
\end{proposition}

\begin{proof}
For every $\infty$-category $\calA$, we let $\calA^{\natural}$ denote the marked simplicial set
$(\calA, M_{\calA})$, where $\calA$ is the collection of all equivalences in $\calA$. Since
the category of marked simplicial sets is Quillen equivalent to the category of simplicial
sets (with the Joyal model structure), it will suffice to prove the following:
\begin{itemize}
\item[$(A)$] The diagram
$$ \xymatrix{ \calX^{\natural} \times (\bd \Delta^{ \{1,2\} })^{\flat} \ar[r] \ar[d] & \calC^{\natural} \times \calD^{\natural} \ar[d] \\
\calX^{\natural} \times ( \Delta^{ \{1,2\} })^{\flat} \ar[r] & \calM^{\natural} }$$
is a homotopy pushout square of marked simplicial sets.
\end{itemize}
To prove this, we let $\calY$ denote the full subcategory of $\Fun( \Delta^1, \calM) \times \Delta^3$
spanned by those pairs $(f: A \rightarrow A', i)$ satisfying one of the following conditions:
\begin{itemize}
\item We have $i=0$ and $f$ is an equivalence in $\calC$.
\item We have $i=1$ or $i=2$ and $f$ belongs to $\calX$.
\item We have $i=3$ and $f$ is an equivalence in $\calD$.
\end{itemize}
For each simplicial subset $K \subseteq \Delta^3$, we let $\calY_{K} = \calY \times_{ \Delta^3} K$, and let $\overline{\calY}_{K}$ denote the marked simplicial set $(\calY_K, M_K)$, where $M_K$ is
the collection of all edges $\alpha: (f,i) \rightarrow (f',i')$ in $\calY_K$ satisfying one of the following three conditions:
\begin{itemize}
\item The map $\alpha$ is an equivalence in $\calY$.
\item We have $i = 0$, $i' = 1$, and $\alpha$ corresponds to a commutative diagram
$$ \xymatrix{ C \ar[d]^{f} \ar[r]^{g} & C'' \ar[d]^{f'} \\
C' \ar[r] & D }$$
for which $g$ is an equivalence.
\item We have $i=2$, $i' = 3$, and $\alpha$ corresponds to a commutative diagram
$$ \xymatrix{ C \ar[d]^{f} \ar[r] & D \ar[d]^{f'} \\
D'' \ar[r]^{g} & D' }$$
for which $g$ is an equivalence.
\end{itemize}

We observe that there is a retraction $r$ of $\calY$ onto the full subcategory $\calY_{ \Delta^{ \{0,2,3\} } }$, which carries an object $f: C \rightarrow D$ of $\calY_{ \{1\} }$ to the object $\id_{C} \in \calY_{ \{0\} }$. 
This retraction is equipped with a natural transformation $r \rightarrow \id_{ \calY}$, which
determines a map of marked simplicial sets $\overline{\calY}_{\Delta^3} \times ( \Delta^1)^{\sharp}
\rightarrow \overline{\calY}_{\Delta^3}$. Using this deformation retraction, we deduce the following:
\begin{itemize}
\item[$(\ast)$] Let $S$ be a subset of $\{0,2,3\}$ containing $\{0\}$. Then the inclusion $\overline{\calY}_{ \Delta^{ S} } \subseteq \overline{\calY}_{ \Delta^{ S \cup \{1\} }}$ is a weak equivalence of marked simplicial sets.
\end{itemize}

A similar argument proves:
\begin{itemize}
\item[$(\ast')$] Let $S$ be a subset of $\{0,1,3\}$ containing $\{3\}$. Then the inclusion
$\overline{\calY}_{ \Delta^{S}} \subseteq \overline{\calY}_{ \Delta^{ S \cup \{2\} }}$
is a weak equivalence of marked simplicial sets.
\end{itemize}

Let $\phi: \Delta^3 \rightarrow \Delta^1$ be the map characterized by $\phi^{-1} \{0\} = \Delta^{ \{0,1\} } \subseteq \Delta^3$, and consider the map
$$ \theta: \calY \subseteq \Fun( \Delta^1, \calM) \times \Delta^3 \stackrel{\id \times \phi}{\rightarrow}
\Fun( \Delta^1, \calM) \times \Delta^1 \rightarrow \calM.$$
Consider the diagram
$$ \xymatrix{ \calC^{\natural} \ar[r] \ar[drr]^{\id} & \overline{\calY}_{\{0\}} \ar[r] & \overline{\calY}_{ \Delta^{ \{0,1\} }} \ar[d]^{\theta_0} \\
& & \calC^{\natural}. }$$
Using $(\ast)$ and the observation that the diagonal inclusion $\calC \rightarrow \calY_{ \{0\} }$ is
an equivalence of $\infty$-categories, we deduce that $\theta_0$ is a weak equivalence of marked simplicial sets. A similar argument gives an equivalence of marked simplicial sets
$\overline{\calY}_{ \Delta^{ \{2,3\} }} \rightarrow \calD^{\natural}$. Using this observation,
we can reformulate $(A)$ as follows:
\begin{itemize}
\item[$(B)$] The diagram
$$ \xymatrix{ \overline{\calY}_{ \{1\} }\coprod \overline{\calY}_{ \{2\}} \ar[r] \ar[d] & \overline{\calY}_{ \Delta^{\{0,1\}}}
 \coprod \overline{\calY}_{\Delta^{\{2,3\}}} \ar[d] \\
\overline{\calY}_{\Delta^{\{1,2\}}} \ar[r] & \calM^{\natural} }$$
is a homotopy pushout square of marked simplicial sets.
\end{itemize}
We have a commutative diagram of marked simplicial sets
$$ \xymatrix{ \overline{\calY}_{ \Delta^{ \{0,3\}}} \ar[d]^{\beta_1} \ar[dr]^{\beta_0} & \\
\overline{\calY}_{\Delta^3} \ar[r] & \calM^{\natural} \\
\overline{\calY}_{K} \ar[u]^{\beta_2} \ar[ur]^{\beta_3} & }$$
where $K = \Delta^{ \{0,1\} } \coprod_{ \{1\} } \Delta^{ \{1,2\} } \coprod_{ \{2\} } \Delta^{ \{2,3\} }
\subseteq \Delta^3$. We wish to prove that $\beta_3$ is a weak equivalence of marked
simplicial sets. Since $\beta_0$ is an isomorphism of marked simplicial sets, it suffices to show
that $\beta_{1}$ and $\beta_{2}$ are weak equivalences of marked simplicial sets. 

To prove that $\beta_1$ is a weak equivalence, we factor $\beta_1$ as a composition
$$ \overline{\calY}_{ \Delta^{ \{0,3\} }} \stackrel{\beta'_{1}}{\rightarrow}
\overline{\calY}_{ \Delta^{ \{0,1,3\} }} \stackrel{ \beta''_{1}}{\rightarrow} \overline{\calY}_{ \Delta^3}.$$
Assertion $(\ast)$ implies that $\beta'_{1}$ is a weak equivalence, and assertion
$(\ast')$ implies that $\beta''_{1}$ is a weak equivalence.

To prove that $\beta_2$ is a weak equivalence, we factor $\beta_{2}$ as a composition
$$ \overline{\calY}_{K} \stackrel{\beta'_{2}}{\rightarrow}
\overline{\calY}_{ \Delta^{ \{0,1,2\} } \coprod_{ \{2\} } \Delta^{ \{2,3\} }}
\stackrel{ \beta''_{2}}{\rightarrow} \overline{\calY}_{ \Delta^{ \{0,1,2\}}
\coprod_{ \Delta^{ \{1,2\} } } \Delta^{ \{1,2,3\} }} \stackrel{\beta'''_{2}}{\rightarrow} \overline{\calY}_{\Delta^3}.$$
The map $\beta'_{2}$ is a pushout of the inclusion
$$\calY^{\natural}_{ \Delta^{ \{0,1\} }} \coprod_{ \calY_{ \{1\}}^{\natural}}
\calY^{\natural}_{ \Delta^{ \{1,2\} }} \rightarrow \calY^{\natural}_{ \Delta^{ \{0,1,2\}}}.$$
Consequently, to prove $\beta'_{2}$ it suffices to show that the map
$\calY_{ \Delta^{ \{0,1,2\} }} \rightarrow \Delta^{ \{0,1,2\} }$ is a flat inner fibration, which
follows from Example \ref{gabbe2}. The same argument shows that $\beta''_{2}$ is a weak equivalence. 
The map $\beta'''_{2}$ is a pushout of the inclusion
$$\calY^{\natural}_{\Delta^{ \{0,1,2\}}
\coprod_{ \Delta^{ \{1,2\} } } \Delta^{ \{1,2,3\} }} \rightarrow \calY^{\natural}.$$
To complete the proof, it will suffice to show that this map is a weak equivalence of marked simplicial sets, which is equivalent to the requirement that the composite map
$$ \calY \rightarrow \Delta^3 \stackrel{\phi'}{\rightarrow} \Delta^2$$
is a flat inner fibration (here $\phi'$ is the map characterized by the property that ${\phi'}^{-1} \{1\} = \Delta^{ \{1,2\} } \subseteq \Delta^3$). 

In view of Proposition \ref{seda}, we must show that for every object $\overline{C}: C \rightarrow C'$ of
$\Fun( \Delta^1, \calC) \simeq \calY_{ \{0\} }$ and every object $\overline{D}: D \rightarrow D'$ of $\Fun( \Delta^1, \calD) \simeq \calY_{ \{3\} }$, the
simplicial set $\calY_{ \overline{C}/ \, / \overline{D}} \times_{ \Delta^3 } \Delta^{ \{1,2\} }$ is
weakly contractible. This simplicial set can be identified with the product
$\Delta^1 \times \calE$, where $\calE = \Fun( \Delta^1, \calM)_{ \overline{C}/ \, / \overline{D} } \times_{ \Fun( \Delta^1, \calM)} \calX$. To complete the proof, we will show that the $\infty$-category $\calE$ is weakly contractible.

We observe that an object of $\calE$ can be identified with a commutative diagram
$$ \xymatrix{ C \ar[d]^{ \overline{C} } \ar[r]^{\gamma} & C'' \ar[d] \ar[r] & D \ar[d]^{\overline{D}} \\
C' \ar[r] & D'' \ar[r]^{\gamma'} & D' }$$
in $\calM$, where $C'' \in \calC$ and $D'' \in \calD$. Let $\calE_0$ denote the full subcategory
of $\calE$ spanned by those objects for which $\gamma$ is an equivalence. The inclusion
$\calE_0 \subseteq \calE$ admits a right adjoint, and is therefore a weak homotopy equivalence.
It will therefore suffice to show that $\calE_0$ is weakly contractible. Let $\calE_1$ denote the
full subcategory of $\calE_0$ spanned by those diagrams for which $\gamma'$ is an equivalence.
The inclusion $\calE_1 \subseteq \calE_0$ admits a left adjoint, and is therefore a weak homotopy equivalence. It therefore suffices to show that $\calE_1$ is weakly contractible. We complete the proof by observing that $\calE_1$ is a contractible Kan complex.
\end{proof}

\begin{lemma}\label{sadly}
Let $q: \calC \rightarrow \calD$ be a coCartesian fibration of $\infty$-categories.
Suppose that $\calD$ is weakly contractible and that each fiber $\calC_{D}$ of $q$ is
weakly contractible. Then $\calC$ is weakly contractible.
\end{lemma}

\begin{proof}
The coCartesian fibration $q$ is classified by a functor $\chi: \calD \rightarrow
\Cat_{\infty}$. Let $F: \Cat_{\infty} \rightarrow \SSet$ denote a left adjoint
to the inclusion $\SSet \subseteq \Cat_{\infty}$. According to Corollary \toposref{tolbot},
the colimit of $\chi$ is represented (in the homotopy category $\h{\Cat_{\infty}}$)
by the marked simplicial set $(\calC, M)$, where $M$ is the collection of $q$-coCartesian morphisms in $\calD$. It follows that $\varinjlim F \circ \chi \simeq F( \varinjlim \chi)$ is represented
(in the homotopy category $\h{ \SSet}$) by the simplicial set $\calC$. Since each
fiber $\calC_{D}$ is weakly contractible, the composition $F \circ \chi$ is equivalent to the
constant functor $\calD \rightarrow \SSet$ taking the value $\Delta^0$. Using Corollary
\toposref{needka}, we deduce that the colimit $\varinjlim F \circ \chi$ can also be represented
by the simplicial set $\calD$. We therefore have an isomorphism between $\calC$ and $\calD$ in the homotopy category of spaces. Since $\calD$ is weakly contractible, we conclude that $\calC$ is weakly contractible.
\end{proof}

\begin{lemma}\label{cast1}
Let $p: \calM \rightarrow \Delta^3$ be a flat inner fibration. Let
$f: C \rightarrow D$ be a morphism in $\calM$, where $C \in \calM_0$ and
$D \in \calM_3$. Then the $\infty$-category
$\calN = \calM_{ C/ \, /D} \times_{ \Delta^3} \Delta^{ \{1,2\} }$ is weakly contractible.
\end{lemma}

\begin{proof}
Let $\calX$ denote the $\infty$-category $\Fun_{ \Delta^{ \{1,2\} }}( \Delta^{ \{1,2\} }, \calN)$.
According to Proposition \ref{balder}, we have a categorical equivalence
$$ \calN_{1} \coprod_{ \calX \times \{1\} } ( \calX \times \Delta^1) \coprod_{ \calX \times \{2\} } \calN_{2} \rightarrow \calN.$$
Since $\calN_{1}$ and $\calN_{2}$ are weakly contractible (by virtue of the assumption that
$p$ is flat over $\Delta^{ \{0,1,3\} }$ and $\Delta^{ \{0,2,3\} }$), it will suffice to show that
$\calX$ is weakly contractible. Let $q: \calX \rightarrow \calN_2$ be the map given by
evaluation at $\{2\}$. Using Corollary \toposref{tweezegork}, we deduce that
$q$ is a coCartesian fibration. Since $\calN_2$ is weakly contractible, it will suffice
to show that the fiber $q^{-1} E$ is weakly contractible, for each $E \in \calN_2$
(Lemma \ref{sadly}). This fiber can be identified with the fiber product
$\{ 1\} \times_{ \Delta^3} ( \calM_{ C/\,/D})^{/E}$, which is categorically
equivalent to $\{1\} \times_{ \Delta^3} ( \calM_{ C/ \, /D})_{/E}$ 
(Proposition \toposref{certs}). Let $E_0$ denote the image of $E$ in $\calM$. We
have a trivial Kan fibration $( \calM_{C/ \, /D})_{/E} \rightarrow \calM_{C/ \, /E_0}$.
It therefore suffices to show that $\{1\} \times_{ \Delta^3} \calM_{ C/ \, /E_0}$ is weakly contractible, which follows from the assumption that $p$ is flat over the $2$-simplex $\Delta^{ \{0,1,2\} }$. 
\end{proof}

\begin{lemma}\label{custer}
Let $p: \calM \rightarrow \Delta^n$ be a flat inner fibration, and let
$q: \Delta^n \rightarrow \Delta^m$ be a map of simplices which is surjective on vertices. Then the
composite map $q \circ p$ is a flat inner fibration.
\end{lemma}

\begin{proof}
If $n - m > 1$, then we can factor $q$ as a composition
$$ \Delta^{n} \stackrel{q'}{\rightarrow} \Delta^{n-1} \stackrel{q''}{\rightarrow} \Delta^m$$
where $q'$ and $q''$ are surjective on vertices. Using descending induction on $n-m$,
we can assume that $n-m \leq 1$. If $n=m$ there is nothing to prove, so we may suppose
that $n = m+1$. To prove that $q \circ p$ is flat, it suffices to show that it is flat over
every nondegenerate $2$-simplex of $\Delta^m$. Replacing
$\calM$ by the pullback $\calM \times_{ \Delta^m} \Delta^2$, we can reduce to the
case where $m=2$ and $n=3$.

Fix objects $C \in (q \circ p)^{-1} \{0\}$ and $D \in (q \circ p)^{-1} \{2\}$ and a morphism
$f: C \rightarrow D$ in $\calM$; we wish to prove
that the $\infty$-category $\calM_{C/ \, /D} \times_{ \Delta^2} \{1\}$ is weakly contractible.
Let $i \in [2]$ be the unique integer such that $q^{-1} \{i\}$ is a $1$-simplex of $\Delta^3$.
If $i = 0$, then the weak contractibility follows from the assumption that
$p$ is flat over $\Delta^{ \{0, 2, 3 \} } \subseteq \Delta^3$. If $i=2$,
then the weak contractibility follows from our assumption that $p$ is flat over
$\Delta^{ \{0,1, 3\} } \subseteq \Delta^3$. If $i = 1$, then the desired result follows from
Lemma \ref{cast1}.
\end{proof}

\begin{lemma}\label{pustule}
Let $p: \calM \rightarrow \Delta^n$ be a flat inner fibration. Let
$f: C \rightarrow D$ be a morphism in $\calM$, where $C \in \calM_0$ and
$D \in \calM_n$. Then the $\infty$-category
$\calN = \calM_{ C/ \, /D} \times_{ \Delta^n} \Delta^{ \{1,\ldots, n-1\} }$ is weakly contractible.
\end{lemma}

\begin{proof}
Apply Lemma \ref{custer} to the map $q: \Delta^n \rightarrow \Delta^2$ characterized
by the requirement that $q^{-1} \{1\} = \{ 1, \ldots, n-1 \}$. 
\end{proof}

\begin{lemma}\label{spaz}
Let $p: \calM \rightarrow \Delta^n \times \Delta^m$ be a flat inner fibration.
Then the induced map $p': \calM \rightarrow \Delta^m$ is a flat inner fibration.
\end{lemma}

\begin{proof}
It suffices to show that $p'$ is flat over every nondegenerate $2$-simplex of
$\Delta^m$. Replacing $\calM$ by $\calM \times_{ \Delta^m} \Delta^2$, we can
reduce to the case $m = 2$. Fix a morphism $f: C \rightarrow D$ in
$\calM$, where $C \in {p'}^{-1} \{0\}$ and $D \in {p'}^{-1} \{2\}$; we wish to show
that the $\infty$-category $\calM_{C/ \, /D} \times_{ \Delta^2} \{1\}$ is weakly contractible.
Let $i$ and $j$ denote the images of $C$ and $D$ in $\Delta^n$, and let
$\phi: \Delta^{ 2+j-i} \rightarrow \Delta^n \times \Delta^2$ be the map
given on vertices by the formula
$$ \phi(k) = \begin{cases} (i,0) & \text{if } k = 0 \\
(i+k-1, 1) & \text{if } 0 < k < 2 + j - i \\
(j, 2) & \text{if } k = 2+j - i. \end{cases}$$
The desired result now follows after applying Lemma \ref{pustule} to the
flat inner fibration $\calM \times_{ \Delta^n \times \Delta^2} \Delta^{ 2 + j - i } \rightarrow
\Delta^{ 2 + j - i}$. 
\end{proof}

\begin{proof}[Proof of Proposition \ref{capper}]
Fix an inner fibration of simplicial sets $p: X \rightarrow S$. By definition, the map
$p$ is flat if it induces a categorical equivalence $X \times_{S} \Lambda^2_1 \rightarrow
X \times_{S} \Delta^2$, for every $2$-simplex of $S$. This proves the implication
$(2) \Rightarrow (1)$. For the converse, let us say that a monomorphism of simplicial sets
$A \rightarrow B$ is {\it good} if it satisfies the following condition:
\begin{itemize}
\item[$(\ast)$] For every map of simplicial sets $B \rightarrow S$, the induced map
$X \times_{S} A \rightarrow X \times_{S} B$ is a categorical equivalence.
\end{itemize}
Since the collection of trivial cofibrations with respect to the Joyal model structure is
weakly saturated (in the sense of Definition \toposref{saturated}), we deduce
that the collection of good morphisms in $\sSet$ is also weakly saturated. We wish to
prove that every inner anodyne morphism is good. In view of Proposition \toposref{usejoyal2}, it will suffice to show that for every monomorphism of simplicial sets $A \rightarrow B$ having only finitely many nondegenerate simplices, the induced map
$$ (A \times \Delta^2) \coprod_{ A \times \Lambda^2_1} (B \times \Lambda^2_{1}) \rightarrow B \times \Delta^2$$ is good. In other words, we must show that for every map $B \times \Delta^2 \rightarrow S$, the induced diagram
$$ \xymatrix{ X \times_{S} (A \times \Lambda^2_{1}) \ar[r] \ar[d] & X \times_{S} (B \times \Lambda^2_1) \ar[d] \\
X \times_{S} (A \times \Delta^2) \ar[r] & X \times_{S} (B \times \Delta^2)}$$
is a homotopy pushout square (with respect to the Joyal model structure). To prove this, it
suffices to show that the vertical maps are categorical equivalences. In other words, we are reduced to proving that the following assertion holds, for every simplicial set $K$ having only finitely many nondegenerate simplices:
\begin{itemize}
\item[$(\ast')$] For every map $K \times \Delta^2 \rightarrow S$, the inclusion
$X \times_{S} (K \times \Lambda^2_1) \rightarrow X \times_{S} (K \times \Delta^2)$ is a categorical equivalence.
\end{itemize}
We now prove $(\ast')$ by induction on the dimension $n$ of $K$ and the number of nondegenerate $n$-simplices of $K$. If $K$ is empty, there is nothing to prove. Otherwise, we have a pushout diagram
$$ \xymatrix{ \bd \Delta^n \ar[r] \ar[d] & \Delta^n \ar[d] \\
K' \ar[r] & K. }$$
Using the left properness of the Joyal model structure, we see that $K$ will satisfy
$(\ast')$ provided that $K'$, $\bd \Delta^n$, and $\Delta^n$ satisfy $(\ast')$. In the first
two cases, this follows from the inductive hypothesis. We are therefore reduced to the
case $K = \Delta^n$. Fix a map $\Delta^n \times \Delta^2 \rightarrow S$, and consider the flat inner fibration $q: X \times_{S} ( \Delta^n \times \Delta^2) \rightarrow \Delta^n \times \Delta^2$.
To prove that $(\ast')$ is satisfied, we must show that the composition
$$ X \times_{S} ( \Delta^n \times \Delta^2) \stackrel{q}{\rightarrow} \Delta^n \times \Delta^2
\rightarrow \Delta^2$$
is flat, which follows from Lemma \ref{spaz}.
\end{proof}

\begin{corollary}\label{silman}
Let $p: \calC \rightarrow \calD$ be a flat categorical fibration between $\infty$-categories.
Then, for every categorical equivalence of simplicial sets $A \rightarrow B$ and
every diagram $B \rightarrow \calD$, the induced map
$\theta: A \times_{\calD} \calC \rightarrow B \times_{\calD} \calC$
is an equivalence of $\infty$-categories.
\end{corollary}

\begin{proof}
Every map $f: B \rightarrow \calD$ factors as a composition
$$ B \stackrel{f'}{\rightarrow} B' \stackrel{f''}{\rightarrow} \calD,$$
where $f'$ is inner anodyne and $f''$ is an inner fibration (so that $B'$ is an
$\infty$-category). We obtain a commutative diagram
$$ \xymatrix{ A \times_{\calD} \calC \ar[rr]^{\beta} \ar[dr]^{\theta} & & B' \times_{\calD} \calC \\
& B \times_{\calD} \calC \ar[ur]^{\alpha}. & }$$
Proposition \ref{capper} implies that $\alpha$ is a categorical equivalence. By the two-out-of-three property, it suffices to show that $\beta$ is a categorical equivalence. We may therefore replace $B$ by $B'$ and thereby reduce to the case where $B$ is an $\infty$-category.

The map $g: A \rightarrow B$ factors as a composition
$$ A \stackrel{g'}{\rightarrow} A' \stackrel{g''}{\rightarrow} B$$
where $g'$ is inner anodyne and $g''$ is an inner fibration (so that $A'$ is an $\infty$-category).
We obtain a commutative diagram 
$$ \xymatrix{ A \times_{\calD} \calC \ar[rr]^{\theta} \ar[dr]^{\gamma} & & B \times_{\calD} \calC \\
& A' \times_{\calD} \calC \ar[ur]^{\delta}. & }$$
Proposition \ref{capper} implies that $\gamma$ is a categorical equivalence. Using the two-out-of-three property, we are reduced to proving that $\delta$ is a categorical equivalence. We may therefore replace $A$ by $A'$ and thereby reduce to the case where $A$ is an $\infty$-category.

Consider the pullback diagram
$$ \xymatrix{ A \times_{\calD} \calC \ar[r]^{\theta} \ar[d] & B \times_{\calD} \calC \ar[d]  \\
A \ar[r]^{g} & B. }$$
Since the vertical maps in this diagram are categorical fibrations and the
simplicial sets $A$ and $B$ are $\infty$-categories,
Proposition \toposref{leftpropsquare} guarantees that this diagram is homotopy Cartesian
(with respect to the Joyal model structure). Since the $g$ is a categorical equivalence, it follows that $\theta$ is a categorical equivalence as well.
\end{proof}

\begin{corollary}\label{stiffum}
Let $f: \calC \rightarrow \calD$ and $g: \calD \rightarrow \calE$ be flat categorical
fibrations between $\infty$-categories. Then $g \circ f$ is a flat categorical fibration.
\end{corollary}

\begin{proof}
Since $g \circ f$ is evidently a categorical fibration, it will suffice to show that $g \circ f$
is flat. Choose a $2$-simplex $\sigma: \Delta^2 \rightarrow \calE$; we wish to show that
the inclusion $\calC \times_{ \calE} \Lambda^2_1 \subseteq \calC \times_{ \calE} \Delta^2$ is a categorical equivalence. Let $\calD' = \calD \times_{ \calE} \Delta^2$ and
$\calD'' = \calD \times_{ \calE} \Lambda^2_1$. Since $g$ is flat, the inclusion
$\calD'' \subseteq \calD'$ is a categorical equivalence. Since $f$ is flat, Corollary
\ref{silman} guarantees that the inclusion
$$ \calC \times_{ \calE} \Lambda^2_1 \simeq \calC \times_{ \calD} \calD''
\subseteq \calC \times_{ \calD} \calD' \simeq \calC \times_{ \calE} \Delta^2$$ is a categorical equivalence, as desired.
\end{proof}

\subsection{Functoriality}

In this section, we will study the behavior of the model structure described in
Theorem \ref{theo} as a functor of the categorical pattern $\CatP$. Our main result is the following:

\begin{theorem}\label{makera}
Suppose we are given categorical patterns
$\CatP = (M_{\calC}, T, \{ p_{\alpha}: K_{\alpha}^{\triangleleft} \rightarrow \calC\}_{\alpha \in A})$ and
$\CatP' = (M_{\calC'}, T', \{ p'_{\alpha}: {K'_{\alpha}}^{\triangleleft} \rightarrow \calC'\}_{\alpha \in A'})$ on $\infty$-categories $\calC$ and $\calC'$. Suppose we are given a diagram of marked simplicial sets
$$ (\calC, M_{\calC}) \stackrel{\pi}{\leftarrow} (\calK, M) \stackrel{\pi'}{\rightarrow} (\calC', M_{\calC'}).$$
Then the construction $\overline{X} \mapsto \overline{X} \times_{ (\calC, M_{\calC}) } (\calK,M)$
determines a left Quillen functor from $\mset{\CatP}$ to $\mset{\CatP'}$ provided that the following
conditions are satisfied:
\begin{itemize}
\item[$(1)$] The map $\pi: \calK \rightarrow \calC$ is a flat categorical fibration.

\item[$(2)$] The collections of morphisms $M_{S}$ and $M$ contain all equivalences
in $\calC$ and $\calK$, respectively, and are closed under composition.

\item[$(3)$] For every $2$-simplex $\sigma$ of $\calK$ such that $\pi(\sigma) \in T$, we have
$\pi'(\sigma) \in T'$. Moreover, $T$ contains all $2$-simplices $\Delta^2 \rightarrow \calC$
whose restriction to $\Delta^{ \{0,1\} }$ is an equivalence in $\calC$.

\item[$(4)$] For every edge $\Delta^1 \rightarrow \calC$ belonging to $M_{\calC}$, the
induced map $\calK \times_{ \calC} \Delta^1 \rightarrow \Delta^1$ is a Cartesian fibration.

\item[$(5)$] Each of the simplicial sets $K_{\alpha}$ is an $\infty$-category, and each of the
induced maps $\pi_{\alpha}: K_{\alpha}^{\triangleleft} \times_{\calC} \calK \rightarrow K_{\alpha}^{\triangleleft}$ is a coCartesian fibration.

\item[$(6)$] For $\alpha \in A$ and every coCartesian section $s$ of $\pi_{\alpha}$, the composite map
$$ K_{\alpha}^{\triangleleft} \rightarrow K_{\alpha}^{\triangleleft} \times_{\calC} \calK
\rightarrow \calK \stackrel{\pi'}{\rightarrow} \calC'$$
can be identified with $p'_{\beta}$, for some $\beta \in A'$.

\item[$(7)$] Suppose we are given a commutative diagram
$$\xymatrix{ & Y \ar[dr]^{g} & \\
X \ar[ur]^{f} \ar[rr]^{h} & & Z }$$
in $\calK$, where $g$ is locally $\pi$-Cartesian, $\pi(g) \in M_{\calC}$, and $\pi(f)$ is an equivalence.
Then $f \in M$ if and only if $h \in M$. (In particular, a locally $\pi$-Cartesian morphism
$g$ of $\calK$ belongs to $M$ if and only if $\pi(g) \in M_{\calC}$.

\item[$(8)$] Suppose we are given $\alpha \in A$ and a commutative diagram
$$ \xymatrix{ & Y \ar[dr]^{g} & \\
X \ar[ur]^{f} \ar[rr]^{h} & & Z }$$
in $K_{\alpha}^{\triangleleft} \times_{\calC} \calK$, where
$f$ is $\pi_{\alpha}$-coCartesian and $\pi_{\alpha}(g)$ is an equivalence. Then the image of $g$ in $\calK$ belongs to $M$
if and only if the image of $h$ in $\calK$ belongs to $M$.
\end{itemize}
\end{theorem}

\begin{proof}
Consider the categorical pattern
$\CatP'' = (M, \pi^{-1}(T), \{ p''_{\alpha,s}: K_{\alpha}^{\triangleleft} \rightarrow \calK\}_{(\alpha,s) \in A''})$
on $\calK$, where $A''$ consists of all pairs $(\alpha, s)$ such $\alpha \in A$ and
$s$ is a coCartesian section of $\pi_{\alpha}$, and $p''_{\alpha,s}$ is the composition
$$K_{\alpha}^{\triangleleft} \stackrel{s}{\rightarrow} \calK \times_{ \calC} K_{\alpha}^{\triangleleft} \rightarrow \calK.$$
The functor in question admits a factorization
$$ \mset{\CatP} \stackrel{ \pi^{\ast}}{\rightarrow} \mset{\CatP''} \stackrel{ \pi'_{!} }{\rightarrow} \mset{\CatP},$$
where $\pi^{\ast}$ and $\pi'_{!}$ are left Quillen functors by virtue of Propositions \ref{candied} and \ref{supper} proven below.
\end{proof}

\begin{definition}
Let $f: S \rightarrow S'$ be a map of simplicial sets. Suppose we are given categorical patterns
$\CatP = (M_S, T, \{ p_{\alpha}: K_{\alpha}^{\triangleleft} \rightarrow S\}_{\alpha \in A})$ and
$\CatP' = (M_S', T', \{ p'_{\alpha}: {K'_{\alpha}}^{\triangleleft} \rightarrow S'\}_{\alpha \in A'})$
on $S$ and $S'$, respectively. We will say that $f$ is {\it compatible} with $\CatP$ and $\CatP'$ if the following conditions are satisfied:

\begin{itemize}
\item The map $f$ carries $M_S$ into $M_{S'}$.
\item The map $f$ carries $T$ into $T'$.
\item For each $\alpha \in A$, the composition
$$K_{\alpha}^{\triangleleft} \stackrel{p_{\alpha}}{\rightarrow} S \stackrel{f}{\rightarrow} S'$$
belongs to $\{ p'_{\alpha}: {K'_{\alpha}}^{\triangleleft} \rightarrow S'\}_{\alpha \in A'})$.
\end{itemize}
\end{definition}

\begin{proposition}\label{supper}
Let $f: S \rightarrow S'$ be a map of simplicial sets, and suppose that $f$ is compatible with
categorical patterns $\CatP$ and $\CatP'$ on $S$ and $S'$, respectively. Then composition
with $f$ induces a left Quillen functor $f_{!}: \mset{ \CatP} \rightarrow \mset{ \CatP'}$.
\end{proposition}

\begin{proof}
It is clear that $f_{!}$ preserves cofibrations. It also admits a right adjoint, given by the pullback functor
$f^{\ast}$ described by the formula $f^{\ast} \overline{X} \simeq \overline{X} \times_{ (S', M_{S'}) } (S, M_S)$. To complete the proof, it will suffice to show that $f_{!}$ preserves $\CatP$-equivalences.
Let $\overline{X}, \overline{Y} \in \mset{\CatP}$, and let $\alpha: \overline{X} \rightarrow \overline{Y}$
be a $\CatP$-equivalence. We wish to show that $f_{!}(\alpha)$ is a $\CatP'$-equivalence.
For this, it suffices to show that for every $\CatP'$-fibered object $\overline{Z} \in \mset{ \CatP'}$, the induced map
$$ \bHom^{\sharp}_{S'}( f_{!} \overline{Y}, \overline{Z}) \rightarrow \bHom^{\sharp}_{S'}(f_{!} \overline{X}, \overline{Z} )$$
is a homotopy equivalence. The left hand side can be identified with $\bHom^{\sharp}_{S}( \overline{Y}, f^{\ast} \overline{Z})$, and the right hand side with $\bHom^{\sharp}_{S}( \overline{X}, f^{\ast} \overline{Z})$. The desired result now follows from the assumption that $\alpha$ is a
$\CatP$-equivalence, and the observation that $f^{\ast} \overline{Z}$ is $\CatP$-fibered.
\end{proof}

\begin{example}\label{megaparnum}
For any categorical pattern $\CatP = (M_{S}, T, \{ p_{\alpha}: K_{\alpha}^{\triangleleft} \rightarrow S\}_{\alpha \in A})$ on any simplicial set $S$, the forgetful functor
$\mset{ \CatP} \rightarrow \mSet$ is a left Quillen functor, where we endow
$\mSet$ with the model structure determined by Theorem \ref{theo} for the
categorical pattern $\CatP_{0} = (M_0, T_0, \{ K_{\alpha}^{\triangleleft} \rightarrow \Delta^0 \}_{\alpha \in A})$ on $\Delta^0$ (here $M_0$ and $T_0$ consist of all edges and $2$-simplices
of $\Delta^0$, respectively). If each of the simplicial sets $K_{\alpha}$ is contractible, then
this coincides with the usual model structure on $\mSet$ (Example \ref{sipper}).
\end{example}

Under a few mild hypotheses on a categorical pattern $\CatP$, we can characterize the fibrations
between fibrant objects in $\mset{ \CatP}$.

\begin{notation}
Let $\CatP = (M_S, T, \{ p_{\alpha}: K_{\alpha}^{\triangleleft} \rightarrow S \}_{\alpha \in A})$
be a categorical pattern on a simplicial set $S$, and suppose we are given a 
$\CatP$-fibered object $\overline{X} = (X,M) \in \mset{ \CatP}$. Let $\pi: X \rightarrow S$
denote the underlying map of simplicial sets. We define a categorical pattern
$\pi^{\ast} \CatP = (M, T', \{ q_{\beta}: K_{\beta}^{\triangleleft} \rightarrow X \}_{\beta \in B})$ 
on $X$ as follows:
\begin{itemize}
\item[$(1)$] The set $M$ is the collection of marked edges of $X$ (in other words, the collection
of locally $\pi$-coCartesian edges $e$ of $X$ such that $\pi(e) \in M_S$).
\item[$(2)$] The set $T' = \pi^{-1} T$ is the collection of all $2$-simplices of $X$ whose images in $S$ belong to $T$.
\item[$(3)$] We let $\{ q_{\beta}: K_{\beta}^{\triangleleft} \rightarrow X \}_{\beta \in B}$ be the collection
of those diagrams $q: K^{\triangleleft} \rightarrow X$ such that $q$ carries each edge of
$K^{\triangleleft}$ into $M$, and $\pi \circ q$ belongs to $\{ p_{\alpha}: K_{\alpha}^{\triangleleft} \rightarrow S \}_{\alpha \in A}$.
\end{itemize}
\end{notation}

\begin{lemma}\label{sock}
Let $\CatP = (M_S, T, \{ p_{\alpha}: K_{\alpha}^{\triangleleft} \rightarrow S \}_{\alpha \in A})$
be a categorical pattern on an $\infty$-category $S$, and let $\overline{X} = (X,M) \in \mset{ \CatP}$ be
$\CatP$-fibered. Assume that $M_{S}$ is the collection of all equivalences in $S$ and that $T$ contains
all $2$-simplices $\Delta^2 \rightarrow S$ whose restriction to $\Delta^{ \{0,1\} }$ in an equivalence in $S$. Then $M$ is the collection of all equivalences in $X$.
\end{lemma}

\begin{proof}
Let $p: X \rightarrow S$ denote the underlying map of simplicial sets. The set $M$ consists
of all locally $p$-coCartesian morphisms $f$ in $X$ such that $p(f)$ is an equivalence in $S$.
In view of Proposition \toposref{universalequiv}, it will suffice to show that every such morphism
is $p$-coCartesian. This follows from Lemma \ref{kisker} together with our assumption on the set of $2$-simplices $T$.
\end{proof}

\begin{proposition}\label{fibraman}
Let $\CatP = (M_S, T, \{ p_{\alpha}: K_{\alpha}^{\triangleleft} \rightarrow S \}_{\alpha \in A})$
be a categorical pattern on an $\infty$-category $S$. Suppose that $M_{S}$ contains all equivalences in $S$ and that $T$ contains
all $2$-simplices $\Delta^2 \rightarrow S$ whose restriction to $\Delta^{ \{0,1\} }$ in an equivalence in $S$. Let $\overline{Y} = (Y,M_Y)$ be a $\CatP$-fibered object of $\mset{\CatP}$, and
let $\pi: Y \rightarrow S$ denote the underlying map of simplicial sets. Let
$\overline{X} = (X, M_X)$ be another object of $\mset{\CatP}$, and let
$f: \overline{X} \rightarrow \overline{Y}$ be a morphism. The following conditions are equivalent:
\begin{itemize}
\item[$(a)$] The map $p$ is a fibration in $\mset{\CatP}$.

\item[$(b)$] The object $\overline{X}$ is $\CatP$-fibered, and the underlying map of simplicial sets
$X \rightarrow Y$ is a categorical fibration.

\item[$(c)$] The map $p$ exhibits $\overline{X}$ as a $\pi^{\ast} \CatP$-fibered object of
$\mset{ \pi^{\ast} \CatP}$. 
\end{itemize}
\end{proposition}

\begin{proof}
We first prove that $(a) \Rightarrow (b)$. If $\overline{Y}$ is $\CatP$-fibered and $p$
is a fibration in $\mset{ \CatP}$, then it is clear that $\overline{X}$ is $\CatP$-fibered. It will therefore suffice to show that
the underlying map of simplicial sets $X \rightarrow Y$ is a categorical fibration. In other words, we must show that every lifting problem of the form
$$ \xymatrix{ A^{\flat} \ar[r] \ar[d]^{i} & \overline{X} \ar[d] \\
B^{\flat} \ar[r] \ar@{-->}[ur] & \overline{Y} }$$
admits a solution, provided that the underlying map of simplicial sets $A \rightarrow B$ is
a trivial cofibration with respect to the Joyal model structure. To prove this, it will suffice to show
that the map $i$ is a $\CatP$-equivalence. By virtue of Proposition \ref{supper}, it will suffice to prove this after replacing $\CatP$ by the categorical pattern $\CatP' = ( M'_{S}, T, \emptyset)$, where
$M'_{S}$ is the collection of all equivalences in $S$. We must now show that for every
$\CatP$-fibered object $\overline{Z} = (Z,M) \in \mset{ \CatP}$, the induced map
$\theta: \bHom^{\sharp}_{S}( B^{\flat}, \overline{Z}) \rightarrow \bHom^{\sharp}_{S}( A^{\flat}, \overline{Z})$ is a
weak homotopy equivalence. We observe that $Z$ is an $\infty$-category and
$M_Z$ can be identified with the collection of all equivalences in $Z$ (Lemma \ref{sock}). For every
simplicial set $K$ and every $\infty$-category $\calC$, let $\Fun(K, \calC)^{0}$ denote the largest Kan complex contained in $\Fun(K, \calC)$. We have a commutative diagram
$$ \xymatrix{ \bHom^{\sharp}_{S}( B^{\flat}, \overline{Z}) \ar[r] \ar[d]^{\theta} & \Fun(B, Z)^{0} \ar[r] \ar[d]^{\theta'} & \Fun(B, S)^{0} \ar[d]^{\theta''} \\
\bHom^{\sharp}_{S}( A^{\flat}, \overline{Z} ) \ar[r] & \Fun( A, Z)^{0} \ar[r] & \Fun(A, S)^{0}. }$$
where the rows are homotopy fiber sequences. Consequently, to prove that $\theta$
is a homotopy equivalence, it suffices to show that $\theta'$ and $\theta''$ are homotopy equivalences.
This follows from the observation that the maps 
$$\Fun(B, Z) \rightarrow \Fun(A,Z) \quad \quad \Fun(B, S) \rightarrow \Fun(A,S)$$
are categorical equivalences (in fact, trivial Kan fibrations), since $A \rightarrow B$ is
a trivial cofibration and the simplicial sets $S$ and $Z$ are fibrant (with respect to the Joyal model structure).

We now show that $(b) \Rightarrow (a)$. We must prove that every lifting problem of the form
$$ \xymatrix{ \overline{A} \ar[r]^{f_0} \ar[d]^{i} & \overline{X} \ar[d]^{p} \\
\overline{B} \ar[r]^{g} \ar@{-->}[ur] & \overline{Y} }$$
admits a solution, provided that $i$ is a monomorphism and a $\CatP$-equivalence. 
Since $\overline{X}$ is $\CatP$-fibered, the lifting problem
$$ \xymatrix{ \overline{A} \ar[r]^{f_0} \ar[d]^{i} & \overline{X} \ar[d] \\
\overline{B} \ar[r] \ar@{-->}[ur]^{f'} & (S,M_S). }$$
admits a solution. The map $g' = p \circ f'$ does not necessarily coincide with $g$. However,
$g$ and $g'$ agree on $\overline{A}$ and therefore determine a map
$$ G_0: ( \overline{A} \times (\Delta^1)^{\sharp} ) \coprod_{ \overline{A} \times ( \bd \Delta^1)^{\sharp}}
( \overline{B} \times (\bd \Delta^1)^{\sharp}) \rightarrow \overline{Y}.$$ Consider the diagram
$$ \xymatrix{  ( \overline{A} \times (\Delta^1)^{\sharp} ) \coprod_{ \overline{A} \times ( \bd \Delta^1)^{\sharp}}
( \overline{B} \times (\bd \Delta^1)^{\sharp}) \ar[rr]^{G_0} \ar[d]^{j} & & \overline{Y} \ar[d] \\
\overline{B} \times ( \Delta^1)^{\sharp} \ar[r] \ar@{-->}[urr]^{G} & \overline{B} \ar[r]^{ \pi \circ g} & (S, M_S). }$$
Since the map $j$ is a $\CatP$-equivalence (Proposition \ref{postprod}) and $\overline{X}$ is
$\CatP$-fibered, there exists a map $G$ rendering this diagram commutative. We
regard $G$ as an equivalence from $g$ to $p \circ f'$ in $\Fun( B, Y)$. Since $p$ is a
categorical fibration, it induces a fibration $X^{\natural} \rightarrow Y^{\natural}$ in the category
of marked simplicial sets; here $X^{\natural} = (X, E_X)$ where $E_{X}$ is the collection of all equivalences in $X$ and $Y^{\natural}$ is defined similarly. It follows that the lifting problem
$$ \xymatrix{ (A^{\flat} \times (\Delta^1)^{\sharp}) \coprod_{ A^{\flat} \times \{1\}^{\sharp} }
( B^{\flat} \times \{1\}^{\sharp}) \ar[r] \ar[d] & A^{\flat} \ar[r]^{f'} & X^{\natural} \ar[d] \\
B^{\flat} \times (\Delta^1)^{\sharp} \ar[rr]^{G} \ar@{-->}[urr]^{F} & & Y^{\natural} }$$
admits a solution. We can regard $F$ as an equivalence from $f$ to $f'$ in
$\Fun(B, X)$, where $f$ is an extension of $f_0$ lifting $g$. Since $f$ is equivalent to
$f'$, it carries marked edges of $\overline{A}$ to marked edges of $\overline{X}$, and therefore constitutes a solution to the original lifting problem.

We next show that $(a) \Rightarrow (c)$. We must prove that every lifting problem of the form
$$ \xymatrix{ \overline{A} \ar[r]^{f_0} \ar[d]^{i} & \overline{X} \ar[d]^{p} \\
\overline{B} \ar[r]^{g} \ar@{-->}[ur] & \overline{Y} }$$
admits a solution, provided that $i$ is a trivial cofibration in $\mset{ \pi^{\ast} \CatP}$. 
Since $p$ is assumed to be a fibration in $\mset{ \CatP}$, it suffices to show that
$i$ is a trivial cofibration in $\mset{ \CatP}$, which follows from Proposition \ref{supper}.

Finally, we must show that $(c) \Rightarrow (b)$. Replacing $\CatP$ by $\pi^{\ast} \CatP$ and
invoking the implication $(a) \Rightarrow (b)$, we deduce that $X \rightarrow Y$ is a categorical fibration. It will therefore suffice to show that $\overline{X}$ is $\CatP$-fibered.
We will show that $\overline{X}$ satisfies conditions $(1)$, $(2)$, $(3)$, $(4)$, and
$(6)$ of Definition \ref{catpat}, together with condition $(5')$ of Remark \ref{ratpat}:
\begin{itemize}
\item[$(1)$] The underlying map of simplicial sets $q: X \rightarrow S$ is an inner fibration.
This is clear, since $q = \pi \circ p$, where both $\pi$ and $p$ are inner fibrations.

\item[$(2)$] For each edge $\Delta^1 \rightarrow S$ belonging to $M_S$, the induced map
$q_{\Delta^1}: X \times_{S} \Delta^1 \rightarrow \Delta^1$ is a coCartesian fibration.
In other words, we must show that for every object $x \in X$ and every edge
$e: q(x) \rightarrow s$ belonging to $M_{S}$, there exists a locally
$q$-coCartesian edge $\overline{e}: x \rightarrow \overline{s}$ with
$q( \overline{e}) = e$. Since $\overline{Y}$ is $\CatP$-fibered, we can choose
a locally $\pi$-coCartesian edge $\widetilde{e}: p(x) \rightarrow \widetilde{s}$
with $\pi( \widetilde{e} ) = e$. Moreover, the edge $\widetilde{e}$ belongs to
$M_{Y}$, so we can choose a locally $p$-coCartesian edge $\overline{e}$ with
$p( \overline{e} ) = \widetilde{e}$ (note that $\overline{e}$ belongs to $M_{X}$).
To complete the proof, it will suffice to show that $\overline{e}$ is locally
$q$-coCartesian: in other words, that it determines a $q_{\Delta^1}$-coCartesian edge
$\overline{e}'$ of $X \times_{S} \Delta^1$. We note that $q_{ \Delta^1}$ factors as a composition
$$ X \times_{S} \Delta^1 \stackrel{ q'_{ \Delta^1} }{\rightarrow} Y \times_{S} \Delta^1
\stackrel{ \pi_{\Delta^1}} \Delta^1,$$
and that $q'_{ \Delta^1}( \overline{e}')$ is $\pi_{ \Delta^1}$-coCartesian by construction.
In view of Proposition \toposref{stuch}, it suffices to show that 
$\overline{e}'$ is $q'_{ \Delta^1}$-coCartesian. This follows from Lemma \ref{kisker}, since
the image of every $2$-simplex $\sigma$ of $X \times_{S} \Delta^1$ in $Y$ is a thin
$2$-simplex with respect to $\pi^{\ast} \CatP$ (since the image of $\sigma$ in $S$ is degenerate).

\item[$(3)$] An edge $\overline{e}: x \rightarrow x'$ of $X$ belongs to $M_X$ if and only if $e = q(\overline{e})$ belongs to $M_S$ and $e$ locally $q$-coCartesian. The ``if'' direction follows from the proof of $(2)$. For the converse, we observe that if $e \in M_{S}$ then we can apply the construction of $(2)$ to
produce a locally $q$-coCartesian edge $\overline{e}': x \rightarrow x''$ of $X$ covering
$e$, where $\overline{e}' \in M_{X}$. If $\overline{e}$ is also locally $q$-coCartesian, then
$\overline{e}$ and $\overline{e}'$ are equivalent, so $\overline{e}$ also belongs to $M_{X}$.

\item[$(4)$] Given a commutative diagram
$$ \xymatrix{ \Delta^{ \{0,1\} } \ar[d] \ar[r]^{e} & X \ar[d] \\
\Delta^2 \ar[r]^{\sigma} & S, }$$
if $e \in M_{X}$ and $\sigma \in T$, then $e$ determines an
$q_{\Delta^2}$-coCartesian edge of $X \times_{S} \Delta^2$, where
$q_{\Delta^2}: X \times_{S} \Delta^2 \rightarrow \Delta^2$ denotes the projection map.
To prove this, we factor $q_{ \Delta^2}$ as a composition
$$ X \times_{S} \Delta^2 \stackrel{ p_{ \Delta^2} }{\rightarrow} Y \times_{S} \Delta^2
\stackrel{ \pi_{\Delta^2}} \Delta^2.$$
Since $\overline{Y}$ is $\CatP$-fibered and $p(e) \in M_Y$, we conclude that
the image of $e$ in $Y \times_{S} \Delta^2$ is $\pi_{ \Delta^2}$-coCartesian.
In view of Proposition \toposref{stuch}, it will suffice to show that
$e$ determines a $p_{ \Delta^2}$-coCartesian edge of $X \times_{S} \Delta^2$.
This follows from Lemma \ref{kisker}, since $e$ determines a locally
$p_{ \Delta^2}$-coCartesian edge of $X \times_{S} \Delta^2$ and the image of
every $2$-simplex of $X \times_{S} \Delta^2$ in $Y$ is thin with respect to $\pi^{\ast} \CatP$.

\item[$(5')$] For each $\alpha \in A$, every lifting problem of the form
$$ \xymatrix{ K_{\alpha}^{\sharp} \ar[r] \ar[d] & \overline{X} \ar[d] \\
(K_{\alpha}^{\triangleleft})^{\sharp} \ar[r]^{p_{\alpha}} \ar@{-->}[ur] & (S, M_S) }$$
admits a solution. We first invoke the fact that $\overline{Y}$ is $\CatP$-fibered
to solve the induced lifting problem
$$ \xymatrix{ K_{\alpha}^{\sharp} \ar[r] \ar[d] & \overline{Y} \ar[d] \\
(K_{\alpha}^{\triangleleft})^{\sharp} \ar[r]^{p_{\alpha}} \ar@{-->}[ur]^{f} & (S, M_S). }$$
We then invoke the assumption that $\overline{X}$ is $\pi^{\ast} \CatP$ fibered to
solve the lifting problem
$$ \xymatrix{ K_{\alpha}^{\sharp} \ar[r] \ar[d] & \overline{X} \ar[d] \\
(K_{\alpha}^{\triangleleft})^{\sharp} \ar[r]^{f} \ar@{-->}[ur] & \overline{Y}. }$$

\item[$(6)$] For every index $\alpha \in A$, every map
$\overline{p}_{\alpha}: (K_{\alpha}^{\triangleleft})^{\sharp} \rightarrow \overline{X}$
lifting $p_{\alpha}: (K_{\alpha}^{\triangleleft})^{\sharp} \rightarrow (X,S)$ is
a $q$-limit diagram. Invoking the assumption that $\overline{Y}$ is $\CatP$-fibered,
we deduce that $\widetilde{p}_{\alpha} = p \circ \overline{p}_{\alpha}$ is
a $\pi$-limit diagram in $Y$. Moreover, $\widetilde{p}_{\alpha}$ is one of the diagrams
defining the categorical pattern $\pi^{\ast} \CatP$, so our assumption that
$\overline{X}$ is $\CatP$-fibered ensures that $\overline{p}_{\alpha}$ is a $p$-limit diagram.
Since $q = \pi \circ p$, the desired result now follows from Proposition \toposref{basrel}.
\end{itemize} 
\end{proof}

Proposition \ref{supper} can be interpreted roughly as saying that the model structure
of Theorem \ref{theo} defines a covariant functor of the underlying categorical
pattern $\CatP$. The remainder of this section is devoted to studying the behavior of
this model structure as a {\em contravariant} functor of $\CatP$: in other words, we wish to study circumstances under which the pullback functor $f^{\ast}$ is a left Quillen functor.
First, let us introduce a bit of terminology.

\begin{notation}\label{cabbel}
Suppose we are given maps of simplicial sets $X \stackrel{\phi}{\rightarrow} Y \stackrel{\pi}{\rightarrow} Z$.
We let $\pi_{\ast}(X)$ denote a simplicial set equipped with a map $\pi_{\ast} X \rightarrow Z$
with the following universal property: for every map of simplicial sets $K \rightarrow Z$, we
have a canonical bijection
$$ \Hom_{ Z}( K, \pi_{\ast}(X) ) \simeq \Hom_{Y}( K \times_{Z} Y, X).$$
\end{notation}

In the situation of Notation \ref{cabbel}, suppose that $\pi$ is a Cartesian fibration and the map
$\phi$ is a coCartesian fibration. Corollary \toposref{presalad} implies that the
map $\pi_{\ast} X \rightarrow Z$ is a coCartesian fibration. We will need some refinements of this result.

\begin{proposition}\label{saeve}
Let $\pi: Y \rightarrow Z$ be a flat categorical fibration of simplicial sets.
Then the functor $\pi_{\ast}: ( \sSet)_{/Y} \rightarrow (\sSet)_{/Z}$ is a right
Quillen functor (with respect to the Joyal model structures). In particular,
if $X \rightarrow Y$ is a categorical fibration, then the induced map
$\pi_{\ast} X \rightarrow Z$ is a categorical fibration.
\end{proposition}

\begin{proof}
The functor $\pi_{\ast}$ admits a left adjoint $\pi^{\ast}$, given by the formula
$\pi^{\ast} A = A \times_{Z} Y$. To prove that $\pi_{\ast}$ is a right Quillen functor, it suffices to show that $\pi^{\ast}$ preserves cofibrations and weak equivalences. The case of cofibrations is clear, and
the case of weak equivalences follows from Corollary \ref{silman}.
\end{proof}

\begin{example}\label{just}
Suppose we are given a diagram of simplicial sets $X \stackrel{\phi}{\rightarrow} Y \stackrel{\pi}{\rightarrow} Z$.
We observe that there is a canonical map $\theta: X \rightarrow \pi_{\ast} X$. If the map $\pi$ is
a trivial Kan fibration, then $\theta$ is a categorical equivalence. To prove this, we first choose
a section $s: Z \rightarrow Y$ of $\pi$. Composition with $s$ yields a map
$r: \pi_{\ast} X \rightarrow X$ such that $r \circ \theta = \id_{X}$. Moreover, since
$s \circ \pi$ is homotopic (over $Z$) to the map $\id_{Y}$. It follows that there exists a contractible
Kan complex $K$ containing a pair of distinct points $x$ and $y$ and a map
$h: K \times Y \rightarrow Y$ compatible with the projection map $\pi$ such that
$h|( \{x\} \times Y) = \id_{Y}$ and $h | (\{y\} \times Y) = s \circ \pi$. The map
$h$ induces a map $h': K \times \pi_{\ast} X \rightarrow \pi_{\ast} X$ such
that $h|( \{x\} \times \pi_{\ast} X) = \id_{ \pi_{\ast} X}$ and $h|( \{y\} \times \pi_{\ast} X) = \theta \circ r$.
It follows that $r$ is a right homotopy inverse to $\theta$ (as well as being a strict left inverse) with respect to the Joyal model structure, so that $\theta$ is a categorical equivalence as desired.
\end{example}

\begin{remark}\label{safi}
Suppose we are given a diagram of simplicial sets
$X \stackrel{\phi}{\rightarrow} Y \stackrel{\pi}{\rightarrow} Z$, where
$\phi$ is a categorical fibration and $\pi$ is a flat categorical fibration.
Let $\psi: Y' \rightarrow Y$ be a trivial Kan fibration, let $\pi' = \pi \circ \psi$, and
let $X' = X \times_{Y} Y'$. Then the canonical map $f: \pi_{\ast} X \rightarrow \pi'_{\ast} X'$
is a categorical equivalence. To prove this, we observe that $\pi'_{\ast} X'
\simeq \pi_{\ast} \psi_{\ast} X'$, and $f$ is induced by applying
$\pi_{\ast}$ to a map $g: X \rightarrow \psi_{\ast} X'$. 
Example \ref{just} shows that $g$ is a categorical equivalence. Since
$\pi_{\ast}$ is a right Quillen functor (Proposition \ref{saeve}), it preserves
categorical equivalences between fibrant objects of $(\sSet)_{/Y}$, so
$f$ is a categorical equivalence.
\end{remark}

\begin{lemma}\label{hodo2}
Let $q: \calC \rightarrow \Delta^n$ and $p: \calD \rightarrow \calE$ be categorical fibrations of
$\infty$-categories, where $n \geq 2$. Let $\calC^{0}$ be a full subcategory of $\calC$ with the following properties:
\begin{itemize}
\item[$(i)$] The subcategory $\calC^{0} \times_{ \Delta^{n} } \Delta^{ \{n-1,n\} }$ is a cosieve on $\calC
\times_{ \Delta^n} \Delta^{ \{n-1,n\}}$: that is, for every morphism
$f: x \rightarrow y$ in $\calC \times_{ \Delta^n} \Delta^{ \{n-1,n\} }$, if $x \in \calC^{0}$, then $y \in \calC^{0}$.
\item[$(ii)$] For every object $y \in \calC^{0} \times_{ \Delta^{n} } \{n-1\}$ and
each $i < n-1$, there exists an object $x \in \calC^{0} \times_{ \Delta^{n} } \{i\}$ and
a $q$-Cartesian morphism $x \rightarrow y$ in $\calC$.
\end{itemize}
Suppose we are given a lifting problem
$$ \xymatrix{ (\calC \times_{ \Delta^n} \Lambda^n_{n})
\coprod_{ \calC^{0} \times_{ \Delta^n} \Lambda^n_{n}} \calC^{0} \ar[r]^{f_0} \ar[d] & \calD \ar[d]^{p} \\
\calC \ar[r]^{g} \ar@{-->}[ur]^{f} & \calE. }$$
Let $X = (\calC \times_{\Delta^n} \{n\}) \coprod_{ \calC^{0} \times_{ \Delta^n} \{n\} }
( \calC^{0} \times_{ \Delta^n} \Delta^{ \{n-1, n\} })$; condition $(i)$ guarantees that
$X$ can be identified with a full subcategory of $\calC$. Assume further that
\begin{itemize}
\item[$(iii)$] The functor $f_0|(\calC \times_{ \Delta^n} \Delta^{ \{n-1, n\}})$ is a $p$-right Kan extension
of $f_0 |X$.
\end{itemize}
Then there exists a dotted arrow $f$ rendering the diagram commutative.
\end{lemma}

\begin{proof}
Using Proposition \toposref{minimod}, we can choose a minimal $\infty$-category
$\calC'$ and a categorical equivalence $\calC' \rightarrow \calC$. Let
${\calC'}^{0}$ denote the fiber product $\calC' \times_{\calC} \calC^{0}$.
We observe that if $K$ is a face of $\Delta^n$, the maps
$$\calC' \times_{ \Delta^n} K \rightarrow \calC \times_{ \Delta^n } K
\quad \quad {\calC'}^{0} \times_{ \Delta^n} K \rightarrow \calC^{0} \times_{ \Delta^n } K
$$ are categorical equivalences. Because the Joyal model structure is left proper,
we deduce that for every inclusion of simplicial subsets $K \subseteq K' \subseteq \Delta^n$,
the induced map
$$ (K \times_{ \Delta^n} \calC') \coprod_{ K \times_{ \Delta^n} {\calC'}^{0}}
(K' \times_{ \Delta^n} {\calC'}^{0}) \rightarrow
 (K \times_{ \Delta^n} \calC) \coprod_{ K \times_{ \Delta^n} {\calC}^{0}}
(K' \times_{ \Delta^n} {\calC}^{0})$$
is a categorical equivalnece. Invoking Proposition \toposref{princex}, we are reduced to solving
the associated lifting problem
$$ \xymatrix{ (\calC' \times_{ \Delta^n} \Lambda^n_{n})
\coprod_{ {\calC'}^{0} \times_{ \Delta^n} \Lambda^n_{n}} {\calC'}^{0} \ar[r] \ar[d] & \calD \ar[d]^{p} \\
\calC' \ar[r] \ar@{-->}[ur] & \calE. }$$
We may therefore replace $\calC$ by $\calC'$ and thereby reduce to the case where $\calC$ is minimal.

Let $\calC^{1}$ denote the simplicial subset of $\calC$ consisting of all those simplices
$\sigma$ satisfying one of the following conditions:
\begin{itemize}
\item The image of $\sigma$ in $\Delta^n$ does not contain
$\Delta^{ \{0, 1, \ldots, n-1 \} }$. 
\item The intersection of $\sigma$ with $\calC \times_{ \Delta^{n} } \{n-1\}$ is contained in
$\calC^{0}$.
\end{itemize}
We first extend $f_0$ to a map $f_1: \calC^{1} \rightarrow \calD$. For every
simplicial subset $S \subseteq X$, let $\calC^{1}_{S}$ denote the simplicial subset
of $\calC^{1}$ generated by $\calC^{0}$, $\calC \times_{ \Delta^n } \Lambda^n_n$, and
those simplices $\sigma$ whose intersection with $X$ belong to $S$. Let
$A$ denote the collection of pairs $(S, f_{S})$, where $S$ is a simplicial subset of
$X$ containing $\calC \times_{ \Delta^n} \{n\}$, and $f_{S}: \calC^{1}_{S} \rightarrow \calD$
is an extension of $f_0$ such that $p \circ f_{S} = g| \calC^{1}_{S}$.
We regard $A$ as a partially ordered set, with $(S, f_{S}) \leq (S', f_{S'})$ if $S \subseteq S'$
and $f_{S} = f_{S'} | S^{+}$. Using Zorn's lemma, we deduce that $A$ admits a maximal element
$(S, f_S)$. If $S = X$, then we set $f_1 = f_{S}$. Otherwise, choose a nondegenerate simplex
$\sigma$ of $X$ which does not belong to $S$ having minimal dimension $k$. Let $S'$ be the simplicial subset of $X$ obtained
by adjoining the simplex $\sigma$ to $S$, let $\calA = \calC_{/\sigma} \times_{ \Delta^n} \Delta^{ \{0, \ldots, n-2\} }$. For every simplicial subset $K \subseteq \Delta^{ \{ 0, \ldots, n-2\} }$, we let
$\calA_{K}$ denote the simplicial subset of $\calA$ generated by
$\calA \times_{ \calC} \calC^{0}$ together with $\calA \times_{ \Delta^{ \{0, \ldots, n-2\} }} K$,
and let $\calA_{\bd}$ denote $\calA_{ \bd \Delta^{ \{0, \ldots, n-2\} }}$. 
We have a diagram
$$ \xymatrix{ (\calA_{\bd} \star \Delta^k) \coprod_{ \calA_{\bd} \star \bd \Delta^k}
( \calA \star \bd \Delta^k) \ar[r] \ar[d] & \calC^{1}_{S} \ar[d] \\
\calA \star \Delta^k \ar[r] & \calC^{1}_{S'}. }$$
Using Proposition \symmetricref{appstrict}, we conclude that this diagram is a pushout square.
The maximality of the pair $(S, f_S)$ guarantees the existence of an insoluble lifting problem
$$ \xymatrix{ (\calA^{0}_{\bd} \star \Delta^k) \coprod_{ \calA^{0}_{\bd} \star \bd \Delta^k}
( \calA \star \bd \Delta^k) \ar[r] \ar[d] & \calD \ar[d]^{p} \\
\calA \star \Delta^k \ar@{-->}[ur] \ar[r] & \calE. }$$
Since $p$ is a categorical fibration, we conclude that the left vertical map is
not inner anodyne. We will obtain a contradiction from Lemma \toposref{precough} by showing
that the inclusion $i: \calA_{\bd} \rightarrow \calA$ is right anodyne.
In fact, we will prove the following more general claim:
\begin{itemize}
\item[$(\ast)$] For every simplex $\tau \subseteq \Delta^{ \{0, \ldots, n-2 \} }$ and
every simplicial subset $K \subseteq \tau$, the inclusion $\calA_{K} \subseteq
\calA_{ \tau}$ is right anodyne.
\end{itemize}
The proof proceeds by induction on the dimension of $\tau$, and on the number of
nondegenerate simplices of $K$. If $K = \tau$ there is nothing to prove. If
$K \neq \tau$ is nonempty, then we can write $K$ as a pushout
$K' \coprod_{ \bd \tau'} \tau'$, where $\tau'$ is a simplex of small dimension than $\tau$. 
The inductive hypothesis shows that the inclusion $\calA_{K'} \subseteq \calA_{\tau}$ is
right anodyne. By virtue of Proposition \toposref{cofbasic}, it will suffice to show that
the inclusion $\calA_{K'} \subseteq \calA_{K}$ is right anodyne. For this, we consider
the pushout diagram
$$ \xymatrix{ \calA_{ \bd \tau'} \ar[r] \ar[d] & \calA_{ \tau' } \ar[d] \\
\calA_{K'} \ar[r] & \calA_{K}. }$$
Since the upper horizontal map is right anodyne by the inductive hypothesis, the lower
horizontal map is right anodyne as well.

It remains to consider the case $K = \emptyset$. In this case, the inclusion $\calA_{K} \subseteq \calA_{\tau}$ is a pushout of the inclusion 
$$i': \calC^{0} \times_{\calC} \calC_{/ \sigma} \times_{ \Delta^n} \tau
\subseteq \calC_{/ \sigma} \times_{ \Delta^n} \tau.$$
It will therefore suffice to show that $i'$ is right anodyne, which is equivalent
(Proposition \toposref{cofbasic}) to the assertion that $i'$ is cofinal.
Using Proposition \toposref{cofbasic} again, we see that it suffices to produce
an object $x \in \calC^{0} \times_{\calC} \calC_{/ \sigma} \times_{ \Delta^n} \tau$
such that the corresponding maps
$$  \calC^{0} \times_{\calC} \calC_{/ \sigma} \times_{ \Delta^n} \tau \hookleftarrow \Delta^0 \hookrightarrow \calC_{/ \sigma} \times_{ \Delta^n} \tau$$
are both cofinal. In other words, it suffices to guarantee that the object
$x$ is final in $\calC_{/ \sigma} \times_{ \Delta^n} \tau$. The existence of such an object
follows from assumption $(ii)$ (applied to the object $y \in \calC \times_{ \Delta^n} \{n-1\}$ given by the initial vertex of $\sigma$). This completes the construction of the map $f_{1}$.

We now show that $f_{1}$ can be extended to the desired map $f: \calC \rightarrow \calD$.
Let $Y$ denote the full subcategory of $\calC$ spanned by those vertices which do not belong to $X$.
For every simplicial subset $T \subseteq Y$, let $\calC_{T}$ denote the simplicial subset
of $\calC$ consisting of those simplices whose intersection with $Y$ belongs to $T$.
Let $B$ denote the collection of pairs $(T, f_{T})$, where $T$ is a simplicial subset of
$Y$ containing $Y \cap ( \calC \times_{ \Delta^n} \bd \Delta^{ \{0, \ldots, n-1 \}})$, and
$f_{T}: \calC_{T} \rightarrow \calD$ is an extension of $f_{1}$ such that
$p \circ f_{T} = g| \calC_{T}$. We regard $B$ as a partially ordered set, with $(T, f_{T}) \leq (T', f_{T'})$ if $T \subseteq T'$
and $f_{T} = f_{T'} | \calC_{T'}$. Using Zorn's lemma, we deduce that $B$ admits a maximal element
$(T, f_T)$. If $T = Y$, then we set $f = f_{T}$ and the proof is complete. Otherwise, choose a nondegenerate simplex $\sigma'$ of $Y$ which does not belong to $T$ having minimal dimension $k'$.
Let $T'$ be the simplicial subset of $Y$ obtained
by adjoining the simplex $\sigma'$ to $T$, and let $X_{\sigma'/}$ denote the fiber product
$X \times_{ \calC} \calC_{\sigma'/}$. Using Proposition \toposref{minstrict}, we deduce the diagram
$$ \xymatrix{ \bd \Delta^{k'} \star X_{\sigma'/} \ar[r] \ar[d] & \calC_{T} \ar[d] \\
\Delta^{k'} \star X_{\sigma'/} \ar[r] & \calC_{T'}. }$$
Using the maximality of $(T, f_{T})$, we deduce that there is no solution to the lifting problem
$$ \xymatrix{ \bd \Delta^{k'} \star X_{\sigma'/} \ar[r]^{h} \ar[d] & \calD \ar[d]^{p} \\
\Delta^{k} \star X_{\sigma'/} \ar[r] & \calE, }$$
so that $h$ does not exhibit $h( \{k' \})$ as a $p$-limit of $h | X_{\sigma'/}$. 
Let $v$ denote the final vertex of $\sigma'$, so that $v \in \calC \times_{ \Delta^n} \{n-1\}$.
Since the projection map $X_{\sigma'/} \rightarrow X \times_{ \calC} \calC_{v/}$ is a trivial
Kan fibration, we conclude that $f_0$ does not exhibit $v$ as a $p$-limit of the diagram
$$X \times_{\calC} \calC_{v/} \rightarrow X \rightarrow \stackrel{f_0}{\rightarrow} \calD.$$
This contradicts our assumption that 
$f_0|(\calC \times_{ \Delta^n} \Delta^{ \{n-1, n\}})$ is a $p$-right Kan extension
of $f_0 |X$. 
\end{proof}

\begin{proposition}\label{critter}
Suppose we are given a diagram of $\infty$-categories $X \stackrel{\phi}{\rightarrow} Y \stackrel{\pi}{\rightarrow} Z$ where $\phi$ is a categorical fibration and $\pi$ is a flat categorical fibration.
Let $Y' \subseteq Y$ be a full subcategory. Let $X' = Y' \times_{Y} X$, let $\pi' = \pi | Y'$, and let
$\psi: \pi_{\ast} X \rightarrow \pi'_{\ast} X'$ denote the projection. Let $K$ be another $\infty$-category, let $\overline{p}: K^{\triangleleft} \rightarrow \pi_{\ast} X$ be a diagram, and suppose that the following conditions are satisfied:
\begin{itemize}
\item[$(i)$] The full subcategory $Y' \times_{Z} K^{\triangleleft} \subseteq Y \times_{Z} K^{\triangleleft}$ is a cosieve on $Y \times_{Z} K^{\triangleleft}$.

\item[$(ii)$] For every object $y \in Y'$ and every morphism $f: z \rightarrow \pi(y)$ in $Z$, 
there exists a $\pi$-Cartesian morphism $\overline{f}: \overline{z} \rightarrow y$ in $Y'$ with
$\pi( \overline{f}) = f$.

\item[$(iii)$] The map $\overline{F}: K^{\triangleleft} \times_{Z} Y \rightarrow X$ classified by
$\overline{p}$ is a $\phi$-right Kan extension of $$F = \overline{F} | ((K \times_{Z} Y) \coprod_{ K \times_{Z} Y'} (K^{\triangleleft} \times_{Z} Y')).$$
\end{itemize}
Then $\overline{p}$ is a $\psi$-limit diagram.
\end{proposition}

\begin{proof}
We must show that for $n \geq 2$, every lifting problem of the form
$$ \xymatrix{ \bd \Delta^{n-1} \star K \ar[r]^{f_0} \ar[d] & \pi_{\ast} X \ar[d]^{\psi} \\
\Delta^{n-1} \star K \ar[r] \ar@{-->}[ur] & \pi'_{\ast} X' }$$
admits a solution, provided that $f_0 | \{n-1\} \star K$ coincides with $\overline{p}$. Let
$\calC = ( \Delta^{n-1} \star K) \times_{Z} Y$, and observe that $\calC$ is equipped with a map
$\calC \rightarrow \Delta^n$. Let $\calC^{0} = \calC \times_{Y} Y'$. 
Unwinding the definitions, we see that it suffices to solve a lifting problem of the form
$$ \xymatrix{ (\calC \times_{ \Delta^n} \Lambda^n_n) \coprod_{ \calC^{0} \times_{ \Delta^n} \Lambda^n_n} \calC^{0} \ar[r] \ar[d] & Y \ar[d]^{\phi} \\
\calC \ar[r] \ar@{-->}[ur] & Z. }$$
The desired result now follows from Lemma \ref{hodo2} and our hypothesis on $\overline{F}$.
\end{proof}

We will typically apply Proposition \ref{critter} in the special case where $Y' = \emptyset$, so that
$\pi'_{\ast} X' \simeq Z$. In this cases, conditions $(i)$ and $(ii)$ are automatic.

\begin{proposition}\label{sabre1}
Suppose we are given a diagram of categorical fibrations $X \stackrel{\phi}{\rightarrow} Y \stackrel{\pi}{\rightarrow} Z$ where $\pi$ is a Cartesian fibration and $\phi$ is a categorical fibration. Suppose that the following condition is satisfied:
\begin{itemize}
\item[$(\ast)$] For every vertex $x$ of $X$ and every $\pi$-Cartesian edge $f: \phi(x) \rightarrow y$ in $Y$, there exists a $\phi$-coCartesian edge $\overline{f}: x \rightarrow \overline{y}$ such
that $f = \phi( \overline{f} )$. 
\end{itemize}
Then:
\begin{itemize}
\item[$(1)$] The map $\psi: \pi_{\ast} X \rightarrow Z$ is a coCartesian fibration.
\item[$(2)$] Let $\overline{e}$ be an edge of $\pi_{\ast} X$ lying over an edge $e: z \rightarrow z'$ in
$Z$, corresponding to a map $F: \Delta^1 \times_{Z} Y \rightarrow X$. Then $\overline{e}$ is $\psi$-coCartesian if and only if the following condition is satisfied:
\begin{itemize}
\item[$(a)$] For every $\pi$-Cartesian edge $\widetilde{e}$ of $Y$ lying over $e$, the image
$F(e)$ is a $\phi$-coCartesian edge of $X$.
\end{itemize}
\end{itemize}
\end{proposition}

\begin{proof}
We use the same strategy as in the proof of Proposition \toposref{presalad}.
Proposition \ref{saeve} guarantees that $\psi$ is a categorical fibration, and in particular
an inner fibration. Let us say that an edge of $\pi_{\ast} X$ is {\it special} if it satisfies condition $(a)$. 
We will prove:
\begin{itemize}
\item[$(i)$] For every vertex $A \in \pi_{\ast} X$ and every edge $e: \psi(A) \rightarrow z$ in
$Z$, there exists a special edge $\overline{e}: A \rightarrow \overline{z}$ in $\pi_{\ast} X$
such that $\psi( \overline{e} ) = e$.
\item[$(ii)$] Every special edge of $\pi_{\ast} X$ is $\psi$-coCartesian.
\end{itemize}
This will prove $(1)$ and the ``if'' direction of $(2)$. To prove the ``only if'' direction, we consider
an arbitrary $\psi$-coCartesian edge $\overline{e}: A \rightarrow B$ in $\pi_{\ast} X$ covering an edge $e: z \rightarrow z'$ in $Z$. Using $(i)$, we can choose a special edge $\overline{e}': A \rightarrow C$ in $\pi_{\ast} X$ covering $e$. Using the assumption that $\overline{e}$ is $\psi$-coCartesian, we can choose
a $2$-simplex
$$ \xymatrix{ & B \ar[dr]^{ \overline{e}'' } & \\
A \ar[ur]^{ \overline{e}'} \ar[rr]^{ \overline{e} } & & C}$$
whose image in $Z$ is degenerate. Since $\overline{e}'$ and $\overline{e}$ are both
$\psi$-coCartesian (by $(ii)$), we conclude that $\overline{e}''$ is an equivalence in the $\infty$-category
$(\pi_{\ast} X)_{z'}$. Since $\overline{e}'$ satisfies $(a)$, we deduce that $\overline{e}$ satisfies
$(a)$, as desired.

We now prove $(i)$. Without loss of generality, we may replace $X$ and $Y$ by their pullbacks along the edge $e: \Delta^1 \rightarrow Z$, and thereby reduce to the case $Z = \Delta^1$. We can identify
$A$ with a section of the projection map $X_0 \rightarrow Y_0$. To produce an edge
$\overline{e}: A \rightarrow \overline{z}$ as in $(i)$, we must solve the lifting problem depicted in the diagram
$$ \xymatrix{ Y_0 \ar[d] \ar[r]^{A} & X \ar[d]^{\phi} \\
Y \ar[r] \ar@{--}[ur]^{A'} & Y. }$$
Moreover, $\overline{e}$ is special if and only if the map $A'$ carries $\pi$-Cartesian morphisms
of $Y$ to $\phi$-coCartesian morphisms in $X$. 
Using Proposition \toposref{simplexplay}, we can choose a functor $\chi: X_1 \rightarrow X_0$
and a quasi-equivalence $M(\phi) \rightarrow X$.
Using Propositions \toposref{princex}, we may reduce to the problem of providing a dotted arrow in the diagram
$$ \xymatrix{ X_{0} \ar@{^{(}->}[d] \ar[r] & X \ar[d]_{\phi} \\
M(\phi) \ar@{-->}[ur] \ar[r] & Y }$$
which carries the marked edges of $M^{\natural}(\phi)$ to $\phi$-coCartesian edges of $X$.
This follows from the fact that $\phi^{X_1}: \Fun(X_1, X) \rightarrow \Fun(X_1, Y)$ is a coCartesian fibration and the description of the $\phi^{X_1}$-coCartesian morphisms (Proposition \toposref{doog}).

The proof of $(ii)$ is similar. We wish to prove that every lifting problem
$$ \xymatrix{ \Lambda^n_0 \ar[r] \ar@{^{(}->}[d] & \pi_{\ast} X \ar[d]^{\psi} \\
\Delta^n \ar[r] \ar@{-->}[ur] & Z }$$
has a solution provided that $n \geq 2$ and the upper horizontal map carries $\Delta^{ \{0,1\} } \subseteq \Lambda^n_0$ to a special edge of $\pi_{\ast} X$. Replacing $X$ and $Y$ by their pullbacks along
$\Delta^n \rightarrow Z$, we can assume that the lower horizontal map is an isomorphism. 
Unwinding the definitions, we are reduced to solving the lifting problem
$$ \xymatrix{ Y \times_{\Delta^n} \Lambda^n_0 \ar[r] \ar@{^{(}->}[d] & X \ar[d]^{\phi} \\
Y \ar[r] \ar@{-->}[ur] & Y. }$$
Using Proposition \toposref{simplexplay}, we can choose a composable sequence of morphisms
$$ \chi: Y_0 \leftarrow \cdots \leftarrow Y_{n} $$
and a quasi-equivalence $M(\chi) \rightarrow Y$. Invoking
Propositions \toposref{princex}, we may reduce to the associated mapping problem
$$ \xymatrix{ M(\psi) \times_{ \Delta^n} \Lambda^n_0 \ar[r] \ar[d] & X \ar[d]^{\phi} \\
M(\psi) \ar[r] \ar@{-->}[ur] & Y.}$$
This is equivalent to the mapping problem
$$ \xymatrix{ X_{n} \times \Lambda^n_i \ar[r] \ar@{^{(}->}[d] & X \ar[d]^{\phi} \\
X_{n} \times \Delta^n \ar[r] & Y. }$$
which admits a solution by virtue of Proposition \toposref{doog}.
\end{proof}

\begin{corollary}\label{falch}
Suppose we are given a diagram of categorical fibrations $X \stackrel{\phi}{\rightarrow} Y \stackrel{\pi}{\rightarrow} Z$. Let $M$ be a collection of edges of $Z$ containing all degenerate edges
and $T$ a collection of $2$-simplices of $Z$ containing all degenerate $2$-simplices.
Suppose that the following conditions are satisfied:
\begin{itemize}
\item[$(a)$] The categorical fibration $\pi$ is flat.
\item[$(b)$] For every vertex $y \in Y$ and every edge $f: z \rightarrow \pi(y)$ of
$Z$ which belongs to $M$, there exists a locally $\pi$-Cartesian edge
$\overline{f}: \overline{z} \rightarrow y$ such that $\pi( \overline{f} ) = f$.
\item[$(c)$] For every vertex $x$ of $X$ and every locally $\pi$-Cartesian edge $f: \phi(x) \rightarrow y$ in $Y$ such that $\pi(f) \in M$, there exists a locally $\phi$-coCartesian edge $\overline{f}: x \rightarrow \overline{y}$ such that $f = \phi( \overline{f} )$.
\item[$(d)$] Let $\overline{f}$ be a locally $\phi$-coCartesian edge of $X$ such that
$f = \phi( \overline{f} )$ is locally $\pi$-Cartesian and $\pi(f) \in M$, and suppose that
$f: \Delta^{ \{0,1\} } \rightarrow Y$ is extended to a $2$-simplex $\sigma: \Delta^2 \rightarrow Y$ 
such that $\pi(\sigma) \in T$. Then
$\overline{f}$ determines a $\phi'$-coCartesian morphism of $X \times_{Y} \Delta^2$, where
$\phi': X \times_{Y} \Delta^2 \rightarrow \Delta^2$ denotes the projection.
\end{itemize}
Then:
\begin{itemize}
\item[$(1)$] The map $\psi: \pi_{\ast} X \rightarrow Z$ is a categorical fibration.
\item[$(2)$] For every vertex $x \in \pi_{\ast} X$ and edge morphism $f: \psi(x) \rightarrow z$
of $Z$ which belongs to $M$, there exists a locally $\psi$-coCartesian edge
$\overline{f}: x \rightarrow \overline{z}$ of $\pi_{\ast} X$ with $\psi( \overline{f}) = f$.
\item[$(3)$] Let $\overline{f}$ be an edge of $\pi_{\ast} X$ lying
over an edge $\psi(\overline{f}) = f: z \rightarrow z'$ which belongs to $M$, corresponding to a map
$F: \Delta^1 \times_{Z} Y \rightarrow X$. Then $\overline{f}$ is locally $\psi$-coCartesian if and only if the following condition is satisfied:
\begin{itemize}
\item[$(\ast)$] For every locally $\pi$-Cartesian edge $\widetilde{f}$ of $Y$ lying over $f$, the image
$F( \widetilde{f} )$ is a locally $\phi$-coCartesian edge of $X$.
\end{itemize}
\item[$(4)$] Let $\sigma: \Delta^2 \rightarrow Z$ be $2$-simplex belonging to $T$ such
that the edge $f = \sigma| \Delta^{ \{0,1\} }$ belongs to $M$, and 
let $\overline{f}$ be a locally $\psi$-coCartesian edge of $\pi_{\ast} X$ lying over
$f$. Then $\overline{f}$ determines a
$\psi'$-coCartesian edge of $\pi_{\ast} X \times_{ Z} \Delta^2$, where
$\psi'$ denotes the projection $\pi_{\ast} \times_{Z} \Delta^2 \rightarrow \Delta^2$.
\end{itemize}
\end{corollary}

\begin{proof}
Assertion $(1)$ follows from Proposition \ref{saeve}, and assertions $(2)$ and $(3)$
follow by applying Proposition \ref{sabre1} to the diagram
$X \times_{Z} \Delta^1 \rightarrow Y \times_{Z} \Delta^1 \rightarrow \Delta^1$.
To prove $(4)$, we are free to replace $Z$ by $\Delta^2$ and thereby reduce to the
case where $\sigma$ is an isomorphism. Let $\overline{f}$ be a locally
$\psi$-coCartesian edge of $\pi_{\ast} X$ lying over $\Delta^{ \{0,1\} }$, which
we can identify with a functor $F: Y \times_{ \Delta^2} \Delta^{ \{0,1\} }
\rightarrow X \times_{ \Delta^2} \Delta^{ \{0,1\} }$. We wish to show that $f$ is $\psi$-coCartesian.
By virtue of (the dual of) Proposition \ref{critter}, it will suffice to show that
the functor $F$ is a $\psi$-left Kan extension of $F|Y_0$, where
$Y_0 = Y \times_{ \Delta^2} \{0\}$. Unwinding the definitions, we must show that for
each object $y \in Y \times_{ \Delta^2} \{1\}$, the map $F$ induces a $\phi$-colimit
diagram
$$ \theta: (Y_0 \times_{Y} Y_{/y})^{\triangleright} \rightarrow (Y \times_{\Delta^2} \Delta^{ \{0,1\}})_{/y}^{\triangleright} \rightarrow Y \times_{ \Delta^2} \Delta^{ \{0,1\} } \stackrel{F}{\rightarrow} X.$$
Condition $(b)$ guarantees that the projection $Y \times_{ \Delta^2} \Delta^{ \{0,1\} } \rightarrow \Delta^{ \{0,1\} }$ is a Cartesian fibration, so the $\infty$-category
$Y_0 \times_{Y} Y_{/y}$ has a final object, given by a locally $\pi$-Cartesian morphism
$f: y' \rightarrow y$. It follows that $\theta$ is a $\phi$-colimit diagram if and only if
$F(f)$ is a $\phi$-coCartesian morphism in $X$. Criterion $(3)$ guarantees that
$F(f)$ is locally $\psi$-coCartesian, which implies that $F(f)$ is $\psi$-coCartesian by virtue of assumption $(d)$.
\end{proof}

\begin{proposition}\label{skitter}
Suppose we are given a diagram of $\infty$-categories $X \stackrel{\phi}{\rightarrow} Y \stackrel{\pi}{\rightarrow} Z$ where $\pi$ is a flat categorical fibration and $\phi$ is a categorical fibration.
Let $\psi: \pi_{\ast} X \rightarrow Z$ denote the projection, let $K$ be an $\infty$-category, let
$\overline{p}_0: K^{\triangleleft} \rightarrow Z$ be a diagram, and assume that the induced map
$\pi': K^{\triangleleft} \times_{Z} Y \rightarrow K^{\triangleleft}$ is a coCartesian fibration.
Let $v$ denote the cone point of $K^{\triangleleft}$, let $\calC = {\pi'}^{-1} \{v\}$, and choose
a map $K^{\triangleleft} \times \calC \rightarrow K^{\triangleleft} \times_{Z} Y$ which is the identity
on $\{v\} \times \calC$ and carries $e \times \{C\}$ to a $\pi'$-coCartesian edge of
$K^{\triangleleft} \times_{Z} Y$, for each edge $e$ of $K^{\triangleleft}$ and each object $C$ of $\calC$.
Then:
\begin{itemize}
\item[$(1)$] Let $\overline{p}: K^{\triangleleft} \rightarrow \pi_{\ast} X$ be a map lifting
$\overline{p}_0$, and suppose that for each $C \in \calC$ the induced map
$$ K^{\triangleleft} \times \{C\} \hookrightarrow
K^{\triangleleft} \times \calC \rightarrow K^{\triangleleft} \times_{Z} Y
\rightarrow X$$
is a $\phi$-limit diagram. Then $\overline{p}$ is a $\psi$-limit diagram.

\item[$(2)$] Suppose that $p: K^{\triangleleft} \rightarrow \pi_{\ast} X$ is a map lifting
$p_0 = \overline{p}_0 | K$, and suppose that for each $C \in \calC$ the induced map
$$ K \times \{C\} \hookrightarrow K \times \calC \rightarrow K \times_{Z} Y
\rightarrow X$$ admits a $\psi$-limit diagram lifting the map
$$K^{\triangleleft} \times \{C\} \hookrightarrow K^{\triangleleft} \times \calC
\rightarrow K^{\triangleleft} \times_{Z} Y \rightarrow Y.$$ 
Then there exists an extension $\overline{p}: K^{\triangleleft} \rightarrow \pi_{\ast} X$
of $p$ lifting $\overline{p}_0$ which satisfies condition $(1)$.
\end{itemize}
\end{proposition}

\begin{proof}
Let $\overline{p}: K^{\triangleleft} \rightarrow \pi_{\ast} X$ satisfy the condition described
in $(1)$. We can identify $\overline{p}$ with a map $\overline{F}: K^{\triangleleft} \times_{Z} Y
\rightarrow X$. In view of Proposition \ref{critter}, it will suffice to show that
$\overline{F}$ is a $\phi$-right Kan extension of $\overline{F} | K \times_{Z} Y$. 
Pick an object $C \in \calC$; we wish to show that $\overline{F}$ is a $\phi$-right Kan extension of
$F$ at $C$. In other words, we wish to show that the map
$$ (K \times_{Z} Y)_{C/}^{\triangleleft} \rightarrow K^{\triangleleft} \times_{Z} Y
\stackrel{\overline{F}}{\rightarrow} X$$
is a $\phi$-limit diagram. Since $\overline{p}$ satisfies $(1)$, it suffices to show that the map
$$ s: K \times \{C\} \rightarrow ((K \times_{Z} Y)_{C/})$$ is the opposite of a cofinal map. It follows from Proposition \toposref{verylonger} that the projection $q: (K^{\triangleleft} \times_{Z} Y)_{C/} \rightarrow K^{\triangleleft}$ is a coCartesian fibration, and that $s$ is a coCartesian section of $q$. To show
that $s^{op}$ is cofinal, it will suffice to show that $s$ admits right adjoint (this follows from Corollary \toposref{hollowtt}). In fact, we will show that the identity map $\id_{K^{\triangleleft}} \rightarrow q \circ s$
exhibits $q$ as a right adjoint to $s$. For this, we must show that for every object
$a \in (K \times_{Z} Y)_{C/}$ and every object $b \in K$, the map
$$\bHom_{ (K \times_{Z} Y)_{C/}}( s(b),a) \rightarrow \bHom_{K}( b, q(a) )$$
is a homotopy equivalence. Let $c = \id_{C}$ denote the initial object of
$(K^{\triangleleft} \times_{Z} Y)_{C/}$, and let $\eta$ denote the unique map from
the cone point $v \in K^{\triangleleft}$ to $b$ in $K^{\triangleleft}$. Using Proposition \toposref{charCart}, we obtain a homotopy pullback diagram
$$ \xymatrix{ \bHom_{ (K^{\triangleleft} \times_{Z} Y)_{C/}}( s(b),a) \ar[d] \ar[r] & \bHom_{  (K^{\triangleleft} \times_{Z} Y)_{C/}}( c,a) \ar[d]^{\theta} \\
\bHom_{ K^{\triangleleft}}( b, q(a) ) \ar[r] & \bHom_{ K^{\triangleleft}}( v, q(a) ).}$$
It therefore suffices to show that the $\theta$ is a homotopy equivalence, which is clear
(both the domain and the codomain of $\theta$ are contractible). This completes the proof
of $(1)$.

We now prove $(2)$. The diagram $p$ gives rise to a map $F: K \times_{Z} Y \rightarrow X$ fitting into a commutative diagram
$$ \xymatrix{ K \times_{Z} Y \ar[r]^{F} \ar[d] & X \ar[d]^{\phi} \\
K^{\triangleleft} \times_{Z} Y \ar[r] \ar@{-->}[ur]^{\overline{F} } & Y. }$$
The above argument shows that a dotted arrow $\overline{F}$ as indicated will correspond to a map $\overline{p}: K^{\triangleleft} \rightarrow \pi_{\ast} X$ satisfying $(1)$ if and only if
$\overline{F}$ is a $\phi$-right Kan extension of $F$. In view of Lemma \toposref{kan2}, the
existence of such an extension is equivalent to the requirement that for each $C \in \calC$, 
the diagram 
$$ (K \times_{Z} Y)_{C/} \rightarrow K \times_{Z} Y
\stackrel{F}{\rightarrow} X$$
can be extended to a $\phi$-limit diagram lifting the map
$$ (K \times_{Z} Y)^{\triangleleft}_{C/} \rightarrow K^{\triangleleft} \times_{Z} Y
\rightarrow Y.$$
This follows from the hypothesis of part $(2)$ together with the fact (established above) that
$s^{op}$ is cofinal.
\end{proof}

\begin{proposition}\label{candied}
Suppose we are given categorical patterns
$\CatP = (M_S, T, \{ p_{\alpha}: K_{\alpha}^{\triangleleft} \rightarrow S\}_{\alpha \in A})$ and
$\CatP' = (M_{S'}, T', \{ p'_{\alpha}: {K'_{\alpha}}^{\triangleleft} \rightarrow S'\}_{\alpha \in A'})$ on $\infty$-categories $S$ and $S'$. Let $\pi: S' \rightarrow S$ be a map satisfying the following conditions:
\begin{itemize}
\item[$(i)$]  For every vertex $s' \in S'$ and every morphism $f: s \rightarrow \pi(s')$ in
$S$ which belongs to $M_{S}$, there exists a locally $\pi$-Cartesian morphism
$\overline{f}: \overline{s} \rightarrow s'$ in $S'$ such that $\pi( \overline{f} ) = f$.
\item[$(ii)$] The map $\pi$ is a flat categorical fibration.
\item[$(iii)$] The map $\pi$ carries $M_{S'}$ into $M_{S}$.
\item[$(iv)$] The collections of morphisms $M_{S}$ and $M_{S'}$ contain all equivalences
and are stable under composition (and are therefore stable under equivalence).
\item[$(v)$] Suppose given a commutative diagram
$$\xymatrix{ & s' \ar[dr]^{g} & \\
s \ar[ur]^{f} \ar[rr]^{h} & & s'' }$$
in $S'$, where $g$ is locally $\pi$-Cartesian, $\pi(g) \in M_{S}$, and $\pi(f)$ is an equivalence. Then
$f \in M_{S'}$ if and only if $h \in M_{S'}$. In particular (taking $f = \id_{s}$), we deduce that
every locally $\pi$-Cartesian morphism $g$ such that $\pi(g) \in M_{S}$ belongs to $M_{S'}$.

\item[$(vi)$] The set of $2$-simplices $T'$ contains $\pi^{-1}(T)$, and $T$ contains
all $2$-simplices $\Delta^2 \rightarrow S$ whose restriction to $\Delta^{ \{0,1\} }$ is an equivalence in $S$.

\item[$(vii)$] Each of the simplicial sets $K_{\alpha}$ is an $\infty$-category, and each of the induced maps $\pi_{\alpha}: K_{\alpha}^{\triangleleft} \times_{S} S' \rightarrow K_{\alpha}^{\triangleleft}$ is a coCartesian fibration.

\item[$(viii)$] Suppose we are given $\alpha \in A$ and a commutative diagram
$$ \xymatrix{ & s' \ar[dr]^{g} & \\
s \ar[ur]^{f} \ar[rr]^{h} & & s'' }$$
in $K_{\alpha}^{\triangleleft} \times_{S} S'$, where
$f$ is $\pi_{\alpha}$-coCartesian and $\pi_{\alpha}(g)$ is an equivalence. Then the image of $g$ in $S'$ belongs to $M_{S'}$
if and only if the image of $h$ in $S'$ belongs to $M_{S'}$. In particular, the image in $S'$ of
any $\pi_{\alpha}$-coCartesian morphism of $K_{\alpha}^{\triangleleft}$ belongs to $M_{S'}$.

\item[$(ix)$] Let $\alpha \in A$, and suppose we are given a map $\overline{p}_{\alpha}: K^{\triangleleft}_{\alpha} \rightarrow S'$ lifting $p_{\alpha}$, such that the corresponding
section of $\pi_{\alpha}$ is $\pi_{\alpha}$-coCartesian. Then $\overline{p}_{\alpha} \simeq p'_{\beta}$ for some $\beta \in A'$. 
\end{itemize}

Let $\pi^{\ast}: \mset{ \CatP} \rightarrow \mset{\CatP'}$ denote the functor
$\overline{X} \mapsto \overline{X} \times_{ (S, M_{S})} (S', M_{S'})$. Then
$\pi^{\ast}$ is a left Quillen functor (with respect to the model structures described in
Theorem \ref{theo}).
\end{proposition}

\begin{proof}
The functor $\pi^{\ast}$ admits a right adjoint $\overline{\pi}_{\ast}$, given by the formula
$\overline{\pi}_{\ast} (X',M') = (X,M)$, where:
\begin{itemize}
\item[$(a)$] The simplicial set $X$ is the full simplicial subset of
$\pi_{\ast} X'$ spanned by those vertices lying over objects $s \in S$ which classify maps
$S'_{s} \rightarrow X'_{s}$ which carry edges of $M_{S'}$ (which belong to $S'_{s}$)
into $M'$.
\item[$(b)$] An edge $e$ of $X$ belongs to $M$ if and only if its image in
$S$ belongs to $M_{S}$, and $e$ classifies a map
$S' \times_{ S} \Delta^1 \rightarrow X'$ which carries the inverse image of
$M_{S'}$ in $S' \times_{S} \Delta^1$ into $M'$. 
\end{itemize}
We wish to prove that the adjoint functors $(\pi^{\ast}, \overline{\pi}_{\ast})$ give
a Quillen adjunction between $\mset{ \CatP}$ and $\mset{ \CatP'}$. 
To prove this, it will suffice to show that $\pi^{\ast}$ preserves cofibrations and
weak equivalences. The case of cofibrations is obvious. To prove that
$\pi^{\ast}$ preserves weak equivalences, consider an arbitrary $\CatP$-equivalence $\overline{Y} \rightarrow \overline{Z}$. We wish to prove that for every $\CatP$-fibered object $\overline{X} \in \mset{ \CatP'}$, the induced map
$$ \bHom_{S'}^{\sharp}( \pi^{\ast} \overline{Z}, \overline{X}') 
\rightarrow \bHom_{S'}^{\sharp}( \pi^{\ast} \overline{Y}, \overline{X} )$$
is a homotopy equivalence. We can identify this with the canonical map
$$ \bHom_{S}^{\sharp}( \overline{Z}, \overline{\pi}_{\ast} \overline{X}')
\rightarrow \bHom_{S}^{\sharp}( \overline{Y}, \overline{\pi}_{\ast} \overline{X}').$$
It will therefore suffice to show that $\overline{\pi}_{\ast} \overline{X}'$ is $\CatP$-fibered.

Write $\overline{X}' = (X', M')$, and let $\overline{\pi}_{\ast} \overline{X}' = (X,M)$ where
$(X,M)$ is described by $(a)$ and $(b)$, and let $p: X' \rightarrow S'$ denote the projection.
Set $W = \pi_{\ast} X'$, so that $X$ can be identified with a full simplicial subset of $W$. 
Let $M_W$ denote the collection of edges $e: \Delta^1 \rightarrow W$ satisfying the following condition:
\begin{itemize}
\item[$(\ast)$] The image of $e$ in $S$ belongs to $M_{S}$, and the edge $e$ classifies a map $S' \times_{S} \Delta^1 \rightarrow X'$ which
carries $\pi_{\Delta^1}$-Cartesian edges of $S' \times_{S} \Delta^1$ into $M'$, where
$\pi_{\Delta^1}: S' \times_{S} \Delta^1 \rightarrow \Delta^1$ denotes the projection.
\end{itemize}
We claim that $M$ is the inverse image of $M_W$ in $X$. To see that $M$ is contained
in this inverse image, it suffices to observe that every locally $\pi$-Cartesian edge of $\pi^{-1} M_S$ belongs
to $M_{S'}$, which follows from $(v)$. Conversely, suppose that $e: x \rightarrow x'$ is an edge of $X$ belonging
to $M_{W}$, and let $e$ classify a map $E: S' \times_{S} \Delta^1 \rightarrow X'$.
We wish to prove that if $f$ is an edge of $S' \times_{S} \Delta^1$ whose image in
$S'$ belongs to $M_{S'}$, then $E(f) \in M'$. If the composite map
$\Delta^1 \stackrel{f}{\rightarrow} S' \times_{S} \Delta^1 \rightarrow \Delta^1$
is not the identity, then the inclusion $E(f) \in M'$ follows from the assumption that
the vertices $x$ and $x'$ belong to $X$. Otherwise, we can factor
$f$ as a composition $f' \circ f''$, where $f''$ is a morphism in
$S' \times_{S} \{0\}$ and $f'$ is $\pi_{\Delta^1}$-Cartesian. Using $(v)$, we see that
the image of $f''$ in $S'$ belongs to $M_{S'}$, so that $E(f'') \in M'$ by virtue of our assumption
that $x \in X$. Condition $(\ast)$ guarantees that $E(f') \in M'$. Using the assumption that
$\overline{X}'$ is $\CatP'$-fibered, we deduce that $E(f) \in M'$ as desired.

We wish to prove that the pair $(X,M)$ is $\CatP$-fibered. For this, we will verify that the map
$q: X \rightarrow S$ satisfies conditions $(1)$, $(2)$, $(3)$, $(4)$, and $(6)$ of Definition \ref{catpat}, together with condition $(5')$ of Remark \ref{ratpat}.

\begin{itemize}
\item[$(1)$] We must show that the map $q: X \rightarrow S$ is an inner fibration.
It follows from Proposition \ref{fibraman} (together with conditions $(iv)$ and $(vi)$) that  $X' \rightarrow S'$ is a categorical fibration. 
Proposition \ref{saeve} and assumption $(ii)$ guarantee that the map $q': W \rightarrow S$ is a categorical fibration, and therefore an inner fibration. Since $X$ is a full simplicial subset of $W$, it follows also that $X \rightarrow S$ is an inner fibration.

\item[$(2)$] For each edge $\Delta^1 \rightarrow S$ belonging to $M_S$, the induced map
$q_{\Delta^1}: X \times_{S} \Delta^1 \rightarrow \Delta^1$ is a coCartesian fibration.
It follows from Corollary \ref{falch} that the map $q'_{\Delta^1}: W \times_{S} \Delta^1 \rightarrow \Delta^1$ is a coCartesian fibration, and that an edge of $W \times_{S} \Delta^1$ is
$q'_{\Delta^1}$-coCartesian if and only if its image in $W$ belongs to $M_{W}$.
To complete the proof, it will suffice to show that if $f: x \rightarrow y$ is
a $q'_{\Delta^1}$-coCartesian morphism in $W \times_{S} \Delta^1$ with nondegenerate
image in $\Delta^1$ and $x \in X \times_{S} \{0\}$, then $y \in X \times_{S} \{1\}$.
We can identify $f$ with a map $F: S' \times_{S} \Delta^1 \rightarrow X'$.
To prove that $y \in X \times_{S} \{1\}$, we must show that for every morphism
$\alpha: t \rightarrow t'$ in $S' \times_{S} \{1\}$ whose image in $S'$ belongs to $M_{S'}$, 
we have $F(\alpha) \in M'$. Form a commutative diagram
$$ \xymatrix{ s \ar[r]^{\beta} \ar[d]^{\alpha'} \ar[dr]^{\gamma} & t \ar[d]^{\alpha} \\
s' \ar[r]^{\beta'} & t' }$$
in $S' \times_{S} \Delta^1$, where $s,s' \in S' \times_{S} \{0\}$ and the horizontal maps
are $\pi_{\Delta^1}$-Cartesian. Condition $(v)$ guarantees that the images of $\beta$ and $\beta'$
in $S'$ belong to $M_{S'}$. Invoking $(iv)$, we deduce that the image of $\gamma$ in $S'$ belongs to $M_{S'}$. Invoking $(v)$ again, the image 
of $\alpha'$ in $S$ belongs to $M_{S'}$. 
Since the image of $f$ in $Y$ belongs to $M_W$, and $x \in X \times_{S} \{0\}$, we
conclude that $F$ carries $\alpha'$, $\beta$, and $\beta'$ into $M'$.

Let
$\sigma: \Delta^2 \rightarrow X'$ be the $2$-simplex 
$$ \xymatrix{ & F(t) \ar[dr]^{F(\alpha)} & \\
F(s) \ar[ur]^{F(\beta)} \ar[rr]^{ F(\gamma) } & & F(t' ). }$$ 
Note that since $(X',M')$ is $\CatP'$-fibered and $F(\alpha'), F(\beta') \in M'$, we have
$F(\gamma) \in M'$. Since the image of this $2$-simplex in $S$ is degenerate, condition $(vi)$ guarantees that its image in $S'$ belongs to $T'$. Because $(X', M')$ is $\CatP'$-fibered, we conclude that the induced map $p': X' \times_{S'} \Delta^2 \rightarrow \Delta^2$ is a coCartesian fibration.
To prove that $F( \alpha) \in M'$, it suffices to show that $F(\alpha)$ is
locally $p$-coCartesian, which is equivalent to the requirement that it is
$p'$-coCartesian when regarded as a morphism of $X' \times_{S'} \Delta^2$.
This follows from Proposition \toposref{protohermes}, since $F(\beta), F(\gamma) \in M'$
implies that $F(\beta)$ and $F(\gamma)$ determine $p'$-coCartesian morphisms
in $X' \times_{S'} \Delta^2$.

\item[$(3)$] A morphism $f$ of $X$ belongs to $M$ if and only if $q(f)$ belongs to $M_S$ and
$f$ is locally $q$-coCartesian. This follows from the proof of $(2)$, since both conditions
are equivalent to the requirement that $f \in M_W$.

\item[$(4)$] Given a commutative diagram
$$ \xymatrix{ \Delta^{ \{0,1\} } \ar[d] \ar[r]^{f} & X \ar[d]^{q} \\
\Delta^2 \ar[r]^{\sigma} & S, }$$
if $f \in M$ and $\sigma \in T$, then $f$ determines an
$q_{\Delta^2}$-coCartesian edge of $X \times_{S} \Delta^2$, where
$q_{\Delta^2}: X \times_{S} \Delta^2 \rightarrow \Delta^2$ denotes the projection map.
In fact, $f$ determines a $q'_{\Delta^2}$-coCartesian edge of $W \times_{S} \Delta^2$,
where $q'_{\Delta^2}: W \times_{S} \Delta^2 \rightarrow \Delta^2$ denotes the projection:
this follows from Corollary \ref{falch}.

\item[$(6)$] For every index $\alpha \in A$ and every coCartesian section
$s$ of the map $q_{\alpha}: X \times_{S} K_{\alpha}^{\triangleleft} \rightarrow K_{\alpha}^{\triangleleft}$,
the map $s$ is a $q$-limit diagram in $X$. To prove this, it will suffice to show that
$s$ is a $q'$-limit diagram in $W$. We will prove this by applying Proposition \ref{skitter}.
Let $s$ classify a map $F: K_{\alpha}^{\triangleleft} \times_{S} S' \rightarrow X'$, and note
that the map $F$ carries the inverse image of $M_{S'}$ into $M'$.
Let $\calC$ denote the fiber of the map $\pi: S' \rightarrow S$ over the image of the cone point
of $K_{\alpha}^{\triangleleft}$. Choose a map $\calC \times K^{\triangleleft}_{\alpha}
\rightarrow K_{\alpha}^{\triangleleft} \times_{S} S' \rightarrow X'$ as in the statement
of Proposition \ref{skitter}. We wish to show that, for each $C \in \calC$, the induced map
$$\theta: K^{\triangleleft}_{\alpha} \times \{C\} \stackrel{\theta_0}{\hookrightarrow} K^{\triangleleft}_{\alpha} \times \calC
\stackrel{\theta_1}{\rightarrow} K_{\alpha}^{\triangleleft} \times_{S} S' \stackrel{F}{\rightarrow} X'$$
is a $p$-limit diagram in $X'$. Let $\pi_{\alpha}: K_{\alpha}^{\triangleleft} \times_{S} S' \rightarrow S'$
denote the restriction of $\pi$. Since $\theta_1 \circ \theta_0$ can be identified with a
$\pi_{\alpha}$-coCartesian section of $\pi_{\alpha}$, condition $(viii)$ and the fact
that $F$ carries the inverse image of $M_{S'}$ into $M'$ guarantee that 
$\theta$ carries every edge of $K^{\triangleleft}_{\alpha}$ into $M'$. Using the assumption that
$(X', M')$ is $\CatP$-fibered and condition $(ix)$, we conclude that $\theta$ is a $p$-limit diagram as desired.

\item[$(5')$] For each $\alpha \in A$ and every coCartesian section $s_0$ of the projection
$X \times_{S} K_{\alpha} \rightarrow K_{\alpha}$, there exists a coCartesian section $s$ of
$X \times_{S} K_{\alpha}^{\triangleleft} \rightarrow K_{\alpha}^{\triangleleft}$ extending
$s_0$. The construction of $s$ amounts to solving a lifting problem
$$ \xymatrix{ S' \times_{S} K_{\alpha} \ar[r]^{F} \ar[d] & X' \ar[d]^{p} \\
S' \times_{S} K^{\triangleleft}_{\alpha} \ar[r] \ar@{-->}[ur]^{\overline{F}} & S'. }$$
Since $s_0$ is a coCartesian section, we have $F(e) \in M'$ for every
edge $e$ of $S' \times_{S} K_{\alpha}$ whose image in $S'$ belongs to
$M_{S'}$. As in the proof of $(6)$, we let $\calC$ denote the fiber of the map $\pi$ over
the image of the cone point of $K^{\triangleleft}_{\alpha}$, and choose
a map $\calC \times K^{\triangleleft}_{\alpha} \rightarrow S' \times_{S} K_{\triangleleft}$ as
in the statement of Proposition \ref{skitter}.
Using condition $(ix)$, the assumption that $(X', M')$ is $\CatP'$-fibered, and
Proposition \ref{skitter}, we conclude that there exists an extension $\overline{F}$
of $F$ such that for each $C \in \calC$, the composite map
$$ \overline{F}_C: \{C \} \times K^{\triangleleft}_{\alpha} \hookrightarrow
\calC \times K^{\triangleleft} \rightarrow S' \times_{S} K^{\triangleleft}_{\alpha}
\stackrel{ \overline{F}}{\rightarrow} X'$$
is a $p$-limit diagram. Invoking again our assumption that $(X',M')$ is $\CatP$-fibered
(and that $\overline{F}_{C}$ carries each edge of $\{C\} \times K_{\alpha}$ into $M'$,
by virtue of $(viii)$), we deduce that $\overline{F}_C$ carries each edge of $K^{\triangleleft}_{\alpha}$ into $M'$. The map $\overline{F}$ corresponds to a section $s$ of the projection
$X \times_{S} K_{\alpha}^{\triangleleft} \rightarrow K_{\alpha}^{\triangleleft}$
extending $s_0$. To complete the verification of $(6')$, it will suffice to show that
$s$ is a coCartesian: in other words, we must show that $\overline{F}(e) \in M'$ 
whenever $e: x \rightarrow y$ is an morphism of $S' \times_{S} K^{\triangleleft}_{\alpha}$ whose image
in $S'$ belongs to $M_{S'}$.

If $x \notin \calC$, then $e$ is a morphism of $S' \times_{S} K_{\alpha}$ so that
$\overline{F}(e) = F(e) \in M'$ as desired. We may therefore assume that $x \in \calC$.
Suppose for the moment that $y \notin \calC$, so that the image of
$y$ in $K_{\alpha}^{\triangleleft}$ is a vertex $y_0 \in K_{\alpha}$. We can factor $e$ as a composition
$e' \circ e''$, where $e''$ is a $\pi_{\alpha}$-coCartesian morphism lying in the image of
the map $\{x\} \times K^{\triangleleft}_{\alpha} \rightarrow S' \times_{S} K^{\triangleleft}_{\alpha}$
and $e'$ is a morphism in the fiber $\{y_0\} \times_{S} S'$. Invoking assumption
$(viii)$, we deduce that $e' \in M_{S'}$, so that $\overline{F}(e') = F(e') \in M'$.
Since $\overline{F}(e'')$ lies in the image of $\overline{F}_{x}$, we conclude that
$\overline{F}(e'') \in M'$. Using the fact that $(X', M')$ is $\CatP$-fibered, we conclude
that $\overline{F}(e) \in M'$, as desired.

We now treat the case where $x,y \in \calC$. Let $\psi$ denote the projection map
$X' \times_{S} K_{\alpha}^{\triangleleft} \rightarrow S' \times_{S} K_{\alpha}^{\triangleleft}$. 
Applying $\overline{F}$ to $e$ yields a morphism $\overline{e}: \overline{x} \rightarrow \overline{y}$ of $X' \times_{S} K_{\alpha}^{\triangleleft}$ with $\psi( \overline{e} ) = e$. Since the image in $S'$ of $e$ belongs to $M_{S'}$, we can factor $\overline{e}$ as a composition $\overline{e}' \circ \overline{e}''$, where
$\overline{e}''$ is locally $\psi$-coCartesian and $\overline{e}'$ is a morphism belonging
to $\psi^{-1} \{x\}$. Using assumption $(vi)$ and Lemma \ref{kisker}, we deduce that
every locally $\psi$-coCartesian morphism is $\psi$-coCartesian provided that its image
in $S'$ belongs to $M_{S'}$; in particular, $\overline{e}''$ is $\psi$-coCartesian.
We wish to prove that $\overline{e}$ is locally $\psi$-coCartesian, which is equivalent to the assertion that $\overline{e}'$ is an equivalence. Choose a $\pi_{\alpha}$-coCartesian
section $\theta$ of the projection $\pi_{\alpha}$ which carries the cone point of
$K^{\triangleleft}_{\alpha}$ to $y$. Since $(X', M')$ is $\CatP$-fibered, the coCartesian fibration
$$X' \times_{S'} K^{\triangleleft}_{\alpha} \rightarrow K^{\triangleleft}_{\alpha}$$
is classified by a limit diagram $\chi: K^{\triangleleft}_{\alpha} \rightarrow \Cat_{\infty}$, so
that $\overline{e}'$ is an equivalence if and only if $\gamma_{!} \overline{e}'$ is an equivalence
in $\psi^{-1} \{y'\}$ for every morphism $\gamma: y \rightarrow y'$ lying in the image of $\theta$.
We have a commutative diagram in $X' \times_{S} K_{\alpha}^{\triangleleft}$
$$ \xymatrix{ \overline{x} \ar[r]^{\overline{e}'} \ar[dr]^{\overline{e}} & \overline{z} \ar[r] \ar[d]^{\overline{e}} & \overline{z}' \ar[d]^{ \gamma_{!} \overline{e}'} \\
& \overline{y} \ar[r] & \overline{y}' }$$
where the horizontal maps are $\psi$-coCartesian. Moreover, the argument of the preceding paragraph shows that the map $\overline{x} \rightarrow \overline{y}'$ is $\psi$-coCartesian. Applying Proposition
\toposref{protohermes}, we deduce that $\gamma_{!} \overline{e}'$ is $\psi$-coCartesian, and
therefore an equivalence because it belongs to a fiber of $\psi$.
\end{itemize}
\end{proof}

\subsection{Scaled Straightening and Unstraightening}\label{bisec3.3}

Our goal in this section is to introduce the scaled version of the straightening
and unstraightening functors of \S \toposref{strsec}. We begin by introducing a bit of notation.

\begin{definition}\label{sputer1}
Let $\overline{X} = (X, M)$ be a marked simplicial set. Let $T$ denote the collection of all
$2$-simplices $\sigma$ of $X \times \Delta^1$ with the following properties:
\begin{itemize}
\item[$(a)$] The image of $\sigma$ in $X$ is degenerate.
\item[$(b)$] Suppose that the projection $\pi: \Delta^2 \stackrel{\sigma}{\rightarrow} X \times \Delta^1 \rightarrow \Delta^1$ satisfies $\pi^{-1} \{0\} = \{0,1\}$. Then the image of $\sigma | \Delta^{ \{0,1\} }$
in $X$ belongs to $M$.
\end{itemize}
We define a scaled simplicial set
$\LCone{\overline{X}}$, the {\it scaled cone} of $X$, by the following formula:
$$\LCone{\overline{X}} = (X \times \Delta^1, T) \coprod_{ (X \times \{0\})_{\flat} } \{v\}_{\flat}.$$

More generally, given a marked simplicial set $\overline{S}$ and an object $\overline{X} \in 
\mset{\overline{S}}$. We let $\LCones{\overline{S}}{\overline{X}}$ denote the coproduct
$$ \LCone{ \overline{X} } \coprod_{ (X \times \{1\})_{\flat} } \overline{S}.$$
\end{definition}

\begin{remark}\label{swite}
The construction $\overline{X} \mapsto \LCone{\overline{X}}$ determines a functor
from $\mSet$ to $(\scSet)_{ \{v\}_{\flat} / }$. This functor preserves monomorphisms and commutes with all colimits. Similarly, if we fix a scaled simplicial set $\overline{S}$, then the construction
$\overline{X} \mapsto \LCones{\overline{S}}{\overline{X}}$ determined a functor from
$\mset{\overline{S}}$ to $(\scSet)_{ (\{v\}_{\flat} \coprod \overline{S})/}$, which again preserves monomorphisms and commutes with all colimits.
\end{remark}

\begin{remark}\label{sagee}
Let $f: \overline{S} \rightarrow \overline{S}'$ be a map of scaled simplicial sets. Composition with
$f$ induces a functor $f_!: \mset{\overline{S}} \rightarrow \mset{ \overline{S}' }$. For every
object $\overline{X} \in \mset{\overline{S}}$, there is a canonical isomorphism
$$ \LCones{\overline{S}}{\overline{X}} \coprod_{ \overline{S} } \overline{S}'
\simeq \LCones{ \overline{S}' }{f_! \overline{X} }.$$
\end{remark}

\begin{definition}\label{sputer2}
Let $\overline{S} = (S,T)$ be a scaled simplicial set and let $\phi: \scCoNerve[ \overline{S} ] \rightarrow \calC$ be a functor between $\mSet$-enriched categories. We define a functor
$\scSt_{\phi}: \mset{ \overline{S} } \rightarrow (\mSet)^{\calC}$ by the formula
$$(\scSt_{\phi} \overline{X})(C) = \bHom_{ \scCoNerve[ \LCones{\overline{S}}{\overline{X}}]
\coprod_{ \scCoNerve[ \overline{S} ] } \calC }(v, C).$$
If $\phi$ is an isomorphism, we will denote the straightening functor
$\scSt_{\phi}$ instead by $\scSt_{\overline{S}}$. 
\end{definition}

The basic formal properties of the scaled straightening functor may be summarized as follows:

\begin{proposition}\label{curp}
Let $\overline{S}$ be a scaled simplicial set, and let $\phi: \scCoNerve[ \overline{S} ] \rightarrow \calC$
be a functor between $\mSet$-enriched categories. Then the straightening functor
$\scSt_{\phi}: \mset{ \overline{S} } \rightarrow (\mSet)^{\calC}$ is a left Quillen functor. Here we
regard $\mset{ \overline{S} }$ as endowed with the $\overline{S}$-marked model structure, and
$(\mSet)^{\calC}$ as endowed with the projective model structure.
\end{proposition}

In particular, the functor $\scSt_{\phi}$ admits a right adjoint, which we will denote by
$\scUn_{\phi}$.

The proof of Proposition \ref{curp} will be given at the end of this section. First, we need to
introduce a bit of additional terminology.

\begin{definition}\label{bicateq}
Let $f: X \rightarrow Y$ be a map of scaled simplicial sets. We will say that
$f$ is a {\it bicategorical equivalence} if the induced map $\scCoNerve[X] \rightarrow \scCoNerve[Y]$ is a weak equivalence of $\mSet$-enriched categories.
\end{definition}

\begin{remark}\label{suskp}
Suppose given a pushout diagram
$$ \xymatrix{ X \ar[r]^{f} \ar[d]^{g} & Y \ar[d]^{f'} \\
X' \ar[r] & Y' }$$
of scaled simplicial sets. Assume that $f$ is a bicategorical equivalence and that either
$f$ or $g$ is a cofibration. Then $f'$ is a bicategorical equivalence. This follows immediately
from Proposition \ref{presus}, since the category $\Cat_{\mSet}$ is left-proper.
\end{remark}

\begin{remark}\label{ilkk}
Every scaled anodyne map between scaled simplicial sets is a bicategorical equivalence.
\end{remark}

\begin{example}\label{swisher}
The composite map
$$ \Delta^{ \{0,2\} }_{\sharp} \subseteq \Delta^2_{\sharp} \rightarrow
\Delta^2_{\sharp} \coprod_{ \Delta^{ \{0,1\}}_{\sharp} } \Delta^0_{\sharp}$$
is a bicategorical equivalence.
\end{example}

For later use, we record the following reformulation of Proposition \ref{twop1}:

\begin{proposition}\label{twop2}
Let $f: S \rightarrow S'$ be a map of simplicial sets. Then $f$ is a categorical equivalence if and only if the induced map $f_{\sharp}: S_{\sharp} \rightarrow S'_{\sharp}$ is a bicategorical equivalence.
\end{proposition}

Our next goal is to prove the following basic result about scaled cone construction:

\begin{proposition}\label{coop}
Let $\overline{S} = (S,T)$ be a scaled simplicial set, let $\overline{X} = (X, M)$ and
$\overline{Y} = (Y,M')$ be objects of $\mset{\overline{S}}$, and let $f: \overline{X} \rightarrow \overline{Y}$ be a $\CatP_{\overline{S}}$-anodyne morphism (here $\CatP_{\overline{S}}$ denotes the categorical
pattern of Example \ref{user}). Then the induced map
$F: \LCones{\overline{S}}{\overline{X}} \rightarrow \LCones{ \overline{S}}{\overline{Y}}$ is
a bicategorical equivalence.
\end{proposition}

\begin{lemma}\label{urplee}
Let $A \subseteq B$ be an inclusion of simplicial sets, and let
$v$ be a vertex of $A$ with the following property: for every simplex
$\sigma$ of $B$ which does not belong to $A$, $v$ is the final vertex of $\sigma$.
Then the inclusion of marked simplicial sets
$$ (( A \times \Delta^1) \coprod_{A \times \{1\} } (B \times \{1\}), M_0)
\subseteq (B \times \Delta^1, M)$$
is marked anodyne. Here $M$ denotes the collection of all degenerate edges
of $B \times \Delta^1$ together with the edge $\{v\} \times \Delta^1$, and 
$M_0$ is defined similarly.
\end{lemma}

\begin{proof}
Working simplex-by-simplex, we can reduce to the case where $B = \Delta^n$,
$A = \bd \Delta^n$, and $v$ is the final vertex of $B$. For $0 \leq i \leq n$, let
$\sigma_i: \Delta^{n+1} \rightarrow \Delta^n \times \Delta^1$ be the simplex given by the map
of partially ordered sets $r: [n+1] \rightarrow [n] \times [1]$ given by the formula
$$ r(j) = \begin{cases} (j, 0) & \text{if } j \leq i \\
(j-1, 1) & \text{if } j > i. \end{cases}$$
Let $X_0 = ( A \times \Delta^1) \coprod_{A \times \{1\} } (B \times \{1\})$, and for
$0 < i \leq n+1$, let $X_{i} = X_0 \cup \sigma_0 \cup \ldots \cup \sigma_{i-1}$.
Finally, let $M_i$ denote the collection of all degenerate edges of $X_i$ together with
$\{v\} \times \Delta^1$. We have inclusions
$$(X_0, M_0) \subseteq (X_1, M_1) \subseteq \ldots \subseteq (X_{n+1}, M_{n+1})
= ( B \times \Delta^1, M).$$
It will therefore suffice to show that each of the inclusions $(X_i, M_i) \subseteq (X_{i+1}, M_{i+1})$ is marked anodyne. This follows from the existence of a pushout diagram
$$ \xymatrix{ \Lambda^{n+1}_{i+1} \ar@{^{(}->}[r] \ar[d] & \Delta^{n+1} \ar[d]^{g} \\
X_{i} \ar@{^{(}->}[r] & X_{i+1}, }$$
together with the observation that $g$ carries $\Delta^{ \{n, n+1\} }$ to
the marked edge $\{v\} \times \Delta^1$ when $i = n$.
\end{proof}

\begin{lemma}\label{prewise}
Let $n \geq 2$. Let $T$ denote the collection of all degenerate simplices of
$\Delta^n$, together with $\Delta^{ \{0,1,n \} }$, and let $T_0$ be the collection of
all elements of $T$ which belong to $\Lambda^n_0 \subseteq \Delta^n$.
Then the inclusion
$$ f: ( \Lambda^n_0, T_0) \coprod_{ \Delta^{ \{0,1\} }_{\flat} } \Delta^0_{\flat}
( \Delta^n, T) \coprod_{ \Delta^{ \{0,1\} }_{\flat} } \Delta^0_{\flat} $$
is a bicategorical equivalence.
\end{lemma}

\begin{proof}
Let $\calC = \scCoNerve[ \Lambda^n_0, T_0]$, let $\calD = \scCoNerve[ \Delta^n, T]$, and let
$F: \calC \rightarrow \calD$ be the $\mSet$-enriched functor induced by the inclusion.
We can identify the objects of $\calC$ and $\calD$ with elements of the set
$\{ 0, \ldots, n \}$. The functor $\scCoNerve[f]$ is bijective on objects, and can be identified
with the restriction of $F$ to the full subcategory of $\calC$ spanned by the objects
$\{0, 2, 3, \ldots, n \}$. It therefore suffices to show that for $i,j \in \{0, 2, 3, \ldots, n \}$, 
the map $\phi: \bHom_{\calC}(i,j) \rightarrow \bHom_{\calD}(i,j)$ is an equivalence of marked simplicial sets. If $i \neq 0$ or $j \neq n$, then the map $\phi$ is an isomorphism. 

We may therefore assume without loss of generality that $i = 0$ and $j=n$. Let
$C$ denote the cube $(\Delta^1)^{n-2}$, let $\bd C$ denote the boundary of $C$, and let
$v = (0, 0, \ldots, 0)$ denote the initial vertex of $C$. Then we can identify $\phi$ with the inclusion
$$ ((C \times \{0\}) \coprod_{ \bd C \times \{0\} } ( \bd C \times \Delta^1), M_0)
\subseteq (C \times \Delta^1, M),$$
where $M$ is the collection of all degenerate edges of $C \times \Delta^1$ together
with the edge $\{v\} \times \Delta^1$, and $M_0$ is defined similarly. The desired conclusion now follows by applying Lemma \ref{urplee} to the morphism $\phi^{op}$.
\end{proof}

\begin{lemma}\label{swww}
Let $\overline{S}$ be a scaled simplicial set, and let
$f: ( \Lambda^n_i)^{\flat} \subseteq (\Delta^n)^{\flat}$ be an inner horn inclusion
in $\mset{\overline{S}}$ (so that $0 < i < n$). Then the induced map $\LCones{\overline{S}}{f}$ is scaled anodyne.
\end{lemma}

\begin{proof}
Set
$$ Z_0 =( \Lambda^n_i \times \Delta^1) \coprod_{ \Lambda^n_i \times \bd \Delta^1}
( \Delta^n \times \bd \Delta^1).$$
Then $F$ is a pushout of a morphism
$F' : ( Z_0, T_0) \rightarrow ( \Delta^n \times \Delta^1, T)$, where
$T$ is the collection of $2$-simplices of Definition \ref{sputer1} together with
all $2$-simplices of $\Delta^n \times \{0\}$, and $T_0$ is the collection of all $2$-simplices
of $Z_0$ which belong to $T$. It will therefore suffice to show that $F'$ is scaled anodyne.

We first define a sequence of $\tau_1, \ldots, \tau_{n-1}$ of $n$-simplices
$$ \Delta^n \hookrightarrow \Delta^{ \{ 0, \ldots, i-1, i+1, \ldots, n \} } \times \Delta^1
\subseteq \Delta^n \times \Delta^1.$$
These simplices can be described by maps of partially ordered sets
$\tau_{k}: [n] \rightarrow [n-1] \times [1]$ given by the formula
$$ \tau_{k}(j) = \begin{cases} (j,0) & \text{if } j < k \\
(j-1, 1) & \text{if } j \geq k. \end{cases}$$
For $1 \leq k \leq n-1$, let $Z_{k} = Z_0 \cup \tau_1 \cup \ldots \cup \tau_{ k} \subseteq \Delta^n \times \Delta^1$, and let $T_{k}$ be the collection of all $2$-simplices of $Z_k$ which belong to $T$.
We claim that each inclusion $(Z_k, T_{k}) \subseteq (Z_{k+1}, T_{k+1})$ is scaled anodyne.
To see this, we observe that there is a pushout diagram
$$ \xymatrix{ \Lambda^{n}_{k+1} \ar[r] \ar[d] & \Delta^n \ar[d]^{g} \\
Z_{k} \ar[r] & Z_{k+1}, }$$
and that $g$ carries $\Delta^{ \{k, k+1, k+2\} }$ to an element of $T$. 

The above argument shows that the inclusion $(Z_0, T_0) \subseteq (Z_{n-1}, T_{n-1})$ is scaled anodyne. To complete the proof, it will suffice to show that the inclusion
$(Z_{n-1}, T_{n-1}) \subseteq ( \Delta^n \times \Delta^1, T)$ is scaled anodyne. To this end,
we introduce a sequence $\sigma_0, \ldots, \sigma_n$ of $(n+1)$-simplices of
$\Delta^n \times \Delta^1$, defined by the maps of partially ordered sets $[n+1] \rightarrow [n] \times [1]$ described by the formula
$$ \sigma_{k}(j) = \begin{cases} (j,0) & \text{if } j \leq k \\
(j-1, 1) & \text{if } j > k. \end{cases}$$
For $n \leq m \leq 2n$, set 
$$Z_{m} = Z_{n-1} \cup \sigma_0 \cup \ldots \cup \sigma_{m-n},$$
and let $T_{m}$ be the collection of all $2$-simplices of $Z_{m}$ which belong to $T$.
Then $(Z_{2n}, T_{2n}) = ( \Delta^n \times \Delta^1, T)$. It will therefore suffice to show that
each inclusion $(Z_{n-1+k}, T_{n-1+k}) \subseteq (Z_{n+k}, T_{n+k})$ is a bicategorical equivalence
for $0 \leq k \leq n$.

For $k < n$ we have a pushout diagram
$$ \xymatrix{ \Lambda^{n+1}_{k+1} \ar[r] \ar[d] & \Delta^{n+1} \ar[d]^{g} \\
Z_{n-1+k} \ar[r] & Z_{n+k}. }$$
The desired result then follows from the observation that $g$ carries
$\Delta^{ \{k, k+1, k+2\} }$ to an element of $T$. If $k=n$, we have instead a pushout diagram
$$ \xymatrix{ \Lambda^{n+1}_{i} \ar[r] \ar[d] & \Delta^{n+1} \ar[d] \\
Z_{2n-1} \ar[r] & Z_{2n}, }$$
where $g$ again carries $\Delta^{ \{i-1, i, i+1 \} }$ to an element of $T$.
\end{proof}

\begin{proof}[Proof of Proposition \ref{coop}]
Since the collection of all morphisms $F$ in $\scSet$ which are both monomorphisms and bicategorical equivalences is weakly saturated, Remark \ref{swite} implies that the collection of all morphisms $f$
satisfying the conclusion of the Proposition is weakly saturated as well. It will therefore suffice to show that the desired result holds when $f$ is one of the generating $\CatP_{\overline{S}}$-anodyne morphisms appearing in Definition \ref{postspunt}. There are five cases to consider:

\begin{itemize}
\item[$(A_0)$] The morphism $f$ is an inclusion
$$( \Lambda^2_1)^{\sharp} \coprod_{ (\Lambda^2_{1})^{\flat}} (\Delta^2)^{\flat}  \subseteq (\Delta^2)^{\sharp}.$$
such that the image of $\Delta^2$ is a thin $2$-simplex of $S$. Let us
identify $\Delta^2 \times \Delta^1$ with the nerve of the partially ordered set depicted below:
$$ \xymatrix{ c \ar[r] & c' \\
b \ar[r] \ar[u] & b' \ar[u] \\
a \ar[r] \ar[u] & a', \ar[u] }$$
and let us identify $2$-simplices of $\Delta^2 \times \Delta^1$ with chains of length three in this partially ordered set. Let $T$ denote the collection of all $2$-simplices of $\Delta^2 \times \Delta^1$ except
for $( a < c < c')$, $(a < b < c')$, and $(a < b' < c')$. Let $T'$ denote the collection of all
$2$-simplices of $\Delta^2 \times \Delta^1$ except for $(a < b < c')$ and $(a < b' < c')$.
The map $F$ is a pushout of the inclusion
$$ ( \Delta^2 \times \Delta^1, T) \subseteq ( \Delta^2 \times \Delta^1, T').$$
In view of Remark \ref{suskp}, it will suffice to show that this inclusion is a bicategorical
equivalence. We have a commutative diagram
$$ \xymatrix{ ( \Delta^2 \times \Delta^1, T) \ar[dr]^{i} \ar[rr] & & ( \Delta^2 \times \Delta^1, T') \ar[dl]^{i'} \\
& (\Delta^2 \times \Delta^1)_{\sharp}. & }$$
By the two-out-of-three property, it will suffice to show that $i$ and $i'$ are bicategorical
equivalences. Since $i'$ is a pushout of $i$, it will suffice to show that $i$ is a bicategorical
equivalence (Remark \ref{suskp} again). We now observe that $i$ is scaled anodyne:
can be obtained as an iterated pushout of three of the morphisms appearing in Remark \ref{slapperB}.

\item[$(A_1)$] The inclusion $Q^{\flat} \subseteq Q^{\sharp}$, where
$Q = \Delta^0 \coprod_{ \Delta^{ \{0,2\} }} \Delta^3 \coprod_{ \Delta^{ \{1,3\} }} \Delta^0$
and the map $Q \rightarrow S$ every $2$-simplex
of $Q$ into $T$. In view of Remarks \ref{suskp} and \ref{sagee}, we can replace
$S$ by the scaled simplicial set $Q_{\sharp}$. The desired result now follows from a simple
explicit computation.


\item[$(B_0)$] The morphism $f$ is an inclusion $\{0\}^{\sharp} \subseteq ( \Delta^1)^{\sharp}$.
In this case, $F$ factors as a composition $F' \circ F''$, where $F''$ is a pushout of
a scaled anodyne morphism of type $(A)$ appearing in Definition \ref{slapper} (with $n=2$) and
$F'$ is a pushout of the bicategorical equivalence of Lemma \ref{prewise} (for $n=2$).

\item[$(C_0)$] The morphism $f$ is an inclusion
$$ (\Lambda^n_0)^{\flat} \coprod_{ (\Delta^{ \{0,1\} })^{\flat} } ( \Delta^{ \{0,1\} })^{\sharp}
\subseteq ( \Delta^n )^{\flat} \coprod_{ ( \Delta^{ \{0,1\} })^{\flat} } ( \Delta^{ \{0,1\} })^{\sharp},$$
where $n > 1$ and $\Delta^{ \{0,1,n\} }$ maps to a thin simplex of $S$. In view of
Remarks \ref{sagee} and \ref{suskp}, we may assume that $S = \Delta^n$.
Let $T$ denote the collection of all $2$-simplices in $\Delta^n \times \Delta^1$ appearing
in Definition \ref{sputer1}, together with
all $2$-simplices of $\Delta^n \times \{0\}$ and the $2$-simplex $\Delta^{ \{0,1,n\} } \times \{1\}$.
Let $Z_0 =( \Lambda^n_0 \times \Delta^1) \coprod_{ \Lambda^n_0 \times \bd \Delta^1}
( \Delta^n \times \bd \Delta^1)$, and let $T_0$ denote the collection of all $2$-simplices of
$Z_0$ which belong to $T$. We wish to show that the map 
$$(Z_0, T_0) \coprod_{ ( \Delta^n \times \{0\} )_{\flat} } \{v\}_{\flat}
\subseteq ( \Delta^n \times \Delta^1, T) \coprod_{ (\Delta^n \times \{0\})_{\flat} } \{v\}_{\flat}$$
is a bicategorical equivalence.

We will define a filtration 
$$(Z_0,T_0) \subseteq (Z_1,T_1) \subseteq \ldots
\subseteq (Z_{2n}, T_{2n}) = (\Delta^n \times \Delta^1, T)$$
and show that the inclusion
$$i_k: (Z_{k-1}, T_{k-1}) \coprod_{ ( \Delta^n \times \{0\} )_{\flat} } \{v\}_{\flat}
\subseteq (Z_k, T_k) \coprod_{ ( \Delta^n \times \{0\} )_{\flat} } \{v\}_{\flat}$$
is a bicategorical equivalence for $1 \leq k \leq 2n$. We first define
a sequence $\tau_1, \ldots, \tau_{n-1}$ of $n$-simplices
$ \Delta^n \hookrightarrow \Delta^n \times \Delta^1$
using the maps of partially ordered sets
$\tau_{k}: [n] \rightarrow [n] \times [1]$ described by the formula
$$ \tau_{k}(j) = \begin{cases} (j+1,0) & \text{if } j < k \\
(j, 1) & \text{if } j \geq k. \end{cases}$$
For $0 \leq k \leq n-1$, let $Z_{k} = Z_0 \cup \tau_1 \cup \ldots \cup \tau_{ k} \subseteq \Delta^n \times \Delta^1$, and let $T_{k}$ be the collection of all $2$-simplices of $Z_k$ which belong to $T$.
The maps $i_{k}$ are scaled anodyne for $1 \leq k \leq n-1$: this follows from the existence of the pushout diagram
$$ \xymatrix{ \Lambda^{n}_{k+1} \ar[r] \ar[d] & \Delta^n \ar[d]^{g} \\
Z_{k} \ar[r] & Z_{k+1}, }$$
where $g$ carries $\Delta^{ \{k, k+1, k+2\} }$ to an element of $T$.

We now introduce a sequence $\sigma_0, \ldots, \sigma_n$ of $(n+1)$-simplices of
$\Delta^n \times \Delta^1$, defined by the maps of partially ordered sets $[n+1] \rightarrow [n] \times [1]$ described by the formula
$$ \sigma_{k}(j) = \begin{cases} (j,0) & \text{if } j \leq k \\
(j-1, 1) & \text{if } j > k. \end{cases}$$
For $n \leq m \leq 2n$, set 
$$Z_{m} = Z_{n-1} \cup \sigma_0 \cup \ldots \cup \sigma_{m-n},$$
and let $T_{m}$ be the collection of all $2$-simplices of $Z_{m}$ which belong to $T$.
We now deduce that $i_{k}$ is scaled anodyne for $n \leq k < 2n$ from the existence of a pushout diagram
$$ \xymatrix{ \Lambda^{n+1}_{k+1} \ar[r] \ar[d] & \Delta^{n+1} \ar[d]^{g} \\
Z_{n-1+k} \ar[r]^{i_k} & Z_{n+k}. }$$
(since $g$ carries $\Delta^{ \{k, k+1, k+2\} }$ to an element of $T$). 

For the map $i_{2n}$, we need to work a bit harder. Let $\gamma$
denote the $2$-simplex of $\Delta^n \times \Delta^1$ spanned by the chain of vertices
$(0,0) < (1,1) < (n,1)$, and let $\gamma'$ denote the $2$-simplex spanned by the chain of vertices
$(0,0) < (1,0) < (n,1)$. We have a commutative diagram
$$ \xymatrix{ (Z_{2n-1}, T_{2n-1}) \coprod_{ ( \Delta^n \times \{0\} )_{\flat} } \{v\}_{\flat} 
\ar[r]^{i_{2n}} \ar[d] & (\Delta^n \times \Delta^1, T) \coprod_{ ( \Delta^n \times \{0\})_{\flat} } \{v\}_{\flat} \ar[d] \\
( Z_{2n-1}, T \cup \{ \gamma, \gamma' \} ) \coprod_{ ( \Delta^n \times \{0\} )_{\flat} } \{v\}_{\flat} \ar[r]^{i'} & (\Delta^n \times \Delta^1, T \cup \{ \gamma, \gamma' \} ) \coprod_{ ( \Delta^n \times \{0\})_{\flat} } \{v\}_{\flat} .}$$
The vertical maps are scaled anodyne (each can be obtained as a pushout of two of the morphisms
appearing in Remark \ref{slapperB}), and therefore bicategorical equivalences.
It follows that $i_{n}$ is a bicategorical equivalence if and only if $i'$ is a bicategorical equivalence.
We now complete the proof by observing that $i'$ is a pushout of the bicategorical equivalence
appearing in Lemma \ref{prewise}.

\item[$(C_1)$]  The morphism $f$ is an inclusion $(\Lambda^n_i)^{\flat} \subseteq (\Delta^n)^{\flat}$, for some $0 < i < n$. In this case, the desired result follows from Lemma \ref{swww}.
\end{itemize}

\end{proof}

We now return to the proof of Proposition \ref{curp}.

\begin{lemma}\label{corp1}
Let $\overline{S}$ be a scaled simplicial set and $\phi: \scCoNerve[\overline{S}] \rightarrow \calC$ be a functor between $\mSet$-enriched categories. Let $f,g: \overline{X} \rightarrow \overline{Y}$ be morphisms in $\mset{\overline{S}}$ which are homotopic in the sense that there exists a map
$h: \overline{X} \times (\Delta^1)^{\sharp} \rightarrow \overline{Y}$ such that
$h | (\overline{X} \times \{0\}^{\sharp}) = f$ and $h | ( \overline{X} \times \{1\}^{\sharp}) = g$.
Then $\scSt_{\phi}(f)$ and $\scSt_{\phi}(g)$ induce the same map in the homotopy category
$\h(\mSet)^{\calC}$. 
\end{lemma}

\begin{proof}
We claim that $\scSt_{\phi}(h)$ is a homotopy from $\scSt_{\phi}(f)$ to $\scSt_{\phi}(g)$. 
To prove this, it suffices to show that the projection $p: \scSt_{\phi}( \overline{X} \times (\Delta^1)^{\sharp})
\rightarrow \scSt_{\phi}( \overline{X})$ is a weak equivalence in $(\mSet)^{\phi}$. The map $p$
has a section $s$, induced by the inclusion $\overline{X} \times \{0\}^{\sharp}
\subseteq \overline{X} \times (\Delta^1)^{\sharp}$. Since this inclusion is an
$\CatP_{\overline{S}}$-anodyne map (Proposition \ref{postprod}), the map
$s$ is a weak equivalence by Proposition \ref{coop}. It follows that $p$ is a weak equivalence as desired.
\end{proof}

\begin{proof}[Proof of Proposition \ref{curp}]
It is easy to see that $\scSt_{\phi}$ preserves colimits and cofibrations. 
The existence of the right adjoint $\scSt_{\phi}$ follows from the adjoint functor theorem.
To complete the proof, it will suffice to show that $\scSt_{\phi}$ preserves weak equivalences.
Choose a weak equivalence $f: \overline{X} \rightarrow \overline{Y}$ in $\mset{ \overline{S} }$;
we wish to show that $\scSt_{\phi}(f)$ is a weak equivalence. Choose an $\CatP_{\overline{S}}$-anodyne map $g: \overline{Y} \rightarrow \overline{Y}'$, where $\overline{Y}'$ is $\overline{S}$-fibered. Proposition \ref{coop} implies that $\scSt_{\phi}(g)$ is a weak equivalence. By the two-out-of-three property, it will suffice to show that
$\scSt_{\phi}(g \circ f)$ is a weak equivalence. Replacing $\overline{Y}$ by $\overline{Y}'$, we may suppose that $\overline{Y}$ is $\overline{S}$-fibered.

The map $f$ admits a factorization
$$ \overline{X} \stackrel{f'}{\rightarrow} \overline{X}' \stackrel{f''}{\rightarrow} \overline{Y}$$
where $f'$ is $\CatP_{\overline{S}}$-anodyne, and $f''$ has the right lifting property with respect to all $\CatP_{\overline{S}}$-anodyne morphisms. Proposition \ref{coop} implies that
$\scSt_{\phi}(f')$ is a weak equivalence. By the two-out-of-three property, it will suffice to prove
that $\scSt_{\phi}(f'')$ is a weak equivalence. We may therefore replace $\overline{X}$ by
$\overline{X}'$, and thereby reduce to the case where $\overline{X}$ satisfies the hypotheses of Proposition \ref{coop}. Invoking Lemma \ref{piner}, we deduce that $f$ admits a simplicial homotopy inverse in $\mset{\overline{S}}$. It follows from Lemma \ref{corp1} that $\scSt_{\phi}(f)$ is a weak equivalence as desired.
\end{proof}

\begin{remark}\label{ba1}
Let $f: \overline{S} \rightarrow \overline{S}'$ be a map of scaled simplicial sets, and let
$$ \Adjoint{f_!}{\mset{\overline{S}}}{\mset{\overline{S}'}}{f^{\ast}}$$ be the Quillen adjunction appearing in Proposition \ref{supper}. For any $\mSet$-enriched functor $\phi: \scCoNerve[\overline{S}'] \rightarrow \calC$, we have isomorphisms of functors
$$ \scSt_{ \phi} \circ f_{!} \simeq \scSt_{ \phi \circ f}$$
$$ f^{\ast} \circ \scUn_{\phi} \simeq \scUn_{ \phi \circ f}.$$
\end{remark}

\begin{remark}\label{ba2}
Let $\overline{S}$ be a scaled simplicial set, and suppose given
$\mSet$-enriched functors
$$ \scCoNerve[ \overline{S} ] \stackrel{ \phi}{\rightarrow} \calC
\stackrel{ \psi}{\rightarrow} \calC'. $$
Let $\psi^{\ast}: (\mSet)^{\calC'} \rightarrow (\mSet)^{\calC}$ be given by composition with
$\psi$, and let $\psi_{!}$ be a left adjoint to $\psi^{\ast}$ (given by left Kan extension along $\psi$). 
Then we have canonical isomorphisms of functors
$$ \psi_{!} \circ \scSt_{\phi} \simeq \scSt_{ \psi \circ \phi }$$
$$ \scUn_{\phi} \circ \psi^{\ast} \simeq \scUn_{ \psi \circ \phi }.$$
\end{remark}

\subsection{Straightening over a Point}\label{bisec3.4}

In this section, we study the straightening functor $\scSt_{\ast}$ associated to the scaled simplicial set $\ast = \Delta^0_{\sharp}$ (and its right adjoint, which we will denote by $\scUn_{\ast}$). We can identify $\scSt_{\ast}$ with a functor from the category $\mSet$ of marked simplicial sets to itself. Note that for every marked simplicial set $\overline{X} = (X,M)$, the collection of vertices of $\scSet_{\ast}( \overline{X} )$ can be identified with the collection of edges of $X$.

There is a natural transformation of functors $\alpha: \scSt_{\ast} \rightarrow \id_{\mSet}$ with the following property: for every marked simplicial set $\overline{X} = (X,M)$, the map
$\alpha( \overline{X} ) \rightarrow \overline{X}$ induces on vertices the restriction map
$\Hom_{\sSet}( \Delta^1, X) \rightarrow \Hom_{\sSet}( \Delta^{\{0\}}, X)$. This property
uniquely determines $\alpha$, since $\alpha$ is determined by the maps
$\alpha( ( \Delta^n)^{\flat} ): \scSt_{\ast} (\Delta^n)^{\flat} \rightarrow (\Delta^n)^{\flat}$
which are in turn determined by the induced maps on vertices. In fact, $\alpha$ is the
{\em unique} natural transformation from $\scSt_{\ast}$ to $\id_{\mSet}$: any other possibility induces (and is determined by) a natural transformation
$\Hom_{\sSet}( \Delta^1, X) \rightarrow \Hom_{\sSet}( \Delta^0, X)$ which
(by Yoneda's lemma) arises from a map of simplicial sets $i: \Delta^{0} \rightarrow \Delta^1$; it therefore suffices to observe that the inclusion $\Delta^{0} \simeq \Delta^{ \{1\} } \subseteq \Delta^1$ does
{\em not} induce a natural transformation from $\scSt_{\ast}$ to $\id_{\mSet}$.

\begin{proposition}\label{uil}
For every marked simplicial set $\overline{X} = (X,M)$, the natural transformation $\alpha$
induces a weak equivalence $\scSt_{\ast} \overline{X} \rightarrow \overline{X}$.
\end{proposition}

\begin{proof}
Let us say that a marked simplicial set $\overline{X}$ is {\em good} if the map
$\alpha_{ \overline{X}}: \scSt_{\ast} \overline{X} \rightarrow \overline{X}$ is a weak equivalence.
The proof proceeds in several steps:
\begin{itemize}
\item[$(a)$] The functor $\scSt_{\ast}$ preserves filtered colimits, and the collection of weak equivalences in $\mSet$ is stable under filtered colimits. Consequently, the collection of good
marked simplicial sets is stable under filtered colimits. It will therefore suffice to show that
$\overline{X}=(X,M)$ is good whenever $X$ is a finite simplicial set.
\item[$(b)$] Suppose given a pushout diagram
$$ \xymatrix{ \overline{X} \ar[r]^{f} \ar[d] & \overline{X}' \ar[d] \\
\overline{Y} \ar[r] & \overline{Y}'. }$$
of marked simplicial sets, where $f$ is a cofibration. If $\overline{X}$, $ \overline{X}'$, and
$\overline{Y}$ are good, then $\overline{Y}'$ is good. This follows from the 
fact that $\scSt_{\ast}$ preserves pushouts and cofibrations, and the fact that the model 
structure on $\mSet$ is left proper.
\item[$(c)$] The marked simplicial sets $( \Delta^0)^{\sharp}$, $(\Delta^1)^{\flat}$, and
$(\Delta^1)^{\sharp}$ are good: this is easy to verify by direct calculation.
\item[$(d)$] We now proceed by induction on the number of nondegenerate $1$-simplices of $X$ which belong to $M$. If this number is not zero, then there is a pushout diagram
$$ \xymatrix{ ( \Delta^1)^{\flat} \ar[r] \ar[d] & ( \Delta^1)^{\flat} \ar[d] \\
\overline{X}_0 \ar[r] & \overline{X}, }$$
where $\overline{X}_0$ has fewer nondegenerate marked edges. Using the inductive hypothesis,
$(c)$, and $(b)$, we conclude that $\overline{X}$ is good. We may therefore reduce to the case
where every marked edge of $X$ is degenerate, so that $\overline{X} = X^{\flat}$.
\item[$(e)$] We now work by induction on the dimension $n$ of $X$ and the number of 
nondegenerate simplices of $X$ of dimension $n$. If $X$ is empty, then $X^{\flat}$ is good and there is nothing to prove. Otherwise, we have a pushout diagram
$$ \xymatrix{ ( \bd \Delta^n)^{\flat} \ar[r] \ar[d] & (\Delta^n)^{\flat} \ar[d] \\
{X'}^{\flat} \ar[r] & X^{\flat}. }$$
The inductive hypothesis implies that $(\bd \Delta^n)^{\flat}$ and ${X'}^{\flat}$ are
good. According to $(b)$, it will suffice to show that $(\Delta^n)^{\flat}$ is good. We may therefore
assume that $X$ is an $n$-simplex. If $n \leq 1$, the desired result follows from $(c)$.
If $n \geq 2$, then we can choose an integer $i$ with $0 < i < n$. We have a commutative diagram
$$ \xymatrix{ \scSt_{\ast} ( \Lambda^n_i)^{\flat} \ar[r] \ar[d] & \scSt_{\ast} (\Delta^n)^{\flat} \ar[d] \\
(\Lambda^n_i)^{\flat} \ar[r] & (\Delta^n)^{\flat}. }$$
The horizontal arrows are weak equivalences (for the upper horizontal map, this follows from
Lemma \ref{swww} and Proposition \ref{presus}), and the left vertical map is a weak equivalence by the inductive hypothesis. It follows that the right vertical map is also a weak equivalence, as desired.
\end{itemize}
\end{proof}

\begin{corollary}\label{eul}
Let $f: \overline{X} \rightarrow \overline{Y}$ be a weak equivalence of marked simplicial sets, regarded
as a morphism in $\mset{\Delta^0_{\flat}}$. Then the induced map
$\LCones{ \Delta^0_{\flat}}{f}$ is a bicategorical equivalence.
\end{corollary}

\begin{proof}
We wish to prove that $\scCoNerve[ \LCones{ \Delta^{0}_{\flat}}{f} ]$ is an equivalence of
$\mSet$-enriched categories. Unwinding the definitions, it suffices to show that the map
$\scSt_{\ast}(f)$ is a weak equivalence. In view of the commutative diagram
$$ \xymatrix{ \scSt_{\ast} \overline{X} \ar[d] \ar[r]^{ \scSt_{\ast}(f)} & \scSt_{\ast} \overline{Y} \ar[d] \\
\overline{X} \ar[r]^{f} & \overline{Y}, }$$ this follows from Proposition \ref{uil}.
\end{proof}

\begin{corollary}\label{ba3}
The Quillen adjunction
$(\scSt_{\ast}, \scUn_{\ast})$ is a Quillen equivalence.
\end{corollary}

\begin{proof}
It will suffice to show that $\scSt_{\overline{S}} = \scSt_{\ast}$ induces an
equivalence from the homotopy category $\h{\mSet}$ to itself. Proposition \ref{uil} implies that this functor is isomorphic to the identity. 
\end{proof}

\begin{corollary}\label{ba4}
Let $\overline{S}$ be a scaled simplicial set, and let $\phi: \scCoNerve[\overline{S} ] \rightarrow \calC$ be an $\mSet$-enriched functor. Assume that $\phi$ is essentially surjective, and let
$\alpha: \calF \rightarrow \calF'$ be a map between fibrant objects of $(\mSet)^{\calC}$. The following conditions are equivalent:
\begin{itemize}
\item[$(1)$] The map $\alpha$ is a weak equivalence in $(\mSet)^{\calC}$.
\item[$(2)$] For every object $C \in \calC$, the induced map
$\calF(C) \rightarrow \calF'(C)$ is a weak equivalence.
\item[$(3)$] For every vertex $s$ of $S$, the map
$$ \{s\}^{\sharp} \times_{S^{\sharp} } \scUn_{\phi}(\calF)
\{s\}^{\sharp} \times_{ S^{\sharp} } \scUn_{\phi}(\calF')$$
is a weak equivalence.
\item[$(4)$] The map $\scUn_{\phi}(\alpha)$ is a weak equivalence in
$\mSet{\overline{S}}$.  
\end{itemize}
\end{corollary}

\begin{proof}
The equivalence of $(1)$ and $(2)$ is obvious. The equivalence of $(3)$ and
$(4)$ follows from Lemma \ref{piner}, since $\scUn_{\phi}(\calF)$ and
$\scUn_{\phi}(\calF')$ are fibrant objects of $\mset{\overline{S}}$. 
Since $\phi$ is essentially surjective, the equivalence of $(2)$ and $(3)$ will follow
from the following:
\begin{itemize}
\item[$(\ast)$] For every vertex $s$ of $\overline{S}$, the map
$\calF( \phi(s) ) \rightarrow \calF'( \phi(s) )$ is a weak equivalence if and only if the map
$ \{s\}^{\sharp} \times_{S^{\sharp} } \scUn_{\phi}(\calF)
\{s\}^{\sharp} \times_{ S^{\sharp} } \scUn_{\phi}(\calF')$ is a weak equivalence.
\end{itemize}
To prove $(\ast)$, we use Remarks \ref{ba1} and \ref{ba2} to reduce to the case
$\overline{S} = \Delta^0_{\sharp}$. The desired result in this case follows from Corollary \ref{ba3}.
\end{proof}

\subsection{Straightening over a Simplex}\label{bisec3.5}

Throughout this section, we let $\calC$ denote the simplicial category
$\sCoNerve[\Delta^n]$, $\calC^{+}$ the $\mSet$-enriched category $\scCoNerve[ \Delta^n_{\flat} ]$, and $\phi: \scCoNerve[ \Delta^n_{\flat} ] \rightarrow \calC^{+}$ the identity map. 
We wish to show that the Quillen adjunction
$(\scSt_{\phi}, \scUn_{\phi})$ is a Quillen equivalence. For this, we need to introduce a bit of notation.

\begin{notation}
Let $T$ denote the collection of all $2$-simplices of $\Delta^{n+1}$ which are either degenerate or contain the final vertex of $\Delta^{n+1}$. We let $\calD^{+}$ denote the $\mSet$-enriched category
$\scCoNerve[ \Delta^{n+1}, T]$, and $\calD = \sCoNerve[ \Delta^{n+1} ]$ its underlying simplicial category. Let $i: \calC \rightarrow \calD$ and $i^{+}: \calC^{+} \rightarrow \calD^{+}$ denote the functors induced by the inclusion $\Delta^{n} \simeq \Delta^{ \{0, 1, \ldots, n \} } \subseteq \Delta^{n+1}$.
Let $D$ denote the object of $\calD^{+}$ corresponding to the final vertex of $\Delta^{n+1}$.

For $0 \leq i \leq n$, we can identify vertices of $\bHom_{\calD}(i, D)$ with
subsets $S \subseteq \{ i, i+1, \ldots, n \}$ which contain $i$. The formula
which assigns to each subset $S$ its largest elements extends uniquely to a map
of simplicial sets $\phi_i: \bHom_{\calD}(i, D) \rightarrow \Delta^n$.

Let $i_{!}: (\sSet)^{\calC} \rightarrow (\sSet)^{\calD}$ denote the functor given by left Kan extension along $i$, and let $\sMap: (\sSet)^{\calC} \rightarrow \sSet$ denote the composition of $i_{!}$ with the functor
$(\sSet)^{\calD} \rightarrow \sSet$ given by evaluation at $D$. Similarly, we define a functor
$\sMap^{+}: (\sSet)^{\calC} \rightarrow \mSet$ as the composition
$$ (\sSet)^{\calC} \stackrel{\flat}{\rightarrow} (\mSet)^{\calC^{+} }
\stackrel{ i^{+}_{!} }{\rightarrow} (\mSet)^{\calD^{+}} \rightarrow \mSet,$$
where the first map carries an object $\calF \in (\sSet)^{\calC}$ to the functor
$\calF^{\flat}: \calC^{+} \rightarrow \mSet$ given by the formula
$\calF^{\flat}(i) = \calF(i)^{\flat}$, and the last map is given by evaluation at $D$.

For every $\calF \in (\sSet)^{\calC}$, we can identify $\sMap(\calF)$ with the underlying simplicial set of the marked simplicial set $\sMap^{+}(\calF)$. We can identify $\sMap(\calF)$ with a quotient of the
disjoint union $\coprod_{ 0 \leq i \leq n} \calF(i) \times \bHom_{\calD}(i,D)$. The maps
$\{ \phi_i \}_{0 \leq i \leq n}$ determine a map of simplicial sets
$\sMap(\calF) \rightarrow \Delta^n$. This map depends functorially
on $\calF$; we may therefore view $\sMap$ as defining a functor from
$(\sSet)^{\calC}$ to $(\sSet)_{/ \Delta^n}$. We will abuse notation by denoting this functor also by $\sMap$. Similarly, we can view $\sMap^{+}$ also as a functor from
$(\sSet)^{\calC}$ to $\mset{ \Delta^n_{\flat} }$. 
\end{notation}

\begin{remark}
For every object $\calF \in (\sSet)^{\calC}$ and every $0 \leq i \leq n$, there is a canonical isomorphism
$\sMap(\calF) \times_{ \Delta^n } \{i\} \simeq \calF(i)$. Moreover, the marking on this simplicial set
provided by $\sMap(\calF)^{+}$ is trivial: only degenerate edges of $\calF(i)$ are marked.
\end{remark}

\begin{proposition}\label{pikkle}
Let $\calF \in (\sSet)^{\calC}$ and let $f: \sMap^{+}(\calF) \rightarrow \overline{X} = (X,M)$ be a morphism
in $\mset{\Delta^n_{\flat}}$. Suppose that $\overline{X}$ is $\Delta^n_{\flat}$-fibered, and suppose that for $0 \leq i \leq n$, the map
$$\calF(i) \simeq \sMap(\calF) \times_{\Delta^n} \{i\}
\rightarrow X_i \times_{ \Delta^n } \{i\}$$
is a categorical equivalence of simplicial sets. Then:
\begin{itemize}
\item[$(1)$] The induced map $\sMap(\calF) \rightarrow X$ is a categorical equivalence of simplicial sets.
\item[$(2)$] The map $f$ is a weak equivalence in $\mset{ \Delta^n_{\flat} }$.
\end{itemize}
\end{proposition}

\begin{proof}
We first prove $(1)$ using induction on $n$. If $n=0$ the result is obvious, so we may suppose $n > 0$.
Let $\calC' = \sCoNerve[ \Delta^{n-1} ]$, regarded as a full subcategory of $\calC$, and let
$\calF' = \calF | \calC'$. Unwinding the definition, we have a canonical isomorphism of simplicial sets
$$ \alpha: \sMap(\calF) \simeq (\sMap(\calF') \times \Delta^1) \coprod_{ \sMap(\calF') \times \{1\} } \calF(n)$$
Let $q$ denote the composition $X \rightarrow \Delta^n \stackrel{q_0}{\rightarrow} \Delta^1$, where
$q_0^{-1} \{0\} = \Delta^{n-1} \subseteq \Delta^n$. The map $q$ is a coCartesian fibration of simplicial sets. Assertion $(1)$ follows from the isomorphism $\alpha$, the inductive hypothesis, and
Proposition \toposref{qequiv}.

To prove $(2)$, choose an arbitrary $\Delta^{n}_{\flat}$-fibered object $\overline{Y} = (Y,M')\in \mset{ \Delta^n_{\flat} }$. We observe that there is a commutative diagram
$$ \xymatrix{ \bHom^{\flat}_{ \Delta^n_{\flat}}( \overline{X}, \overline{Y} ) \ar[r] \ar[d] & 
\bHom^{\flat}_{ \Delta^n_{\flat} }( \sMap^{+}(\calF), \overline{Y} ) \ar[d] \\
\Fun( X, Y ) \ar[r] & \Fun( \sMap(\calF), Y ). }$$
Since $Y$ is an $\infty$-category, assertion $(1)$ guarantees that the lower horizontal map is an equivalence. We now complete the proof by observing that the vertical maps are the inclusions of full subcategories which are stable under equivalence, and that the diagram is a pullback square.
\end{proof}

\begin{proposition}\label{unple}
Let $\overline{X} = (X,M)$ be a $\Delta^{n}_{\flat}$-fibered object in $\mset{ \Delta^n_{\flat} }$. Then there exists a strongly cofibrant diagram $\calF \in (\sSet)^{\calC}$ and
a map $f: \sMap(\calF) \rightarrow \overline{X}$ which satisfies the hypotheses of Proposition
\ref{pikkle}.
\end{proposition}

\begin{proof}
The proof goes by induction on $n$. The result is obvious if $n=0$, so assume $n > 0$.
Let $\overline{X}' = \overline{X} \times_{ (\Delta^n)^{\sharp} } ( \Delta^{n-1} )^{\sharp}$, and let
$\calC' = \sCoNerve[ \Delta^{n-1}]$. The inductive hypothesis guarantees the existence of a strongly cofibrant diagram $\calF' \in (\sSet)^{\calC'}$ and a map $f': \sMap^{+}(\calF') \rightarrow \overline{X}'$
satisfying the hypotheses of Proposition \ref{pikkle}. Let $q: X \rightarrow \Delta^1$ be defined
as in the proof of Proposition \ref{pikkle}. Then $q$ is a coCartesian fibration, and $f'$ determines a map
of simplicial sets $h_0: \sMap(\calF') \times \{0\} \rightarrow X \times_{ \Delta^1 } \{0\}$. We can therefore choose a $q$-coCartesian extension of $h_0$ to a map $h: \sMap( \calF') \times \Delta^1 \rightarrow X$, 
where $h | \sMap(\calF') \times \{1\}$ determines a map $h_1: \sMap(\calF') \rightarrow X \times_{ \Delta^n} \{n\}$.
Choose a factorization of $h_1$ as a composition
$$ \sMap(\calF') \stackrel{g'}{\rightarrow} Y \stackrel{g''}{\rightarrow} X \times_{ \Delta^n} \{n\},$$
where $g'$ is a cofibration of simplicial sets and $g''$ is a categorical equivalence. The
map $g'$ determines an extension of $\calF'$ to a functor $\calF \in (\sSet)^{\calC}$ with
$\calF(n) = Y$, and the maps $h$ and $g''$ can be amalgamated to a map of marked
simplicial sets $\sMap^{+}(\calF) \rightarrow \overline{X}$ with the desired properties.
\end{proof}

\begin{remark}\label{scootnerve}
Let $P$ be a partially ordered set. The simplicial category
$\calE = \sCoNerve[ \Nerve(P) ]$ can be described as follows:
\begin{itemize}
\item[$(a)$] The objects of $\calE$ are the elements of $P$.
\item[$(b)$] Given elements $i,j \in P$, the simplicial set
$\bHom_{ \calE}(i,j)$ can be identified with the nerve $\Nerve C_{i,j}$. Here
$C_{i,j}$ denotes the collection of all linearly ordered subsets $S \subseteq P$ with
least element $i$ and largest element $j$, regarded as a partially ordered set
with respect to inclusions.
\item[$(c)$] The composition $\bHom_{\calE}(i,j) \times \bHom_{\calE}(j,k)
\rightarrow \bHom_{\calE}(i,k)$ is induced by the union map
$$ C_{i,j} \times C_{j,k} \rightarrow C_{i,k}$$
$$ (S, S') \mapsto S \cup S'.$$
\end{itemize}
\end{remark}

\begin{lemma}\label{carp}
Let $\overline{S} = (S,T)$ be a scaled simplicial set, and let $A \subseteq B$ be an inclusion of simplicial sets. Let $f: \Delta^1 \times B \rightarrow S$ be a map with the following
properties:
\begin{itemize}
\item For every simplex $\sigma: \Delta^n \rightarrow B$ which does not belong to $A$
and let $\tau$ be the $2$-simplex of $\Delta^1 \times \Delta^n$ spanned by
$(0,0)$, $(1,0)$ and $(1,n)$. Then the induced map
$$ \Delta^2 \stackrel{\tau}{\rightarrow} \Delta^1 \times \Delta^n
\stackrel{\sigma}{\rightarrow} \Delta^1 \times B \stackrel{f}{\rightarrow} S$$
is a thin $2$-simplex of $S$.
\end{itemize}
Then the inclusion 
$$((\Delta^1)^{\sharp} \times A^{\flat}) \coprod_{ \{0\}^{\sharp} \times A^{\flat} }
(\{0\}^{\sharp} \times B^{\flat}) \subseteq (\Delta^1)^{\sharp} \times B^{\flat}$$
is $\CatP_{\overline{S}}$ anodyne (where $\CatP_{\overline{S}}$ is the categorical
pattern of Example \ref{user}).
\end{lemma}

\begin{proof}
Working simplex-by-simplex, we can reduce to the case where
$B = \Delta^n$ and $A = \bd \Delta^n$. The simplicial
set $\Delta^1 \times \Delta^n$ admits a filtration
$$ (\{0\} \times \Delta^n) \coprod_{ \{0\} \times \bd \Delta^n} ( \Delta^1 \times \bd \Delta^n)
= Z_0 \subset Z_1 \subset \ldots \subset Z_{n} \subseteq Z_{n+1} = \Delta^1 \times \Delta^n,$$
where each $Z_{i+1}$ is obtained from $Z_{i}$ by adjoining the $(n+1)$-simplex of
$\Delta^1 \times \Delta^n$ corresponding to the map
$$ \sigma_i: [n+1] \rightarrow [1] \times [n]$$
$$ \sigma_i(j) = \begin{cases} (0,j) & \text{if } j \leq n-i \\
(1, j-1) & \text{if } j > n-i. \end{cases}$$
Let $\overline{Z_i} = (Z_i, M_i)$ denote the marked simplicial set whose marked edges
are precisely those edges which are marked in $(\Delta^1)^{\sharp} \times (\Delta^1)^{\flat}$. We wish to show that the inclusion
$\overline{Z}_0 \subseteq \overline{Z}_{n+1}$ is $\CatP_{\overline{S} \times \overline{S}'}$ anodyne. For this, it will suffice to show that each of the inclusions $h_i: \overline{Z}_i \subseteq \overline{Z}_{i+1}$ is $\CatP_{\overline{S}}$-anodyne. If $0 \leq i < n$, then
$h_i$ is a pushout of a morphism of type $(C_1)$ appearing in Definition \ref{postspunt}. If $i = n=0$, then $h_i$ is a pushout of a morphism of the type $(B_0)$. 
If $i=n>0$, then $h_i$ is a pushout of a morphism of the type $(C_0)$. 
\end{proof}

\begin{lemma}\label{toughboy}
Let $K$ be a simplicial set, and define $\calF_{K} \in (\sSet)^{\calC}$ by
the formula $\calF_{K}(i) = \bHom_{\calF}(0,i) \times K$, and let $f$ denote the inclusion
$$ K \simeq \sMap(\calF) \times_{ \Delta^n } \{0\} \rightarrow \sMap(\calF).$$
Then $f$ induces a $\CatP_{\Delta^n_{\flat}}$-anodyne morphism
$K^{\flat} \rightarrow \sMap^{+}(\calF)$.
\end{lemma}

\begin{proof}
The proof uses induction on $n$, the case $n=0$ being trivial. Assume that
$n > 0$, and let $\calF' = \calF | \sCoNerve[ \Delta^{n-1} ]$. Using the inductive hypothesis, we deduce that the inclusion $K^{\flat} \subseteq \sMap^{+}(\calF')$ is $\CatP_{\Delta^{n-1}_{\flat}}$-anodyne, and therefore also $\CatP_{\Delta^{n}_{\flat}}$-anodyne. It will therefore suffice to show that the map $g: \sMap^{+}(\calF') \rightarrow \sMap^{+}(\calF)$ is $\CatP_{\Delta^{n}_{\flat}}$-anodyne.
We observe that $g$ is a pushout of the inclusion $\sMap(\calF')^{\flat} \times \{0\}^{\sharp}
\subseteq \sMap( \calF')^{\flat} \times (\Delta^1)^{\sharp}$, which is $\CatP_{\Delta^{n}_{\flat}}$-anodyne by Lemma \ref{carp}.
\end{proof}

\begin{lemma}\label{maincase}
Let $K$ be a simplicial set, and define $\calF_{K} \in (\sSet)^{\calC}$ by the formula
$$ \calF_{K}(i) = \bHom_{\calC}(0,i) \times K.$$
Let $X = \sMap( \calF_{K}) \times_{ \Delta^n} \{n\} \subseteq \sMap(\calF_{K})$. Then the canonical map
$$\alpha_{K}: (\scSt_{\phi} X^{\flat})(n) \rightarrow (\scSt_{\phi} \sMap^{+}(\calF_{K}))(n)$$
is an equivalence of marked simplicial sets.
\end{lemma}

\begin{proof}
For $0 \leq i \leq j \leq n$, let $P_{i,j}$ denote the partially ordered set of all subsets
of $\{ i, i+1, \ldots, j \}$ which contain $i$ and $j$, so that $\bHom_{\calC}(i,j) = \Nerve P_{i,j}$.
Let $P$ denote the collection of all subsets of $[n]$ which contain $0$.
We let $\chi: P \rightarrow [n]$ denote the map which carries a subset
$S \subseteq [n]$ to its largest element. The simplicial set
$\sMap( \calF_{K} )$ can be identified with the product $K \times \Nerve(P)$, and the
projection $\sMap( \calF_{K} ) \rightarrow \Delta^n$ is given by the composition
$$ \sMap(\calF_{K} ) \simeq K \times \Nerve(P) \rightarrow \Nerve(P) \stackrel{ \Nerve(\chi)}{\rightarrow}
\Nerve( [n] ) \simeq \Delta^n.$$

Let $F: \sSet \rightarrow \mSet$ denote the functor
$K \mapsto (\scSt_{\phi} \sMap^{+}(\calF_K))(n)$. We will define a natural transformation of functors
$$\beta_{K}: F(K) \rightarrow \calF(n)^{\flat} = (K \times \Nerve P_{0,n})^{\flat}.$$
Since the functor $F$ commutes with colimits, it will suffice to define $\beta_{K}$ in the case
where $K$ is a simplex; more generally, we will describe $\beta_{K}$ in the case where
$K$ is the nerve of a partially ordered set $Q$. We can then identify 
$\sMap(\calF)$ with the nerve $\Nerve(Q \times P)$. 

Let $\overline{\calE} = \sCoNerve[ Nerve(Q \times P \times [1] ) ]$, and set
$$ \calE = \sCoNerve[ \{v\} \coprod_{ \Nerve(Q \times P \times \{0\}) } \Nerve(Q \times P \times [1] ) \coprod_{\Nerve( Q \times P \times \{1\}} [n] ].$$
The underlying simplicial set of $F(K)$ can be identified with
$\bHom_{\calE}(v, n)$, which is a quotient of the disjoint union
$$ \coprod_{ q, q' \in Q, S \in P, S' \in P_{0, n}}
\bHom_{\overline{\calE}}( (q,S,0), (q', S', 1) ).$$
According to Remark \ref{scootnerve}, we can identify
$\bHom_{ \overline{\calE} }( (q,S,0), (q', S',1) )$ with the nerve of the partially ordered
set of chains
$$ (q,S,0) = (q_0, S_0,0) <
(q_1, S_1,0) < \ldots < (q_i, S_i, 0)
< (q_{i+1}, S_{i+1}, 1) < \ldots < (q_m, S_m, 1) = (q', S', 1)$$
in the partially ordered set $Q \times P \times [1]$. We define the map $\beta_{K}$ so that
the image of such a chain is the pair $(q_i, S_{i} \cup \{ \chi(S_j) \}_{i < j \leq n})
\in Q \times P_{0,n}$. It is not difficult to check that $\beta_{K}$ is well-defined, functorial in $Q$, and
determines a natural transformation as indicated.

For every simplicial set $K$, we have a diagram
$$ K \simeq \sMap(\calF_K) \times_{\Delta^n} \{0\}
\rightarrow \sMap(\calF_K) \leftarrow \sMap(\calF_K) \times_{ \Delta^n} \{n\}
\simeq K \times \Nerve P_{0,n},$$
This diagram determines subfunctors $F_0, F_1 \subseteq F$, given by the formulae
$$ F_0(K) = (\scSt_{\phi} K^{\flat})(n)$$
$$ F_1(K) = (\scSt_{\phi} (K \times \Nerve P_{0,n})^{\flat})(n).$$
We have a commutative diagram of marked simplicial sets
$$ \xymatrix{ F_0(K) \ar[r]^{\gamma_0} \ar[d]^{\gamma_1} & F(K) \ar[d]^{\beta_K} & F_1(K) \ar[dl]^{\gamma_2} \ar[l]^{\alpha_{K}} \\
\scSt_{\ast}(K^{\flat}) \times (\Nerve P_{0,n})^{\flat} \ar[r]^{\gamma_3} & (K \times \Nerve P_{0,n})^{\flat}, & }$$
where:
\begin{itemize}
\item[$(0)$] The map $\gamma_0$ is induced by the inclusion $i_{K}: K^{\flat} \subseteq
\sMap^{+}( \calF_{K} )$. It therefore suffices to show that $i_{K}$ is a
$\CatP_{\Delta^n_{\flat}}$-anodyne morphism of $\mset{ \Delta^n_{\flat} }$., which follows
from Lemma \ref{toughboy}.
\item[$(1)$] The map $\gamma_1$ is the isomorphism
$F_0(K) \simeq \scSt_{\ast}(K^{\flat} ) \times \bHom_{\calC^{+}}(0,n)$ supplied by
Remarks \ref{ba1} and \ref{ba2}.
\item[$(2)$] The map $\gamma_2$ is the composition of the isomorphism
$F_1(K) \simeq \scSt_{\ast} X^{\flat}$ (see Remarks \ref{ba1} and \ref{ba2}) with the weak equivalence
$\scSt_{\ast} X^{\flat} \rightarrow X^{\flat}$ of Proposition \ref{uil}.
\item[$(3)$] The map $\gamma_3$ is the product of the identity map from
$\Nerve( P_{0,n})^{\flat}$ to itself with the weak equivalence
$\scSt_{\ast} K^{\flat} \rightarrow K^{\flat}$ of Proposition \ref{uil}.
\end{itemize}
In particular, the maps $\gamma_0$, $\gamma_1$, $\gamma_2$, and $\gamma_3$ are weak equivalences in $\mSet$. It follows from a diagram chase that $\alpha_{K}$ is also a weak equivalence, as desired.
\end{proof}

\begin{proposition}\label{jurr}
The Quillen adjunction $(\scSt_{\Delta^{n}_{\flat} }, \scUn_{\Delta^{n}_{\flat}})$ is a Quillen equivalence.
\end{proposition}

\begin{proof}
Corollary \ref{ba4} implies that the right derived functor $R \scUn_{\Delta^{n}_{\flat}}$ is conservative. 
It will therefore suffice to show that the unit transformation $\id \rightarrow R \scUn_{\Delta^{n}_{\flat}}
\circ L \scSt_{\Delta^{n}_{\flat}}$ is an isomorphism of functors from the homotopy category
$\h{\mset{\Delta^{n}_{\flat}}}$ to itself. In other words, we must show that for every object
$\overline{X} \in \mset{ \Delta^{n}_{\flat} }$ and every weak equivalence
$\scSt_{ \Delta^{n}_{\flat}} \overline{X} \rightarrow \calF$ in $(\mSet)^{\calC^{+}}$, where
$\calF$ is fibrant, the adjoint map $\overline{X} \rightarrow \scUn_{ \Delta^{n}_{\flat}} \calF$
is also a weak equivalence. Without loss of generality, we may assume that $\overline{X}$ is fibrant, so that $\overline{X}$ is $\Delta^n_{\flat}$-fibered. Choose a strongly cofibrant object
$\calG \in (\sSet)^{\calC}$ and a map $\sMap^{+}(\calG) \rightarrow \overline{X}$
satisfying the hypotheses of Proposition \ref{pikkle} (this is possible, in view of Proposition \ref{unple}). Then the map $\scSt_{\Delta^{n}_{\flat}} \sMap^{+}(\calG) \rightarrow \calF$ is a weak equivalence.
It will therefore suffice to show that the adjoint map $\alpha: \sMap^{+}(\calG) \rightarrow \scUn_{\Delta^{n}_{\flat}}(\calF)$ is a weak equivalence. In view of Proposition \ref{pikkle}, it will suffice to show that
the map $\alpha: \calG(i)_{\flat} \rightarrow \scUn_{\Delta^{n}_{\flat}}(\calF) \times_{ (\Delta^n)^{\sharp}} \{i\}^{\sharp}$ is an equivalence of marked simplicial sets for each $0 \leq i \leq n$. Using Remarks \ref{ba1} and \ref{ba2}, we can replace
$\Delta^{n}$ by $\Delta^{ \{0, \ldots, i \} }$ and thereby reduce to the case $i = n$. 

In view of Remark \ref{ba1} and Corollary \ref{ba3}, it will suffice to show that the map
$\scSt_{\Delta^{n}_{\flat}} \calG(n)^{\flat} \rightarrow \calF$ induces an equivalence
$(\scSt_{\Delta^{n}_{\flat}} \calG(n)^{\flat})(n) \rightarrow \calF(n)$. By the two-out-of-three property, we are reduced to proving the following:

\begin{itemize}
\item[$(\ast)$] Let $\calG$ be a strongly cofibrant object of $(\sSet)^{\calC}$. Then the inclusion
$\calG(n)^{\flat} \rightarrow \sMap^{+}(\calG)$ induces a weak equivalence
$$(\scSt_{ \Delta^{n}_{\flat} } \calG(n)^{\flat})(n) \rightarrow (\scSt_{\Delta^{n}_{\flat}} \sMap^{+}(\calG))(n).$$
\end{itemize}

Let us say that an object $\calG \in (\sSet)^{\calC}$ is {\em good} if the conclusion of $(\ast)$ holds for
$\calG$. The functors $\calG \mapsto (\scSt_{\Delta^{n}_{\flat}} \calG(n)^{\flat})(n)$ and
$\calG \mapsto (\scSt_{\Delta^{n}_{\flat}} \sMap^{+}(\calG))(n)$ both preserve cofibrations and pushout squares. Since the model category $\mSet$ is left-proper, we deduce the following:

\begin{itemize}
\item[$(a)$] Suppose given a pushout diagram
$$ \xymatrix{ \calG_0 \ar[r]^{p} \ar[d] & \calG \ar[d] \\
\calG'_0 \ar[r] & \calG' }$$
in $(\sSet)^{\calC}$, where $\calG_0$, $\calG$, and $\calG'_0$ are good. 
If $p$ is a strong cofibration, then $\calG'$ is also good.
\end{itemize}

Fix a strongly cofibrant object $\calG \in (\sSet)^{\calC}$. For $-1 \leq j \leq n$, let $\calG_{j} \in (\sSet)^{\calC}$ be a left Kan extension of $\calG | \sCoNerve[ \Delta^{ \{0, 1, \ldots, j \} }]$. Then
$\calG_{-1}$ is an initial object of $(\sSet)^{\calC}$, and $\calG_{n} \simeq \calG$. We will prove
by induction on $j$ that each $\calG_{j}$ is good. Let $r^{j}$ denote the inclusion of the object $\{j\}$ into $\calC$, and let $r^{j}_{!}: \sSet \rightarrow (\sSet)^{\calC}$ be the functor of left Kan extension along $r^{j}$. We then have a pushout diagram
$$ \xymatrix{ r^{j}_{!} \calG_{j-1}(j) \ar[r]^{p} \ar[d] & r^{j}_{!} \calG(j) \ar[d] \\
\calG_{j-1} \ar[r] & \calG_{j}. }$$
Since $\calG$ is strongly cofibrant, the map $p$ is a strong cofibration.
Consequently, to show that $\calG_{j}$ is good, it will suffice to show that
$\calG_{j-1}$, $r^{j}_{!} \calG_{j-1}(j)$, and $r^{j}_{!} \calG(j)$ are good. In the first case, this follows from the inductive hypothesis. We are therefore reduced to proving the following:

\begin{itemize}
\item[$(\ast')$] Let $0 \leq j \leq n$, and let $K$ be a simplicial set. Then
$r^{j}_{!} K \in (\sSet)^{\calC}$ is good.
\end{itemize}

Using Remark \ref{ba1}, we can replace $\Delta^{n}$ by $\Delta^{ \{j, j+1, \ldots, n \} }$ and
thereby reduce to the case $j=0$. The desired result now follows from Lemma \ref{maincase}.
\end{proof}

\subsection{Straightening in General}\label{bisec3.6}

Our goal in this section is to prove the following result:

\begin{theorem}\label{cupper}
Let $\overline{S}$ be a scaled simplicial set, and let $\phi: \scCoNerve[ \overline{S} ] \rightarrow \calC$ be an equivalence of $\mSet$-enriched categories. Then the Quillen adjunction $( \scSt_{\phi}, \scUn_{\phi})$ is a Quillen equivalence
from $\mset{ \overline{S} }$ to $(\mSet)^{\calC}$.
\end{theorem}

\begin{remark}\label{jul2}
We regard $\mset{ \overline{S} }$ as endowed with the simplicial model structure described
by Theorem \ref{theo}. We will regard $(\mSet)^{\calC}$ as a simplicial model
category as well, via the simplicial structure on the category $\mSet \simeq \mset{ \Delta^0_{\flat}}$. More concrete, given a simplicial set $K$ and an object $\calF \in (\mSet)^{\calC}$, we can define a new object $\calF \otimes K$ by the formula $(\calF \otimes K)(C)
= \calF(C) \times K^{\sharp}$.

The unstraightening functor $\scUn_{\phi}$ admits the structure of a simplicial functor. For suppose we are given objects $\calF, \calF' \in (\mSet)^{\calC}$ and a map
of simplicial sets $K \rightarrow \bHom_{ (\mSet)^{\calC}}( \calF, \calF' )$, which we may view as a morphism $\alpha: \calF \otimes K \rightarrow \calF'$ in $(\mSet)^{\calC}$. We then have a morphism
$$ \scUn_{\phi} \calF \times K^{\sharp}
\rightarrow \scUn_{\phi} \calF \times \scUn_{\ast} K^{\sharp}
\simeq \scUn_{\phi}( \calF \otimes K)
\rightarrow \scUn_{\phi} \calF',$$
which classifies a map of simplicial sets
$K \rightarrow \bHom_{ \mset{ \overline{S} }}( \scUn_{\phi} \calF,
\scUn_{\phi} \calF' ).$
\end{remark}

\begin{remark}\label{jul1}
Let $\phi: \scCoNerve[ \overline{S} ] \rightarrow \calC$ be an equivalence of
$\mSet$-enriched categories.
Proposition \toposref{lesstrick} implies that the left Kan extension functor
$\phi_{!}: (mSet)^{ \scCoNerve[ \overline{S} ]} \rightarrow (\mSet)^{\calC}$ is
a Quillen equivalence. Using Remark \ref{ba2}, we see that Theorem \ref{cupper}
holds for the functor $\phi$ if and only if it holds for the identity functor
from $\scCoNerve[ \overline{S} ]$ to itself.
\end{remark}

Combining Remarks \ref{jul1}, \ref{jul2}, \toposref{tuccan}, and Proposition
\toposref{weakcompatequiv}, we deduce that Theorem \ref{cupper} is equivalent to the following assertion:

\begin{proposition}\label{tornave}
Let $\overline{S}$ be a scaled simplicial set.
Let $(\mSet)^{\scCoNerve[ \overline{S}], \degree}$ denote the full subcategory
of $(\mSet)^{\scCoNerve[ \overline{S}] }$ spanned by those objects which
are strongly cofibrant and weakly fibrant, and let $\mset{ \overline{S} }^{\degree}$ denote the full subcategory of $\mset{ \overline{S} }$ spanned by the fibrant objects. Then
the functor $\scUn_{\overline{S}}$ induces an equivalence of simplicial categories
from $(\mSet)^{\scCoNerve[ \overline{S}], \degree}$ to $\mset{\overline{S}}^{\degree}$.
\end{proposition}

\begin{lemma}\label{gutta}
Let $i: \overline{S} \rightarrow \overline{S}'$ be a cofibration of scaled simplicial sets,
and let $\overline{X}, \overline{Y} \in \mset{ \overline{S}' }$ be fibrant objects. Then the restriction map
$$ \bHom^{\sharp}_{ \overline{S}' }( \overline{X}, \overline{Y} )
\rightarrow \bHom^{\sharp}_{ \overline{S} }( \overline{X} \times_{ {S'}^{\sharp}}
S^{\sharp}, \overline{Y} \times_{ { S'}^{\sharp} } S^{\sharp} )$$
is a Kan fibration.
\end{lemma}

\begin{proof}
We can identify the right hand side with the mapping space
$\bHom^{\sharp}_{ \overline{S} }( \overline{X} \times_{ {S'}^{\sharp} } S^{\sharp},
\overline{Y} )$. Since $\mset{ \overline{S} }$ is a simplicial model category and
$\overline{Y}$ is fibrant, it suffices to observe that the map
$\overline{X} \times_{ {S'}^{\sharp} } S^{\sharp} \rightarrow \overline{X}$ is a cofibration.
\end{proof}

\begin{lemma}\label{cuppa}
The conclusion of Proposition \ref{tornave} holds when $\overline{S} = \Delta^n_{\flat}$, for $n \geq 0$.
\end{lemma}

\begin{proof}
This is a reformulation of Proposition \ref{jurr}.
\end{proof}

\begin{lemma}\label{culpa}
The conclusion of Proposition \ref{tornave} holds when $\overline{S}$ is the thin
$2$-simplex $\Delta^2_{\sharp}$.
\end{lemma}

\begin{proof}
We have a commutative diagram of simplicial categories
$$ \xymatrix{ 
( \mSet)^{ \scCoNerve[ \Delta^2_{\sharp}], \degree} 
\ar[r]^{ \scUn_{\Delta^2_{\sharp}}} 
\ar[d] & \mset{ \Delta^2_{\sharp}}^{\degree} \ar[d] \\
( \mSet)^{ \scCoNerve[ \Delta^2_{\flat} ], \degree} \ar[r]^{ \scUn_{ \Delta^2_{\flat} }} & \mset{ \Delta^2_{\flat}}^{\degree}. }$$
The vertical arrows are inclusions of full subcategories, and the bottom horizontal arrow is an equivalence by Lemma \ref{cuppa}. It follows that $\scUn_{ \Delta^2_{\sharp} }$ is fully faithful. To complete the proof, we must show that $\scUn_{ \Delta^2_{\sharp} }$ is essentially surjective. Let $\overline{X}$ be a fibrant object of $\mset{ \Delta^2_{\sharp} }$, corresponding to a coCartesian fibration of simplicial sets $p: X \rightarrow \Delta^2$. We wish to prove that $\overline{X}$ lies in the essential image of $\scUn_{ \Delta^2_{\sharp} }$. 
Lemma \ref{cuppa} guarantees that $\overline{X} \simeq \scUn_{ \Delta^2_{\flat} } \calF$
for some fibrant-cofibrant object $\calF: \scCoNerve[ \Delta^2_{\flat} ] \rightarrow \mSet$. 
To prove that $\overline{X}$ lies in the essential image of $\scUn_{ \Delta^2_{\sharp} }$, it will suffice to show that $\calF$ belongs to $( \mSet)^{\scCoNerve[ \Delta^2_{\sharp}}$.
In other words, we must show that the canonical map
$$ \calF(0) \times (\Delta^1)^{\flat} \simeq \calF(0) \times \bHom_{ \scCoNerve[ \Delta^2_{\flat}]}(0,2) \rightarrow \calF(2)$$
factors through $\calF(0) \times ( \Delta^1)^{\sharp} \simeq \calF(0)
\times \bHom_{ \scCoNerve[ \Delta^2_{\sharp}] }(0,2)$. Unwinding the definitions, we see that this is equivalent to the requirement that the collection of locally $p$-coCartesian morphisms
of $X$ is stable under composition, which follows from the assumption that $p$ is a coCartesian fibration.
\end{proof}

\begin{proof}[Proof of Proposition \ref{tornave}]
For every scaled simplicial set $\overline{S}$, let $(\mSet)^{\scCoNerve[\overline{S}]}_{f}$ denote the category of weakly fibrant objects of $(\mSet)^{\scCoNerve[\overline{S}]}$, and let
$W_{\overline{S}}$ be the class of weak equivalences in $( \mSet)^{\scCoNerve[\overline{S}]}_{f}$. Let $W'_{S}$ be the collection of weak equivalences in $\mset{ \overline{S} }^{\degree}$.
We have a commutative diagram of simplicial categories
$$ \xymatrix{ ((\mSet)^{\scCoNerve[S]})^{\degree} \ar[r]^{ \scUn_{\overline{S}} } \ar[d] &
\mset{ \overline{S} }^{\degree} \ar[d]^{\psi_{ \overline{S} }} \\
( \mSet)^{\scCoNerve[\overline{S}]}_{f}[W_{\overline{S}}^{-1} ] \ar[r]^{\phi_{\overline{S}}}& \mset{ \overline{S} }^{\degree}[ {W'}^{-1}_{\overline{S}} ] }$$ (see Notation \toposref{localdef}). 
We wish to prove that the upper horizontal functor is an equivalence of simplicial categories.
Lemma \toposref{kur} implies that the left vertical map is an equivalence.
Using Lemma \toposref{postcuse} and Remark \toposref{uppa}, we deduce that the right vertical map is also an equivalence. Consequently, the Proposition is equivalent to the assertion that $\phi_{ \overline{S} }$ is an equivalence.

Let us say that a scaled simplicial set $\overline{S}$ is {\em good} if the functor
$\phi_{ \overline{S} }$ is an equivalence of simplicial categories. Our goal is to show that every scaled simplicial set is good.

Let $F: \scSet^{op} \rightarrow \sCat$ be the functor given by the formula
$\overline{S} \mapsto ( \mSet)^{\scCoNerve[ \overline{S} ] }_{f}[W_{ \overline{S} }^{-1}]$. 
Using Corollary \toposref{uspin} (and the fact that $\mSet$ is left proper), we deduce the following:
\begin{itemize}
\item[$(a)$] Suppose given a pushout diagram of scaled simplicial sets
$$ \xymatrix{ \overline{S} \ar[r]^{i} \ar[d] & \overline{S}' \ar[d] \\
\overline{T} \ar[r] & \overline{T}', }$$
 where $i$ is a cofibration. Then the induced diagram
$$ \xymatrix{ F( \overline{S} ) & F( \overline{S}' ) \ar[l] \\
F( \overline{T} ) \ar[u] & F( \overline{T}') \ar[u] \ar[l] }$$
is a homotopy pullback square of simplicial categories.
\item[$(b)$] Suppose that a scaled simplicial set $\overline{S}$ is written as a union
of a transfinite sequence $\{ \overline{S}_{\beta} \}_{\beta < \alpha}$ of scaled simplicial subsets. Then $F( \overline{S} )$ is the homotopy limit of the diagram of simplicial
categories $\{ F( \overline{S}_{\alpha} ) \}$. 
\end{itemize} 

We will prove the following:
\begin{itemize}
\item[$(a')$] Suppose given a pushout diagram of scaled simplicial sets
$$ \xymatrix{ \overline{S} \ar[r]^{i} \ar[d] & \overline{S}' \ar[d] \\
\overline{T} \ar[r] & \overline{T}', }$$
where $i$ is a cofibration. If $\overline{S}$, $\overline{S}'$, and $\overline{T}$ is
good, then $\overline{T}'$ is good.
\item[$(b')$] Suppose that a scaled simplicial set $\overline{S}$ is written as a union
of a transfinite sequence $\{ \overline{S}_{\beta} \}_{\beta < \alpha}$ of scaled simplicial subsets. If each $\overline{S}_{\beta}$ is good, then $\overline{S}$ is good.
\end{itemize}

We will prove $(a')$; the proof of $(b')$ is similar. We have a commutative diagram
$$ \xymatrix{ ( \mSet)^{\scCoNerve[\overline{T}']}_{f}[W_{\overline{T}'}^{-1}] \ar[dr]^{\phi_{\overline{T}'}} \ar[ddr] \ar[drr] & & \\
& \mset{ \overline{T}'}^{\degree}[ {W'}_{\overline{T}'}^{-1}] \ar[r]^{u} \ar[d]^{v} \ar[dr]^{w} & \mset{\overline{T}}^{\degree}[ {W'}_{\overline{T}}^{-1}] \ar[d] \\
& \mset{\overline{S}'}^{\degree}[ {W'}_{\overline{S}'}^{-1}] \ar[r] & \mset{\overline{S}}^{\degree}[ {W'}_{\overline{S}}^{-1} ]. }$$
Using $(a)$ and the assumption that $\overline{S}$, $\overline{S}'$, and $\overline{T}$ are good, we deduce that the outer square in this diagram is a homotopy pullback square of
simplicial categories. According to Corollary \toposref{wspin}, the functor
$\phi_{ \overline{T}'}$ is an equivalence of simplicial categories if and onyl if, for every
pair of objects $x,y \in \mset{ \overline{T}'}^{\degree}[ {W'}_{\overline{T'}}^{-1}]$, the diagram of simplicial sets
$$ \xymatrix{ \bHom_{ \mset{\overline{T}'}^{\degree}[ {W'}_{\overline{T}'}^{-1}] }( x, y) \ar[r] \ar[d] 
& \bHom_{ \mset{\overline{T}}^{\degree}[ {W'}_{\overline{T}}^{-1}] }( u(x), u(y) ) \ar[d] \\
\bHom_{ \mset{ \overline{S}'}^{\degree}[ {W'}_{\overline{S}'}^{-1}] }( v(x), v(y) ) \ar[r] &
\bHom_{ \mset{ \overline{S} }^{\degree}[ {W}_{\overline{S}}^{-1}] }( w(x), w(y) ) }$$ 
is homotopy Cartesian. Since $\psi_{\overline{T}'}$ is a weak equivalence of simplicial categories, we may assume without loss of generality that $x = \psi_{Y'} \overline{X} $ and
$y = \psi_{Y'} \overline{Y}$, for a pair of fibrant objects
$\overline{X}, \overline{Y} \in \mset{ \overline{T}'}$. It will therefore suffice to show that
the equivalent diagram
$$ \xymatrix{ \bHom^{\sharp}_{ \overline{T}' }( \overline{X}, \overline{Y}) \ar[r] \ar[d] 
& \bHom^{\sharp}_{\overline{T} }( \overline{X} \times_{ {T'}^{\sharp}} T^{\sharp}, 
\overline{Y} \times_{ {T'}^{\sharp}} T^{\sharp} ) \ar[d] \\
\bHom^{\sharp}_{ \overline{S}' }( \overline{X} \times_{ {T'}^{\sharp}} {S'}^{\sharp}, \overline{Y} \times_{ {T'}^{\sharp} } {S'}^{\sharp} ) \ar[r]^{g} &
\bHom^{\sharp}_{\overline{S} }( \overline{X} \times_{ {T'}^{\sharp} } S^{\sharp}, \overline{Y} \times_{ {T'}^{\sharp} } S^{\sharp} ) }$$ 
is homotopy Cartesian. Here $S, S', T$ and $T'$ denote the underlying simplicial sets of
$\overline{S}$, $\overline{S}'$, $\overline{T}$, and $\overline{T}'$, respectively.
This diagram is a pullback square, and the map $g$ is a Kan fibration by Lemma \ref{gutta}.
This completes the verification of assertion $(a')$.

We now show that every scaled simplicial set is good. The proof proceeds in several steps.
\begin{itemize}
\item[$(i)$] For every $n \geq 0$, the scaled simplicial set $\Delta^{n}_{\flat}$ is good.
This is simply a reformulation of Lemma \ref{cuppa}.
\item[$(ii)$] For every finite simplicial set $S$, the scaled simplicial set
$S_{\flat}$ is good. The proof goes by induction on the dimension $n$ of
$S$ and the number of nondegenerate simplices of $S$. If $S$ is empty, the result is obvious. Otherwise, we have a pushout diagram
$$ \xymatrix{ (\bd \Delta^n)_{\flat} \ar[r] \ar[d] & \Delta^n_{\flat} \ar[d] \\
S'_{\flat} \ar[r] & S_{\flat}. }$$
The inductive hypothesis guarantees that $(\bd \Delta^n)_{\flat}$ and
$S'_{\flat}$ are good, and $\Delta^n_{\flat}$ is good by $(i)$. It follows
from $(a')$ that $S_{\flat}$ is good, as desired.
\item[$(iii)$] Let $S$ be an arbitrary simplicial set. Then $S_{\flat}$ is good.
To prove this, we write $S$ as the union of a transfinite sequence of
simplicial subsets $\{ S_{\beta} \}_{\beta < \alpha}$, where each
$S_{\beta}$ is obtained from $S_{< \beta} = \bigcup_{ \gamma < \beta} S_{\gamma}$ by adjoining a single nondegenerate simplex. In view of $(b')$, it will suffice to show that
each $(S_{\beta})_{\flat}$ is good. We prove this using induction on $\beta$. The inductive hypothesis and $(b')$ imply that $( S_{< \beta})_{\flat}$ is good. We now observe that there is a pushout diagram
$$ \xymatrix{ ( \bd \Delta^n)_{\flat} \ar[r] \ar[d] & \Delta^n_{\flat} \ar[d] \\
(S_{< \beta} )_{\flat} \ar[r] & (S_{\beta})_{\flat}. }$$
Since $(S_{< \beta})_{\flat}$ is good by assumption, 
$\Delta^{n}_{\flat}$ is good by $(i)$, and $( \bd \Delta^n)_{\flat}$ is good by
$(ii)$, we deduce from $(a')$ that $(S_{\beta})_{\flat}$ is good.

\item[$(iv)$] The thin $2$-simplex $\Delta^2_{\sharp}$ is good. This is a reformulation of Lemma \ref{culpa}.

\item[$(v)$] Every scaled simplicial set $\overline{S} = (S,T)$ is good.
To prove this, we let $\{ \sigma_{\beta} \}_{ \beta < \alpha}$ be a well-ordering of
the collection of all nondegenerate thin $2$-simplices of $\overline{S}$.
For $\beta < \alpha$, let $T_{\beta}$ denote the collection of all
degenerate $2$-simplices of $S$ together with $\{ \sigma_{\gamma} \}_{\gamma \leq \beta}$, and let $T_{< \beta}$ be defined similarly. In view of $(b')$, it will suffice to show that
each of the scaled simplicial sets $(S, T_{\beta})$ is good. We prove this by induction on $\beta$. Using the inductive hypothesis and $(b')$ again, we deduce that
$(S, T_{< \beta})$ is good. We now observe that there is a pushout diagram
$$ \xymatrix{ \Delta^2_{\flat} \ar[r] \ar[d] & \Delta^2_{\sharp} \ar[d] \\
(S, T_{< \beta}) \ar[r] & (S, T_{\beta} ). }$$
Since $\Delta^2_{\flat}$ is good by $(i)$ and $\Delta^2_{\sharp}$ is good
by $(iv)$, assertion $(a')$ guarantees that $(S, T_{\beta} )$ is good as desired.
\end{itemize}
\end{proof}

\begin{corollary}\label{stlik}
Let $\overline{S}=(S,T)$ be a scaled simplicial set containing a vertex $y$, and let
$\overline{X}$ be a fibrant object of $\mset{ \overline{S}}$. Then the canonical map
$\scSt ( \overline{X} \times_{ S^{\sharp} } \{y\}^{\sharp} )
\rightarrow (\scSt_{ \overline{S}} \overline{X})(y)$ is an equivalence of marked simplicial sets.
\end{corollary}

\begin{proof}
Choose a weak equivalence $\scSt_{ \overline{S} } \overline{X} \rightarrow \calF$, where
$\calF$ is a fibrant object of $( \mSet)^{ \scCoNerve[ \overline{S} ]}$. 
It follows from Proposition \ref{tornave} that the adjoint map $\overline{X} \rightarrow \scUn_{\overline{S}} \calF$ is a weak equivalence between fibrant objects of $\mset{ \overline{S} }$, so that Lemma \ref{piner} guarantees a weak equivalence of fibers 
$$\overline{X} \times_{ S^{\sharp} } \{y\}^{\sharp}
\rightarrow ( \scUn_{ \overline{S} } \calF) \times_{ S^{\sharp} } \{y\}^{\sharp} \simeq
\scUn \calF(y)$$
in the category of marked simplicial sets. It follows from Corollary \ref{ba3} that the adjoint map
$\scSt (\overline{X} \times_{ S^{\sharp} } \{y\}^{\sharp}) \rightarrow \calF(y)$ is again an equivalence of marked simplicial sets. The desired result now follows from the two-out-of-three property, applied to the diagram
$$ \xymatrix{ \scSt (\overline{X} \times_{ S^{\sharp} } \{y\}^{\sharp}) \ar[rr] \ar[dr] & & \calF(y) \\
& (\scSt_{ \overline{S}} \overline{X})(y). \ar[ur] & }$$
\end{proof}

\section{$\infty$-Bicategories}\label{bisecD}

In \S \ref{bisec3.1} we introduced the category $\scSet$ of {\it scaled simplicial sets}.
Our main goal in this section is to endow $\scSet$ with the structure of a model category
and to show that the underlying homotopy theory is the theory of $(\infty,2)$-categories.
More precisely, we will show that the scaled nerve functor
$\scNerve: \Cat_{\mSet} \rightarrow \scSet$ is a right Quillen equivalence
(Theorem \ref{slai}). The main step in the proof is to show that if
$\overline{\calC}= (\calC,T)$ is a scaled simplicial set satisfying appropriate filling conditions,
then there is a recipe for determining the homotopy type of the mapping objects $\bHom_{ \scCoNerve[ \overline{\calC}]}(x,y)$ directly in terms of $\overline{X}$. In fact, we will show that
$\bHom_{ \scCoNerve[ \overline{\calC}]}(x,y)$ is weakly equivalent to the fiber of a map
of marked simplicial sets $\overline{\calC}^{x/} \rightarrow \calC^{\sharp}$ over the vertex $y$.
The construction of this map will be given in \S \ref{bisec4.1}. We will apply this construction
(in combination with the scaled straightening functor of \S \ref{bisecC}) in
\S \ref{bisec4.2} to give a proof of Theorem \ref{slai}.

It follows from Theorem \ref{slai}, Theorem \ref{castle2}, Proposition \ref{curt}, and
Proposition \ref{camperr} that there is a chain of right Quillen equivalences
$$ \scSet \leftarrow \Cat_{\mSet} \rightarrow \PreSeg{\mSet} \leftarrow
\Fun( \cDelta^{op}, \mSet) \leftarrow \mset{ \Nerve(\cDelta)^{op} }$$
which relates $\scSet$ to the category $\mset{ \Nerve(\cDelta)^{op} }$
(endowed with the complete Segal model structure of Proposition
\ref{camperr}). In \S \ref{bisec4.3}, we will complete the proof of
Theorem \ref{toothygrin} by directly constructing a left Quillen equivalence
$\sdd^{+}: \scSet \rightarrow \mset{ \Nerve(\cDelta)^{op} }$.
Here the functor $\sdd^{+}$ is a close relative of the {\it barycentric subdivision}
functor on simplicial sets. Our proof that $\sdd^{+}$ is a left Quillen equivalence is very indirect,
and makes use of all of the models for $(\infty,2)$-categories appearing in this paper.

The remainder of this section is devoted to describing some applications of 
the subdivision construction of \S \ref{bisec4.3}. In \S \ref{bisec4.4}, we will use it to obtain
a classification for self-equivalences of the $\infty$-category $\Cat_{\infty}$.
In \S \ref{bisec4.5}, we will apply this classification result to obtain a compatibility between
the straightening constructions of \S \ref{bisecC} and \S \toposref{strsec}.

\subsection{The Scaled Slice Construction}\label{bisec4.1}

Let $\calC$ be an $\infty$-category containing an object $x$. Then we can define
a new $\infty$-category $\calC^{x/}$ equipped with a left fibration
$\calC^{x/} \rightarrow \calC$, whose objects are morphisms
$x \rightarrow y$ in $\calC$. Our goal in this section is to describe
an analogous construction in the $(\infty,2)$-categorical setting,
where the simplicial set $\calC$ is replaced by a scaled simplicial set.
Our first step is to formulate some conditions on $\calC$ which will guarantee
that this construction is well-behaved.

\begin{definition}
An {\it weak $\infty$-bicategory} is a scaled simplicial set $(X,T)$ which has the
extension property with respect to every scaled anodyne morphism of $\scSet$.
\end{definition}

\begin{remark}
Let $X$ be a scaled simplicial set, and suppose that {\em every} $2$-simplex of $X$ is thin.
Then $X$ is a weak $\infty$-bicategory if and only if the underlying simplicial set is an $\infty$-category.
The only nontrivial point is to verify that $X$ has the extension property with respect to morphisms of type $(C)$ appearing in Definition \ref{slapper}, which follows from Lemma \toposref{greenlem}.
\end{remark}

\begin{remark}\label{copus}
Let $(X,T)$ be a weak $\infty$-bicategory. Let $X' \subseteq X$ be the simplicial subset spanned
by those simplices $\sigma$ of $X$ such that every two-dimensional face of $\sigma$ is thin.
Then $X'$ is an $\infty$-category. We will refer to $X'$ as the {\em underlying $\infty$-category}
of $(X,T)$.
\end{remark}

\begin{example}
Let $\calC$ be a fibrant $\mSet$-enriched category. Then $\scNerve(\calC)$ is a weak
$\infty$-bicategory.
\end{example}

\begin{notation}\label{caff}
Let $\overline{\calC} = (\calC, T)$ be a scaled simplicial set, and let $X$ be a vertex of $\calC$. 
We can identify edges of the simplicial set $\calC^{X/}$ with diagrams
$$ \xymatrix{ X \ar[d]^{\id} \ar[r] \ar[dr] & Y \ar[d] \\
X \ar[r] & Z }$$
in $\calC$. Let $M$ denote the collection of all edges of $\calC^{X/}$ such that the upper right $2$-simplex in the diagram belongs $T$, and let $M_0 \subseteq M$ be the subset consisting of those edges for which {\em both} nondegenerate $2$-simplices in the diagram belong to $T$.
We let $\overline{\calC}^{X/}$ denote the marked simplicial set $( \calC^{X/}_0, M_0)$, where
$\calC^{X/}_0$ denotes the simplicial subset of $\calC^{X/}$ spanned by those simplices
$\sigma$ for which every edge of $\sigma$ belongs to $M$.

The canonical map $\calC^{X/} \rightarrow \calC$ determines a map from $\calC^{X/}_0$ to $\calC$, so that we can regard $\overline{\calC}^{X/}$ as an object of $\mset{ \overline{\calC} }$. 
\end{notation}

\begin{proposition}\label{ilk}
Let $\overline{\calC} = (\calC, T)$ be a weak $\infty$-bicategory and let $X$ be a vertex of $\calC$. Then
$\overline{\calC}^{X/}$ is a fibrant object of $\mset{ \overline{\calC} }$. 
\end{proposition}

\begin{proof}
The proof proceeds in several steps, each of which amounts to the verification of
a certain extension property. 

\begin{itemize}
\item[$(1)$] The underlying map of simplicial sets $p: \calC^{X/}_0 \rightarrow \calC$ is
an inner fibration: that is, it has the extension property with respect to 
$\Lambda^n_i \subseteq \Delta^n$ for $0 < i < n$. We need to prove the existence of solutions to extension problems of the form
$$ \xymatrix{ K_0 \ar@{^{(}->}[d] \ar[r]^{g_0} & \calC \\
\Delta^n \times \Delta^1. \ar@{-->}[ur]^{g} & }$$
Here $g_0 | \Delta^n \times \{0\}$ is constant at the vertex $X$, and the
restriction of $g_0$ to the $2$-simplex of $\Delta^n \times \Delta^1$ spanned by $(i,0), (i,1)$, and $(j,1)$ is thin, for all $0 \leq i \leq j \leq n$ (except in the case $0 = i < j = n = 2$, which case we must
guarantee that the restriction of $g$ to this simplex is thin).

We first define a sequence of $n$-simplices 
$$\tau_{1}, \ldots, \tau_{n-1}: \Delta^n \rightarrow \Delta^{n-1} \times \Delta^1
\simeq \Delta^{ \{0, \ldots, i-1, i+1, \ldots, n\} } \times \Delta^1 \subseteq \Delta^n \times \Delta^1.$$
On the level of vertices, these simplices are given by maps of partially ordered sets
$\tau_{j}: [n] \rightarrow [n-1] \times [1]$ described by the formula
$$ \tau_{j}(k) = \begin{cases} (k,0) & \text{if } k < j \\
(k-1,1) & \text{if } k \geq j. \end{cases}$$
For $j \leq n-1$, let $K_{j} = K_0 \cup \tau_{1} \cup \ldots \cup \tau_{j}$. 
We prove by induction on $j$ that the function $g_0$ can be extended to a map
$g_{j}: K_j \rightarrow \calC$. To prove this, we observe that there is a pushout diagram
$$ \xymatrix{ \Lambda^{n}_{j} \ar[d] \ar[r] & K_{j-1} \ar[d] \\
\Delta^n \ar[r] & K_j,}$$
and that (if $n > 2$) the composition
$$ \Delta^{ \{j-1,j,j+1\}} \rightarrow \Lambda^n_j \rightarrow K_{j-1} \stackrel{g_{j-1}}{\rightarrow} \calC$$
is a thin $2$-simplex. If $n=2$ (so that $j=1$), we instead choose $g_{j}$ to guarantee that the composition
$$ \Delta^{ \{j-1,j,j+1\}} \simeq \Delta^n \rightarrow K_{j} \stackrel{g_{j}}{\rightarrow} \calC$$
is a thin $2$-simplex.

We now define a sequence of $(n+1)$-simplices 
$$ \sigma_0, \ldots, \sigma_n: \Delta^{n+1} \rightarrow \Delta^n \times \Delta^1.$$
On vertices, these simplices are defined by maps of partially ordered sets
$\sigma_{j}: [n+1] \rightarrow [n] \times [1]$ given by the formulae
$$ \sigma_{j}(k) = \begin{cases} (k,0) & \text{if } k \leq j \\
(k-1,1) & \text{if } k > j. \end{cases}$$
For $0 \leq j \leq n$, let $K_{n+j} \subseteq \Delta^{n} \times \Delta^1$ denote the union
$K_{n-1} \cup \sigma_0 \cup \ldots \cup \sigma_j$, so that we have a chain of inclusions
$$ K_{n-1} \subseteq K_{n} \subseteq \ldots \subseteq K_{2n} = \Delta^n \times \Delta^1.$$
We will prove that $g_{n-1}$ can be extended to a map $g_{n+j}: K_{n+j} \rightarrow \calC$,
using induction on $j$. For $j < n$, it suffices to observe that there is a pushout diagram
$$ \xymatrix{ \Lambda^{n+1}_{j+1} \ar[d] \ar[r] & K_{n+j-1} \ar[d] \\
\Delta^{n+1} \ar[r] & K_{n+j},}$$
and that the composition
$$ \Delta^{ \{j,j+1, j+2\} } \subseteq \Lambda^{n+1}_{j+1}
\rightarrow K_{n+j-1} \stackrel{g_{n+j-1}}{\rightarrow} \calC$$
is a thin $2$-simplex of $\calC$. For $j=n$, we have instead a pushout diagram
$$ \xymatrix{ \Lambda^{n+1}_{i} \ar[r] \ar[d] & K_{2n-1} \ar[d] \\
\Delta^{n+1} \ar[r] & K_{2n} }$$
and the composition
$$ \Delta^{ \{i-1, i, i+1\}} \subseteq \Lambda^{n+1}_{i} \rightarrow K_{2n-1} \stackrel{g_{2n-1}}{\rightarrow} \calC$$
corresponds to the $2$-simplex of $\calC$ which is degenerate at the vertex $X$, and therefore thin.

\item[$(2)$] Every marked edge of $\overline{\calC}^{X/}$ is locally $p$-coCartesian.
To prove this, it suffices to show that $\overline{\calC}^{X/}$ has the extension property with
respect to every inclusion
$$ f: (\Lambda^n_0)^{\flat} \coprod_{ (\Delta^{ \{0,1\} })^{\flat} } ( \Delta^{ \{0,1\} })^{\sharp}
\subseteq ( \Delta^n )^{\flat} \coprod_{ ( \Delta^{ \{0,1\} })^{\flat} } ( \Delta^{ \{0,1\} })^{\sharp},$$
where $n > 1$ and the image of $\Delta^{ \{0,1, n\} }$ in $\calC$ is thin. Let 
$K_0 = ( \Lambda^n_0 \times \Delta^1) \coprod_{ 
\Lambda^n_0 \times \bd \Delta^1} ( \Delta^n \times \bd \Delta^1)$. To show that $\overline{\calC}^{X/}$ has the extension property with respect to $f$, we need to prove the existence of solutions to extension problems of the form
$$ \xymatrix{ K_0 \ar@{^{(}->}[d] \ar[r]^{g_0} & \calC \\
\Delta^n \times \Delta^1. \ar@{-->}[ur]^{g} & }$$
Here $g_0$ has the following properties:
\begin{itemize}
\item[$(a)$] The map $g_0$ carries $\Delta^n \times \{0\}$ to the vertex $X$,
\item[$(b)$] The map $g_0$ carries $\Delta^{ \{0,1,n\} } \times \{1\}$ to a thin $2$-simplex of $\calC$, 
\item[$(c)$] The map $g_0$ carries the $2$-simplex spanned by the vertices $(0,0)$, $(1,0)$, 
and $(1,1)$ to a thin $2$-simplex of $\calC$.
\item[$(d)$] For $0 \leq i \leq j \leq n$, the map $g_0$ carries the $2$-simplex spanned by the vertices
$(i,0)$, $(i,1)$, and $(j,1)$ to a thin $2$-simplex of $\calC$ (except in the case
where $i=1$ and $j=n=2$).
\end{itemize}
Moreover, if $n=2$, then we must guarantee that $g$ carries the $2$-simplex spanned by the vertices
$(1,0)$, $(1,1)$, and $(2,1)$ to a thin $2$-simplex of $\calC$.

We define a sequence of $n$-simplices
$$ \tau_1, \ldots, \tau_{n-1}: \Delta^n \rightarrow
\Delta^n \times \Delta^1,$$
which are given on vertices by the formula
$$ \tau_{i}(j) = \begin{cases} (j+1,0) & \text{if } j < i \\
(j,1) & \text{if } j \geq i. \end{cases}$$
For $i < n$, let $K_i = K_0 \cup \tau_1 \cup \ldots \cup \tau_i \subseteq \Delta^n \times \Delta^1$.
We prove by induction on $i$ that $g_0$ can be extended to a map $g_i: K_i \rightarrow \calC$.
To prove this, we observe the existence of a pushout diagram
$$ \xymatrix{ \Lambda^n_i \ar[r] \ar[d] & K_{i-1} \ar[d] \\
\Delta^n \ar[r] & K_i, }$$
where the composition
$$\Delta^{ \{i-1, i, i+1\} } \subseteq \Lambda^n_i \rightarrow K_{i-1} \stackrel{g_{i-1}}{\rightarrow} \calC$$ is a thin simplex by virtue of assumption $(d)$ (unless $n=2$, in which case we instead choose
$g_{i}$ so that the composition
$\Delta^{ \{i-1, i, i+1\} } \simeq \Delta^n \rightarrow K_i \stackrel{ g_i}{\rightarrow} \calC$ is a thin
$2$-simplex). 

We now define a sequence of $(n+1)$-simplices
$$ \sigma_0, \ldots, \sigma_n: \Delta^{n+1} \rightarrow \Delta^n \times \Delta^1,$$
which are given on vertices by the formulae
$$ \sigma_{i}(j) = \begin{cases} (j,0) & \text{if } j \leq i \\
(j-1, 1) & \text{otherwise.} \end{cases}$$
For $0 \leq i \leq n$, let $K_{n+i} = K_{n-1} \cup \sigma_0 \cup \ldots \cup \sigma_{n}$.
We prove by induction on $i$ that the map $g_{n-1}$ can be extended to a map
$g_{n+i}: K_{n+i} \rightarrow \calC$. If $i < n$, this follows from the pushout diagram
$$ \xymatrix{ \Lambda^{n+1}_{i+1} \ar@{^{(}->} \ar[r] & K_{n+i-1} \ar[d] \\
\Delta^{n+1} \ar[r] & K_{n+i} }$$
together with the observation that the map
$$ \Delta^{ \{i, i+1, i+2 \} } \subseteq \Lambda^{n+1}_{i+1}
\rightarrow K_{n+i-1} \stackrel{g_{n+i-1} }{\rightarrow} \calC$$
is a thin $2$-simplex of $\calC$ (by virtue of assumption $(d)$). If $i=n$, we have instead a pushout diagram
$$ \xymatrix{ \Lambda^{n+1}_0 \ar[r] \ar@{^{(}->}[d] & K_{2n-1} \ar[d] \\
\Delta^{n+1} \ar[r] & K_{2n}. }$$
To prove the existence of the desired extension, it will suffice to show that the composition
$$ \Delta^{ \{0,1\} } \subseteq \Lambda^{n+1}_0 \rightarrow K_{2n-1} \stackrel{g_{2n-1}}{\rightarrow} \calC$$ is degenerate (which follows from $(a)$) and that the composition
$$ \Delta^{ \{0,1, n+1\} } \subseteq \Lambda^{n+1}_0 \rightarrow K_{2n-1} \stackrel{g_{2n-1}}{\rightarrow} \calC$$ is a thin $2$-simplex of $\calC$. To prove the latter result, we apply
Remark \ref{slapperB} to the image in $\calC$ of the $3$-simplex spanned by
the vertices $(0,0)$, $(1,0)$, $(1,1)$, and $(n,1)$. Using assumptions $(c)$ and $(d)$, we are
reduced to proving that the image in $\calC$ of the $2$-simplex spanned by $(0,0)$, $(1,1)$, and $(n,1)$ is thin. This follows by applying Remark \ref{slapperB} to the image of the $3$-simplex spanned by
$(0,0)$, $(0,1)$, $(1,1)$ and $(n,1)$, by virtue of the thinness guaranteed by assumptions
$(b)$ and $(d)$.

\item[$(3)$] For every vertex $v \in \calC^{X/}_0$ and every edge $e: p(v) \rightarrow w$, there
exists a marked edge $\overline{e}: v \rightarrow \overline{w}$ in $\calC^{X/}_0$ lifting $e$
(this edge will automatically be locally $p$-coCartesian, by virtue of $(2)$, so that $p$ is a locally
coCartesian fibration). To prove this, we must show that $\overline{\calC}^{X/}$ has the
extension property with respect to inclusions of the form
$f: \{0\}^{\sharp} \subseteq (\Delta^1)^{\sharp}$. Let
$K = (\Delta^1 \times \bd \Delta^1) \coprod_{ \{0\} \times \bd \Delta^1 } ( \{0\} \times \Delta^1)$.
To prove that $\overline{\calC}^{X/}$ has the extension property with respect to $f$, it suffices to show that every extension problem of the form
$$ \xymatrix{ K_{\sharp} \ar[r]^{g_0} & \overline{\calC} \\
(\Delta^1 \times \Delta^1)_{\sharp} \ar[ur]^{g} & }$$
admits a solution, provided that $g_0$ carries $\Delta^1 \times \{0\} \subseteq K$ to the vertex
$X$ in $\calC$. Let $\sigma$ denote the $2$-simplex of $\Delta^1 \times \Delta^1$ spanned by
the vertices $(0,0)$, $(0,1)$, and $(1,1)$, and let $K'$ denote the union of $K$ with $\sigma$.
The inclusion $K_{\sharp} \subseteq K'_{\sharp}$ is a pushout of an inclusion of type
$(A)$ appearing in Definition \ref{slapper}. Since $\overline{\calC}$ is a weak $\infty$-bicategory, we can extend $g_0$ to a map $g'_0: K'_{\sharp} \rightarrow \overline{\calC}$. Let $\phi$
denote the edge of $\calC$ obtained by applying $g'_0$ to the edge spanned by the vertices
$(0,0)$ and $(1,1)$, and let $\overline{\tau}$ denote the $2$-simplex of $\calC$ defined by the composition
$$ \Delta^2 \stackrel{p}{\rightarrow} \Delta^1 \stackrel{\phi}{\rightarrow} \calC,$$
where $p^{-1} \{1\} = \{2\}$. Let $\tau$ denote
the $2$-simplex of $\Delta^1 \times \Delta^1$ spanned by the vertices $(0,0)$, $(1,0)$, and $(1,1)$.
Then there is a unique extension $g$ of $g'_0$ to $\Delta^1 \times \Delta^1$ such that
$g(\tau) = \overline{\tau}$; it is easy to see that the map $g$ has the desired properties.

\item[$(4)$] Every equivalence appearing in a fiber of the map $p$ is marked in
$\overline{\calC}^{X/}$. To prove this, it suffices to show that $\overline{\calC}^{X/}$ has the
extension property with respect to every inclusion $f: K^{\flat} \subseteq K^{\sharp}$, where
$K$ is a Kan complex and the map $K \rightarrow \calC$ is constant at a vertex $Y \in \calC$.
To prove that $\overline{\calC}^{X/}$ has the extension property with respect to $K$, it suffices to prove the following:
\begin{itemize}
\item Let $g: K \times \Delta^1 \rightarrow \calC$ be a map of simplicial sets. Assume that:
\begin{itemize}
\item[$(i)$] The map $g$ carries $K \times \{0\}$ to the vertex $X$.
\end{itemize}
Then for every edge $e: \Delta^1 \rightarrow K$, the restriction of $g$ to
$\Delta^1 \times \Delta^1$ carrries the $2$-simplex spanned by the vertices
$(0,0)$, $(1,0)$, and $(1,1)$ to a thin $2$-simplex of $\calC$.
\end{itemize}
Since $K$ is a Kan complex, we can choose a simplex $\sigma: \Delta^3 \rightarrow K$
with the following properties:
\begin{itemize}
\item[$(ii)$] The restriction of $\sigma$ to $\Delta^{ \{0,2,3\} }$ is given by the composition
$$ \Delta^{ \{0,2,3\} } \stackrel{p}{\rightarrow} \Delta^1 \stackrel{e}{\rightarrow} K,$$
where $p^{-1} \{1\} = \{3\}$.
\item[$(iii)$] The restriction of $\sigma$ to $\Delta^{ \{1,3\}}$ is degenerate.
\end{itemize}
Let $g'$ denote the map $\Delta^3 \times \Delta^1 \rightarrow \calC$ obtained by composing
$g$ with $\sigma \times \id_{\Delta^1}$. To prove $(iv)$, it will suffice to show that
$g'(\tau)$ is thin in $\calC$, where $\tau$ is the $2$-simplex of $\Delta^3 \times \Delta^1$ spanned by
the vertices $(0,0)$, $(3,0)$, and $(3,1)$. Let $\overline{\tau}$ denote the $4$-simplex of
$\Delta^3 \times \Delta^1$ spanned by the vertices $(0,0)$, $(1,0)$, $(2,0)$, $(3,0)$, and $(3,1)$.
Since $\overline{\calC}$ is a weak $\infty$-bicategory, we can apply part $(B)$ of Definition \ref{slapper}
to the simplex $g(\overline{\tau})$ to reduce to proving the following five assertions:
\begin{itemize}
\item The simplex $g(\tau_0)$ is thin in $\calC$, where $\tau_0$ is the $2$-simplex
of $\Delta^3 \times \Delta^1$ spanned by the vertices $(1,0)$, $(2,0)$, and $(3,0)$. This follows immediately from $(i)$.
\item The simplex $g(\tau_1)$ is thin in $\calC$, where $\tau_1$ is the $2$-simplex
of $\Delta^3 \times \Delta^1$ spanned by the vertices $(0,0)$, $(2,0)$, and $(3,1)$. This follows immediately from $(ii)$.
\item The simplex $g(\tau_2)$ is thin in $\calC$, where $\tau_2$ is the $2$-simplex of
$\Delta^3 \times \Delta^1$ spanned by the vertices $(0,0)$, $(1,0)$, and $(3,0)$. This again follows immediately from $(i)$.
\item The simplex $g(\tau_3)$ is thin in $\calC$, where $\tau_3$ is the $2$-simplex of
$\Delta^3 \times \Delta^1$ spanned by the vertices $(1,0)$, $(3,0)$, and $(3,1)$. This follows immediately from $(iii)$.
\item The simplex $g(\tau_4)$ is thin in $\calC$, where $\tau_4$ is the $2$-simplex of
$\Delta^3 \times \Delta^1$ spanned by the vertices $(0,0)$, $(1,0)$, and $(2,0)$. This follows from $(i)$.
\end{itemize}

\item[$(5)$] The object $\overline{\calC}^{X/} \in \mset{ \overline{\calC} }$ has the extension
property with respect to every inclusion of the form  
$f: ( \Lambda^2_1)^{\sharp} \coprod_{ (\Lambda^2_{1})^{\flat}} (\Delta^2)^{\flat}  \subseteq (\Delta^2)^{\sharp}$, provided that $\Delta^2$ maps to a thin
$2$-simplex of $\calC$. To prove that $\overline{\calC}^{X/}$ has the extension property with respect to $f$, it will suffice to prove the following:
\begin{itemize}
\item[$(\ast)$] Let $T$ denote the collection of all degenerate $2$-simplices of
$\Delta^2 \times \Delta^1$, together with the following nondegenerate $2$-simplices:
\begin{itemize}
\item The simplex $\sigma_0$ spanned by the vertices $(0,0)$, $(1,0)$ and $(2,0)$.
\item The simplex $\sigma_1$ spanned by the vertices $(0,1)$, $(1,1)$, and $(2,1)$.
\item The simplex $\sigma_2$ spanned by the vertices $(0,0)$, $(0,1)$, and $(1,1)$.
\item The simplex $\sigma_3$ spanned by the vertices $(0,0)$, $(0,1)$, and $(2,1)$.
\item The simplex $\sigma_4$ spanned by the vertices $(1,0)$, $(1,1)$, and $(2,1)$.
\item The simplex $\sigma_5$ spanned by the vertices $(0,0)$, $(1,0)$ and $(1,1)$.
\item The simplex $\sigma_6$ spanned by the vertices $(1,0)$, $(2,0)$, and $(2,1)$.
\end{itemize}
Let $g: ( \Delta^2 \times \Delta^1, T) \rightarrow \overline{\calC}$ be a map of scaled simplicial sets.
Then $g(\tau)$ is a thin simplex of $\calC$, where $\tau$ denotes the $2$-simplex spanned by the vertices $(0,0)$, $(2,0)$, and $(2,1)$.
\end{itemize}
To see this, we argue as follows. Since $g(\sigma_1)$, $g(\sigma_2)$, and $g(\sigma_3)$ are
thin, $g(\tau_0)$ is thin, where $\tau_0$ is spanned by the vertices $(0,0)$, $(1,1)$, and $(2,1)$.
Since $g(\tau_0)$, $g(\sigma_5)$ and $g(\sigma_4)$ are thin, $g(\tau_1)$ is thin, where
$\tau_1$ is spanned by the vertices $(0,0)$, $(1,0)$, and $(2,1)$. Finally, since
$g(\tau_1)$, $g(\sigma_0)$, and $g(\sigma_6)$ are thin, $g(\tau)$ is thin as desired.

\item[$(6)$] Every locally $p$-coCartesian edge $e: v \rightarrow w$ in $\calC^{X/}_{0}$ is marked.
To prove this, we begin by choosing a marked edge $e': v \rightarrow w'$ lifting
$p(e)$ (invoking $(3)$). Since $e'$ is locally $p$-coCartesian, we can find a $2$-simplex
$$ \xymatrix{ & w' \ar[dr]^{e''} & \\
v \ar[ur]^{e'} \ar[rr]^{e} & & w }$$
in $\calC^{X/}_0$. Since $e$ and $e'$ are both locally $p$-coCartesian (by $(2)$), the edge
$e''$ is an equivalence in the fiber $p^{-1} \{ p(w) \}$. Invoking $(4)$, we conclude that
$e''$ is marked. It follows from $(5)$ that $e$ is marked, as desired.

\item[$(7)$] The restriction of $p$ to every thin $2$-simplex $\sigma$ of $S$ is a coCartesian fibration.
This follows immediately from $(2)$, $(5)$, $(6)$, and Proposition \toposref{gotta}.
\end{itemize}
\end{proof}

\begin{lemma}\label{kappus}
Let $\overline{S} = (S,T)$ be a scaled simplicial set, let
$f: B \times \Delta^1 \rightarrow S$ be a map of simplicial sets, $M$ a collection of edges of $B$ (containing every nondegenerate edge), and $(A,M_0)$ a marked simplicial subset of $(B,M)$.
Suppose that the following condition is satisfied:
\begin{itemize}
\item[$(\ast)$] For every $n$-simplex $\Delta^n \rightarrow B$ which does not factor through $A$, the composite map $\Delta^n \times \Delta^1 \rightarrow S$
carries the $2$-simplex spanned by $(0,0)$, $(0,1)$, and $(n,1)$ to a thin simplex in $S$.
\end{itemize}
Then the inclusion 
$$ f: ((A,M_0) \times (\Delta^{1})^{\sharp})
\coprod_{ (A,M_0) \times \{0\}^{\sharp} } ((B,M) \times \{0\}^{\sharp} )
\subseteq (B,M) \times (\Delta^1)^{\sharp}$$
is $\CatP_{\overline{S}}$-anodyne.
\end{lemma}

\begin{proof}
Working simplex-by-simplex, we can reduce to one of the following two cases:
\begin{itemize}
\item The marked simplicial set $(B,M)$ is $(\Delta^1)^{\sharp}$, and 
$(A,M_0)$ is $(\Delta^1)^{\flat}$. In this case, the inclusion $f$ is a pushout of a morphism of type $(A_0)$ appearing in Definition \ref{postspunt}. 

\item The marked simplicial set $(B,M)$ is $(\Delta^n)^{\flat}$, and the marked simplicial subset
$(A,M_0)$ is $( \bd \Delta^n)^{\flat}$. We define a sequence of $(n+1)$-simplices
$$ \sigma_0, \ldots, \sigma_n: \Delta^{n+1} \rightarrow \Delta^n \times \Delta^1,$$
given on vertices by the formulae
$$ \sigma_i(j) = \begin{cases} (j,0) & \text{if } j \leq n-i \\
(j-1, 1) & \text{if } j > n-i. \end{cases}$$
Let $K_0 = (\bd \Delta^n \times \Delta^1) \coprod_{ \bd \Delta^n \times \{0\} } ( \Delta^n \times \{0\})$,
and for $1 \leq i \leq n+1$ let $K_i = K_0 \cup \sigma_0 \cup \ldots \cup \sigma_{i-1}$.
Let $N_i$ denote the collection of all marked edges of $(\Delta^n)^{\flat} \times (\Delta^1)^{\sharp}$ that belong to $K_i$. We wish to prove that the inclusion $(K_0, N_0) \subseteq (K_{n+1}, N_{n+1})$
is $\CatP_{\overline{S}}$-anodyne. It will therefore suffice to prove for each
$i \leq n$ that the inclusion $(K_i, N_i) \subseteq (K_{i+1}, N_{i+1})$ is $\CatP_{\overline{S}}$-anodyne. If $i < n$, then the inclusion $f_i$ is a pushout of the inclusion
$(\Lambda^{n+1}_{n-i})^{\flat} \subseteq (\Delta^{n+1})^{\flat}$, which is a morphism of type
$(C_1)$ appearing in Definition \ref{postspunt}.
If $n=0$, then $f_{n}$ is a morphism of type $(B_0)$ appearing in Definition \ref{postspunt}.
If $n > 0$, then $f_{n}$ is a pushout of the inclusion
$$( \Lambda^{n+1}_0)^{\flat} \coprod_{ (\Delta^{ \{0,1\} })^{\flat} } (\Delta^{ \{0,1\} })^{\sharp}
\subseteq (\Delta^{n+1})^{\flat} \coprod_{ (\Delta^{ \{0,1\} })^{\flat} } ( \Delta^{ \{0,1\}})^{\sharp},$$
which is of the type $(C_0)$ appearing in Definition \ref{postspunt} by virtue of assumption $(\ast)$.
\end{itemize}
\end{proof}

\begin{proposition}\label{cast2}
Let $\overline{\calC} = (\calC, T)$ be a weak $\infty$-bicategory containing a vertex $x$, and let
$X$ denote the vertex of $\overline{\calC}^{x/}$ corresponding to the degenerate edge
$\id_{x}$ in $\overline{\calC}$. Then the inclusion
$$ i: \{ X \}^{\sharp} \subseteq \overline{\calC}^{x/}$$
is a $\CatP_{\overline{\calC}}$-anodyne morphism.
\end{proposition}

\begin{proof}
As in Notation \ref{caff}, we let $\calC^{x/}_0$ denote the underlying simplicial set
of $\overline{\calC}^{x/}$. There is a canonical evaluation map
$e: \calC^{x/}_0 \times \Delta^1 \rightarrow \calC$. Let
$p: \Delta^1 \times \Delta^1 \rightarrow \Delta^1$ be the map which is given on vertices by the formula $(i,j) \mapsto ij$. The composition
$$ \calC^{x/}_0 \times \Delta^1 \times \Delta^1 \stackrel{\id \times p}{\rightarrow}
\calC^{x/}_{0} \times \Delta^1 \stackrel{e}{\rightarrow} \calC$$
is adjoint to a map of simplicial sets
$$ \calC^{x/}_0 \times \Delta^1 \rightarrow \calC^{x/}_0.$$
This underlies a map of marked simplicial sets
$$ r: \overline{\calC}^{x/} \times (\Delta^1)^{\sharp} \rightarrow \overline{\calC}^{x/}.$$
We have a commutative diagram
$$ \xymatrix{
\{X\}^{\sharp} \times \{1\}^{\sharp} \ar[r]^{i} \ar[d] & \overline{\calC}^{x/} \times \{1\}^{\sharp} \ar[d] \\
( \{X\} \times \Delta^1)^{\sharp} \coprod_{ \{X\}^{\sharp} \times \{0\}^{\sharp}  }
( \overline{\calC}^{x/} \times \{0\}^{\sharp} ) \ar[r]^{j} \ar[d] & \overline{\calC}^{x/} \times (\Delta^1)^{\sharp} \ar[d]^{r} \\
\{X\}^{\sharp} \ar[r]^{i} & \overline{\calC}^{x/} }$$
in the category $\mset{\overline{S}}$. This diagram exhibits $i$ as a retract of the inclusion $j$; it will therefore suffice to show that $j$ is $\CatP_{\overline{S}}$-anodyne, which follows from Lemma \ref{kappus}.
\end{proof}

\subsection{$\infty$-Bicategories and $\mSet$-Enriched Categories}\label{bisec4.2}

Our goal in this section is to prove the main result of this paper: Theorem
\ref{slai}, which establishes the existence of a model structure on $\scSet$ such that
the scaled nerve functor $\scNerve: \Cat_{\mSet} \rightarrow \scSet$ is a right Quillen
equivalence. The main ingredient is Theorem \ref{csi}, which gives a direct description of 
the homotopy type of the mapping objects $\bHom_{ \scCoNerve[ \overline{S} ] }(x,y)$
when $\overline{S}$ is a weak $\infty$-bicategory.

To begin, suppose that $\overline{S}=(S,T)$ be a scaled simplicial set. For every pair of vertices $x,y \in \overline{S}$, we let $\Hom_{\overline{S}}(x,y) \in \mSet$ denote the fiber product
$\overline{S}^{x/} \times_{ S^{\sharp} } \{y\}^{\sharp}$. If $\overline{S}$ is a weak $\infty$-bicategory, then it follows from Proposition \ref{ilk} that $\Hom_{\overline{S}}(x,y)$ is a fibrant object of $\mSet$: that is, it has the form
$(\calC, M)$ where $\calC$ is an $\infty$-category and $M$ is the collection of all equivalences in $\calC$.

\begin{remark}\label{pal}
Let $\calC$ be an $\mSet$-enriched category containing two objects $x$ and $y$. Then
we have a canonical isomorphism of marked simplicial sets $\scUn \bHom_{\calC}(x,y)
\simeq \Hom_{ \scNerve \calC}(x,y)$. 

More generally, given a scaled simplicial set $\overline{S}$, an $\mSet$-enriched functor
$\phi: \scCoNerve[ \overline{S} ] \rightarrow \calC$, and two vertices $x$ and $y$ of $\overline{S}$, we can compose the adjoint $\overline{S} \rightarrow \scNerve \calC$ of $\phi$ with the adjoint
$\scSt \Hom_{ \scNerve \calC}( \phi x, \phi y) \rightarrow \bHom_{\calC}( \phi x, \phi y)$ to obtain a morphism
$$ \alpha_{\phi}: \scSt \Hom_{ \overline{S} }(x,y) \rightarrow \bHom_{\calC}( \phi x, \phi y).$$
\end{remark}

\begin{theorem}\label{csi}
Let $\overline{S}$ be a weak $\infty$-bicategory and let $\phi: \scCoNerve[ \overline{S}] \rightarrow \calC$ be a weak equivalence of $\mSet$-enriched categories, where $\calC$ is a fibrant $\mSet$-enriched category. Let $\psi: \overline{S} \rightarrow \scNerve \calC$ be the morphism adjoint to $\phi$.
Then, for every pair of vertices $x$ and $y$ in $\overline{S}$, the induced map
$$ f: \Hom_{\overline{S}}(x,y) \rightarrow \Hom_{ \scNerve \calC}( \psi x, \psi y)$$
is a weak equivalence of marked simplicial sets.
\end{theorem}

\begin{proof}
Let $\phi': \scCoNerve[ \scNerve \calC ] \rightarrow \calC$ denote the counit map.
We have a commutative diagram of marked simplicial sets
$$ \xymatrix{ \Hom_{ \overline{S} }(x,y) \ar[d]^{f} & \scSt \Hom_{\overline{S}}(x,y) \ar[l] \ar[d]^{ \scSt f} 
\ar[r]^{ \alpha_{\id}} & \bHom_{ \scCoNerve[ \overline{S} ] }(x,y) \ar[d]^{f'} \\
\bHom_{ \scNerve(\calC)}( \phi x, \phi y) & \scSt \bHom_{ \scNerve(\calC)}( \phi x, \phi y) \ar[l] \ar[r]^-{ \alpha_{\phi_0}} & \bHom_{\calC}( \phi x, \phi y), }$$
where $\alpha_{\id}$ and $\alpha_{\phi_0}$ are defined as in Remark \ref{pal}. We wish to prove that $f$ is a weak equivalence. Proposition \ref{uil} implies that the left horizontal maps are weak equivalences, so it will suffice to show that $\scSt f$ is a weak equivalence. For this, it will suffice to prove the following:
\begin{itemize}
\item[$(i)$] The map $f'$ is a weak equivalence. This follows from our assumption that $\phi$ is
a weak equivalence.
\item[$(ii)$] The map $\alpha_{\phi_0}$ is a weak equivalence. This map is adjoint to the isomorphism
$\Hom_{ \scNerve \calC}( \phi x, \phi y) \simeq \scUn \bHom_{\calC}( \phi x, \phi y)$ of Remark \ref{pal}. This follows from the fact that $( \scSt, \scUn)$ is a Quillen equivalence (Corollary \ref{ba3}), since
the marked simplicial set $\bHom_{\calC}( \phi x, \phi y)$ is fibrant.
\item[$(iii)$] The map $\alpha_{\id}$ is a weak equivalence. To prove this, let
$\LCones{ \overline{S} }{\overline{S}^{x/}}$ be defined as in Definition \ref{sputer1}, let
$\ast$ denote the cone point of $\LCones{ \overline{S}}{\overline{S}^{x/}}$, and observe that we have a canonical map of scaled simplicial sets $\pi: \LCones{ \overline{S} }{\overline{S}^{x/}} \rightarrow \overline{S}$. This map fits into a commutative diagram
$$ \xymatrix{ \scSt( \overline{S}^{x/} \times_{ S^{\sharp}} \{y\}^{\sharp}) \ar[r]^{\beta} \ar[d]^{\delta_0}
& (\scSt_{\overline{S}} \overline{S}^{x/})(y) \ar[d]^{\delta_1} & (\scSt_{ \overline{S}} \{x\}^{\sharp})(y) \ar[l]^{\gamma} \ar[d]^{\delta^2} \\
\scSt \Hom_{ \overline{S}}(x,y) \ar[dr]^{\alpha_{\id}} & \bHom_{ \LCones{ \overline{S}}{\overline{S}^{x/}}} \ar[d]^{\epsilon} & \bHom_{ \scCoNerve[ \overline{S} ]}(x,y) \ar[dl]^{\id} \\
& \bHom_{ \scCoNerve[ \overline{S} ]}(x,y). & }$$
Here the maps $\delta_0$, $\delta_1$, and $\delta_2$ are isomorphisms of marked simplicial sets.
Propositions \ref{cast2} and \ref{curp} imply that $\gamma$ is a weak equivalence. By the two-out-of-three property, we deduce that $\epsilon$ is a weak equivalence. It follows from Corollary \ref{stlik} and Proposition \ref{ilk} that $\beta$ is also a weak equivalence, so that $\alpha_{\id}$ is a weak equivalence by another two-out-of-three argument.
\end{itemize}
\end{proof}

\begin{lemma}\label{precab}
Let $\overline{S}$ be a weak $\infty$-bicategory containing a pair of vertices $x$ and $y$.
Then there is a canonical isomorphism
$\Hom_{ \overline{S} }(x,y) \simeq \bHom_{ \sCoNerve[ \overline{S}]}(x,y)$ in the homotopy category of marked simplicial sets.
\end{lemma}

\begin{proof}
We have a chain of maps
$$\bHom_{ \scCoNerve[ \overline{S}]}(x,y) \simeq
(\scSt_{ \overline{S}} \{x\}^{\sharp})(y) \rightarrow (\scSt_{ \overline{S}} \overline{S}^{x/})(y) \leftarrow \scSt \Hom_{\overline{S}}(x,y) \rightarrow \Hom_{ \overline{S}}(x,y),$$
each of which is a weak equivalence in $\mSet$ (as in the proof of Theorem \ref{csi}).
\end{proof}

\begin{lemma}\label{cabbus}
Let $f: \overline{S} \rightarrow \overline{S}'$ be a map of $\infty$-bicategories which is surjective on vertices. Then $f$ is bicategorical equivalence if and only if, for every pair of vertices
$x$ and $y$ of $\overline{S}$, the map $\Hom_{ \overline{S} }(x,y) \rightarrow
\Hom_{ \overline{S}}( fx, fy)$ is a weak equivalence of marked simplicial sets.
\end{lemma}

\begin{proof}
The map $f$ is a bicategorical equivalence if and only if $\scCoNerve[f]$ is a weak equivalence of $\mSet$-enriched categories. Since $\scCoNerve[f]$ is essentially surjective, this is equivalent
to the requirement that for every pair of vertices $x$ and $y$ in $\overline{S}$, the map
$\bHom_{ \scCoNerve[ \overline{S} ]}(x,y) \rightarrow \bHom_{ \scCoNerve[ \overline{S}']}(fx,fy)$ is a weak equivalence of marked simplicial sets. We now invoke Lemma \ref{precab}.
\end{proof}

\begin{lemma}\label{kilpus}
Let $f: \calC \rightarrow \calD$ be a functor between fibrant $\mSet$-enriched categories.
Assume that $f$ is surjective on objects, and that the induced map $\scNerve(f)$ is a bicategorical equivalence. Then $f$ is a weak equivalence of $\mSet$-enriched categories.
\end{lemma}

\begin{proof}
Since $f$ is essentially surjective, it suffices to show that the map
$\bHom_{\calC}(x,y) \rightarrow \bHom_{\calD}(fx,fy)$ is a weak equivalence of marked simplicial sets for every pair of objects $x,y \in \calC$. Since $\calC$ and $\calD$ are fibrant, this is equivalent to the assertion that the map $\scUn \bHom_{\calC}(x,y) \rightarrow \scUn \bHom_{\calD}(fx,fy)$ is a weak equivalence (Corollary \ref{ba3}). The desired result now follows from Remark \ref{pal} and Lemma \ref{cabbus}.
\end{proof}

\begin{lemma}\label{carpus}
Let $f: X \rightarrow X'$ and $g: Y \rightarrow Y'$ be bicategorical equivalences of scaled simplicial sets. Then the induced map $f \times g: X \times Y \rightarrow X' \times Y'$ is a bicategorical equivalence.
\end{lemma}

\begin{proof}
Working one variable at a time, we may assume that $Y = Y'$ and that $g$ is the identity map.
Choose a scaled anodyne map $Y \rightarrow Z$, where $Z$ is a weak $\infty$-bicategory.
We have a commutative diagram
$$ \xymatrix{ X \times Y \ar[r] \ar[d] & X \times Z \ar[d] \\
X' \times Y \ar[r] & X' \times Z. }$$
Using Proposition \ref{notred}, we deduce that the horizontal maps are scaled anodyne, and therefore bicategorical equivalences. Consequently, the left vertical map is a bicategorical equivalence if and only if the right vertical map is a bicategorical equivalence. We may therefore replace $Y$ by $Z$ and reduce to the case where $Y$ is a weak $\infty$-bicategory.

Choose a scaled anodyne map $X' \rightarrow X''$, where $X''$ is a weak $\infty$-bicategory.
We have a commutative diagram
$$ \xymatrix{ & X' \times Y \ar[dr]^{h \times g} & \\
X \times Y \ar[ur]^{f \times g} \ar[rr]^{h' \times g} & & X'' \times Y. }$$
Proposition \ref{notred} guarantees that $h \times g$ is scaled anodyne, and therefore a categorical equivalence. By the two-out-of-three property, it will suffice to show that $h \times g$ is a bicategorical equivalence. We may therefore replace $X'$ by $X''$ and thereby reduce to the case where $X'$ is a weak $\infty$-bicategory.

Choose a factorization of $f$ as a composition
$$ X \stackrel{f'}{\rightarrow} X_0 \stackrel{f''}{\rightarrow} X'$$
where $f'$ is scaled anodyne and the map $f''$ has the right lifting property with respect to all scaled anodyne morphisms. We have a commutative diagram
$$ \xymatrix{ & X_0 \times Y \ar[dr]^{f'' \times g} & \\
X \times Y \ar[ur]^{f' \times g} \ar[rr]^{f \times g} & & X' \times Y. }$$
Proposition \ref{notred} guarantees that $f' \times g$ is scaled anodyne, and therefore a bicategorical equivalence. By the two-out-of-three property, it will suffice to show that $f'' \times g$ is a bicategorical equivalence. We may therefore replace $X$ by $X_0$, and thereby reduce to the case where
$X$ is a weak $\infty$-bicategory.

We now prove that the map $\scCoNerve[f \times g]$ is fully faithful. Choose vertices
$x_0, x_1 \in X$ and $y_0, y_1 \in Y$; we wish to show that the map
$$ \bHom_{ \scCoNerve[ X \times Y]}( (x_0, y_0), (x_1, y_1))
\rightarrow \bHom_{ \scCoNerve[ X' \times Y]}( (fx_0, y_0), (fx_1, y_1))$$
is a weak equivalence in $\mSet$. In view of Lemma \ref{precab}, it suffices to show that the induced map
$$ \Hom_{ X \times Y}( (x_0, y_0), (x_1, y_1) ) \rightarrow \Hom_{ X' \times Y}( (fx_0, y_0), (fx_1, y_1))$$
is a weak equivalence of marked simplicial sets. Since the collection of weak equivalences in $\mSet$ is stable under the formation of products, it suffices to show that the map
$\Hom_{ X}(x_0, x_1) \rightarrow \Hom_{X'}(fx_0, fx_1)$ is a weak equivalence. This follows from Lemma \ref{precab}, since the functor $\scCoNerve[f]$ is fully faithful.

To complete the proof, it suffices to show that $\scCoNerve[f \times g]$ is essentially surjective.
Choose an object $(x',y)$ of $\scCoNerve[ X' \times Y]$. Since $\scCoNerve[f]$ is essentially surjective, there exists a vertex $x$ of $X$ such that $fx$ is equivlaent to $x'$ in $\scCoNerve[X']$. It is then easy to see that $(x',y)$ is equivalent to $(fx,y)$ in $\scCoNerve[X' \times Y]$.
\end{proof}

\begin{theorem}\label{slai}
There exists a left proper, combinatorial model structure on the category
$\scSet$ of scaled simplicial sets with the following properties:
\begin{itemize}
\item[$(C)$] The cofibrations in $\scSet$ are monomorphisms of scaled simplicial sets.
\item[$(W)$] The weak equivalences in $\scSet$ are the bicategorical equivalences.
\item[$(F)$] A morphism in $\scSet$ is a fibration if and only if it has the right lifting property with respect to all morphisms satisfying both $(C)$ and $(W)$.
\end{itemize}
Moreover, the adjoint functors $\Adjoint{ \scCoNerve}{\scSet}{\mCat}{\scNerve}$ determine a Quillen equivalence of $\scSet$ with $\mCat$.
\end{theorem}

\begin{proof}
To prove the first assertion, we will show that $\scSet$ satisfies the hypotheses of Proposition \toposref{goot}:
\begin{itemize}
\item[$(1)$] The class of weak equivalences in $\scSet$ is perfect; this follows by applying Corollary
\toposref{perfpull} to the functor $\scCoNerve$.
\item[$(2)$] The collection of weak equivalences in $\scSet$ is stable under the formation of pushouts; this follows from the fact that $\mSet$ is left-proper.
\item[$(3)$] Let $f: (X,T) \rightarrow (Y,T')$ be a map of scaled simplicial sets which has the
right lifting property with respect to {\em all} cofibrations. We wish to prove that $f$ is a bicategorical equivalence. Our hypothesis implies that the underlying map
of simplicial sets $X \rightarrow Y$ is a trivial Kan fibration, and $T = f^{-1} T'$. It follows that
$f$ admits a section $s$. Moreover, the composition $s \circ f$ is homotopic to the identity, in the sense that there exists a contractible Kan complex $K$ and a map $h: K_{\sharp} \times (X,T) \rightarrow (X,T)$ with the following properties:
\begin{itemize}
\item[$(i)$] The composition $f \circ h$ coincides with the composition $f \circ \pi$, where
$\pi: K_{\sharp} \times (X,T) \rightarrow (X,T)$ denotes the projection onto the second factor.
\item[$(ii)$] There exists a pair of vertices $x,y \in K$ such that $h | (\{x\}_{\sharp} \times (X,T))$ is the identity and $h | ( \{y\}_{\sharp} \times (X,T))$ coincides with $s \circ f$. 
\end{itemize}
We then have a commutative diagram
$$ \xymatrix{ \{y\}_{\sharp} \times (X,T) \ar[r]^-{f} \ar[d] & (Y,T') \ar[d]^{s} \\
K_{\sharp} \times (X,T) \ar[r]^-{h} \ar[d]^{\pi} & (X,T) \ar[d]^{f} \\
(X,T) \ar[r]^{f} & (Y,T') }$$
which exhibits $f$ as a retract of $h$; it will therefore suffice to show that $h$ is a bicategorical equivalence. The map $h$ has a right inverse, given by the inclusion
$i: (X,T) \simeq \{x\}_{\sharp} \times (X,T) \subseteq K_{\sharp} \times (X,T)$. It will therefore suffice to show that $i$ is a bicategorical equivalence, which follows from Lemma \ref{carpus} and Proposition
\ref{twop2}.
\end{itemize}
This completes the proof that $\scSet$ is a left proper, combinatorial model category. 
The functor $\scCoNerve: \scSet \rightarrow \mSet$ preserves cofibrations and weak equivalences, and is therefore a left Quillen functor. We will complete the proof by showing that
$(\scCoNerve, \scNerve)$ is a Quillen equivalence. Let $\overline{S}$ be a scaled simplicial set
and $\calC$ an $\mSet$-enriched category; we wish to show that the unit and counit maps
$$ u_{\overline{S}}: \overline{S} \rightarrow R \scNerve( L \scCoNerve[\overline{S}])$$
$$ v_{\calC}: L \scCoNerve[ R \scNerve \calC] \rightarrow \calC$$
are isomorphisms in the homotopy categories $\h{\scSet}$ and $\h{ \mCat}$, respectively.
We first show that $u_{ \overline{S}}$ is an isomorphism. Choose a weak equivalence
$\scCoNerve[ \overline{S} ] \rightarrow \calD$ of $\mSet$-enriched categories which is bijective on objects, where $\calD$ is fibrant. We wish to show that the adjoint map
$\overline{S} \rightarrow \scNerve \calD$ is a bicategorical equivalence. 
This follows from Lemma \ref{cabbus} and Theorem \ref{csi}.

We now show that each of the maps $v_{\calC}$ is an isomorphism in the homotopy category
$\h{ \mCat}$. In other words, we must show that for every fibrant object $\calC \in \mCat$, the
counit map $v: \scCoNerve[ \scNerve \calC] \rightarrow \calC$ is a weak equivalence of $\mSet$-enriched categories. The map $v$ admits a factorization
$$ \scCoNerve[ \scNerve \calC] \stackrel{v'}{\rightarrow} \calC' \stackrel{v''}{\rightarrow} \calC,$$
where $v'$ is a weak equivalence and $v''$ is a fibration (so that $\calC'$ is fibrant).
By the two-out-of-three property, it will suffice to show that $v''$ is a weak equivalence of
$\mSet$-enriched categories. Since $v''$ is surjective on objects, it will suffice to show that
$\scNerve(v'')$ is a bicategorical equivalence (Lemma \ref{kilpus}). We conclude by observing that $\scNerve(v'')$ a left homotopy inverse to the weak equivalence $u_{ \scNerve \calC}$.
\end{proof}

\begin{definition}
We will say that a scaled simplicial set $\calC$ is an {\it $\infty$-bicategory} if it is a fibrant object of
$\scSet$, with respect to the model structure described in Theorem \ref{slai}
\end{definition}

\begin{remark}\label{coslai}
Let $\calC$ be a fibrant $\mSet$-enriched category. Then the scaled nerve
$\scNerve \calC$ is an $\infty$-bicategory. This follows immediately from Theorem \ref{slai}.
\end{remark}

\subsection{Subdivision}\label{subber}\label{bisec4.3}

In this section, we will study the {\it subdivision} functor
$\sdd^{+}: \scSet \rightarrow \mset{ \Nerve(\cDelta)^{op} }$. 
Our main result, Theorem \ref{cuti}, asserts that this functor is a left Quillen
equivalence (where $\scSet$ is endowed with the bicategorical model structure
and $\mset{ \Nerve(\cDelta)^{op}}$ with the complete Segal model structure).

We begin by introducing some definitions.

\begin{definition}\label{conf}
Let $X$ be a simplicial set. We let $\sdd(X)$ denote the simplicial set
$\Nerve( \cDelta_{X})^{op}$, where $\cDelta_{X}$ denotes the category of simplices of
$X$ (see \S \toposref{quasilimit1}). The forgetful functor $\cDelta_{X} \rightarrow \cDelta$
determines a map of simplicial sets $\sdd(X) \rightarrow \Nerve(\cDelta)^{op}$. 
We will regard $\sdd$ as defining a functor $\sSet \rightarrow (\sSet)_{ / \Nerve(\cDelta)^{op} }$.

Suppose that $(X, T)$ is a scaled simplicial set. We let $\sdd^{+}(X,T)$ denote the
marked simplicial set $( \sdd(X), \calE)$, where $\calE$ denotes the collection of all edges $e$ of
of $\sdd(X)$ which satisfy one of the following conditions:
\begin{itemize}
\item[$(a)$] The edge $e$ corresponds to a diagram of simplicial sets
$$ \xymatrix{ \Delta^m \ar[rr] \ar[dr] & & \Delta^n \ar[dl] \\
& X & }$$
such that image of $[m]$ is a convex subset of $[n]$.
\item[$(b)$] The edge $e$ corresponds to a diagram of simplicial sets
$$ \xymatrix{ \Delta^{ \{0,2\} } \ar@{^{(}->}[rr] \ar[dr] & & \Delta^2 \ar[dl]^{\sigma} \\
& X & }$$
such that $\sigma \in T$. 
\end{itemize}
We will regard $\sdd^{+}$ as defining a functor from $\scSet$ to $\mset{ \Nerve(\cDelta)^{op} }$.
\end{definition}

\begin{remark}
The functor $\sdd: \sSet \rightarrow ( \sSet)_{/ \Nerve(\cDelta)^{op} }$ is fully faithful.
Its essential image consists precisely of those maps $Z \rightarrow \Nerve( \cDelta)^{op}$ such
that $Z$ is isomorphic to the nerve of a category $\calC$, and the induced functor
$p: \calC \rightarrow \cDelta^{op}$ is cofibered in sets; in this case we can recover 
the underlying simplicial set by the formula $[n] \mapsto p^{-1} \{ [n] \}$. 
\end{remark}

\begin{notation}
Let $F$ denote the composite functor
$$ \sSet \stackrel{ \sdd}{\rightarrow} (\sSet)_{/ \Nerve(\cDelta)^{op} }
\stackrel{ \lNerve_{\bigdot}( \cDelta^{op} )}{\rightarrow} \Fun( \cDelta^{op}, \sSet).$$
Using Example \toposref{spacek}, we can describe the functor $F$ as follows:
it carries a simplicial set $X$ to the simplicial object $F(X)_{\bigdot}$ described by the formula
$$ F(X)_{n} = \Nerve( \cDelta_{X} \times_{\cDelta} \cDelta_{[n]/})^{op}.$$
Note that $F(X)_{n}$ admits a natural descomposition $F(X)_n \simeq \coprod_{\sigma} F(X)_{\sigma}$,
where $\sigma$ ranges over all $n$-simplices of $X$.

Let $F^{+}$ denote the composite functor
$$ \scSet \stackrel{ \sdd^{+}}{\rightarrow} \mset{ \Nerve(\cDelta)^{op} }
\stackrel{ \lNerve^{+}_{\bigdot}( \cDelta^{op} )}{\rightarrow} \Fun( \cDelta^{op}, \mSet).$$
Then $F^{+}$ carries a scaled simplicial set $(X,T)$ to the simplicial object
$F^{+}(X,T)_{\bigdot}$ of $\mSet$ described by the formula
$$F^{+}(X,T)_{n} = (F(X)_n, \calE_{n})$$
where $\calE_{n}$ denotes the collection of all edges of $F(X)_n$ which correspond to diagrams
$$ \Delta^n \rightarrow \Delta^k \stackrel{f}{\rightarrow} \Delta^{k'} \stackrel{\sigma}{\rightarrow} X$$
satisfying one of the following two conditions:
\begin{itemize}
\item[$(a)$] The image of $[k]$ under the map $f$ is a convex subset of $[k']$.
\item[$(b)$] The map $f$ is isomorphic to an inclusion $\Delta{ \{0,2\} } \subseteq \Delta^2$
and the $2$-simplex $\sigma$ belongs to $T$.
\end{itemize}

We define another functor $F': \sSet \rightarrow \Fun( \cDelta^{op}, \sSet)$ to be the composition
$$ \sSet \stackrel{\sCoNerve}{\rightarrow} \Cat_{\sSet}
\stackrel{i}{\rightarrow} \Cat_{\sSet} \subseteq \PreSeg{\sSet} \stackrel{\UnPre}{\rightarrow}
\Fun( \cDelta^{op}, \sSet),$$
where $i: \Cat_{\sSet} \rightarrow \Cat_{\sSet}$ is the functor induced by the equivalence
$X \mapsto X^{op}$ from $\sSet$ to itself. More concretely, $F'$ carries a simplicial set
$X$ to the simplicial object $F'(X)_{\bigdot}$ of $\sSet$ described by the formula
$$ F'(X)_n = \coprod_{ x_0, \ldots, x_n} \bHom_{ \scCoNerve[X]}(x_0, x_1)^{op}
\times \ldots \times \bHom_{\scCoNerve[X]}( x_{n-1}, x_n)^{op}.$$

We observe that there is a natural transformation of functors
$\alpha: F \rightarrow F'$, which is uniquely determined by the following conditions:
\begin{itemize}
\item For every simplicial set $X$ and every $n$-simplex $\sigma$ of $X$ with
vertices $x_0, \ldots, x_n$, the map $\alpha_{X}: F(X)_{\bigdot} \rightarrow F'(X)_{\bigdot}$
restricts to a map 
$$\alpha_{X}^{\sigma}: F(X)_{\sigma} \rightarrow \bHom_{ \scCoNerve[X]}(x_0, x_1)^{op}
\times \ldots \times \bHom_{\scCoNerve[X]}( x_{n-1}, x_n)^{op} \subseteq F'(X)_{n}.$$
\item The map $\alpha_{X}^{\sigma}$ is the product of the opposite of maps
$$\beta_{i}: F(X)^{op}_{\sigma} \rightarrow \bHom_{ \scCoNerve[X]}( x_{i}, x_{i+1}).$$ Here the maps $\beta_i$ are defined as follows. Suppose given
a $k$-simplex $\tau$ of $F(X)^{op}_{\sigma}$, corresponding to a commutative diagram
$$ \xymatrix{ \Delta^{m_0} \ar[r] & \Delta^{m_1} \ar[r] & \ldots \ar[r] & \Delta^{m_k} \ar[d] \\
\Delta^n \ar[u] \ar[rrr]^{\sigma} & & & X. }$$
Let $y$ and $z$ denote the images of $i, i+1 \in [n]$, and for $0 \leq j \leq k$ let
$S_{j}$ denote the subset of $[m_k]$ consisting of those elements
$m_0$ such that $y \leq m_0 \leq z$ and such that $m_0$ lies in the image of the map
$[m_j] \rightarrow [m_k]$. The chain of subsets $S_0 \subseteq \ldots \subseteq S_k$
determines a $k$-simplex $\tau'$ in $\bHom_{ \sCoNerve[ \Delta^{m_k}]}(y,z)$, and we define
$\beta_i(\tau)$ to be the image of $\tau'$ in $\bHom_{ \sCoNerve[X]}( x_{i}, x_{i+1})$. 
\end{itemize}

We define a functor ${F'}^{+}: \scSet \rightarrow \Fun( \cDelta^{op}, \mSet)$ to be the composition
$$ \scSet \stackrel{\scCoNerve}{\rightarrow} \Cat_{\mSet}
\stackrel{i}{\rightarrow} \Cat_{\mSet} \subseteq \PreSeg{\mSet} \stackrel{\UnPre}{\rightarrow}
\Fun( \cDelta^{op}, \mSet),$$
where $i: \Cat_{\mSet} \rightarrow \Cat_{\mSet}$ is again the functor induced by the equivalence
$X \mapsto X^{op}$ from $\mSet$ to itself. We observe that the natural transformation $\alpha$
extends uniquely to a natural transformation $\alpha^{+}: F^{+} \rightarrow {F'}^{+}$.
\end{notation}

\begin{remark}\label{cubbly}
Let $(X,T)$ be a scaled simplicial set. Assume that $X$ satisfies the following condition:
\begin{itemize}
\item[$(\ast)$] Every face of a nondegenerate simplex of $X$ is again nondegenerate.
\end{itemize}
We let $\sdd_0(X) \subseteq \sdd(X)$ denote the
full simplicial subset spanned by the {\em nondegenerate} simplices of $X$.
Let $\sdd^{+}_0(X,T) = (\sdd_0(X), \calE)$, where $\calE$ is the collection of all edges of
$\sdd_0(X)$ which are marked in $\sdd(X)$. 

Condition $(\ast)$ implies that the inclusion $\sdd_0(X) \subseteq \sdd(X)$ admits a right adjoint $G$, so that we have a canonical homotopy $h: \sdd(X) \times \Delta^1 \rightarrow \sdd(X)$ from
$G$ to the identity, which is trivial on $\sdd_0(X)$. This homotopy determines a map
$$ \sdd^{+}(X,T) \times (\Delta^1)^{\sharp} \rightarrow \sdd^{+}(X,T).$$
The diagram
$$ \xymatrix{ \sdd_0^{+}(X,T) \times \{1\}^{\sharp} \ar@{^{(}->}[r] \ar@{^{(}->}[d] & \sdd^{+}(X,T) \times \{1\}^{\sharp} \ar@{^{(}->}[d] \\
( \sdd_0^{+}(X,T) \times (\Delta^1)^{\sharp} )
\coprod_{ \sdd_0^{+}(X,T) \times \{0\}^{\sharp} } ( \sdd^{+}(X,T) \times \{0\}^{\sharp} ) \ar@{^{(}->}[r] \ar[d] & \sdd^{+}(X,T) \times (\Delta^1)^{\sharp} \ar[d] \\
\sdd^{+}_0(X,T) \ar@{^{(}->}[r] & \sdd^{+}(X,T) }$$
exhibits the inclusion $\sdd^{+}_0(X,T) \subseteq \sdd^{+}(X,T)$ as a retract of the map
$$( \sdd_0^{+}(X,T) \times (\Delta^1)^{\sharp} )
\coprod_{ \sdd_0^{+}(X,T) \times \{0\}^{\sharp} } ( \sdd^{+}(X,T) \times \{0\}^{\sharp} )
\subseteq \sdd^{+}(X,T) \times (\Delta^1)^{\sharp}$$
which is the opposite of a marked anodyne map (by Proposition \toposref{markanodprod}).
It follows that the inclusion $\sdd_0^{+}(X,T) \subseteq \sdd^{+}(X,T)$ is a coCartesian
equivalence in $\mset{ \Nerve(\cDelta)^{op} }$.
\end{remark}

\begin{notation}\label{clubb}
Let $(X,T)$ be a scaled simplicial set such that $X$ satisfies condition $(\ast)$ of Remark \ref{cubbly}. Let $F^{+}_0(X,T)_{\bigdot}$ denote the 
simplicial object of $\mSet$ obtained by applying the functor $\lNerve_{\bigdot}(\cDelta^{op} ):
\mset{ \Nerve(\cDelta)^{op} } \rightarrow \Fun( \cDelta^{op}, \mSet)$ to
$\sdd^{+}_0(X,T)$. It follows from Remark \ref{cubbly} (and the fact that $\lNerve_{\bigdot}(\cDelta^{op})$ is a left Quillen functor) that the induced map $F^{+}_0(X,T)_{\bigdot} \rightarrow F^{+}(X,T)_{\bigdot}$ is a trivial cofibration with respect to the projective model structure on $ \Fun( \cDelta^{op}, \mSet)$.

For each $n \geq 0$, we define a marked simplicial subset $F^{+}_1(X,T)_{n} =
( \Nerve(\calC)^{op}, \calE) \subseteq F^{+}_0(X,T)_n$ as follows:
$\calC$ is the full subcategory of $\cDelta_{X} \times_{\cDelta} \cDelta_{[n]/}$
spanned by those objects which correspond to diagrams
$$ \Delta^n \stackrel{f}{\rightarrow} \Delta^k \stackrel{\sigma}{\rightarrow} X$$
such that $\sigma$ is nondegenerate, $f(0) = 0$, and $f(n) = k$; and
$\calE$ is the collection of all edges of $\Nerve(\calC)^{op}$ which are marked edges
of $F^{+}(X,T)$. We note that the inclusion
$\Nerve(\calC)^{op} \subseteq \sdd_0(X)$ admits a left adjoint $F$.
This left adjoint determines a homotopy $H: \sdd_0(X) \times \Delta^1 \rightarrow \sdd_0(X)$ from
the identity to $F$, and $H$ induces a map of marked simplicial sets
$$ F^{+}_0(X,T)_n \times (\Delta^1)^{\sharp} \rightarrow F^{+}_0(X,T)_n.$$
Arguing as in Remark \ref{cubbly}, we deduce that the inclusion
$F^{+}_1(X,T)_n \subseteq F^{+}_0(X,T)_n$ is marked anodyne and therefore a weak equivalence of marked simplicial sets. It follows that $F^{+}_1(X,T)_n \subseteq F^{+}(X,T)_n$ is again an equivalence of marked simplicial sets.
\end{notation}

\begin{warning}
Let $(X,T)$ be as in Notation \ref{clubb}. The construction $[n] \mapsto F^{+}_{1}(X,T)_n$ is functorial only with respect to maps $f: [n] \rightarrow [m]$ such that $f(0) = 0$ and $f(n) = m$, so we cannot
regard $F^{+}_1(X,T)_{\bigdot}$ as a simplicial object of $\mSet$.

Note also that the constructions $(X,T) \mapsto \sdd^{+}_0(X,T)$ and $(X,T) \mapsto F^{+}_0(X,T)$
are not functorial in the pair $(X,T)$, since the image of a nondegenerate simplex of $X$
under a map $X \rightarrow Y$ need not be nondegenerate in $Y$.
\end{warning}

\begin{lemma}\label{urqi}
Suppose given a Quillen equivalence between combinatorial model categories
$\Adjoint{F}{\bfA}{\bfB}{G}.$
\begin{itemize}
\item[$(1)$] If the functor $F$ preserves weak equivalences, then it carries homotopy
colimit diagrams in $\bfA$ to homotopy colimit diagrams in $\bfB$ and homotopy limit
diagrams in $\bfA$ to homotopy limit diagrams in $\bfB$.
\item[$(2)$] If the functor $G$ preserves weak equivalences, then it carries homotopy
colimit diagrams in $\bfA$ to homotopy colimit diagrams in $\bfB$ and homotopy limit
diagrams in $\bfA$ to homotopy limit diagrams in $\bfB$.
\end{itemize}
\end{lemma}

\begin{proof}
We will give the proof of $(1)$; the proof of $(2)$ is identical.

Let $\calJ$ be a small category and $\overline{p}: \calJ \star [0] \rightarrow \bfA$
a homotopy colimit diagram in $\bfA$. We wish to prove that $F \circ \overline{p}$ is a homotopy colimit
diagram in $\bfB$. Since $F$ preserves weak equivalences, we are free to replace $\overline{p}$ by a weakly equivalent diagram if necessary; we may therefore suppose that $\overline{p}$ is a colimit diagram and that $p = \overline{p}|\calJ$ is projectively cofibrant, in which case the result is obvious.

In the case of limit diagrams, we must work slightly harder. Let $\overline{q}: [0] \star \calJ \rightarrow \bfA$ be a homotopy limit diagram in $\bfA$. Let $q = \overline{q} | \calJ$, and choose a
weak equivalence $F \circ q \rightarrow q'$, where $q': \calJ \rightarrow \bfB$ is injectively fibrant.
Let $\overline{q}': [0] \star \calJ \rightarrow \bfB$ be a limit of $q'$. Then
$G \overline{q}'$ is a homotopy limit diagram in $\bfA$. For each $J \in \calJ$, the map
$F q(J) \rightarrow q'(J)$ is a weak equivalence. Since $(F,G)$ is a Quillen equivalent and $F$ preserves weak equivalences, it follows that the adjoint map $q(J) \rightarrow G q'(J)$ is a weak equivalence. Since $\overline{q}$ and $G \overline{q}'$ are both homotopy limit diagrams, we
deduce that the map $\overline{q} \rightarrow G \overline{q'}$ is a weak equivalence of diagrams, so that the adjoint map $F \overline{q} \rightarrow \overline{q}'$ is a weak equivalence of diagrams as well. It follows that $F \overline{q}$ is a homotopy limit diagram in $\bfB$, as desired.
\end{proof}

\begin{lemma}\label{slappus}
The functor ${F'}^{+}: \scSet \rightarrow \Fun( \cDelta^{op}, \mSet)$ preserves weak equivalences
and homotopy colimit diagrams, and induces an equivalence of homotopy categories.
\end{lemma}

\begin{proof}
The functor ${F'}^{+}$ factors as a composition of left and right Quillen equivalences
$$ \scSet \stackrel{ \scCoNerve}{\rightarrow} \Cat_{\mSet} \stackrel{i}{\simeq}
\Cat_{\mSet}
\subseteq \PreSeg{\mSet} \stackrel{\UnPre}{\rightarrow} \Fun( \cDelta^{op}, \mSet),$$
(Theorem \ref{slai}, Theorem \ref{castle2}, and Proposition \ref{curt}),
each of which preserves weak equivalences (Definition \ref{bicateq}, 
Remark \ref{schwe}, and Lemma \ref{sacurt}). We now conclude by invoking Lemma \ref{urqi}.
\end{proof}

\begin{lemma}\label{jinker}
Let $(X,T)$ be a scaled simplicial set which is isomorphic either to $\Delta^2_{\sharp}$ or
to $\Delta^n_{\flat}$, for some $n \geq 0$. Then the map
$$ \alpha^{+}_{(X,T)}: F^{+}(X,T)_{\bigdot} \rightarrow {F'}^{+}(X,T)_{\bigdot}$$
is a weak equivalence in $\Fun( \cDelta^{op}, \mSet)$ with respect to the injective model structure.
\end{lemma}

\begin{proof}
We will show that the induced map $f: F^{+}(X,T)_{k} \rightarrow {F'}^{+}(X,T)_{k}$ is a weak equivalence in $\mSet$ for each $k \geq 0$. We have a commutative diagram
$$ \xymatrix{ & F^{+}(X,T)_{k} \ar[dr]^{f} & \\
F^{+}_1(X,T)_{k} \ar[ur]^{g} \ar[rr]^{h} & & {F'}^{+}(X,T)_{k} }$$
where $g$ is the weak equivalence of Notation \ref{clubb}. A direct calculation shows that
$h$ is an isomorphism of marked simplicial sets, so that $f$ is a weak equivalence by the two-out-of-three property.
\end{proof}

\begin{proposition}\label{kium}
Let $(X,T)$ be a scaled simplicial set. Then the induced map
$\alpha_{(X,T)}: F^{+}(X,T)_{\bigdot} \rightarrow {F'}^{+}(X,T)_{\bigdot}$ is a weak equivalence with
respect to the complete Segal model structure on $\Fun( \cDelta^{op}, \mSet)$.
\end{proposition}

\begin{proof}
Let us say that a scaled simplicial set $(X,T)$ is {\it good} if the map
$\alpha_{(X,T)}$ is a weak equivalence. We wish to prove that every scaled simplicial set $(X,T)$ is good. For this, we argue as follows:
\begin{itemize}
\item[$(a)$] The functors $F^{+}$ and ${F'}^{+}$ commute with filtered colimits, and the collection of weak equivalences in $\Fun( \cDelta^{op}, \mSet)$ is stable under filtered colimits 
(Proposition \ref{smashet} and Example \ref{smashett}). It follows that the collection of good scaled simplicial sets is stable under filtered colimits.
\item[$(b)$] Suppose given a pushout diagram
$$ \xymatrix{ (X,T) \ar[r] \ar[d] & (X', T') \ar[d] \\
(Y,S) \ar[r] & (Y', S') }$$
of scaled simplicial sets, in which the vertical morphisms are cofibrations.
If $(X,T)$, $(X',T')$, and $(Y,S)$ are good, then $(Y', S')$ is good. To prove this, it suffices
to show that the diagrams
$$ \xymatrix{ F^{+}(X,T)_{\bigdot} \ar[r] \ar[d] & F^{+}(X',T')_{\bigdot} \ar[d] & {F'}^{+}(X,T)_{\bigdot} \ar[r] \ar[d] & {F'}^{+}(X',T')_{\bigdot} \ar[d] \\
F^{+}(Y,S)_{\bigdot} \ar[r] & F^{+}(Y',S')_{\bigdot} & {F'}^{+}(Y,S)_{\bigdot} \ar[r] & {F'}^{+}(Y',S')_{\bigdot} }$$
are both homotopy pushout squares with respect to the complete Segal model structure on 
the category $\Fun(\cDelta^{op}, \mSet)$. For the square on the left this is obvious, since it is a pushout square
and the vertical maps are cofibrations. For the square on the right, we invoke Lemma \ref{slappus}.

\item[$(c)$] By virtue of $(a)$, it will suffice to prove that $(X,T)$ is good whenever
the simplicial set $X$ has only finitely many nondegenerate simplices. We now work
by induction on the number of nondegenerate thin simplices of $X$. If this number is not zero, then
we have a pushout diagram
$$ \xymatrix{ \Delta^2_{\flat} \ar[r] \ar[d] & (X,T_0) \ar[d] \\
\Delta^2_{\sharp} \ar[r] & (X,T) }$$
and we conclude using the inductive hypothesis together with $(b)$ and Lemma \ref{jinker}. We may therefore
assume that $(X,T) = X_{\flat}$, where $X$ is a finite simplicial set.

\item[$(d)$] We now work by induction on the dimension $n$ of $X$ and the number of nondegenerate simplices of maximal dimension. If $X$ is empty, the result is obvious. Otherwise, we have a pushout
diagram 
$$ \xymatrix{ \bd \Delta^n \ar[r] \ar[d] & \Delta^n \ar[d] \\
X' \ar[r] & X. }$$
Using the inductive hypothesis, we deduce that $(\bd \Delta^n)_{\flat}$ and
$X'_{\flat}$ are good. Lemma \ref{jinker} implies that $\Delta^n_{\flat}$ is good. Invoking
$(b)$, we deduce that $X_{\flat}$ is good, as desired.
\end{itemize}
\end{proof}

\begin{proposition}\label{slape}
The functor $\sdd^{+}: \scSet \rightarrow \mset{ \Nerve(\cDelta)^{op} }$ carries bicategorical
equivalences of scaled simplicial sets to weak equivalences in $\mset{ \Nerve(\cDelta)^{op} }$ $($with respect to the complete Segal model structure$)$.
\end{proposition}

\begin{proof}
Let $f: (X,T) \rightarrow (Y,S)$ be a map of scaled simplicial sets. We wish to prove htat
$\sdd^{+}(f)$ is a weak equivalence. By virtue of Proposition \ref{campe}, it will suffice to show that
$F^{+}(f)$ is a weak equivalence with respect to the complete Segal model structure on
$\Fun( \cDelta^{op}, \mSet)$. This follows by applying Proposition \ref{kium} and
Lemma \ref{slappus} to the diagram
$$ \xymatrix{ F^{+}(X,T)_{\bigdot} \ar[r]^{F^{+}(f)} \ar[d] & F^{+}(Y,S)_{\bigdot} \ar[d] \\
{F'}^{+}(X,T)_{\bigdot} \ar[r]^{{F'}^{+}(f)} & {F'}^{+}(Y,S)_{\bigdot}. }$$
\end{proof}

\begin{corollary}
The functor $\sdd: \sSet \rightarrow (\sSet)_{/ \Nerve(\cDelta)^{op} }$ carries categorical
equivalences in $\sSet$ to weak equivalences in $(\sSet)_{/ \Nerve(\cDelta)^{op} }$
(with respect to the complete Segal model structure).
\end{corollary}

\begin{proof}
Let $f: X \rightarrow Y$ be a categorical equivalence of simplicial sets. Then
the induced map $f_{\sharp}: X_{\sharp} \rightarrow Y_{\sharp}$ is a bicategorical equivalence
(Proposition \ref{twop2}). It follows that $\sdd^{+}(f_{\sharp}): \sdd^{+} X_{\sharp} \rightarrow \sdd^{+} Y_{\sharp}$ is
a weak equivalence with respect to the complete Segal model structure on
$\mset{ \Nerve(\cDelta)^{op} }$. Since the forgetful functor
$F: \mset{ \Nerve(\cDelta)^{op} } \rightarrow (\sSet)_{/ \Nerve(\cDelta)}^{op}$ is a left
Quillen functor (Remark \ref{cooperr}), we conclude that $\sdd(f) = F \sdd^{+}(f_{\sharp})$ is a weak equivalence as desired.
\end{proof}

\begin{theorem}\label{cuti}
\begin{itemize}
\item[$(1)$]
The functor $\sdd^{+}$ is a left Quillen equivalence from $\scSet$ $($endowed with the bicategorical model structure$)$ to $\mset{ \Nerve(\cDelta)^{op} }$ $($endowed with the complete Segal model structure$)$.
\item[$(2)$] 
The functor $\sdd$ is a left Quillen equivalence from $\sSet$ $($endowed with the Joyal model structure$)$ to $\sSet_{/ \Nerve(\cDelta)^{op}}$ $($endowed with the complete Segal model structure$)$.
\end{itemize}
\end{theorem}

\begin{proof}
We will give the proof of $(1)$; the proof of $(2)$ is similar (but slightly easier). It is easy to see
that $\sdd^{+}$ preserves cofibrations and admits a right adjoint. Proposition \ref{slape} implies that $\sdd^{+}$ preserves weak equivalences. To complete the proof, it will suffice to show that
$\sdd^{+}$ induces an equivalence from the homotopy category $\h{ \scSet}$ to the homotopy
category $\h{ \mset{ \Nerve(\cDelta)^{op} }}$. Invoking Proposition \ref{curt}, we can reduce to proving that the functor $F^{+} = \lNerve^{+}_{\bigdot}( \cDelta^{op}) \circ \sdd^{+}$ induces an
equivalence from $\h{ \scSet}$ to $\h{ \Fun( \cDelta^{op}, \mSet)}$. Proposition \ref{kium} allows
us to replace $F^{+}$ by the functor ${F'}^{+}$, so that the desired result follows from Lemma \ref{slappus}. 
\end{proof}

\begin{corollary}[Joyal-Tierney]\label{JT}
Let $Q: \sSet \rightarrow \Fun( \cDelta^{op}, \sSet)$ be the functor described as follows:
for every simplicial set $X$, $Q(X)_{n}$ is the discrete simplicial set corresponding to the
set $\Hom_{\sSet}( \Delta^n, X)$ of $n$-simplices of $X$. Then
$Q$ is a left Quillen equivalence from $\sSet$ $($endowed with the Joyal model structure$)$
to $\Fun( \cDelta^{op}, \SSet)$ $($endowed with the complete Segal model structure$)$.
\end{corollary}

\begin{proof}
The functor $Q$ evidently has a right adjoint, which carries a simplicial object
$X_{\bigdot}$ of $\sSet$ to the simplicial set given by $[n] \mapsto \Hom_{\sSet}(\Delta^0, X_{n} )$.
It is easy to see that $Q$ preserves cofibrations. To complete the proof, it will suffice to show that
$Q$ preserves weak equivalences and induces an equivalence of homotopy categories.
To prove this, we first construct a natural transformation $\beta: F \rightarrow Q$
of functors from $\sSet$ to $\Fun( \cDelta^{op}, \sSet)$, which may be described as follows:
for every simplicial set $X$, $F(X)_{n}$ can be described as the $\Nerve( \calC)^{op}$, where
$\calC$ denotes the category of diagrams of the form
$$ \Delta^n \stackrel{f}{\rightarrow} \Delta^m \stackrel{g}{\rightarrow} X.$$
This category breaks up as a disjoint union of categories, indexed by the set
$\Hom_{\sSet}( \Delta^n, X)$ of $n$-simplices of $X$, and this decomposition induces a map of simplicial sets $(\beta_X)_n: \Nerve(\calC)^{op} \rightarrow Q(X)_n$. These maps depend functorially on $X$ and $n$, and therefore determine a natural transformation $\beta: F \rightarrow Q$
as desired. Note that the fiber of $(\beta_X)_n$ over an $n$-simplex $\sigma \in \Hom_{\sSet}( \Delta^n, X)$ is isomorphic to the nerve of a category $\Nerve( \calC_{\sigma})^{op}$, where
$\calC_{\sigma}$ denotes the category of diagrams of the form
$$ \xymatrix{ & \Delta^m \ar[dr] & \\
\Delta^n \ar[ur]^{f} \ar[rr]^{\sigma} & & X. }$$
This category evidently has an initial object, where we take $f$ to be the identity map.
It follows that $\Nerve( \calC_{\sigma})^{op}$ is weakly contractible for each $\sigma$, so that
$(\beta_X)_n$ is a weak homotopy equivalence. Consequently, $\beta_X$ is a weak
equivalence with respect to the injective model structure on $\Fun( \cDelta^{op}, \sSet)$, and in particular with respect to the complete Segal model structure. It will therefore suffice to show
that the functor $F$ preserves weak equivalences and induces an equivalence on homotopy categories. This follows from the factorization $F \simeq \lNerve_{\bigdot}( \cDelta^{op} ) \circ \sdd$ and
Theorem \ref{cuti}.
\end{proof}

\begin{corollary}\label{dumont}
Let $E: \Fun( \cDelta^{op}, \sSet) \rightarrow \mSet$ denote the functor given by the coend
$$ X_{\bigdot} \mapsto \int_{ \cDelta} X_{n}^{\sharp} \times (\Delta^n)^{\flat}.$$
Then $E$ determines a left Quillen equivalence from $\Fun( \cDelta^{op}, \sSet)$ 
$($endowed with the complete Segal model structure$)$ to $\mSet$.
\end{corollary}

\begin{proof}
We observe that $E$ has a right adjoint, given by the formula
$$X \mapsto ( [n] \mapsto \bHom_{ \mSet}^{\sharp}( (\Delta^n)^{\flat}, X)).$$
Let $\bfA$ denote the category $\Fun( \cDelta^{op}, \sSet)$ of bisimplicial sets, endowed with
the injective model structure. Note that this coincides with the Reedy model structure on
$\sSet$ (Example \toposref{tetsu}). Since the standard simplex $[n] \mapsto ( \Delta^n)^{\flat}$ is a Reedy cofibrant cosimplicial object of $\mSet$, Proposition \toposref{intreed} guarantees that
the functor $E: \bfA \rightarrow \mSet$ is a left Quillen functor. We wish to prove that $E$
is also a left Quillen functor with respect to the complete Segal model structure on
$\Fun( \cDelta^{op}, \sSet)$. In view of Proposition \toposref{stake}, this is equivalent to the assertion
that the right adjoint of $E$ carries fibrant objects of $\mSet$ to fibrant objects of 
$\Fun( \cDelta^{op}, \sSet)$. Let $X \in \mSet$ be fibrant, so that $X = ( \calC, \calE)$ where
$\calC$ is an $\infty$-category and $\calE$ is the collection of all equivalences in $\calC$.
Then the right adjoint of $E$ carries $X$ to the simplicial object $Y_{\bigdot}$ of $\sSet$,
where $Y_{n}$ is the largest Kan complex contained in the $\infty$-category
$\Fun( \Delta^n, \calC)$. Then $Y_{\bigdot}$ is injectively fibrant, and determines
a simplicial object $Z_{\bigdot}$ of $\SSet$; we wish to prove that $Z_{\bigdot}$ is a complete
Segal space object of $\SSet$. We first claim that $Z_{\bigdot}$ is a category object of $\SSet$.
In other words, we claim that for each $n \geq 0$, the canonical map $Z_{n} \rightarrow 
\rightarrow Z_1 \times_{Z_0} \ldots
\times_{Z_0} Z_1$ is a weak equivalence. This follows from the observation that the inclusion
$$ \Delta^{ \{0,1\} } \coprod_{ \{1\} } \Delta^{ \{1,2\} } \coprod_{ \{2\} } \ldots \coprod_{ \{n-1\} }
\Delta^{ \{n-1, n\} } \subseteq \Delta^n$$
is a categorical equivalence. Let $Z'_{\bigdot}$ be the underlying groupoid object of $Z_{\bigdot}$. Unwinding the definitions, we can identify $Z'_{\bigdot}$ with the simplicial object of $\SSet$ given by the formula
$$[n] \mapsto \bHom_{ \mSet}^{\sharp}( (\Delta^n)^{\sharp}, X) \subseteq \bHom_{ \mSet}^{\sharp}( (\Delta^n)^{\flat}, X).$$
Since every map $[m] \rightarrow [n]$ in $\cDelta$ induces a weak equivalence
$(\Delta^{m})^{\sharp} \rightarrow (\Delta^n)^{\sharp}$ of marked simplicial sets, we conclude
that $Z'_{\bigdot}$ is a constant groupoid object of $\SSet$, so that $Z_{\bigdot}$ is a complete Segal space object as desired. This completes the proof that $E$ is a left Quillen functor.

To prove that $E$ is a left Quillen equivalence, it will suffice (by virtue of Corollary \ref{JT}) to show that
the composite functor $Q \circ E: \sSet \rightarrow \mSet$ is a left Quillen equivalence; this is 
a special case of Theorem \toposref{bigdiag}.
\end{proof}

\begin{corollary}\label{presquare}
Let $f: \Nerve(\cDelta) \rightarrow \Cat_{\infty}$ be the functor induced by the 
Yoneda embedding $\cDelta \rightarrow \sSet$, and let $F: \calP( \Nerve(\cDelta)) \rightarrow \Cat_{\infty}$ be a functor which preserves small colimits such that the composition of $F$
with the Yoneda embedding $j: \Nerve(\cDelta) \rightarrow \calP( \Nerve(\cDelta) )$
is equivalent to $f$ (the functor $F$ exists and is unique up to equivalence, by virtue of
Theorem \toposref{charpresheaf}). Then $F$ admits a fully faithful right adjoint $G$.
Moreover, the essential image of $\calG$ consists precisely of the complete Segal space objects of $\SSet$. In particular, we have an equivalence $\Cat_{\infty} \rightarrow \CSS{\SSet}$.
\end{corollary}

\begin{proof}
Let $\bfA$ be the simplicial model category $\Fun( \cDelta^{op}, \sSet)$, endowed with the injective model structure. Let $F': \Nerve( \bfA^{\degree} ) \rightarrow \Nerve( (\mSet)^{\degree}) \simeq \Cat_{\infty}$ be the functor induced by the simplicial left Quillen functor $E$ of Corollary \ref{dumont}, so that $F'$ corresponds to $F$ under the equivalence $\phi: \Fun( \Nerve(\cDelta)^{op}, \SSet) \simeq \calP( \Nerve(\cDelta) )$ of Proposition \toposref{gumby444}. Corollary \ref{dumont} implies that $F'$ admits a fully faithful right adjoint whose essential image corresponds (under the equivalence $\phi$) to the full subcategory spanned by the complete Segal space objects.
\end{proof}

\subsection{Application: Automorphisms of $\Cat_{\infty}$}\label{cataut}\label{bisec4.4}

For every $\infty$-category $\calC$, the opposite simplicial set
$\calC^{op}$ is again an $\infty$-category. We will see below that the construction
$\calC \mapsto \calC^{op}$ determines an equivalence from the $\infty$-category
$\Cat_{\infty}$ to itself. In fact, this is essentially the {\em only} nontrivial self-equivalence of
$\Cat_{\infty}$. More precisely, we have the following result due to To\"{e}n (see \cite{toenchar}):

\begin{theorem}\label{cabbi}[To\"{e}n]
Let $\calE$ denote the full subcategory of $\Fun( \Cat_{\infty}, \Cat_{\infty})$ spanned by the equivalences. Then $\calE$ is equivalent to the $($nerve of the$)$ discrete category
$\{ \id, r \}$, where $r: \Cat_{\infty} \rightarrow \Cat_{\infty}$ is a functor which associates to every
$\infty$-category its opposite. 
\end{theorem}

Our goal in this section is to give a proof of Theorem \ref{cabbi}. We first outline the basic strategy. Fix an equivalence $f$ from $\Cat_{\infty}$ to itself. The first step is to argue that $f$ is determined by its restriction to a reasonably small subcategory of $\Cat_{\infty}$. To prove this, we will introduce the notion of a subcategory $\calC^{0} \subseteq \calC$ which {\it strongly generates} $\calC$
(Definitions \ref{coughball} and \ref{ballcough}). We will then show that $\Cat_{\infty}$ is strongly generated by the subcategory consisting of nerves of partially ordered sets (in fact, it is generated
by an even smaller subcategory: see Example \ref{upsquare}). This will allow us to reduce to the problem of understanding the category of self-equivalences of the category of partially ordered sets, which is easy to tackle directly: see Proposition \ref{cape}. 

We begin by introducing some definitions.

\begin{definition}\label{coughball}
Let $f: \calC \rightarrow \calD$ be a functor between $\infty$-categories.
We will say that $f$ is {\em strongly generates} the $\infty$-category $\calD$ if the identity transformation
$\id: f \rightarrow f$ exhibits the identity functor $\id_{\calD}$ as a left Kan extension of $f$
along $f$.
\end{definition}

\begin{remark}
In other words, a functor $f: \calC \rightarrow \calD$ strongly generates the $\infty$-category
$\calD$ if and only if, for every object $D \in \calD$, the evident diagram
$( \calC \times_{\calD} \calD_{/D} )^{\triangleright} \rightarrow \calD_{/D}^{\triangleright} \rightarrow \calD$
exhibits $D$ as a colimit of the diagram
$(\calC \times_{\calD} \calD_{/D}) \rightarrow \calC \stackrel{f}{\rightarrow} \calD.$
In particular, this implies that every object of $\calD$ can be obtained as the colimit of a diagram
which factors through $\calC$. Moreover, if $\calC$ is small and $\calD$ is locally small, then the diagram can be assumed small.
\end{remark}

\begin{remark}\label{copse}
Let $f: \calC \rightarrow \calD$ be a functor between $\infty$-categories, where $\calC$ is small, $\calD$ is locally small, and $\calD$ admits small colimits. In view of Theorem \toposref{charpresheaf}, we may assume without loss of generality that $f$ factors as a composition
$$ \calC \stackrel{j}{\rightarrow} \calP(\calC) \stackrel{F}{\rightarrow} \calD,$$
where $j$ denotes the Yoneda embedding and $F$ preserves small colimits.
Corollary \toposref{coughspaz} implies that $F$ has a right adjoint $G$, given by the composition
$$ \calD \stackrel{j'}{\rightarrow} \Fun( \calD^{op}, \SSet) \stackrel{\circ f}{\rightarrow} \calP(\calC),$$
where $j'$ denotes the Yoneda embedding for $\calD$; moreover, the transformation
$$ f = F \circ j \rightarrow (F \circ (G \circ F)) \circ j \simeq (F \circ G) \circ f$$
exhibits $(F \circ G)$ as a left Kan extension of $f$ along itself. It follows that $f$
strongly generates $\calD$ if and only if the counit map $F \circ G \rightarrow \id_{\calD}$ is an equivalence of functors. This is equivalent to the requirement that the functor $G$ is fully faithful.

In other words, the functor $f: \calC \rightarrow \calD$ strongly generates $\calD$ if and only if the
induced functor $F: \calP(\calC) \rightarrow \calD$ exhibits $\calD$ as a localization of
$\calP(\calC)$. In particular, the Yoneda embedding $\calC \rightarrow \calP(\calC)$
strongly generates $\calP(\calC)$, for any small $\infty$-category $\calC$.
\end{remark}

\begin{remark}\label{copo}
Let $f: \calC \rightarrow \calD$ be as in Remark \ref{copse}, and let $\calE$ be an $\infty$-category which admits small colimits. Let $\Fun^{0}(\calD, \calE)$ denote the full subcategory of
$\Fun(\calD, \calE)$ spanned by those functors which preserve small colimits. Then composition
with $f$ induces a fully faithful functor $\Fun^{0}(\calD, \calE) \rightarrow \Fun(\calC, \calE)$. 
This follows from Theorem \toposref{charpresheaf}, Proposition \toposref{unichar}, and Remark \ref{copse}.
\end{remark}

\begin{definition}\index{gen}{generates!strongly}\index{gen}{strongly generates}\label{ballcough}
Let $\calC$ be an $\infty$-category. We will say that a full subcategory $\calC^{0} \subseteq \calC$ 
{\it strongly generates} $\calC$ if the inclusion functor $\calC^{0} \rightarrow \calC$ strongly generates $\calC$, in the sense of Definition \ref{coughball}.
\end{definition}

\begin{remark}\label{sobre}
In other words, $\calC^0$ strongly generates $\calC$ if and only if the identity functor
$\id_{\calC}$ is a left Kan extension of $\id_{\calC} | \calC^{0}$. It follows from Proposition \toposref{acekan} that if $\calC^{0} \subseteq \calC^{1} \subseteq \calC$ are full subcategories and
$\calC^{0}$ strongly generates $\calC$, then $\calC^{1}$ also strongly generates $\calC$.
\end{remark}

\begin{example}
The $\infty$-category $\SSet$ of spaces is strongly generated by its final object;
this follows immediately from Remark \ref{copse}.
\end{example}

\begin{example}\label{upsquare}
The $\infty$-category $\Cat_{\infty}$ is strongly generated by the full subcategory consisting of the objects $\{ \Delta^n \}_{n \geq 0}$. This follows immediately from Corollary \ref{presquare}.
\end{example}

It follows from Remark \ref{sobre} and Example \ref{upsquare} that $\Cat_{\infty}$ is strongly generated by the full subcategory spanned by those $\infty$-categories of the form $\Nerve P$, where
$P$ is a partially ordered set. Our next step is to describe this subcategory in more intrinsic terms.

\begin{proposition}
Let $\calC$ be an $\infty$-category. The following conditions are equivalent:
\begin{itemize}
\item[$(1)$] The $\infty$-category $\calC$ is equivalent to the nerve of a partially ordered set $P$.
\item[$(2)$] For every $\infty$-category $\calD$ and every pair of functors
$F, F': \calD \rightarrow \calC$ such that $F(x) \simeq F'(x)$ for each object $x \in \calD$, 
the functors $F$ and $F'$ are equivalent as objects of $\Fun(\calD, \calC)$.
\item[$(3)$] For every $\infty$-category $\calD$, the map of sets
$$ \pi_0 \bHom_{\Cat_{\infty}}( \calD, \calC)
\rightarrow \Hom_{ \Set}( \pi_0 \bHom_{ \Cat_{\infty}}( \Delta^0, \calD),
\pi_0 \bHom_{ \Cat_{\infty} }( \Delta^0, \calC) )$$
is injective.
\end{itemize}
\end{proposition}

\begin{proof}
The implication $(1) \Rightarrow (2)$ is obvious, and $(3)$ is just a restatement of $(2)$.
Assume $(2)$; we will show that $(1)$ is satisfied. Let $P$ denote the collection of equivalence
classes of objects of $\calC$, where $x \leq y$ if the space $\bHom_{\calC}(x,y)$ is nonempty. 
There is a canonical functor $\calC \rightarrow \Nerve P$. To prove that this functor is an equivalence, it will suffice to show the following:
\begin{itemize}
\item[$(\ast)$] For every pair of objects $x,y \in \calC$, the space $\bHom_{\calC}(x,y)$ is either empty or contractible.
\end{itemize}
To prove $(\ast)$, we may assume without loss of generality that $\calC$ is the nerve of a fibrant
simplicial category $\overline{\calC}$. Let $x$ and $y$ be objects of $\overline{\calC}$ such that
the Kan complex $K = \bHom_{ \overline{\calC} }(x,y)$ is nonempty.
We define a new (fibrant) simplicial category $\overline{\calD}$ so that $\overline{\calD}$ consists of a pair of objects $\{x', y'\}$, with
$$ \bHom_{\overline{\calD}}(x',x') \simeq \bHom_{\overline{\calD}}(y',y') \simeq \Delta^0$$
$$ \bHom_{\overline{\calD}}(x',y') \simeq K \quad \bHom_{\overline{\calD}}(y',x') \simeq \emptyset.$$
We let $\overline{F}, \overline{F}': \overline{\calD} \rightarrow \overline{\calC}$ be simplicial functors
such that $\overline{F}(x') = \overline{F}'(x') = x$, $\overline{F}(y') = \overline{F}'(y')=y$, where
$\overline{F}$ induces the identity map from $\bHom_{\overline{\calD}}(x',y') = K = \bHom_{\overline{\calC}}(x,y)$ to itself, while $\overline{F}'$ induces a constant map from $K$ to itself.
Then $\overline{F}$ and $\overline{F}'$ induce functors $F$ and $F'$ from $\sNerve( \overline{\calD} )$ to $\calC$. It follows from assumption $(2)$ that the functors $F$ and $F'$ are equivalent, which implies that the identity map from $K$ to itself is homotopic to a constant map; this proves that $K$ is contractible. 
\end{proof}

\begin{corollary}\label{bigegg}
Let $\sigma: \Cat_{\infty} \rightarrow \Cat_{\infty}$ be an equivalence of $\infty$-categories, and let
$\calC$ be an $\infty$-category $($which we regard as an object of $\Cat_{\infty}${}$)$. Then $\calC$ is equivalent to the nerve of a partially ordered set if and only if $\sigma(\calC)$ is equivalent to the nerve of a partially ordered set.
\end{corollary}

\begin{lemma}\label{calcul}
Let $\sigma, \sigma' \in \{ \id_{\cDelta}, r \} \subseteq \Fun( \cDelta, \cDelta)$, where
$r$ denotes the reversal functor from $\cDelta$ to itself. Then
$$ \Hom_{ \Fun(\cDelta, \cDelta)}(\sigma, \sigma') = \begin{cases} \emptyset & \text{if } \sigma= \sigma' \\
\{ \id \} & \text{if } \sigma = \sigma'. \end{cases}$$
\end{lemma}

\begin{proof}
Note that $\sigma$ and $\sigma'$ are both the identity at the level of objects. 
Let $\alpha: \sigma \rightarrow \sigma'$ be a natural transformation. Then, for each $n \geq 0$,
$\alpha_{[n]}$ is a map from $[n]$ to itself. We claim that $\alpha_{[n]}$ is given by the formula
$$ \alpha_{[n]}(i) = \begin{cases} i & \text{if } \sigma = \sigma' \\
n-i & \text{if } \sigma \neq \sigma'.\end{cases}$$
To prove this, we observe that a choice of $i \in [n]$ determines a map $[0] \rightarrow [n]$, which allows us to reduce to the case $n=0$ (where the result is obvious) by functoriality.

It follows from the above argument that the natural transformation $\alpha$ is uniquely determined, if it exists. Moreover, $\alpha$ is a well-defined natural transformation if and only if each $\alpha_{[n]}$ is an order-preserving map from $[n]$ to itself; this is true if and only if $\sigma = \sigma'$.
\end{proof}

\begin{proposition}\label{cape}
Let $\calP$ denote the category of partially ordered sets, and let
$\sigma: \calP \rightarrow \calP$ be an equivalence of categories. Then
$\sigma$ is isomorphic either to the identity functor $\id_{\calP}$ or the functor
$r$ which carries every partially ordered set $X$ to the same set with the opposite ordering.
\end{proposition}

\begin{proof}
Since $\sigma$ is an equivalence of categories, it carries the final object $[0] \in \calP$ to itself (up to canonical isomorphism). It follows that for every
partially ordered set $X$, we have a canonical bijection of sets
$$ \eta_{X}: X \simeq \Hom_{ \calP}( [0], X) \simeq
\Hom_{\calP}( \sigma([0]), \sigma(X) ) \simeq \Hom_{\calP}( [0], \sigma(X)) \simeq \sigma(X).$$

We next claim that $\sigma([1])$ is isomorphic to $[1]$ as a partially ordered set.
Since $\eta_{[1]}$ is bijective, the partially ordered set $\sigma([1])$ has precisely two elements.
Thus $\sigma([1])$ is isomorphic either to $[1]$ or to a partially ordered set $\{x,y\}$ with two elements, neither larger than the other. In the second case, the set $\Hom_{\calP}( \sigma([1]), \sigma([1]))$
has four elements. This is impossible, since $\sigma$ is an equivalence of categories and
$\Hom_{\calP}( [1], [1])$ has only three elements. Let $\alpha: \sigma([1]) \rightarrow [1]$
be an isomorphism (automatically unique, since the ordered set $[1]$ has no automorphisms in $\calP$). 

The map $\alpha \circ \eta_{[1]}$ is a bijection from the set $[1]$ to itself. We will assume that
this map is the identity, and prove that $\sigma$ is isomorphic to the identity functor $\id_{\calP}$.
The same argument, applied to $\sigma \circ r$, will show that if $\alpha \circ \eta_{[1]}$ is not the identity, then $\sigma$ is isomorphic to $r$.

To prove that $\sigma$ is equivalent to the identity functor, it will suffice to show that for every
partially ordered set $X$, the map $\eta_{X}$ is an isomorphism of partially ordered sets.
In other words, we must show that both $\eta_{X}$ and $\eta_{X}^{-1}$ are maps of partially ordered sets. We will prove that $\eta_{X}$ is a map of partially ordered sets; the same argument, applied to 
an inverse to the equivalence $\sigma$, will show that $\eta^{-1}_{X}$ is a map of partially ordered sets.
Let $x,y \in X$ satisfy $x \leq y$; we wish to prove that $\eta_{X}(x) \leq \eta_{X}(y)$ in $\sigma(X)$.
The pair $(x,y)$ defines a map of partially ordered sets $[1] \rightarrow X$. By functoriality, we
may replace $X$ by $[1]$, and thereby reduce to the problem of proving that $\eta_{[1]}$ is a map of partially ordered sets. This follows from our assumption that $\alpha \circ \eta_{[1]}$ is the identity map.
\end{proof}

\begin{proof}[Proof of Theorem \ref{cabbi}]
Let $\calC$ be the full subcategory of $\Cat_{\infty}$ spanned by those $\infty$-categories
which are equivalent to the nerves of partially ordered sets, and let $\calC^{0}$ denote the full subcategory of $\calC$ spanned by the objects $\{ \Delta^n \}_{n \geq 0}$. 
Corollary \ref{bigegg} implies that every object $\sigma \in \calE$ restricts to an equivalence from $\calC$ to itself. According to Proposition \ref{cape}, $\sigma| \calC$ is equivalent either to the identity functor, or to the restriction $r|\calC$. In either case, we conclude that $\sigma$ also induces an equivalence from $\calC^{0}$ to itself.

Using Example \ref{upsquare} and Remark \ref{copo}, we deduce that the restriction functor
$\calE \rightarrow \Fun(\calC^{0}, \calC^{0})$ is fully faithful. In particular, any object $\sigma \in \calE$ is determined by the restriction $\sigma | \calC$, so that $\sigma$ is equivalent to either $\id$ or
$r$ by virtue of Proposition \ref{cape}. Since $\calC^{0}$ is equivalent to the nerve of the category
$\cDelta$, Lemma \ref{calcul} implies the existence of a fully faithful embedding from
$\calE$ to the nerve of the discrete category $\{ \id, r \}$. To complete the proof, it will suffice to show that this functor is essentially surjective. In other words, we must show that there exists a functor
$R: \Cat_{\infty} \rightarrow \Cat_{\infty}$ whose restriction to $\calC$ is equivalent to $r$. 

To carry out the details, it is convenient to replace $\Cat_{\infty}$ by an equivalent
$\infty$-category with a slightly more elaborate definition. Recall that $\Cat_{\infty}$ is defined to be the simplicial nerve of a simplicial category $\Cat_{\infty}^{\Delta}$, whose objects are $\infty$-categories, where $\bHom_{\Cat_{\infty}^{\Delta}}(X,Y)$ is the largest Kan complex contained in
$\Fun(X,Y)$. We would like to define $R$ to be induced by the functor $X \mapsto X^{op}$, but this
is not a simplicial functor from $\Cat_{\infty}^{\Delta}$ to itself; instead we have a canonical isomorphism $\bHom_{\Cat_{\infty}^{\Delta}}( X^{op}, Y^{op} )
\simeq \bHom_{\Cat_{\infty}^{\Delta}}(X,Y)^{op}.$
However, if we let $\Cat_{\infty}^{\top}$ denote the topological category obtained by geometrically realizing the morphism spaces in 
$\Cat_{\infty}^{\Delta}$, then $i$ induces an autoequivalence of $\Cat_{\infty}^{\top}$ as a topological category (via the natural homeomorphisms $| K | \simeq |K^{op}|$, which is defined for every simplicial set $K$). We now define $\Cat'_{\infty}$ to be the topological nerve of $\Cat_{\infty}^{\top}$ (see 
Definition \toposref{topnerve}). Then $\Cat'_{\infty}$ is an $\infty$-category equipped with a canonical equivalence $\Cat_{\infty} \rightarrow \Cat'_{\infty}$, and the involution $i$ induces
an involution $I$ on $\Cat'_{\infty}$, which carries each object $\calD \in \Cat'_{\infty}$ to 
the opposite $\infty$-category $\calD^{op}$. We now define $R$ to be the composition
$$ \Cat_{\infty} \rightarrow \Cat'_{\infty} \stackrel{I}{\rightarrow} \Cat'_{\infty} \rightarrow \Cat_{\infty},$$
where the last map is a homotopy inverse to the equivalence $\Cat_{\infty} \rightarrow \Cat'_{\infty}$.
It is easy to see that $R$ has the desired properties (moreover, we note that for every
object $\calD \in \Cat_{\infty}$, the image $R \calD$ is canonically equivalent with the opposite $\infty$-category $\calD^{op}$).
\end{proof}

\subsection{Application: Comparison of Universal Fibrations}\label{bisec4.x}\label{bisec4.5}

Let $\overline{S} = (S,T)$ be a scaled simplicial set. The scaled straightening and unstraightening
functors $\scSt_{\overline{S}}$ and $\scUn_{\overline{S}}$ of \S \ref{bisec3.3} are analogous
to the functors $\St_{S}$ and $\Un_{S}$ introduced in \S \toposref{strsec}. However,
it is difficult to relate them by a direct combinatorial construction. Our goal in this section is to show that they are nevertheless related by virtue of universal properties enjoyed by both constructions.
Our argument makes use of the classification of self-equivalences of $\Cat_{\infty}$
established in \S \ref{cataut}.

\begin{definition}\label{scorpus}
Let $\ICAT$ denote the scaled nerve $\scNerve( \mSet^{\degree})$, where
$\mSet^{\degree}$ denotes the full subcategory of $\mSet$ spanned by the fibrant objects
(viewed as an $\mSet$-enriched category). Then $\ICAT$ is a (large) scaled simplicial set, which is an $\infty$-bicategory by Remark \ref{coslai}. We will refer to $\ICAT$ as the
{\it $\infty$-bicategory of $\infty$-categories}. Note that the underlying $\infty$-category
of $\ICAT$ (obtained by discarding all simplices which contain non-thin faces as in Remark \ref{copus}) is canonically isomorphic to the $\infty$-category $\Cat_{\infty}$.

Let $\phi: \scCoNerve[ \ICAT ] \rightarrow \mSet$ be the composition of the counit map
with the inclusion $\mSet^{\degree} \subseteq \mSet$. It follows from Proposition \ref{curp} that
$\scUn_{ \ICAT}(\phi)$ is a fibrant object of $\mset{ \ICAT}$, which we can identify with a locally coCartesian fibration $q: \calZ \rightarrow \ICAT$. We will refer to $q$ as the {\em universal locally coCartesian fibration}.

Given a scaled simplicial set $\overline{S}= (S,T)$ and a fibrant object $\overline{X} \in \mset{ \overline{S}}$, we will say that $\overline{X}$ is {\it classified by} a map $\chi: \overline{S} \rightarrow \ICAT$
if $\overline{X}$ is isomorphic to $\scUn_{ \chi'}(\phi)$, in the homotopy category
$\h{\mset{\overline{S}}}$, where $\chi': \scCoNerve[ \overline{S}] \rightarrow \mSet^{\degree}$ is adjoint to $\chi$. In this case, we will also say that $\chi$ is classified by the fibrant object
$\overline{X} \in \mset{ \overline{S} }$, or by the underlying locally coCartesian fibration
$X \rightarrow S$.
\end{definition}

\begin{proposition}\label{urple}
Let $\overline{S}$ be a small marked simplicial set, and $\phi: \overline{S} \rightarrow \calC$ a weak equivalence of $\mSet$-enriched categories. Let $\calF, \calG: \calC \rightarrow (\mSet)^{\degree}$ be 
$\mSet$-enriched functors. The following conditions are equivalent:
\begin{itemize}
\item[$(1)$] The functors $\calF$ and $\calG$ are homotopic in the homotopy category of
(large) $\mSet$-enriched categories.
\item[$(2)$] The functors $\calF$ and $\calG$ are isomorphic in the homotopy category
$(\mSet)^{\calC}$.
\item[$(3)$] The objects $\scUn_{\phi} \calF$ and $\scUn_{\phi} \calG$ are isomorphic in the homotopy category of $\mset{ \overline{S} }$.
\item[$(4)$] The compositions $\scNerve(\calF) \circ \phi: \overline{S} \rightarrow \ICAT$
and $\scCoNerve(\calG) \circ \phi: \overline{S} \rightarrow \ICAT$ are homotopic 
in the model category of (large) scaled simplicial sets.
\end{itemize}
\end{proposition}

\begin{proof}
The equivalence of $(2) \Leftrightarrow (3)$ follows from Theorem \ref{cupper}, and
the equivalence $(1) \Leftrightarrow (4)$ from Theorem \ref{slai}. It will therefore suffice to
prove that $(1)$ and $(2)$ are equivalent. Suppose first that $(1)$ is satisfied.
Choose a weak equivalence $f: \calC' \rightarrow \calC$, where $\calC$ is cofibrant. 
Then the adjoint functors $(f_!, f^{\ast})$ determine a Quillen equivalence of
$(\mSet)^{\calC'}$ with $(\mSet)^{\calC}$. Consequently, it will suffice to show that
$f^{\ast} \calF$ and $f^{\ast} \calG$ are isomorphic in the homotopy category of
$(\mSet)^{\calC'}$. We may therefore replace $\calC$ by $\calC'$, and thereby reduce to the case where $\calC$ is cofibrant. Since $\calF$ and $\calG$ are homotopic, there exists a cylinder object
$$ \calC \coprod \calC \rightarrow \calC'' \stackrel{\pi}{\rightarrow} \calC$$
for $\calC$, such that the map $\calF \coprod \calG: \calC \coprod \calC \rightarrow (\mSet)^{\degree}$ extends to a functor $\calH: \calC'' \rightarrow (\mSet)^{\degree}$. Since $\pi$ is a weak equivalence, the adjoint functors $(\pi_{!}, \pi^{!})$ determine a Quillen equivalence of
$(\mSet)^{\calC''}$ with $(\mSet)^{\calC}$, so that there exists an object
$\calH_0 \in (\mSet)^{\calC}$ such that $\calH$ is isomorphic to $\pi^{\ast} \calH_0$
in the homotopy category of $(\mSet)^{\calC''}$. It follows that $\calF$ and $\calG$ are both
isomorphic (in the homotopy category of $(\mSet)^{\calC}$) to the functor $\calH_0$, and are therefore isomorphic to one another.

Now assume that $(2)$ holds. We define a (large) $\mSet$-enriched category $\calM$ as follows:
\begin{itemize}
\item The objects of $\calM$ are trivial fibrations $\alpha: \overline{X} \rightarrow \overline{Y}$
in $\mSet$, where $\overline{Y} \in \mSet$ is fibrant.
\item Given a pair of objects $\alpha: \overline{X} \rightarrow \overline{Y}$ and
$\beta: \overline{X}' \rightarrow \overline{Y}'$ in $\calM$, we let
$\bHom_{ \calM}(\alpha, \beta)$ denote the marked simplicial set
$$ \bHom_{\mSet}( \overline{X}, \overline{X}') \times_{ \bHom_{ \mSet}( \overline{X}, \overline{Y}')} \bHom_{ \mSet}( \overline{Y}, \overline{Y}').$$
\item Composition is defined in the evident way.
\end{itemize}
We have canonical projections $\pi_0, \pi_1: \calM \rightarrow (\mSet)^{\degree}$, given by the formula
$$\pi_i ( \alpha: \overline{X} \rightarrow \overline{Y} ) = \begin{cases} \overline{X} & \text{if } i=0 \\
\overline{Y} & \text{if } i=1.\end{cases}$$
Moreover, we have a diagonal embedding $\delta: (\mSet)^{\degree} \rightarrow \calM$, which carries an object $\overline{X} \in (\mSet)^{\degree}$ to the identity map $\id: \overline{X} \rightarrow \overline{X}$. We claim that $\pi_1$ is an equivalence of $\mSet$-enriched categories. The essential surjectivity of $\pi_1$ follows from the equation $\pi_1 \circ \delta = \id$; to prove that $\pi_1$ is fully faithful, we observe that the projection map
$$ \bHom_{\mSet}( \overline{X}, \overline{X}') \times_{ \bHom_{ \mSet}( \overline{X}, \overline{Y}')} \bHom_{ \mSet}( \overline{Y}, \overline{Y}') \rightarrow \bHom_{\mSet}( \overline{Y}, \overline{Y}')$$
is a trivial fibration whenever the map $\overline{X}' \rightarrow \overline{Y}'$ is a trivial fibration.
It follows from the two-out-of-three property that the maps $\delta$ and $\pi_0$ are also equivalences of $\mSet$-enriched categories.

Since $\calF$ and $\calG$ are fibrant objects of $(\mSet)^{\calC}$, there exists a fibrant-cofibrant object $\calE \in (\mSet)^{\calC}$ and a pair of trivial fibrations
$\calF \leftarrow \calE \rightarrow \calG$. The trivial fibration $\calE \rightarrow \calF$ determines a functor $\calH:  \calC \rightarrow \calM$, so that $\calE = \pi_0 \circ \calH$ and $\calF = \pi_1 \circ \calH$. We can therefore regard $\calH$ as a right homotopy from $\calE$ to $\calF$. The same argument shows that $\calE$ and $\calG$ are homotopic, so that $\calF$ is homotopic to $\calG$ by transitivity; this completes the proof of $(1)$.
\end{proof}

\begin{remark}\label{callie}
Let $\overline{S}$ be a scaled simplicial set, $\calC$ an $\infty$-bicategory, and
$f,g: \overline{S} \rightarrow \calC$ a pair of maps. Fix a contractible Kan complex $K$ containing a pair of distinct vertices $x$ and $y$. The following conditions are equivalent:
\begin{itemize}
\item The maps $f$ and $g$ are homotopic (with respect to the model structure on
$\scSet$ described in Theorem \ref{slai}). 
\item There exists a map $h: K_{\sharp} \times \overline{S} \rightarrow \calC$
such that $f = h| \{x\}_{\sharp} \times \overline{S}$ and
$g = h| \{y\}_{\sharp} \times \overline{S}$.
\end{itemize}
In this case, we will say that $h$ is a {\it homotopy} from $f$ to $g$. To establish the equivalence, it suffices to show that the diagram
$$ \overline{S} \coprod \overline{S} \simeq \{x,y\}_{\sharp} \times \overline{S}
\subseteq K_{\sharp} \times \overline{S} \rightarrow \overline{S}$$
exhibits $K_{\sharp} \times \overline{S}$ as a cylinder object for $\overline{S}$, which follows from
Lemma \ref{carpus} and Proposition \ref{twop2}.
\end{remark}

\begin{corollary}\label{preil}
Let $\overline{S} = (S,T)$ be a small marked simplicial set, and let
$\overline{X} \in \mset{ \overline{S}}$ be a (small) fibrant object. Then there exists
a map $\chi: \overline{S} \rightarrow \ICAT$ which classifies $\overline{X}$. Moreover,
$\chi$ is uniquely determined up to homotopy.
\end{corollary}

In other words, the ``classification'' relation of Definition \ref{scorpus} determines a bijection
between equivalence classes of fibrant objects of $\mset{ \overline{S} }$ and homotopy classes
of diagrams $\overline{S} \rightarrow \ICAT$, for every small scaled simplicial set $\overline{S}$. 
For technical purposes, it will be important to have a generalization of this result when $\overline{S}$ is not assumed to be small. To establish this generalization, we need to introduce a bit of notation.

\begin{notation}
Let $\kappa$ be an uncountable regular cardinal. We let $\ICAT^{\kappa}$ denote the
nerve $\scNerve( \mSet^{\degree, \kappa})$, where $\mSet^{\degree, \kappa}$ denotes the full subcategory of $\mSet$ spanned by the $\kappa$-small fibrant-cofibrant objects.
\end{notation}

\begin{corollary}\label{usqus}
Let $\overline{S} = (S,T)$ be a (small) marked simplicial set, let
$\overline{X} = (X,M)\in \mset{ \overline{S} }$ be a (small) fibrant object, and let $\kappa$ be an uncountable regular cardinal.
The following conditions are equivalent:
\begin{itemize}
\item[$(1)$] For every vertex $s$ of $S$, the fiber $X_{s}$ is essentially $\kappa$-small.
\item[$(2)$] The object $\overline{X}$ is classified by a map $\chi: \overline{S} \rightarrow \ICAT^{\kappa}$. 
\end{itemize}
Moreover, if these conditions are satisfied, the map $\chi$ is uniquely determined up to homotopy.
\end{corollary}

\begin{proof}
Let $\calC$ denote the scaled nerve $\scNerve(\calC_0)$, where $\calC_0$ is the full subcategory of
$\mSet$ spanned by the fibrant-cofibrant objects which are essentially $\kappa$-small.
Since the inclusion $\mSet^{\degree, \kappa} \subseteq \calC_0$ is a weak equivalence between fibrant $\mSet$-enriched categories, the inclusion $i: \ICAT^{\kappa} \subseteq \calC$ is
a weak equivalence between fibrant $\infty$-bicategories. Consequently, assertion
$(2)$ is equivalent to the following:
\begin{itemize}
\item[$(2')$] The object $\overline{X}$ is classified by a map $\chi': \overline{S} \rightarrow \calC$, which is determined uniquely up to homotopy. 
\end{itemize}
The equivalence of $(1)$ with $(2')$ follows immediately from Corollary \ref{preil}. Now suppose that $(1)$ is satisfied, so that there exist maps $\chi$ and $\chi'$ satisfying $(2)$ and $(2')$, respectively.
We wish to show that $\chi$ is uniquely determined up to homotopy. Since $i$ is a weak equivalence, it will suffice to show that $\chi'$ is uniquely determined up to homotopy. This follows immediately from the description of homotopies supplied by Remark \ref{callie}.
\end{proof}

\begin{corollary}
Let $\overline{S}=(S,T)$ be a scaled simplicial set (not necessarily small), and let
$\overline{X} \in \mset{ \overline{S}}$ be 
$\overline{S}$-fibered. The following conditions are equivalent:
\begin{itemize}
\item For every vertex $s$ in $S$, the fiber $X_{s}$ is essentially small.
\item The object $\overline{X}$ is classified by a map $\chi: \overline{S} \rightarrow \ICAT$.
\end{itemize}
Moreover, in this case, $\chi$ is determined uniquely up to homotopy.
\end{corollary}

\begin{proof}
This is a special case of Corollary \ref{usqus}, applied in a larger universe.
\end{proof}

\begin{corollary}\label{snoball}
Let $S$ be a simplicial set (not necessarily small), and let $p: X \rightarrow S$ be a locally coCartesian fibration. The following conditions are equivalent:
\begin{itemize}
\item For every vertex $s \in S$, the fiber $X_{s}$ is essentially small.
\item The object $(X,M) \in \mset{ S_{\flat}}$ is classified by a map $\chi: S_{\flat} \rightarrow \ICAT$. Here $M$ denotes the collection of all locally $p$-coCartesian edges of $X$.
\end{itemize}
If these conditions are satisfied, then $\chi$ is uniquely determined. Moreover, 
$p$ is a coCartesian fibration if and only if $\chi$ factors through $S_{\sharp}$.
\end{corollary}

Suppose that $p: X \rightarrow S$ is a Cartesian fibration of simplicial sets
whose fibers are essentially small. Corollary \ref{snoball} implies that
$p$ is classified by a map of scaled simplicial sets $S_{\sharp} \rightarrow \ICAT$, which we can identify with a map of ordinary simplicial sets $\chi: S \rightarrow \Cat_{\infty}$. In this case, we will also say that $\chi$ {\it classifies} the coCartesian fibration $p$, or that $p$ classifies the map $\chi$.
However, there is now some danger of confusion: in \S \toposref{universalfib}, we said that a 
coCartesian fibration $p: X \rightarrow S$ is {\em classified by} a map $\chi: S \rightarrow \Cat_{\infty}$ if the Cartesian fibration $X^{op} \rightarrow S^{op}$ is classified by $\chi$, in the sense of 
Definition \toposref{classer}. These definitions are {\em not} equivalent. Instead, we have the following:

\begin{proposition}\label{juspus}
Let $p: X \rightarrow S$ be a coCartesian fibration of simplicial sets, and assume that the fibers of $p$ are essentially small. Let $\chi: S_{\sharp} \rightarrow \ICAT$ be a map of scaled simplicial sets
which classifies the object $X^{\natural} \in \mset{ S_{\sharp}}$ (in the sense of Definition \ref{scorpus}), corresponding to a map of simplicial sets $\chi_0: S \rightarrow \Cat_{\infty}$. Let
$\xi: S \rightarrow \Cat_{\infty}$ be a functor which classifies the Cartesian fibration
$X^{op} \rightarrow S^{op}$, in the sense of Definition \toposref{classer}. Then
$\chi_0$ is homotopic to $r \circ \xi$, where $r: \Cat_{\infty} \rightarrow \Cat_{\infty}$
is the functor which carries each $\infty$-category to its opposite (up to homotopy;
see Theorem \ref{cabbi}).
\end{proposition}

\begin{proof}
Let $S = \Cat_{\infty}$ and let $p: X \rightarrow S$ be the universal coCartesian fibration.
In this case, the classifying functors $\chi_0, \xi: S \rightarrow \Cat_{\infty}$ are equivalences of $\infty$-categories so that $\chi_0$ is homotopic to $\psi \circ \xi$ for some equivalence
$\psi$ from $\Cat_{\infty}$ to itself. It follows by functoriality that $\chi_0$ is homotopic to $\psi \circ \xi$ for {\em any} coCartesian fibration $p: X \rightarrow S$ with essentially small fibers.
Theorem \ref{cabbi} implies that $\psi$ is equivalent to either $r$ or the identity functor
$\id: \Cat_{\infty} \rightarrow \Cat_{\infty}$. To complete the proof, it will suffice to show that the latter case is impossible. To see this, we take $S$ to consist of a single vertex. In this case, we can
identify $\chi_0$ with the vertex of $\Cat_{\infty}$ corresponding to the $\infty$-category
$X$, while $\xi$ corresponds to the vertex of $\Cat_{\infty}$ corresponding to the $\infty$-category
$X^{op}$. In general the $\infty$-categories $X$ and $X^{op}$ are not equivalent. 
\end{proof}

\begin{remark}
The proof of Proposition \ref{juspus} is somewhat unsatisfying: it proceeds via a classification
of all equivalences of $\Cat_{\infty}$ with itself, rather than by establishing a direct relationship between the functors $\chi_0$ and $\xi$. It seems difficult to describe this relationship using the simplicial language presented here. The classifying functor $\chi_0$ is defined using the
scaled straightening functor of Definition \ref{sputer2}, while $\xi$ is defined in
terms of the ordinary straightening functor defined in \S \toposref{markmodel2}. These
straightening functors play similar philosophical roles, but are implemented differently in the formalism of simplicial sets and do not seem to be directly comparable to one another.
\end{remark}

\section{The Goodwillie Calculus}\label{bisecE}

Let $\calC$ be a presentable pointed $\infty$-category. Then we can consider also the $\infty$-category $\Spectra(\calC)$ of spectrum objects of $\calC$ (see \S 
\stableref{stable9.1}). The following question arises naturally: to what extent can $\Spectra(\calC)$ be regarded as a functor of $\calC$? For example, suppose that $F: \calC \rightarrow \calD$ is
a functor between pointed presentable $\infty$-categories. Under what conditions does
$F$ determine a functor from $\Spectra(\calC)$ to $\Spectra(\calD)$?
The most obvious case to consider is when the functor $F$ is left exact. In this case, composition
with $F$ determines a $f: \Spectra(\calC) \rightarrow \Spectra(\calD)$.
This functor $f$ fits into a commutative diagram
$$ \xymatrix{ \Spectra(\calC) \ar[r]^{f} \ar[d]^{ \Omega^{\infty}_{\calC} } & \Spectra(\calD) \ar[d]^{\Omega^{\infty}_{\calD}} \\
\calC \ar[r]^{F} & \calD. }$$

There is a dual situation which is equally important. Suppose that the functor $F$
preserves small colimits. Applying Corollary \toposref{adjointfunctor}, we deduce that $F$ admits
a right adjoint $G$. Since $G$ is left exact, we can apply the above reasoning to obtain an
induced functor $g: \Spectra(\calD) \rightarrow \Spectra(\calC)$. We can use Corollary
\toposref{adjointfunctor} again to deduce that $g$ admits a left adjoint $f$.
We can regard $f$ as an ``extension'' of $F$, in the
sense that the diagram
$$ \xymatrix{ \Spectra(\calC) \ar[r]^{f} & \Spectra(\calD) \\
\calC \ar[u]^{\Sigma^{\infty}_{\calC} } \ar[r]^{F} & \calD \ar[u]^{\Sigma^{\infty}_{\calD} } }$$
commutes up to homotopy. 

This raises a number of questions. For example, suppose that a functor $F: \calC \rightarrow \calD$ preserves small colimits {\em and} finite limits. In this case, we can apply either of the above constructions to produce a functor $\Spectra(\calC) \rightarrow \Spectra(\calD)$: do the resulting functors coincide up to homotopy? On the other hand, suppose that $F$ satisfies neither condition; can one still hope to find an exact functor $f: \Spectra(\calC) \rightarrow \Spectra(\calD)$ which is somehow related to $F$?

To answer these questions, let us recall a bit of terminology. A functor $F: \calC \rightarrow \calD$
is {\it excisive} if $F$ carries zero objects of $\calC$ to zero objects of $\calD$, and carries
pushout diagrams in $\calC$ to pullback diagrams in $\calC$. (Our terminology here is
somewhat nonstandard; most authors do not require the preservation of zero objects in the
definition of an excisive functor.) According to Proposition \ref{urbusk}, the
$\infty$-category of colimit preserving functors $\LFun( \Spectra(\calC), \Spectra(\calD) )$
is equivalent to the full subcategory of $\Fun(\calC, \calD)$ spanned by the
excisive functors which preserve filtered colimits. Consequently, we may rephrase our problem as follows: given a functor $F: \calC \rightarrow \calD$, can we choose an excisive functor
$F': \calC \rightarrow \calD$ which preserves filtered colimits and is, in some sense, a best approximation to the original functor $F$? 

The Goodwillie calculus allows us to address these questions, provided that some mild assumptions are satisfied. Suppose that $\calC$ and $\calD$ be {\it well-pointed} $\infty$-categories (see Definition \ref{wellp}). Let $\Fun_{0}(\calC), \calD)$ denote the full subcategory of
$\Fun(\calC, \calD)$ spanned by those functors which preserve zero objects and sequential colimits, and
let $\Fun_{\exc}(\calC, \calD)$ denote the full subcategory of $\Fun_0(\calC, \calD)$ spanned by
the excisive functors. Then the inclusion $\Fun_{\exc}(\calC, \calD) \subseteq \Fun_0(\calC, \calD)$
admits a left adjoint (Corollary \ref{leftywing}), which carries a functor $F \in \Fun_0(\calC, \calD)$
to its {\it $($Goodwillie$)$ derivative} $DF = \varinjlim_{n} \Omega^{n}_{\calD} \circ F \circ \Sigma^{n}_{\calC}$. In this case, we can write $DF$ as a composition $\Omega^{\infty}_{\calD} \circ f \circ \Sigma^{\infty}_{\calC}$, where $f: \Spectra(\calC) \rightarrow \Spectra(\calD)$ is an exact functor which we call
the {\it linearization} of $f$. Finally, there is a {\em chain rule}: given a composable pair of functors
$F: \calC \rightarrow \calD$ and $G: \calD \rightarrow \calE$ with linearizations
$f: \Spectra(\calC) \rightarrow \Spectra(\calD)$ and $g: \Spectra(\calD) \rightarrow \Spectra(\calE)$,
we can identify the composition $g \circ f$ with the linearization of $G \circ F$ (Proposition \ref{clambake}).

Our goal in this section is to give an overview of some rudimentary parts of the Goodwillie calculus (specifically, the theory of first derivatives) using the language of $\infty$-categories. (For
a more comprehensive study of the Goodwillie calculus in the classical setting,
we refer the reader to Goodwillie's work (\cite{goodwillieI}, \cite{goodwillieII}, and
\cite{goodwillie}); see also \cite{klein} and \cite{chingarone} for a discussion of the chain rule).
We will begin in \S \ref{bisec5.1} by defining the linearization and Goodwillie derivative of a (well-pointed) functor $F: \calC \rightarrow \calD$, and establishing some of their basic formal properties.
In \S \ref{bisec5.2} we will discuss the situation more systematically by introducing a notion of
linearization for locally coCartesian fibrations: this can be regarded as a parametrized version of the
stabilization construction $\calC \mapsto \Stab(\calC)$. We will apply this parametrized linearization
construction in \S \ref{bisec5.3} to reformulate the theory of first derivatives using the $(\infty,2)$-categorical language introduced earlier in this paper. Finally, in \S \ref{bisec5.4} we will consider
consider the relationship between the linearization of a functor $F: \calC \rightarrow \calD$ and the
linearization of its adjoint $G$.

\subsection{Derivatives and Linearizations of Functors}\label{linf}\label{bisec5.1}

Our goal in this section is to describe some of the basic notions from
Goodwillie's calculus of functors. With an eye toward future applications, we will try to work
in as general a setting is possible. We begin by axiomatizing the types of
$\infty$-categories and functors to which the calculus can be applied.

\begin{definition}\label{wellp}
We will say that an $\infty$-category $\calC$ is {\it well-pointed}
if the following conditions are satisfied:
\begin{itemize}
\item[$(1)$] The $\infty$-category $\calC$ is pointed.
\item[$(2)$] The $\infty$-category $\calC$ admits finite limits. In particular,
the loop functor $\Omega_{\calC}: \calC \rightarrow \calC$ is well-defined.
\item[$(3)$] The $\infty$-category $\calC$ admits countable colimits.
\item[$(4)$] The loop functor $\Omega_{\calC}: \calC \rightarrow \calC$ preserves
sequential colimits.
\end{itemize}
\end{definition}

\begin{definition}
We will say that a functor $F: \calC \rightarrow \calD$ between $\infty$-categories
is {\it well-pointed} if the following conditions are satisfied:
\begin{itemize}
\item[$(1)$] The $\infty$-categories $\calC$ and $\calD$ are well-pointed.
\item[$(2)$] The functor $F$ preserves zero objects and sequential colimits.
\end{itemize}
\end{definition}

\begin{notation}
If $F: \calC \rightarrow \calD$ is a well-pointed functor between $\infty$-categories,
then we let $F^{+}: \PreSpectra(\calC) \rightarrow \PreSpectra(\calD)$ denote the functor
given by composition with $F$.
\end{notation}

\begin{proposition}\label{coomer}
Let $F: \calC \rightarrow \calD$ be a well-pointed functor. Let $L_{\calC}: \PreSpectra(\calC) \rightarrow \Spectra(\calC)$ and $L_{\calD}: \PreSpectra(\calD) \rightarrow \Spectra(\calD)$ denote left adjoints to the inclusions.
Suppose that $\alpha: X \rightarrow X'$ is a map of prespectrum objects of $\calC$ such that
$L_{\calC}( \alpha)$ is an equivalence in $\Spectra(\calC)$. Then $L_{\calD} F^{+}(\alpha)$ is an equivalence in $\Spectra(\calD)$.
\end{proposition}

\begin{proof}
Let us say that a morphism $\alpha$ in $\PreSpectra(\calC)$ is {\it good} if
the induced map $L_{\calD} F^{+}(\alpha)$ is an equivalence in $\Spectra(\calD)$.
We wish to prove that if $L_{\calC}(\alpha)$ is an equivalence, then $\alpha$ is good.
The proof proceeds in several steps.
\begin{itemize}
\item[$(a)$] Let $\alpha: X \rightarrow X'$ be a map of prespectrum objects of $\calC$ which
induces equivalences $X[n] \rightarrow X'[n]$ for $n \gg 0$. Then the map $F^{+}(\alpha)$
has the same property, so that $L_{\calD} F^{+}(\alpha)$ is an equivalence by virtue of
Remark \stableref{jilly}. It follows that $\alpha$ is a good morphism.
\item[$(b)$] Let $\alpha: X \rightarrow X'$ be such that $L_{\calC}(\alpha)$ is an equivalence. We have a commutative diagram $$ \xymatrix{ X \ar[r]^{\alpha} \ar[d] & X' \ar[d] \\
L_{\calC} X \ar[r] & L_{\calC} X', }$$
where the bottom horizontal map is an equivalence (and therefore good). Since the collection of
good morphisms in $\PreSpectra(\calC)$ has the two-out-of-three property, it will suffice to prove
that the vertical maps are good.
\item[$(c)$] Since the functors $F^{+}$ and $L_{\calD}$ commute with sequential colimits,
the collection of good morphisms in $\PreSpectra(\calC)$ is stable under sequential colimits.
\item[$(d)$] Corollary \stableref{postzing} implies that the localization functor $L_{\calC}$
can be written as the colimit of the sequence of functors $L_{n}: \PreSpectra(\calC) \rightarrow
\PreSpectra(\calC)$ described in Corollary \stableref{camber}. It will therefore suffice to show that
each of the maps $X \rightarrow L_n X$ is good. This follows from $(a)$.
\end{itemize}
\end{proof}

\begin{definition}\label{cool}
Let $F: \calC \rightarrow \calD$ be a well-pointed functor, let $f: \Spectra(\calC) \rightarrow \Spectra(\calD)$ be an arbitrary functor. Let $\alpha: F^{+}| \Spectra(\calC) \rightarrow f$ be a natural transformation of functors from $\Spectra(\calC)$ to $\PreSpectra(\calD)$. We will say that $\alpha$
{\it exhibits $f$ as a linearization of $F^{+}$} if, for every object $X \in \Spectra(\calC)$, the induced map $\alpha_{X}: F^{+}(X) \rightarrow f(X)$ exhibits $f(X)$ as a $\Spectra(\calD)$-localization of
$F^{+}(X)$.
\end{definition}

\begin{remark}\label{coola}
In the situation of Definition \ref{cool}, there exists a natural transformation
$\alpha: F^{+}| \Spectra(\calC) \rightarrow f$ which exhibits $f$ as a linearization of $F$,
and the functor $f$ is uniquely determined up to equivalence. For example, we can
take $f = (L_{\calD} \circ F^{+}) | \Spectra(\calC)$, where $L_{\calD}$ denotes a left
adjoint to the inclusion $\Spectra(\calD) \subseteq \PreSpectra(\calD)$.
\end{remark}

\begin{remark}\label{sweety}
Let $F: \calC \rightarrow \calD$ be a well-pointed functor.
It follows from Proposition \ref{coomer}, Proposition \toposref{unlap}, and Remark \ref{coola} that the composition $L_{\calD} \circ F^{+}: \PreSpectra(\calC) \rightarrow \Spectra(\calD)$ is equivalent to $f \circ L_{\calC}$, where $f$ denotes a linearization of $F$.
\end{remark}

\begin{proposition}\label{exman}
Let $F: \calC \rightarrow \calD$ be a well-pointed functor, and let $\alpha: F^{+} \rightarrow f$ exhibit $f$ as a linearization of $F$. Then $f$ is an exact functor from $\Spectra(\calC)$ to $\Spectra(\calD)$.
\end{proposition}

\begin{proof}
In view of Corollary \stableref{cuuple}, it will suffice to show that for every object
$X \in \Spectra(\calC)$, the canonical map $\Sigma_{ \Spectra(\calD)} f(X) \rightarrow
f( \Sigma_{ \Spectra(\calC)} X)$ is an equivalence in $\Spectra(\calD)$. Let
$L_{\calC}: \PreSpectra(\calC) \rightarrow \Spectra(\calC)$ denote a left adjoint to the inclusion, and let $L_{\calD}$ be defined similarly. Without loss of generality we may assume $X = L_{\calC} Y$
for some $Y \in \PreSpectra(\calC)$. According to Proposition \stableref{makesig}, we have a pushout diagram
$$ \xymatrix{ L_{\calC} Y \ar[r] \ar[d] & L_{\calC} S_{+}(Y) \ar[d] \\
L_{\calC} S_{-}(Y) \ar[r] & L_{\calC} S(Y) }$$
in $\Spectra(\calC)$ (see Notation \stableref{corpae}); we wish to show that the induced diagram
$$ \xymatrix{ f L_{\calC} Y \ar[r] \ar[d] & f L_{\calC} S_{+}(Y) \ar[d] \\
f L_{\calC} S_{-}(Y) \ar[r] & f L_{\calC} S(Y) }$$
is a pushout diagram in $\Spectra(\calD)$.
Remark \ref{sweety} implies that $f L_{\calC} \simeq L_{\calD} F^{+}$. We observe
that we have identities 
$$F^{+} S(Y) = S( F^{+} Y) \quad \quad F^{+} S_{+}(Y) = S_{+}( F^{+} Y)
\quad \quad F^{+} S_{-}(Y) = S_{-}( F^{+} Y).$$
It will therefore suffice to show that the diagram
$$ \xymatrix{ F^{+} Y \ar[r] \ar[d] & S_{+}( F^{+} Y) \ar[d] \\
S_{-}( F^{+} Y) \ar[r] & S( F^{+} Y) }$$
becomes a pushout square after applying the functor $L_{\calD}$, which follows from
Proposition \stableref{makesig}.
\end{proof}

The following result characterizes the linearization of a functor $F: \calC \rightarrow \calD$ by a universal property:

\begin{proposition}\label{swuft}
Let $F: \calC \rightarrow \calD$ be a well-pointed functor, let $\alpha: F^{+} \rightarrow f$ exhibit
$f$ as a linearization of $F$, and let $g: \Spectra(\calC) \rightarrow \Spectra(\calD)$ be any exact functor. Let $L_{\calC}$ and $L_{\calD}$ denote left adjoints to the inclusions
$\Spectra(\calC) \subseteq \PreSpectra(\calC)$ and $\Spectra(\calD) \subseteq \PreSpectra(\calD)$.
Then the composite map
\begin{eqnarray*}
\bHom_{ \Fun( \Spectra(\calC), \Spectra(\calD)}(f,g) & \stackrel{\phi_0}{\rightarrow} & 
\bHom_{ \Fun( \PreSpectra(\calC), \Spectra(\calD)}( f \circ L_{\calC}, g \circ L_{\calC} ) \\
& \stackrel{\phi_1}{\simeq} & \bHom_{ \Fun( \PreSpectra(\calC), \Spectra(\calD)}( L_{\calD} \circ F^{+}, g \circ L_{\calC} ) \\
& \stackrel{\phi_2}{\rightarrow} & \bHom_{ \Fun( \PreSpectra(\calC), \PreSpectra(\calD)}( F^{+}, g \circ L_{\calC}) \\
& \rightarrow & \bHom_{ \Fun(\calC, \calD)}( \Omega^{\infty}_{\calD} \circ F^{+} \circ \widetilde{\Sigma}^{\infty}_{\calC}, \Omega^{\infty}_{\calD} \circ g \circ L_{\calC} \circ \widetilde{\Sigma}^{\infty}_{\calC}) \\
& \simeq & \bHom_{ \Fun(\calC, \calD)}( F, \Omega^{\infty}_{\calD} \circ g \circ \Sigma^{\infty}_{\calC})
\end{eqnarray*}
is a homotopy equivalence. Here $\phi_1$ denotes the equivalence provided by Remark \ref{sweety}.
\end{proposition}

\begin{proof}
We first observe that $g \circ L_{\calC}$ is a right Kan extension of the restriction
$(g \circ L_{\calC}) | \Spectra(\calC) \simeq g$. This proves that $\phi_0$ is a homotopy
equivalence. The map $\phi_1$ is a homotopy equivalence by construction, and the map
$\phi_2$ is a homotopy equivalence because the image of $g \circ L_{\calC}$ is contained in the
essential image of the localization functor $L_{\calD}$ (namely, $\Spectra(\calD)$). 

For every integer $n$, let $\tau_n: \PreSpectra(\calD) \rightarrow \PreSpectra^{n}_{-\infty}(\calD)$
denote the restriction functor. The mapping space
$\bHom_{ \Fun( \PreSpectra(\calC), \PreSpectra(\calD))}( F^{+}, g \circ L_{\calC})$ can be described as the homotopy inverse limit of the tower of spaces
$$ \{ Z_n = \bHom_{ \Fun( \PreSpectra(\calC), \PreSpectra^{n}_{- \infty}(\calD)}( \tau_n \circ F^{+},
\tau_n \circ g \circ L_{\calC}) \}.$$
It will therefore suffice to prove the following:
\begin{itemize}
\item[$(1)$] Composition with $\widetilde{\Sigma}^{\infty}_{\calC}$ and
the map $e: \PreSpectra^{0}_{- \infty}(\calD) \rightarrow \calD$ given by evaluation at $(0,0)$ 
gives a homotopy equivalence
$$Z_0 \rightarrow \bHom_{ \Fun( \calC, \calD)}( e \circ \tau_0 \circ F^{+} \circ \widetilde{\Sigma}^{\infty}_{\calC},
e \circ \tau_0 \circ g \circ L_{\calC} \circ \widetilde{\Sigma}^{\infty}_{\calC})\simeq
\bHom_{ \Fun(\calC, \calD)}( F, \Omega^{\infty}_{\calD} \circ g \circ \Sigma^{\infty}_{\calC} ).$$
\item[$(2)$] For each $n > 0$, the restriction map $Z_{n} \rightarrow Z_{n-1}$ is a homotopy equivalence.
\end{itemize}

Our next step is to compute the homotopy types of the spaces $Z_n$. 
We will employ the conventions of Notation \stableref{jirl}.
\begin{itemize}
\item[$(a)$] Let $\tau: \PreSpectra(\calD) \rightarrow \PreSpectra^{n}_{n}(\calD)$
denote the restriction map. Using Lemma \stableref{atwas2}, we deduce that
$\tau_n \circ g$ takes values among those functors $\Nerve( Q(-\infty,n)) \rightarrow \calD$
which are right Kan extensions of their restrictions to $\Nerve( Q(n,n) )$. It
follows that the restriction map
$$ Z_n \rightarrow \bHom_{ \Fun( \PreSpectra(\calC), \PreSpectra^{n}_{n}(\calD))}(
\tau \circ F^{+}, \tau \circ g \circ L_{\calC})$$
is a homotopy equivalence.
\item[$(b)$] Combining $(a)$ with Lemma \stableref{atwas3}, we deduce that
the restriction map
$$ Z_{n} \rightarrow \bHom_{ \Fun( \PreSpectra(\calC), \calD)}( \Omega^{\infty-n}_{\calD} \circ F^{+},
\Omega^{\infty-n}_{\calD} \circ g \circ L_{\calC}) =
\bHom_{ \Fun( \PreSpectra(\calC), \calD)}( F \circ \Omega^{\infty-n}_{\calC}, \Omega^{\infty-n}_{\calD} \circ g \circ L_{\calC} )$$
is a homotopy equivalence.
\item[$(c)$] Let $\calE$ denote the full subcategory of
$\PreSpectra(\calC)$ spanned by those functors which are suspension prespectra above $n$.
Using Lemma \stableref{atwas4}, we deduce that $F \circ \Omega^{\infty-n}_{\calC}$ is a left
Kan extension of its restriction to $\calE$. Combining this
observation with $(b)$, we deduce that the restriction map
$$ Z_{n} \rightarrow \bHom_{\Fun(\calE, \calD)}( 
F \circ \Omega^{\infty-n}_{\calC}, \Omega^{\infty-n}_{\calD} \circ g \circ L_{\calC} )$$
is a homotopy equivalence.
\item[$(d)$] Let $\calE_0$ denote the full subcategory of
$\calE$ spanned by the $n$-suspension prespectra.
Using Lemma \stableref{atwas2} and Remark \stableref{jilly}, we deduce that 
$\Omega^{\infty-n}_{\calD} \circ g \circ L_{\calC}| \calE$ is a right Kan extension of its restriction to $\calE_0$. Combining this observation with
$(c)$, we deduce that the restriction map
$$ Z_n \rightarrow \bHom_{ \Fun( \calE_0, \calD)}( 
F \circ \Omega^{\infty-n}_{\calC}, \Omega^{\infty-n}_{\calD} \circ g \circ L_{\calC} )$$
is a homotopy equivalence.
\item[$(e)$] According to Proposition \stableref{swagger}, evaluation at $(n,n)$ induces a trivial Kan
fibration $\calE_0 \rightarrow \calC$, which admits a section
$\widetilde{\Sigma}^{\infty-n}_{\calC}$. Combining this observation with
$(d)$, we deduce that the map
$$ Z_{n} \rightarrow \bHom_{ \Fun(\calC, \calD)}( F \circ \Omega^{\infty-n}_{\calC}
\circ \widetilde{\Sigma}^{\infty-n}_{\calC},
\Omega^{\infty-n}_{\calD} \circ g \circ L_{\calC} \circ \widetilde{\Sigma}^{\infty-n}_{\calC})
= \bHom_{ \Fun(\calC, \calD)}( F, \Omega^{\infty-n}_{\calD} \circ g \circ \Sigma^{\infty-n}_{\calD})$$
is a homotopy equivalence.
\end{itemize}
We note that in the case $n=0$, the map described by $(e)$ agrees with the map appearing
in statement $(1)$, which is therefore a homotopy equivalence as desired. To prove $(2)$,
we observe that the restriction map $Z(n) \rightarrow Z(n-1)$ can be identified with a map
$$ \psi: \bHom_{ \Fun(\calC, \calD)}( F, \Omega^{\infty-n}_{\calD} \circ g \circ \Sigma^{\infty-n}_{\calD})
\rightarrow \bHom_{ \Fun(\calC, \calD)}( F, \Omega^{\infty-n+1}_{\calD} \circ g \circ \Sigma^{\infty-n+1}_{\calD}),$$
which is well-defined up to homotopy. We wish to show that $\psi$ is a homotopy equivalence.
Unwinding the definitions, we observe that $\psi$ is induced by composition with the canonical natural transformation $g \rightarrow \Omega_{ \Spectra(\calD)} \circ g \circ \Sigma_{ \Spectra(\calD) }$.
This transformation is an equivalence, by virtue of our assumption that $g$ is an exact functor.
\end{proof}

We would like to use Proposition \ref{swuft} to reformulate the definition of linearization.
We begin with the following result:

\begin{proposition}\label{urbusk}
Let $\kappa$ be a regular cardinal. Let $\calC$ and $\calD$ be pointed $\infty$-categories
which admit finite limits and $\kappa$-small colimits, and suppose that the loop functor
$\Omega_{\calD}$ commutes with $\kappa$-small filtered colimits.

Let $\Fun_0(\Spectra(\calC), \Spectra(\calD) )$ denote the full
subcategory of $\Fun( \Spectra(\calC), \Spectra(\calD) )$ spanned by those functors which preserve $\kappa$-small colimits, and let 
$$ \theta: \Fun_0( \Spectra(\calC), \Spectra(\calD) ) \rightarrow \Fun(\calC, \calD)$$
be given by the formula $F \mapsto \Omega_{\calD}^{\infty} \circ F \circ \Sigma^{\infty}_{\calC}$.
Then $\theta$ is fully faithful, and the essential image of $\theta$ consists of those functors
which are excisive and preserve $\kappa$-small filtered colimits.
\end{proposition}


The proof of Proposition \ref{urbusk} requires the following lemma:

\begin{lemma}\label{smulk}
Let $\calC$ be a pointed $\infty$-category which admits finite colimits, and let 
$\calD$ be a pointed $\infty$-category which admits finite limits. Let $K$ be a simplicial set
such that $\calC$ and $\calD$ both admit $K$-indexed colimits, and suppose that the loop functor
$\Omega_{\calD}$ preserves $K$-indexed colimits. Then:
\begin{itemize}
\item[$(1)$] The $\infty$-category $\Spectra(\calD)$ admits $K$-indexed colimits. 
\item[$(2)$] A diagram $\overline{p}: K^{\triangleright} \rightarrow \Spectra(\calD)$ is a colimit if and only if
$\Omega_{\calD}^{\infty-n} \circ \overline{p}: K^{\triangleright} \rightarrow \Spectra(\calD)$ is a colimit, for each $n \geq 0$.
\item[$(3)$] A functor $F: \calC \rightarrow \Spectra(\calD)$ preserves $K$-indexed colimits
if and only if $\Omega_{\calD}^{\infty-n} \circ F: \calC \rightarrow \calD$ preserves $K$-indexed colimits, for every $n \geq 0$.
\item[$(4)$] A right exact functor $F: \calC \rightarrow \Spectra(\calD)$ preserves $K$-indexed colimits
if and only if the excisive functor $\Omega_{\calD}^{\infty} \circ F: \calC \rightarrow \calD$ preserves $K$-indexed colimits.
\end{itemize}
\end{lemma}

\begin{proof}
Assertions $(1)$ and $(2)$ follow immediately from the description of $\Spectra(\calD)$ as the
homotopy inverse limit of the tower
$$ \ldots \stackrel{\Omega_{\calD}}{\rightarrow} \calD \stackrel{ \Omega_{\calD}}{\rightarrow} \calD$$
(Proposition \stableref{camer}). The implication $(2) \Rightarrow (3)$ is obvious. We now prove $(4)$. The ``only if'' direction follows immediately from $(3)$. For the converse, it suffices to observe that since $F$ is right exact, the composition $\Omega_{\calD}^{\infty-n} \circ F$ is equivalent to the functor $\Omega_{\calD}^{\infty} \circ F \circ \Sigma_{\calC}^{n}$, and both
$\Omega_{\calD}^{\infty} \circ F$ and $\Sigma_{\calC}^{n}$ preserve $K$-indexed colimits.
\end{proof}

\begin{proof}[Proof of Proposition \ref{urbusk}]
Using Proposition \stableref{urtusk21} and Lemma \ref{smulk}, we deduce that composition with
$\Omega_{\calD}^{\infty}$ induces an equivalence from $\Fun_0( \calC, \Spectra(\calD) )$
to the full subcategory of $\Fun(\calC, \calD)$ spanned by the excisive functors which
preserve $\kappa$-small filtered colimits. We now conclude by applying Corollary \stableref{mapprop}.
\end{proof}

\begin{remark}\label{seqco}
Let $\calJ$ be a filtered $\infty$-category with only countably many simplices.
Then there exists a cofinal map $\Nerve( \Z_{\geq 0} ) \rightarrow \calJ$.
To prove this, we first invoke Proposition \toposref{rot} to choose a cofinal map
$\Nerve(A) \rightarrow \calJ$, where $A$ is a filtered partially ordered; note that
the proof of Proposition \toposref{rot} produces a {\em countable} partially
ordered set $A$ in the case where $\calC$ has only countably many simplices.
Let $A = \{ a_0, a_1, a_2, \ldots \}$. Let $b_0 = a_0$, and for each $n > 1$
choose an element $b_n \in A$ which is an upper bound for the set
$\{ b_{n-1}, a_n \}$. The sequence $b_0 \leq b_1 \leq b_2 \leq \ldots$
determines a map $\Nerve( \Z_{\geq 0} ) \rightarrow \Nerve(A)$; Theorem \toposref{hollowtt} implies that this map is cofinal.
\end{remark}

\begin{remark}\label{swz}
Using Remark \ref{seqco}, we deduce the following:
\begin{itemize}
\item[$(1)$] Let $\calC$ be an $\infty$-category. Then $\calC$ admits countable filtered colimits
if and only if it admits sequential colimits. (In particular, if $\calC$ also admits finite colimits and
sequential colimits, then $\calC$ admits all countable colimits.)
\item[$(2)$] Let $F: \calC \rightarrow \calD$ be a functor where the $\infty$-category
$\calC$ satisfies the equivalent conditions of $(1)$. Then $F$ preserves countable filtered colimits 
if and only if $F$ preserves sequential colimits. 
\end{itemize}
\end{remark}

\begin{theorem}\label{spiller}
Let $F: \calC \rightarrow \calD$ be a well-pointed functor and let $f: \Spectra(\calC) \rightarrow \Spectra(\calD)$ be an exact functor which preserves sequential colimits.
Suppose we are given a natural transformation $\alpha: F^{+} \rightarrow f$ of functors from
$\Spectra(\calC)$ to $\PreSpectra(\calD)$. The following conditions are equivalent:
\begin{itemize}
\item[$(1)$] The transformation $\alpha$ exhibits $f$ as a linearization of $F^{+}$.
\item[$(2)$] Let $\beta$ denote the composition
$$ \Sigma^{\infty}_{\calD} \circ F \circ \Omega^{\infty}_{\calC}
= \Sigma^{\infty}_{\calD} \circ \Omega^{\infty}_{\calD} \circ F^{+}
\rightarrow F^{+} \stackrel{\alpha}{\rightarrow} f.$$
Then, for every exact functor $g: \Spectra(\calC) \rightarrow \Spectra(\calD)$, 
composition with $\beta$ induces a homotopy equivalence
$$\bHom_{ \Fun( \Spectra(\calC), \Spectra(\calD) )}( f, g) \rightarrow
\bHom_{ \Fun( \calC, \calD)}( F, \Omega^{\infty}_{\calD} \circ g \circ \Sigma^{\infty}_{\calC}).$$

\item[$(3)$] For every well-pointed exact functor $g: \Spectra(\calC) \rightarrow \Spectra(\calD)$, composition with $\beta$ induces a homotopy equivalence
$$\bHom_{ \Fun( \Spectra(\calC), \Spectra(\calD) )}( f, g) \rightarrow
\bHom_{ \Fun( \calC, \calD)}( F, \Omega^{\infty}_{\calD} \circ g \circ \Sigma^{\infty}_{\calC}).$$

\item[$(4)$] Let $\gamma$ denote the composition
$$ F \rightarrow F \circ \Omega^{\infty}_{\calC} \circ \Sigma^{\infty}_{\calC}
 = \Omega^{\infty}_{\calD} \circ F^{+} \circ \Sigma^{\infty}_{\calC}
 \rightarrow \Omega^{\infty}_{\calD} \circ f \circ \Sigma^{\infty}_{\calC}.$$
Then for every well-pointed excisive functor $G: \calC \rightarrow \calD$, composition with $\alpha$ induces a homotopy
equivalence 
$$ \bHom_{ \Fun( \calC, \calD)}( \Omega^{\infty}_{\calD} \circ f \circ \Sigma^{\infty}_{\calC}, G) \rightarrow \bHom_{ \Fun( \calC, \calD)}( F,G).$$
\end{itemize}
\end{theorem}

\begin{proof}
The implication $(1) \Rightarrow (2)$ follows from Proposition \ref{swuft}, 
the implication $(2) \Rightarrow (3)$ is obvious, and the equivalence
$(3) \Leftrightarrow (4)$ follows from Proposition \ref{urbusk} and Remark \ref{swz}.
To prove that $(3)$ implies $(1)$, let $f' = L_{\calC} \circ F^{+}$, where
$L_{\calC}: \PreSpectra(\calC) \rightarrow \Spectra(\calC)$ denotes a left adjoint to the inclusion.
Since $F$ preserves sequential colimits, the functor $F^{+}$ has the same property, so that
$f'$ again preserves sequential colimits. Since $f'$ is exact (Proposition \ref{exman}), we conclude
that $f'$ preserves countable colimits. Since $f$ takes values in $\Spectra(\calD)$, the map $\alpha$ factors as a composition
$$ F^{+} \stackrel{\alpha'}{\rightarrow} f' \stackrel{\alpha''}{\rightarrow} f$$
where $\alpha'$ exhibits $f'$ as a linearization of $F$. We wish to prove that if
condition $(3)$ is satisfied by $\alpha$, then $\alpha''$ is an equivalence.
To prove this, it will suffice to show that composition with $\alpha''$ determines
a homotopy equivalence
$$ \bHom_{ \Fun( \Spectra(\calC), \Spectra(\calD)}( f, g) \rightarrow \bHom_{ \Fun( \Spectra(\calC), \calD)}(f', g)$$
for every functor $g \in \LFun( \Spectra(\calC), \calD) )$. This follows by applying the two-out-of-three property to the diagram
$$ \xymatrix{  \bHom_{ \Fun( \Spectra(\calC), \Spectra(\calD)}( f, g) \ar[rr] \ar[dr] & &  \bHom_{ \Fun( \Spectra(\calC), \Spectra(\calD)}( f', g) \ar[dl] \\
& \bHom_{ \Fun( \calC, \calD)}( F, \Omega^{\infty}_{\calD} \circ g \circ \Sigma^{\infty}_{\calC}); & }$$
here the vertical maps are both homotopy equivalences because both $\alpha$ and $\alpha'$
satisfy condition $(3)$.
\end{proof}

\begin{remark}\label{rezza}
In the situation of Theorem \ref{spiller}, the natural transformation
$\beta: \Sigma^{\infty}_{\calD} \circ F \circ \Omega^{\infty}_{\calC} \rightarrow f$
is determined up to equivalence by property of $(2)$. We will therefore adopt
the following variation on the terminology of Definition \ref{cool}: we will say that
an arbitrary natural transformation $\beta: \Sigma^{\infty}_{\calD} \circ F \circ \Omega^{\infty}_{\calC} \rightarrow f$ {\it exhibits $f$ as a linearization of $F$} if the functor $f$ is exact and condition $(2)$ is
satisfied (it follows in this case that $f$ automatically preserves countable colimits).
Note in this case that $\beta$ is determined up to homotopy by a number of adjoint
transformations
$$ \gamma: F \rightarrow \Omega^{\infty}_{\calD} \circ f \circ \Sigma^{\infty}_{\calC} \quad \quad 
\Sigma^{\infty}_{\calD} \circ F \rightarrow f \circ \Sigma^{\infty}_{\calC} \quad \quad F \circ \Omega^{\infty}_{\calC} \rightarrow \Omega^{\infty}_{\calD} \circ f;$$
we will abuse terminology by saying that any one of these transformations
{\it exhibits $f$ as a linearization of $F$}.

Similarly, we will say that a natural transformation $\gamma: F \rightarrow F'$
{\it exhibits $F'$ as a derivative of $F$} if the functor $F'$ is well-pointed and excisive, and condition $(4)$ of Theorem \ref{spiller} is satisfied.
In this case, Proposition \ref{urbusk} implies that $F'$ is equivalent to a composition
$\Omega^{\infty}_{\calD} \circ f \circ \Sigma^{\infty}_{\calC}$, where
$f: \Spectra(\calC) \rightarrow \Spectra(\calD)$ is a functor which preserves small colimits.
Then $\gamma$ exhibits $F'$ as a derivative of $F$ if and only if
it exhibits $f$ as a linearization of $F$; in other words, the notations of derivative
and linearization are interchangable.
\end{remark}

\begin{remark}\label{derexi}
Let $F: \calC \rightarrow \calD$ be a well-pointed functor. Then we can produce a derivative
of $F$ via the composition
$$ F': \calC \stackrel{ \widetilde{ \Sigma}^{\infty}_{\calC} }{\rightarrow} 
\PreSpectra(\calC) \stackrel{F^{+}}{\rightarrow} \PreSpectra(\calD)
\stackrel{ L_{\calD}}{\rightarrow} \Spectra(\calD) \stackrel{ \Omega^{\infty}_{\calD} }{\rightarrow} \calD.$$
Combining this observation with the equivalence $L_{\calD} \simeq \varinjlim_{n} L_{n}$
of Corollary \stableref{postzing}, we deduce that $F'$ is equivalent to the colimit
$\varinjlim_{n} \Omega^{n}_{\calD} \circ F \circ \Sigma^{n}_{\calC}$. 
\end{remark}

\begin{corollary}\label{leftywing}
Let $\calC$ and $\calD$ be well-pointed $\infty$-categories.
Let $\calE$ denote the full subcategory of $\Fun(\calC, \calD)$ spanned by the
well-pointed functors, and $\calE_0 \subseteq \calE$ the full subcategory spanned
by the well-pointed excisive functors. Then $\calE_0$ is a localization of $\calE$.
\end{corollary}

\begin{proof}
For every functor $F \in \calE$, let $f$ denote a linearization of $f$. Then
the induced map $F \rightarrow \Omega^{\infty}_{\calD} \circ f \circ \Sigma^{\infty}_{\calC}$
exhibits $\Omega^{\infty}_{\calD} \circ f \circ \Sigma^{\infty}_{\calC}$ as a 
$\calE_0$-localization of $F$, by Theorem \ref{spiller}.
\end{proof}

We now study the linearization of functors in some special cases.

\begin{proposition}\label{kunst}
Let $\calC$ and $\calD$ be well-pointed $\infty$-categories, and let 
$F: \calC \rightarrow \calD$ and $f: \Spectra(\calC) \rightarrow \Spectra(\calD)$ be
left-exact functors which preserve sequential colimits.
Then a natural transformation
$\alpha: F \circ \Omega^{\infty}_{\calC} \rightarrow \Omega^{\infty}_{\calD} \circ f$
exhibits $f$ as a linearization of $F$ if and only if $\alpha$ is an equivalence.
\end{proposition}


\begin{proof}
Since $F$ is left exact, the functor $F^{+}: \PreSpectra(\calC) \rightarrow \PreSpectra(\calD)$
given by composition with $F$ carries $\Spectra(\calC)$ into $\Spectra(\calD)$. It follows that
the identity transformation exhibits $F^{+} | \Spectra(\calC)$ as a linearization of $F$.
To prove the ``only if'' direction, we may assume without loss of generality that
$f = F^{+} | \Spectra(\calC)$, in which case we can identify $\alpha$ with the identity
map from $F \circ \Omega^{\infty}_{\calC}$ to itself.

For the converse, let $f_0 = F^{+} | \Spectra(\calC)$ be as above, let
$\alpha: F \circ \Omega^{\infty}_{\calC} \rightarrow \Omega^{\infty}_{\calD} \circ f$
be any natural transformation, and let $\beta: \Sigma^{\infty}_{\calD} \circ F \circ \Omega^{\infty}_{\calC} \rightarrow f$ be adjoint to $\alpha$. Since $f_0$ is a linearization of $F$, the map
$\beta$ factors as a composition
$$  \Sigma^{\infty}_{\calD} \circ F \circ \Omega^{\infty}_{\calC} \stackrel{\beta'}{\rightarrow}
f_0 \stackrel{\beta''}{\rightarrow} f.$$
We wish to prove that $\beta''$ is an equivalence. For this, it suffices to prove that the
induced map $\Omega^{\infty}_{\calD} \circ f_0 \rightarrow \Omega^{\infty}_{\calD} \circ f$
coincides with $\alpha$, and is therefore an equivalence as desired.
\end{proof}

\begin{corollary}\label{abler}
Let $F: \calC \rightarrow \calD$ be a left exact well-pointed functor, and let
$f: \Spectra(\calC) \rightarrow \Spectra(\calD)$ be given by composition with $F$.
Then the identity transformation $F \circ \Omega^{\infty}_{\calC} \rightarrow
\Omega^{\infty}_{\calD} \circ f$ exhibits $f$ as a linearization of $F$.
\end{corollary}

\begin{proposition}\label{kunst2}
Let $\calC$ and $\calD$ be well-pointed $\infty$-categories 
and let $F: \calC \rightarrow \calD$ and $f: \Spectra(\calC) \rightarrow \Spectra(\calD)$ be functors which preserve countable colimits. A natural transformation $\alpha: \Sigma^{\infty}_{\calD} \circ F \rightarrow f \circ \Sigma^{\infty}_{\calC}$ exhibits
$f$ as a linearization of $F$ if and only if $\alpha$ is an equivalence.
\end{proposition}


\begin{proof}
We will prove the ``only if'' direction; the converse will then follow as in the proof of
Proposition \ref{kunst}. We observe that $f \circ \Sigma^{\infty}_{\calC}
\simeq f \circ L_{\calC} \circ \widetilde{\Sigma}^{\infty}_{\calC} \simeq L_{\calD} \circ F^{+}
\circ \widetilde{\Sigma}^{\infty}_{\calC}$, where the second equivalence follows from
Remark \ref{sweety}. Similarly, we can write $\Sigma^{\infty}_{\calD} \circ F$
as a composition $L_{\calD} \circ \widetilde{\Sigma}^{\infty}_{\calD} \circ F$.
In terms of these identifications, the map $\alpha$ is obtained by applying
$L_{\calD}$ to the natural transformation
$$ \alpha_0: \widetilde{\Sigma}^{\infty}_{\calD} \circ F \rightarrow F^{+} \circ \widetilde{ \Sigma}^{\infty}_{\calC}.$$
To prove that $L_{\calD}(\alpha_0)$ is an equivalence, it will suffice to show that for every
object $C \in \calC$, the map of prespectra
$$ \widetilde{\Sigma}^{\infty}_{\calD} F(C) \rightarrow F^{+} \circ \widetilde{\Sigma}^{\infty}_{\calC}(C)$$
induces an equivalence after evaluation at $(n,n)$ for each $n \geq 0$ (Remark \stableref{jilly}).
Unwinding the definitions, we must show that the canonical map
$$ \Sigma^{n}_{\calD} F(C) \rightarrow F( \Sigma^{n}_{\calC} C)$$
is an equivalence for $n \geq 0$, which follows from the assumption that $F$ is right exact.
\end{proof}

\begin{proposition}\label{staleman}
Let $F: \calC \rightarrow \calD$ be a well-pointed functor, $f: \Spectra(\calC) \rightarrow
\Spectra(\calD)$ a well-pointed exact functor, and 
$\alpha: F \circ \Omega^{\infty}_{\calC} \rightarrow \Omega^{\infty}_{\calD} \circ f$
a natural transformation. Then $\alpha$ exhibits $f$ as a linearization of $F$ if and only if
$\alpha$ exhibits $\Omega^{\infty}_{\calD} \circ f$ as
a derivative of $F \circ \Omega^{\infty}_{\calC}$.
\end{proposition}

\begin{proof}
The functors $f$ and $\Omega^{\infty}_{\calD}$ are both left exact and well-pointed, so the
composition $\Omega^{\infty}_{\calD} \circ f$ is left exact and well-pointed. Since the domain
of $\Omega^{\infty}_{\calD} \circ f$ is stable, it follows that $\Omega^{\infty}_{\calD} \circ f$
is excisive. Suppose first that $\alpha$ exhibits $\Omega^{\infty}_{\calD} \circ f$
as a derivative of $F \circ \Omega^{\infty}_{\calC}$. We wish to prove that
$\alpha$ exhibits $f$ as a linearization of $F$. Let
$\beta: \Sigma^{\infty}_{\calD} \circ F \circ \Omega^{\infty} \rightarrow f$ be the
map adjoint to $\alpha$, and let $g: \Spectra(\calC) \rightarrow \Spectra(\calD)$ be
an exact functor which preserves sequential colimits; we must show that
composition with $\beta$ induces a homotopy equivalence
$$ \psi: \bHom_{ \Fun( \Spectra(\calC), \Spectra(\calD))}(f,g) \rightarrow
\bHom_{ \Fun(\Spectra(\calC), \Spectra(\calD))}( \Sigma^{\infty}_{\calD}
\circ F \circ \Omega^{\infty}_{\calC}, g).$$
Using Proposition \stableref{urtusk21} and the adjointness between
$\Sigma^{\infty}_{\calD}$ and $\Omega^{\infty}_{\calD}$, we can identify
$\psi$ with the map
$$ \bHom_{ \Fun( \Spectra(\calC), \calD)}( \Omega^{\infty}_{\calD} \circ f,
\Omega^{\infty}_{\calD} \circ g) \rightarrow \bHom_{ \Fun( \Spectra(\calC), \calD)}(
F \circ \Omega^{\infty}_{\calC}, \Omega^{\infty}_{\calD} \circ g)$$
induced by composition with $\alpha$. This map is a homotopy equivalence
by virtue of our assumption on $\alpha$, and the fact that $\Omega^{\infty}_{\calD} \circ g$
is a well-pointed excisive functor (this follows from the fact that $\Omega^{\infty}_{\calD}$ and $g$
are well-pointed and left exact, and the domain of $\Omega^{\infty}_{\calD} \circ g$ is stable).

To prove the converse, let us suppose that $\alpha$ exhibits $f$ as a linearization of $F$.
The argument above shows that composition with $\alpha$ induces a homotopy
equivalence 
$$ \bHom_{ \Fun( \Spectra(\calC), \calD)}( \Omega^{\infty}_{\calD} \circ f,
G) \rightarrow \bHom_{ \Fun( \Spectra(\calC), \calD)}(
F \circ \Omega^{\infty}_{\calC}, G)$$
for every functor $G: \Spectra(\calC) \rightarrow \calD$ which can be obtained as a composition
$\Omega^{\infty}_{\calD} \circ g$, where $g: \Spectra(\calC) \rightarrow \Spectra(\calD)$
is exact and commutes with sequential colimits. To prove that $\alpha$
exhibits $\Omega^{\infty}_{\calD} \circ f$ as a derivative of $F \circ \Omega^{\infty}_{\calC}$,
it will suffice to show that every well-pointed excisive functor $G: \Spectra(\calC) \rightarrow \calD$
can be obtained in this way. Using the excisiveness of $G$ and Proposition \stableref{urtusk21},
we deduce that $G \simeq \Omega^{\infty}_{\calD} \circ g$, where $g: \Spectra(\calC) \rightarrow
\Spectra(\calD)$ is exact. Using Lemma \ref{smulk} and the compatibility of $G$
with sequential colimits, we deduce that $g$ preserves sequential colimits as desired.
\end{proof}

\begin{corollary}
Let $\calC$ be a well-pointed $\infty$-category. Then the identity transformation
$\id_{\calC} \circ \Omega^{\infty}_{\calC} \rightarrow \Omega^{\infty}_{\calC} \circ \id_{ \Spectra(\calC)}$ exhibits the identity functor $\id_{\Spectra(\calC)}: \Spectra(\calC) \rightarrow \Spectra(\calC)$
as a linearization of the identity functor $\id_{\calC}: \calC \rightarrow \calC$.
\end{corollary}

\begin{proof}
By virtue of Proposition \ref{staleman}, it will suffice to show that
the identity transformation exhibits $\Omega^{\infty}_{\calC}: \Spectra(\calC) \rightarrow \calC$
as a derivative of itself. This is clear, since $\Omega^{\infty}_{\calC}$ is well-pointed and excisive.
\end{proof}

\begin{proposition}[Chain Rule for First Derivatives]\label{clambake}
Let $F: \calC \rightarrow \calD$ and $G: \calD \rightarrow \calE$ be well-pointed functors.
\begin{itemize}
\item[$(1)$] Assume that $\alpha: F \circ \Omega^{\infty}_{\calC} \rightarrow \Omega^{\infty}_{\calD}
\circ f$ exhibits $f: \Spectra(\calC) \rightarrow \Spectra(\calD)$ as a linearization of
$F$, and that $\beta: G \circ \Omega^{\infty}_{\calD} \rightarrow \Omega^{\infty}_{\calE} \circ g$
exhibits $g: \Spectra(\calD) \rightarrow \Spectra(\calE)$ as a linearization of $G$.
Then the composite map
$$ \gamma: G \circ F \circ \Omega^{\infty}_{\calC} \stackrel{\alpha}{\rightarrow}
G \circ \Omega^{\infty}_{\calD} \circ f \stackrel{\beta}{\rightarrow} \Omega^{\infty}_{\calE} \circ g
\circ f$$
exhibits $g \circ f$ as a linearization of $G \circ F$.
\item[$(2)$] Assume that $\alpha: \Sigma^{\infty}_{\calD} \circ F \rightarrow f \circ \Sigma^{\infty}_{\calC}$ exhibits $f: \Spectra(\calC) \rightarrow \Spectra(\calD)$ as a linearization of
$F$, and that $\beta: \Sigma^{\infty}_{\calE} \circ G \rightarrow g \circ \Sigma^{\infty}_{\calD}$
exhibits $g: \Spectra(\calD) \rightarrow \Spectra(\calE)$ as a linearization of $G$.
Then the composite map
$$ \gamma: \Sigma^{\infty}_{\calE} \circ G \circ F \stackrel{\alpha}{\rightarrow}
g \circ \Sigma^{\infty}_{\calD} \circ F \stackrel{\beta}{\rightarrow} g
\circ f \circ \Sigma^{\infty}_{\calC}$$
exhibits $g \circ f$ as a linearization of $G \circ F$.
\end{itemize}
\end{proposition}

\begin{proof}
We will prove $(1)$; the proof of $(2)$ is similar. Let $L_{\calD}: \PreSpectra(\calD) \rightarrow
\Spectra(\calD)$ and $L_{\calE}: \PreSpectra(\calE) \rightarrow \Spectra(\calE)$ denote left
adjoints to the inclusion functors. In view of Theorem \ref{spiller}, we may assume without loss
of generality that $f = L_{\calD} \circ F^{+}$, $g = L_{\calE} \circ G^{+}$, and that
$\alpha$ and $\beta$ are given by the compositions
$$ F \circ \Omega^{\infty}_{\calC} = \Omega^{\infty}_{\calD} \circ F^{+}
\rightarrow \Omega^{\infty}_{\calD} \circ L_{\calD} \circ F^{+} = \Omega^{\infty}_{\calD} \circ f
$$
$$ G \circ \Omega^{\infty}_{\calD} = \Omega^{\infty}_{\calE} \circ G^{+}
\rightarrow \Omega^{\infty}_{\calE} \circ L_{\calE} \circ G^{+} = \Omega^{\infty}_{\calE} \circ g.$$
Then the map $\gamma$ factors as a composition
$$ G \circ F \circ \Omega^{\infty}_{\calC}
\stackrel{\gamma'}{\rightarrow} \Omega^{\infty}_{\calE} \circ L_{\calE} \circ
(G \circ F)^{+}
\stackrel{\gamma''}{\rightarrow} \Omega^{\infty}_{\calE} \circ L_{\calE} \circ G^{+}
\circ L_{\calD} \circ F^{+},$$
where $\gamma'$ exhibits $L_{\calE} \circ (G \circ F)^{+}$ as a linearization of
$G \circ F$. Consequently, it will suffice to show that $\gamma''$ is an equivalence.
In other words, it will suffice to show that for $X \in \PreSpectra(\calD)$, the canonical map
$L_{\calE} G^{+}(X) \rightarrow L_{\calE} G^{+} (L_{\calD} X)$ is an equivalence; this follows
immediately from Proposition \ref{coomer}.
\end{proof}

The chain rule is most conveniently stated in terms of linearizations. However,
in some cases it can also be rephrased in terms of derivatives:

\begin{proposition}\label{easychain}
Let $F: \calC \rightarrow \calD$ and $G: \calD \rightarrow \calE$ be well-pointed functors.
\begin{itemize}
\item[$(1)$] Suppose that $F$ is right exact, and let $\alpha: G \rightarrow G'$ exhibit
$G'$ as a derivative of $G$. Then the induced map $G \circ F \rightarrow G' \circ F$ exhibits
$G' \circ F$ as a derivative of $G \circ F$.
\item[$(2)$] Suppose that $G$ is left exact, and let $\beta: F \rightarrow F'$ exhibit
$F'$ as derivative of $F$. Then the induced map $G \circ F \rightarrow G \circ F'$ exhibits
$G \circ F'$ as a derivative of $G \circ F$.
\end{itemize}
\end{proposition}

\begin{proof}
We will prove $(2)$; the proof of $(1)$ is similar. Let $L_{\calD}: \PreSpectra(\calD) \rightarrow
\Spectra(\calD)$ and $L_{\calE}: \PreSpectra(\calE) \rightarrow \Spectra(\calE)$ denote left adjoints to the inclusions. By virtue of Theorem \ref{spiller}
may assume without loss of generality that 
$F' = \Omega^{\infty}_{\calD} \circ L_{\calD} \circ F^{+} \circ \widetilde{\Sigma}^{\infty}_{\calC}$,
and that $\beta$ is induced by the natural transformation $\id_{ \PreSpectra(\calD)} \rightarrow L_{\calD}$. We have a commutative diagram
$$ \xymatrix{ \Omega^{\infty}_{\calE} \circ G^{+} \circ F^{+} \circ \widetilde{\Sigma}^{\infty}_{\calC}
\ar[r]^{u} \ar[d]^{v} & \Omega^{\infty}_{\calE} \circ L_{\calE} \circ G^{+} \circ F^{+} \circ \widetilde{\Sigma}^{\infty}_{\calC} \ar[d]^{v'} \\
\Omega^{\infty}_{\calE} \circ G^{+} \circ L_{\calD} \circ F^{+} \circ \widetilde{\Sigma}^{\infty}_{\calC}
\ar[r]^{u'} & \Omega^{\infty}_{\calE} \circ L_{\calE} \circ G^{+} \circ L_{\calD} \circ F^{+} \circ
\widetilde{\Sigma}^{\infty}_{\calC} }$$
where $u$ exhibits $\Omega^{\infty}_{\calE} \circ L_{\calE} \circ G^{+} \circ F^{+} \circ \widetilde{\Sigma}^{\infty}_{\calC}$ as a derivative of $G \circ F$, and we wish to prove that
$v$ exhibits $G \circ F' = \Omega^{\infty}_{\calE} \circ G^{+} \circ L_{\calD} \circ F^{+} \circ \widetilde{\Sigma}^{\infty}_{\calC}$ as a derivative of $G \circ F$. To prove this, it will suffice to show
that the natural transformations $u'$ and $v'$ are equivalences.

To prove that $u'$ is an equivalence, it will suffice to show that the functor
$G^{+} \circ L_{\calD}: \PreSpectra(\calD) \rightarrow \PreSpectra(\calE)$ takes values
in $\Spectra(\calE)$. Since $L_{\calD}$ takes values in $\Spectra(\calD)$, this follows
from the observation that $G^{+}: \PreSpectra(\calD) \rightarrow \PreSpectra(\calE)$ carries
$\Spectra(\calD)$ into $\Spectra(\calE)$, since $G$ is left exact. To prove
that $v'$ is an equivalence, it will suffice to show that for every object $X \in \PreSpectra(\calD)$,
the canonical map $L_{\calE} \circ G^{+}(X) \rightarrow L_{\calE} \circ G^{+}( L_{\calD} X)$
is an equivalence, which follows from Proposition \ref{coomer}.
\end{proof}

\subsection{Linearization in Families}\label{bisec5.2}

In \S \ref{bisec5.1}, we introduced the {\it linearization} $f$ of a well-pointed functor
$F: \calC \rightarrow \calD$. To the functor $F$, we can associate a coCartesian
fibration $p: \calM \rightarrow \Delta^1$ such that $\calC \simeq \calM \times_{ \Delta^1} \{0\}$
and $\calD \simeq \calM \times_{ \Delta^1 } \{1\}$. In this case, the linearization of
$F$ can be described directly in terms of the map $p$: it is the functor associated to
another coCartesian fibration $\Spectra(p) \rightarrow \Delta^1$, such that
$\Spectra(p) \times_{ \Delta^1} \{0\} \simeq \Spectra(\calC)$ and
$\Spectra(p) \times_{ \Delta^1} \{1\} \simeq \Spectra(\calD)$. Our goal in this section
is to define the $\infty$-category $\Spectra(p)$ for a general locally coCartesian fibration
of $\infty$-categories, and to study its basic properties. Our main result is 
Theorem \ref{kalet}, which characterizes $\Spectra(p)$ by a universal property.

\begin{definition}\label{linfib}
Let $p: X \rightarrow S$ be an inner fibration of simplicial sets. 
We define simplicial sets
$$\Spectra(p) \subseteq \PreSpectra(p) \subseteq \widetilde{ \PreSpectra}(p)
\rightarrow S$$
as follows:
\begin{itemize}
\item For every map of simplicial sets $K \rightarrow S$, we have a canonical
bijection $$\Hom_{S}(K, \widetilde{\PreSpectra}(p))
\simeq \Hom_{S}( K \times \Nerve( \Z \times \Z), X).$$
In particular, we can identify vertices of $\widetilde{\PreSpectra}(p)$ with pairs
$(s,F)$, where $s$ is a vertex of $S$ and $F: \Nerve( \Z \times \Z) \rightarrow X_{s}$
is a functor.
\item We let $\PreSpectra(p)$ denote the full simplicial subset of
$\widetilde{ \PreSpectra}(p)$ spanned by those vertices $(s,F)$ such that
$F$ is a prespectrum object of $X_{s}$.
\item We let $\Spectra(p)$ denote the full simplicial subset of $\widetilde{\PreSpectra}(p)$
spanned by those vertices $(s,F)$ such that $F$ is a spectrum object of $X_{s}$.
\end{itemize}
\end{definition}

\begin{definition}
Let $p: X \rightarrow S$ be a locally coCartesian fibration of simplicial sets.
We will say that $p$ is {\it well-pointed} if the following conditions are satisfied:
\begin{itemize}
\item[$(i)$] For every vertex $s$ in $S$, the fiber $X_{s}$ is well-pointed.
\item[$(ii)$] For every edge $s \rightarrow s'$ in $S$, the associated functor
$X_{s} \rightarrow X_{s'}$ is well-pointed.
\end{itemize}
We will say that $p$ is {\it stable} if the following stronger conditions are satisfied:
\begin{itemize}
\item[$(i')$] For every vertex $s$ in $S$, the fiber $X_{s}$ is stable.
\item[$(ii')$] For every edge $s \rightarrow s'$ in $S$, the associated functor
$X_{s} \rightarrow X_{s'}$ is exact.
\end{itemize}
\end{definition}

\begin{proposition}\label{swingg}
Let $p: X \rightarrow S$ be an inner fibration of simplicial sets. Then:
\begin{itemize}
\item[$(1)$] The maps $\Spectra(p) \rightarrow \PreSpectra(p) \rightarrow S$
are inner fibrations.

\item[$(2)$] Assume that:
\begin{itemize}
\item[$(a)$] The map $p$ is a coCartesian fibration.
\item[$(b)$] For every vertex $s \in S$, the $\infty$-category $X_{s}$ is pointed.
\item[$(c)$] For every edge $f: s \rightarrow s'$ in $S$, the associated functor
$f_{!}: X_{s} \rightarrow X_{s'}$ preserves zero objects.
\end{itemize}
Then:
\begin{itemize}
\item[$(i)$] The induced map $q: \PreSpectra(p) \rightarrow S$ is a coCartesian fibration.
\item[$(ii)$] An edge $x \rightarrow x'$ in $\PreSpectra(p)$ is $q$-coCartesian if and only
if the induced map $x(n,n) \rightarrow x'(n,n)$ is a $p$-coCartesian edge in $X$, for every
integer $n$.
\item[$(iii)$] Evaluation at $(0,0)$ determines a functor $\PreSpectra(p) \rightarrow X$
which carries $q$-coCartesian edges to $p$-coCartesian edges.
\end{itemize}

\item[$(2')$] Assume that:
Assume that:
\begin{itemize}
\item[$(a)$] The map $p$ is a Cartesian fibration.
\item[$(b)$] For every vertex $s \in S$, the $\infty$-category $X_{s}$ is pointed.
\item[$(c)$] For every edge $f: s \rightarrow s'$ in $S$, the associated functor
$f^{\ast}: X_{s'} \rightarrow X_{s}$ preserves zero objects.
\end{itemize}
Then:
\begin{itemize}
\item[$(i)$] The induced map $q: \PreSpectra(p) \rightarrow S$ is a Cartesian fibration.
\item[$(ii)$] An edge $x \rightarrow x'$ in $\PreSpectra(p)$ is $q$-Cartesian if and only
if the induced map $x(n,n) \rightarrow x'(n,n)$ is a $p$-Cartesian edge in $X$, for every
integer $n$.
\item[$(iii)$] Evaluation at $(0,0)$ determines a functor $\PreSpectra(p) \rightarrow X$
which carries $q$-Cartesian edges to $p$-Cartesian edges.
\end{itemize}

\item[$(3)$] Suppose that $p$ satisfies the conditions $(a)$ of $(2)$, together
with the following stronger versions of $(b)$ and $(c)$:
\begin{itemize}
\item[$(b')$] For every vertex $s$ of $S$, the fiber $X_{s}$ admits finite limits.
\item[$(c')$] For every edge $f: s \rightarrow s'$ in $S$, the induced functor
$f_{!}: X_{s} \rightarrow X_{s'}$ is left exact.
\end{itemize}
Then:
\begin{itemize}
\item[$(i)$] The induced map $q: \Spectra(p) \rightarrow S$ is a coCartesian fibration.
\item[$(ii)$] An edge $x \rightarrow x'$ in $\Spectra(p)$ is $q$-coCartesian if and only
if the induced map $x(n,n) \rightarrow x'(n,n)$ is a $p$-coCartesian edge in $X$, for every
integer $n$.
\item[$(iii)$] Evaluation at $(0,0)$ determines a functor $\Spectra(p) \rightarrow X$
which carries $q$-coCartesian edges to $p$-coCartesian edges.
\end{itemize}

\item[$(3')$] Suppose that $p$ satisfies the conditions $(a)$ of $(2')$, together
with the following stronger versions of $(b)$ and $(c)$:
\begin{itemize}
\item[$(b')$] For every vertex $s$ of $S$, the fiber $X_{s}$ admits finite limits.
\item[$(c')$] For every edge $f: s \rightarrow s'$ in $S$, the induced functor
$f^{\ast}: X_{s'} \rightarrow X_{s}$ is left exact.
\end{itemize}
Then:
\begin{itemize}
\item[$(i)$] The induced map $q: \Spectra(p) \rightarrow S$ is a coCartesian fibration.
\item[$(ii)$] An edge $x \rightarrow x'$ in $\Spectra(p)$ is $q$-coCartesian if and only
if the induced map $x(n,n) \rightarrow x'(n,n)$ is a $p$-coCartesian edge in $X$, for every
integer $n$.
\item[$(iii)$] Evaluation at $(0,0)$ determines a functor $\Spectra(p) \rightarrow X$
which carries $q$-Cartesian edges to $p$-Cartesian edges.
\end{itemize}

\item[$(4)$] Assume that $p$ is a well-pointed coCartesian fibration.
Then:
\begin{itemize}
\item[$(i)$] The induced map $q: \Spectra(p) \rightarrow S$ is a coCartesian fibration.
\item[$(ii)$] Let $\overline{f}: x \rightarrow x'$ be an edge of $\Spectra(p)$ lying over an
edge $f: s \rightarrow s'$ in $S$. Then $\overline{f}$ is $q$-coCartesian if and only if the
induced map $f_{!} \circ x \rightarrow x'$ exhibits $x' \in \Spectra( X_{s'})$ as a
$\Spectra( X_{s'})$-localization of $f_{!} \circ x$ in the $\infty$-category
$\PreSpectra( X_{s'} )$.
\end{itemize}
\end{itemize}
\end{proposition}

\begin{proof}
To prove $(1)$, we consider the sequence of maps
$$ \Spectra(p) \rightarrow \PreSpectra(p) \rightarrow \widetilde{ \PreSpectra}(p) \rightarrow
S.$$
The first two maps are inclusions of full simplicial subsets, and therefore inner fibrations.
The third map is a pullback of $\Fun( \Nerve(\Z \times \Z), X) \rightarrow \Fun( \Nerve( \Z \times \Z), S)$, and therefore an inner fibration by Corollary \toposref{slammfan}.

We next prove $(2)$; the proofs of $(2')$, $(3)$, and $(3')$ are identical.
Assume that $p$ satisfies conditions $(a)$, $(b)$, and $(c)$ of $(2)$.
Combining $(a)$ with Proposition \toposref{doog}, we deduce:
\begin{itemize}
\item[$(i')$] The map $\widetilde{q}: \widetilde{ \PreSpectra}(p) \rightarrow S$
is a coCartesian fibration.
\item[$(ii')$] An edge $x \rightarrow x'$ in $\widetilde{\PreSpectra}(p)$ is $\widetilde{q}$-Cartesian if and only if the induced map $x(m,n) \rightarrow x'(m,n)$ is a $p$-Cartesian edge in $X$, for every
pair of integers $m,n \in \Z$.
\end{itemize}
To deduce $(i)$ from $(i')$, it suffices to observe that $(ii')$ and $(c)$ imply that
for every $\widetilde{q}$-coCartesian edge $x \rightarrow y$, if $x \in \PreSpectra(p)$, then
$y \in \PreSpectra(p)$. Moreover, an edge of $\PreSpectra(p)$ is $q$-coCartesian if and only if
it is a $\widetilde{q}$-coCartesian edge of $\widetilde{ \PreSpectra(p) }$. Assertion $(ii)$ follows
from $(ii')$, together with the observation that an edge $x \rightarrow y$ in $\PreSpectra(p)$
{\em automatically} induces $p$-coCartesian edges $x(m,n) \rightarrow x'(m,n)$ for
$m \neq n$, by virtue of $(c)$. Assertion $(iii)$ follows immediately from $(ii)$.

We conclude by observing that $(4)$ follows from $(2)$, Proposition \ref{coomer}, and
Lemma \monoidref{surine}.
\end{proof}

\begin{remark}\label{campar}
Let $p: X \rightarrow S$ be a map of simplicial sets satisfying the hypotheses of
part $(4)$ of Proposition \ref{swingg}. We observe that for every edge
$f: s \rightarrow s'$ in $S$, the associated functor
$$\PreSpectra(p)_{s} \simeq \PreSpectra(X_s)
\rightarrow \PreSpectra(X_{s'}) \simeq \PreSpectra(p)_{s'}$$
can be identified with the functor $f_{!}^{+}$ given by composition with $f_{!}: X_{s} \rightarrow X_{s'}$.
Consequently, the associated functor 
$$\Spectra(p)_{s} \simeq \Spectra(X_s)
\rightarrow \Spectra(X_{s'}) \simeq \Spectra(p)_{s'}$$
is given by the composition
$$ \Spectra(X_s) \subseteq \PreSpectra(X_s) 
\stackrel{ f_{!}^{+}}{\rightarrow} \PreSpectra(X_{s'})
\stackrel{L}{\rightarrow} \Spectra(X_{s'}),$$
where $L$ denotes a left adjoint to the inclusion. In other words,
the associated functor $\Spectra(p)_{s} \rightarrow \Spectra(p)_{s'}$ can be
identified with the linearization of the functor $f_{!}: X_{s} \rightarrow X_{s'}$.
\end{remark}

In the situation of part $(3)$ (or $(3')$) of Proposition \ref{swingg}, it is easy to
characterize $\Spectra(p)$ by a universal property:

\begin{proposition}\label{jagger}
\begin{itemize}
\item[$(1)$] Suppose given maps of simplicial sets $p: X \rightarrow S$, $q: Y \rightarrow S$
satisfying the following conditions:
\begin{itemize}
\item[$(a)$] The maps $p$ and $q$ are coCartesian fibrations.
\item[$(b)$] For every vertex $s \in S$, the fiber $X_{s}$ is pointed and admits finite limits, while the fiber $Y_{s}$ is stable.
\item[$(c)$] For every edge $s \rightarrow s'$ in $S$, the induced functors
$X_{s} \rightarrow X_{s'}$ and $Y_{s} \rightarrow Y_{s'}$ are left exact.
\end{itemize}
Let $e: \Spectra(p) \rightarrow X$ be the functor given by evaluation at $(0,0)$.
Let $\calC$ denote the full subcategory of $\Hom_{S}(Y, X)$ spanned by those
objects which carry $q$-coCartesian edges of $Y$ to $p$-coCartesian edges of
$X$, and induce left exact functors $Y_{s} \rightarrow X_{s}$ for each $s \in S$.
Let $\overline{\calC} \subseteq \Hom_{S}( Y, \Spectra(p) )$ be defined similarly.
Then composition with $e$ induces a categorical equivalence $\overline{\calC} \rightarrow \calC$.
\item[$(2)$] Suppose given maps of simplicial sets $p: X \rightarrow S$, $q: Y \rightarrow S$
satisfying the following conditions:
\begin{itemize}
\item[$(a)$] The maps $p$ and $q$ are Cartesian fibrations.
\item[$(b)$] For every vertex $s \in S$, the fiber $X_{s}$ is pointed and admits finite limits, while the fiber $Y_{s}$ is stable.
\item[$(c)$] For every edge $s \rightarrow s'$ in $S$, the induced functors
$X_{s'} \rightarrow X_{s}$ and $Y_{s'} \rightarrow Y_{s}$ are left exact.
\end{itemize}
Let $e: \Spectra(p) \rightarrow X$ be the functor given by evaluation at $(0,0)$.
Let $\calC$ denote the full subcategory of $\Hom_{S}(Y, X)$ spanned by those
objects which carry $q$-Cartesian edges of $Y$ to $p$-Cartesian edges of
$X$, and induce left exact functors $Y_{s} \rightarrow X_{s}$ for each $s \in S$.
Let $\overline{\calC} \subseteq \Hom_{S}( Y, \Spectra(p) )$ be defined similarly.
Then composition with $e$ induces a categorical equivalence $\overline{\calC} \rightarrow \calC$.
\end{itemize}
\end{proposition}

\begin{proof}
We will prove $(2)$; the proof of $(1)$ is similar.
It will suffice to show that for every simplicial set $K$, the induced map
$\Fun(K, \overline{\calC}) \rightarrow \Fun(K, \calC)$ induces a bijection between
equivalence classes of objects. Replacing $X$ by $\Fun(K,X) \times_{ \Fun(K,S) } S$, we
are reduced to the problem of showing that the functor $\overline{\calC} \rightarrow \calC$ is
bijective on equivalence classes of objects. 

The Cartesian fibrations $p$ and $q$ are classified by functors
$\chi_{p}: S^{op} \rightarrow \Cat_{\infty}$ and $\chi_{q}: S^{op} \rightarrow \Cat_{\infty}$,
respectively. Let $\chi': S^{op} \rightarrow \Cat_{\infty}$ classify the map $\Spectra(p) \rightarrow S$, 
which is a Cartesian fibration by virtue of Proposition \ref{swingg}. 
Let $\calC'$ denote the full subcategory of $\Hom_{S}(Y,X)$ spanned by those functors
which carry $q$-Cartesian edges to $p$-Cartesian edges, and let $\overline{\calC}'
\subseteq \bHom_{S}(Y, \Spectra(p) )$ be defined similarly. Using
Theorem \toposref{straightthm} and Proposition \toposref{gumby444}, we deduce that the collection of equivalence classes
of objects of $\calC'$ can be identified with $\pi_0 \bHom_{ \Fun( S^{op}, \Cat_{\infty})}( \chi_{q}, \chi_{p})$, and the collection of equivalence classes of objects of $\overline{\calC}'$ can be identified
with $\pi_0 \bHom_{ \Fun( S^{op}, \Cat_{\infty})}( \chi_{q}, \chi')$. Let
$\calE$ denote the subcategory of $\Cat_{\infty}$ spanned by pointed $\infty$-categories
which admit finite limits, and left exact functors between them. Then
$\chi_p$, $\chi_q$, and $\chi'$ all factor through $\calE$. Moreover, the set of
isomorphism classes of objects in $\calC$ can be identified with
$$\pi_0 \bHom_{ \Fun( S^{op}, \calE)}( \chi_{q}, \chi_{p})
\subseteq \pi_0 \bHom_{ \Fun( S^{op}, \calE)}( \chi_q, \chi_{p}),$$
and the set of isomorphism classes of objects in $\overline{\calC}$ can be identified with
$$\pi_0 \bHom_{ \Fun( S^{op}, \calE)}( \chi_{q}, \chi') \subseteq
\pi_0 \bHom_{ \Fun( S^{op}, \Cat_{\infty})}( \chi_q, \chi' ).$$

Let $\calE_0 \subseteq \calE$ denote the full subcategory spanned by the
stable $\infty$-categories. By assumption, $\chi_q$ factors through $\calE_0$.
Consequently, to complete the proof, it will suffice to show that the canonical map
$\chi' \rightarrow \chi_p$ exhibits $\chi'$ as a $\Fun( S^{op}, \calE_0)$-colocalization
of $\chi_p$. For this, it suffices to prove that for every vertex $s \in S$, the induced functor
$\chi'(s) \simeq \Spectra(X_s) \rightarrow X_s \simeq \chi_p(s)$ exhibits
$\chi'(s)$ as a $\calE_0$-colocalization of $\chi_p(s) \in \calE$, which follows
from Proposition \stableref{urtusk21}.
\end{proof}

\begin{remark}\label{inger}
Suppose given a commutative diagram of simplicial sets
$$ \xymatrix{ X \ar[rr]^{r} \ar[dr]^{p} & & Y \ar[dl]^{q} \\
& S & }$$
where $p$ and $q$ are locally coCartesian fibrations, but the map
$r$ is {\em not} assumed to preserve locally coCartesian edges. Let
$f: s \rightarrow s'$ be an edge of $S$, and let $f^{X}_!: X_{s} \rightarrow X_{s'}$ and
$f^{Y}_{!}: Y_{s} \rightarrow Y_{s'}$ denote the associated functors. There is a locally $p$-coCartesian natural transformation $\alpha: \id_{X_s} \rightarrow f^{X}_{!}$. Applying
$r$, we obtain a natural transformation from $r_{s}$ to $r_{s'} \circ f^{X}_{!}$.
Invoking the universal property of $f^{Y}_{!}$, we deduce that this map factors
through a natural transformation $f^{Y}_{!} \circ r_{s} \rightarrow r_{s'} \circ f^{X}_{!}$.
\end{remark}

\begin{definition}\label{hagg}
Let $p: X \rightarrow S$ be a locally coCartesian fibration of simplicial sets.
Assume that each fiber $X_{s}$ of $p$ is well-pointed, and that the functor
$f^{X}_{!}: X_{s} \rightarrow X_{s'}$ associated to each edge $f: s \rightarrow s'$
of $S$ is well-pointed. We will say that a map $r: Y \rightarrow X$
{\it exhibits $Y$ as a linearization of $p$} if the following conditions are satisfied:
\begin{itemize}
\item[$(1)$] The composite map $q: Y \rightarrow S$ is a locally coCartesian fibration.
\item[$(2)$] For each vertex $s \in S$, the fiber $Y_{s}$ is stable.
\item[$(3)$] For each vertex $s \in S$, the induced map $r_{s}: Y_{s} \rightarrow X_{s}$
is left exact. Consequently, the map $r_{s}$ admits an essentially unique factorization
$Y_{s} \stackrel{ \overline{r}_{s}}{\rightarrow} \Spectra(X_{s}) \stackrel{ \Omega^{\infty}_{X_s}}{\rightarrow} X_{s}$, where $\overline{r}_s$ is exact.
\item[$(4)$] For each vertex $s \in S$, the functor $\overline{r}_s$ is an equivalence
of $\infty$-categories, and therefore admits a homotopy inverse which we will denote
by $\overline{r}_{s}^{-1}$.
\item[$(5)$] For every edge $f: s \rightarrow s'$ in $S$, the natural transformation
$$f^{X}_{!} \circ \Omega^{\infty}_{X_s} \simeq f^{X}_{!} \circ r_{s} 
\circ \overline{r}_{s}^{-1} \rightarrow r_{s'} \circ f^{Y}_{!} \circ \overline{r}_{s}^{-1}
\simeq \Omega^{\infty}_{ X_{s'} } \circ ( \overline{r}_{s'} \circ f^{Y}_{!} \circ \overline{r}_{s}^{-1}$$ 
determined by Remark \ref{inger} exhibits $\overline{r}_{s'} \circ f^{Y}_{!} \circ \overline{r}_{s}^{-1}: \Spectra(X_s) \rightarrow \Spectra( X_{s'} )$ as a linearization of $f^{X}_{!}$.
\end{itemize}
\end{definition}

\begin{proposition}\label{hungr}
Let $p: X \rightarrow S$ be a locally coCartesian fibration of simplicial sets.
Assume that each fiber $X_{s}$ of $p$ is well-pointed, and that the functor
$f^{X}_{!}: X_{s} \rightarrow X_{s'}$ associated to each edge $f: s \rightarrow s'$
of $S$ is well-pointed. Let $e: \Spectra(p) \rightarrow X$ be given by evaluation
at $(0,0)$. Then $e$ exhibits $\Spectra(p)$ as a linearization of $p$.
\end{proposition}

\begin{proof}
We must show that conditions $(1)$ through $(5)$ of Definition \ref{hagg} are satisfied.
It follows from Proposition \ref{swingg} that the map $q: \Spectra(p) \rightarrow S$ is an inner fibration.
To prove that $q$ is a locally coCartesian fibration, we can reduce to the case where
$S = \Delta^1$; in this case, $p$ is a coCartesian fibration and the desired result follows
again from Proposition \ref{swingg}. This proves $(1)$. Conditions $(2)$ through $(4)$ are obvious (we can take $\overline{r}_{s}$ and $\overline{r}_{s}^{-1}$ to be the identity maps).
Finally, assertion $(5)$ follows from Remark \ref{campar} after unwinding the definitions.
\end{proof}

\begin{lemma}\label{huffle}
Suppose given a commutative diagram of simplicial sets
$$ \xymatrix{ \overline{X} \ar[dr]^{ \overline{p} } \ar[rr]^{e} & & X \ar[dl]^{p} \\
& S. & }$$
Assume that $p$ is a well-pointed locally coCartesian fibration and that
$e$ exhibits $\overline{X}$ as a linearization of $p$. Let $\calC$ be
any well-pointed $\infty$-category. Let $V$ be the full simplicial subset
of $\Fun( \calC, X) \times_{ \Fun( \calC, S)} S$ spanned by the collection
of well-pointed excisive functors $\calC \rightarrow X_{s}$ for $s \in S$, and
let $\overline{V}$ be defined similarly. Then composition with $e$ induces
a categorical equivalence $\overline{V} \rightarrow V$.
\end{lemma}

\begin{proof}
In view of Lemma \ref{piner}, it will suffice to prove the following:
\begin{itemize}
\item[$(i)$] The projection $\overline{q}: \overline{W} \rightarrow S$ is a locally coCartesian fibration.
\item[$(ii)$] The projection $q: W \rightarrow S$ is a locally coCartesian fibration.
\item[$(iii)$] The map $\overline{W} \rightarrow W$ carries locally $\overline{q}$-coCartesian
edges to $q$-coCartesian edges.
\item[$(iv)$] For every vertex $s \in S$, the induced map $\overline{V}_s \rightarrow V_{s}$
is an equivalence of $\infty$-categories.
\end{itemize}
The map $q$ admits a factorization
$$ V \rightarrow \Fun( \calC, X) \times_{ \Fun(\calC,S)} S \rightarrow S.$$
The first factor is the inclusion of a full simplicial subset, and therefore an inner fibration;
the second factor is a pullback of $\Fun(\calC,X) \rightarrow \Fun(\calC, S)$,
and therefore an inner fibration by Corollary \toposref{slammfan}. It follows that
$q$ is an inner fibration; likewise $\overline{q}$ is an inner fibration. To prove
the remaining assertions, we can reduce to the case where $S$ is a simplex of
dimension $\leq 1$; in particular, we may assume that $p$ and $\overline{p}$ are coCartesian
fibrations.

We now prove $(i)$. Proposition \toposref{doog} implies that the projection
$\overline{q}': \Fun( \calC, \overline{X}) \times_{ \Fun(\calC, S) } S \rightarrow S$
is a coCartesian fibration. Moreover, an edge $f \rightarrow g$ in the fiber 
product $\Fun( \calC, \overline{X}) \times_{ \Fun(\calC, S) } S$ is
$\overline{q}'$-coCartesian if and only if, for each $C \in \calC$, the induced edge
$f(C) \rightarrow g(C)$ is an $\overline{p}$-coCartesian edge of $X$. 
Since every edge $s \rightarrow s'$ induces a well-pointed exact functor
$\overline{X}_s \rightarrow \overline{X}_{s'}$, we conclude that
if $f \in \overline{W}$, then $g \in \overline{W}$. This proves that
$\overline{q} = \overline{q}' | \overline{W}$ is a coCartesian fibration, and that
an edge of $\overline{W}$ is $\overline{q}$-coCartesian if and only if it is
$\overline{q}'$-coCartesian.

Assertion $(iv)$ follows immediately from Proposition \stableref{urtusk21}.
To prove $(iii)$, we may assume without loss of generality that $S= \Delta^1$, so that
we can view $X$ as the correspondence associated to a functor $F: X_0 \rightarrow X_1$.
Since $e$ exhibits $\overline{X}$ as a linearization of $F$, it is the correspondence associated
to the linearization $f: \Spectra(X_0) \rightarrow \Spectra(X_1)$ of $F$.
Let $G: \calC \rightarrow \Spectra(X_0)$ be a well-pointed excisive functor, and let
$e: G \rightarrow f \circ G$ be the corresponding $\overline{q}$-coCartesian edge of
$\overline{V}$; we wish to show that the image of $e$ in $V$ is $q$-coCartesian.
Unwinding the definitions, we are reduced to proving that the canonical natural transformation
$F \circ \Omega^{\infty}_{X_0} \circ G \rightarrow \Omega^{\infty}_{X_1} \circ f \circ G$
exhibits $\Omega^{\infty}_{X_1} \circ f \circ G$ as a derivative of
$F \circ \Omega^{\infty}_{s} \circ G$. Since $G$ is an excisive functor with
stable codomain, it is right exact. Invoking Proposition \ref{easychain}, we may
reduce to showing that the canonical map $F \circ \Omega^{\infty}_{X_0} \rightarrow \Omega^{\infty}_{X_1} \circ f$ exhibits $\Omega^{\infty}_{X_1} \circ f$ as a derivative of
$F \circ \Omega^{\infty}_{X_0}$, which follows from immediately from
Proposition \ref{staleman} (and our assumption regarding $e$).

We now prove $(ii)$. Suppose we are given
an vertex $f: \calC \rightarrow X_{s}$ of $V$ and an edge $e: s \rightarrow s'$ in
$S$; we wish to prove that $e$ can be lifted to a $q$-coCartesian edge $f \rightarrow g$.
Using $(iv)$, we may assume without loss of generality that $f$ can be lifted to a vertex $\overline{f}: \calC \rightarrow \overline{X}_s$ in $\overline{V}$. Using $(i)$, we can choose an
$\overline{q}$-coCartesian edge $\overline{f} \rightarrow \overline{g}$ lifting $e$.
Applying $(iii)$, we deduce that the image of this edge in $V$ is a $q$-coCartesian
lift $f \rightarrow g$ of $e$, as desired.
\end{proof}

\begin{theorem}\label{kalet}
Suppose given a commutative diagram of simplicial sets
$$ \xymatrix{ Y \ar[dr]^{q} & \overline{X} \ar[r]^{e} \ar[d]^{ \overline{p}} & X \ar[dl]^{p} \\
& S & }$$
satisfying the following conditions:
\begin{itemize}
\item[$(1)$] The map $q$ is a stable locally coCartesian fibration.
\item[$(2)$] The map $p$ is a well-pointed locally coCartesian fibration.
\item[$(3)$] The map $e$ exhibits $\overline{X}$ as a linearization of
$p$ (in particular, $\overline{p}$ is a stable locally coCartesian fibration).
\end{itemize}
Let $\calC$ denote the full subcategory of $\Fun_{S}(Y,X)$ spanned by those
maps which induced well-pointed left exact functors $Y_{s} \rightarrow X_{s}$ for every vertex $s \in S$, and let $\overline{\calC} \subseteq \Fun_{S}( Y, \overline{X} )$ be defined similarly.
Then composition with $e$ induces an equivalence of $\infty$-categories
$\overline{\calC} \rightarrow \calC$.
\end{theorem}

\begin{proof}
Without loss of generality, we may suppose that $e$ is a categorical fibration.
We define a simplicial set $Z$ by the following universal property: for
every simplicial set $K$, $\Hom_{ \sSet}(K, Z)$ can be identified
with the set of pairs $(b, \phi)$, where $b: K \rightarrow S$ is a map
of simplicial sets and $\phi: K \times_{S} Y \rightarrow K \times_{S} X$
is a map which is compatible with the projection to $K$, and induces
a left exact functor $Y_{b(k)} \rightarrow X_{b(k)}$ for each vertex $k$ of $K$.
Let $\overline{Z}$ be defined similarly, using $\overline{X}$ in place of $X$.
The map $\overline{\calC} \rightarrow \calC$ is a pullback of the canonical map
$\Fun(S, \overline{Z}) \rightarrow \Fun(S, Z)$. It will therefore suffice to show that
the map $\overline{Z} \rightarrow Z$ is a trivial Kan fibration. In other words, we need only show that every lifting problem of the form
$$ \xymatrix{ \bd \Delta^n \ar@{^{(}->}[d] \ar[r] & \overline{Z} \ar[d] \\
\Delta^n \ar[r] & Z }$$
admits a solution. Without loss of generality, we may replace $S$ by $\Delta^n$;
let $Y' = Y \times_{ \Delta^n} \bd \Delta^n$. Unwinding the definitions, we are required
to solve a lifting problem of the form
$$ \xymatrix{ Y' \ar@{^{(}->}[d] \ar[r] & \overline{X} \ar[d]^{e} \\
Y \ar[r]^{\phi_0} \ar@{-->}[ur]^{\phi} & X. }$$
Moreover, if $n=0$, we must further guarantee that the functor $\phi$ is left
exact and preserves sequential colimits.

Let us first consider the case $n=0$. By assumption, the map $e$ is equivalent to
the functor $\Omega^{\infty}_{X}: \Spectra(X) \rightarrow X$, and $\phi_0$
is a left exact functor whose domain is stable, which is therefore excisive.
Invoking Proposition \stableref{urtusk21}, we deduce that $\phi_0 \simeq e \circ \phi'$,
where $\phi': Y \rightarrow \overline{X}$ is an exact functor. Since $e$ is a categorical
fibration, any equivalence of $e \circ \phi'$ with $\phi_0$ can be lifted to an equivalence
of $\phi'$ with an exact functor $\phi: Y \rightarrow \overline{X}$ satisfying $e \circ \phi = \phi_0$.
The compatibility of $\phi$ with sequential colimits follows from Lemma \ref{smulk}.

We now treat the case $n > 0$. Since $q$ is a locally coCartesian fibration,
Proposition \ref{unple} guarantees the existence of a simplicial functor
$\calF: \sCoNerve[ \Delta^n ] \rightarrow \sSet$ and a map $u: \sMap(\calF) \rightarrow Y$
which induces categorical equivalences $\calF(i) \rightarrow Y \times_{ \Delta^n} \{i\}$ for
$0 \leq i \leq n$. For every face $\sigma \subseteq \Delta^n$, let
$W_{\sigma} = \sMap( \calF | \sCoNerve[\sigma])$. Finally, for every
simplicial subset $S' \subseteq S$, let $W_{S'}$ denote the colimit
$\colim_{ \sigma \in S'} W_{\sigma}$. For each $S' \subseteq S$, we have a canonical map
$\psi_{S'}: W_{S'} \rightarrow Y \times_{S} S'$. Using Proposition \ref{pikkle}, we deduce
that $\psi_{S'}$ is a categorical equivalence whenever $S'$ is a simplex. Since the domain
and codomain of $\psi_{S'}$ both carry pushout squares of simplicial subsets of $S$
to homotopy pushout squares of simplicial sets, we deduce that $\psi_{S'}$ is a categorical
equivalence for all $S' \subseteq S$. Invoking Proposition \toposref{princex}, we are reduced
to solving the lifting problem depicted in the diagram
$$ \xymatrix{ W_{ \bd \Delta^n} \ar[r] \ar@{^{(}->}[d] & \overline{X} \ar[d]^{e} \\
W_{ \Delta^n} \ar[r] \ar@{-->}[ur] & X. }$$
Let $C = ( \Delta^1)^{n}$ denote an $n$-dimensional cube, and $\bd C$ its boundary.
Then the left vertical map is a pushout of the inclusion $(\bd C) \times \calF(0) \subseteq
C \times \calF(0)$. Consequently, the above lifting problem is equivalent to providing
a dotted arrow in the diagram
$$ \xymatrix{ (\bd C) \times \calF(0) \ar[r] \ar[d] & \overline{X} \ar[d]^{e} \\
C \times \calF(0) \ar[r] \ar@{-->}[ur] & X. }$$

We may assume without loss of generality that the functor $\calF$ is projectively fibrant
(otherwise, we simply make a fibrant replacement for $\calF$), so that $\calF(0)$ is an
$\infty$-category which is equivalent to the fiber $Y \times_{ \Delta^n} \{0\}$.
In particular, $\calF(0)$ is stable. Let $V$ denote the full simplicial subset of $\Fun( \calF(0), X) \times_{ \Fun( \calF(0), S)} S$ spanned by those vertices which correspond to well-pointed left
exact functors $\calF(0) \rightarrow X_{s}$, for some vertex $s$ in $S$, and let
$\overline{V}$ be defined similarly. We can now rewrite our lifting problem yet again:
$$ \xymatrix{ \bd C \ar[r] \ar@{^{(}->}[d] & \overline{V} \ar[d]^{e'} \\
C \ar[r] & V. }$$
To solve this lifting problem, it suffices to show that $e'$ is a trivial Kan fibration.
Since $e$ is a categorical fibration, we deduce that $e'$ is a categorical fibration.
We complete the proof by observing that Lemma \ref{huffle} guarantees that
$e'$ is a categorical equivalence (since $\calF(0)$ is stable, a functor from
$\calF(0)$ to another $\infty$-category is excisive if and only if it is left exact).
\end{proof}

\subsection{Linearization as a Functor}\label{finalx}\label{bisec5.3}

Our goal in this section is to formulate Goodwillie's theory of first derivatives in
the language of $(\infty,2)$-categories.

\begin{definition}
We define a pair of $\mSet$-enriched categories $(\dICAT)^{\ex} \subseteq
(\dICAT)^{\wellp}$ as follows:
\begin{itemize}
\item[$(1)$] The objects of $(\dICAT)^{\wellp}$ are small well-pointed $\infty$-categories.
An object of $(\dICAT)^{\wellp}$ belongs to $(\dICAT)^{\ex}$ if and only if it is stable.
\item[$(2)$] For every pair of objects $\calC, \calD \in (\dICAT)^{\wellp}$, we let
$$\bHom_{ (\dICAT)^{\wellp}}( \calC, \calD) = ( \Fun^{\wellp}(\calC, \calD), M)$$
where $\Fun^{\wellp}(\calC, \calD)$ denotes the full subcategory of
$\Fun(\calC, \calD)$ spanned by the well-pointed functors, and $M$ is the collection
of all equivalences in $\Fun^{\wellp}(\calC, \calD)$. Similarly, we let
$\bHom_{ (\dICAT)^{\ex}}(\calC, \calD) = ( \Fun^{\ex}(\calC, \calD), M)$, where
$\Fun^{\ex}(\calC, \calD)$ is the full subcategory of $\Fun(\calC,\calD)$ spanned
by the well-pointed exact functors, and $M$ is the collection of all equivalences
in $\Fun^{\ex}(\calC, \calD)$.
\item[$(3)$] Composition in $(\dICAT)^{\wellp}$ and $(\dICAT)^{\ex}$ are defined in the obvious way.
\end{itemize}
\end{definition}

\begin{definition}
We define scaled simplicial sets $\ICAT^{\wellp}$ and $\ICAT^{\ex}$ by the formulas
$$ \ICAT^{\wellp} = \scNerve( (\dICAT)^{\wellp}) \quad \quad \ICAT^{\ex} = \scNerve( \ICAT^{\ex} ).$$
\end{definition}

\begin{remark}
Since the marked simplicial categories $(\dICAT)^{\wellp}$ and $(\dICAT)^{\ex}$ are not small,
the scaled simplicial sets $\ICAT^{\wellp}$ and $\ICAT^{\ex}$ are likewise not small.
\end{remark}

\begin{remark}
By construction, the marked simplicial categories $(\dICAT)^{\wellp}$ and
$(\dICAT)^{\ex}$ are fibrant. It follows that the scaled simplicial sets
$\ICAT^{\ex}$ and $\ICAT^{\wellp}$ are $\infty$-bicategories (Remark \ref{coslai}).
\end{remark}

\begin{remark}
Unwinding the definitions, we deduce that the underlying $\infty$-category
of $\ICAT^{\wellp}$ is the subcategory of $\Cat_{\infty}$ spanned by the well-pointed
$\infty$-categories and well-pointed functors between them. Similarly, the underlying
$\infty$-category of $\ICAT^{\ex}$ is the subcategory of $\Cat_{\infty}$ spanned by the
stable well-pointed $\infty$-categories and exact well-pointed functors between them.
\end{remark}

\begin{remark}\label{intertime}
Let $\overline{S} = (S,T)$ be a scaled simplicial set. According to Corollary
\ref{snoball}, the scaled unstraightening functor $\scUn_{ \overline{S} }$ determines
a bijective correspondence between homotopy classes of maps
$\overline{S} \rightarrow \ICAT$ in $\scSet$ and equivalence classes
of locally coCartesian fibrations $X \rightarrow S$ with essentially small fibers,
whose restriction to every simplex of $T$ is a coCartesian fibration. Restricting
our attention to maps $\overline{S} \rightarrow \ICAT$ that factor through
$\ICAT^{\wellp}$, we deduce the following analogue:
\begin{itemize}
\item[$(a)$] The unstraightening functor $\scUn_{\overline{S}}$ induces a bijective
correspondence between homotopy classes of maps $\overline{S} \rightarrow
\ICAT^{\wellp}$ and equivalence classes of well-pointed locally coCartesian fibrations $X \rightarrow S$
with essentially small fibers, whose restriction to every simplex of $T$ is a coCartesian fibration.
\end{itemize}
Similarly, we obtain the following characterization of $\ICAT^{\ex}$:
\begin{itemize}
\item[$(b)$] The unstraightening functor $\scUn_{ \overline{S} }$ induces a bijective
correspondence between homotopy classes of maps $\overline{S} \rightarrow \ICAT{\wellp}$
and equivalence classes of stable well-pointed locally coCartesian fibrations
$X \rightarrow S$ with essentially small fibers, whose restriction to every simplex of $T$ is a coCartesian fibration.
\end{itemize}
\end{remark}

Let $\ICAT^{\wellp} = (S,T)$. Applying assertion $(a)$ of Remark \ref{intertime} to the identity
map $(S,T) \rightarrow \ICAT^{\wellp}$, we deduce the existence of a {\em universal}
well-pointed locally coCartesian fibration $p: X \rightarrow S$ whose fibers are essentially
small and whose restriction to every simplex of $T$ is a coCartesian fibration.
Choose a linearization $X' \rightarrow X$ of $p$. The characterization
of linearizations given in Theorem \ref{kalet} shows that $X'$ is well-defined up
to equivalence. 

\begin{proposition}\label{sabre}
In the above situation, the induced map $p': X' \rightarrow S$ is a stable well-pointed
locally coCartesian fibration with essentially small fibers. Moreover, the
restriction of $p'$ to every thin $2$-simplex of $S$ is a coCartesian fibration.
\end{proposition}

\begin{proof}
Without loss of generality, we may suppose that $X' = \Spectra(p)$ (Proposition
\ref{hungr}). In this case, the desired result follows from Proposition \ref{swingg}.
\end{proof}

\begin{remark}
Alternatively, we can deduce Proposition \ref{sabre} directly from the definition
of a linearization. The only nontrivial point is to verify that the restriction
of $p'$ to every thin $2$-simplex of $S$ is a coCartesian fibration; this is simply
a translation of the chain rule (Proposition \ref{clambake}). 
\end{remark}

Combining Proposition \ref{sabre} with part $(b)$ of Remark \ref{intertime}, we deduce that 
the map $p': X' \rightarrow S$ is classified by a map of scaled simplicial sets $(S,T) \rightarrow \ICAT^{\ex}$.

\begin{definition}\label{swit}
We let $\Lin: \ICAT^{\wellp} \rightarrow \ICAT^{\ex}$ denote the map of $\infty$-bicategories constructed above. We will refer to $\Lin$ as the {\it linearization functor}.
\end{definition}

\begin{remark}
The functor $\Lin$ of Definition \ref{swit} is well-defined up to equivalence. 
Moreover, the map $X' \rightarrow X$ in the above discussion determines a natural
transformation $\alpha$ from $\Lin$ to the identity functor $\id_{ \ICAT^{\wellp} }$. The universal
property of linearizations (Theorem \ref{kalet}) translates in this context to a universal
property of the natural transformation $\alpha$, which determines $\Lin$
and $\alpha$ up to a contractible ambiguity; we will not pursue the matter further here.
\end{remark}

\begin{remark}
We can summarize Definition \ref{swit} informally as follows: the functor
$\Lin$ carries a well-pointed $\infty$-category $\calC$ to the $\infty$-category
$\Spectra(\calC)$ of spectrum objects of $\calC$, and a well-pointed functor
$F: \calC \rightarrow \calD$ to its linearization $f: \Spectra(\calC) \rightarrow \Spectra(\calD)$.
The compatibility of this procedure with composition of morphisms is guaranteed by
the chain rule (Proposition \ref{clambake}).
\end{remark}

\begin{variant}
Let $\widehat{\ICAT}^{\wellp}$ and $\widehat{\ICAT}^{\ex}$ be defined like
$\ICAT{\wellp}$ and $\ICAT^{\ex}$, but without the requirement that the objects
of $\widehat{\ICAT}^{\wellp}$ and $\widehat{\ICAT}^{\ex}$ be small.
Then $\widehat{\ICAT}^{\wellp}$ and $\widehat{\ICAT}^{\ex}$ are {\em very large}
$\infty$-bicategories (for example, the $\infty$-categories of maps between objects of
$\widehat{\ICAT}^{\wellp}$ and $\widehat{\ICAT}^{\ex}$ are not small). Nevertheless,
all of the constructions of this section can be carried out without essential change, to obtain
a linearization functor $\widehat{ \ICAT}^{\wellp} \rightarrow \widehat{\ICAT}^{\ex}$ which
we will also denote by $\Lin$. This is the setting in which the Goodwillie calculus 
is usually studied.
\end{variant}


\subsection{Adjoint Functors and Linearization}\label{linadj}\label{bisec5.4}

Our starting point for this section is the following observation:

\begin{proposition}\label{summm}
Suppose given a pair of adjoint functors
$\Adjoint{F}{\calC}{\calD}{G}$
between well-pointed $\infty$-categories $\calC$ and $\calD$.
Let $f$ be a linearization of $F$, and let $g: \Spectra(\calD) \rightarrow \Spectra(\calC)$
be given by composition with $G$. Then $f$ and $g$ are adjoint to one another.
In particular, if $G$ preserves sequential colimits, then the linearizations
of $F$ and $G$ are adjoint to one another.
\end{proposition}

\begin{proof}
Let $p: \calM \rightarrow \Delta^1$ be a correspondence associated to the adjoint functors
$F$ and $G$. Then the induced map $q: \Spectra(p) \rightarrow \Delta^1$ is a correspondence between $\Spectra(\calC)$ and $\Spectra(\calD)$. 
Proposition \ref{swingg} implies that $q$ is both a Cartesian and a coCartesian fibration.
The associated functor $\Spectra(\calC) \rightarrow \Spectra(\calD)$ can be identified
with a linearization of $F$ (by Remark \ref{campar}), while the associated functor
$\Spectra(\calD) \rightarrow \Spectra(\calC)$ is given by composition with $G$.
It follows that the correspondence $\Spectra(p)$ exhibits the functors $f$ and $g$ as adjoint to one another. The final assertion follows from Corollary \ref{abler}.
\end{proof}

We now ask the question: given an adjunction $\Adjoint{F}{\calC}{\calD}{G}$ as in
Proposition \ref{summm}, when is the induced adjunction
$$ \Adjoint{f}{\Spectra(\calC)}{\Spectra(\calD)}{g}$$ an equivalence of $\infty$-categories?
Corollary \ref{hurpek} below allows us to give an affirmative answer to this question in a variety of situations. Before we can state it, we need to review a bit of terminology.

Suppose given a pair of adjoint functors $\Adjoint{F}{\calC}{\calD}{G}$.
In \S \monoidref{barri}, we showed the that composition $T=G \circ F$ admits the structure of
a monad on the $\infty$-category $\calC$. Moreover, the functor $G$ factors (up to canonical homotopy) as a composition
$$ \calD \stackrel{G'}{\rightarrow} \Mod_{T}(\calC) \stackrel{G''}{\rightarrow} \calC,$$
where $\Mod_{T}(\calC)$ denotes the $\infty$-category of $T$-modules in $\calC$ (that is,
the $\infty$-category of objects $C \in \calC$ equipped with a map $TC \rightarrow C$ which is coherently associative in a suitable sense). If the functor $G'$ is an equivalence, then we say that
the $\infty$-category $\calD$ is {\it monadic} over $\calC$. In this case, we think of 
$G$ as a forgetful functor, so that an object $D \in \calD$ can be identified with its image
$GD \in \calC$, together with some additional data (a $T$-module structure on $GD$).

\begin{proposition}\label{poofle}
Let $G: \calD \rightarrow \calC$ be a functor between pointed $\infty$-categories which
exhibits $\calD$ as monadic over $\calC$. Assume that $\calD$ and $\calC$ admit finite limits.
The functor $G$ is left exact, and therefore induces a functor 
$g: \Spectra(\calD) \rightarrow \Spectra(\calC)$. Suppose that $g$ admits a left adjoint. Then $g$ exhibits $\Spectra(\calD)$ as monadic over $\Spectra(\calC)$.
\end{proposition}

We will give the proof of Proposition \ref{poofle} at the end of this section. 

\begin{corollary}
Suppose given a pair of adjoint functors
$$\Adjoint{F}{\calC}{\calD}{G}$$
between pointed presentable $\infty$-categories. Suppose further that
the functors $G$, $\Omega_{\calC}$, and $\Omega_{\calD}$ are continuous, and
that $G$ exhibits $\calD$ as monadic over $\calC$.
Then:
\begin{itemize}
\item[$(1)$] The resulting adjoint functors
$\Adjoint{f}{\Spectra(\calC)}{\Spectra(\calD)}{g}$ exhibit $\Spectra(\calD)$ as monadic over
$\Spectra(\calC)$.

\item[$(2)$] The monad $g \circ f$ associated to the adjunction of $(1)$ is equivalent to the
linearization of $G \circ F$.
\end{itemize}
\end{corollary}

\begin{proof}
Assertion $(1)$ follows immediately from Proposition \ref{poofle}, and $(2)$ follows
from Proposition \ref{easychain}.
\end{proof}

\begin{corollary}\label{hurpek}
Suppose given an adjunction
$\Adjoint{F}{\calC}{\calD}{G}$
between pointed presentable $\infty$-categories. Suppose further that
the functors $G$, $\Omega_{\calC}$, and $\Omega_{\calD}$ are continuous, and
that $G$ exhibits $\calD$ as monadic over $\calC$. If the unit map
$\id_{\calC} \rightarrow GF$ induces an equivalence after linearization, then
$G$ induces an equivalence $\Spectra(\calD) \rightarrow \Spectra(\calC)$.
\end{corollary}

The proof of Proposition \ref{poofle} relies on the following lemma:

\begin{lemma}\label{precougher}
Suppose given a diagram $q: S \rightarrow \Cat_{\infty}$ whose limit is an $\infty$-category $\calC$.
For each vertex $s \in S$, we will denote by $\calC_{s}$ the image of $s$ in $\Cat_{\infty}$.
Let $p: K \rightarrow \calC$ be a diagram. Assume that:
\begin{itemize}
\item[$(i)$] For every vertex $s \in S$, the resulting diagram
$p_{s}: K \rightarrow \calC_{s}$ admits a colimit
$\overline{p}_{s}: K^{\triangleright} \rightarrow \calC_{s}$.
\item[$(ii)$] For every edge $s \rightarrow s'$ in $S$, the induced functor
$f: \calC_{s} \rightarrow \calC_{s'}$ carries $\overline{p}_{s}$ to a colimit diagram
$K^{\triangleright} \rightarrow \calC_{\beta}$.
\end{itemize}
Then:
\begin{itemize}
\item[$(1)$] The diagram $p$ admits a colimit in $\calC$.
\item[$(2)$] Let $\overline{p}: K^{\triangleright} \rightarrow \calC$ be an arbitrary extension of $p$. Then $\overline{p}$ is a colimit diagram if and only if each of the induced diagrams $K^{\triangleright} \rightarrow \calC_{s}$ is a colimit diagram.
\end{itemize}
\end{lemma}

\begin{proof}
The diagram $q$ is classified by a coCartesian fibration $g: X \rightarrow S$ (see \S \toposref{universalfib}). Let $\calD = \bHom_{S}(S,X)$ denote the $\infty$-category of sections of $g$. 
According to Corollary \toposref{blurt}, we can identify $\calC$ with the full subcategory of
$\calD$ spanned by coCartesian sections of $g$. Under this identification, we may view
$p$ as defined by a map $P: K \times S \rightarrow X$. Using Lemma \monoidref{surtybove}
and Proposition \toposref{relcolfibtest}, we conclude that there is an extension
$\overline{P}: K^{\triangleright} \times S \rightarrow X$ which classifies a 
colimit diagram $\overline{p}: K^{\triangleright} \rightarrow \calD$, and having
the property that for each $s \in S$ the induced map $\overline{p}_{s}: K^{\triangleright} \rightarrow \calC_{s}$ is a colimit diagram. Using condition $(ii)$, we deduce that $\overline{p}$ factors
through $\calC \subseteq \calD$, and is therefore a colimit diagram in $\calC$.
This proves $(1)$. The ``only if'' direction of $(2)$ follows from the uniqueness of colimit diagrams.

To prove the ``if'' direction of $(2)$, let $\overline{p}: K^{\triangleright} \rightarrow
\calC$ be the colimit diagram constructed above, and let $\overline{p}': K^{\triangleright} \rightarrow \calC$ be an arbitrary extension of $p$. Then there exists a map $\alpha: \overline{p} \rightarrow \overline{p}'$ in $\calC_{p/}$. If $\overline{p}'$ induces colimit diagrams in
each $\calC_{s}$, then we conclude that each of the induced transformations
$\alpha_{s}: \overline{p}_s \rightarrow \overline{p}'_{s}$ is an equivalence. It follows
that $\alpha$ is an equivalence, so that $\overline{p}'$ is a colimit diagram as desired.
\end{proof}

\begin{proof}[Proof of Proposition \ref{poofle}]
According to the Barr-Beck Theorem (Theorem \monoidref{barbeq}), it will suffice to verify the following conditions:
\begin{itemize}
\item[$(1)$] The functor $g$ is conservative.
\item[$(2)$] If $U_{\bigdot}$ is a simplicial object of $\Spectra(\calD)$ which is $g$-split, then
$U_{\bigdot}$ admits a colimit in $\Spectra(\calD)$, and that colimit is preserved by $g$.
\end{itemize} 

Let $\alpha: X \rightarrow Y$ be a morphism in $\Spectra(\calD)$. Suppose that $g(\alpha)$ is an
equivalence. We wish to show that $\alpha$ is an equivalence. It will suffice to show that for
each $n \geq 0$, the induced map $\Omega^{\infty-n}_{\calD}(\alpha)$ is an equivalence in
$\calD$. This follows from the fact that $\Omega^{\infty-n}_{\calC}( g(\alpha) )$ is
an equivalence in $\calC$, since $G$ induces a conservative functor from
$\calD$ to $\calC$. This proves $(1)$.

We now prove $(2)$. Let $U_{\bigdot}$ be a $g$-split simplicial object of $\Spectra(\calD)$.
For each $n \geq 0$, the composition $\Omega^{\infty-n}_{\calD} U_{\bigdot}$ is a $G$-split
simplicial object of $\calD$. It follows from Theorem \monoidref{barbeq} that
$\Omega^{\infty-n}_{\calD} U_{\bigdot}$ admits a colimit $V^{n}_{\bigdot} \in \calD$, which is preserved by $G$. It will therefore suffice to show that we can assemble the augmented simplicial objects
$\{ V^{n}_{\bigdot} \}_{n \geq 0}$ into colimit diagram in $\Spectra(\calD)$. 
Applying Lemma \ref{precougher} to the tower
$$ \ldots \stackrel{ \Omega_{\calD} }{\rightarrow} \calD \stackrel{ \Omega_{\calD} }{\rightarrow} \calD,$$
we are reduced to showing that each of the induced maps
$V^{n}_{-1} \rightarrow \Omega_{\calD} V^{n+1}_{-1}$ is an equivalence in $\calD$. 
Since $G$ is conservative, we are reduced to proving that the induced map
$GV^{n}_{-1} \rightarrow G\Omega_{\calD} V^{n+1}_{-1} \simeq \Omega_{\calC} G V^{n+1}_{-1}$
is an equivalence in $\calC$. This follows from the fact that $\Omega_{\calC}$ preserves the colimits
of split simplicial objects in $\calC$ (Remark \monoidref{bsplit}). 
\end{proof}

For applications elsewhere, we record the following consequence of the Barr-Beck theorem:

\begin{proposition}\label{cuppp}
Suppose given a commutative diagram of $\infty$-categories
$$ \xymatrix{ & \calD \ar[dr]^{G''} & \\
\calE \ar[ur]^{G'} \ar[rr]^{G} & & \calC. }$$
Suppose that $G$ exhibits $\calE$ as monadic over $\calC$, that $G''$ is conservative, and that
$G'$ admits a left adjoint. Then $G'$ exhibits $\calE$ as monadic over $\calD$.
\end{proposition}

\begin{warning}
Monadicity is not transitive: in the situation of Proposition \ref{cuppp}, if $G''$ exhibits $\calD$ as monadic over $\calC$ and $G'$ exhibits $\calE$ as monadic over $\calD$, then $G$ need not exhibit $\calE$ as monadic over $\calC$.
\end{warning}

\begin{proof}
We will show that the functor $G'$ satisfies the hypotheses of Theorem \monoidref{barbeq}:
\begin{itemize}
\item[$(1)$] The functor $G'$ is conservative. For suppose that $\alpha: E \rightarrow E'$ is a morphism in $\calE$ such that $G'(\alpha)$ is an equivalence. Then $G(\alpha)$ is an equivalence. Since $G$ is conservative (by Theorem \monoidref{barbeq}), we deduce that $\alpha$ is an equivalence.

\item[$(2)$] Let $E_{\bigdot}$ be a simplicial object of $\calE$, and suppose that
$G' E_{\bigdot}$ is a split simplicial object of $\calD$. Then $G E_{\bigdot} = G'' G' E_{\bigdot}$ is a split simplicial object of $\calC$. Since $G$ exhibits $\calE$ as monadic over $\calC$, Theorem \monoidref{barbeq} implies that $E_{\bigdot}$ admits a colimit $| E_{\bigdot} |$ in $\calE$, and that the canonical map $\alpha: | G E_{\bigdot} | \rightarrow G | E_{\bigdot} |$ is an equivalence in $\calC$. We wish to show that $\beta: | G' E_{\bigdot} | \rightarrow G' | E_{\bigdot} |$ is an equivalence in $\calD$. Since
$G''$ is conservative, it will suffice to show that $G''(\beta)$ is an equivalence in $\calC$. Using the commutative diagram
$$ \xymatrix{ & G'' | G' E_{\bigdot} \ar[dr]^{ G''(\beta)} & \\
| G E_{\bigdot} | \ar[ur]^{\gamma} \ar[rr]^{\alpha} & & G | E_{\bigdot} |, }$$
we are reduced to proving that $\gamma$ is an equivalence. In other words, we must show that
$G''$ preserves the colimit of the split simplicial object $G' E_{\bigdot}$, which follows from
Remark \monoidref{bsplit}.
\end{itemize}
\end{proof}



\begin{thebibliography}{99}

\bibitem{chingarone} Arone, G. and M. Ching. {\it Operads and Chain Rules for the Calculus
of Functors.} Preprint.

\bibitem{giraud} Artin, M. {\it Th\'{e}orie des topos et cohomologie
\'{e}tale des sch\'{e}mas.} SGA 4. Lecture Notes in Mathematics
269, Springer-Verlag, Berlin and New York, 1972.

\bibitem{artinmazur} Artin, M. and B. Mazur. {\it \'{E}tale Homotopy.} Lecture Notes in Mathematics 100, Springer-Verlag, Berlin and New York, 1969.

\bibitem{basterra} Basterra, M. {\it \Andre-Quillen cohomology of commutative $S$-algebras.} Journal of Pure and Applied Algebra 144 (1999) no. 2, 111-143.

\bibitem{virtual} Behrend, K. and B. Fantechi. {\it The intrinsic
normal cone.} Inventiones Mathematicae 128 (1997) no. 1, 45-88.

\bibitem{BBD} Beilinson, A. , Bernstein, J. and P. Deligne.
{\it Faisceaux pervers.} Asterisuqe 100, Volume 1, 1982.

\bibitem{bergner} Bergner, J.E. {\it A Model Category Structure on the Category of Simplicial Categories.} Transactions of the American Mathematical Society 359 (2007), 2043-2058.

\bibitem{bergner2} Bergner, J.E. {\it A survey of $(\infty,1)$-categories.} Available at math.AT/0610239

\bibitem{bergner3} Bergner, J.E. {\it Rigidification of algebras over multi-sorted theories.}
Algebraic and Geometric Topoogy 7, 2007.

\bibitem{bergner4} Bergner, J.E. {\it Three models for the homotopy theory of homotopy theories,} Topology 46 (2007), 397-436.

\bibitem{LKM} Bosch, Guntzer, U., and R. Remmert, R. {\it Non-Archimedean Analysis: a Systematic Approach to Rigid Analytic Geometry.} Springer-Verlag, Berlin and Heidelberg, 1984.

\bibitem{bousfieldkan} Bousfield, A.K. and D.M. Kan. {\it Homotopy
limits, completions, and localizations.} Lecture Notes in
Mathematics 304, Springer-Verlag, 1972.

\bibitem{cismoer} Cisinski, D-C and I. Moerdijk. {\it Dendroidal sets as models for homotopy operads.} Available for download as arXiv:0902.1954v1.

\bibitem{combmodel} Dugger, D. {\it Combinatorial model categories have presentations.} Advances in Mathematics 164, 2001, 177-201.

\bibitem{eilenbergsteenrod} Eilenberg, S. and N.E. Steenrod. {\it Axiomatic approach to homology theory.} Proc. Nat. Acad. Sci. U.S.A. 31, 1945, 117-120.

\bibitem{eisenbud} Eisenbud, D. {\it Commutative algebra.} Springer-Verlag, New York, 1995. 

\bibitem{EKMM} Elmendorf, A.D., Kriz, I. , Mandell, M.A., and J.P.
May. {\it Rings, modules and algebras in stable homotopy theory.}
Mathematical Surveys and Monographs 47, American Mathematical
Society, 1997.

\bibitem{bezout} Fulton, W. {\it Algebraic curves.} W.A.
Benjamin, Inc., New York, 1969.

\bibitem{goerssjardine} Goerss, P. and J.F. Jardine. {\it Simplicial Homotopy Theory.} Progress in Mathematics, Birkhauser, Boston, 1999.

\bibitem{goodwillieI} Goodwillie, T. {\it Calculus I: The Þrst derivative of pseudoisotopy theory.} K-Theory 4 (1990), no. 1, 1Ð27. 

\bibitem{goodwillieII} Goodwillie, T. {\it Calculus. II. Analytic functors.} K-Theory 5 (1991/92), no. 4, 295Ð332.

\bibitem{goodwillie} Goodwillie, T. {\it Calculus III: Taylor Series.} Geometry and Topology, Volume 7 (2003) 645-711.

\bibitem{stein} Grauert, H. and R. Remmert. {\it Theory of Stein Spaces.} Springer-Verlag, Berlin Heidelberg, 2004.

\bibitem{gunning} Gunning, R. and H. Rossi. {\it Analytic functions of several complex variables.} Prentice-Hall, Englewood Cliffs, N.J, 1965.

\bibitem{hatcher} Hatcher, A. {\it Algebraic Topology}. Cambridge University Press, 2002.

\bibitem{bordism} Hook, E.C. {\it Equivariant cobordism and
duality.} Transactions of the American Mathematical Society 178
(1973) 241-258.

\bibitem{stablemodel} Hovey, M. {\it Model Categories.}
Mathematical Surveys and Monographs 63, AMS, Providence, RI, 1999.

\bibitem{symmetricspectra} Hovey, M., Shipley, B. and J. Smith. {\it Symmetric spectra.} Journal of the American Mathematical Society 13, 2000, no. 1, 149-208.

\bibitem{illusie} Illusie, L. {\it Complexe cotangent et d\'{e}formations
I}. Lecture Notes in Mathematics 239, Springer-Verlag, 1971.

\bibitem{illusie2} Illusie, L. {\it Complexe cotangent et d\'{e}formations
II}. Lecture Notes in Mathematics 283, Springer-Verlag, 1972.

\bibitem{mainjoyal} Joyal, A. {\it Notes on quasi-categories.}

\bibitem{joyalsimp} Joyal, A. {\it Simplicial categories vs. quasi-categories.}

\bibitem{joyalt} Joyal, A. and M. Tierney. {\it Quasi-categories vs. Segal Spaces.} Preprint available at 
math.AT/0607820. 

\bibitem{kerz} Kerz, M. {\it The complex of words and Nakaoka stability.} Homology, Homotopy and Applications, volume 7(1), 2005, pp. 77-85.

\bibitem{klein} Klein, J. and J. Rognes. {\it A chain rule in the calculus of homotopy functors.} Geom. Topol. 6 (2002), 853Ð887.

\bibitem{knutson} Knutson, D. {\it Algebraic spaces.} Lecture
Notes in Mathematics 203, Springer-Verlag, 1971.

\bibitem{categoricalring} Laplaza, M. {\it Coherence for
distributivity.} Coherence in categories, 29-65. Lecture Notes in
Mathematics 281, Springer-Verlag, 1972.

\bibitem{stacks} Laumon, G. and L. Moret-Bailly. {\it Champs
algebriques.} Springer-Verlag, 2000.

\bibitem{lazard} Lazard, Daniel. {\it Sur les modules plats.} C.R.
Acad. Sci. Paris 258, 1964, 6313-6316.

\bibitem{topoi} Lurie, J. {\it Higher Topos Theory.} Available for download at http://www.math.harvard.edu/~lurie/ .

\bibitem{DAGStable} Lurie, J. {\it Derived Algebraic Geometry I: Stable $\infty$-Categories.} Available for download.

\bibitem{monoidal} Lurie, J. {\it Derived Algebraic Geometry II: Noncommutative Algebra.} Available for download.

\bibitem{symmetric} Lurie, J. {\it Derived Algebraic Geometry III: Commutative Algebra.} Available for download.

\bibitem{deformation} Lurie, J. {\it Derived Algebraic Geometry IV: Deformation Theory.} Available for download.

\bibitem{structured} Lurie, J. {\it Derived Algebraic Geometry V: Structured Spaces.} Available for download.

\bibitem{spectral} Lurie, J. {\it Derived Algebraic Geometry VI: Spectral Schemes.} In preparation.

\bibitem{derivative} Lurie, J. {\it $(\infty,2)$-Categories and the Goodwillie Calculus I.} Available for download.

\bibitem{calculus} Lurie, J. {\it $(\infty,2)$-categories and the Goodwillie Calculus II.} In preparation.

\bibitem{elliptic1} Lurie, J. {\it Elliptic curves in spectral algebraic geometry.} In preparation.

\bibitem{ellipticloop} Lurie, J. {\it Toric varieties, elliptic cohomology at infinity, and loop group representations.} In preparation.

\bibitem{maclane} MacLane, S. {\it Categories for the Working Mathematician.} Second edition. Graduate Txts in Mathematics, 5. Springer-Verlag, New York, 1998.

\bibitem{gabriel} Mitchell, B. {\it A quick proof of the
Gabriel-Popesco theorem.} Journal of Pure and Applied Algebra 20
(1981), 313-315.

\bibitem{neeman} Neeman, A. {\it Triangulated categories.} Annals
of Mathematics Studies, 148. Princeton University Press, 2001.

\bibitem{homotopicalalgebra} Quillen, D. {\it Homotopical Algebra.} Lectures Notes in Mathematics 43, SpringerÐVerlag, Berlin, 1967. 

\bibitem{completesegal} Rezk, C. {\it A model for the homotopy theory of homotopy theory.} Transactions of the American Mathematical Society 35 (2001), no. 3, 973-1007.

\bibitem{homotopyvarieties} Rosicky, J. {\it On Homotopy Varieties.} Advances in Mathematics
214, 2007 no. 2, 525-550.

\bibitem{schwede} Schwede, S. {\it Spectra in model categories and applications to the algebraic
cotangent complex.} Journal of Pure and Applied Algebra 120 (1997), pp.
77-104.

\bibitem{monmod} Schwede, S. and B. Shipley. {\it Algebras and Modules in Monoidal Model Categories.} Proceedings of the London Mathematical Society (80) 2000, 491-511.

\bibitem{schwedeshipley} Schwede, S. and B. Shipley. {\it Stable model categories are categories of modules.} Topology 42, 2003, no. 1, 103-153.

\bibitem{intersection} Serre, Jean-Pierre. {\it Local algebra.}
Springer-Verlag, 2000.

\bibitem{shipley} Shipley, B. {\it A Convenient Model Category for Commutative Ring Spectra.} Homotopy theory: relations with algebraic geometry, group cohomology, and algebraic $K$-theory. Contemp. Math. volume 346 pp. 473-483, American Mathematical Society, Providence, RI, 2004. 

\bibitem{spivak} Spivak, D. {\it Quasi-smooth Derived Manifolds.} PhD dissertation.

\bibitem{srinivas} Srinivas, V. {\it Algebraic K-Theory.} Birkhauser, Boston, 1993.

\bibitem{toen} To\"{e}n, B. {\it Champs affines.} Available for
download: math.AG/0012219.

\bibitem{toenchar} To\"{e}n, B. {\it Vers une axiomatisation de la th\'{e}orie des cat\'{e}gories sup\'{e}riures.} K-theory 34 (2005), no. 3, 233-263.

\bibitem{toen2} To\"{e}n, B. and G. Vezzosi. {\it From HAG to DAG:
derived moduli stacks.} Available for download: math.AG/0210407.

\bibitem{toen3} To\"{e}n, B. and G. Vezzosi. {\it Algebraic
geometry over model categories.} Available for download:
math.AG/0110109.

\bibitem{toen4} To\"{e}n, B. and G. Vezzosi. {\it ``Brave New''
Algebraic Geometry and global derived moduli spaces of ring
spectra.} Available for download: math.AT/0309145.

\bibitem{toen5} To\"{e}n, B. and G. Vezzosi. {\it Segal topoi and
stacks over Segal categories.} Available for download:
math.AG/0212330.

\bibitem{toenK} To\"{e}n, B. and G. Vezzosi. {\it A remark on K-theory and S-categories.} Topology 43, No. 4 (2004), 765-791


\bibitem{verity} Verity, D. {\it Weak complicial sets, a simplicial weak omega-category
theory. Part I: basic homotopy theory.} 

\bibitem{verity2} Verity, D. {\it Weak complicial sets, a simplicial weak omega-category
theory. Part II: nerves of complicial Gray-categories.}

\bibitem{weibel} Weibel, C. {\it An Introduction to Homological Algebra.} Cambridge University Press, 1995.

\end{thebibliography}
\end{document}